\documentclass[a4paper]{amsart}
\usepackage[english]{babel}
\usepackage[utf8]{inputenc}
\usepackage[T1]{fontenc}
\usepackage{csquotes} 
\usepackage{amssymb}
\usepackage{float} 
\usepackage{amsthm}
\usepackage[top=2cm, bottom=2cm, left=2cm, right=2cm,twoside=false]{geometry}
\usepackage{mathtools} 
\usepackage[usenames,x11names]{xcolor}
\usepackage[all]{xy} 
\usepackage[normalem]{ulem} 
\usepackage{multirow} 
\usepackage{multicol}

\usepackage{nicefrac}

\usepackage[shortlabels]{enumitem}
\setlist[itemize]{leftmargin=2em}
\setlist[1]{leftmargin=*}
\setlist[enumerate,1]{label=(\alph*), topsep=0.2em}
\setlist[enumerate,2]{label=(\roman*), ref=(\alph{enumi}.\roman*)}
\newlist{enumerate-alt}{enumerate}{1}
\setlist[enumerate-alt,1]{label=(\roman*)}
\newlist{longlist}{enumerate}{1}
\setlist[longlist]{label=\small{(\arabic*)}, itemsep=0.2em}
\newlist{shortlist}{enumerate}{1}
\setlist[shortlist]{label=\small{(\Alph*)}}
\newlist{parlist}{enumerate}{2}
\setlist[parlist,1]{leftmargin=0cm, itemindent=2\parindent, label=(\alph*), itemsep=0.2em}
\setlist[parlist,2]{label=(\roman*), ref=(\alph{parlisti}.\roman*), itemindent=\parindent}
\newlist{parts}{enumerate}{3}
\setlist[parts,1]{leftmargin=0cm, itemindent=2\parindent, label=(\textbf{\arabic*}), ref=(\arabic*), itemsep=0.2em}
\setlist[parts,2]{label=(\alph*), ref=(\arabic{partsi}.\alph*), topsep=0.2em}
\setlist[parts,3]{label=(\roman*), ref=(\arabic{partsi}.\alph{partsii}.\roman*)}
\newlist{steps}{enumerate}{1}
\setlist[steps]{align=left, listparindent=\parindent, parsep=\parskip, leftmargin=0em, labelwidth=0pt, itemindent=1em,labelsep=.4em, topsep=.4em, itemsep=.4em}	
\setlist[steps,1]{label={\textbf{Step~\arabic*.}},  ref=\textbf{\arabic*}}
\newcounter{foo} 
\newlist{casesp}{enumerate}{5} 
\setlist[casesp]{align=left, 
	listparindent=\parindent, 
	parsep=\parskip, 
	font=\normalfont\bfseries, 
	leftmargin=0pt, 
	labelwidth=0pt, 
	itemindent=.4em,labelsep=.4em, 
	topsep=.4em, 
	itemsep=.4em, 
}
\setlist[casesp,1]{label=Case~\arabic*:,ref=\arabic*}
\setlist[casesp,2]{label=Subcase~\thecasespi.\arabic*:,ref=\thecasespi.\arabic*}
\setlist[casesp,3]{label=Subcase~\thecasespii.\arabic*:,ref=\thecasespii.\arabic*}
\setlist[casesp,4]{label=Subcase~\thecasespiii.\arabic*:,ref=\thecasespiii.\arabic*}
\setlist[casesp,5]{label=Subcase~\thecasespiv.\arabic*:,ref=\thecasespiv.\arabic*}
\newcommand\litem[1]{\item{\bfseries #1.\enspace}}

\newlist{casesp*}{enumerate}{2}
\setlist[casesp*]{align=left, listparindent=\parindent, parsep=\parskip, font=\normalfont\bfseries, leftmargin=0pt, 	labelwidth=0pt, itemindent=.4em,labelsep=.4em, topsep=.6em, itemsep=.6em }
\setlist[casesp*,1]{label=Case,ref=\arabic*}
\setlist[casesp*,2]{label=Subcase,ref=\arabic*}
\newlist{types}{enumerate}{1}
\setlist[types]{align=left, listparindent=\parindent, parsep=\parskip, font=\normalfont\bfseries, leftmargin=0pt, 	labelwidth=0pt, itemindent=.4em,labelsep=.4em, topsep=.6em, itemsep=.6em }
\setlist[types,1]{label=Type,ref=\arabic*}

\usepackage{subcaption} 
\usepackage{tikz}
\usepackage{tikz-cd}
\usepackage{pgfplots}
\usetikzlibrary{calligraphy,decorations}
\tikzset{%
	symbol/.style={%
		draw=none,
		every to/.append style={%
			edge node={node [sloped, allow upside down, auto=false]{$#1$}}
		}
	}
}

\usepackage{hyperref} 
\hypersetup{
	linkcolor=DodgerBlue3, 
	citecolor  = teal,
	urlcolor   = DodgerBlue4,
	colorlinks = true, 
	bookmarksdepth=3,
	linktoc=all
}
\makeatletter
\@namedef{subjclassname@2020}{%
	\textup{2020} Mathematics Subject Classification}
\makeatother

\makeatletter
\newcommand{\myitem}[1]{%
	\item[#1]\protected@edef\@currentlabel{#1}%
}
\makeatother

\makeatletter 
\renewcommand\subsection{\@startsection{subsection}{3}
	\z@{.5\linespacing\@plus.7\linespacing}{.5\linespacing}
	{\bfseries\itshape}} 
\renewcommand\paragraph{\@startsection{paragraph}{4}%
	\z@{.5\linespacing\@plus.7\linespacing}{-.5\linespacing}%
	{\normalfont\bfseries}}
\makeatother

\usepackage{indentfirst}

\makeatletter \renewenvironment{proof}[1][\proofname]{
	\par\pushQED{\qed}\normalfont
	\topsep6\p@\@plus6\p@\relax
	\trivlist\item[\hskip\labelsep\bfseries#1\@addpunct{.}]
	\ignorespaces}{
	\popQED\endtrivlist\@endpefalse} \makeatother

\theoremstyle{plain}
\newtheorem{theorem}{Theorem}[section]
\newtheorem{theoremA}{Theorem} 
\newtheorem{propositionA}[theoremA]{Proposition} 
\newtheorem*{theorem*}{Theorem}
\newtheorem{conjecture}[theorem]{Conjecture}

\theoremstyle{definition}

\newtheorem{corollary}[theorem]{Corollary}
\newtheorem{definition}[theorem]{Definition}
\newtheorem{lemma}[theorem]{Lemma}
\newtheorem{proposition}[theorem]{Proposition}
\newtheorem{example}[theorem]{Example}
\newtheorem{notation}[theorem]{Notation}
\newtheorem{observation}[theorem]{Observation}
\newtheorem{remark}[theorem]{Remark}
\newtheorem{claim}{Claim}
\newtheorem*{claim*}{Claim}

\let\sec\S 

\newcommand{\A}{\mathbb{A}}
\newcommand{\C}{\mathbb{C}}
\newcommand{\F}{\mathbb{F}}
\newcommand{\G}{\mathbb{G}}

\renewcommand{\P}{\mathbb{P}}
\newcommand{\Q}{\mathbb{Q}}

\newcommand{\Z}{\mathbb{Z}}

\newcommand{\cA}{\mathcal{A}}

\newcommand{\cC}{\mathcal{C}}
\newcommand{\cD}{\mathcal{D}}
\newcommand{\cE}{\mathcal{E}}
\newcommand{\cF}{\mathcal{F}}

\newcommand{\cH}{\mathcal{H}}

\newcommand{\cL}{\mathcal{L}}
\newcommand{\cM}{\mathcal{M}}
\newcommand{\cMst}{\mathcal{M}_{\star}}
\newcommand{\cN}{\mathcal{N}}
\newcommand{\cO}{\mathcal{O}}
\newcommand{\cP}{\mathcal{P}}
\newcommand{\Pht}{\cP_{\height\leq 2}}
\newcommand{\Phtt}{\cP_{\height=1}}
\newcommand{\Phtw}{\cP_{\height=2}^{\width=2}}
\newcommand{\Phtsep}{\cP_{\height=2}^{\textnormal{sep}}}
\newcommand{\Phtinsep}{\cP_{\height=2}^{\textnormal{insep}}}
\newcommand{\Pdeb}{\cP_{\textnormal{deb}}}
\newcommand{\Pcusp}{\cP_{\textnormal{cusp}}}

\newcommand{\Pnode}{\cP_{\textnormal{node}}}

\newcommand{\PKM}{\cP_{\textnormal{\ref{ex:ht=2_twisted_cha=2}}}}
\newcommand{\PKMres}{\tilde{\cP}_{\textnormal{\ref{ex:ht=2_twisted_cha=2}}}}
\newcommand{\Fib}{\operatorname{Fib}_{+}}

\newcommand{\Phttres}{\tilde{\cP}_{\height=1}}
\newcommand{\Phtwres}{\tilde{\cP}_{\height=2}^{\width=2}}

\newcommand{\Phtinsepres}{\tilde{\cP}_{\height=2}^{\textnormal{insep}}}

\newcommand{\cQ}{\mathcal{Q}}
\newcommand{\cR}{\mathcal{R}}
\newcommand{\cS}{\mathcal{S}}
\newcommand{\cT}{\mathcal{T}}

\newcommand{\cV}{\mathcal{V}}
\newcommand{\cW}{\mathcal{W}}
\newcommand{\cX}{\mathcal{X}}
\newcommand{\cY}{\mathcal{Y}}
\newcommand{\cZ}{\mathcal{Z}}

\newcommand{\rA}{\mathrm{A}}
\newcommand{\rD}{\mathrm{D}}
\newcommand{\rE}{\mathrm{E}}
\newcommand{\rC}{\mathrm{C}}
\newcommand{\rN}{\mathrm{N}}

\usepackage[scr=boondoxo]{mathalfa}

\renewcommand{\ll}{\mathscr{l}}
\newcommand{\cc}{\mathscr{c}}
\newcommand{\qq}{\mathscr{q}}
\newcommand{\pp}{\mathscr{p}}
\newcommand{\hh}{\mathscr{h}}
\newcommand{\kk}{\mathscr{k}}

\newcommand{\rr}{\mathscr{r}}

\renewcommand{\epsilon}{\varepsilon}
\renewcommand{\phi}{\varphi}
\renewcommand{\theta}{\vartheta}

\DeclareFontFamily{U}{mathx}{}
\DeclareFontShape{U}{mathx}{m}{n}{<-> mathx10}{}
\DeclareSymbolFont{mathx}{U}{mathx}{m}{n}
\DeclareMathAccent{\widehat}{0}{mathx}{"70}
\DeclareMathAccent{\widecheck}{0}{mathx}{"71}

\renewcommand{\tilde}{\widetilde}
\renewcommand{\check}{\widecheck}
\renewcommand{\hat}{\widehat}
\renewcommand{\bar}{\overline}

\renewcommand{\leq}{\leqslant}
\renewcommand{\geq}{\geqslant}

\renewcommand{\to}{\longrightarrow}
\newcommand{\map}{\dashrightarrow}

\newcommand{\into}{\hookrightarrow}

\newcommand{\sqto}{\rightsquigarrow}

\newcommand{\cha}{\operatorname{char}} 
\newcommand{\hot}{\textnormal{h.o.t.}} 

\newcommand{\Proj}{\operatorname{Proj}}

\newcommand{\Aut}{\operatorname{Aut}}
\newcommand{\Sing}{\operatorname{Sing}}

\newcommand{\NS}{\operatorname{NS}}
\newcommand{\Exc}{\operatorname{Exc}}
\newcommand{\Supp}{\operatorname{Supp}}

\newcommand{\redd}{_{\mathrm{red}}} 
\newcommand{\Bs}{\operatorname{Bs}} 
\newcommand{\reg}{^{\mathrm{reg}}} 
\newcommand{\trp}{^{\scriptscriptstyle{\top}}} 

\newcommand{\PGL}{\mathrm{PGL}}

\newcommand{\cf}{\operatorname{cf}}

\newcommand{\lts}[2]{\mathcal{T}_{#1}(-\log #2)} 

\newcommand{\hor}{_{\mathrm{hor}}} 
\renewcommand{\vert}{_{\mathrm{vert}}} 
\newcommand{\lt}{_{\mathrm{lt}}} 
\newcommand{\ftip}[1]{\mathrm{tip}^{+}(#1)} 
\newcommand{\ltip}[1]{\mathrm{tip}^{-}(#1)} 

\newcommand{\cp}[1]{^{(#1)}} 

\newcommand{\am}{_{\textnormal{am}}}

\newcommand{\height}{\operatorname{ht}} 
\newcommand{\width}{\operatorname{wd}} 

\renewcommand{\d}{\partial}
\newcommand{\id}{\mathrm{id}}
\newcommand{\pt}{\mathrm{pt}}

\newcommand{\toin}[1]{\overset{#1}{\to}} 

\newcommand{\Sec}{\Xi}
\newcommand{\bs}[1]{\boldsymbol{#1}} 
\newcommand{\ub}[1]{\uline{\bs{#1}}} 
\newcommand{\fh}[1]{\bs{#1^{+}}} 
\newcommand{\lh}[1]{\bs{#1^{-}}} 
\newcommand{\gforkd}[2]{\prescript{| #1 |}{} \!\langle\!\langle #2 \rangle\!\rangle} 
\newcommand{\gfork}[1]{\! \langle\!\langle #1 \rangle\!\rangle} 

\newcommand{\dec}[1]{^{|#1|}}
\newcommand{\adec}[1]{^{\langle #1 \rangle}} 
\newcommand{\ldec}[1]{\prescript{|#1|}{}} 
\newcommand{\decb}[2]{^{|#1|}_{|#2|}}
\newcommand{\ldecb}[2]{\prescript{| #1 |}{| #2 |}}
\newcommand{\lbr}{[\![}
\newcommand{\rbr}{]\!]}
\newcommand{\pr}{\operatorname{pr}}

\newcommand{\de}{\coloneqq} 

\begin{document}
	\title[del Pezzo surfaces I \& II]{Classification of del {P}ezzo surfaces of rank one \\  I. Height 1 and 2, \\ II. Descendants with elliptic boundaries}
	
	\author{Karol Palka}
	\address{Institute of Mathematics, Polish Academy of Sciences, \'{S}niadeckich 8, 00-656 Warsaw, Poland}
	\email{palka@impan.pl}
	
	\author{Tomasz Pe{\l}ka}
	\address{University of Warsaw, Faculty of Mathematics, Informatics and Mechanics, Banacha 2, 02-097 Warsaw, Poland}
	\email{tpelka@mimuw.edu.pl}


	\thanks{This project was funded by the National Science Centre, Poland, grant no.\ 2021/41/B/ST1/02062. For the purpose of Open Access, the authors have applied a CC-BY public copyright license to any Author Accepted Manuscript version arising from this submission.}

\begin{abstract}
This is the first article in a series aimed at classifying normal del Pezzo surfaces of Picard rank one over algebraically closed fields of arbitrary characteristic up to an isomorphism. Our guiding invariant is the \emph{height} of a del Pezzo surface, defined as the minimal intersection number of the exceptional divisor of the minimal resolution and a fiber of some $\P^1$-fibration. The geometry of del Pezzo surfaces gets more constrained as the height grows; in characteristic $0$ no example of height bigger than $4$ is known.

In this article, we classify del Pezzo surfaces of Picard rank one and height at most $2$; in particular we describe the non-log terminal ones. We also describe a natural class of del Pezzo surfaces which have \emph{descendants with elliptic boundary}, i.e.\ whose minimal resolution has a birational morphism onto a canonical del Pezzo surface of rank one mapping the exceptional divisor to an anti-canonical curve. 
\end{abstract}
%

\maketitle
\setcounter{tocdepth}{1}
\tableofcontents
\section{Introduction}

A normal surface $\bar{X}$ is \emph{del Pezzo} if its anti-canonical divisor $-K_{\bar{X}}$ is ample. This is the first article in a series aimed at a complete classification of del Pezzo surfaces of Picard rank one over any algebraically closed field $\kk$ up to an isomorphism. We will provide an explicit construction of every such surface; and a list of singularity types of all log canonical ones. For each log canonical type, we will compute the number of non-isomorphic del Pezzo surfaces realizing it, or, in case this number is infinite, their moduli dimension, see Definition \ref{def:moduli}\ref{item:def-family-faithful}. Here, by the \emph{singularity type} we mean the weighted graph of the exceptional divisor of the minimal resolution, see Section \ref{sec:log_surfaces}. Non--log terminal del Pezzo surfaces are described in Proposition \ref{prop:non-log-terminal}.
\smallskip

Del Pezzo surfaces are possible  outcomes of the two-dimensional Minimal Model Program, and as such they play an important role in the birational classification of varieties. They were investigated by many authors from various points of view. For example, in case $\kk=\C$, Keel--McKernan \cite{Keel-McKernan_rational_curves} proved that that the smooth locus of any log terminal del Pezzo surface is  dominated by images of $\A^{1}$, and Gurjar--Zhang \cite{GurZha_1} showed that its fundamental group is finite.

The proof of Keel--McKernan's result provides in fact a description of \enquote{all but a bounded family} of del Pezzo surfaces of rank one. This description was recently completed by Lacini \cite{Lacini} in case $\cha\kk\neq 2,3$, who arranged all log terminal del Pezzo surfaces of rank one into 24 series. Nonetheless, these series are non-disjoint and rather broad, hence questions such as the existence of a del Pezzo surface of rank one with prescribed singularities are still not easy to answer. In general the problem of uniqueness remained untouched. 
Many other partial classification results were achieved by various authors, see e.g.\  \cite{Zhang_dP3,Nakayama_delPezzo-index-2,Kojima_index-2,Liu-Shokurov,Belousov_4-sings}. We compare some of them with our classification in Section \ref{sec:comparison}.

\subsection{The height and the classification program}

In a series of articles we pursue a new approach, which gives a new geometric insight into the structure of del Pezzo surfaces. Our guiding invariant is the \emph{height}. 

 \begin{definition}[The height]
 	\label{def:height}
 	Let $\bar{X}$ be a normal projective surface and $\bar{D}$ a reduced Weil divisor on it. Let $(X,D)$ be the minimal log resolution of $(\bar{X},\bar{D})$.  The \emph{height} of $(\bar{X},\bar{D})$, denoted by $\height(\bar{X},\bar{D})$, is the infimum of the set of integers $\hh$ such that $X$ admits a $\P^1$-fibration whose fiber $F$ satisfies $F\cdot D=\hh$. We put $\height(\bar{X})\de \height(\bar{X},0)$.
\end{definition}

We note that the height of a normal surface is finite only if $\bar{X}$ has negative Kodaira dimension; in particular the height of a singular del Pezzo surface is finite and nonzero, see Remark \ref{rem:ht_finite_nonzero}. The key property which makes the height important is that, contrary to other invariants (like the Fano index), it is bounded. This has been observed by the first author several years ago, who, to his surprise, noticed that the bound seems to be very small, and proposed a classification program based of studying the 
$\P^1$-fibrations realizing the infimum in Definition \ref{def:height}. The following result will be proved in a forthcoming article. 

\begin{conjecture}[Palka]
Let $\bar{X}$ be a singular del Pezzo surface. 
If $\cha\kk\neq 2,3$ then $\height(\bar{X})\leq 4$. 
\end{conjecture}

The exceptions with $\height(\bar{X})>4$ in $\cha\kk=2,3$, cf.\ Theorem \ref{thm:GK}\ref{item:GK-ht}, are related to the exceptional geometry of canonical del Pezzo surfaces in low characteristics, and will be classified. Crucially, the geometry of del Pezzo surfaces gets more constrained as the height increases, and our knowledge on $\P^1$-fibrations gives strong control over the configurations of possible degenerate fibers and various auxiliary curves, including the cases of low field characteristics. For this reason the height is a natural uniform invariant for the purpose of classification.

\smallskip
In \cite{PaPe_MT} we have already shown that for singular \emph{open} del Pezzo surfaces, i.e.\ the ones admitting a nonzero reduced boundary $\bar{D}$ such that $-K_{\bar {X}}-\bar{D}$ remains ample, the height of $(\bar{X},\bar{D})$, and hence of $\bar{X}$, is bounded by~$2$, with some exceptions if $\cha\kk=2,3,5$. In this article we obtain the following results.
\begin{itemize}
\item In Theorem \ref{thm:ht=1,2} we describe all normal del Pezzo surfaces of rank 1 and height at most $2$.
\item In Proposition \ref{prop:non-log-terminal} we show that with some exceptions if $\cha\kk=2,3,5$ the above class contains all non--log terminal del Pezzo surfaces of rank $1$, and we describe their structure.
\item In Propositions \ref{prop:moduli} and \ref{prop:moduli-hi} we discuss the issues of uniqueness, rigidity, and moduli for log canonical ones.
\item In Theorem \ref{thm:GK} we describe a notable class of del Pezzo surfaces of rank $1$ which can be constructed directly from canonical ones and singular members of anticanonical linear systems of the latter.
\end{itemize}

The classification results we achieve are complete, in the sense that we not only concentrate on studying possible singularities of del Pezzo surfaces, but we also carefully control their constructions. This allows us to discuss moduli of surfaces of fixed singularity type or to prove their uniqueness up to an isomorphism. 

We warn the reader that the price paid for including in the analysis fields of low characteristic, especially $\cha\kk=2$, are additional calculations, which sometimes become tedious. There seems to be no royal road here, as the resulting classification contains many exotic cases, see Tables \ref{table:ht=1}--\ref{table:exceptions-to-moduli}. 

\subsection{Del Pezzo surfaces of height at most two}

The main part of this article treats the case $\height(\bar{X})\leq 2$, so 
our $\P^1$-fibrations are relatively simple, but flexible enough to produce a lot of examples, sometimes with moduli. 
Describing their degenerate fibers allows us to arrange the minimal log resolutions of all del Pezzo surfaces of rank 1 and height at most 2 into series whose subsequent elements are connected by 
birational transformations called \emph{elementary vertical swaps}.
\begin{definition}[Swap, see Figure \ref{fig:swap}]\label{def:vertical_swap}
	Let $X$ be a smooth projective surface, let $D$ be a reduced divisor on $X$ and let $p\colon X\to B$ be a $\P^1$-fibration.
	\begin{parlist}
		\item\label{item:def-swap-elementary} 
		An \emph{elementary swap} (respectively, an \emph{elementary swap over $B$}) is a blowdown $\phi\colon X\to X'$ such that the $(-1)$-curve $A\de \Exc\phi$ meets $D$ normally and satisfies $A\cdot D\leq 2$ and $A\cdot C=1$ for exactly one $(-2)$-curve $C\subseteq D$ (and, respectively, $p_{*}A=p_{*}C=0$).  We write $\phi$ as $(X,D)\sqto (X',D')$, where $D'=\phi_{*}(D-C)$.
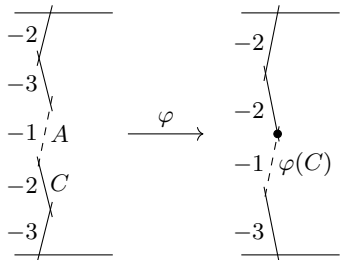
\begin{figure}[htbp]
	\begin{tikzpicture}
		\begin{scope}
			\draw (-0.3,3.1) -- (1,3.1);
			\draw (0.2,3.2) -- (0,2.4);
			\node at (-0.2,2.8) {\small{$-2$}};
			\draw (0,2.6) -- (0.2,1.8);
			\node at (-0.2,2.15) {\small{$-3$}};
			\draw[dashed] (0.2,2) -- (0,1);
			\node at (-0.2,1.5) {\small{$-1$}};
			\node at (0.3,1.5) {\small{$A$}};
			\draw (0,1.2) -- (0.2,0.4);
			\node at (-0.2,0.8) {\small{$-2$}};
			\node at (0.3,0.85) {\small{$C$}};
			\draw (0.2,0.6) -- (0,-0.2);
			\node at (-0.2,0.2) {\small{$-3$}};
			\draw (-0.3,-0.1) -- (1,-0.1);
			\draw[->] (1.2,1.5) -- (2.2,1.5);
			\node at (1.7,1.7) {\small{$\phi$}}; 
			%
		\end{scope}
		\begin{scope}[shift={(3,0)}]
			\draw (-0.3,3.1) -- (1,3.1);
			\draw (0.2,3.2) -- (0,2.2);
			\node at (-0.2,2.7) {\small{$-2$}};
			\draw (0,2.4) -- (0.2,1.4);
			\node at (-0.2,1.8) {\small{$-2$}};
			\draw[dashed] (0.2,1.6) -- (0,0.6);
			\node at (-0.2,1.1) {\small{$-1$}};
			\node at (0.55,1.1) {\small{$\phi(C)$}};
			\filldraw (0.18,1.5) circle (0.05);
			\draw (0,0.8) -- (0.2,-0.2);
			\node at (-0.2,0.2) {\small{$-3$}};
			\draw (-0.3,-0.1) -- (1,-0.1);
			%
		\end{scope}
	\end{tikzpicture}
	\vspace{-1em}
	\caption{Example of an elementary vertical swap $\phi\colon (X,D)\sqto(X',D')$, see Definition \ref{def:vertical_swap}.}
\label{fig:swap}
\vspace{-1em}
\end{figure}		
		\item\label{item:def-swap-general} A \emph{swap} (respectively, a \emph{swap over $B$}) is a composition of a sequence $(X,D)\sqto\dots\sqto (X',D')$ of elementary swaps (respectively, elementary swaps over $B$). We write it shortly as $(X,D)\sqto (X',D')$. Whenever the fixed $\P^1$-fibration $X\to B$ is clear from the context, we refer to a swap over $B$ as a \emph{vertical swap}.
		\item\label{item:def-swapping}  We say that 
		$(X,D)$ \emph{swaps vertically} to $(X',D')$ if there is a  $\P^1$-fibration $X\to B$ and a swap $(X,D)\sqto (X',D')$ over $B$.
		\item\label{item:def-swapping-normal}  Given two normal surfaces $\bar{X}$, $\bar{X}'$, we say that $\bar{X}$ \emph{swaps vertically to} $\bar{X}'$ if the minimal log resolution of $\bar{X}$ swaps vertically to the minimal log resolution of $\bar{X}'$.				
		\item\label{item:def-swap-basic} We say that $(X,D)$ is \emph{primitive} (over $B$) if it does not admit any elementary vertical swaps (over $B$). 
		A normal surface $\bar{X}$ is \emph{primitive} if so is its minimal log resolution.
	\end{parlist} 
\end{definition}
If a log surface $(X,D)$ with a fixed $\P^1$-fibration $p\colon X\to B$ is primitive over $B$ (and $(X,D)$ is a minimal log resolution of a normal surface $\bar{X}$), we say that $(X,D)$ (and $\bar{X}$) are \emph{vertically primitive}, keeping the fixed $\P^1$-fibration $p$ implicit. This way $\bar{X}$ can be vertically primitive but not primitive, see Figure \ref{fig:2D4-intro} for an example; it can also be vertically primitive with respect to one $\P^1$-fibration but not another, see Example \ref{ex:ht=2_twisted}.
\begin{theoremA}[Classification of del Pezzo surfaces of rank one and height $\leq 2$]\label{thm:ht=1,2}
	Let $\bar{X}$ be a normal del Pezzo surface of rank $1$ and height at most $2$, and let $(X,D)$ be its minimal log resolution. Then the following hold.
	\begin{parts}
		\item\label{part:swaps} 
		The log surface 
		$(X,D)$ swaps vertically to a vertically primitive log surface $(Y,D_Y)$ with a $\P^1$-fibration shown in Figure \ref{fig:basic-intro}. Moreover, if  $\height(\bar{X})=2$ then $(Y,D_Y)$ is a minimal log resolution of a del Pezzo surface $\bar{Y}$ which is either canonical, or, in case $\cha\kk=2$, is a surface constructed in  Example \ref{ex:ht=2_twisted_cha=2}, see Figure \ref{fig:KM_surface-intro}. 
\begin{figure}[htbp]
	\subcaptionbox{$e>2g-2+\nu$ \\ Example \ref{ex:ht=1} \label{fig:ht=1-intro}}[.23\linewidth]{
		\begin{tikzpicture}[scale=0.9]
			\draw[very thick] (0,2) -- (3.4,2);
			\node at (2.2,2.2) {\small{$-e$}};
			\node at (2.2,1.75) {\small{genus $g$}};
			\draw (0.2,2.5) -- (0,1.5);
			\draw[dashed] (0,1.7) -- (0.2,0.7);
			\draw (0.2,0.9) -- (0,-0.1);
			\draw (1.2,2.5) -- (1,1.5);
			\draw[dashed] (1,1.7) -- (1.2,0.7);
			\draw (1.2,0.9) -- (1,-0.1);
			\node at (2.2,1.2) {\Large{$\cdots$}};
			\draw (3.4,2.5) -- (3.2,1.5);
			\draw[dashed] (3.2,1.7) -- (3.4,0.7);
			\draw (3.4,0.9) -- (3.2,-0.1);
			\draw [decorate, decoration = {calligraphic brace}, thick] (0.1,2.7) -- (3.5,2.7);
			\node at (2,3) {\scriptsize{$\nu$ fibers}};			
		\end{tikzpicture}
	}	
	\subcaptionbox{ $3\rA_{2}$ \\ Example  \ref{ex:ht=2}\ref{item:3A2_construction} \label{fig:3A2-intro}}[.14\linewidth]{\centering
		\begin{tikzpicture}[scale=0.9]
			\draw (0.1,3) -- (1.3,3);
			\draw[dashed] (0.2,3.1) -- (0,2.1);
			\draw (0,2.3) -- (0.2,1.4);
			\draw (0.2,1.6) -- (0,0.7);
			\draw[dashed] (0,0.9) -- (0.2,-0.1);
			%
			\draw (1.2,3.1) -- (1,1.9);
			\draw[dashed] (1,2.1) -- (1.2,0.9);
			\draw (1.2,1.1) -- (1,-0.1);
			%
			\draw (0.1,0.05) -- (1.1,0.05);
		\end{tikzpicture}
	}
	\subcaptionbox{$\rA_{1}+2\rA_{3}$\\ Example \ref{ex:ht=2}\ref{item:A1+2A3_construction} \label{fig:2A3+A1-intro}}[.19\linewidth]{\centering
		\begin{tikzpicture}[scale=0.9]
			\draw (0.1,3) -- (2.3,3);
			\draw (-0.1,0) -- (2.1,0);
			%
			\draw[dashed] (0.2,3.1) -- (0,1.9);
			\draw (0,2.1) -- (0.2,0.9);
			\draw[dashed] (0.2,1.1) -- (0,-0.1);
			%
			\draw (1.2,3.1) -- (1,1.9);
			\draw[dashed] (1,2.1) -- (1.2,0.9);
			\draw (1.2,1.1) -- (1,-0.1);
			\draw (2.2,3.1) -- (2,1.9);
			\draw[dashed] (2,2.1) -- (2.2,0.9);
			\draw (2.2,1.1) -- (2,-0.1);			
		\end{tikzpicture}
	}
	\subcaptionbox{$\rA_1+\rA_2+\rA_5$\\ Example \ref{ex:ht=2}\ref{item:A1+A2+A5_construction} \label{fig:A5+A2+A1-intro}}[.2\linewidth]{\centering
		\begin{tikzpicture}[scale=0.9]
			\draw (0.1,3) -- (2.4,3);
			\draw (0,0) -- (1,0) to[out=0,in=180] (2.3,2.8) -- (2.4,2.8);
			%
			\draw[dashed] (0.2,3.1) -- (0,2.1);
			\draw (0,2.3) -- (0.2,1.4);
			\draw (0.2,1.6) -- (0,0.7);
			\draw[dashed] (0,0.9) -- (0.2,-0.1);
			%
			\draw (1.2,3.1) -- (1,1.9);
			\draw[dashed] (1,2.1) -- (1.2,0.9);
			\draw (1.2,1.1) -- (1,-0.1);
			\draw (2.3,3.1) -- (2.1,1.9);
			\draw[dashed] (2.1,2.1) -- (2.3,0.9);
			\draw (2.3,1.1) -- (2.1,-0.1);			
		\end{tikzpicture}
	}	
	\subcaptionbox{$2\rD_4$\\ Example \ref{ex:ht=2}\ref{item:2D4_construction} \label{fig:2D4-intro}}[.21\linewidth]{\centering
		\begin{tikzpicture}[scale=0.9]
			\draw (0.1,3) -- (2.9,3);
			\draw (-0.1,0) -- (2.9,0);
			%
			\draw (0.2,3.1) -- (0,1.9);
			\draw[dashed] (0,2.1) -- (0.2,0.9);
			\draw (0.2,1.1) -- (0,-0.1);
			%
			\draw (1.2,3.1) -- (1,1.9);
			\draw[dashed] (1,2.1) -- (1.2,0.9);
			\draw (1.2,1.1) -- (1,-0.1);
			\draw (2.2,3.1) -- (2,1.9);
			\draw[dashed] (2,2.1) -- (2.2,0.9);
			\draw (2.2,1.1) -- (2,-0.1);
			\draw[dashed] (2.8,3.1) -- (3.1,1.4);
			\draw[dashed] (2.8,-0.1) -- (3.1,1.6);
		\end{tikzpicture}			
	}	
	\subcaptionbox{$2\rA_4$\\ Example \ref{ex:ht=2_meeting} \label{fig:2A4-intro}}[.2\linewidth]{\centering
		\begin{tikzpicture}[scale=0.9]
			\draw (0,3) -- (1.2,3) to[out=0,in=120] (2,1.4);
			\draw (0,0) -- (1,0) to[out=0,in=-120] (2,1.6);
			%
			\draw[dashed] (0.2,3.1) -- (0,2.4);
			\draw (0,2.55) -- (0.2,1.95);
			\draw (0.2,2.05) -- (0,1.45);
			\draw (0,1.55) -- (0.2,0.95);
			\draw (0.2,1.05) -- (0,0.45);
			\draw[dashed] (0,0.55) -- (0.2,-0.1);
			%
			\draw (1.1,3.1) -- (0.9,1.9);
			\draw[dashed] (0.9,2.1) -- (1.1,0.9);
			\draw (1.1,1.1) -- (0.9,-0.1);	
		\end{tikzpicture}
	}	
	\subcaptionbox{$\rA_3+\rD_5$\\ Example \ref{ex:ht=2_twisted}\ref{item:A3+D5_construction} \label{fig:D5+A3_twisted-intro}}[.22\linewidth]{
		\begin{tikzpicture}[scale=0.9]
			\draw (0.1,3) -- (0.3,3) to[out=0,in=170] (1.2,1.56) to[out=-10,in=-135] (1.3,1.61);
			\draw (-0.1,0) -- (0.2,0) to[out=0,in=170] (1.2,1.56) to[out=-10,in=135] (1.3,1.49);
			%
			\draw (0.2,3.1) -- (0,1.9);
			\draw[dashed] (0,2.1) -- (0.2,0.9);
			\draw (0.2,1.1) -- (0,-0.1);
			\draw[dashed] (1,1.6) -- (2.1,1.4);
			\draw (1.9,1.4) -- (3,1.6);
			\draw (2.8,1.6) -- (3.9,1.4);
			\draw (3.9,3.1) -- (3.7,1.9);
			\draw (3.7,2.1) -- (3.9,0.9);
			\draw (3.9,1.1) -- (3.7,-0.1);
		\end{tikzpicture}
	}
	\subcaptionbox{$2\rA_1+2\rA_3$, $\cha\kk\neq 2$\\ Example \ref{ex:ht=2_twisted}\ref{item:2A1+2A3_construction} \label{fig:2A3+2A1-intro}}[.27\linewidth]{
		\begin{tikzpicture}[scale=0.9]
			\draw (0.2,3.1) -- (0,1.9);
			\draw[dashed] (0,2.1) -- (0.2,0.9);
			\draw (0.2,1.1) -- (0,-0.1);
			\draw (1.2,3.1) -- (1,1.9);
			\draw[dashed] (1,2.1) -- (1.2,0.9);
			\draw (1.2,1.1) -- (1,-0.1);
			\draw[dashed] (1.7,1.6) -- (2.8,1.4);
			%
			\draw (2.8,3.1) -- (2.6,1.9);
			\draw (2.6,2.1) -- (2.8,0.9);
			\draw (2.8,1.1) -- (2.6,-0.1);
			\draw (-0.05,1.55) to[out=0,in=180] (0.05,1.5) to[out=0,in=180] (1.1,2.9) -- (1.3,2.9) 
			to[out=0,in=170] (1.9,1.56) to[out=-10,in=-135] (2,1.61);
			\draw (-0.05,1.45) to[out=0,in=180] (0.05,1.5) to[out=0,in=180] (0.9,0.1) -- (1.1,0.1) 
			to[out=0,in=170] (1.9,1.56) to[out=-10,in=135] (2,1.51);
		\end{tikzpicture}
	}
	\subcaptionbox{$2\nu\rA_1+[2\nu-4]$, $\cha\kk=2$\\ Example \ref{ex:ht=2_twisted_cha=2} \label{fig:KM_surface-intro}}[.27\linewidth]{
		\begin{tikzpicture}[scale=0.9]
			\draw[very thick] (0,1.2) -- (3.1,1.2);
			\node at (2.1,1.4) {\small{$4-2\nu$}};
			\draw (0.2,2.5) -- (0,1.5);
			\draw[dashed] (0,1.7) -- (0.2,0.7);
			\draw (0.2,0.9) -- (0,-0.1);
			\draw (1.2,2.5) -- (1,1.5);
			\draw[dashed] (1,1.7) -- (1.2,0.7);
			\draw (1.2,0.9) -- (1,-0.1);
			\node at (2.1,2) {\Large{$\cdots$}};
			\draw (3,2.5) -- (2.8,1.5);
			\draw[dashed] (2.8,1.7) -- (3,0.7);
			\draw (3,0.9) -- (2.8,-0.1);
			\draw [decorate, decoration = {calligraphic brace}, thick] (0,2.7) -- (3.2,2.7);
			\node at (1.6,3) {\scriptsize{$\nu$ fibers}};			
		\end{tikzpicture}
	}
	\caption{Vertically primitive log surfaces $(Y,D_Y)$ in  Theorem \ref{thm:ht=1,2}, see Examples \ref{ex:ht=2}--\ref{ex:ht=2_twisted_cha=2}. \\ Solid lines correspond to components of $D_Y$: except for the horizontal ones in Figures \ref{fig:ht=1-intro} and \ref{fig:KM_surface-intro}, they are $(-2)$-curves. Dashed lines correspond to vertical $(-1)$-curves, see Notation \ref{not:figures}.
	}
	\label{fig:basic-intro}
\end{figure}
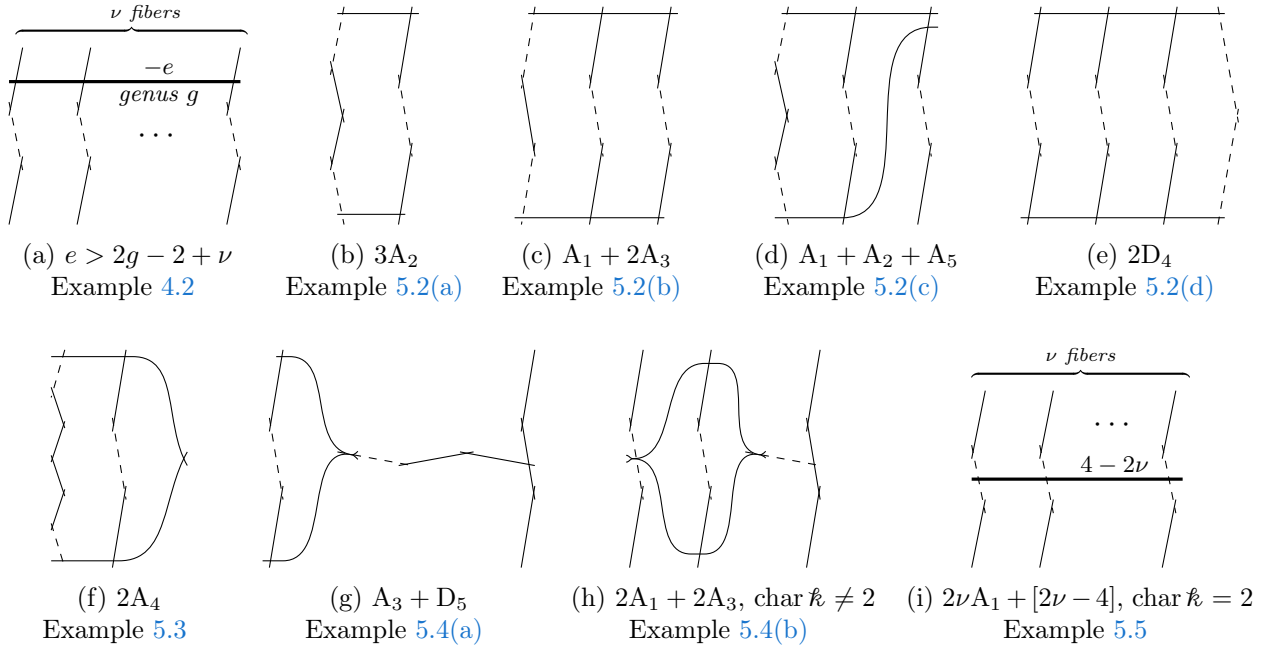
	 \item\label{part:uniqueness} Assume that all singularities of $\bar{X}$ are rational and log canonical. Then the singularity type of $\bar{X}$ determines $\bar{X}$ uniquely up to an isomorphism, unless $\bar{X}$ is one of the exceptions described in Proposition  \ref{prop:moduli}, see Table \ref{table:exceptions}. In the latter case we have either $\height(\bar{X})=1$, or $\height(\bar{X})=2$ and the surface $\bar{Y}$ from \ref{part:swaps} can be chosen as in Example \ref{ex:ht=2}\ref{item:3A2_construction},\ref{item:2D4_construction} or \ref{ex:ht=2_twisted_cha=2}, see Figure \ref{fig:3A2-intro}, \ref{fig:2D4-intro} or \ref{fig:KM_surface-intro}. 
	 \item\label{part:types} One of the following holds.
	\begin{parts}
		\item\label{item:ht=1} We have $\height(\bar{X})=1$. 
		If $\bar{X}$ is log canonical then its singularity type is listed in Lemma  \ref{lem:ht=1_types}, see Table \ref{table:ht=1}.
		\item\label{item:ht=2_width=2} We have $\height(\bar{X})=2$, and $X$ admits a $\P^1$-fibration such that the horizontal part of $D$ consists of two $1$-sections. The surface $\bar{Y}$ from \ref{part:swaps} can be chosen as in Example \ref{ex:ht=2} or \ref{ex:ht=2_meeting}, see Figures \ref{fig:3A2-intro}--\ref{fig:2A4-intro}. If $\bar{X}$ is log canonical then its singularity type is listed in Lemma \ref{lem:ht=2,untwisted}, see Table \ref{table:ht=2_char-any}. 
		\item\label{item:ht=2_width=1} We have $\height(\bar{X})=2$, and $X$ has no $\P^1$-fibration as in \ref{item:ht=2_width=2}. 
		Then one of the following holds. 
		\begin{parts}
		\item\label{item:ht=2_char-neq-2} The surface $\bar{Y}$ from \ref{part:swaps} can be chosen as in Example  \ref{ex:ht=2_twisted}, see Figures \ref{fig:D5+A3_twisted-intro}, \ref{fig:2A3+2A1-intro}. The surface $\bar{X}$ is log terminal, of singularity type listed in Lemma \ref{lem:ht=2_twisted-separable}, see Tables \ref{table:ht=2_char-any} and \ref{table:ht=2_char-neq-2}.
		\item\label{item:ht=2_char=2} We have $\cha\kk=2$. The surface $\bar{Y}$ from \ref{part:swaps} can be chosen as  in Example \ref{ex:ht=2_twisted_cha=2} for some $\nu\geq 3$, see Figure \ref{fig:KM_surface-intro}. If $\bar{X}$ is log canonical then its singularity type is listed in Lemma \ref{lem:ht=2_twisted-inseparable}, see Table \ref{table:ht=2_char=2}.
		\end{parts}
	\end{parts} 
\end{parts}
\end{theoremA}
\begin{remark}[Cascades of log canonical del Pezzo surfaces]\label{rem:swap-lt}
	Assume that $(X,D)$ is a minimal log resolution of a log canonical del Pezzo surface of rank one, and $(X,D)\sqto(X_1,D_1)\sqto\dots\sqto (X_n,D_n)$ are elementary swaps. Then by Lemma \ref{lem:swap_lc}\ref{item:swap_lc_dP} each $(X_i,D_i)$ is a minimal log resolution of a log canonical del Pezzo surface, too. Thus Theorem \ref{thm:ht=1,2} arranges all log canonical del Pezzo surfaces of height at most two in series constructed inductively by blowing up within the degenerate fibers of simple $\P^1$-fibrations shown in Figure \ref{fig:basic-intro}. 	Similar natural cascades were constructed e.g.\ by Hwang \cite{Hwang_cascades} in the toric case, see Section \ref{sec:Hwang}.
\end{remark}
Theorem  \ref{thm:ht=1,2}\ref{part:uniqueness} gives the following description of moduli for log terminal del Pezzo surfaces of rank one and height at most $2$ in non-exceptional field characteristics. For a complete result see Proposition \ref{prop:moduli} and Table \ref{table:exceptions}. 

\begin{notation}\label{not:PS}
	For a singularity type $\cS$ 
	we denote by $\cP(\cS)$ the set of isomorphism classes of  del Pezzo surfaces of rank one and type $\cS$. We denote by $\Pht(\cS)$ the subset of $\cP(\cS)$ consisting of surfaces of height at most $2$.
\end{notation}

\begin{corollary}[Moduli, see Proposition \ref{prop:moduli} and Table \ref{table:exceptions}]\label{cor:moduli}
	Assume $\cha\kk\neq 2,3,5$. Let $\cS$ be a log terminal singularity type. Then any two surfaces in $\Pht(\cS)$ are equivalent under an equisingular deformation. Moreover, if $\#\Pht(\cS)\geq 2$ 
	 then one of the following holds. 
\begin{enumerate}
	\item\label{item:cor-2D4} We have $\#\Pht(\cS)=\infty$, and there is a one-parameter family such that $\bar{X}\in \Pht(\cS)$ if and only if $\bar{X}$ is isomorphic to one of its fibers. Moreover, every surface $\bar{X}\in \Pht(\cS)$ swaps vertically to a surface of type $2\rD_4$, see Figure \ref{fig:2D4-intro} and Example \ref{ex:ht=2}\ref{item:2D4_construction}. 
	\item\label{item:cor-2} We have $\#\Pht(\cS)=2$, and $\cS$ is one of the types shown in Figure \ref{fig:cor-2}, see Table \ref{table:exceptions}. 
\end{enumerate}
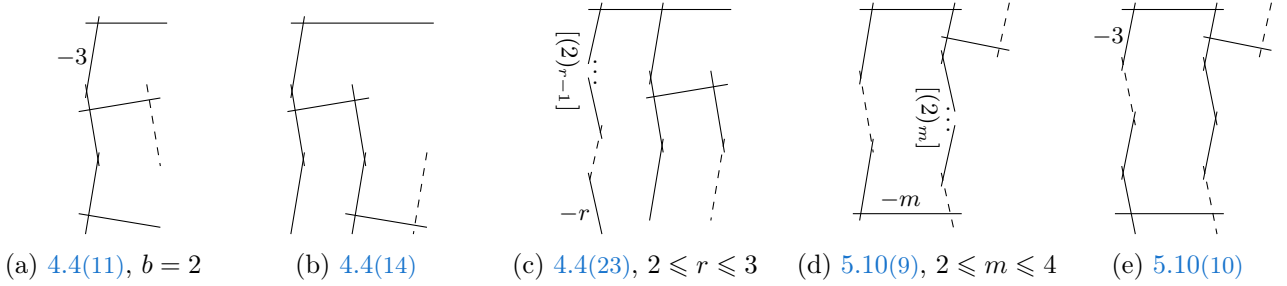
\begin{figure}[htbp]\vspace{-0.5em}
	\subcaptionbox{\ref{lem:ht=1_types}\ref{item:T2*=[2,2]_b>2}, $b=2$ \label{fig:T2*=[2,2]_b>2}}[.18\linewidth]{
	\begin{tikzpicture}[scale=0.9]
		\draw (0,3) -- (1.2,3);
		\draw (0.2,3.1) -- (0,1.9);
		\node at (-0.2,2.5) {\small{$-3$}};
		\draw (0,2.1) -- (0.2,0.9);
		\draw (0.2,1.1) -- (0,-0.1);
		\draw (-0.1,0.2) -- (1.1,0);
		\draw (-0.1,1.7) -- (1.1,1.9);
		\draw[dashed] (0.9,2.1) -- (1.1,0.9);
	\end{tikzpicture}
}	
	\subcaptionbox{\ref{lem:ht=1_types}\ref{item:TH=[2,2]_b>2} 
	\label{fig:E8-intro}
	}
	[.2\linewidth]
	{
		\begin{tikzpicture}[scale=0.9]
			\draw (0,3) -- (2.1,3);
			\draw (0.2,3.1) -- (0,1.9);
			\draw (0,2.1) -- (0.2,0.9);
			\draw (0.2,1.1) -- (0,-0.1);
			\draw (-0.05,1.7) -- (1.15,1.9);
			\draw (0.9,2.1) -- (1.1,0.9);
			\draw (1.1,1.1) -- (0.9,-0.1);
			\draw (0.8,0.2) -- (2,0);
			\draw[dashed] (2,1.1) -- (1.8,-0.1);
		\end{tikzpicture}
	}
	\subcaptionbox{\ref{lem:ht=1_types}\ref{item:T2=[2]_b>2}, $2\leq r \leq 3$ \label{fig:A1+E7-intro} 
}
	[.22\linewidth]
	{
		\begin{tikzpicture}[scale=0.9]
			\draw (0,3) -- (2.1,3);
			\draw (0.2,3.1) -- (0,2.2);
			\node[rotate=-90] at (-0.3,2.1) {\small{$[(2)_{r-1}]$}};
			\node at (0.1,2.2) {\small{$\vdots$}};
			\draw (0,2) -- (0.2,1.1);
			\draw[dashed] (0.2,1.3) -- (0,0.4);
			\draw (0,0.6) -- (0.2,-0.3);
			\node at (-0.2,0) {\small{$-r$}};
			\draw (1.1,3.1) -- (0.9,1.9);
			\draw (0.9,2.1) -- (1.1,0.9);
			\draw (1.1,1.1) -- (0.9,-0.1);
			\draw (0.85,1.7) -- (2.05,1.9);
			\draw (1.8,2.1) -- (2,0.9);
			\draw[dashed] (2,1.1) -- (1.8,-0.1);
		\end{tikzpicture}
	}
	\subcaptionbox{\ref{lem:ht=2,untwisted}\ref{item:[T_1,n]=[2,2]_r=2}, $2\leq m\leq 4$ \label{fig:A2+E6-intro}}[.22\linewidth]{
		\begin{tikzpicture}[scale=0.9]
			\draw (0,3) -- (1.5,3);
			\draw (0.2,3.1) -- (0,1.9);
			\draw[dashed] (0,2.1) -- (0.2,0.9);
			\draw (0.2,1.1) -- (0,-0.1);
			\draw (1.2,2.6) -- (2.2,2.4); 
			\draw[dashed] (2.2,3.1) -- (2,2.2);
			\draw (1.4,3.1) -- (1.2,2.2);
			\draw (1.2,2.4) -- (1.4,1.5);
			\node[rotate=-90] at (1,1.4) {\small{$[(2)_{m}]$}};
			\node at (1.3,1.5) {\small{$\vdots$}};
			\draw (1.4,1.3) -- (1.2,0.4);
			\draw[dashed] (1.2,0.6) -- (1.4,-0.3);
			\draw (-0.1,0) -- (1.5,0);
			\node at (0.6,0.2) {\small{$-m$}};
		\end{tikzpicture}	
	}	
	\subcaptionbox{\ref{lem:ht=2,untwisted}\ref{item:[T_1,n]=[3,2]_r=2}\label{fig:[T_1,n]=[3,2]_r=2}}[.15\linewidth]{
		\begin{tikzpicture}[scale=0.9]
			\draw (0,3) -- (1.5,3);
			\draw (0.2,3.1) -- (0,2.1);
			\node at (-0.2,2.6) {\small{$-3$}};
			\draw[dashed] (0,2.3) -- (0.2,1.3);
			\draw (0.2,1.5) -- (0,0.5);
			\draw (0,0.7) -- (0.2,-0.3);
			\draw (1.2,2.6) -- (2.2,2.4); 
			\draw[dashed] (2.2,3.1) -- (2,2.2);
			\draw (1.4,3.1) -- (1.2,2.1);
			\draw (1.2,2.3) -- (1.4,1.3);
			\draw (1.4,1.5) -- (1.2,0.5);
			\draw[dashed] (1.2,0.7) -- (1.4,-0.3);
			\draw (-0.1,0) -- (1.5,0);
		\end{tikzpicture}
	}
	\caption{Corollary \ref{cor:moduli}\ref{item:cor-2}: log terminal singularity types realized regardless of the characteristic of the base field by exactly two non-isomorphic del Pezzo surfaces, cf.\ Table \ref{table:exceptions}.
}
	\label{fig:cor-2}\vspace{-0.5em}
\end{figure}
\end{corollary}

\begin{remark}[The number of singularities]\label{rem:number-of-sing}
Let $\bar{X}$ be a log terminal del Pezzo surface of rank one. Theorem~\ref{thm:ht=1,2} gives a complete list of possible singularity types of $\bar{X}$ under the assumption $\height(\bar{X})\leq 2$. If $\cha\kk\neq 2$ it implies that such $\bar{X}$ has at most $4$ singularities. For $\cha\kk\neq 2,3,5$ this so-called \emph{Bogomolov bound} is known for any log terminal del Pezzo surface of rank one, regardless of the height \cite[Corollary 1.2]{Lacini}. If $\kk=\C$, it can be proved using Bogomolov-type inequalities  \cite[Corollary 1.8.1]{Keel-McKernan_rational_curves}, \cite{Belousov_Bogomolov-bound}. 

In turn, if $\cha\kk=2$ then in \cite[p.\ 68]{Keel-McKernan_rational_curves} the authors construct del Pezzo surfaces of rank one, height $2$ and type  $[2\nu-4]+2\nu \cdot \rA_1$ for any $\nu\geq 3$, see Figure \ref{fig:KM_surface-intro} and Example \ref{ex:ht=2_twisted_cha=2}. Theorem \ref{thm:ht=1,2} implies that in fact every del Pezzo surface of height at most two and at least $5$ singular points swaps vertically to one of those surfaces. 

For del Pezzo surfaces $\bar{X}$ of bigger height the inequality $\#\Sing\bar{X}\leq 4$ fails for $\cha\kk\in \{2,3,5\}$. Indeed, in Theorem \ref{thm:GK} below we will see that if $\cha\kk=3$ or $5$ there exist del Pezzo surfaces of rank one and type $4\rA_2+[5]+[3]+[2]$ or $2\rA_4+[5]+[3]+[2]$, respectively (the latter was also constructed in \cite[6.23]{Lacini}). In our forthcoming articles we will see that these surfaces have the maximal possible number of singular points if $\cha\kk=3$ or $5$, respectively; while if $\cha\kk=2$, every del Pezzo surface of rank $1$ with at least $12$ singularities swaps vertically to a surface from \cite[p.\ 68]{Keel-McKernan_rational_curves}.
\end{remark}

There are lots of del Pezzo surfaces of rank $1$ which are not log terminal. Even if $\kk=\C$ they can have arbitrarily many singularities with complicated minimal resolutions, see Section \ref{sec:non-lt-examples}. Nevertheless, their height is low, and their structure is easy to understand in terms of explicit primitive models and swaps. Indeed, in Section \ref{sec:non-lt-proof} we use the known description of \emph{open} del Pezzo surfaces \cite{Miy_Tsu-opendP,PaPe_MT} to prove the following result. For a more detailed version see Proposition \ref{prop:non-lt-more}.

\begin{propositionA}[Non--log terminal del Pezzo surfaces]\label{prop:non-log-terminal}
	Let $\bar{X}$ be a normal del Pezzo surface of rank one. Assume that $\bar{X}$ is not log terminal. Then the non-lt singularity of $\bar{X}$ is unique, and one of the following holds.
	\begin{enumerate}
		\item\label{item:non-lt_ht} We have $\height(\bar{X})\leq 2$. If $\height(\bar{X})=2$ then $\bar{X}$ swaps vertically to one of the canonical surfaces  from Examples \ref{ex:ht=2} or \ref{ex:ht=2_meeting}, see Figures \ref{fig:3A2-intro}--\ref{fig:2A4-intro}, or, if $\cha\kk=2$, to a surface from Example \ref{ex:ht=2_twisted_cha=2}, see Figure \ref{fig:KM_surface-intro}. 
		\item\label{item:non-lt_descendant} We have $\cha\kk\in \{2,3,5\}$. The surface $\bar{X}$ has a descendant $(\bar{Y},\bar{T})$ with elliptic boundary (see Definition \ref{def:GK}) such that $\bar{T}\in |-K_{\bar{Y}}|$ is cuspidal, and $\bar{Y}$ is a canonical del Pezzo surface of rank one, not of type $\rE_6$, $\rE_7$ or $\rE_8$, such that the minimal resolution of $\bar{Y}$ has at least $6$ exceptional components, cf.\ Table \ref{table:canonical}. In this case we have  $\height(\bar{X})\leq 6$,  see Theorem \ref{thm:GK}.
		\item\label{item:non-lt-swap} We have $\cha\kk=2$, $\height(\bar{X})=3$, and $\bar{X}$ swaps vertically to the surface $\bar{Y}$ from Example \ref{ex:ht=3}, see Figure \ref{fig:ht=3-intro}.
	\end{enumerate}	
	Moreover, if $\bar{X}$ is a non-rational surface then $\height(\bar{X})=1$, the resolution of the non--log terminal singularity of $\bar{X}$ is a tree with exactly one non-rational curve, and the resolutions of the log terminal ones  are admissible chains.
\end{propositionA}
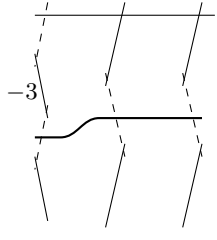
\begin{figure}[htbp]
	\begin{tikzpicture}[scale=0.85]
		\path[use as bounding box] (-0.2,3) rectangle (2.6,0);
		\draw (0,3) -- (2.8,3);
		\draw[dashed] (0.2,3.2) -- (0,2.2);
		\draw (0,2.4) -- (0.2,1.4);
		\node at (-0.2,1.8) {\small{$-3$}};
		\draw[dashed] (0.2,1.6) -- (0,0.6); 
		\draw (0,0.8) -- (0.2,-0.2);
		\draw (1.4,3.2) -- (1.1,1.9);
		\draw[dashed] (1.1,2.1) -- (1.4,0.8);
		\draw (1.4,1) -- (1.1,-0.3); 
		\draw (2.6,3.2) -- (2.3,1.9);
		\draw[dashed] (2.3,2.1) -- (2.6,0.8);
		\draw (2.6,1) -- (2.3,-0.3); 
		\draw[thick] (0,1.1) -- (0.4,1.1) to[out=0,in=180] (1,1.4) -- (2.6,1.4);
	\end{tikzpicture}
	\caption{Del Pezzo surface $\bar{Y}$ from Proposition \ref{prop:non-log-terminal}\ref{item:non-lt-swap}, $\cha\kk=2$, see  Example  \ref{ex:ht=3}.}
\vspace{-0.5em}
	\label{fig:ht=3-intro}
\end{figure}
We note that the uniqueness of a non-lt singularity for log canonical del Pezzo surfaces was proved by Kojima in \cite{Kojima_non-lt}, our more detailed result follows a similar approach.

\begin{remark}[Non-rational del Pezzo surfaces]\label{rem:non-rational}
	\begin{parlist}\ 
	\item The last statement of Proposition \ref{prop:non-log-terminal}, concerning non-rational surfaces, is known, see \cite{Badescu_non-rational,Schroer_non-rational}. 
	It leads to combinatorial descriptions of such surfaces, given in \cite[Theorem 4.9]{Fujisawa}, \cite[\sec 5]{Cheltsov_non-rational} for $\kk=\C$ and in \cite[\sec 5]{Schroer_non-rational} for an arbitrary perfect field $\kk$. 
	
	\item We note that a del Pezzo surface of rank one has a non-rational singularity if and only if $\bar{X}$ itself is non-rational \cite[Theorem 2.2]{Schroer_non-rational}, cf.\ \cite[Lemma 2.7]{PaPe_MT}. If furthermore $\bar{X}$ is log canonical then, as we will see in Lemma \ref{lem:ht=1_types}, $\bar{X}$ is an elliptic cone as in \cite[\sec 14.38]{Badescu}, see Example \ref{ex:ht=1}.
	
	\item In fact, Proposition \ref{prop:non-log-terminal}, as well as more detailed descriptions cited above, hold also for \emph{numerically del Pezzo surfaces}, i.e.\ without the assumption that $K_{\bar{X}}$ is $\Q$-Cartier, see \cite{Cheltsov_non-rational}. In this article, however, we focus mostly on rational or log canonical singularities, for which $K_{\bar{X}}$ is automatically $\Q$-Cartier \cite[17.1]{Lipman}. 
	\end{parlist}
\end{remark}

\subsection{The question of uniqueness of del Pezzo surfaces of a  given singularity type}

Corollary \ref{cor:moduli} implies that usually two log terminal del Pezzo surfaces of rank $1$ and height at most $2$ which share the  singularity type are usually isomorphic. Proposition \ref{prop:moduli} below gives a more precise description of the exceptions, showing in particular that they are equivalent under an equisingular deformation. 
\begin{definition}[Representing families]\label{def:moduli} \ 
	\begin{parlist}
		\item \label{item:def-family-smooth} A \emph{smooth family} $f\colon (\cX,\cD)\to B$ consists of a smooth surjective morphism $f\colon \cX\to B$ between connected smooth varieties, and a reduced divisor $\cD$ on $\cX$ such that for every component $\cD_0$ of $\cD$, the restriction $f|_{\cD_0}\colon \cD_0\to B$ is a smooth surjective morphism with irreducible fibers, cf.\ \cite[Definition 2.13]{CTW}. 
		\item \label{item:def-family-fiber} A \emph{fiber} of $f$ over $b\in B$ is a pair $(X_b,D_b)$, where $X_{b}$ is the scheme-theoretic fiber $f^{-1}(b)$, and $D_{b}$ is a reduced divisor on $X_{b}$ given by the restriction of $\cD$.
		\item \label{item:def-family-representing} 
		Let $\cC$ be a set consisting of isomorphism classes of some log smooth log surfaces. We say that a smooth family  $f\colon (\cX,\cD)\to B$ represents $\cC$ if $\cC$ equals the set of isomorphism classes of fibers of $f$.
		\item \label{item:def-family-faithful} Let $f\colon (\cX,\cD)\to B$ be a smooth family. 
		We say that $f$ is \emph{almost faithful} if it is equivariant with respect to the action of some finite group $G$ such that two fibers $(X_b,D_b)$ and $(X_{b'},D_{b'})$ of $f$ are isomorphic if and only if $b$ and $b'$ 
		lie in the same orbit. In this case, we say that $f$ \emph{has dimension} $\dim B$. We call the image of $G$ in $\Aut(B)$ the \emph{symmetry group} of $f$.
		\item \label{item:def-family-stratified} Let $f\colon (\cX,\cD)\to B$ be a smooth family. We say that $f$ is \emph{stratified}, by a stratification $B=\bigsqcup_{i=0}^{d}B_i$, if
		\begin{parlist}
			\item $\dim B=d$, and each closure $\bar{B}_{i}=\bigsqcup_{j\geq i} B_i$ is a smooth connected subvariety of codimension $j$ in $B$,
			\item  two fibers $(X_{b},D_{b})$ and $(X_{b'},D_{b'})$ of $f$ are isomorphic if and only if $b$ and $b'$ lie in the same stratum.
		\end{parlist}
		\item \label{item:def-family-bar} We say that a set $\bar{\cC}$ of isomorphism classes of normal projective surfaces or log surfaces is  \emph{represented by an (almost faithful, or stratified) family} (of dimension $d$), if the same holds for the set of isomorphism classes of minimal log resolutions of surfaces in $\bar{\cC}$, cf.\ Lemma \ref{lem:blowdown}.
	\end{parlist}
\end{definition}

\begin{propositionA}[Exceptions to uniqueness]\label{prop:moduli}
	Let $\cS$ be a rational log canonical singularity type. Let ${\cP=\Pht(\cS)}$. 
		\begin{enumerate} 
			\item\label{item:moduli-finite} If $\cP$ is finite then it has at most three elements and it is represented by a stratified family.
			\item\label{item:moduli-infinite} If $\cP$ is infinite then it is represented by an almost faithful family.
			\item\label{item:moduli-table} If $\#\cP\geq 2$ then the type $\cS$, the number $\#\cP$ in case $\#\cP<\infty$, and the dimension of some almost faithful family representing $\cP$ in case $\#\cP=\infty$, are listed in Table~\ref{table:exceptions}.
		\end{enumerate}
\end{propositionA}
If $\#\cP=\infty$ then the base and symmetry group of some almost faithful family representing $\cP$ is described in Table \ref{table:bases-primitive} (case $\height=1$) and Remark \ref{rem:ht=2_bases} (case $\height=2$). For instance, in Corollary \ref{cor:moduli}\ref{item:cor-2D4} the base $B\cong \P^1\setminus \{0,1,\infty\}$ parametrizes the fourth fiber in Figure \ref{fig:2D4-intro}, and the symmetry group is a subgroup of $\Aut(B)\cong S_3$.
\smallskip

Note that in Proposition \ref{prop:moduli} we do not claim that the representing family is unique. Proposition \ref{prop:moduli-hi} below shows that (up to some minor exceptions in positive characteristic, listed in Table \ref{table:exceptions-to-moduli}) one can choose such a family so that it satisfies certain natural properties, related to infinitesimal deformations of minimal log resolutions of surfaces in $\Pht(\cS)$. In case $\#\Pht(\cS)=\infty$ it allows to intrinsically define the \emph{moduli dimension} of $\Pht(\cS)$, see Definition \ref{def:moduli-hi}\ref{item:def-moduli}. To state this result, we need to introduce some further terminology.

We denote by $\lts{X}{D}$ the logarithmic tangent sheaf $\lts{X}{D}$, i.e.\  the sheaf of vector fields on $X$ which are tangent to $D$, see \cite{Kawamata_deformations,FZ-deformations}. We write $h^{i}(\lts{X}{D})=\dim H^{i}(X,\lts{X}{D})$. The number $h^{1}(\lts{X}{D})$ is the dimension of the tangent space to the base of a semiuniversal formal deformation of $(X,D)$, see \cite[Definition 2.2.6]{Sernesi_deformations} for a definition (and \cite{Kawamata_deformations} for analytic version). If $h^{2}(\lts{X}{D})=0$ then this base is smooth, see  \cite[Proposition 3.4.17]{Sernesi_deformations} or \cite[8.6]{Kato_deformation-theory}.

\begin{definition}[$h^{1}$-stratified and almost universal families]
	\label{def:moduli-hi}
	Let $f\colon (\cX,\cD)\to B$ be a smooth family.
	\begin{parlist}
		\item\label{item:def-h1-stratified} We say that $f$ is \emph{$h^{1}$-stratified} if it is stratified by $B=\bigsqcup_{i=0}^{d}B_i$, see Definition \ref{def:moduli}\ref{item:def-family-stratified}, and
		\begin{parlist}
			\item for every $i\in \{0,\dots, d\}$ and $b\in B_i$ we have $h^{1}(\lts{X_b}{D_b})=i$,
			\item the formal germ of $f$ at the deepest stratum $B_{d}=\{\bar{b}\}$ is a semiuniversal deformation of $(X_{\bar{b}},D_{\bar{b}})$. 
		\end{parlist} 	
		\item\label{item:def-universal} We say that $f$ is \emph{almost universal} if the following hold:
	\begin{parlist}
		\item\label{item:def-universal-af} $f$ is almost faithful, see Definition \ref{def:moduli}\ref{item:def-family-faithful}, 
		\item\label{item:def-universal-su} for every $b\in B$ the formal germ of $f$ at $b$ is a semiuniversal deformation of $(X_b,D_b)$.
\end{parlist}
	\item\label{item:def-moduli} We say that a set $\cC$ of isomorphism classes of log smooth surfaces \emph{has moduli dimension $d$} if it is represented by an almost universal family of dimension $d$. In this case, $d=h^{1}(\lts{X}{D})$ for every $(X,D)\in \cC$. We say that a set of isomorphism classes of normal projective surfaces, or log surfaces, \emph{has moduli dimension $d$} if the same holds for the set of isomorphism classes of their minimal log resolutions, cf.\ Definition \ref{def:moduli}\ref{item:def-family-bar}.
	\end{parlist}
\end{definition}

\begin{propositionA}[Moduli and rigidity]\label{prop:moduli-hi}
	Let $\cS$ be a rational log canonical singularity type, ${\cP=\Pht(\cS)}$. 
	\begin{enumerate} 
		\item\label{item:moduli-hi-infinite} If $\cP$ is infinite then $\cP$ has moduli dimension as in Table~\ref{table:exceptions}, unless $\cha\kk=2$ and $\cS$ is as in Table \ref{table:exceptions-to-moduli}. 
		\item\label{item:moduli-hi-finite} If $\cP$ is finite and $\#\cP\geq 2$ then $\cP$ is represented by an $h^{1}$-stratified family, unless $\cha\kk=2$ and $\cS=\rE_8$.
		\item\label{item:moduli-hi-1} Let $(X,D)$ be the minimal log resolution of a surface in $\cP$, and let $h^{i}\de h^{i}(\lts{X}{D})$. Then
		\begin{enumerate}
			\item\label{item:moduli-h1} If $\#\cP=1$ then $h^{1}=0$, unless $\cha\kk>0$, $\cS$ is as in Table \ref{table:exceptions-to-moduli}, and $h^1=1$. 
			\item\label{item:moduli-h2} We have $h^{2}=0$, unless $\cha\kk=2$, $\cS$ is as in Theorem \ref{thm:ht=1,2}\ref{item:ht=2_char=2}, and $h^2=\nu-2$.
		\end{enumerate}
	\end{enumerate}
\end{propositionA}
For most of the exceptions in  Proposition \ref{prop:moduli-hi}\ref{item:moduli-hi-infinite} the number $h^{1}(\lts{X}{D})$ varies among surfaces in $\cP$, see Table \ref{table:exceptions-to-moduli}, so $\cP$ is not represented by an almost universal family. In the exceptional case of Proposition \ref{prop:moduli-hi}\ref{item:moduli-hi-finite}, the set $\cP(\rE_8)$ has three elements, with $h^1=0,1,3$, so it is not represented by an $h^1$-stratified family. Note that by the Riemann--Roch theorem the Euler characteristic $\chi\de h^0-h^1+h^2$ depends only on the singularity type~$\cS$. It is listed in Tables \ref{table:ht=1}--\ref{table:ht=2_char=2}: together with the above formulas for $h^1$ and $h^2$ it allows to compute $h^0$.

\subsection{Del Pezzo surfaces admitting descendants with elliptic boundary}\label{sec:intro_GK}

We now introduce another broad class of del Pezzo surfaces with relatively simple geometry. Just like in case $\height\leq 2$ discussed above, these surfaces are constructed from simple (in fact canonical) del Pezzo surfaces of rank one by a sequence of blowups. The difference is that now, instead of resolving a particular $\P^{1}$-pencil on $\bar{Y}$, we resolve the singularity of an elliptic boundary $\bar{T}\subseteq \bar{Y}\reg$; see Figure \ref{fig:GK} for an example. We say that a boundary $\bar{T}$ of a log surface $(\bar{Y},\bar{T})$ is  \emph{elliptic} if $\bar{T}$ is an irreducible curve of arithmetic genus one (usually rational singular) contained in the smooth locus of $\bar{Y}$. 

\begin{definition}[Descendant]\label{def:GK}
	Let $\bar{X}$ be a normal projective surface. Let $(X,D)$ be its minimal log resolution. A log surface $(\bar{Y},\bar{T})$ is a \emph{descendant} of $\bar{X}$ if there is a birational morphism $\phi\colon X\to \bar{Y}$ such that $\phi_{*}D=\bar{T}$.
\end{definition}

Classification of del Pezzo surfaces of rank one admitting descendants with elliptic boundary is summarized in Theorem \ref{thm:GK} below. It is based on the following elementary lemma, proved in Section \ref{sec:GK}, which shows that the geometry of their descendants is highly restricted.

\begin{lemma}[Descendants with elliptic boundary for surfaces of rank one]\label{lem:GK_intro}
	Let $\bar{X}$ be a normal projective surface of rank one which admits a descendant $(\bar{Y},\bar{T})$ with elliptic boundary. If $\bar{T}$ is smooth then $\bar{X}$ is del Pezzo if and only if it is an elliptic cone. 
	Assume that $\bar{T}$ is singular. Let $\pi\colon (X,D)\to (\bar{X},0)$ be the minimal log resolution, let $\phi\colon (X,D)\to (\bar{Y},\bar{T})$ be as in Definition \ref{def:GK}, and let $T=\phi^{-1}_{*}\bar{T}$. Then the  following hold.
	\begin{enumerate}
		\item\label{item:GK-intro-Y} $\bar{Y}$ is a canonical del Pezzo surface of rank one and $\bar{T}\in |-K_{\bar{Y}}|$. 
		\item\label{item:GK-intro-L} $\Exc\phi=D-T+L$ for some $(-1)$-curve $L\not\subseteq D$; and $\phi(L)=\Sing\bar{T}$ is the only base point of $\phi^{-1}$ in $\bar{Y}\reg$.
		\item\label{item:GK-intro-ht} If $\bar{Y}\not\cong \P^2$ then $\height(\bar{X})\leq \height(\bar{Y})+2\leq 6$ (and $\leq 4$ if $\cha\kk\not\in \{2,3\}$). If $\bar{Y}\cong \P^2$ then $\height(\bar{X})\leq 2$.
		\item\label{item:GK-intro-cf} The surface $\bar{X}$ is del Pezzo if and only if $\cf(T)<1$, where $\cf(T)$ is the coefficient of $T$ in $\pi^{*}K_{\bar{X}}-K_{X}$. 
	\end{enumerate}
	In particular, if $\bar{X}$ is log terminal then $\bar{X}$ is a del Pezzo surface and $\bar{T}$ is singular.
\end{lemma}
\begin{figure}[htbp]
	\centering
	\begin{tikzpicture}[scale=0.9]
		\begin{scope}[shift={(-6.8,0)}]
	\draw[dashed] (-0.03,1) -- (0.2,2.6) to[out=80,in=180] (1,2.8) -- (1.6,2.8) to[out=0,in=60] (1.1,1.5) -- (0.9,1.1);
	\filldraw (1.1,2.8) circle (0.08);
	\node[below right] at (1.1,2.8) {\small{$\rA_2$}}; 
	\filldraw (0.05,1.5) circle (0.08);
	\node[right] at (0.05,1.5) {\small{$\rA_2$}};
	\filldraw (1.1,0.2)  circle (0.08);
	\node[above right] at (1.1,0.1) {\small{$\rA_2$}};
	\draw[dashed] (-0.03,2) -- (0.2,0.4) to[out=-80,in=180] (1,0.2) -- (1.6,0.2) to[out=0,in=-60] (1.1,1.5) -- (0.9,1.9);
	\draw[dashed] (1.1,0) -- (1.1,3);
	\filldraw (1.02,1.42) rectangle (1.18,1.58);
	\node[above right] at (2.2,0.1) {\small{$[3,2,3,2]$}};
	\draw[->, gray] (2.5,0.7) to[out=135,in=-45] (1.25,1.35);
	\draw[thick, densely dashed] (0.8,1.6) to[out=-30,in=180] (1.1,1.5) to[out=0,in=90] (2.7,1.5) to[out=-90,in=0] (1.1,1.5) to[out=180,in=30] (0.8,1.4);
	\node at (2.4,1.95) {\small{$L$}};
%
%
	%
	\node at (3.6,1.75) {\scriptsize{min.\ res.}};
	\node at (3.6,1.25) {\scriptsize{of $\bar{X}$}};
	\draw[<-] (3.2,1.5) -- (4,1.5);
\end{scope}
		\begin{scope}
			\draw (0.1,3) -- (1.3,3);
			\draw[dashed] (0.2,3.1) -- (0,2.1);
			\draw (0,2.3) -- (0.2,1.4);
			\draw (0.2,1.6) -- (0,0.7);
			\draw[dashed] (0,0.9) -- (0.2,-0.1);
			\draw (1.2,3.1) -- (1,1.9);
			\draw[dashed] (1,2.1) -- (1.2,0.9);
			\draw (1.2,1.1) -- (1,-0.1);
			\draw (0.1,0.05) -- (1.1,0.05);
			\draw (-0.8,3.1) -- (-1,2.1);
			\draw (-1,2.3) -- (-0.8,1.4);
			\node at (-1.2,1.8) {\small{$-3$}};
			\draw[thick, densely dashed] (-0.8,1.6) -- (-1,0.7);
			\node at (-1.1,1.2) {\small{$L$}};
			\draw (-1,0.9) -- (-0.8,-0.1);
			\draw[thick] (1.3,1.6) to[out=-150,in=0] (1.1,1.5) to[out=180,in=0] (0.1,2.7)
			-- (-1,2.7) to[out=180,in=90] (-1.9,1.5) to[out=-90,in=180] (-1,0.4) -- (0.1,0.4) to[out=0,in=180] (1.1,1.5) to[out=0,in=150] (1.3,1.4);
			\node at (-2.2,1.5) {\small{$-3$}};
			\node at (-1.7,1.5) {\small{$T$}};
		\end{scope}
		\begin{scope}[shift={(4.7,0)}]
			\draw (0.1,3) -- (1.3,3);
			\draw[dashed] (0.2,3.1) -- (0,2.1);
			\draw (0,2.3) -- (0.2,1.4);
			\draw (0.2,1.6) -- (0,0.7);
			\draw[dashed] (0,0.9) -- (0.2,-0.1);
			\draw (1.2,3.1) -- (1,1.9);
			\draw[dashed] (1,2.1) -- (1.2,0.9);
			\draw (1.2,1.1) -- (1,-0.1);
			\draw (0.1,0.05) -- (1.1,0.05);
			\draw[thick] (1.3,1.6) to[out=-150,in=0] (1.1,1.5) to[out=180,in=0] (0.1,2.7) to[out=180,in=0] (-1.1,1.3) to[out=180,in=-90] (-1.3,1.5) to[out=90,in=180] (-1.1,1.7) to[out=0,in=180] (0.1,0.4) to[out=0,in=180] (1.1,1.5) to[out=0,in=150] (1.3,1.4);
			\node at (-1,1.9) {\small{$T$}};
			\node at (-1,1.1) {\small{$3$}};
			\draw[<-] (-2,1.5) -- (-2.8,1.5);
		\end{scope}
		\begin{scope}[shift={(9.5,0)}]
			\draw[dashed] (0,1.3) -- (0.2,2.6) to[out=80,in=180] (1,2.8) -- (1.4,2.8);
			\filldraw (1.1,2.8) circle (0.08);
			\node[below right] at (1.1,2.8) {\small{$\rA_2$}}; 
			\filldraw (0.05,1.5) circle (0.08);
			\node[right] at (0.05,1.5) {\small{$\rA_2$}};
			\filldraw (1.1,0.2)  circle (0.08);
			\node[above right] at (1.1,0.1) {\small{$\rA_2$}};
			\draw[dashed] (0,1.7) -- (0.2,0.4) to[out=-80,in=180] (1,0.2) -- (1.4,0.2);
			\draw[dashed] (1.1,0) -- (1.1,3);
			\draw[thick] (1.3,1.6) to[out=-150,in=0] (1.1,1.5) to[out=180,in=0] (0.1,2.7) to[out=180,in=0] (-1.1,1.3) to[out=180,in=-90] (-1.3,1.5) to[out=90,in=180] (-1.1,1.7) to[out=0,in=180] (0.1,0.4) to[out=0,in=180] (1.1,1.5) to[out=0,in=150] (1.3,1.4);
			\node at (-1,1.9) {\small{$\bar{T}$}};
			\node at (-1,1.1) {\small{$3$}};
			%
			\node at (-2.4,1.75) {\scriptsize{min.\ res.}};
			\node at (-2.4,1.25) {\scriptsize{of $\bar{Y}$}};
			\draw[<-] (-2,1.5) -- (-2.8,1.5);
		\end{scope}
	\end{tikzpicture}
	\caption{Del Pezzo surface $\bar{X}$ of rank one and type $3\rA_2+[3,2,3,2]$ with a descendant $(\bar{Y},\bar{T})$ of type $3\rA_2$, cf.\ Figure \ref{fig:3A2-intro}. The elliptic boundary $\bar{T}$ is nodal.}	\vspace{-0.5em}
\label{fig:GK}
\end{figure}
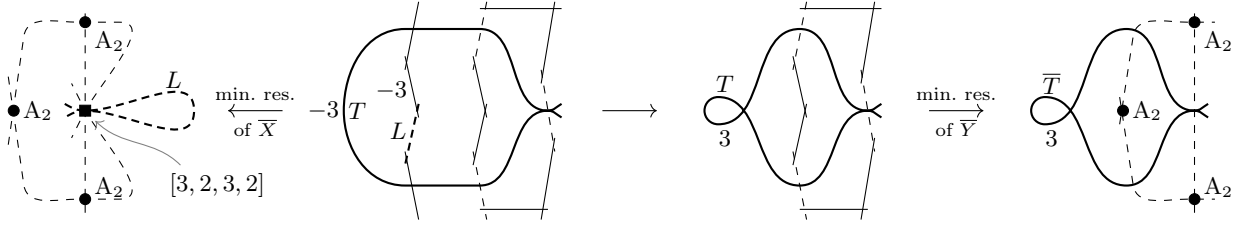

Log surfaces $(\bar{Y},\bar{T})$ as in Lemma \ref{lem:GK_intro}\ref{item:GK-intro-Y} 
are classified in \cite[Proposition 1.5]{PaPe_MT} summarized in Table \ref{table:canonical} as follows. In the row corresponding to a canonical singularity type $\cS$, the entry in the last column contains $\rN$ (respectively, $\rC$ or $\rC_d$) if and only if some del Pezzo surface $\bar{Y}$ of rank 1 and type $\cS$ contains a nodal (respectively, cuspidal) curve $\bar{T}$ with $p_{a}(\bar{T})=1$ in its smooth locus. Moreover, the log surface $(\bar{Y},\bar{T})$ is unique up to an isomorphism unless this entry is $\rC_{d}$. In the latter case, the set $\Pcusp(\cS)$ of isomorphism classes of log surfaces $(\bar{Y},\bar{T})$ where $\bar{Y}$ is a canonical del Pezzo surface of rank 1 and type $\cS$ and $\bar{T}\subseteq \bar{Y}\reg$ is a cuspidal member of $|-K_{\bar{Y}}|$ has moduli dimension $d$. If $\#\Pcusp(\cS)=1$ 
then $\Pcusp(\cS)$ has moduli dimension $0$, see Lemma \ref{lem:GK_hi-3}.

Lemma \ref{lem:GK_intro}\ref{item:GK-intro-L} shows that to reconstruct $\bar{X}$ from its descendant $(\bar{Y},\bar{T})$, we need to blow up over the singular point of $\bar{T}$ in such a way that the exceptional divisor has exactly one $(-1)$-curve $L$, and the total transform of $\bar{T}$, minus $L$, contracts to normal surface singularities. In the most important case, when the resulting surface $\bar{X}$ is log terminal, Lemma \ref{lem:GK_intro}\ref{item:GK-intro-cf} shows that it is automatically del Pezzo. We note that an analogous result holds for surfaces of height at most $2$, see Lemma \ref{lem:delPezzo_criterion}, but outside these two classes the ampleness of $-K_{\bar{X}}$ becomes a much stronger restriction.

The above strategy leads to a classification summarized in Theorem \ref{thm:GK} below. Given Theorem \ref{thm:ht=1,2} and Proposition \ref{prop:non-log-terminal} we may restrict our attention to log canonical del Pezzo surfaces of height at least $3$. For a normal surface $\bar{X}$ we denote by $\cS(\bar{X})$ its singularity type, and we write $\#\cS(\bar{X})$ for the number of exceptional components of its minimal resolution. For a singularity type $\cS$ we denote by $\Pdeb(\cS)$ the set of isomorphism classes of  of del Pezzo surfaces $\bar{X}$ of singularity type $\cS$ which have a descendant with elliptic boundary.

\begin{theoremA}[Classification of del Pezzo surfaces having descendants with elliptic boundary]\label{thm:GK}
	Let $\bar{X}$ be a log canonical del Pezzo surface of rank one, having a descendant $(\bar{Y},\bar{T})$ with elliptic boundary, see Definition \ref{def:GK}. Let $(X,D)$ be its minimal log resolution, and let $\cS$ be its singularity type. 
	If $\height(\bar{X})\geq 3$ then the following hold.
	\begin{enumerate}
		\item\label{item:GK-unique} The descendant $(\bar{Y},\bar{T})$ with elliptic boundary of the surface $\bar{X}$ is unique up to an isomorphism. Moreover, the singularity type $\cS(\bar{Y})$ of $\bar{Y}$, and the singularity type of $\bar{T}$, are uniquely determined by $\cS$.
		\item\label{item:GK-Y} 
		We have $\#\cS(\bar{Y})\geq 6$, see the classification in Table~\ref{table:canonical}. Moreover, if $\bar{T}$ is cuspidal then $\cS(\bar{Y})\not\in \{\rE_{6},\rE_7,\rE_8\}$.
		\item\label{item:GK-types} We have $\cS=\cS(\bar{Y})+\cT$, where $\cT$ is one of the types listed in  Lemma \ref{lem:cuspidal_resolution}.
		\item\label{item:GK_uniqueness-nodal} Assume that $\bar{T}$ is nodal. Then $\#\Pdeb(\cS)=1$, i.e.\ $\bar{X}$ is the unique del Pezzo surface of rank one and singularity type $\cS$ admitting a descendant with elliptic boundary.
		Moreover, $h^{i}(\lts{X}{D})=0$. 
		\item\label{item:GK_uniqueness-cuspidal} Assume that $\bar{T}$ is cuspidal, so $\cha\kk\in \{2,3,5\}$. 
		Then $\Pdeb(\cS)$ is represented by an almost faithful family of dimension $d+t$ (whose bases and symmetry groups are described in Remark \ref{rem:GK-bases}), where 
		\begin{enumerate}
			\item\label{item:GK-d} $d$ is the moduli dimension of $\Pcusp(\cS(\bar{Y}))$, listed in Table \ref{table:canonical} (i.e.\ $d$ for types with \enquote{$\rC_{d}$} and $0$ otherwise),
			\item\label{item:GK-t} $t=1$ if the type $\cT$ is as in Lemma \ref{lem:cuspidal_resolution}\ref{item:C_smooth}, and $t=0$ otherwise.
		\end{enumerate}
		Moreover, if $\#\cT\geq 3$ then $\Pdeb(\cS)$ has moduli dimension $d+t$ and we have $h^{i}(\lts{X}{D})=0,d+t,d+1$ for $i=0,1,2$, respectively. For the corresponding result in case $\#\cT\leq 2$ see Proposition \ref{prop:GK-small}.
		\item\label{item:GK-ht} We have $\height(\bar{X})=\height(\bar{Y})+2$, unless $\cS$ is as in Proposition \ref{prop:GK-ht-exceptions}, in which case $\height(\bar{X})=3$, $\height(\bar{Y})=2$. In particular, we have optimal bounds $\height(\bar{X})\leq 6$, and $\height(\bar{X})\leq 4$ if $\cha\kk\neq 2,3$.
	\end{enumerate}
	Conversely, if $\#\cS(\bar{Y})\geq 6$ then either $\height(\bar{X})\geq 3$ or $\bar{X}$ is one of the exceptions listed in Lemma \ref{lem:GK_exceptions}.
\end{theoremA}

\begin{remark}[Cascades, cf.\ Remark \ref{rem:swap-lt} and Proposition \ref{prop:GK_swaps}]\label{rem:GK-cascades}
		Let $\bar{X}$ be a log canonical del Pezzo surface of rank one, having a descendant with elliptic boundary. If $\bar{X}$ is non-primitive then by Proposition \ref{prop:GK_swaps}\ref{item:GK_swaps-lc} it admits an elementary swap  onto a log canonical del Pezzo surface $\bar{X}'$ of rank one, which either has a (possibly different) descendant with elliptic boundary or is canonical or, if $\cha\kk=2$, is one of the primitive surfaces from Example \ref{ex:ht=2_twisted_cha=2}, see Figure \ref{fig:KM_surface-intro}. In turn, if $\bar{X}$ is primitive then  it is either canonical, or its singularity type is listed in Lemma \ref{lem:GK_prim}; in particular if $\cha\kk\not\in \{2,3\}$ then it is of type $2\rA_4+[3]$ or $\rA_1+\rA_2+\rA_5+[3]$.
		
		In most cases, the above elementary swap $\bar{X}\sqto \bar{X}'$ is given by the contraction of the $(-1)$-curve $L$ from Lemma \ref{lem:GK_intro}\ref{item:GK-intro-L}. Then $\bar{X}'$ has the same descendant with elliptic boundary as $\bar{X}$. For example, if $\bar{X}$ is as in Figure \ref{fig:GK} then contracting $L$ we get an elementary swap onto $\bar{X}'$ of type $3\rA_2+[2,2,3]$. In the remaining cases, when this particular swap is not possible then by  Proposition \ref{prop:GK_swaps}\ref{item:GK_swaps-L} $\bar{X}$ is rather simple: either it is already canonical, or its minimal log resolution is birationally dominated by the minimal log resolution of its descendant.
\end{remark} 

\begin{remark}\label{rem:Pht=P}
	Let $\bar{X}_1,\bar{X}_2$ be del Pezzo surfaces of rank one and the same log canonical singularity type $\cS$. In forthcoming articles we will show that $\height(\bar{X}_1)=\height(\bar{X}_2)$, and if $\bar{X}_1$ has a descendant with elliptic boundary then so does $\bar{X}_2$. This result implies that for each singularity type $\cS$ as in Theorem \ref{thm:ht=1,2} and \ref{thm:GK}, the sets $\Pht(\cS)$ and $\Pdeb(\cS)$ described there are equal to the sets of isomorphism classes of \emph{all} del Pezzo surfaces of rank one and singularity type $\cS$, i.e.\ $\Pht(\cS)=\cP(\cS)$ and $\Pdeb(\cS)=\cP(\cS)$.
\end{remark}

\paragraph{Acknowledgment} We thank Joachim Jelisiejew for helpful discussions concerning deformation theory.

\clearpage
\section{Preliminaries}

We  now settle some notation and we recall some known results which will be useful in this and upcoming articles. For more general statements and further motivation see \cite{Fujita-noncomplete_surfaces} and references therein.

\subsection{Divisors on surfaces}\label{sec:log_surfaces}

Curves and surfaces are understood to be irreducible. A \emph{log surface} is a pair $(X,D)$, where $X$ is a normal projective surface and the \emph{boundary} $D$ is a Weil $\Q$-divisor on $X$ with coefficients between $0$ and $1$ such that $K_{X}+D$ is $\Q$-Cartier. It is \emph{log smooth} if $X$ is smooth and $D$ has simple normal crossings. We identify $(X,0)$ with $X$. Given a birational morphism of surfaces $\phi\:X'\to X$ we write $\phi\colon (X',D')\to (X,D)$ if $D=\phi_{*}D'$. We denote the reduced exceptional divisor of $\phi$ by $\Exc\phi$, the proper transform of a curve $C$ on $X'$ by $\phi^{-1}_{*}C$, and the base locus of a birational map $\psi$ by $\Bs\psi$. 

A \emph{log resolution} of a log surface $(\bar{X},\bar{D})$ is a proper birational morphism $\phi\colon (X,D)\to(\bar{X},\bar{D})$, where $D=\phi^{-1}_{*}\bar{D}+\Exc\phi$. It is \emph{minimal} if it is not birationally dominated by any other log resolution.

In this article we consider only log surfaces with reduced boundary. 
\smallskip

Let $X$ be a smooth projective surface. A curve $C$ on $X$ is called an \emph{$n$-curve} if $C\cong \P^1$ and $C^2=n$. 
Let $D$ be a divisor on $X$. We say that $D$ is \emph{connected} if $\Supp D$ is connected in the Zariski topology. We write $D\redd$ for $D$ with reduced structure. By a \emph{component} of $D$ we mean an \emph{irreducible} component of $D\redd$ (not a connected one); we denote their number by $\#D$. If divisors $D,T$ and $D-T$ are effective, we say that $T$ is  a \emph{subdivisor} of $D$. For two effective divisors $D_1,D_2$ we write $D_1\cap D_2$ for the intersection of their supports, and, if they are reduced, we write $D_1\wedge D_2$ for the sum of their common components. 
\smallskip

We now recall the definition of a weighted graph of a reduced divisor on a smooth projective surface. Abstractly, by a \emph{weighted graph} we mean a graph $\Gamma$ together with a weight function associating a pair of integers $(v\cdot v,p_{a}(v))$ to each vertex $v$; and an integer $(v\cdot w)_{e}$ to each edge $e$ between vertices $v$ and $w$. We write $v\cdot w=\sum_{e}(v\cdot w)_{e}$, where the sum runs over all edges $e$ between $v$ and $w$. The \emph{intersection matrix} of $\Gamma$ is $M_{\Gamma}\de [v_i\cdot v_j]_{1\leq i,j\leq n}$, where $v_{1},\dots,v_n$ are the vertices of $\Gamma$. We say that $\Gamma$ is \emph{negative definite} if so is $M_{\Gamma}$. The \emph{discriminant} of $\Gamma$ is $d(\Gamma)\de \det(-M_{\Gamma})$, or $1$ if $\Gamma$ is empty; see \cite[\S 3]{Fujita-noncomplete_surfaces} for its elementary properties. 
	
Let $D$ be a reduced divisor on a smooth projective surface. Its \emph{weighted graph} $\Gamma(D)$ is a weighted graph together with a bijection $\epsilon$ from the set of its vertices to the set of components of $D$ such that for each component $C$ of $D$, the weight of the vertex $\epsilon^{-1}(C)$ is $(C^2,p_{a}(C))$, and for two different components $C$, $C'$ of $D$, the vertices $\epsilon^{-1}(C)$, $\epsilon^{-1}(C')$ are connected by an edge of weight $(C\cdot C')_{p}$ for each $p\in C\cap C'$.  We write $D^v\de \epsilon(v)$, $v_C\de\epsilon^{-1}(C)$.  We often skip the edge weight if it is equal to $1$, and we skip $p_a(C)$ if it is equal to $0$. A morphism $\phi\colon (X',D')\to (X,D)$ induces a morphism of weighted graphs $\Gamma(D')\to \Gamma(D)$ by $v_{C'}\mapsto v_{\phi(C')}$.
		
We say that $D$ is \emph{negative definite} if so is its weighted graph; we put $d(D)\de d(\Gamma(D))$.  The \emph{branching number} of a component  $C$ of $D$ is 
	$\beta_{D}(C)\de C\cdot (D-C)$, 
i.e.\ the sum of weights of all edges of $\Gamma$ adjacent to $v_C$. 
We say that $C$  is a \emph{tip} of $D$ if $\beta_{D}(C)\leq 1$, and is \emph{branching} in $D$ if $\beta_{D}(C)\geq 3$.
\smallskip

A connected snc divisor  with no branching component is a \emph{chain} if it has a tip and \emph{circular} if it does not. A \emph{tree} is a connected snc divisor with no circular subdivisor. A tree $T$ whose unique branching component $B$ has  $\beta_{T}(B)=3$ is called a \emph{fork}. A \emph{$(-2)$-chain} (\emph{$(-2)$-fork}) is a chain (fork) whose all components are $(-2)$-curves.

Let $T$ be a chain with a chosen \emph{first} tip of $T$, call it $\ftip{T}$. Then the components $T\cp{1},\dots,T\cp{\#T}$  of $T$ are ordered in a unique way such that $T\cp{1}=\ftip{T}$ and $T\cp{i}\cdot T\cp{i+1}=1$, $1\leq i\leq \#T-1$. Then $\ltip{T}\de T\cp{\#T}$ is the \emph{last} tip of $T$. We write $T\trp$ for the same chain $T$ with an  opposite order.

A \emph{type} of an ordered rational chain $T$ is a sequence of integers $[a_1,\dots,a_{\#T}]$, where $a_{i}=-(T\cp{i})^{2}$. We  often abuse notation and identify a chain with its type. Moreover, for types $T_1=[a_1,\dots, a_k]$, $T_2=[b_1,\dots, b_l]$ we write $[T_1,T_2]=[T_1,b_1,\dots,b_l]=[a_1,\dots,a_k,b_1,\dots,b_l]$ etc. We warn the reader about a slight abuse of notation: if $T$ is a zero divisor ($T=0$), then its type is \emph{empty}, as opposed to $[0]$, which refers to a single $0$-curve. Hence if $T=0$ then we have $[a,T,b]=[a,b]$, \emph{not} $[a,0,b]$. 

We write $(m)_{k}$ for an integer $m$ repeated $k$ times. Following \cite{Tono_nie_bicuspidal}, we define the type $T_1*T_2$ as 
\begin{equation} \label{eq:convention-Tono_star}
	[a_{1},\dots,a_{k}]*[b_{1},\dots,b_{l}]\de [a_{1},\dots, a_{k-1},a_{k}+b_{1}-1,b_{2},\dots, b_{l}].
\end{equation}
If $T_1$ is empty, we put $0 *[b_{1},\dots,b_{l}]=[b_{2},\dots,b_{l}]$. The following convention will be useful:
\begin{equation}\label{eq:convention_2-1}
	[(2)_{-1}]*[b_{1},\dots, b_{l}]\de [b_2+1,\dots, b_{l}],\quad [(2)_{-1},b_{1},\dots,b_{l}]\de [b_{2},\dots, b_{l}].
\end{equation}

We say that a divisor $D$ \emph{contracts to} a divisor $D'$ (or to a singularity $p$), if there is a birational morphism $\phi$ such that $\Exc\phi\subseteq D$ and $\phi_{*}D=D'$ (or $\phi(D)=p$).

Let $T$ be an ordered rational chain on a surface $X$. Assume that $T=[1]$ or  $T$ is \emph{admissible}, i.e.\ $T=[a_1,\dots, a_{k}]$ with $a_j\geq 2$ for all $j$. Then there is a unique type $T^{*}$ such that $[T,1,T^{*}]$ contracts to $[0]$, see \cite[Proposition 4.7]{Fujita-noncomplete_surfaces}. For an admissible chain $[a_1,\dots,a_{k}]$, \cite[Lemma 5.ii]{Tono_nie_bicuspidal} gives:
\begin{equation}\label{eq:adjoint_chain}
	[a_1,\dots,a_k]^{*}=[(2)_{a_{k}-1}]*\dots*[(2)_{a_{1}-1}],
\end{equation}
and, clearly, $[1]^{*}$ is an empty chain. A chain $[T,1,T']$ contracts to a smooth point if and only if $T'=T^{*}*[(2)_{k}]$ for some $k\geq -1$. If a curve $C$ meets  a chain $[T,1,T^{*}]*[(2)_{k}]$ once, in the first tip, then the contraction of that chain increases the self-intersection number of $C$ by $k+2$.
\smallskip

Let again $D$ be an snc divisor. An ordered subchain $T\leq D$ is a \emph{twig} of $D$ if $\ftip{T}$ is a tip of $D$, and no component of $T$ is branching in $D$. In this case, $\ltip{T}$ meets $D-T$, or $T$ is a connected component of $D$. A \emph{$(-2)$-twig} is a twig whose all components are $(-2)$-curves. We say that a twig of $D$ is \emph{maximal} if it is not properly contained in any other twig of $D$. 
\smallskip

Let $T$ be a rational fork. Then $T=B+T_1+T_2+T_3$, where $B$ is the branching component of $T$, and $T_{j}$ are maximal twigs of $T$.  A \emph{type} of the fork $T$ consists of an integer $b\de -B^2$ and an (unordered) set of types $T_1$, $T_2$, $T_3$. We denote this type by 
	$\langle b;T_1,T_2,T_3 \rangle.$

\subsection{Log terminal and log canonical surface singularities}\label{sec:singularities}

Let $\pi\colon X\to \bar{X}$ be a resolution of a normal surface. Let $D\de\Exc\pi$ be its reduced exceptional divisor, and let $C_1,\dots, C_{k}$ be all components of $D$. Since the intersection matrix of $D$ is negative definite, the formula
\begin{equation}\label{eq:discrepancy}
	\pi^{*}K_{\bar{X}}=K_{X}+\sum_{i=1}^{k}\cf(C_i)\, C_i 
\end{equation}
uniquely defines rational numbers $\cf(C_i)$ called the \emph{coefficients} of $C_i$. We note that $\cf(C_i)$ depends only on the valuation of $\kk(\bar{X})$ associated to $C_i$, and not on $\pi$ \cite[Remark 2.23]{KollarMori-bir_geom}. More generally, if $\Gamma$ is a negative definite weighted graph, see Section \ref{sec:log_surfaces}, then for each vertex $v$ of $\Gamma$ we define $\cf_{\Gamma}(v)\in \Q$ by the formula 
\begin{equation*}
	0=2p_{a}(v)-2-v^2+\sum_{w} \cf_{\Gamma}(w)\, w\cdot v, 
\end{equation*}
where the sum runs over all vertices $w$ of $\Gamma$. Explicit formulas for these coefficients are given e.g.\ in \cite[\sec 3.1.10, \sec 3.2]{Flips_and_abundance}. If $D$ is a reduced divisor with weighted graph $\Gamma$, and $C$ is a component of $D$  corresponding to a vertex $v$, we write $\cf_{D}(C)\de \cf_{\Gamma}(v)$. This way, if $D$ is the exceptional divisor of a resolution of a normal surface $\bar{X}$, we have $\cf_{D}(C)=\cf(C)$ for every component $C$ of $D$. 

A normal surface is called \emph{canonical}, \emph{log terminal} (\emph{lt}), or \emph{log canonical} (\emph{lc}) if for every exceptional divisor $E$ of some (equivalently, any) resolution we have $\cf(E)\leq 0$, $\cf(E)<1$ or $\cf(E)\leq 1$, respectively. 

A surface singularity is canonical if and only if it is du Val, i.e.\ its minimal resolution is a $(-2)$-chain or fork \cite[3.26]{Kollar_singularities_of_MMP}. We now recall the classification of lt and lc surface singularities and fix some notation and terminology, see 3.39 and 3.40 loc.\ cit. A chain $T$ on $X$ is \emph{admissible} if it is rational, and all its components have self-intersection numbers at most $-2$. Let $F$ be rational fork on $X$ of type $\langle b;T_1,T_2,T_3\rangle$, with $b\geq 2$ and admissible twigs $T_1,T_2,T_3$. Put $\delta_{F}=\sum_{i=1}^3 d(T_{i})^{-1}$. We say that $F$ is \emph{admissible} if $\delta_{F}>1$; and $F$ is \emph{log canonical} if $\delta_{F}=1$ and $F$ is not a $(-2)$-fork. Every admissible chain or fork contracts to a log terminal singularity, and conversely, every log terminal singularity is of this type \cite[3.40]{Kollar_singularities_of_MMP}. Every log canonical fork contracts to a log canonical singularity, which is not log terminal because the branching component has coefficient $1$. Solving the inequality $\delta_{F}\geq 1$, one gets the following well-known description.

\begin{lemma}[{Admissible forks, cf.\ \cite[I.5.3.4]{Miyan-OpenSurf}}] \label{lem:admissible_forks}
	Let $F$ be a rational fork of type $\langle b;T_1,T_2,T_3\rangle$, with $b\geq 2$. 
	\begin{enumerate}
		\item\label{item:lt-fork} The fork $F$ is admissible if and only if the triple $\{d(T_1),d(T_2),d(T_3)\}$ is one of the following: $\{2,2,k\}$ for some $k\geq 2$; $\{2,3,3\}$, $\{2,3,4\}$, or $\{2,3,5\}$. 
		\item\label{item:lc-fork} The fork $F$ is log canonical if and only if  $\{d(T_1),d(T_2),d(T_3)\}=\{2,3,6\},\{2,4,4\}$ or $\{3,3,3\}$.
		\item\label{item:d=3} Let $T$ be an admissible chain, and let $d\de d(T)$. Then $d\geq 2$ and:
		\begin{enumerate}
			\item if $d\in \{2,3,4,6\}$ then $T=[d]$ or $[(2)_{d-1}]$.
			\item if $d=5$ then $T=[5]$, $[3,2]$, $[2,3]$ or $[2,2,2,2]$.
		\end{enumerate}
		\setcounter{foo}{\value{enumi}}
	\end{enumerate}
	In particular, the following hold.
	\begin{enumerate}\setcounter{enumi}{\value{foo}}		
		\item\label{item:has_-2} If the fork $F$ is admissible then one of its maximal twigs is $[2]$.
		\item\label{item:long-twig} Assume that $F$ is an admissible (respectively, log canonical) fork, and some twig $T$ of $F$ has $\#T\geq 3$. Then either $F=\langle b,[2],[2],T\rangle$, or $T$ is a $(-2)$-twig and $\#T\leq 4$ (respectively, $\#T=5$).
	\end{enumerate}
\end{lemma}

We say that an snc divisor $D$ is a \emph{bench} if $D=T+T_1^{+}+T_2^{+}+T_{1}^{-}+T_{2}^{-}$, where $T$ is a rational chain, called the \emph{central chain}, and $T_{i}^{\pm}=[2]$ is a twig of $D$ meeting $\textnormal{tip}^{\pm}(T)$. If $T=[a_1,\dots,a_n]$ we say that $D$ is \emph{a bench of type $\lbr a_1,\dots,a_n \rbr$}.  A bench is \emph{log canonical} if its central chain is admissible, but not a $(-2)$-chain. Every log canonical bench contracts to a log canonical singularity; in this case $\cf(C)=1$ for each component $C$ of the central chain, and $\cf(T_{i}^{\pm})=\frac{1}{2}$ for each tip. Conversely, if $E$ is the exceptional divisor of a minimal resolution of a log canonical, but not log terminal surface singularity, then $E$ is one of the following: a log canonical fork, a log canonical bench, a smooth elliptic curve, or a circular divisor whose components are rational and have self intersection numbers at most $-2$, with at least one strict inequality  \cite[3.39]{Kollar_singularities_of_MMP}. In the last two cases, we have $\cf(C)=1$ for each component $C$ of $E$.

A normal surface singularity $\bar{X}$ is \emph{rational} if some (equivalently, any) resolution $\pi\colon X\to \bar{X}$ satisfies $R^{1}\pi_{*}\cO_{X}=0$. By Artin's criterion  \cite{Artin-Contractibility}, cf.\ \cite[Theorem 7.1.2]{Nemethi_book}, a negative definite snc divisor $D$ contracts to a rational singularity if and only if its \emph{fundamental cycle}  $Z$ satisfies $p_{a}(Z)=0$. We recall that $Z$ is the minimal divisor with support equal to the support of $D$ such that $Z\cdot E\leq 0$ for every component $E$ of $D$. It can be computed from the weighted graph of $D$ by the Laufer algorithm, see \cite[Corollary 6.6.7(b)]{Nemethi_book}. This criterion implies that every log terminal surface singularity is rational, and, more generally, a log canonical surface singularity is rational if and only if the exceptional divisor of its minimal log resolution is an admissible or log canonical chain, fork, or bench, 
see \cite[Example 7.1.5(g),(h)]{Nemethi_book}. 

A \emph{singularity type} of a point $p\in \bar{X}$ is the weighted graph of the exceptional divisor of its minimal resolution. 
If the latter is a chain, fork or bench of type $T$, we say that $p$ is \emph{of type $T$}. For $(-2)$-chains and forks we also use notation for Dynkin types $\rA_{k}$, $\rD_{k}$, $\rE_{k}$. Hence we have the following equalities of singularity types:
\begin{equation*}
	\rA_{k}=[(2)_{k}],\ \rD_{k}=\langle 2;[(2)_{k-3}],[2],[2]\rangle,\ \rE_{k}=\langle 2;[(2)_{k-4}],[2,2],[2]\rangle.
\end{equation*}
We remark that if $\cha\kk=0$, the type of a log terminal singularity determines the isomorphism type of its local ring \cite[4.9(2)]{KollarMori-bir_geom}, but it is not so if $\cha\kk>0$, see \cite{Artin_coindices}. 
A \emph{singularity type} of a normal surface $\bar{X}$ is a formal sum  of singularity types of all singular points of $\bar{X}$. 

More abstractly, by a \emph{singularity type} we mean a negative definite weighted graph, see Section \ref{sec:log_surfaces}. We say that a singularity type $\Gamma$ is canonical (respectively, lt, lc) if $\cf_{\Gamma}(v)\leq 0$ (respectively, $\cf_{\Gamma}(v)<1$, $\cf_{\Gamma}(v)\leq 1$) for every vertex $v$ of $\Gamma$; and is rational if and only if it satisfies Artin's criterion. This way, if $\Gamma$ is a singularity type of a normal surface $\bar{X}$, then all singularities of $\bar{X}$ are canonical,  lt, lc, or rational if and only if so is $\Gamma$.

\subsection{\texorpdfstring{$\P^1$}{P1}-fibrations}\label{sec:P1-fibrations}

Let $(X,D)$ be a log smooth surface. A \emph{$\P^{1}$-fibration} of $X$ is a morphism $p\colon X\to B$ whose general fiber $F$ is isomorphic to $\P^1$. We call $F\cdot D$ the \emph{height} of $p$ (with respect to $D$). Recall from Definition \ref{def:height} that the \emph{height} of $(X,D)$ is the infimum of the heights of all $\P^1$-fibrations of $X$ with respect to $D$. A \emph{witnessing} $\P^1$-fibration is any $\P^1$-fibration realizing this infimum. Height of a normal surface is the height of its minimal log resolution.
\smallskip

Fix a $\P^1$-fibration $p$ of $X$. A curve $C\subseteq X$ is \emph{vertical} if $p(C)$ is a point; otherwise $C$ is horizontal. Any divisor $T$ decomposes uniquely as a sum of its vertical and horizontal parts, i.e.\ $T=T\vert+T\hor$, where all components of $T\vert$ are vertical and all components of $T\hor$ are horizontal. We say that $T$ is \emph{horizontal} if $T=T\hor$. 

A curve $H$ is an $n$-section if $H\cdot F=n$ for a fiber $F$ of $p$. In the course of our classification, whenever $D$ contains such a multi-section, we need to consider two cases: when $p|_{H}$ is separable, so the Hurwitz formula applies; and, if $\cha\kk | n$, when it is not. To avoid such case-splitting, we prefer $\P^1$-fibrations with as few multi-sections in $D$ as possible. Therefore, we first study cases of large width, defined as follows.

\begin{definition}[Width]\label{def:width}
	Let $(X,D)$ be a log smooth surface of finite height.  
	We define the \emph{width} of $(X,D)$ as the maximal value of $\#D\hor$ among all witnessing $\P^{1}$-fibrations. We denote it by $\width(X,D)$. 
\end{definition}
Clearly, $\width(X,D)\leq \height(X,D)$, and the equality holds if and only if $X$ admits a $\P^1$-fibration such that $D\hor$ consists of $\height(X,D)$ $1$-sections.
\smallskip

A fiber $F$ of a $\P^1$-fibration is \emph{degenerate} if it is not isomorphic to $\P^1$. Any such $F$ contracts to a $0$-curve, i.e.\ can be recovered by inductively blowing up over a $0$-curve. This leads to the following description of $F$.

\begin{lemma}[{Degenerate fibers, cf.\ \cite[\S 4]{Fujita-noncomplete_surfaces}}]\label{lem:degenerate_fibers}
	Let $F$ be a degenerate fiber of a $\P^{1}$-fibration of a smooth projective surface. Then $F\redd$ is a rational tree with no branching $(-1)$-curve, and the following hold.
	\begin{enumerate}
		\item\label{item:unique_-1-curve} If a $(-1)$-curve $L$ has multiplicity $1$ in $F$ then  $\beta_{F\redd}(L)=1$ and $F\redd-L$ contains another $(-1)$-curve.
		\item\label{item:adjoint_chain}
		Assume that $F$ contains exactly one $(-1)$-curve $L$. Then $F$ has exactly two components of multiplicity $1$, they are tips of $F$, and one of the following holds. 
		\begin{enumerate}
			\item \label{item:columnar} The fiber $F$ is \emph{columnar}, i.e.\ $F\redd=[T,1,T^{*}]$ for some admissible chain $T$. 
			\item \label{item:not_columnar} Both components of multiplicity one belong to the same connected component of $F\redd-L$, say $C$. The divisor $F\redd-L-C$ is empty or a chain.
		\end{enumerate}
	\end{enumerate}
\end{lemma}

Contracting all fibers to $0$-curves yields the following formula, see \cite[4.16]{Fujita-noncomplete_surfaces}, cf.\  \cite[Lemma 2]{Palka-AMS_LZ}.

\begin{lemma}
	\label{lem:fibrations-Sigma-chi}
	Let $(X,D)$ be a log smooth surface. Fix a $\P^1$-fibration of $X$. Let $\nu_{\infty}$ be the number of fibers contained in $D$. For a fiber $F$ let $\sigma(F)$ be the number of components of $F\redd$ not contained in $D$, i.e.\ $\sigma(F)\de \#(F\redd-D\wedge F\redd)$. Put  $\Sigma=\sum_{F}(\sigma(F)-1)$, where the sum runs over all fibers  not contained in $D$. Then
	\begin{equation*}\#D\hor+\nu_{\infty}+\rho(X)=\#D+2+\Sigma.\end{equation*}
\end{lemma}

\subsection{Del Pezzo surfaces}

We now summarize basic properties of $\P^1$-fibrations of the minimal log resolutions of del Pezzo surfaces of rank one. Those properties will be used frequently in our series of articles. Crucially for this one, Lemma \ref{lem:delPezzo_criterion} implies that a log terminal surface of rank one and height at most two is automatically del Pezzo.
\begin{remark}[Finiteness of height]
	\label{rem:ht_finite_nonzero}
	Let $\bar{X}$ be a del Pezzo surface of rank one, not isomorphic to $\P^2$. Its minimal resolution has negative Kodaira dimension, so by the Enriques--Kodaira classification of surfaces it is $\P^1$-fibered, and therefore $\height(\bar{X})<\infty$. Also, $\height(\bar{X})>0$, as otherwise a general fiber of a witnessing $\P^1$-fibration would map to a $0$-curve on $\bar{X}$, contrary to the assumption $\rho(\bar{X})=1$.
\end{remark}

\begin{lemma}[Criterion for ampleness of $-K_{\bar{X}}$ using a $\P^1$-fibration]\label{lem:delPezzo_criterion}
	Let $\bar{X}$ be a normal surface of Picard rank one such that $K_{\bar{X}}$ is $\Q$-Cartier. Let $\pi\colon X\to \bar{X}$ be its resolution. Assume that $X$ admits a $\P^1$-fibration, let $F$ be its fiber and let $H_1,\dots, H_{h}$ be all horizontal components of $\Exc\pi$. Then $\bar{X}$ is del Pezzo if and only if
	\begin{equation}\label{eq:ld_phi_H}
		\sum_{j=1}^{h}\cf(H_j)\, H_j\cdot F<2.
	\end{equation}
	In particular, every log terminal surface of rank one and height at most two is del Pezzo.
\end{lemma}
\begin{proof}
	If $\bar{X}$ is del Pezzo then $-K_{\bar{X}}\cdot \pi(F)>0$, and conversely, if $\rho(\bar{X})=1$, $K_{\bar{X}}$ is $\Q$-Cartier and $-K_{\bar{X}}\cdot \pi(F)>0$ then $\bar{X}$ is del Pezzo. The inequality  $-K_{\bar{X}}\cdot \pi(F)>0$ is equivalent to \eqref{eq:ld_phi_H}. Indeed, by formula \eqref{eq:discrepancy} we have  
	\begin{equation*}
		-K_{\bar{X}}\cdot \pi(F)=-\Big(K_{X}+\sum_{T}\cf(T)\, T\Big)\cdot F=2-\sum_{j=1}^{h} \cf(H_{j})\, H_{j}\cdot F,
	\end{equation*}
	where the first sum runs over all components $T$ of $\Exc \pi$. This proves the first part of the lemma.
	
	Assume now that $\bar{X}$ is log terminal (hence $K_{\bar{X}}$ is $\Q$-Cartier), $\rho(\bar{X})=1$ and $\height(\bar{X})\leq 2$. Since $\height(\bar{X})\leq 2$, the minimal log resolution $(X,D)$ of $\bar{X}$ has a $\P^1$-fibration whose fiber $F$ satisfies $D\cdot F\leq 2$.  Because $\bar{X}$ is log terminal, we have $\sum_{j}\cf(H_j)\, H_j\cdot F<\sum_{j} H_{j}\cdot F=D\cdot F \leq 2$, so \eqref{eq:ld_phi_H} holds, as needed.
\end{proof}

\begin{lemma}[Elementary properties of minimal resolutions of del Pezzo surfaces]\label{lem:min_res}
	Let $\bar{X}$ be a del Pezzo surface. Let $\pi\colon X\to \bar{X}$ be its minimal resolution, and let $D\de \Exc \pi$. Then the following hold.
	\begin{enumerate}
		\item\label{item:ld<=1} Every component $C$ of $D$ satisfies $\cf(C)\geq 0$.
		\item\label{item:one-elliptic} Assume that $D$ has a component $C\not\cong \P^1$. Then $\cf(C)\geq 1$ and $2p_{a}(C)+\beta_{D}(C)<2-C^2$. Moreover, $C$ is a $1$-section of some $\P^1$-fibration of $X$ such that $D-C$ is vertical. In particular, $\height(\bar{X})=1$.
		\item\label{item:D-snc} The divisor $D$ is snc. In particular, $(X,D)$ is the minimal log resolution of $\bar{X}$.
		\item\label{item:-1-curves_off_D} Every curve $L\subseteq X$, $L\not\subseteq D$ satisfies $L^{2}\geq 2p_{a}(L)-1$. In particular, if $L^2<0$ then $L$ is a $(-1)$-curve.
	\end{enumerate}
\end{lemma}
\begin{proof}
	\ref{item:ld<=1} This follows from the minimality of $\pi$ and from \cite[Lemma 3.39(1)]{KollarMori-bir_geom}, cf.\ 
	  \cite[Lemma 7.3(3)]{Palka_almost_MMP}.

	\ref{item:one-elliptic} Put $g=p_{a}(C)$, $e=-C^2$, $\beta=\beta_{D}(C)$. By the adjunction formula $C\cdot K_{X}=2g-2+e\geq e$, as $C\not\cong \P^1$. By \cite[3.1.3]{Flips_and_abundance} we have $\cf(C)\geq \cf_{C}(C)=\frac{C\cdot K_{X}}{e}\geq 1$, which proves the first assertion.
		
	Since $\kappa(K_{X})\leq \kappa(K_{\bar{X}})=-\infty$ and $X\not\cong \P^2$, the surface $X$ admits a $\P^1$-fibration $p\colon X\to B$. Let $F$ be a fiber of $p$. Since $C\not\cong \P^1$, the curve  $C$ is horizontal, i.e.\ $C\cdot F\geq 1$. We have $C\cdot F\leq \cf(C)\, C\cdot F<2$ by \eqref{eq:ld_phi_H}, so $C\cdot F=1$. 
	Suppose $D-C$ has a horizontal component $C'$. Since $B$ has genus $g\geq 1$, we have $C'\not\cong \P^1$, so $C'$ is another section of $p$ with $\cf(C')\geq 1$. Thus $\cf(C)+\cf(C')\geq 2$, a contradiction with the inequality \eqref{eq:ld_phi_H}.
	
	It remains to prove that $2g+\beta<2+e$. Let $V_1,\dots,V_{\beta}$ be the components of $D$ meeting $C$, let $R\de C+\sum_{i}V_i$ and $\delta\de \sum_{i}\frac{1}{-V_i^2}$. Negative definiteness of $R$ is equivalent to the inequality $e>\delta$. Using \cite[3.1.10]{Flips_and_abundance} we compute $\cf_{R}(C)=1-\frac{2-2g-\beta+\delta}{e-\delta}$. By \cite[Lemma 3.1.3]{Flips_and_abundance} and \eqref{eq:ld_phi_H} we have $\cf_{R}(C)\leq \cf(C)<2$, so $e-\delta -2+2g+\beta-\delta<2e-2\delta$ and therefore $2g+\beta<2+e$, as needed.
		
	\ref{item:D-snc} Part \ref{item:one-elliptic} implies that all components of $D$ are smooth. Suppose $D$ has two components $C_1,C_2$ such that $C_1\cdot C_2\geq 2$. If, say, $C_1\not\cong \P^1$ then part \ref{item:one-elliptic} shows that $C_1$ is a section and $C_2$ is vertical, so $C_1\cdot C_2\leq 1$, which is false. Thus $C_1,C_2\cong \P^1$. 
	Because $D$ is negative definite, we have, say, $C_1^2\leq -3$ and $C_2^2\leq -2$. As before, \cite[Lemma 3.1.3]{Flips_and_abundance} implies that $\cf(C_i)\geq \cf_{E}(E_i)=1$, where $E=E_1+E_2$, $E_{1}\cdot E_{2}=2$ and $p_{a}(E_i)=0$. Consider any $\P^1$-fibration of $X$. If both $C_1$ and $C_2$ are horizontal then the inequality \eqref{eq:ld_phi_H} fails. Thus, say, $C_1$ is vertical. Then $C_2$ is a multi-section, so by \eqref{eq:ld_phi_H} and part \ref{item:ld<=1} we have $\cf(C_2)<1$, a contradiction.
	
	Similarly, if $D$ has three components $C_1$, $C_2$, $C_3$ meeting normally at one point then applying \cite[Lemma 3.1.3]{Flips_and_abundance} as before we see that $\cf(C_i)\geq 1$ for each $i$. Fix a $\P^1$-fibration of $X$. By the inequality  \eqref{eq:ld_phi_H} at most one of the $C_i$'s is horizontal, and it is a $1$-section, which is impossible, since it meets two components of $D\vert$.
	
	\ref{item:-1-curves_off_D} 	Part \ref{item:ld<=1} and formula \eqref{eq:discrepancy} give $K_{X}\cdot L\leq \pi^{*}K_{\bar{X}}\cdot L=K_{\bar{X}}\cdot \pi(L)<0$, since $-K_{\bar{X}}$ is ample. By adjunction,  $2p_{a}(L)-2=K_{X}\cdot L+L^2<L^2$, as needed.
\end{proof}

We will also need the following lemma. 
Recall that given a reduced divisor $D$ on a smooth surface with a fixed $\P^1$-fibration, we write $D\hor$ and $D\vert$ for horizontal and vertical part of $D$, see Section \ref{sec:P1-fibrations}.

\begin{lemma}[$\P^1$-fibrations on a minimal resolution of a surface of rank one]\label{lem:delPezzo_fibrations}
	Let $\bar{X}$ be a normal surface of Picard rank one, and let $\pi\colon (X,D)\to (\bar{X},0)$ be its minimal log resolution. Fix a $\P^1$-fibration of $X$, and let $F$ be its degenerate fiber. Then the following hold.
	\begin{enumerate}
		\item\label{item:Sigma} We have $\#D\hor=\Sigma+1$, where $\Sigma$ is as in Lemma \ref{lem:fibrations-Sigma-chi}.
		\item\label{item:-1_curves} Assume that $\bar{X}$ is del Pezzo. Then a component $L$ of $F$ is a $(-1)$-curve if and only if $L\not\subseteq D$.
		\item\label{item:rivet} Assume that a component $R$ of $F$ satisfies $R\cdot D\hor=F\cdot D\hor$. Then $R\subseteq D$ and $F\redd$ has exactly one component not contained in $D$.
	\end{enumerate}
\end{lemma}
\begin{proof}
	Parts \ref{item:Sigma} and \ref{item:-1_curves} are direct consequences of Lemmas \ref{lem:fibrations-Sigma-chi} and \ref{lem:min_res}\ref{item:-1-curves_off_D}, proved in \cite[Lemma 2.6]{PaPe_MT}. 
	
	We prove part \ref{item:rivet}. Put $F'=F\redd-R$, $D'=D-(F\redd\wedge D)$. Since $R\cdot D\hor=F\cdot D\hor$, $F'$ is disjoint from $D'$ and $R$ has multiplicity one in $F$. The first condition implies that $F'+D'$ is negative definite, so by the Hodge index theorem $\#(F'+D')\leq \rho(X)-1=\rho(\bar{X})+\#D-1=\#D=\#(D-D\wedge R)+\#(F'\wedge D)+\#D'$. Therefore, $\#F'-\#(F'\wedge D)\leq \#(D-D\wedge R)\leq 1$.
	
	Since $R$ has multiplicity one in $F\redd$, by Lemma \ref{lem:degenerate_fibers}\ref{item:unique_-1-curve} $F'$ contains a $(-1)$-curve, say $L$. If $L\subseteq D$ then, since $\pi$ is minimal, we have $\beta_{D}(L)\geq 3>\beta_{F'}(L)$, so $L$ meets $D\hor\subseteq D'$, which is false. Thus $L\not\subseteq D$, so $\#F'-\#(F'\wedge D)\geq 1$. It follows that $\#(D-D\wedge R)=1$, i.e.\ $R\subseteq D$, and $\#F'-\#(F'\wedge D)=1$, i.e.\ $L$ is the unique component of $F'$ which is not contained in $D$, as needed.
\end{proof}

\begin{lemma}\label{lem:swap_lc}
	Let $(X,D)$ be a minimal resolution of a log canonical del Pezzo surface $\bar{X}$. Let $(X,D)\sqto (X',D')$ be a swap, see Definition \ref{def:vertical_swap}. Then the following hold.
	\begin{enumerate}
		\item\label{item:swap_lc_res} The log surface $(X',D')$ is a minimal log resolution of a log canonical surface $\bar{X}'$.
		\item \label{item:swap_lc_rho} We have 	$\rho(\bar{X}')=\rho(\bar{X})$.
		\item\label{item:swap_lc_cf}  For every component $T'$ of $D'$ we have $\cf_{D'}(T')\leq \cf_{D}(T)$, where $T$ is the proper transform of $T'$ on $X$.
		\item\label{item:swap_lc_dP} If $\rho(\bar{X})=1$ then $\bar{X}'$ is a log canonical del Pezzo surface of rank one.
	\end{enumerate}	
\end{lemma}
\begin{proof}
	Denote the contraction of $D$ by  $\pi\colon X\to \bar{X}$. By induction, we can assume that the swap $\phi\colon (X,D)\sqto (X',D')$ is elementary. 
	Let $L,C$ be as in Definition \ref{def:vertical_swap}\ref{item:def-swap-elementary}. Let $B\neq [2]$ be the component of $D-C$ meeting $L$; put $B=0$ if there is no such. Let $\check{D}_0$ be a connected component of $D+L$ containing $L$, let $\check{D}_{0}'=\phi(\check{D}_{0})$, and let $D_0=\check{D}_0-L$, $D'_{0}=\check{D}_{0}'-\phi(C)$; so $D_{0}'=\phi_{*}D_0\wedge D'$ and $\phi$ is an isomorphism near $D-D_0$. The weighted graph of $D_{0}'$ is obtained from the one of $D_{0}$ by removing the vertex corresponding to $C$, and, if $B\neq 0$, by decreasing the weight of the vertex corresponding to $B$, which is at most $-3$. 
	
	\ref{item:swap_lc_res}  It is enough to prove that every connected component of $D_0'$ is an admissible chain or an admissible fork. Suppose the contrary. 
	Then the above description shows that $D_0$ has a connected component $V$ which is not an admissible chain or fork. Let $G\in \{B,C\}$ be a component of $V$ meeting $L$. By the adjunction formula and by Lemma \ref{lem:min_res}\ref{item:ld<=1} we have $0>K_{\bar{X}}\cdot \pi(L)\geq -1+\cf(G)$. Thus $\cf(G)<1$. Since $\bar{X}$ is log canonical, we get that $G$ lies in a twig of a log canonical fork or bench. The claim now follows from the above description of $D_0'$. 
	
	\ref{item:swap_lc_rho} We have $\rho(\bar{X}')=\rho(X')-\#D'=(\rho(X)-1)-(\#D-1)=\rho(X)-\#D=\rho(\bar{X})$, as claimed.  
	
	\ref{item:swap_lc_cf} Part \ref{item:swap_lc_res} implies that $\cf_{D'}(\phi(B))\leq 1$ and that $D'$ is negative definite, in particular $d(D')>0$. Now \ref{item:swap_lc_cf} follows from the above description of $D_0'$ and \cite[Lemma 7.5]{Palka_almost_MMP}, cf.\ \cite[L.1(2)]{Keel-McKernan_rational_curves}.
	
	\ref{item:swap_lc_dP} Denote the contraction of $D'$ by $\pi'\colon X'\to \bar{X}'$.  By \ref{item:swap_lc_res} and \ref{item:swap_lc_cf}, $\bar{X}'$ is a log canonical surface of rank one. It remains to prove that $K_{\bar{X'}}\cdot \bar{R}'<0$ for some curve $\bar{R}'$. Take for $\bar{R}'$ any curve on $\bar{X}'$ other than $\pi'(\phi(C))$. Let $R'=(\pi')^{-1}_{*}\bar{R}'$, $R=\phi^{*}R'$ and $\bar{R}=\pi_{*}R$, so $\bar{R}$ is a nonzero, effective divisor on $\bar{X}$. Since $\bar{X}$ is del Pezzo, we have $0>K_{\bar{X}}\cdot \bar{R}=\phi_{*}\pi^{*}K_{\bar{X}}\cdot R'$. By Lemma \ref{lem:min_res}\ref{item:ld<=1} we have $\cf_{D}(C)\geq 0$. Hence part \ref{item:swap_lc_cf} implies that the divisor $\phi_{*}\pi^{*}K_{\bar{X}}-(\pi')^{*}K_{\bar{X}'}=\cf_{D}(C)\, \phi(C)+\sum_{T\subseteq D-C}(\cf_{D}(T)-\cf_{D'}(\phi(T)))\, \phi(T)$ is effective. Since its support does not contain $R'$, we infer that $\phi_{*}\pi^{*}K_{\bar{X}}\cdot R'\geq (\pi')^{*}K_{\bar{X}}'\cdot R'=K_{\bar{X}'}\cdot \bar{R}'$. Hence $0>K_{\bar{X}'}\cdot \bar{R}'$, as needed. 
\end{proof}

\subsection{Inner and outer blowups, and basic properties of the logarithmic tangent sheaf}\label{sec:hi}

Let $(X,D)$ be a log smooth surface. Let $\phi\colon (X',D')\to (X,D)$ be a blowup at a point $p\in D$, with $D'=(\phi^{*}D)\redd$. The blowup $\phi$ is called \emph{inner} if $p\in \Sing D$ and \emph{outer} if $p\in D\reg$, see \cite[Definitions 2.3]{FKZ-weighted-graphs}. 

To motivate the discussion here and in the next section, we outline a crucial observation underlying our proof of Propositions \ref{prop:moduli}, \ref{prop:moduli-hi}. We will formalize it in Lemmas \ref{lem:inner} and \ref{lem:outer} below, after introducing suitable notation.
\begin{observation}[{Inner vs.\ outer blowups, cf.\ Example \ref{ex:h1}}]\label{obs:blowups}
	Fix a log smooth surface $(X,D)$, and let $\phi$ be a blowup as above. If $\phi$ is inner, then the isomorphism class of $(X',D')$ depends only on $(X,D)$ and on a combinatorial datum: the edge of the weighted graph of $D$ corresponding to the center $p$ of $\phi$. On the other hand, if $\phi$ is outer, then $(X',D')$ depends also on a continuous datum: the position of $p$ on the corresponding component of $D$, up to the action of $\Aut(X,D)$. 
	 Thus, from the viewpoint of our classification, constructions involving only inner blowups are easier, while outer blowups require some control of the automorphism groups. 
\end{observation}	 

We call a birational map $(X',D')\map (X,D)$ \emph{inner} if it factors as a composition of inner blowups and blowdowns. We  distinguish two useful classes of inner maps.

Let $(X,D)$ be a log smooth surface such that $D$ contains a $0$-curve $F$ meeting exactly two components of $D-F$, once each. An \emph{elementary transformation on $F$}, see  \cite[Definition 2.10]{FKZ-weighted-graphs}, is an inner map $\sigma_{1}\circ \sigma_{2}^{-1}\colon (X',D')\map (X,D)$, where $\sigma_{1}$ is a blowup at one of the two points in $F\cap (D-F)$, and $\sigma_{2}$ is the contraction of $(\sigma_{1})^{-1}_{*}F$. Note that if $D$ contains a subchain  $[a,0,b]$, with $F$ being the middle component, then the reduced total transform of this subchain in $D'$ is of type $[a+1,0,b-1]$ or $[a-1,0,b+1]$, see \cite[\sec 2.3]{FKZ-weighted-graphs}.

For a fixed $\P^1$-fibration of $X$, a \emph{vertical inner snc-minimalization} $(X,D)\to (Y,B)$ is the contraction of vertical $(-1)$-curves of branching number $2$ in $D$ and its subsequent images so that $B$ has no such curves.
\smallskip

Following \cite[\sec 1]{FZ-deformations} we now describe the effect of inner and outer blowups on the cohomology of the logarithmic tangent sheaf, and we compute the latter in some basic cases.

\begin{lemma}[Change of $h^{i}(\lts{X}{D})$ under a blowup]\label{lem:blowup-hi}
	Let $(X,D)$ be a log smooth surface, let $\phi\colon X'\to X$ be a blowup at a point $p\in X$. Let $D'$ be either $ \phi^{-1}_{*}D$ or $(\phi^{*}D)\redd$. Put $\cT\de \lts{X}{D}$, $\cT'\de \lts{X'}{D'}$. 
	\begin{enumerate}
		\item\label{item:blowup-hi-exact} Let $r\in \{0,1,2\}$ be the number of components of $D$ containing $p$. We have an exact sequence
		\begin{equation*} 
		\begin{tikzcd}[cramped]
				0 \ar[r] & H^{0}(X',\cT') \ar[r, "d\phi"] & H^{0}(X,\cT)  \ar[r, "\textnormal{ev}_{p}"] & \kk^{r}  \ar[r] & H^{1}(X',\cT') \ar[r] & H^{1}(X,\cT) \ar[r] & 0.
		\end{tikzcd}		
		\end{equation*} 
		\item\label{item:blowup-h2} We have $h^{2}(\cT')=h^{2}(\cT)$. 
		\item\label{item:blowup-hi-inner} If $p\in \Sing D$, i.e.\ $\phi$ is inner, then $h^{i}(\cT')=h^{i}(\cT)$ for every $i$.
		\item\label{item:blowup-hi-outer} Assume that $p$ is a smooth point of $D$, i.e.\ $\phi$ is outer. If every vector field on $X$ tangent to $D$ vanishes at $p$ then $h^{1}(\cT')=h^{1}(\cT)+1$ and $h^0(\cT')=h^0(\cT)$, otherwise $h^{1}(\cT')=h^{1}(\cT)$ and $h^{0}(\cT')=h^0(\cT)-1$.
	\end{enumerate}
\end{lemma}
\begin{proof}
		First, we note that by \cite[Proposition 1.7(3)]{FZ-deformations} removing the $(-1)$-curve from $D'$ does not change $h^{i}(\lts{X'}{D'})$, see Lemma \ref{lem:h1}\ref{item:h1-1_curve} below, so we can assume $D'=\phi^{*}D\redd$. A direct computation, see Lemma 1.5(3) loc.\ cit.\ or \cite[Exercise 10.5]{Hartshorne_deformations} gives an exact sequence of sheaves 
		\begin{equation*} 
			\begin{tikzcd}[cramped]
				0 \ar[r] & \cT' \ar[r, "d\phi"] & \cT  \ar[r, "\textnormal{ev}_{p}"] & \kk^{r}_{p} \ar[r] & 0.
			\end{tikzcd}		
		\end{equation*}
		The lemma follows from the induced exact sequence in cohomology.
\end{proof}

\begin{remark}[Extension of Lemma \ref{lem:blowup-hi}\ref{item:blowup-hi-outer} in case  $\kk=\C$]\label{rem:outer-C}
		Assume $\kk=\C$. Let $(X',D')\to (X,D)$ be an outer blowup at $p\in D$ as in Lemma \ref{lem:blowup-hi}\ref{item:blowup-hi-outer}, and let $U$ be the orbit of $\Aut(X,D)$ containing $p$, so $\dim U\in \{0,1\}$. If $\dim U =1$ then $h^{1}(\lts{X'}{D'})=h^{1}(\lts{X}{D})$, otherwise $h^{1}(\lts{X'}{D'})=h^{1}(\lts{X}{D})+1$.
		
		Indeed, if $\dim U=1$ then then the differential at $g=\id$ of the morphism $\Aut(X,D)\ni g\mapsto g(p)\in X$ is nonzero, hence its image contains a vector field tangent to $D$ which does not vanish at $p$, so $h^{1}(\lts{X'}{D'})=h^{1}(\lts{X}{D})$ by  Lemma \ref{lem:blowup-hi}\ref{item:blowup-hi-outer}. Conversely, if such a vector field exists then its flow yields a one-dimensional family of biholomorphisms of $(X,D)$, which are regular since $X$ is projective, and move $p$, as needed.
		
		This growth of $h^1$ matches the growth of dimensions of representing families constructed in Lemma \ref{lem:outer}, see Observation \ref{obs:blowups} and Figure \ref{fig:outer}. The following Example \ref{ex:h1} illustrates the failure of the above statement in positive characteristic, which leads to exceptions in  Proposition \ref{prop:moduli-hi}.
\end{remark}

\begin{example}[Uniqueness and change of $h^1$ under outer blowups, see Figure \ref{fig:h1}]\label{ex:h1}
	To illustrate Observation \ref{obs:blowups} and Remark \ref{rem:outer-C}, we now follow the construction of minimal resolutions of del Pezzo surfaces of rank one and types $2\rA_1+\rA_3$, $\rA_1+\rA_5$, $\rA_7$, and $\rD_8$, cf.\ Figure \ref{fig:ht=1}. We outline the computation of $h^1(\lts{X}{D})$ and the classification of the resulting surfaces up to an isomorphism: they are unique except for type $\rD_8$ in case $\cha\kk=2$, when there is a one-dimensional family of pairwise non-isomorphic ones, see Table \ref{table:canonical}. This is a particular instance of  Lemma \ref{lem:ht=1_uniqueness}, in cases when $\cZ$ is as in \ref{lem:ht=1_reduction}\ref{item:uniq_easy}, \ref{item:uniq_n=3}, \ref{item:uniq_n=2}, \ref{item:uniq_fork}, see Figure \ref{fig:basic_ht=1} and Table \ref{table:ht=1_exceptions}.
	
	As a consequence, we will see that if $\cha\kk=2$ the sets $\cP(\rA_7)$ and $\cP(\rD_8)$ are not represented by almost universal families in the sense of Definition \ref{def:moduli-hi}. Indeed, the surface of type $\rA_7$ is unique up to an isomorphism, but its minimal log resolution has $h^1=1$. Similarly, there are infinitely many pairwise non-isomorphic  surfaces of type $\rD_8$, one of which has $h^1=2$ and the others have $h^1=1$. These exceptional cases of Proposition \ref{prop:moduli-hi}\ref{item:moduli-hi-1} and \ref{item:moduli-infinite} are listed in Table \ref{table:exceptions-to-moduli} as \ref{lem:ht=1_types}\ref{item:chains_not-columnar} with $T,T_i=[2]$, $m,r_i=2$; and \ref{lem:ht=1_types}\ref{item:TH_long_other_b>2} with $T=[2]$, $r,m=2$.
	
	\begin{figure}[htbp]
		\begin{tikzpicture}[scale=0.9]
			\begin{scope}
					\draw (0,3) -- (1.2,3);
					\node at (0.6,2.75) {\small{$0$}};
					\draw (0.1,3.1) -- (0.1,0.5);
					\node at (0.3,1.8) {\small{$0$}};
					\node at (1.2,1.8) {\small{$0$}};
					\draw (1,3.1) -- (1,0.5);		
					\draw[<-] (1.6,1.8) -- (2.6,1.8);
					\filldraw (0.1,3) circle (0.06);
					\draw[-stealth] (0.2,2.9) -- (0.2,2.5);
					\filldraw (1,3) circle (0.06);
					\draw[-stealth] (1.1,2.9) -- (1.1,2.5);
					\node at (0.6,0.2) {\small{$\P^1\times \P^1$}};
					\node at (0.6,-0.2) {\small{$(3,0)$}};
			\end{scope}	
			\begin{scope}[shift={(3,0)}] 
					\draw (0,3) -- (1.2,3);
					\draw (0.2,3.1) -- (0,2);
					\draw[dashed] (0,2.3) -- (0.2,1.3);
					\filldraw (1,1.8) circle (0.06);
					\draw (0.2,1.5) -- (0,0.5);
					\draw (1.1,3.1) -- (0.9,2.1);
					\draw[dashed] (0.9,2.3) -- (1.1,1.3);
					\draw (1.1,1.5) -- (0.9,0.5);
					\draw[<-] (1.5,1.8) -- (2.5,1.8);
					\node at (0.6,0.2) {\small{$2\rA_1+\rA_3$}};
					\node at (0.6,-0.2) {\small{$(3,0)$}};		
			\end{scope}	
			\begin{scope}[shift={(6,0)}]	
					\draw (0,3) -- (2,3);
					\draw (0.2,3.1) -- (0,2);
					\draw[dashed] (0,2.3) -- (0.2,1.3);
					\filldraw (0.1,1.8) circle (0.06);
					\draw (0.2,1.5) -- (0,0.5);
					\draw (1.1,3.1) -- (0.9,2.1);
					\draw (0.9,2.3) -- (1.1,1.3);
					\draw (1.1,1.5) -- (0.9,0.5);
					\draw[dashed] (0.9,1.8) -- (1.9,2);
					\draw[<-] (2.2,1.8) -- (3.2,1.8);
					\node at (0.6,0.2) {\small{$\rA_1+\rA_5$}};
					\node at (0.6,-0.2) {\small{$(2,0)$}};		
			\end{scope}	
			\begin{scope}[shift={(10.5,0)}]	
					\draw (-0.8,3) -- (1.8,3);
					\draw (0.2,3.1) -- (0,2.1);
					\draw (0,2.3) -- (0.2,1.3);
					\draw (0.2,1.5) -- (0,0.5);
					\draw (1.1,3.1) -- (0.9,2.1);
					\draw (0.9,2.3) -- (1.1,1.3);
					\draw (1.1,1.5) -- (0.9,0.5);
					\draw[dashed] (0.9,1.8) -- (1.9,2);	
					\draw[dashed] (0.2,1.8) -- (-0.8,1.6);
					\filldraw (1.56,1.95) circle (0.06);
					\draw[<-] (2.2,1.8) -- (3.2,1.8);
					\node at (0.6,0.3) {\small{$\rA_7$}};
					\node[right] at (-1.3,-0.1) {\small{$\cha\kk\neq 2: (1,0)$}};
					\node[right] at (-1.3,-0.5) {\small{$\cha\kk=2: (2,1)$}};
		\end{scope}	
		\begin{scope}[shift={(15,0)}]		
					\draw (-0.8,3) -- (1.8,3);
					\draw (0.2,3.1) -- (0,2.1);
					\draw (0,2.3) -- (0.2,1.3);
					\draw[dashed] (0.2,1.8) -- (-0.8,1.6);
					\draw (0.2,1.5) -- (0,0.5);
					\draw (1.1,3.1) -- (0.9,2.1);
					\draw (0.9,2.3) -- (1.1,1.3);
					\draw (1.1,1.5) -- (0.9,0.5);
					\draw (0.85,1.9) -- (1.85,2.1);
					\draw[dashed] (1.6,2.3) -- (1.8,1.3);
					\node at (0.6,0.3) {\small{$\rD_8$}};
					\node[right] at (-1.6,-0.1) {\small{$\cha\kk\neq 2: (0,0)$}};
					\node[right] at (-1.6,-0.5) {\small{$\cha\kk=2: (2,2)$ or $(1,1)$}};
		\end{scope}
		\end{tikzpicture}
		\caption{Example \ref{ex:h1}: $(h^{0}(\lts{X}{D}),h^1(\lts{X}{D}))$ after subsequent blowups.}
		\label{fig:h1} 
	\end{figure}
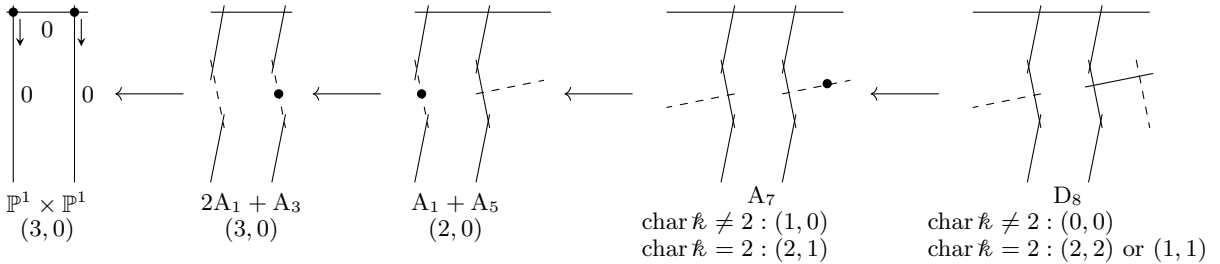

We will define a sequence of log surfaces 
$\begin{tikzcd}[cramped] (X_{0},\check{D}_{0}) & \ar[l,"\phi_1"'] (X_1,\check{D}_1) & \ar[l,"\phi_2"'] \dots \end{tikzcd}$,
where each $\check{D}_{j+1}$ is the reduced total transform of $\check{D}_j$; each $X_j$ for $j=1,2,3,4$ is a minimal resolution of a del Pezzo surface of rank one and type $2\rA_1+\rA_3$, $\rA_1+\rA_5$, $\rA_7$, $\rD_8$, respectively, with exceptional divisor $D_j$ equal to $\check{D}_j$ with all $(-1)$-curves removed. Put $h^{i}_{j}\de h^{i}(\lts{X_j}{\check{D}_j})$. Then $h^i_j=h^{i}(\lts{X_j}{D_j})$ by Lemma \ref{lem:h1}\ref{item:h1-1_curve}.
	
We start with $X_0=\P^1\times \P^1$, $D_0=V_0+H+V_{\infty}$, where $V_0=\{x=0\}$, $V_{\infty}=\{x=\infty\}$ are vertical lines, and $H=\{y=0\}$ is a horizontal line. Clearly, such $(X_0,D_0)$ is unique up to an isomorphism. We have $h^{0}_0=3$, $h^1_0=0$. In fact, every $\xi\in H^{0}(\lts{X_0}{D_0})$ is of the form $ax\frac{\d}{\d x}+by\frac{\d}{\d y}+cy^2\frac{\d}{\d y}$ for some $a,b,c\in \kk$. We denote such $\xi$, or its lift to $X_j$, by $\xi(a,b,c)$. The connected component of $\Aut(X_0,D_0)$ is generated by $G_{1}\de \{(x,y)\mapsto (\lambda x,y):\lambda\in \kk^{*}\}$, $G_{2}\de \{(x,y)\mapsto (x,\lambda  y):\lambda\in \kk^{*}\}$ and $G_{3}\de \{(x,y)\mapsto (x,\frac{y}{1+\lambda y}):\lambda\in \kk\}$. 

Let $\phi_1$ be a blowup at each $V_z\cap H$, $z=0,\infty$, and at its infinitely near point on the proper transform of $V_z$. It is an inner morphism, so $(h^0_1,h^1_1)=(h^0_0,h^1_0)=(3,0)$ by Lemma \ref{lem:blowup-hi}\ref{item:blowup-hi-inner}, and we can identify $\Aut(X_1,D_1)$ with $\Aut(X_0,D_0)$. Let $V_z'=(\phi_{1}^{-1})_{*}V_z$, let $E_{z}$ be the $(-1)$-curve in the fiber over $z$, and let $E_{z}^{\circ}=E_{z}\setminus (D_1-E_{z})$. In local coordinates $(x,y)$ with $V_0'=\{x=0\}$, $E_{0}=\{y=0\}$ we have $\phi_{1}(x,y)=(xy^2,y)$. A computation in these local coordinates shows that the vector field $\xi(a,b,c)$ vanishes along $E_0$ if and only if $a=2b$. Similarly, $\xi(a,b,c)$ vanishes along $E_{\infty}$ if and only if $a=2c$. These two conditions are nontrivial; they are linearly independent if $\cha\kk\neq 2$, and equivalent to $a=0$ if $\cha\kk=2$. Since $G_3\cong \G_a$, it acts trivially on $E_0^{\circ},E_{\infty}^{\circ}\cong \A^1_*$; and a computation shows that the subgroup generated by $G_1$ and $G_2$ acts transitively on $E_{0}^{\circ}\times E_{\infty}^{\circ}$.

Fix $p_{0}\in E_0^{\circ}$, $p_{\infty}\in E_{\infty}^{\circ}$. Let $\phi_2$ be a blowup at $p_0$, and let $\phi_3$ be a blowup at $\phi_{2}^{-1}(p_{\infty})$. Lemma \ref{lem:blowup-hi}\ref{item:blowup-hi-exact} and the above computation give $(h^{0}_2,h^{1}_2)=(2,0)$, and subsequently   $(h^0_3,h^1_3)=(1,0)$ if $\cha\kk\neq 2$ and $(2,1)$ if $\cha\kk=2$. The resulting log surfaces $(X_2,D_2)$ and $(X_3,D_3)$ are unique up to an isomorphism.

Let $E\subseteq D_3$ be a $(-1)$-curve in the fiber over $0$, let $E^{\circ}\de E\setminus (D_3-E)\cong \A^1$. We can choose $p\in E^{\circ}$ and local coordinates $(x,y)$ at $p$ such that $E^{\circ}=\{y=0\}$ and $\phi_2\circ\phi_3 (x,y)=(xy+1,y)$. The connected component of $\Aut(X_3,D_3)$ can be identified with $G_3$, acting in the above coordinates by $(x,y)\mapsto  ((1+\lambda y)((1+\lambda y)^2x+\lambda^2y+2\lambda),\frac{y}{1+\lambda y})$. Thus $\Aut(X_3,D_3)$ acts transitively on $E^{\circ}$ if $\cha\kk\neq 2$ and trivially if $\cha\kk=2$. 

If $\cha\kk\neq 2$, we compute that the vector field $\xi(2b,b,b)$ on $X_1$ lifts to $-b(3xy+x+2)\frac{\d}{\d x}+by(1+y)\frac{\d}{\d y}$, so it vanishes along $E^{\circ}$ if and only if $b=0$. If $\cha\kk\neq 2$ then the vector field $\xi(0,b,c)$ lifts to $bx(1+y)\frac{\d}{\d x}+(by+cy^2)\frac{\d}{\d y}$, so it always vanishes at $p$, and vanishes along $E^{\circ}$ if and only if $b=0$. 

Let $\phi_4$ be a blowup at $q\in E^{\circ}$. By Observation \ref{obs:blowups} if $\cha\kk\neq 2$ we get a unique log surface $(X_4,D_4)$, and if $\cha\kk=2$ we get a family of pairwise non-isomorphic ones, parametrized by the choice of $q\in E^{\circ}$. 
 By Lemma \ref{lem:blowup-hi}\ref{item:blowup-hi-exact} we have $(h^{0}_{4},h^{1}_{4})=(h^0_3,h^1_3+1)$ if $\cha\kk=2$ and $q=p$ and $(h^0_4,h^1_4)=(h^0_3-1,h^1_3)$ otherwise; so $(h^0_4,h^1_4)$ equals $(2,2)$ if $\cha\kk=2$, $q=p$; $(1,1)$ if $\cha\kk=2$, $q\neq p$; and $(0,0)$ if $\cha\kk\neq 2$.
\end{example}

\begin{lemma}[Elementary properties of $\lts{X}{D}$]\label{lem:h1}
	Let $(X,D)$ be a log smooth surface.
	\begin{enumerate}
		\item\label{item:h1_exact} For a component $C$ of $D$, denote by $\cN_{C}$ be the normal bundle to $C$. We have exact sequences 
		\begin{equation*}
			\begin{split}
				0\to \lts{X}{D}&\to \lts{X}{(D-C)}\to \cN_{C}\to 0, \quad\mbox{and}  \\
				0\to \lts{X}{D}&\to \cT_{X}\to\bigoplus_{i=1}^{k} \cN_{C_i}\to 0 \quad\mbox{where } C_1,\dots,C_k \mbox{ are the components of }D.
			\end{split}
		\end{equation*}
		\item\label{item:h1_positive_curve} If $C\subseteq D$ is an $m$-curve for some $m\geq 0$ then $h^{1}(\lts{X}{D})\geq h^{1}(\lts{X}{D-C})$.
		\item\label{item:h1-1_curve} If $C\subseteq D$ is a $(-1)$-curve then $h^{i}(\lts{X}{D})=h^{i}(\lts{X}{D-C})$.
		\item\label{item:h1-easy-vanishing} We have $h^{i}(\lts{\P^1}{D})=0$ for $i\in \{0,1,2\}$ in the following cases:
		\begin{enumerate} 
			\item\label{item:h1-P2} $X=\P^{2}$ and $D$ is a sum of four lines in a general position.
			\item\label{item:h1-diagonal} $X=\P^1\times \P^1$ and $D$ is a sum of two fibers, a horizontal line, and a diagonal. 
			\item\label{item:h1-grid} $X=\P^1\times \P^1$ and $D$ is a sum of $3$ horizontal and $3$ vertical lines.
		\end{enumerate}
		\item\label{item:h1-Fm} Let $X$ be the Hirzebruch surface $\F_{m}=\P(\cO_{\P^1}(m)\oplus\cO_{\P^1})$ for some $m\geq 0$, and let $D$ be a sum of $v$ fibers, a negative section $\Sec_m$, and possibly a section disjoint from $\Sec_m$. Then $h^{1}(\lts{X}{D})=\max\{0,v-3\}$.
		\item\label{item:h1-h2=h0} We have $h^{2}(\lts{X}{D})=h^{0}(\lts{X}{D}\otimes \cO_{X}(2K_{X}+D))$.
		\item\label{item:h1-h2-fibration} If $X$ has a $\P^1$-fibration with a fiber $F$ such that $F\cdot D\leq 3$ and $F$ meets $D$ normally then $h^{2}(\lts{X}{D})=0$. 
	\end{enumerate}
\end{lemma}
\begin{proof}
	The first exact sequence in \ref{item:h1_exact} follows from the definition of $\lts{X}{D}$, the second one follows from the first by induction. For \ref{item:h1_positive_curve} and \ref{item:h1-1_curve} note that $\cN_{C}=\cO_{\P^{1}}(m)$, so $h^{i}(\cN_{C})=0$ for $i>0$ in case \ref{item:h1_positive_curve} and $i\geq 0$ in case \ref{item:h1-1_curve}; so the assertions follow from \ref{item:h1_exact}. To prove \ref{item:h1-easy-vanishing}, we compute directly that in each case  $h^{0}(\lts{X}{D})=0$ and $h^{0}(\cT_{X})=\sum_{i}(C_{i}^{2}+1)=\sum_{i}h^{0}(\cN_{C_i})$, so by \ref{item:h1_exact} for $i>0$ we have $h^{i}(\lts{X}{D})=h^i(\cT_{X})=0$. 
	
	Let $(X,D)$ be as in \ref{item:h1-Fm}. Put $h=\#D\hor \leq 2$. If $m=0$ then $h^{1}(\lts{X}{D})=h^{1}(\cO_{\P^{1}}(2-h))+h^{1}(\cO_{\P^{1}}(2-v))=h^{0}(\cO_{\P^{1}}(h-4))+h^{0}(\cO_{\P^{1}}(v-4))=\max \{0,v-3\}$, and similarly $h^{2}(\lts{X}{D})=0$, as needed. 
	
	Assume $m\geq 1$. Let $F$ be a fiber in $D$. If $h=2$ write $\{p\}=F\cap (D-\Sec_m)$; if $h=1$ let $p$ be any point of $F^{\circ}\de F\setminus \Sec_{m}$. Let $(X,D)\map (X',D')$ be the elementary transformation given by blowing up $p$ and contracting the proper transform of $F$. Then $(X',D')$ is as in \ref{item:h1-Fm}, too, with $m$ replaced by $m-1$. If $h=1$, it is easy to see that 
	there is a vertical vector field which vanishes only along $\Sec_{m}$ and some fiber. Thus in any case Lemma \ref{lem:blowup-hi} gives  $h^{i}(\lts{X}{D})=h^{i}(\lts{X'}{D'})$ for $i\in \{1,2\}$, so \ref{item:h1-Fm} follows by induction on $m$. 
	
	Parts \ref{item:h1-h2=h0} and \ref{item:h1-h2-fibration} are proved in \cite[Proposition 6.2]{FZ-deformations}, as follows. Let $\Omega^{k}(\log D)$ be the sheaf of $k$-forms with logarithmic poles along $D$. We have a perfect pairing $\Omega^{1}(\log D)\times \Omega^{1}(\log D)\to \Omega^{2}(\log D)=\cO_{X}(K_{X}+D)$. Thus $\lts{X}{D}^{\vee}=\Omega^{1}(\log D)=\lts{X}{D} \otimes \cO_{X}(K_{X}+D)$ and \ref{item:h1-h2=h0} follows from Serre duality.
	
	Let $F$ be as in \ref{item:h1-h2-fibration}, $\xi\in H^{0}(X,\lts{X}{D}\otimes \cO_{X}(2K_{X}+D))$. The sheaf $\cN_{F}\otimes \cO_{X}(2K_{X}+D)|_{F}\cong \cO_{F}(F\cdot D-4)$ has no global sections, so by \ref{item:h1_exact} $\xi|_{F}\in H^{0}(F,\cT_{F}\otimes \cO_{X}(2K_{X}+D)|_{F})=H^{0}(F,\cO_{F}(F\cdot D-2))$. Since $F$ meets $D$ normally, $\xi|_{F}$ has $F\cdot D$ zeros, so $\xi|_{F}=0$. Since $F$ is a general fiber we get $\xi=0$, as needed.
\end{proof}

\subsection{Representing families and moduli}\label{sec:moduli}

Recall that Proposition \ref{prop:moduli} asserts that there is a smooth family whose fibers are precisely all minimal log resolutions $(X,D)$ of del Pezzo surfaces of a fixed singularity type and small height; and Proposition \ref{prop:moduli-hi} allows to choose such a family related to infinitesimal deformations of $(X,D)$. In this section, we introduce some technical language which will be convenient to inductively construct such families. 

First, we explain our motivation. In the proof of Theorem \ref{thm:ht=1,2}, we will construct vertical swaps to certain simple log surfaces, shown in Figure \ref{fig:basic-intro}. To prove Propositions \ref{prop:moduli} and \ref{prop:moduli-hi} we will apply applying Observation \ref{obs:blowups} to each elementary vertical swap. The technical issue here is that the combinatorial data in \ref{obs:blowups}, i.e.\ the vertex or edge of the weighted graph where we need to blow up, is determined by the weighted graph of $D$ only up to the action of a certain symmetry group $G$. Therefore, we will use  \emph{$G$-equivariant blowups} and construct \emph{$G$-faithful families}. This approach requires to keep a consistent order of vertices of our graphs, which is possible due to the irreducibility assumption in   Definition \ref{def:moduli}\ref{item:def-family-smooth}. Roughly speaking, once we forget about this order, a $G$-faithful family will become \emph{almost faithful} in the sense of Definition \ref{def:moduli}\ref{item:def-family-faithful}. 
\smallskip

Now we introduce precise definitions. A \emph{combinatorial type} is a pair $(\Gamma,c)$, where $\Gamma$ is a weighted graph as in Section \ref{sec:log_surfaces} and $c\in \Z$. For a log surface $(X,D)$ with smooth $X$ and reduced boundary $D$ its combinatorial type consists of the weighted graph of $D$ (\emph{without} the fixed identification of vertices and components of $D$) and the degree of the top Chern class of $X$.  We will sometimes abuse the notation and use the same letter for a singularity type of a del Pezzo surface of rank one and the combinatorial type of its minimal log resolution. 

Let $\cS\de (\Gamma,c)$ be a combinatorial type. We denote by $\cP(\cS)$ the set of isomorphism classes of log surfaces of combinatorial type $\cS$, modulo an isomorphism. We write $\Aut(\cS)\de \Aut(\Gamma)$. 

To study families of log smooth surfaces we need slightly more information than the combinatorial type, namely, a consistent way to enumerate boundary components. To this end, we define $\cP_{+}(\cS)$ as the set of triples $(X,D,\gamma)$ such that $(X,D)\in \cP(\cS)$ and $\gamma\colon \Gamma\to \Gamma(D)$ is an isomorphism of weighted graphs, modulo the following equivalence relation: we say that $(X_1,D_1,\gamma_1)$ is \emph{equivalent to} $(X_2,D_2,\gamma_2)$ if there is an isomorphism $(X_1,D_1)\to (X_{2},D_{2})$ of log surfaces such that the induced isomorphism of weighted graphs, see Section \ref{sec:log_surfaces}, equals $\gamma_{2}\circ \gamma_{1}^{-1}$. In other words, the additional datum $\gamma$ enumerates the components (and singular points) of $D$, and in $\cP_{+}(\cS)$ we identify two log surfaces if some isomorphism between them preserves that fixed order. 

Given a triple $(X,D,\gamma)\in \cP_{+}(\cS)$, a vertex $v$ and an edge $e$ of $\Gamma$, we denote by $D^v$ the component of $D$ corresponding to the vertex $\gamma(v)$, and by $D^e$ the singular point of $D$ corresponding to the edge $\gamma(e)$.
\smallskip

Let $f\colon (\cX,\cD)\to B$ be a smooth family whose fibers $(X_b,D_b)$ are log smooth surfaces. For a component $\cC$ of $\cD$ write $C_{b}=\cC|_{X_b}$. By Definition \ref{def:moduli}\ref{item:def-family-representing}, $C_b$ is an irreducible component of $D_b$.

As in Section \ref{sec:log_surfaces}, we define the weighted graph $\Gamma(f)$ together with a bijection $v\mapsto \cD^v$ from the set of its vertices to the set of components of $\cD$, such that the weight of $\cD^v$ equals $((D^v_b)^2,p_{a}(D^v_b))$, with two vertices $v,w$ being connected by $(D^v_b\cdot D^w_b)$ edges of weight $1$, for some $b\in B$. Put $\cS(f)=(\Gamma(f),c_2(X_b))$. 

The combinatorial type $\cS(f)$ is well defined, i.e.\ does not depend on $b\in B$. Indeed, for any two components $\cC$, $\cC'$ of $\cD$, and any $b\in B$, we have $X_{b}\cdot \cC\cdot \cC'=C_b\cdot \cC'=C_b\cdot C_{b}'$, where the second equality follows from the projection formula for the embedding $X_b\into \cX$. Since $f$ is flat, the fibers $X_b$ are numerically equivalent, so the number $C_b\cdot C_{b}'$ does not depend on $b\in B$. Since the restriction $f|_{\cC}\colon \cC\to B$ is flat, the arithmetic genus of $C_b$ does not depend on $b\in B$, either \cite[Corollary III.9.10]{Hartshorne_AG}. Hence $\Gamma(f)$. Similarly, using adjunction we have $X_{b}\cdot K_{\cX}^2=K_{X_b}\cdot K_{\cX}=K_{X_{b}}^2$, so $K_{X_{b}}^2$ does not depend on $b\in B$, and since $\chi(\cO_{X_{b}})$ does not depend on $b\in B$, either, by Noether's formula neither does the number $c_{2}(X_b)$, as claimed.

We can now  relate families of del Pezzo surfaces to families of their minimal log  resolutions.

\begin{lemma}[{Fiberwise contractions give a global contraction, cf.\ \cite[Theorem 1.4]{Wahl_deformation-theory}}]\label{lem:blowdown}
	Let $f\colon (\cX,\cD)\to B$ be a smooth family. Assume that for every $b\in B$, we have a contraction $\pi_{b}\colon X_b\to \bar{X}_b$ such that $D_b=\Exc\pi_b$ and $\bar{X}_b$ is a del Pezzo surface. 
	Then there is a factorization $f=\bar{f}\circ \pi$, where $\pi|_{f^{-1}(b)}=\pi_{b}$ for every $b\in B$.
\end{lemma}
\begin{proof}
	Since $\Gamma(f)$ is negative definite, for each vertex $v$ of $\Gamma(f)$ we have a rational number $\cf_{\Gamma}(v)$ such that $\pi_{b}^{*}K_{\bar{X}_b}=K_{X_b}+\sum_{v} \cf_{\Gamma}(v)\, D_{b}^{v}$ for all $b\in B$, see Section \ref{sec:singularities}. Put $\cH=-K_{\cX}-\sum_{v}\cf(v)\, \cD^{v}$. By adjunction  $K_{X_{b}}=K_{\cX}|_{X_{b}}$, see \cite[Proposition II.8.20]{Hartshorne_AG}, so  $\pi_{b}^{*}K_{\bar{X}_{b}}=-\cH|_{X_{b}}$. Since $\bar{X}_{b}$ is del Pezzo, it follows that $\cH|_{X_{b}}$ is $\pi_{b}$-semi-ample. Thus for every $b\in B$, the linear system $|m_{b}\cH|_{X_{b}}|$ is base point free for some $m_{b}\geq 1$. Since $B$ is Noetherian, $|m\cH|$ is base point free for some $m\geq 1$. Let $\Phi_{|m\cH|}\colon \cX\to \P^{N}$ be the induced morphism, let $\pi= (\Phi_{|m\cH|},f)\colon \cX\to \P^{N}\times B$, and let $\bar{f}$ be the restriction of the second projection to the image of $\pi$. Replacing $\pi$ by its Stein factorization we can assume that $\pi$ has connected fibers. Now the restriction $\pi|_{X_{b}}$ contracts precisely the curves which intersect $m\cH|_{X_{b}}=-m\pi_{b}^{*}K_{X_{b}}$ trivially, so $\Exc\pi|_{X_{b}}=D_{b}$, as needed.
\end{proof}
We remark that the morphism $\pi$ from Lemma \ref{lem:blowdown} above is a very weak simultaneous resolution of $\bar{f}$, see \cite{Tessier_simult-res,KSB}. If the surfaces $\bar{X}_{b}$ are log terminal then $\pi$ can be obtained by an $f$-MMP run for $(\cX,\cD)$.
\smallskip

We now continue with abstract definitions.  Let $f\colon (\cX,\cD)\to B$ be a smooth family of log surfaces, and let $\cS(f)=(\Gamma(f),c_2)$ be the combinatorial type of its fibers. For $b\in B$ let $\tilde{\gamma}_b\colon \Gamma(f)\to \Gamma(D_b)$ be the isomorphism of weighted graphs mapping a vertex $v$ to the vertex corresponding to $D^v_b$. Given an isomorphism of weighted graphs $\gamma\colon \Gamma\to \Gamma(f)$ write $\gamma_b\de \tilde{\gamma}_b\circ \gamma$ for $b\in B$.  We call the equivalence class $(X_b,D_b,\gamma_b)\in \cP_{+}(\cS)$ a \emph{fiber of $f$ over $b$ with respect to $\gamma$}, and we skip $\gamma$ whenever it is clear from the context.

Note that the definition of a fiber $(X_b,D_b,\gamma_b)$ depends on the choice of $\gamma$. On the other hand, the equivalence of two such fibers does not. Indeed, the fibers of $f$ over $b_1,b_2\in B$ with respect to $\gamma$ are equivalent if and only if there is an isomorphism $(X_{b_1},D_{b_{1}})\to(X_{b_2},D_{b_2})$ whose induced isomorphism of weighted graphs equals  $\gamma_{b_2}\circ\gamma_{b_{1}}^{-1}=(\tilde{\gamma}_{b_2}\circ \gamma)\circ (\tilde{\gamma}_{b_1}\circ \gamma)^{-1}=\tilde{\gamma}_{b_2}\circ\tilde{\gamma}_{b_1}^{-1}$, and the latter expression does not depend on $\gamma$.

Let $\Fib(f,\gamma)\subseteq \cP_{+}(\cS)$ be the set of equivalence classes of fibers $(X_b,D_b,\gamma_b)$, $b\in B$. We say that the family $f$ \emph{represents} a subset $\cR_{+}\subseteq \cP_{+}(\cS)$ if $\cR_{+}=\Fib(f,\gamma)$ for some isomorphism $\gamma\colon \Gamma\to \Gamma(f)$. We say that $f$ is \emph{faithful} if its fibers over distinct points of $B$ are non-equivalent.

Let $\cR$ be the image of $\cR_{+}$ by the forgetful map $\cP_{+}(\cS)\to \cP(\cS)$. Note that if $f$ represents $\cR_{+}$ then it represents $\cR$ in the sense of Definition \ref{def:moduli}\ref{item:def-family-representing}. We say that $f$ is \emph{$h^{1}$-stratified} if it is $h^{1}$-stratified as a family representing $\cR$, see Definition \ref{def:moduli-hi}\ref{item:def-h1-stratified}, and the fibers over each stratum are pairwise equivalent. In this case, the numbers $\#\cR_{+}$ and $\#\cR$ are both equal to the dimension of the base. We say that a faithful family $f$ is \emph{universal} if at every point of the base its formal germ is semiuniversal.

The following example illustrates the difference between $\cP_{+}(\cS)$ and $\cP(\cS)$; all non-trivial examples in this article will be constructed in a similar way. 

\begin{example}[Varying the fourth fiber on $\F_0$]\label{ex:4-points}
	Denote by $H_{z}\de\P^1\times \{z\}$, $V_{z}\de \{z\}\times \P^1$, $z\in \P^1$ the horizontal and vertical lines on $\P^1\times \P^1$. Put $B=\P^{1}\setminus \{0,1,\infty\}$. For any $b\in B$, the log surface $(\P^1\times \P^1,H_{0}+H_{\infty}+V_0+V_{1}+V_{\infty}+V_{b})$ has combinatorial type $\cS\de (\Gamma,4)$, where $\Gamma$ is a graph with six vertices $h_{0},h_{\infty},v_{0},v_{1},v_{\infty},v_{*}$ of weight $0$, and eight edges connecting $h_{0}$, $h_{\infty}$ with $v_0$, $v_1$, $v_{\infty}$, $v_{*}$, see Figure \ref{fig:grid}.
	\begin{figure}[htbp]\vspace{-0.5em}
		\begin{tikzpicture}[scale=2]
			\begin{scope}
				\draw (0.2,0.1) -- (0.2,0.9);
				\draw (0.6,0.1) -- (0.6,0.9);
				\draw (1,0.1) -- (1,0.9);
				\draw (1.4,0.1) -- (1.4,0.9);
				\node[right] at (0.15,0.45) {\small{$V_0$}};
				\node[right] at (0.55,0.45) {\small{$V_1$}};
				\node[right] at (0.95,0.45) {\small{$V_{\infty}$}};
				\node[right] at (1.35,0.45) {\small{$V_{\star}\! \leftarrow\! $ moves}};
				\draw (0,0.2) -- (2.2,0.2);
				\draw (0,0.8) -- (2.2,0.8);
				\node at (2.35,0.2) {\small{$H_0$}};
				\node at (2.35,0.75) {\small{$H_{\infty}$}};
				\draw[->] (0.8,0.1) -- (0.8,-0.2);
				\draw (0,-0.3) -- (2.2,-0.3);
				\filldraw (0.2,-0.3) circle (0.03); \node[above] at (0.2,-0.3) {\small{$0$}};
				\filldraw (0.6,-0.3) circle (0.03); \node[above] at (0.6,-0.3) {\small{$1$}};
				\filldraw (1,-0.3) circle (0.03); \node[above] at (1,-0.3) {\small{$\infty$}};
				\filldraw (1.4,-0.3) circle (0.03); \node[above right] at (1.3,-0.3) {\small{$\star\! \leftarrow\! $ moves}}; 
				\node at (2.35,-0.25) {\small{$\P^1$}};
			\end{scope}
			\begin{scope}[shift={(3.5,-0.25)}]
				\filldraw (0.2,0.5) circle (0.03);
				\filldraw (0.6,0.5) circle (0.03);
				\filldraw (1,0.5) circle (0.03);
				\filldraw (1.4,0.5) circle (0.03);
				\filldraw (0.8,0) circle (0.03);
				\filldraw (0.8,1) circle (0.03);
				\draw (0.8,0) -- (0.2,0.5) -- (0.8,1);
				\draw (0.8,0) -- (0.6,0.5) -- (0.8,1);
				\draw (0.8,0) -- (1,0.5) -- (0.8,1);
				\draw (0.8,0) -- (1.4,0.5) -- (0.8,1);
				\node at (1,0) {\small{$h_0$}};
				\node at (1,1) {\small{$h_{\infty}$}};
				\node at (0.1,0.4) {\small{$v_0$}};
				\node at (0.5,0.4) {\small{$v_1$}};
				\node at (1.1,0.6) {\small{$v_{\infty}$}};
				\node at (1.5,0.6) {\small{$v_{\star}$}};
				\node at (1.4,0.1) {\small{$\Gamma$}};
			\end{scope}
		\end{tikzpicture}
	\vspace{-0.5em}
		\caption{Example \ref{ex:4-points}.}
		\label{fig:grid}
	\end{figure}
	\vspace{-0.5em}
	
	Put $\cX=B\times \P^{1}\times \P^{1}$ and $\cD=\cH_{0}+\cH_{\infty}+\cV_0+\cV_1+\cV_{\infty}+\cV_{\star}$, where $\cH=B\times H_{z}$, $\cV_{z}=B\times V_{z}$, and $\cV_{\star}=\{(z,z)\in B\times \P^{1}\}\times \P^1$. The projection $f\colon (\cX,\cD)\to B$ is a faithful family representing $\cP_{+}(\cS)$. Indeed, define $\gamma\colon \Gamma\to \Gamma(f)$ so that $\gamma(v_z)$, $\gamma(h_z)$ correspond to $\cV_z$, $\cH_z$. Fix $(X,D,\gamma)\in \cP_{+}(\cS)$. Let $V_z=D^{v_z}$, $z\in \{0,1,\infty,\star\}$, $H_w=D^{h_w}$, $w\in \{0,\infty\}$. Then $|V_0|$ and $|H_0|$ induce two $\P^1$-fibrations of $X$, so $X$ is a rational surface of rank $c_{2}(X)-2=2$, and since $X$ admits two $\P^1$-fibrations, we have $X\cong \P^1\times \P^1$. Let $\pr_1,\pr_2$ be the two projections $X\to \P^1$. We can choose coordinates on each factor of $X$  such that $\pr_1(V_z)=z$ for $z\in \{0,1,\infty\}$ and  $\pr_2(H_w)=w$ for $w\in \{0,\infty\}$. This way,  $(X,D,\gamma)$ gets identified with the fiber of $f$ over $\pr_1(V_{\star})\in B$.

	We will see in Example \ref{ex:4-points-Aut} that $f$, viewed as a family representing $\cP(\cS)$, is almost faithful, see Definition \ref{def:moduli}\ref{item:def-family-faithful}. In fact, two fibers $(X_b,D_b)$ and $(X_{b'},D_{b'})$ are isomorphic if and only if $b$ and $b'$ lie in the same orbit of $\Aut(B)\cong S_3$. As a consequence, the restriction of $f$ over $B\setminus \{b\}$ for any $b\in B$ still represents $\cP(\cS)$ in the sense of Definition \ref{def:moduli}\ref{item:def-family-representing}: indeed, the missing fiber $(X_b,D_b)$ is isomorphic to $(X_{b'},D_{b'})$ for any $b'\neq b$ in the same $S_3$-orbit; such $b'$ exists since the $S_3$-action on $B$ has no fixed point. However, since $f$ is faithful, the fiber $(X_b,D_b,\gamma_b)$ is not equivalent to $(X_{b'},D_{b'},\gamma_{b'})$ for any $b'\neq b$, so the restriction of $f$ over $B\setminus \{b\}$ does not represent $\cP_{+}(\cS)$ in the sense of definitions above.
	
	Eventually, we note that $f$ is universal. By Lemma \ref{lem:h1}\ref{item:h1-Fm} for each $b\in B$ we have  $h^{1}(\lts{X_b}{D_{b}})=1=\dim B$, so it is enough to prove that $f$ is formally versal at $b$. Let $\check{f}\colon (\check{\cX},\check{\cD})\to T$ be a formal deformation of $(X_b,D_b)$. Write  $\check{\cD}=\check{\cH}_{0}+\check{\cH}_{\infty}+\check{\cV}_{0}+\check{\cV}_{1}+\check{\cV}_{\infty}+\check{\cV}_{\star}$ so that $\check{\cH}_{w}|_{X_{b}}=H_{w}$ for $w\in \{0,\infty\}$ and $\check{\cV}_{z}|_{X_{b}}=V_z$ for $z\in \{0,1,\infty\}$. Then $\check{\cV}_{\star}|_{X_{b}}=V_b$. 
	Since $\P^1\times \P^1$ is rigid, cf.\ \cite[Exercise III.9.10]{Hartshorne_AG}, we have an isomorphism $\psi\colon \check{\cX}\to T\times \P^1\times \P^1=\P^1_{T}\times_{T}\P^1_{T}$. As before, we choose coordinates on each factor $\P^1_{T}$ so that $\pr_1(\psi(\check{\cV}_z))=z$ for $z=0,1,\infty$ and $\pr_2(\psi(\check{\cH}_w))=w$ for $w= 0,\infty$: note that now we change coordinates over $T$, i.e.\ using matrices whose entries are regular functions on $T$. The point  $\pr_1(\psi(\cV_{\star}))\in \P^1_{T}\setminus \{0,1,\infty\}$ is given by a smooth morphism $\alpha\colon T\to \P^1\setminus \{0,1,\infty\}=B$, 
	and $\check{f}$ is the pullback of $f$ through $\alpha$, as claimed. 
\end{example}

The notions of an inner and outer blowup, reviewed in Section \ref{sec:hi}, extend naturally to combinatorial types. More precisely, fix a combinatorial type $\cS=(\Gamma,c_2)$ and a vertex $v$ (respectively, an edge $e$ joining $v_1$ and $v_2$) of $\Gamma$. We define an \emph{outer} (respectively, \emph{inner}) blowup $\cS'\to \cS$ \emph{at $v$} (respectively, \emph{at $e$}) by putting $\cS'=(\Gamma',c_2+1)$, where $\Gamma'$ is obtained from $\Gamma$ by reducing the weight of $v$ (respectively, of $v_1$ and $v_2$) by $1$, adding a new vertex $w$ of weight $-1$, and adding an edge between $w$ and $v$ (respectively, replacing $e$ by edges joining $w$ with $v_1$ and $v_2$). 
We call the inverse operation an \emph{outer} (respectively, \emph{inner}) \emph{blowdown} of the vertex $w$. 

A blowup $\beta\colon \cS'\to \cS$ as above yields a surjection $\beta_{+}\colon \cP_{+}(\cS')\to \cP_{+}(\cS)$ given by $\beta_{+}(X',D',\gamma')=(X,D,\gamma)$, where $(X,D)\in \cP(\cS)$ is a log surface obtained from $(X',D')$ by contracting the $(-1)$-curve $(D')^{w}$, and the isomorphism $\gamma \colon \Gamma\to \Gamma(D)$ is given (on vertices) by the restriction of $\gamma'$. If $\beta$ is an inner blowup at an edge $e$ then $\beta_{+}$ is bijective: indeed, $\beta_{+}^{-1}$ associates to each $(X,D,\gamma)\in \cP_{+}(\cS)$ the log surface $(X',D')\in \cP(\cS')$ obtained by blowing up the point $D^{e}$; together with the unique isomorphism $\gamma'$ extending $\gamma$. If $\beta$ is an outer blowup at a vertex $v$, then the elements of the fiber $\beta_{+}^{-1}(X,D,\gamma)$ are in one-to-one correspondence with the orbits of the action  on $D^{v}\setminus (D-D^{v})$ of the group of those automorphisms of $X$ which fix $D$ componentwise. Indeed, fix $(X_{i}',D_{i}',\gamma_{i}')\in \beta_{+}^{-1}(X,D,\gamma)$, $i\in \{1,2\}$, let $\beta_{i}\colon (X_i,D_i)\to (X,D)$ be the contraction of $(D'_{i})^w$ and let $p_i=\beta((D'_{i})^w)\in D^{v}\setminus (D-D^{v})$. Since $\gamma_i'$ restricts to $\gamma$, we have $\beta_{i}((D_i')^{q})=D^{q}$ for every vertex $q$ of $\Gamma$. If $(X'_{i},D'_{i},\gamma_i)$ for $i\in\{1,2\}$ are equivalent then there is an isomorphism $\alpha'\colon (X_1',D_1')\to (X_2',D_2')$ such that $\alpha'((D_{1}')^{w})=(D_{2}')^{w}$ and $\alpha'((D_{1}')^{q})=(D_{2}')^{q}$ for each $q$, so $\alpha'$ descends to $\alpha\in \Aut(X,D)$ such that $\alpha(p_1)=p_2$ and $\alpha(D^{q})=D^q$ for each $q$. Conversely, any such $\alpha$ lifts to an equivalence between $(X_i,D_i,\gamma_i')$, as claimed.
\smallskip

Let $G$ be a subgroup of $\Aut(\cS)$. A \emph{$G$-equivariant blowup} at an edge $e$ (respectively, at a vertex $v$) is a composition $\cS'\to \cS$ of blowups at $ge$ (respectively, at $gv$) for all $g\in G$. Note that each element of $G$ uniquely extends to an automorphism of the new graph, so we can identify $G$ with a subgroup of $\Aut(\cS')$ and unambiguously speak about $G$-equivariant morphisms or birational maps.

 Given two elements $(X_{i},D_{i},\gamma_{i})\in \cP_{+}(\cS)$ and an isomorphism of log surfaces $\phi\colon (X_1,D_1)\to (X_2,D_2)$ we put $\phi_{+}\de \gamma_{2}^{-1}\circ \Gamma(\phi)\circ \gamma_{1}\in \Aut(\cS)$, where $\Gamma(\phi)\colon \Gamma(D_{1})\to \Gamma(D_2)$ is the induced isomorphism of weighted graphs, see Section \ref{sec:log_surfaces}. Then $\phi(D_{1}^{v})=D_{2}^{\phi_{+}(v)}$ for each vertex $v$.
 
 Fix $\cR_{+}\subseteq \cP_{+}(\cS)$, and let $f\colon (\cX,\cD)\to B$ be a family representing $\cR_{+}$, i.e.\ $\cR_{+}=\Fib(f,\gamma)$ for some isomorphism $\gamma\colon \Gamma\to \Gamma(f)$. We denote by $\Aut(f)\leq \Aut(\cX,\cD)$ the  group of automorphisms mapping fibers of $f$ to fibers of $f$. Any $\alpha\in \Aut(f)$ induces an automorphism of $B$, denoted by the same letter; and an element $\alpha_{+}\in \Aut(\cS)$ given by  $\alpha(\cD^{\gamma(v)})=\cD^{\gamma(\alpha_{+}(v))}$ for each vertex $v$ of $\Gamma$. Note that for every $b\in B$, $\alpha_{+}$ is equal to the element $(\alpha|_{X_{b}})_{+}$ of $\Aut(\cS)$ associated to the isomorphism $\alpha|_{X_{b}}\colon (X_b,D_b)\to (X_{\alpha(b)},D_{\alpha(b)})$.

We say that $f$ is \emph{$G$-faithful} if there is a finite subgroup $G_{f}\leq \Aut(f)$ whose image in $\Aut(\cS)$ equals $G$, such that for every isomorphism $\phi\colon (X_{b_{1}},D_{b_{1}})\to (X_{b_{2}},D_{b_{2}})$ with $\phi_{+}\in G$ there is $\alpha\in G_f$ such that $\alpha_{+}=\phi_{+}$ and $\alpha(b_1)=b_2$. As before, a $G$-faithful family is \emph{universal} if it is formally semiuniversal at each $b\in B$.

If $\cR$ is the image of $\cR_{+}$ by the forgetful map $\cP_{+}(\cS)\to \cP(\cS)$, then an  $\Aut(\cS)$-faithful family representing $\cR_{+}$ is an almost faithful family representing $\cR$ in the sense of Definition \ref{def:moduli}\ref{item:def-family-faithful}.

\begin{example}[Varying the fourth fiber on $\F_0$, continued]\label{ex:4-points-Aut}
	As in Example \ref{ex:4-points}, let $\cS$ be the combinatorial type of $\P^1\times \P^1$ with two horizontal and four vertical lines. Then $\Aut(\cS)=\Z/2\times S_4$. Let $f\colon (\cX,\cD)\to B$ be the faithful family representing $\cP_{+}(\cS)$ constructed in  Example \ref{ex:4-points}. We claim that $f$ is $\Aut(\cS)$-faithful, and therefore $f$ is an almost faithful as a family representing $\cP(\cS)$ in the sense of   Definition \ref{def:moduli}\ref{item:def-family-faithful}.
	
	To see this, we define actions of $\Z/2$ and $S_4$ on $B=\P^1\setminus \{0,1,\infty\}$ and $\cX=B\times \P^1\times \P^1$, as follows. For a generator $\tau$ of $\Z/2$, we put $\tau(b)=b$ for $b\in B$ and $\tau(b,x,y)=(b,x,y^{-1})$ for $(b,x,y)\in \cX$. To define the $S_{4}$-action, we use the following notation. Put $x_{1}=\infty$, $x_{2}=0$, $x_{3}=1$, and for $b\in B$ put $x_{i}^{b}=x_{i}$, $i\in \{1,2,3\}$ and $x_{4}^{b}=b$, so $b$ is the cross-ratio of $(x_1^b,x_2^b,x_3^b,x_4^b)$. Now for $\sigma\in S_{4}$, $b\in B$ define $\sigma_{b}\in \Aut(\P^1)$ by $\sigma_{b}(x_{\sigma^{-1}(i)}^{b})=x_{i}$, $i\in \{1,2,3\}$, and let $\sigma(b)\de \sigma_{b}(x_{\sigma^{-1}(4)}^{b})$. Then $\sigma_{b}(x_{i}^{b})=x_{\sigma(i)}^{\sigma(b)}$ for all $i\in \{1,2,3,4\}$, so $\sigma(b)$ is the cross-ratio of $(x_{\sigma(1)}^{\sigma(b)},x_{\sigma(2)}^{\sigma(b)},x_{\sigma(3)}^{\sigma(b)},x_{\sigma(4)}^{\sigma(b)})$. Hence the formula $\sigma(b,x,y)=(\sigma(b),\sigma_{b}(x),y)$ defines the required element of $\Aut(f)$.
	
	For an isomorphism $\phi\colon (X_{b_{1}},D_{b_{1}})\to (X_{b_{2}},D_{b_{2}})$, the induced element $\phi_{+}\in \Aut(\cS)$ is given by the permutation $\tau$ of $\{0,\infty\}$ and $\sigma$ of $\{0,1,\infty,\star\}$ such that $\phi(H_{z})=H_{\tau(z)}$ and $\phi(\cV_{z}|_{X_{b_1}})=\cV_{\sigma(z)}|_{X_{b_2}}$. In this case, the cross-ratios of $(x_{1}^{b_1},x_{2}^{b_1},x_{3}^{b_1},x_{4}^{b_1})$ and $(x_{\sigma(1)}^{b_2},x_{\sigma(2)}^{b_2},x_{\sigma(3)}^{b_2},x_{\sigma(4)}^{b_2})$ are equal, so $b_2=\sigma(b_1)$, as needed. 
\end{example}

We now explain how to lift a family representing $\cR_{+}\subseteq \cP_{+}(\cS)$ by inner and outer blowups $\cS'\to \cS$, thus making Observation \ref{obs:blowups} precise. For most applications we will take $\cR_{+}=\cP_{+}(\cS)$, so $\cR_{+}'=\cP_{+}(\cS')$.

\begin{lemma}[Inner blowups]\label{lem:inner}
	Let $\cS$ be a combinatorial type, let $G$ be a subgroup of $\Aut(\cS)$, and let $\beta\colon \cS'\to \cS$ be a $G$-equivariant inner blowup. 
	Fix $\cR_{+}\subseteq \cP_{+}(\cS)$ and let $\cR_{+}'=\beta_{+}^{-1}(\cR_{+})$. Then $\#\cR_{+}=\#\cR_{+}'$, and $\cR_{+}$ is represented by a ($G$-faithful, universal, or ($h^{1}$)-stratified) family if and only if the same holds for $\cR_{+}'$. Moreover, both families can be chosen having the same base, and the same stratification in the stratified case. 
\end{lemma}
\begin{proof}
	Assume first that $\cR_{+}$ is represented by a family $f\colon (\cX,\cD)\to B$, i.e.\ $\cR_{+}=\Fib(f,\gamma)$ for some $\gamma$. Let $\pi\colon\cX'\to \cX$ be a blowup at $\cD^{\gamma(e)}$ for each blown up edge $e$, $\cD'\de (\pi^{*}\cD)\redd$, let $\gamma'$ be the unique isomorphism extending $\gamma$, and let $f'\de f\circ \pi\colon (\cX,\cD)\to B$. Then $\cR_{+}'=\Fib(f',\gamma')$, so $f'$ is a family representing $\cR_{+}'$, and if $f$ is $G$-faithful then so is $f'$. Conversely, let $f'\colon (\cX',\cD')\to B$ be a ($G$-faithful) family representing $\cR_{+}'$, so $\cR_{+}'=\Fib(f',\gamma')$, and let $\cE\de \sum_{v}(\cD')^{\gamma'(v)}$, where the sum runs over the new vertices. Then $\cE$ is a disjoint union of families of $(-1)$-curves, so it can be contracted by a composition of $K_{\cX/B}$-extremal contractions $\pi\colon (\cX',\cD')\to (\cX,\cD)$. 	The induced morphism $f\colon (\cX,\cD)\to B$ is a ($G$-faithful) family representing $\cR_{+}$, more precisely, $\cR_{+}=\Fib(f,\gamma)$, where $\gamma$ is the restriction of $\gamma'$. 
	
	Similarly, every small deformation of $(X_b,D_b)$ lifts to one of $(X_b',D_b')$, and conversely, every small deformation of $(X_b',D_b')$ blows down to one of $(X_b,D_b)$, cf.\ \cite[Proposition 1.6(1)]{FZ-deformations}. This way, we get a morphism of functors from deformations of $(X_{b'},D_{b'})$ to those of $(X_b,D_b)$, mapping $f'$ to $f$. Since $h^{2}(\lts{X_{b'}},D_{b'})=h^{2}(\lts{X_b}{D_b})$ by Lemma \ref{lem:blowup-hi}\ref{item:blowup-h2}, this morphism is smooth, see  \cite[Proposition 2.3.6]{Sernesi_deformations}. It follows that  $f'$ is versal at $b\in B$ if and only if so is $f$. By Lemma \ref{lem:blowup-hi}\ref{item:blowup-hi-inner} we have $h^{1}(\lts{X'_b}{D'_b})=h^{1}(\lts{X_b}{D_b})$, so $f'$ is $h^{1}$-stratified (or universal) if and only if so is $f$, with the same stratification of $B$.
\end{proof}

\begin{lemma}[Outer blowups]\label{lem:outer}
	Let $\cS$ be a combinatorial type, and let $\beta\colon \cS'\to \cS$ be an outer blowup at a vertex~$a$. Let $\cR_{+}\subseteq \cP_{+}(\cS)$ and let $\cR'_{+}=\beta_{+}^{-1}(\cR_{+})$.  Assume that $\cR_{+}$ is represented by a family $f\colon (\cX,\cD)\to B$. 
	
	Let $\cA\de \cD^{a}$ be the component of $\cD$ corresponding to the blown up vertex, $\cA^{\circ}\de \cA\setminus (\cD-\cA)$. For each $b\in B$ let $A^{\circ}_{b}=\cA^{\circ}|_{X_b}$ and let $H_b\leq \Aut(A_{b}^{\circ})$ be the restriction of $\Aut(X_b,D_b,A_b)$. 
	Then the following hold.
\begin{parts}
	\item\label{item:outer-trivial} Assume that $f$ is $G$-faithful for some group $G$ fixing the blown-up vertex $a$. Let $G_{f}$ be the corresponding subgroup of $\Aut(f)$. Assume that each group $H_b$ is the restriction of some $H\leq  \Aut(f)$ such that the subgroup $G'_{f}\leq \Aut(f)$ generated by $G_f$ and $H$ is finite. Let $G'$ be the image of $G'_{f}$ in $\Aut(\cS')$. Then $\cR_{+}'$ is represented by an $G'$-faithful family $f'\colon (\cX',\cD')\to\cA^{\circ}$. Moreover, if $f$ is universal then $f'$ is universal, too. 
	\item\label{item:outer-open} Assume that $f$ is stratified by $B=\bigsqcup_{j=0}^{k}B_j$, and write $B_k=\{\bar{b}\}$. 
	Then the following hold.
	\begin{parts}
		\item\label{item:outer-fixed-point} Assume that for each $b\neq \bar{b}$ the group $H_{b}$ acts transitively on $A^{\circ}_{b}$; whereas $H_{\bar{b}}$ has a fixed point $\bar{z}\in A^{\circ}_{\bar{b}}$ and an open orbit $A^{\circ}_{\bar{b}}\setminus \{\bar{z}\}$. Then $\cR_{+}'$ is represented by family $(\cX',\cD')\to\cA^{\circ}$, stratified by $\cA^{\circ}=\bigsqcup_{j=0}^{k+1}\cA^{\circ}_{j}$, where $\cA^{\circ}_{k+1}=\{\bar{z}\}$, $\cA^{\circ}_{k}=A^{\circ}_{\bar{b}}\setminus \{\bar{z}\}$, and $\cA^{\circ}_{j}=f^{-1}(B_j)\cap \cA^{\circ}$ for $j<k$, see Figure \ref{fig:outer}. 
		\item\label{item:outer-transitive} Assume that for every $b\in B$ the group $H_{b}$ acts transitively on $A^{\circ}_{b}$. Then $\cR_{+}'$ is represented by stratified family $(\cX',\cD')\to \tilde{B}$ over an open neighborhood of $\bar{b}\in B$, with the stratification restricted from $B$. 
		\item \label{item:outer-h1}
		Put $h_{j}\de h^{1}(\lts{X_{b}}{D_{b}})$ for $b\in B_j$. Assume that the conditions of \ref{item:outer-fixed-point} or \ref{item:outer-transitive} hold, and put $h_{j}'\de h^{1}(\lts{X_{z}'}{D_{z}'})$ for  $z\in \cA_{j}^{\circ}$ in case \ref{item:outer-fixed-point} and for $z\in \tilde{B}_j$ in case \ref{item:outer-transitive}. Then the following hold.
		\begin{parts}
			\item\label{item:outer-h1-stays} If for some $b\in B_j$, $z\in A_{b}^{\circ}$ the sheaf $\lts{X_{b}}{D_b}$ has a global section which does not vanish at $z$ then  $h_{j}'=h_{j}$. This holds for instance if the derivative at $g=\id$ of the morphism $H_{b}\ni g \mapsto g(z)\in X_{b}$ (which we call simply \emph{the derivative of the $H_b$-action}), is nonzero.
			\item\label{item:outer-h1-stratified} Assume that $f$ is $h^1$-stratified, and for all $j\in \{0,\dots, k\}$ we have $h_j=h_j'$. In case \ref{item:outer-transitive}, assume furthermore that the group $\Aut(f)$ acts transitively on $A_{b}^{\circ}$ for each $b\in B$. Then $f'$ is $h^{1}$-stratified.
			\item\label{item:outer-h1-grows} If for some  $b\in B_j$ all sections of $\lts{X_b}{D_b}$ vanish along $A_{b}$ then $h_{j}'=h_{j}+1$.
		\end{parts}
	\end{parts} 
	\end{parts}
\end{lemma}
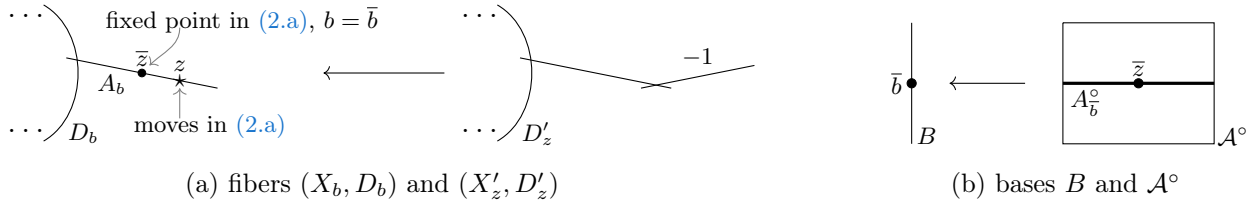
\begin{figure}[htbp]
	\subcaptionbox{fibers $(X_{b},D_{b})$ and $(X_{z}',D_{z}')$}{
	\begin{tikzpicture}
		\begin{scope}
			\draw (-0.5,0.9) to[out=-30,in=30] (-0.5,-0.9); 
			\node[below left] at (-0.3,0.9) {\Large{$\dots$}};
			\node[above left] at (-0.3,-0.9) {\Large{$\dots$}};
			\node[right] at (-0.3,-0.8) {\small{$D_{b}$}};
			\draw (-0.2,0.2) -- (1.8,-0.2);
			\filldraw (0.8,0) circle (0.05); \node at (0.8,0.2) {\small{$\bar{z}$}};
			\draw[<-, black!50] (1.3,-0.2) to[out=-90,in=90] (1.3,-0.6);  
			\node at (1.3,-0.1) {$\star$}; 
			\node at (1.3,0.1) {\small{$z$}};
			\node at (1.7,-0.7) {\small{moves in \ref{item:outer-fixed-point}}};
			\draw[<-, black!50] (0.9,0.1) to[out=45,in=-90] (1.3,0.6); 
			\node[right] at (0.2,0.7) {\small{fixed point in \ref{item:outer-fixed-point}, $b=\bar{b}$}}; 
			\node at (0.4,-0.15) {\small{$A_b$}};
			\draw[<-] (3.2,0) -- (4.8,0);
		\end{scope}
		\begin{scope}[shift={(6,0)}]
			\draw (-0.5,0.9) to[out=-30,in=30] (-0.5,-0.9); 
			\node[below left] at (-0.3,0.9) {\Large{$\dots$}};
			\node[above left] at (-0.3,-0.9) {\Large{$\dots$}};
			\node[right] at (-0.3,-0.8) {\small{$D_{z}'$}};
			\draw (-0.2,0.2) -- (1.8,-0.2);
			\draw (1.4,-0.2) -- (2.9,0.1);
			\node at (2.15,0.2) {\small{$-1$}};
		\end{scope}
	\end{tikzpicture}
	}
\hfill
	\subcaptionbox{bases $B$ and $\cA^{\circ}$}{
	\begin{tikzpicture}
		\draw (0,0.8) -- (0,-0.8); \node at (0.2,-0.7) {\small{$B$}};
		\filldraw (0,0) circle (0.06); \node at (-0.2,0) {\small{$\bar{b}$}};
		\draw[<-] (0.5,0) -- (1.5,0);
		\draw (2,0.8) -- (4,0.8) -- (4,-0.8) -- (2,-0.8) -- (2,0.8); \node at (4.25,-0.7) {\small{$\cA^{\circ}$}};
		\draw[very thick] (2,0) -- (4,0); \node at (2.3,-0.25) {\small{$A_{\bar{b}}^{\circ}$}};
		\filldraw (3,0) circle (0.06); \node at (3,0.2) {\small{$\bar{z}$}};
	\end{tikzpicture}
	}
\vspace{-0.5em}
	\caption{Lemma \ref{lem:outer}\ref{item:outer-trivial},\ref{item:outer-fixed-point}: outer blowup increases the dimension of a representing family.} \vspace{-0.5em}
	\label{fig:outer}
\end{figure}
\begin{proof}
	Write $\cS=(\Gamma,c)$. We have $\cR_{+}=\Fib(f,\gamma)$ for some isomorphism $\gamma\colon \Gamma\to \Gamma(f)$. Consider cases \ref{item:outer-trivial} and \ref{item:outer-fixed-point}. The pullback $\hat{f}\colon (\cX,\cD)\times_{B} \cA^{\circ}\to B$ of $f$ along $f|_{\cA^{\circ}}\colon \cA^{\circ}\to B$ is another family representing $\cR_{+}$, in fact, $\cR_{+}=\Fib(\hat{f},\hat{\gamma})$, where $(\cD\times_{B}\cA^{\circ})^{\hat{\gamma}(v)}=\cD^{\gamma(v)}\times_{B}\cA^{\circ}$. Let $\pi\colon \cX'\to \cX \times_{B}\cA^{\circ}$ be a blowup of the diagonal, and let $\cD'=\pi^{*}(\cD\times_{B}\cA^{\circ})\redd$. Then $f'\de \hat{f}\circ \pi\colon(\cX',\cD')\to \cA^{\circ}$ satisfies $\cR_{+}=\Fib(f',\gamma')$ for the unique extension $\gamma'$ of $\hat{\gamma}$, so $f'$ represents $\cR_{+}'$. We will show that it has properties required by \ref{item:outer-trivial} and \ref{item:outer-fixed-point}. 
	
	Consider case \ref{item:outer-trivial}. We claim that $f'$ is $G'$-faithful. 
	Suppose that for some $z_{1},z_{2}\in \cA^{\circ}$ we have an isomorphism $\alpha'\colon (X'_{z_1},D'_{z_1})\to (X_{z_{2}}',D_{z_{2}}')$ such that $\alpha'_{+}\in G'$. Let $b_i=f(z_i)\in B$, $i\in \{1,2\}$. Since $\alpha'_{+}\in G'$, $\alpha'_{+}$ fixes $a$ and the new vertex, so by the universal property of blowing up there is an isomorphism $\alpha\colon (X_{b_1},D_{b_{1}})\to (X_{b_2},D_{b_{2}})$ between the corresponding fibers of $f$ such that $\alpha(z_1)=z_2$. Since $f$ is $G$-faithful, we have $b_2=g(b_1)$ for some $g\in G_{f}$ such that $g_{+}=\alpha_{+}$. Since $g^{-1}\circ\alpha\in \Aut(X_{b_{1}})$ restricts to an element of $H_{b_{1}}$, by assumption $g^{-1}\circ \alpha=h|_{X_{b}}$ for some $h\in H$. We have $g\circ h|_{X_{b}}=\alpha$, so $g\circ h(z_1)=\alpha(z_1)=z_2$ and $(g\circ h)_{+}=\alpha_{+}$, so $g\circ h$ lifts to the required element of $\Aut(f')$. Thus $f'$ is $G'$-faithful. The assertion about universality will be proved below.
\smallskip

	Consider case \ref{item:outer-fixed-point}. 
	Fix $j\in \{0,\dots, k\}$. For any two points $z_1,z_2\in \cA^{\circ}_j$ the fibers $(X_{z_i}',D_{z_i}',\gamma_{z_i}')$ for $i\in \{1,2\}$ are equivalent. Indeed, the fibers of $f$ over $f(z_i)$ are equivalent, say by an isomorphism $\alpha\colon (X_{f(z_1)},D_{f(z_1)})\to (X_{f(z_2)},D_{f(z_2)})$ with $\alpha_{+}=\id$. Composing $\alpha$ with an element of the connected component of $H_{f(z_2)}$ we can assume $\alpha(z_1)=z_2$, so $\alpha$ lifts to an isomorphism $\alpha'\colon (X_{z_1}',D_{z_1}')\to (X_{z_2}',D_{z_2}')$ with $\alpha'_{+}=\id$, which yields the required equivalence. Conversely, if for some $z\in \cA_{j}^{\circ}$, $z'\in \cA_{j'}^{\circ}$ the fibers $(X'_{z},D'_{z})$ and $(X'_{z'},D'_{z'})$ are isomorphic then $j=j'$. Indeed, by the universal property of blowing up such an isomorphism descends to an isomorphism between $(X_{f(z)},D_{f(z)})$ and $(X_{f(z')},D_{f(z')})$, so if $j\neq j'$ then $\{j,j'\}=\{k,k+1\}$ and we get an automorphism of $(X_{\bar{b}},D_{\bar{b}})$ which maps $\bar{z}$ to a point of $A_{\bar{b}}\setminus \{\bar{z}\}$, contrary to our assumptions. 
\smallskip

	Consider now case \ref{item:outer-transitive}. Shrinking $B$ around $\bar{b}$ if needed, we can assume that $f|_{\cA^{\circ}}$ has a section $\cH\subseteq \cA^{\circ}$. Blowing up $\cH$ we obtain a family $f'\colon (\cX',\cD')\to B$ representing $\cR_{+}'$. For every two points $b_{1},b_{2}\in B_{j}$ the fibers $(X_{b_i},D_{b_i},\gamma_{b_{i}})$ are equivalent, and there is an isomorphism $(X_{b_{1}},D_{b_1})\to (X_{b_2},D_{b_2})$ mapping the point $D_{b_{1}}\cap \cH$ to $D_{b_2}\cap \cH$, so the fibers $(X_{b_i}',D_{b_i}',\gamma_{b_{i}}')$ are equivalent, too. Conversely, if $(X'_b,D'_b)$ and $(X'_{b'},D_{b'})$ are isomorphic then as before we get $b,b'\in B_j$ for some $j$. 
\smallskip
	
	Parts \ref{item:outer-h1-stays} and \ref{item:outer-h1-grows} follow from Lemma \ref{lem:blowup-hi}\ref{item:blowup-hi-outer}. It remains to prove part \ref{item:outer-h1-stratified} and the assertion about universality in \ref{item:outer-trivial}. In case \ref{item:outer-trivial} fix any $z\in \cA^{\circ}$ and let $b=f(z)$, in case \ref{item:outer-fixed-point} put $z\de \bar{z}$, $b\de \bar{b}$, and in case \ref{item:outer-transitive} put $z=b=\bar{b}$. Then by assumption $f$ is semiuniversal at $b$, in particular $h^{1}(\lts{X_b}{D_b})=\dim B$.
	
	We claim that $f'$ is versal at $z$. We apply the argument from \cite[Proposition 1.6(2)]{FZ-deformations}, which can be outlined as follows. Consider a small deformation $\check{f}'\colon (\check{\cX}',\check{\cD}')\to T$ of $(X'_{z},D'_{z})$. Let $\check{\pi}\colon (\check{\cX}',\check{\cD}')\to (\check{\cX},\check{\cD})$ be the contraction of the component $\check{\cE}$ of $\check{\cD}'$ corresponding to the new vertex: it exists since $\check{\cE}$ is a family of $(-1)$-curves, so it spans a $K_{\check{\cX}'/T}$-negative extremal ray. We get a small deformation $\check{f}\colon (\check{\cX},\check{\cD})\to T$ of $(X_{b},D_{b})$. Since $f$ is versal at $b$, $\check{f}$ is a pullback of $f$. Let $\psi\colon \check{\cX}\to \cX$ be the induced morphism. For $t\in T$ let  $p_{t}\de \check{\pi}(\check{E}_{t})\in \check{D}_{t}$ be the center of the blowup $\check{\pi}|_{\check{X}_{t}'}$. Now $\check{f}'$ is a pullback of $f'$ along the morphism $\bar{\psi}\colon T\ni t\mapsto \psi(p_t)\in \cA^{\circ}$. This proves the claim in cases \ref{item:outer-trivial} and \ref{item:outer-fixed-point}. In case \ref{item:outer-transitive}, by assumption we can compose $\bar{\psi}$ with an element of $\Aut(f)$ so that its image is the section $\cH$, and then $\check{f}'$ is a pullback of $f'$ along the composition $f\circ \bar{\psi}$, as needed.
	
	To conclude, by \cite[Proposition 2.5.8]{Sernesi_deformations} it is enough to prove that the dimension of the base of $f'$ equals $h^{1}(\lts{X'_{z}}{D'_{z}})$. In case \ref{item:outer-transitive} it follows from our assumptions. In the remaining cases, the base of $f'$ is $\cA^{\circ}$, and by Lemma \ref{lem:blowup-hi}\ref{item:blowup-hi-outer} we have  $h^{1}(\lts{X'_z}{D'_z})\leq \dim \cA^{\circ}$. Suppose the inequality is strict. Let $\hat{f}'$ be the formal deformation of $(X_{z}',D_{z}')$ associated with $f'$, it is a pullback of a semiuniversal formal deformation by some morphism $\hat{\cA}^{\circ}\to \hat{U}$. Since $\dim \hat{\cA}^{\circ}>\dim \hat{U}$, this morphism contracts a curve $\hat{C}\subseteq \hat{\cA}^{\circ}$, so $\hat{f}'$ restricts to a trivial formal deformation along this curve. By Artin approximation theorem \cite[Theorem 1.7]{Artin_algebraization_I}, the restriction of $f'$ along some curve through $z$ becomes trivial after an \'etale base change. This is impossible because the central fiber $(X_{z}',D'_{z})$ is not isomorphic to $(X'_{z'},D'_{z'})$ for any $z'\neq z$ near $z$, a contradiction.
\end{proof}

In the proof of  Propositions \ref{prop:moduli} and \ref{prop:moduli-hi} we will deduce properties of certain subsets of $\cP(\cS)$ from the properties of $\cP_{+}(\check{\cS})$, where $\check{\cS}$ is a combinatorial type obtained from $\cS$ by adding vertices corresponding to some vertical $(-1)$-curves. This will be done using the following elementary observations. 

\begin{lemma}[Adding $(-1)$-curves]\label{lem:adding-1}
	Let $\check{\cS}=(\check{\Gamma},c_2)$ be a combinatorial type, let $\Gamma_E$ be a subgraph of $\check{\Gamma}$ whose all vertices have weight $-1$, and let $\cS=(\check{\Gamma}-\Gamma_E,c_2)$. Fix $\check{\cR}_{+}\subseteq \cP_{+}(\check{\cS})$,  let $\cR_{+}$ be the image of $\check{\cR}_{+}$ by the restriction map $\cP_{+}(\check{\cS})\to \cP_{+}(\cS)$; and let $\cR$ be the image of $\cR_{+}$ by the forgetful map $\cP_{+}(\cS)\to \cP(\cS)$.
	
	Let $\check{f}\colon (\cX,\check{\cD})\to B$ be a smooth family representing $\cR_{+}$.  Let $\cE\subseteq \check{\cD}$ be a subdivisor corresponding to $\Gamma_{E}$. Write $\cD=\check{\cD}-\cE$ and let $f\colon (\cX,\cD)\to B$ be the restriction of $\check{f}$. Then the following hold. 
	\begin{enumerate}
		\item\label{item:adding-1-semiuniversal} The formal germ of $\check{f}$ at $b\in B$ is semiuniversal if and only if so is the germ of $f$. 
		\item\label{item:adding-1-h1} If $\check{f}$ is an ($h^{1}$-)stratified family representing $\check{\cR}_{+}$ then $f$ is an ($h^{1}$-)stratified family representing $\cR$, with the same stratification. In particular, $\#\cR=\#\check{\cR}_{+}=\dim B+1$, see Definition \ref{def:moduli}\ref{item:def-family-stratified}. 
		\item\label{item:adding-1-faithful} Assume that $\check{f}$ is an $\Aut(\check{\cS})$-faithful family representing $\check{\cR}_{+}$ such that 
		\begin{equation}\label{eq:E-invariant}
				\mbox{for every }b_{1},b_{2}\in B,\mbox{ every isomorphism }(X_{b_1},D_{b_{1}})\to (X_{b_{2}},D_{b_{2}})\mbox{ maps }E_{b_{1}}\mbox{ to }E_{b_{2}}.
		\end{equation}
	 	Then 
	 	$f$ is an almost faithful family representing $\cR$. Moreover, if $\check{f}$ is a universal family representing $\check{\cR}_{+}$ then $f$ is an almost universal family representing $\cR$, see Definition \ref{def:moduli-hi}\ref{item:def-universal}.
\end{enumerate}
\end{lemma}
\begin{proof}
	By \cite[Proposition 1.7(3)]{FZ-deformations}, a semiuniversal deformation of a log smooth surface remains such after removing a $(-1)$-curve from the boundary, which proves \ref{item:adding-1-semiuniversal}. For the remaining parts, note that $f$ represents $\cR_{+}$ and $\cR$. Part  \ref{item:adding-1-h1} now follows directly from \ref{item:adding-1-semiuniversal} and definitions; and to prove part \ref{item:adding-1-faithful}, it remains to prove that $f$ is almost faithful. Let $\phi\colon (X_{b_1},D_{b_{1}})\to (X_{b_{2}},D_{b_{2}})$ be an isomorphism between fibers of $f$. By condition \eqref{eq:E-invariant}, $\phi$ extends to an isomorphism $\check{\phi}\colon (X_{b_{1}},\check{D}_{b_{1}})\to (X_{b_{2}},\check{D}_{b_2})$; so since $f$ is $\Aut(\check{\cS})$-faithful, 
	the points $b_1$ and $b_2$ lie in the same orbit of the $\Aut(\check{\cS})_{f}$-action, as needed. 
\end{proof}

\begin{lemma}[Checking condition \eqref{eq:E-invariant}]\label{lem:adding-1-criterion}
	We keep the notation and assumptions from Lemma \ref{lem:adding-1}. Assume further that for each $b\in B$ the following conditions \ref{item:adding-1-criterion-delPezzo}, \ref{item:adding-1-criterion-fibration}  hold. 	 Then condition \eqref{eq:E-invariant} holds, too. 
	\begin{enumerate}
		\item\label{item:adding-1-criterion-delPezzo} The  N\'eron-Severi group $\NS_{\Q}(X_{b})$ is generated by $K_{X_{b}}$ and the components of $D_b$ (for instance, $(X_{b},D_{b})$ is a log resolution of a del Pezzo surface of rank one).
		\item\label{item:adding-1-criterion-fibration} The surface $X_{b}$ admits a $\P^{1}$-fibration such that $E_{b}$ is the sum of all vertical $(-1)$-curves, and $\Aut(\cS)$ maps $k$-sections to $k$-sections, that is, the set of vertices of $\Gamma$ admits an $\Aut(\cS)$-invariant decomposition $\bigsqcup_{j\geq 0} \Gamma_{j}$ such that the fiber $F_b$ satisfies $F_{b}\cdot D_{b}^{v}=j$ for every $j\geq 0$ and every $v\in \Gamma_j$.
	\end{enumerate}
\end{lemma}
\begin{proof}
	Fix an isomorphism $\phi\colon (X_{b_1},D_{b_{1}})\to (X_{b_{2}},D_{b_{2}})$ and let $\psi=\phi_{+}^{-1}\in \Aut(\cS)$. Then $D_{b_{2}}^{v}=\phi(D_{b_{1}}^{\psi(v)})$ for every vertex $v$ of $\Gamma(f)$.  
	By assumption \ref{item:adding-1-criterion-fibration}, for every $v\in \Gamma_j$ we have 
	$\psi(v)\in \Gamma_j$ and 
	\begin{equation*}
		F_{b_{2}}\cdot D_{b_{2}}^{v}=j=F_{b_{1}}\cdot D_{b_{1}}^{v}=
		F_{b_{1}}\cdot D_{b_{1}}^{\psi(v)}=\phi(F_{b_{1}})\cdot \phi(D_{b_{1}}^{\psi(v)})=\phi(F_{b_{1}})\cdot D_{b_{2}}^{v},
	\end{equation*}
	and by adjunction $F_{b_{2}}\cdot K_{X_{b_{2}}}=-2=F_{b_{1}}\cdot K_{X_{b_{1}}}$. Hence by assumption \ref{item:adding-1-criterion-delPezzo} we get a numerical equivalence $F_{b_{2}}\equiv \phi(F_{b_{1}})$. It follows that the $\P^{1}$-fibrations induced by $|F_{b_{2}}|$ and $|\phi(F_{b_{1}})|$ are the same up to an automorphism of the base. In particular, they have the same vertical $(-1)$-curves, so $E_{b_2}=\phi(E_{b_{1}})$, as needed.
\end{proof}

\subsection{Summary of notation}\label{sec:notation}

We now introduce notation which will allow us to write down those weighted graphs which appear in all the lemmas cited in Theorem \ref{thm:ht=1,2}. First, we recall the notation for rational log canonical singularity types. They are disjoint sums of rational chains, forks, and benches, see Section \ref{sec:singularities}. We denote them as follows:
\begin{itemize}
	\item $[a_1,\dots,a_n]$ is a chain with components of subsequent self-intersection numbers $-a_1,\dots, -a_n$,
	\item $\langle b;T_1,T_2,T_3 \rangle$ is a fork with a branching $(-b)$-curve, and maximal twigs of types $T_1,T_2,T_3$,
	\item $\lbr T \rbr$ is a bench with central chain $T$.
\end{itemize}
We also use the following simplifying conventions:
\begin{itemize}
	\item $(m)_{k}$ stands for an integer $m$ repeated $k$ times, and $[(m)_0]$ stands for an empty chain,
	\item $[a_1,\dots,a_k]*[b_1,\dots,b_l]\de [a_1,\dots,a_{k-1},a_k+b_k-1,b_2,\dots, b_l]$, see formula  \eqref{eq:convention-Tono_star},
	\item $(\mbox{empty chain})*[a_1,\dots,a_k]\de [a_2,\dots,a_k]$,
	\item $[(2)_{-1}]*[a_1,\dots,a_k]\de [a_2+1,\dots,a_k]$, see formula \eqref{eq:convention_2-1},
	\item $[(2)_{-1},a_1,\dots,a_n] \de [a_2,\dots, a_k]$, see  formula \eqref{eq:convention_2-1},
	\item $T\trp$ is the transpose of the chain $T$, i.e.\ $[a_1,\dots,a_{n}]\trp=[a_n,\dots,a_1]$,
	\item $T^{*}$ is the chain adjoint to the chain $T$, given by the formula \eqref{eq:adjoint_chain}.
\end{itemize}

We now introduce additional notation, which allows to encode both the singularity type of a surface, and a structure of a $\P^1$-fibration on its minimal resolution.

\begin{notation}[Decorated types, see Example \ref{ex:notation}]\label{not:fibrations}
	Let $\bar{X}$ be a normal surface, and let $(X,D)\to(\bar{X},0)$ be its minimal log resolution. Assume that $D$ is a sum of rational chains, forks and benches. Fix a $\P^1$-fibration of $X$, and let $A_1,\dots, A_n$ be all components of its degenerate fibers which are not contained in $D$. Assume that every $A_j$ is a $(-1)$-curve, meeting every component of $D$ at most once (this will always hold in our case). We describe the weighted graph of $\check{D}\de D+\sum_{j}A_j$, together with the decomposition $\check{D}=\check{D}\hor+\check{D}\vert$, as follows. First, we write the type of each chain, fork or bench in $D$ using the conventions summarized above; so each number corresponds to a component of $D$.  Next, we decorate it as follows.
	\begin{enumerate}
		\item\label{not:bold} We put in boldface those numbers which correspond to horizontal components of $D$. 
		\begin{itemize}
			\item If a symbol $T$ stands for a chain $[a_1,\dots, a_n]$, we write $\fh{T}\de [\bs{a_1},a_2,\dots,a_{n}]$, $\lh{T}\de [a_1,\dots,a_{n-1},\bs{a_{n}}]$.
		\end{itemize}
		\item\label{not:Aj} We decorate by $\dec{j}$ the number corresponding to each component of $D$ meeting $A_j$.
		\begin{itemize}
			\item If $T$ stands for a chain $[a_1,\dots, a_n]$, we write $\ldec{j}T\de [a_{1}\dec{j},a_{2},\dots,a_{n}]$, $T\dec{j}\de [a_1,\dots,a_{n-1},a_{n}\dec{j}]$.
			\item For a bench $\lbr T \rbr=[2,a_1,2]+[a_{2},\dots, a_{n-1}]+[2,a_{n},2]$ with  central chain $T=[a_1,\dots,a_{n}]$, we write 
		\begin{equation*}
			\begin{split}
			\ldec{j}\lbr T \rbr& \de [2\dec{j},a_1,2]+[a_{2},\dots, a_{n-1}]+[2,a_{n},2],\quad 
			\ldecb{i}{j}\lbr T  \rbr\de [2\dec{i},a_1,2\dec{j}]+[a_{2},\dots, a_{n-1}]+[2,a_{n},2] \\
			\lbr T \rbr\dec{j}& \de [2,a_1,2]+[a_{2},\dots, a_{n-1}]+[2\dec{j},a_{n},2],\quad 
			\lbr T \rbr\decb{i}{j}\de [2,a_1,2]+[a_{2},\dots, a_{n-1}]+[2\dec{i},a_{n},2\dec{j}].
			\end{split}
		\end{equation*}
		\end{itemize} 
	\end{enumerate}
Note that if $\check{D}$ is snc then by assumption every two components of $\check{D}$  meet in at most one point, so decorations introduced in \ref{not:Aj} unambiguously encode the weighted graph of $\check{D}$. In some rare cases our $\check{D}$ will not be snc, because some $(-1)$-curve $A_j$ will pass through a common point of two adjacent components of $D$. Although our notation does not capture such a degeneracy, it will usually be clear from the constructions when it occurs.
\end{notation}

\begin{notation}[Figures]\label{not:figures}
	Unless stated otherwise, we use the following conventions to draw $\check{D}=D+\sum_j A_j$.
	\begin{enumerate}
		\item A solid line annotated by \enquote{$k$} denotes an $k$-curve in $D$. Weight $k=-2$ is skipped.
		\item A dashed line denotes a $(-1)$-curve in $\check{D}-D$.
	\end{enumerate}
\end{notation}

\begin{example}[Figure \ref{fig:cor-2} expressed using Notation \ref{not:fibrations}]\label{ex:notation}
	The weighted graphs of $\check{D}$ in Figure \ref{fig:cor-2} are 
	as follows:
	{\setlength\multicolsep{2pt}\begin{multicols}{2}
	\begin{enumerate}
		\item
		$\langle 2;\ldec{1}[2],[2,2],[\bs{2},3]\rangle$,
		\item
		$\langle 2;\ldec{1}[2,2,2,2],[2],[\bs{2},2] \rangle$,
		\item
		$\langle \bs{2}, [(2)_{r-1},\bs{2},2],\ldec{2}[2,2],[2]\rangle+[r]\dec{1}$,
		\item
		$\langle 2;\ldec{1}[2,\bs{2}],\ldec{2}[(2)_{m}],\ldec{3}[2]\rangle+\ldec{1}[2,\bs{m}]\dec{2}$,
		\item
		$\langle 2;\ldec{1}[3,\bs{2}],\ldec{2}[2,2],\ldec{3}[2]\rangle+\ldec{1}[2,2,\bs{2}]\dec{2}$.
	\end{enumerate}
\end{multicols}}
\end{example}

\begin{notation}[Symbol $\cM^{d}$]\label{not:Md}
	In Tables \ref{table:ht=1_exceptions}, \ref{table:exceptions-to-exceptions} and \ref{table:ht=2_char-any}--\ref{table:canonical} we list numbers $\#\cC$ for some sets $\cC$ of isomorphism classes of log surfaces. In those tables, symbol $\cM^{d}$ for $d\geq 1$ means that we have $\#\cC=\infty$, and $\cC$ is represented by an almost faithful family of dimension $d$, see Definition \ref{def:moduli}\ref{item:def-family-faithful}. 
\end{notation}

\clearpage
\section{Non--log terminal del Pezzo surfaces}\label{sec:non-lt-proof}

In this section, we study non--log terminal del Pezzo surfaces of rank one. Looking at minimal models of their log resolutions, described in \cite[Proposition 3.2]{PaPe_MT}, we prove  Proposition \ref{prop:non-lt-more} below, which is a more detailed version of Proposition \ref{prop:non-log-terminal}. In Section \ref{sec:non-lt-examples} we illustrate each of its cases by non--log canonical examples.

\begin{proposition}[Non--log terminal del Pezzo surfaces]\label{prop:non-lt-more}
	Let $\bar{X}$ be a non--log terminal del Pezzo surface of rank one, and let $(X,D)$ be its minimal log resolution. Write $D=D_0+D\lt$, where $D\lt$ is the preimage of all log terminal singularities of $\bar{X}$. Then the following hold.
	\begin{enumerate}
		\item\label{item:non-lt-B} The divisor $D_0$ is a tree with at most one component $B$ such that  $p_{a}(B)>0$ or $\beta_{D}(B)\geq 4$. In particular, $\bar{X}$ has exactly one non--log terminal singularity.
		\item\label{item:non-lt-ht=1} If $\height(\bar{X})=1$ then every connected component of $D\lt$ is an admissible chain.
		\item\label{item:non-lt-ht>=2} Assume $\height(\bar{X})\geq 2$. Then $\bar{X}$ is rational, has only rational singularities, and every component $C$ of $D_0$ satisfies $p_{a}(C)=0$ and  $\beta_{D}(C)\leq 4$.
		\item\label{item:non-lt-ht=2} Assume $\height(\bar{X})=2$. Then one of the following holds.
		\begin{enumerate}
			\item\label{item:non-lt-chains} Every connected component of $D\lt$ is an admissible chain. In this case, $\bar{X}$ swaps vertically to a surface of type $3\rA_2$, $\rA_1+2\rA_3$, $\rA_1+\rA_2+\rA_5$ or $2\rA_4$ from Examples \ref{ex:ht=2}, \ref{ex:ht=2_meeting}, see Fig.\ \ref{fig:3A2-intro}--\ref{fig:A5+A2+A1-intro}, \ref{fig:2A4-intro}. 
			\item\label{item:non-lt-Platonic} The surface $\bar{X}$ is constructed from a Platonic $\A^{1}_{*}$-fiber space  as in Example \ref{ex:permissible_fork}, see Figure \ref{fig:Platonic-examples}. In particular, $D\lt$ is a disjoint union of one admissible fork and at most one admissible chain; $X$ admits a witnessing  $\P^1$-fibration such that $D\hor$ consists of two branching components of $D$, and $\bar{X}$ swaps vertically to a surface of type $\rA_1+2\rA_3$ or $2\rD_4$ from Example \ref{ex:ht=2}\ref{item:A1+2A3_construction},\ref{item:2D4_construction}, see Fig.\ \ref{fig:2A3+A1-intro}, \ref{fig:2D4-intro}. 
			\item \label{item:non-lt-twisted} We have $\cha \kk=2$, and $\bar{X}$ swaps vertically to a surface from Example \ref{ex:ht=2_twisted_cha=2}, see Figure \ref{fig:KM_surface-intro}.  
		\end{enumerate}
		\item\label{item:non-lt-ht>=3} Assume $\height(\bar{X})\geq 3$. Then $\cha\kk\in \{2,3,5\}$, and  $\bar{X}$ satisfies condition  \ref{item:non-lt_descendant} or \ref{item:non-lt-swap} of Proposition \ref{prop:non-log-terminal}. Moreover, every connected component of $D\lt$ is an admissible chain or a $(-2)$-fork.
	\end{enumerate}
\end{proposition}

\subsection{Proof of Proposition \ref{prop:non-lt-more}: structure of non--log terminal del Pezzo surfaces of rank one}

The proof is based on \cite[Proposition 3.2]{PaPe_MT}, which describes possible almost minimal models of $(X,D)$. To describe the process of almost minimalization, we will need the following elementary lemma. 

\begin{lemma}[Birational morphism with exactly one exceptional $(-1)$-curve]\label{lem:one-base-point}
			Let $(Y,B)$ be a log smooth surface. Let $\tau\colon Y'\to Y$ be a birational morphism such that $\Exc\tau$ contains exactly one $(-1)$-curve, say $A$, and the point $\tau(A)$ lies on $B$. Put $E=\Exc\tau-A$, $B'=\tau^{-1}_{*}B'$. Then the following hold.
			\begin{enumerate}
				\item\label{item:one-base-point-beta} Every component of $E$ has branching number at most $3$ in $B'+E$.
				\item\label{item:one-base-point-b0} The divisor $E$ has at most two connected components.
				\item\label{item:one-base-point-chain} A connected component of $E$ which does not meet $B'$ is an admissible chain.
			\end{enumerate}
\end{lemma}
\begin{proof} Straightforward by induction on the number of blowups in the decomposition of $\tau$. \end{proof}

\begin{proof}[Proof of Proposition \ref{prop:non-lt-more}]
	Assume first that 	$\height(\bar{X})=1$. Fix a witnessing $\P^{1}$-fibration of $(X,D)$. 
	By Lemma \ref{lem:delPezzo_fibrations} every degenerate fiber has a unique $(-1)$-curve. It follows from Lemma \ref{lem:degenerate_fibers}\ref{item:adjoint_chain}, see also Lemma \ref{lem:ht=1_basics} below, that each fiber meets $D\hor$ in a tip, and each connected component of $D$ not containing $D\hor$ is an admissible chain. This proves \ref{item:non-lt-ht=1}  and case $\height=1$ of \ref{item:non-lt-B}.
	
	Thus by Remark \ref{rem:ht_finite_nonzero} we can assume $\height(\bar{X})\geq 2$. Let $\psi_{\min}\colon (X,D)\to (X_{\min},D_{\min})$ be an MMP run. Since $\height(\bar{X})\geq 2$, the minimal model $(X_{\min},D_{\min})$ is a dlt log del Pezzo surface of rank one, see \cite[Lemma 4.1(a)]{PaPe_MT}. We have  $\#\Exc\psi_{\min}-\#D=\#\Exc\psi_{\min}-\rho(X)+\rho(\bar{X})=-\rho(X_{\min})+\rho(\bar{X})=0$, so $\#\Exc \psi_{\min}=\#D$. Since $X_{\min}$ is log terminal, we have $\Exc\psi_{\min}\neq D$. Hence $D_{\min}\neq 0$. Put $n=\#D_{\min}$. The morphism $\psi_{\min}$ contracts precisely $n$ curves not contained in $D$;  they are $(-1)$-curves by Lemma \ref{lem:min_res}\ref{item:-1-curves_off_D}. Since $(X,D)$ is log smooth, $\psi_{\min}$ factors as $\psi_{\min}=\alpha\circ\psi$, where $\alpha\colon (X\am,D\am)\to (X_{\min},D_{\min})$ is the minimal log resolution. This way, the analysis of non-lt del Pezzo surfaces of rank one, with no boundary, reduces to the analysis of dlt log del Pezzo surfaces \emph{with} nonzero boundary, so-called \emph{open del Pezzo}. 
	
	These surfaces are described in \cite[Proposition 3.2]{PaPe_MT}, which implies that $(X\am,D\am)$ is as in (1) or (3)--(6) loc.\ cit: we refer to each of these cases simply by its number. 
	In any case, the surface $X\am$, hence $\bar{X}$, is rational, which by \cite[Lemma 2.7]{PaPe_MT} is equivalent to $\bar{X}$ having only rational singularities. This proves the first three statements of \ref{item:non-lt-ht>=2}, and shows that every connected component of $D$ is a rational tree.
	
	\begin{casesp*}
		\litem{(1)} In this case $X\am=\P^2$ and $\deg D\am\leq 2$. Suppose $n=1$. Then $D\am$ is a line or a conic. Since $D$ is negative definite, and $\psi^{-1}$ has exactly one base point, we have $\psi=\beta\circ\psi_1$, where $\beta$ is a composition of $D\am^{2}$ blowups over a point of $D\am$, on the proper transforms of $D\am$. Now $|\beta^{-1}_{*}D\am|$ induces a $\P^1$-fibration which pulls back to a $\P^1$-fibration of height one with respect to $D$, contrary to our assumptions.
		
		Thus $n=2$, so $D\am$ is a sum of two lines, say $\ll_1$ and $\ll_2$. Since $\ll_{i}^2>0$, each $\ll_{i}$ contains a base point of $\psi^{-1}$. If the unique base point of $\psi^{-1}$ on $\ll_{i}$ is $\ll_1\cap \ll_{2}$, then the proper transform of $\ll_{i}$ after a blowup at that point is a fiber of a $\P^1$-fibration which pulls back to a $\P^1$-fibration of height $1$ with respect to $D$, which is impossible. Thus $\psi^{-1}$ has two base points, say $p_{i}\in \ll_{i}$, $p_{i}\not\in \ll_{1}\cap \ll_2$, and each $\psi^{-1}(p_i)$ contains exactly one $(-1)$-curve. By Lemma \ref{lem:one-base-point}\ref{item:one-base-point-chain} all connected components of $D$, except possibly the one containing $\psi^{-1}_{*}D\am$, are admissible chains; and by \ref{lem:one-base-point}\ref{item:one-base-point-beta} all components of $D$ have $\beta_{D}\leq 3$, so conditions \ref{item:non-lt-B}, \ref{item:non-lt-ht>=2} and \ref{item:non-lt-chains} hold. The proper transform of $\ll_{1}$ after a blowup at $p_1$ is a fiber of a $\P^1$-fibration whose pullback to $(X,D)$ has height $2$, as needed. 
		
		\litem{(3.1)--(3.3)} We have $2=\height(X\am,D\am)\geq \height(\bar{X})\geq 2$ by assumption, so $\height(\bar{X})=2$. By \cite[Remark 3.3]{PaPe_MT}, we can choose the minimal model $(X_{\min},D_{\min})$ so that $(X\am,D\am)$ is as in Figure 2 loc.\ cit. We have $n\in \{1,2\}$. We claim that conditions \ref{item:non-lt-B}, \ref{item:non-lt-ht>=2} and \ref{item:non-lt-chains} hold.
	
	\begin{casesp*}
		\litem{$n=1$} Let $R\de \alpha^{-1}_{*}D_{\min}$ be the component of $D\am$ with nonnegative self-intersection number. Since $D$ has no such component, the unique base point $p$ of $\psi^{-1}$ lies on $R$. Since $D\am$ is a chain, Lemma \ref{lem:one-base-point}\ref{item:one-base-point-chain},\ref{item:one-base-point-beta} implies that all connected components of $D\am$, except possibly the one containing $\psi^{-1}_{*}R$ are admissible chains, and $\beta_{D}(C)\leq 3$ for every component $C$ of $D$. 
		It remains to prove that $\bar{X}$ swaps vertically to one of the surfaces listed there. This will follow from a more general Lemma \ref{lem:ht=2_swaps}, let us briefly sketch the argument here. 
		
		If $\#D\am \geq 3$ and $p\in D\am\reg$, see \cite[Figure 2(c)]{PaPe_MT}, then the fact that $(\psi^{-1}_{*}R)^2\leq -2$, see  Lemma \ref{lem:min_res}\ref{item:D-snc}, easily implies that $(X,D)$ swaps vertically to the surface of type $2\rA_4$ from Example \ref{ex:ht=2_meeting}, see Figure \ref{fig:2A4-intro}. Assume that $\#D\am=2$ or $p\not\in D\am\reg$, and write $\psi=\beta\circ \psi_{1}$, where $\beta$ is a composition of $R^2$ blowups at $p$ and its infinitely near points on the proper transforms of $R$.  Then  $|\beta^{-1}_{*}R|$ pulls back to a $\P^1$-fibration of $X$ which is either of height one with respect to $D$, which is impossible, or has exactly two degenerate fibers, and $(X,D)$ with this $\P^1$-fibration swaps vertically to the surface of type $3\rA_2$ from Example \ref{ex:ht=2}\ref{item:3A2_construction}; as needed.
	
		\litem{$n=2$} Now  $D\am$ is a chain with a subchain $C=[0,1-m]$ for some $m\geq 1$. If the node of $C$ is a base point of $\psi^{-1}$ let $\beta$ be a blowup at this point; otherwise put $\beta=\id$. Then $\psi$ factors as $\psi=\beta\circ\psi_{1}$ and $(\beta^{*}C)\redd$ is a chain whose both tips have self-intersection at least $-1$, so they contain base points of $\psi_{1}^{-1}$, because $D$ has no $(-1)$-curves by Lemma \ref{lem:min_res}\ref{item:D-snc}. By construction those two base points are different, so the preimage of each contains at most one $(-1)$-curve. Now Lemma \ref{lem:one-base-point} implies, as before, that all components of $D$ have $\beta_{D}\leq 3$, and all but one connected components of $D$ are admissible chains. Moreover, if $\psi$ has a base point off $(D\am)\vert$ then $\bar{X}$ swaps vertically to the surface from Example \ref{ex:ht=2}\ref{item:3A2_construction} or \ref{item:2A1+2A3_construction}, see Figures \ref{fig:3A2-intro}, \ref{fig:3A2-intro}; otherwise it swaps vertically to the one from Example \ref{ex:ht=2}\ref{item:A1+A2+A5_construction}, see Figure \ref{fig:A5+A2+A1-intro}. For more details see Lemma \ref{lem:ht=2_swaps} below.
	\end{casesp*} 

	\litem{(3.4)} We have $n=1$. Like in cases (3.1)--(3.3) above, $\height(\bar{X})=2$. Let $H\am\subseteq D\am$ be the branching component of the non-admissible fork, and let $H=\psi^{-1}_{*}H\am$. Since $H^2\leq -2$ by Lemma \ref{lem:min_res}\ref{item:D-snc}, the base point of $\psi^{-1}$ lies on $H$, and $\psi$ factors through $H\am^2+2$ blowups over that point, as described in Example \ref{ex:permissible_fork}. Thus $\bar{X}$ is as in \ref{item:non-lt-Platonic}. Lemma \ref{lem:one-base-point} implies that $D$ all components of $D$ except the proper transform of $B$ have $\beta_{D}\leq 3$, and the latter has $\beta_{D}\leq 4$, so \ref{item:non-lt-ht>=2} holds. Moreover, $\psi^{-1}$ is an isomorphism near the admissible fork in $D\am$; and $D\wedge \Exc\psi$ consists of an admissible chain and a connected component meeting $B$, so \ref{item:non-lt-B} holds, too, as needed.
	\end{casesp*}
	
	In the remaining cases (4)--(6), $n=1$ and $D\am$ has a fork with a branching $(-1)$-curve $A$. Since by Lemma \ref{lem:min_res}\ref{item:D-snc} $D$ has no $(-1)$-curve, the base point of $\psi^{-1}$ lies on $A$. Put $B=\psi^{-1}_{*}A$. Lemma \ref{lem:one-base-point} implies that for every connected component $C$ of $D$ one of the following holds: either $\psi$ is an isomorphism near $C$, or $C\subseteq \Exc\psi$ is an admissible chain, or $C$ is a tree containing $B$. Hence $D_0$ is a rational tree. Moreover, since every component of $D\am$ has $\beta_{D\am}\leq 3$, every component of $D-B$ has $\beta_{D}\leq 3$, and $\beta_{D}(B)\leq 4$. This proves \ref {item:non-lt-B} and \ref{item:non-lt-ht>=2}.
	
	\begin{casesp*}
	\litem{(4)} In this case, we have $\cha\kk=2$, and the surface $X_{\min}$ swaps vertically to the one from Example \ref{ex:ht=2_twisted_cha=2}. Since the base point of $\psi^{-1}$ lies on some vertical $(-1)$-curve, decomposing $\psi$ into elementary vertical swaps shows that $\bar{X}$ swaps to $\bar{X}_{\min}$, and therefore $\height(\bar{X})\leq \height(\bar{X}_{\min})\leq 2$ and  \ref{item:non-lt-twisted} holds.
	
	\litem{(5)}\label{case:5} In this case, we have $\cha\kk=2$, and the log surface $(X\am,D\am)$ is described in \cite[Example 3.9]{PaPe_MT}, see Figures 4(b), 4(c) loc.\ cit. These figures are repeated in the leftmost and rightmost part of Figure \ref{fig:ht=3-remark}, explained in Remark \ref{rem:ht=3}. In particular, we have $3=\height(X\am,D\am)\geq \height(\bar{X})$. 
	
	We claim that Proposition  \ref{prop:non-log-terminal}\ref{item:non-lt-swap} holds, i.e.\ 
	$\bar{X}$ swaps vertically to a surface $\bar{Y}$ of type $\rA_3+4\rA_1+[3]$, with a $\P^1$-fibration of height $3$ constructed in Example \ref{ex:ht=3}, see the middle of Figure \ref{fig:ht=3-remark}. For this we need another $\P^1$-fibration of $X\am$, used in \cite[Lemmas 4.8, 4.9]{PaPe_MT}. It is shown in Figure \ref{fig:ht=3-remark}, too, and is constructed as follows. Let $A+T+T_2+T_3$ be the fork in $D\am$, with branching component $A=[1]$ and twigs $T$, $T_2=[2]$, $T_{3}=[3]$. In case (5.1) of \cite[Proposition 3.2]{PaPe_MT}, we consider the $\P^1$-fibration induced by $|T_2+2A+\ltip{T}|$, and in case (5.2) loc.\ cit., the one with a degenerate fiber supported on $A'+T+A+T_2$, where $A'$ is a $(-1)$-curve denoted in loc.\ cit.\ by $\dec{1}$. The horizontal part of $D\am$ consists of a $2$-section $T_3$ and a $1$-section $H$; and degenerate fibers are supported on: $[2,1,2]$, $[2,1,2]$, $\langle 2;[1,3],[1],[2]\rangle$ in case (5.1) and $[1,2,3,1,2]$, $[2,1,2]$, $\langle 2;[1],[2],[2]\rangle$ in case (5.2). They meet $T_3$ in the $(-1)$-curve of multiplicity $2$; and $H$ either in a $(-1)$-curve of multiplicity $1$ (if such occurs) or in a $(-2)$-tip. With respect to this $\P^1$-fibration, the surface $X_{\min}$ swaps vertically to $\bar{Y}$. Furthermore, as in case (4), since the base point of $\psi^{-1}$ lies on a vertical $(-1)$-curve, we see that $\bar{X}$ swaps vertically to $X_{\min}$. 
	
	Moreover, all connected components of $D\am$ not containing $A$ are admissible chains, so all connected components of $D$ not containing $\psi^{-1}_{*}A$ are admissible chains, too, as claimed in \ref{item:non-lt-ht>=3}.
	
	\litem{(6)} Here the surface $\bar{X}_{\min}$ admits a descendant with elliptic boundary as in \ref{prop:non-log-terminal}\ref{item:non-lt_descendant}, hence so does $\bar{X}$. Moreover all connected components of $D\am$ not containing $A$ are $(-2)$-chains and $(-2)$-forks, so by Lemma \ref{lem:one-base-point}\ref{item:one-base-point-chain} all components of $D$ not containing $\psi^{-1}_{*}A$ are admissible chains or $(-2)$-forks, as claimed in \ref{item:non-lt-ht>=3}.
	\qedhere\end{casesp*}
\end{proof}

\subsection{Examples of non--log canonical del Pezzo surfaces of rank one}\label{sec:non-lt-examples}

We now construct some non--log canonical del Pezzo surfaces of rank one falling in each case of Proposition \ref{prop:non-lt-more}. Although they are all constructed from a simple $\P^1$-fibrations,  their singularities can still be complicated.

We begin with an example where the minimal log resolution has a boundary component with an arbitrarily high branching number. Note that by Proposition \ref{prop:non-lt-more}\ref{item:non-lt-ht>=2} such examples are necessarily of height one.

\begin{example}[Case \ref{prop:non-lt-more}\ref{item:non-lt-ht=1}: $\height=1$, high branching number, see Figure \ref{fig:non-lt_ht=1}]\label{ex:non-lt-ht=1} 
We construct two examples: \ref{item:non-lt-ht=1_many} with $\nu+1$ singularities, of type $\langle m+\nu;\nu\cdot [2]\rangle+\nu\cdot \rA_1$;  and \ref{item:non-lt-ht=1_one} with only one singularity, of type $\langle m+\nu;\nu\cdot [2,2,2] \rangle$. Here $\langle b,\nu\cdot T\rangle$ denotes a tree with one branching component $[b]$ and $\nu$ maximal twigs of type $T$.

\begin{parlist} 
	\item\label{item:non-lt-ht=1_many} Let $\F_{m}=\P(\cO_{\P^1}(m)\oplus\cO_{\P^1})$ be a Hirzebruch surface for some $m\geq 1$. Choose $\nu$ points $p_1,\dots, p_{\nu}$ on the negative section $\Sec_{m}=[m]$, and denote by $F_{i}$ the fiber passing through $p_i$. Blow up at each $p_i$ and its infinitely near point on the proper transform of $F_i$. Denote this morphism by $\psi\colon X\to\F_{m}$, let $A_i$, $B_i$ be the $(-1)$- and $(-2)$-curves in $\psi^{-1}(p_i)$, and let $H=\psi^{-1}_{*}\Sec_m$, $C_{i}=\psi^{-1}_{*}F_i$,  $D=H+\sum_{i}(B_i+C_i)$. Then $C_i=[2]$ is a connected component of $D$, and $H=[m+\nu]$ meets each $B_i=[2]$. In particular, $\beta_{D}(H)=\nu$. Note that $(X,D)$ is the log surface from Example \ref{ex:ht=1} for $g=0$, $e=m+\nu$, see Figure \ref{fig:ht=1-intro}.  
	
	The divisor $D_0\de H+\sum_{i}B_i$ contracts to a rational singularity; with fundamental cycle $D_0$. Let $X\to \bar{X}$ be the contraction of $D$. Then $\rho(\bar{X})=1$. We compute that $\cf(H)=1+\frac{\nu-4}{\nu+2m}<2$, so by Lemma \ref{lem:delPezzo_criterion} the surface $\bar{X}$ is del Pezzo. It 
	is log terminal (respectively, log canonical but not log terminal) if and only if $\nu\leq 3$ (resp.\ $\nu=4$): these surfaces will appear in cases \ref{item:chains_columnar_T1=0,T2=0}--\ref{item:chains_columnar}, 
	\ref{item:beta=3_columnar} and \ref{item:ht=1_bench} of Lemma \ref{lem:ht=1_types}.
\begin{figure}[htbp]
		\begin{tikzpicture}[scale=0.9]
			\begin{scope}	
				\draw (0,2.4) -- (2.8,2.4);
				\node at (1.85,2.2) {\small{$-m-\nu$}};
				\draw (0,0) -- (0.2,1);
				\draw[dashed] (0.2,0.8) -- (0,1.8);
				\draw (0,1.6) -- (0.2,2.6);
				%
				\draw (1,0) -- (1.2,1);
				\draw[dashed] (1.2,0.8) -- (1,1.8);
				\draw (1,1.6) -- (1.2,2.6);
				%
				\node at (1.85,0.4) {\large{$\dots$}};
				\draw (2.5,0) -- (2.7,1);
				\draw[dashed] (2.7,0.8) -- (2.5,1.8);
				\draw (2.5,1.6) -- (2.7,2.6);
				\draw [decorate, decoration = {calligraphic brace}, thick] (2.8,-0.1) -- (-0.2,-0.1);
				\node at (1.5,-0.4) {\small{$\nu$ fibers}};
			\end{scope}
			\begin{scope}[shift={(8,0)}]
				\draw (0,2.4) -- (2.8,2.4);
				\node at (1.85,2.2) {\small{$-m-\nu$}};
				\draw (0,0) -- (0.2,1);
				\draw (0.2,0.8) -- (0,1.8);
				\draw[dashed] (0.2,1.3) -- (-0.8,1.1);
				\draw (0,1.6) -- (0.2,2.6);
				%
				\draw (1,0) -- (1.2,1);
				\draw (1.2,0.8) -- (1,1.8);
				\draw[dashed] (1,1.1) -- (2,1.3);
				\draw (1,1.6) -- (1.2,2.6);
				%
				\node at (1.85,0.4) {\large{$\dots$}};
				\draw (2.5,0) -- (2.7,1);
				\draw (2.7,0.8) -- (2.5,1.8);
				\draw[dashed] (2.5,1.1) -- (3.5,1.3);
				\draw (2.5,1.6) -- (2.7,2.6);
				\draw [decorate, decoration = {calligraphic brace}, thick] (2.8,-0.1) -- (-0.2,-0.1);
				\node at (1.5,-0.4) {\small{$\nu$ fibers}};
			\end{scope}
		\end{tikzpicture}
	\caption{Example \ref{ex:non-lt-ht=1}: non--log canonical del Pezzo surfaces of rank one and height one, with a highly branching component in the exceptional divisor of the minimal resolution.}
	\label{fig:non-lt_ht=1}
\end{figure}
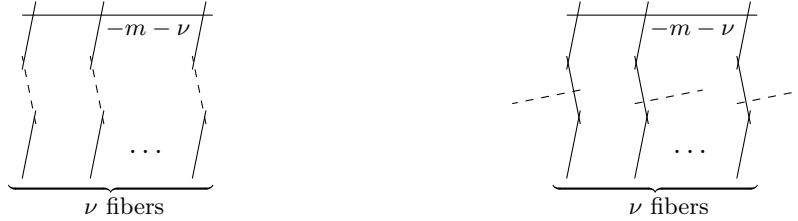
	\item\label{item:non-lt-ht=1_one} 
	We keep the notation from \ref{item:non-lt-ht=1_many}. Blow up once at $A_{i}\setminus (D-A_i)$ for all $i\in \{1,\dots, \nu\}$, denote this morphism by $\sigma\colon X'\to X$ and put $D'=\sigma^{-1}_{*}(D+\sum_{i}A_i)$. As before, there is a morphism $X'\to \bar{X}'$ contracting $D$ to a rational singularity, with reduced fundamental cycle. We have $\cf(H')=1+\frac{3\nu-8}{\nu+4m}$, so by Lemma \ref{lem:delPezzo_criterion}, the surface $\bar{X}'$ is del Pezzo if and only if $\nu\leq 2m+3$. For $\nu\leq 2$ it is log terminal, and as such will appear in Lemma \ref{lem:ht=1_types}\ref{item:chains_both_T1=0},\ref{item:chains_not-columnar}. For $\nu\geq 3$ it is not log canonical.
\end{parlist}	
\end{example}

We  now construct a del Pezzo surface $\bar{X}$ of rank one and height two whose non-lt singular point has a non-star-shaped resolution graph, see Figure \ref{fig:non-star-shaped}. Note that if $\cha\kk\not\in \{2,3,5\}$ then such a surface necessarily falls into case \ref{item:non-lt-chains} of Proposition \ref{prop:non-lt-more}, so its log terminal singularities are images of admissible chains. 

\begin{example}[Case \ref{prop:non-lt-more}\ref{item:non-lt-chains}: a non-star-shaped resolution, see Figure \ref{fig:non-star-shaped}]\label{ex:non-star-shaped}
	Let $(Y,D_Y)$ be the minimal log resolution of a surface of type $\rA_1+\rA_2+\rA_5$, together with the $\P^1$-fibration constructed in Example \ref{ex:ht=2}\ref{item:A1+A2+A5_construction}. We define a vertical swap $(X,D)\sqto (Y,D_Y)$, as follows. We keep the notation from Figure \ref{fig:A5+A2+A1}, and put $H=(D_Y)\hor-C$. Fix $k\geq 4$. Let $\psi\colon X\to Y$ be a composition of one blowup at $A_3\cap (D_Y-L_3)$, and $k-1$ blowups over $A_0\cap H$, such that $\psi^{-1}(A_0\cap H)=[2,1,3,(2)_{k-4}]$, and the first tip of this chain meets $\psi^{-1}_{*}H$. Define $D$ as $\psi^{*}(D_Y+A_0+A_3)\redd$ minus the $(-1)$-curves. Then $D$ is a disjoint union of admissible chains $[3]$, $[3,(2)_{k-4},3,2]$, and a connected component $D_0$. The latter is a sum of a chain $\psi^{-1}_{*}(H+L_3)=[k,2]$, two twigs $[2]$, $[2]$ meeting $\psi^{-1}_{*}H$ and two twigs $[2]$, $[2,\underline{2}]$ meeting $\psi^{-1}_{*}L_3$, where the underlined \enquote{$2$} corresponds to $\psi^{-1}_{*}C$. 
	
	Artin's criterion implies that $D_0$ contracts to a rational singularity, with fundamental cycle $D_0+\psi^{-1}_{*}(L_3+C)$. Let $X\to \bar{X}$ be the contraction of $D$. Then $\bar{X}$ is a normal surface of rank one, which is not log canonical. We compute $\cf(\psi^{-1}_{*}H)=1+\frac{1}{5k-11}$, $\cf(\psi^{-1}_{*}C)=1-\frac{k-3}{5k-11}$, see formulas in  \cite[3.2.5]{Flips_and_abundance}. The inequality \eqref{eq:ld_phi_H} shows that the surface $\bar{X}$ is del Pezzo if and only if $k\geq 5$.
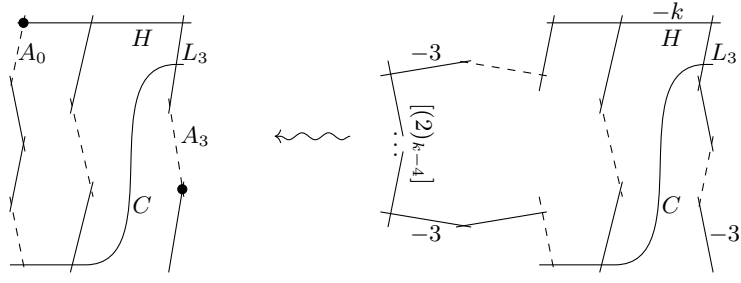
\begin{figure}[ht]
	\begin{tikzpicture}
	\begin{scope}
		\draw (0.1,3) -- (2.4,3);
		\node at (1.75,2.8) {\small{$H$}};
		\draw (0,-0.2) -- (1,-0.2) to[out=0,in=180] (2.2,2.45) -- (2.3,2.45);
		\node at (1.75,0.6) {\small{$C$}};
		\draw[dashed] (0.2,3.1) -- (0,2.1);
		\node at (0.3,2.6) {\small{$A_0$}};
		\filldraw (0.18,3) circle (0.06);
		\draw (0,2.3) -- (0.2,1.3);
		\draw (0.2,1.5) -- (0,0.5);
		\draw[dashed] (0,0.7) -- (0.2,-0.3);
		\draw (1.1,3.1) -- (0.8,1.8);
		\draw[dashed] (0.8,2) -- (1.1,0.7);
		\draw (1.1,0.9) -- (0.8,-0.3);
		\draw (2.3,3.1) -- (2.1,1.8);
		\node at (2.45,2.6) {\small{$L_3$}};
		\draw[dashed] (2.1,2) -- (2.3,0.7);
		\node at (2.45,1.5) {\small{$A_{3}$}};
		\filldraw (2.28,0.8) circle (0.06);
		\draw (2.3,0.9) -- (2.1,-0.3);
		\draw[->,decorate,decoration={snake,amplitude=1.3pt}] (4.5,1.5) -- (3.5,1.5);
	\end{scope}			
	\begin{scope}[shift={(7,0)}]
		\draw (0.1,3) -- (2.4,3);
		\node at (1.75,2.8) {\small{$H$}};
		\node at (1.7,3.15) {\small{$-k$}};
		\draw (0,-0.2) -- (1,-0.2) to[out=0,in=180] (2.2,2.45) -- (2.3,2.45);
		\node at (1.75,0.6) {\small{$C$}};
		\draw (0.2,3.1) -- (0,2.1);
		\draw[dashed] (0.1,2.3) -- (-1.1,2.5);
		\draw (-0.9,2.5) -- (-2.1,2.3);
		\node at (-1.5,2.6) {\small{$-3$}};
		\draw (-2,2.5) -- (-1.8,1.5);
		\node at (-1.9,1.5) {\small{$\vdots$}};
		\node[rotate=-90] at (-1.55,1.4) {\small{$[(2)_{k-4}]$}};
		\draw (-1.8,1.3) -- (-2,0.3);
		\draw (-2.1,0.5) -- (-0.9,0.3);
		\draw (-1.1,0.3) -- (0.1,0.5);
		\node at (-1.5,0.2) {\small{$-3$}};
		\draw[dashed] (0,0.7) -- (0.2,-0.3);
		\draw (1.1,3.1) -- (0.8,1.8);
		\draw[dashed] (0.8,2) -- (1.1,0.7);
		\draw (1.1,0.9) -- (0.8,-0.3);
		\draw (2.3,3.1) -- (2.1,2.1);
		\node at (2.45,2.6) {\small{$L_3$}};
		\draw (2.1,2.3) -- (2.3,1.3);
		\draw[dashed] (2.3,1.5) -- (2.1,0.5);
		\draw (2.1,0.7) -- (2.3,-0.3);
		\node at (2.45,0.2) {\small{$-3$}};
	\end{scope}
	\end{tikzpicture}\vspace{-0.5em}
	\caption{Example \ref{ex:non-star-shaped}: non-lc del Pezzo surface swapping vertically to~$\rA_1\!+\rA_2\!+\rA_5$.
		\label{fig:non-star-shaped}.} 
\end{figure} 
\end{example}

We now consider case \ref{prop:non-lt-more}\ref{item:non-lt-Platonic}. These surfaces are of height $2$ and their minimal log resolutions are constructed from Platonic $\A^{1}_{*}$-fiber spaces, studied in \cite{Miy_Tsu-opendP}. In Example \ref{ex:permissible_fork} we give a general definition of such del Pezzo surfaces. In  Example \ref{ex:permissible_fork-example} we construct a particular non--log canonical del Pezzo surface of this kind, see Figure \ref{fig:permissible_fork-example}. Note that if $\cha\kk\not\in \{2,3,5\}$ then case \ref{prop:non-lt-more}\ref{item:non-lt-Platonic} covers all non-lt del Pezzo surfaces of rank one which have a log terminal singularity whose minimal resolution is an admissible fork.

\begin{example}[Case \ref{prop:non-lt-more}\ref{item:non-lt-Platonic}: blowing up over a Platonic $\A^{1}_{*}$-fiber space]\label{ex:permissible_fork}
	First, we recall a construction of a Platonic $\A^{1}_{*}$-fiber space from \cite{Miy_Tsu-opendP} and its log smooth completion $(X\am,D\am)$, see Figure \ref{fig:Platonic}. 
	
	Fix $m\geq 2$, and consider a Hirzebruch surface $\F_{m}$. Denote by $H_{1}''=[m]$ the negative section. Choose three fibers $F_1,F_2,F_3$ and a section $H_{0}''$ disjoint from $H_{1}''$, so $H_{0}''=[-m]$. 
	Let $\tau\colon X\am\to \F_{m}$ be a composition of blowups over $F_{i}\cap H_0''$ for each $i\in \{1,2,3\}$ such that $(\tau^{*}F_{i})\redd=[T_{i},1,T_{i}^{*}]$ for some admissible chains $T_{i}$ satisfying $\sum_{i=1}^{3} \frac{1}{d(T_i)}>1$. Put $H_{j}'=\tau^{-1}_{*}H_{j}''$, and define $D\am$ as $\tau^{*}(H_1''+H_2''+F_1+F_2+F_3)\redd$ without the $(-1)$-curves. Then $D\am$ has two connected components: a fork $D\am^{0}=\langle 3-m;T_1\trp,T_2\trp,T_3\trp \rangle$ which is not negative definite, and an admissible fork $D\am^{1}=\langle m, T_{1}^{*},T_{2}^{*},T_{3}^{*}\rangle$.
	
	We now construct a class of del Pezzo surfaces from Proposition \ref{prop:non-lt-more}\ref{item:non-lt-Platonic}. Choose any point $p\in H_{0}'$, and blow up $m-1$ times at $p$ and its infinitely near points on the proper transform of $H_{0}'$. Denote this morphism by $Y\to X\am$, and let $D_{Y}$ be the reduced total transform of $D\am$ minus the exceptional $(-1)$-curve $A$. Now $(Y,D_Y)$ is a minimal resolution of a del Pezzo surface $\bar{Y}$ of rank one and type $\langle 2;T_1\trp,T_2\trp,T_3\trp\rangle+\langle m;T_1^{*},T_{2}^{*},T_{3}^{*}\rangle+[(2)_{m-2}]$ if $p\in D\am\reg$ or $[T_{1}\trp,2,T_{2}]+\langle m,T_{1}^{*},T_{2}^{*},T_{3}^{*}\rangle+[(2)_{m-2}]*T_{3}$ if $p\not\in D\am\reg$, cf.\ Lemma \ref{lem:ht=2,untwisted}\ref{item:c=3} or \ref{item:tau=id_fork_chain}. We define $(X,D)$ by a vertical swap $(X,D)\sqto (Y,D_Y)$ whose unique base point lies on $A$. 
	
	For $i\in \{0,1\}$ let $H_i$ be the proper transform of $H_i'$, $i\in \{0,1\}$, and let $D_i$ be the connected component of $D$ containing $H_i$. Then  $D=D_0+D_1+E$, where $D_1$ is an admissible fork, and $E$ is zero or an admissible chain. Note that if $p\in D\am\reg$ then $(Y,D_Y)$, hence $(X,D)$, swaps vertically to the surface of type $2\rD_4$ from Example \ref{ex:ht=2}\ref{item:2D4_construction},  otherwise it swaps vertically to the surface of type $\rA_1+2\rA_3$ from \ref{ex:ht=2}\ref{item:A1+2A3_construction}.
	
	The log surface $(X,D)$ is a minimal log resolution of a del Pezzo surface of rank one if and only if $D_0$ contracts to a rational singularity and $\cf(H_0)+\cf(H_1)<2$, see Lemma \ref{lem:delPezzo_criterion}. This holds in particular if $D_0$ is an admissible or a log canonical fork: in this case, we get del Pezzo surfaces from Lemma \ref{lem:ht=2,untwisted}\ref{item:tau=id_forks}, \ref{item:c=3} or \ref{item:ht=2_bench}, respectively. 
\end{example}
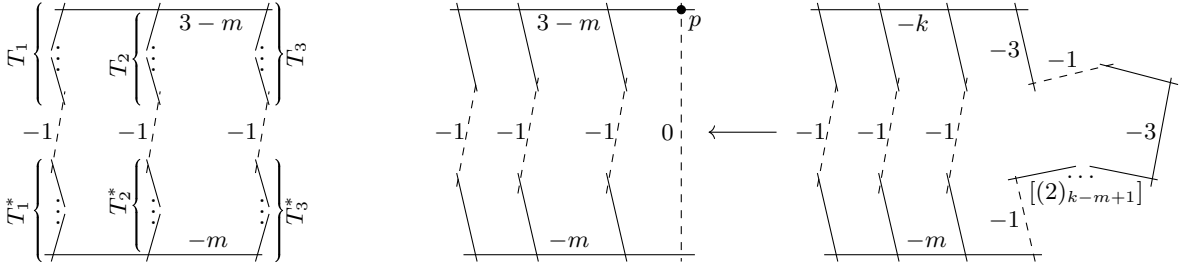
\begin{figure}[htbp]\vspace{-0.5em}
	\subcaptionbox{General Platonic $\A^{1}_{*}$-fiber space \label{fig:Platonic}}
	[.34\linewidth]
	{
		\begin{tikzpicture}[scale=0.9]
			\draw (-0.55,3.5) -- (2.65,3.5);
			\node at (1.75,3.3) {\small{$3-m$}};
			%
			\draw (-0.6,-0.2) -- (-0.4,0.5);
			\node at (-0.5,0.65) {$\vdots$};
			\draw (-0.4,0.6) -- (-0.6,1.3);
			\draw [decorate, decoration = {calligraphic brace}, thick] (-0.75,-0.2) --  (-0.75,1.3);
			\node[rotate=90] at (-1.1,0.55) {\small{$T_1^{*}$}}; 
			\draw[dashed] (-0.6,1.1) -- (-0.4,2.3);
			\node at (-0.8,1.7) {\small{$-1$}};
			\draw (-0.4,2.1) -- (-0.6,2.8);
			\node at (-0.5,2.9) {$\vdots$};
			\draw (-0.6,2.9) -- (-0.4,3.6);
			\draw [decorate, decoration = {calligraphic brace}, thick] (-0.75,2.1) --  (-0.75,3.6);
			\node[rotate=90] at (-1.1,2.85) {\small{$T_{1}$}};
			\draw (0.8,-0.2) -- (1,0.5);
			\node at (0.9,0.65) {$\vdots$};
			\draw (1,0.6) -- (0.8,1.3);
			\draw [decorate, decoration = {calligraphic brace}, thick] (0.7,-0.05) --  (0.7,1.3);
			\node[rotate=90] at (0.4,0.7) {\small{$T_2^{*}$}}; 
			\draw[dashed] (0.8,1.1) -- (1,2.3);
			\node at (0.6,1.7) {\small{$-1$}};
			\draw (1,2.1) -- (0.8,2.8);
			\node at (0.9,2.9) {$\vdots$};
			\draw (0.8,2.9) -- (1,3.6);
			\draw [decorate, decoration = {calligraphic brace}, thick] (0.7,2.1) --  (0.7,3.45);
			\node[rotate=90] at (0.4,2.7) {\small{$T_{2}$}};
			\draw (2.4,-0.2) -- (2.6,0.5);
			\node at (2.5,0.65) {$\vdots$};
			\draw (2.6,0.6) -- (2.4,1.3);
			\draw [decorate, decoration = {calligraphic brace}, thick] (2.7,1.3) -- (2.7,-0.2);
			\node[rotate=90] at (3,0.55) {\small{$T_3^{*}$}}; 
			\draw[dashed] (2.4,1.1) -- (2.6,2.3);
			\node at (2.2,1.7) {\small{$-1$}};
			\draw (2.6,2.1) -- (2.4,2.8);
			\node at (2.5,2.9) {$\vdots$};
			\draw (2.4,2.9) -- (2.6,3.6);
			\draw [decorate, decoration = {calligraphic brace}, thick] (2.7,3.6) -- (2.7,2.1);
			\node[rotate=90] at (3,2.85) {\small{$T_{3}$}};
			\draw (-0.7,-0.1) -- (2.5,-0.1);
			\node at (1.7,0.1) {\small{$-m$}};
		\end{tikzpicture}
	}
\subcaptionbox{Example \ref{ex:permissible_fork-example}: blowing up over a specific Platonic $\A^{1}_{*}$-fiber space
	\label{fig:permissible_fork-example}}[.65\linewidth]{
		\begin{tikzpicture}[scale=0.9]
		\begin{scope}
			\draw (-0.5,3.5) -- (3.1,3.5);
			\node at (1.25,3.3) {\small{$3-m$}};
			\filldraw (2.9,3.5) circle (0.06);
			\node at (3.1,3.3) {\small{$p$}};
			\draw (-0.1,-0.2) -- (-0.4,1.1);
			\draw[dashed] (-0.4,0.9) -- (-0.1,2.5);
			\node at (-0.5,1.7) {\small{$-1$}};
			\draw (-0.1,2.3) -- (-0.4,3.6);
			\draw (0.8,-0.2) -- (0.5,1.1);
			\draw[dashed] (0.5,0.8) -- (0.8,2.5);
			\node at (0.4,1.7) {\small{$-1$}};
			\draw (0.8,2.3) -- (0.5,3.6);
			\draw (2.1,-0.2) -- (1.8,1.1);
			\draw[dashed] (1.8,0.8) -- (2.1,2.5);
			\node at (1.7,1.7) {\small{$-1$}};
			\draw (2.1,2.3) -- (1.8,3.6);
			\draw[dashed] (2.9,3.6) -- (2.9,-0.2);
			\node at (2.7,1.7) {\small{$0$}};
			\draw (-0.3,-0.1) -- (3.1,-0.1);
			\node at (1.25,0.1) {\small{$-m$}};
			\draw[->]
			(4.3,1.7) -- (3.3,1.7);
		\end{scope}
		\begin{scope}[shift={(5.3,0)}]	
			\draw (-0.5,3.5) -- (2.7,3.5);
			\node at (1,3.3) {\small{$-k$}};
			%
			\draw (-0.1,-0.2) -- (-0.4,1.1);
			\draw[dashed] (-0.4,0.9) -- (-0.1,2.5);
			\node at (-0.5,1.7) {\small{$-1$}};
			\draw (-0.1,2.3) -- (-0.4,3.6);
			\draw (0.8,-0.2) -- (0.5,1.1);
			\draw[dashed] (0.5,0.8) -- (0.8,2.5);
			\node at (0.4,1.7) {\small{$-1$}};
			\draw (0.8,2.3) -- (0.5,3.6);
			\draw (1.8,-0.2) -- (1.5,1.1);
			\draw[dashed] (1.5,0.8) -- (1.8,2.5);
			\node at (1.4,1.7) {\small{$-1$}};
			\draw (1.8,2.3) -- (1.5,3.6);
			\draw (2.8,2.3) -- (2.5,3.6);
			\node at (2.35,2.9) {\small{$-3$}};
			\draw[dashed] (2.7,2.4) -- (3.95,2.7);
			\node at (3.2,2.75) {\small{$-1$}};
			\draw (3.75,2.7) -- (4.9,2.4);
			\draw (4.5,0.9) -- (4.8,2.5);
			\node at (4.35,1.7) {\small{$-3$}};
			\draw (4.6,1) -- (3.6,1.2);
			\node at (3.5,1.05) {\small{$\dots$}};
			\node at (3.55,0.8) {\small{$[(2)_{k-m+1}]$}};
			\draw (3.4,1.2) -- (2.4,1);
			\draw[dashed] (2.8,-0.2) -- (2.5,1.1);
			\node at (2.35,0.4) {\small{$-1$}};
			\draw (-0.3,-0.1) -- (2.9,-0.1);
			\node at (1.2,0.1) {\small{$-m$}};
			%
		\end{scope}
		\end{tikzpicture}
	}\vspace{-0.5em}	
	\caption{Examples \ref{ex:permissible_fork} and \ref{ex:permissible_fork-example}: constructing del Pezzo surfaces in Proposition \ref{prop:non-lt-more}\ref{item:non-lt-Platonic}.}\vspace{-1em}
	\label{fig:Platonic-examples}
\end{figure}

\begin{example}[A non--log canonical del Pezzo surface in case \ref{prop:non-lt-more}\ref{item:non-lt-Platonic}, see Figure \ref{fig:permissible_fork-example}]\label{ex:permissible_fork-example}
	We now give a concrete example of a non--log canonical del Pezzo surface of rank one obtained by the construction of Example \ref{ex:permissible_fork}.
	
	We keep the notation from Example \ref{ex:permissible_fork}. Take $T_i=[2]$ for $i\in \{1,2,3\}$. Choose $p\in D\am\reg$, fix $k\geq 3$ and define $\psi\colon X\to X\am$ so that $\psi^{-1}(p)\redd=[3,1,2,3,(2)_{k-m+1}]$, where the first tip meets $H_0$. Now $D_1=\langle m,[2],[2],[2]\rangle$, $E=[2,3,(2)_{k-m+1}]$ and $D_0$ is a sum of $H_1=[k]$ and four twigs: three of type $[2]$ and one of type $[3]$. The divisor $D_0$ contracts to a rational singularity, with fundamental cycle $D_0$ if $k\geq 4$ and $D_0+H_0$ if $k=3$. 
	
	Let $X\to \bar{X}$ be the contraction of $D$. Then $\bar{X}$ is a normal, non--log canonical  surface of rank one. Using \cite[3.2.3]{Flips_and_abundance} we compute $\cf(H_0)=1+\frac{1}{6k-11}$ and $\cf(H_1)=1-\frac{1}{2m-3}$, so by Lemma \ref{lem:delPezzo_criterion} the surface $\bar{X}$ is del Pezzo if and only if $m\leq 3k-5$.
\end{example}

We now move on to some examples specific to characteristic $2,3,5$. First, we construct a non--log canonical surface covered by \ref{prop:non-lt-more}\ref{item:non-lt-twisted}. The log canonical ones are classified in Lemma \ref{lem:ht=2_twisted-separable}, see Theorem \ref{thm:ht=1,2}\ref{item:ht=2_char=2}.

\begin{example}[Case \ref{prop:non-lt-more}\ref{item:non-lt-twisted}, $\cha\kk=2$, see Figure \ref{fig:ht=2_non-lc}]\label{ex:ht=2_non-lc}
	Assume $\cha\kk=2$. Fix $\nu\geq 3$ and let $(Y,D_Y)$ be the minimal log resolution of the surface of type $[2\nu-4]+2\nu \cdot \rA_1$ from \cite[p.\ 68]{Keel-McKernan_rational_curves}, see Example \ref{ex:ht=2_twisted_cha=2}. It admits a $\P^{1}$-fibration such that the $-(2\nu-4)$-curve, call it $H_Y$, is a $2$-section, and all $\nu$ degenerate fibers are supported on chains $[2,1,2]$, see Figure \ref{fig:KM_surface-intro}, meeting $H_Y$ in the middle component. 
	
	Let $A_1\dots, A_{\nu}$ be the vertical $(-1)$-curves and let $L_i$ be a $(-2)$-curve meeting $A_i$. We construct a vertical swap $(X,D)\sqto (Y,D_Y)$, as follows. First, blow up $k\geq 2$ times over $A_1\cap H_Y$, at the proper transforms of $H_Y$, and at the common point of the exceptional $(-1)$- and $(-2)$-curve. Next, blow up at the preimage of $L_2\cap A_2$ and at  its infinitely near point on the proper transform of $A_2$, and let $A$ be the new $(-1)$-curve. Choose a point $p\in A$ away from the other components of the total transform of $D_Y$; and blow up at $p$ and its infinitely near point on the proper transform of $A$. Eventually, blow up at the common point of the two exceptional curves of $p$. Denote the resulting surface by $X$, and let $D$ be the reduced total transform of $D_Y$ without the  $(-1)$-curves, see Figure \ref{fig:ht=2_non-lc}. Then $D=D_0+D\lt$, where $D\lt=(2\nu-4)\cdot [2]+[3]+\langle 2;[2],[2],[3,(2)_{k-2}]\rangle$; and $D_0$ is a sum of a chain $[3,3]$, twigs $[2]$, $[3,2]$ meeting the first $(-3)$-curve and twigs $[2]$, $[2,m]$ meeting the second one, where $m=2\nu-4+k$ and the $(-m)$-curve is the proper transform of $H_Y$, call it $H$. 	
	
	The $\P^1$-fibration of $Y$ pulls back to a $\P^1$-fibration of $X$ such that $D\hor=H$ is a $2$-section.  The divisor $D_0$ contracts to a rational singularity, with fundamental cycle $D_0$. Let $X\to \bar{X}$ be the contraction of $D$.  Using \cite[3.1.10]{Flips_and_abundance} we compute $\cf(H)=1-\frac{25}{150m-151}<1$, so by Lemma \ref{lem:delPezzo_criterion} the surface $\bar{X}$ is del Pezzo of rank one. 
	\begin{figure}[htbp]\vspace{-1.5em}
	\begin{tikzpicture}
		\begin{scope}[scale=1.2]
			\draw[very thick] (0,1.2) -- (3.1,1.2);
			\node at (2.1,1.4) {\small{$H_Y$}};
			\node at (2.1,1) {\small{$4-2\nu$}};
			\draw (0.2,2.5) -- (0,1.5);
			\node at (0.3,2) {\small{$L_1$}};
			\draw[dashed] (0,1.7) -- (0.2,0.7);
			\node at (0.3,1.4) {\small{$A_1$}};
			\filldraw (0.1,1.2) circle (0.05);
			\draw (0.2,0.9) -- (0,-0.1);
			\draw (1.2,2.5) -- (1,1.5);
			\node at (1.3,2) {\small{$L_2$}};
			\draw[dashed] (1,1.7) -- (1.2,0.7);
			\node at (1.3,1.4) {\small{$A_2$}};
			\filldraw (1.02,1.6) circle (0.05);
			\draw (1.2,0.9) -- (1,-0.1);
			\node at (2.1,2) {\Large{$\cdots$}};
			\draw (3,2.5) -- (2.8,1.5);
			\node at (3.1,2) {\small{$L_{\nu}$}};
			\draw[dashed] (2.8,1.7) -- (3,0.7);
			\node at (3.1,1.4) {\small{$A_{\nu}$}};
			\draw (3,0.9) -- (2.8,-0.1);
			\draw [decorate, decoration = {calligraphic brace}, thick] (0,2.6) -- (3.2,2.6);
			\node at (1.6,2.8) {\scriptsize{$\nu$ fibers}};
			\draw[->,decorate,decoration={snake,amplitude=1.3pt}] (5,1.2) -- (4,1.2);
		\end{scope}		
		\begin{scope}[shift={(12,0)}]
			\draw[very thick] (-0.6,1.1) to[out=30,in=180] (0,1.2) -- (3.3,1.2);
			\node at (2.2,1.4) {\small{$H$}};
			\node at (2.2,1) {\small{$4-2\nu-k$}};
			\draw (-0.4,1.1) -- (-1.4,1.3);
			\draw[dashed] (-1.2,1.3) -- (-2.2,1.1);
			\draw (-2,1.1) -- (-3,1.3);
			\node at (-2.4,1.4) {\small{$-3$}};
			\draw (-2.8,1.3) -- (-3.8,1.1);
			\node at (-3.85,1.25) {$\dots$};
			\node at (-3.9,1.5) {\small{$[(2)_{k-2}]$}};
			\draw (-4,1.1) -- (-5,1.3);
			\draw (-4.8,2.5) -- (-5,1.5);
			\draw (-5,1.7) -- (-4.8,0.7);
			\draw (-4.8,0.9) -- (-5,-0.1);
			\draw (0,3.3) -- (0.2,2.3);
			\draw (0.2,3.1) -- (-0.8,3.3);
			\node at (-0.6,3) {\small{$-3$}};
			\draw (0.2,2.5) -- (0,1.5);
			\node at (0.35,2) {\small{$-3$}};
			\draw (0.2,2.1) -- (-0.8,2.3);
			\draw[dashed] (-0.6,2.3) -- (-1.6,2.1);
			\draw (-1.4,2.1) -- (-2.4,2.3);
			\node at (-1.9,2.35) {\small{$-3$}}; 
			\draw (0,1.7) -- (0.2,0.7);
			\node at (0.35,1.4) {\small{$-3$}};
			\draw (0.2,0.9) -- (0,-0.1);
			\draw (1.2,2.5) -- (1,1.5);
			\draw[dashed] (1,1.7) -- (1.2,0.7);
			\draw (1.2,0.9) -- (1,-0.1);
			\node at (2.2,2) {\Large{$\cdots$}};
			\draw (3.2,2.5) -- (3,1.5);
			\draw[dashed] (3,1.7) -- (3.2,0.7);
			\draw (3.2,0.9) -- (3,-0.1);
			\draw [decorate, decoration = {calligraphic brace}, thick] (1,2.6) -- (3.4,2.6);
			\node at (2.4,2.9) {\scriptsize{$\nu-2$ fibers}};	
		\end{scope}	
	\end{tikzpicture}
	\caption{Example \ref{ex:ht=2_non-lc}, $\cha\kk=2$: a non--log canonical del Pezzo surface covered by Theorem \ref{thm:ht=1,2}\ref{item:ht=2_char=2} and Proposition \ref{prop:non-lt-more}\ref{item:non-lt-twisted}, swapping vertically to the surface from Example \ref{ex:ht=2_twisted_cha=2}.}\vspace{-1em}
	\label{fig:ht=2_non-lc}
\end{figure}
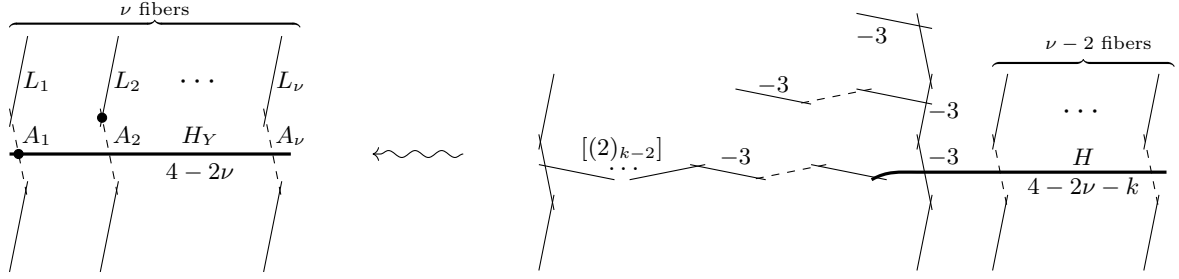
\end{example}

We now focus on case \ref{prop:non-lt-more}\ref{item:non-lt-ht>=3}, that is, on non-lt del Pezzo surfaces of rank one and  height at least $3$, described in the exceptional cases \ref{item:non-lt_descendant} and \ref{item:non-lt-swap} of Proposition \ref{prop:non-log-terminal}. Surfaces in case \ref{prop:non-log-terminal}\ref{item:non-lt_descendant} admit descendants with elliptic boundary: the log canonical ones are classified in Theorem \ref{thm:GK}, we now construct a non--log canonical one. 

\begin{example}[Case \ref{prop:non-log-terminal}\ref{item:non-lt_descendant}: non--log canonical del Pezzo surfaces with elliptic-boundary descendants, see Figure \ref{fig:cusp}]\label{ex:non-lt-descendant}
	Let $\bar{Y}$ be a canonical del Pezzo surface of rank one, and let $\bar{T}\subseteq \bar{Y}\reg$ be a cuspidal curve with $p_{a}(\bar{T})=1$. Such log surfaces $(\bar{Y},\bar{T})$ are classified in \cite[Proposition 1.5]{PaPe_MT}, see Table \ref{table:canonical}. By Noether's formula we have $\bar{T}^2=9-\#\cS(\bar{Y})$, where $\cS(\bar{Y})$ is the singularity type of $\bar{Y}$, see formula \eqref{eq:Noether}. Put $k=\#\cS(\bar{Y})-3=6-\bar{T}$. 
	
	Fix $b\geq 2$ and fix a type $V$ of an admissible chain. Let $\nu\colon (Y',D')\to (\bar{Y},\bar{T})$ be the minimal log resolution, and let $T'=\nu^{-1}_{*}\bar{T}$. Then $T'=[k]$ and the connected component of $D'$ containing $T'$ is a fork $\langle 1;[2],[3],[k]\rangle$. Let $\psi\colon X\to Y'$ be a sequence of blowups over $T'\cap (D'-T')$ such that $\psi^{*}T'\redd$ is a chain $[k,(2)_{b-2}]*[V^{*},1,V]$ whose first tip is $T\de \psi^{-1}_{*}T'$ and last tip meets $\psi^{-1}_{*}(D'-T')$, cf.\ Lemma \ref{lem:cuspidal_resolution}\ref{item:C_[rest]}. Let $A$ be the exceptional $(-1)$-curve of $\psi$, let $D=\psi^{*}D'\redd-A$, and let $D_{T}$ be the connected component of $D$ containing $T$. We have $D=D_0+D_T+D_Y$, where $D_0=\langle b;[2],[3],V\rangle$, $D_{T}=[k,(2)_{b-2}]*V^{*}$ and $D_Y$ is the preimage of $\Sing \bar{Y}$, hence a sum of $(-2)$-chains and forks. Since $D_0+A+D_T-T$ contracts to a smooth point, by Artin's criterion its subdivisor $D_0$ contracts to a rational singularity, see \cite[Example 7.1.5(i)]{Nemethi_book}. If $d(V)>6$, it is not lc.
	
	Let $X\to \bar{X}$ be the contraction of $D$; so $\bar{X}$ is a normal surface of rank one. Since the connected component $D_T$ of $D$ containing $T$ is an admissible chain, we have $\cf(T)<1$, so by Lemma \ref{lem:GK_intro}\ref{item:GK-intro-cf} $\bar{X}$ is del Pezzo. 
	
	Assume that $\bar{Y}$ is as in Proposition \ref{prop:non-log-terminal}\ref{item:non-lt_descendant}, that is, $\#\cS(\bar{Y})\geq 6$ and $\bar{Y}$ is not of type $\rE_6$, $\rE_7$ or $\rE_8$. Note that in this case, a cuspidal curve $\bar{T}$ as above exists  if and only if $\cha\kk\in \{2,3,5\}$, see Table \ref{table:canonical}. We claim that $\height(\bar{X})\geq 3$, so $\bar{X}$ is indeed as in Proposition \ref{prop:non-lt-more}\ref{item:non-lt-ht>=3}. We have $-K_{\bar{Y}}=\bar{T}$, see Lemma \ref{lem:GK_intro}\ref{item:GK-intro-Y}. Pulling this equality back to $X$, we see that $-K_{X}$ is linearly equivalent to an effective divisor supported on $D_0+A+D_{T}$, such that the multiplicity of $A$ in this divisor is at least $3$. Let $F$ be a fiber of any $\P^1$-fibration of $X$. Since by adjunction $-K_{X}\cdot F=2$, we conclude that $A$ is vertical. Applying elementary vertical swaps we can assume that $V=[2]$. Then $\bar{X}$ is log canonical, and $\height(\bar{X})\geq 3$ by Lemma \ref{lem:GK_exceptions}, as needed.
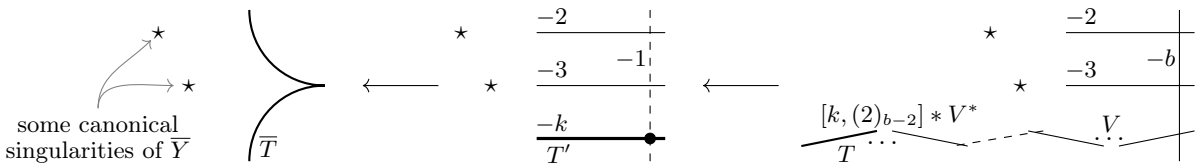
\begin{figure}[htbp]
	\begin{tikzpicture}\vspace{-1em}
		\begin{scope}
			\draw[thick] (1,2) to[out=-90,in=180] (2,1) to[out=180,in=90] (1,0);
			\node at (1.25,0.2) {\small{$\bar{T}$}};
			\node at (-0.2,1.7) {$\star$};
			\node at (0.2,1) {$\star$};
			\draw[gray, ->] (-1,0.7) to[out=90,in=-135] (-0.3,1.6);
			\draw[gray, ->] (-1,0.7) to[out=90,in=180] (0,1);
			\node at (-1,0.5) {\small{some canonical}};
			\node at (-1,0.15) {\small{singularities of $\bar{Y}$}};
			\draw[<-] (2.5,1) -- (3.5,1);  
		\end{scope}
		\begin{scope}[shift={(4,0)}]
			\node at (-0.2,1.7) {$\star$};
			\node at (0.2,1) {$\star$};
			\draw[dashed] (2.3,2) -- (2.3,0);
			\node at (2.05,1.35) {\small{$-1$}};
			\draw (0.8,1.7) -- (2.5,1.7);
			\node at (1,1.9) {\small{$-2$}};
			\draw (0.8,1) -- (2.5,1);
			\node at (1,1.2) {\small{$-3$}};
			\draw[very thick] (0.8,0.3) -- (2.5,0.3);
			\node at (1,0.5) {\small{$-k$}};
			\node at (1.1,0.1) {\small{$T'$}};
			\filldraw (2.3,0.3) circle (0.07);
			\draw[<-] (3,1) -- (4,1);
		\end{scope}
		\begin{scope}[shift={(11,0)}]
			\node at (-0.2,1.7) {$\star$};
			\node at (0.2,1) {$\star$};
			\draw (2.3,2) -- (2.3,0);
			\node at (2.05,1.35) {\small{$-b$}};
			\draw (0.8,1.7) -- (2.5,1.7);
			\node at (1,1.9) {\small{$-2$}};
			\draw (0.8,1) -- (2.5,1);
			\node at (1,1.2) {\small{$-3$}};
			\draw (2.5,0.4) -- (1.5,0.2);
			\node at (1.4,0.3) {$\dots$};
			\node at (1.4,0.5) {\small{$V$}};
			\draw (1.3,0.2) -- (0.3,0.4);
			\draw[dashed] (0.5,0.4) -- (-0.7,0.2);
			\draw (-0.5,0.2) -- (-1.5,0.4);
			\node at (-1.6,0.25) {$\dots$};
			\node at (-1.4,0.6) {\small{$[k,(2)_{b-2}]*V^{*}$}};
			\draw[thick] (-1.7,0.4) -- (-2.7,0.2);
			\node at (-2.1,0.1) {\small{$T$}};
		\end{scope}
	\end{tikzpicture}
	\caption{Example \ref{ex:non-lt-descendant}: a non-lc del Pezzo surface with an elliptic-boundary descendant.}
	\label{fig:cusp}
\end{figure}
\end{example}

Eventually, we consider case  \ref{prop:non-log-terminal}\ref{item:non-lt-swap}. First, we construct the vertically primitive surface $\bar{Y}$.

\begin{example}[Vertically primitive surface of type {$4\rA_1+\rA_3+[3]$} in Proposition  \ref{prop:non-log-terminal}\ref{item:non-lt-swap}, see Figure \ref{fig:ht=3-intro}]\label{ex:ht=3} \ 
	Assume $\cha\kk=2$. As in Example \ref{ex:ht=2_twisted_cha=2}, let $\cc\subseteq \P^2$ be a smooth conic and let $\ll_{1},\ll_{2},\ll_{3}$ be lines tangent to $\cc$. Since $\cha\kk=2$, these lines meet at one point $p_0$. For $j\in \{1,2,3\}$ write $\{p_j\}=\ll_j\cap \cc$. Blow up twice over $p_0$ and twice over each $p_j$, each time on the proper transform of $\ll_1$ and $\cc$, respectively. The resulting surface $Y$ is a minimal resolution of a del Pezzo surface $\bar{Y}$ of rank one and type $4\rA_1+\rA_3+[3]$, together with a $\P^1$-fibration of height $3$, given by the pencil of lines passing through $p_0$, see Figure \ref{fig:ht=3-intro} or the middle part of Figure \ref{fig:ht=3-remark}. 
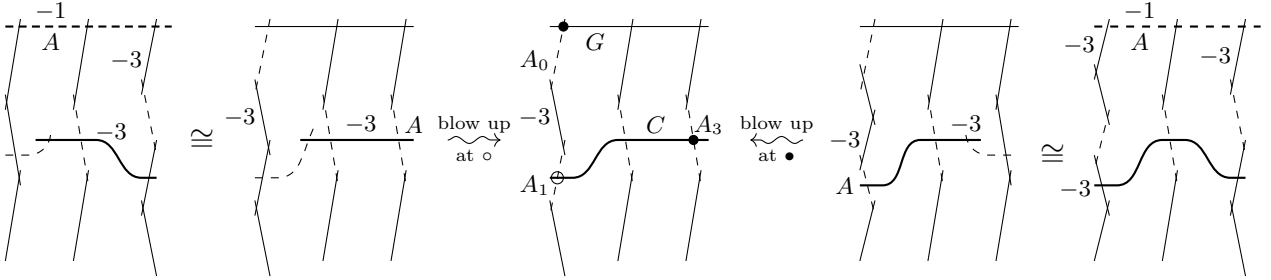
\begin{figure}[htbp]
	\begin{tikzpicture}
		\begin{scope}
			\draw (0,3.1) -- (2.1,3.1);
			\node at (0.6,2.9) {\small{$G$}};
			\draw[dashed] (0.2,3.2) -- (0,2.2);
			\node at (-0.2,2.65) {\small{$A_0$}};
			\filldraw (0.18,3.1) circle (0.06);
			\draw (0,2.4) -- (0.2,1.4);
			\node at (-0.2,1.9) {\small{$-3$}};
			\node at (-0.2,1) {\small{$A_1$}};
			\draw[dashed] (0.2,1.6) -- (0,0.6); 
			\draw (0,0.8) -- (0.2,-0.2);
			\draw (1.1,3.2) -- (0.9,2);
			\draw[dashed] (0.9,2.2) -- (1.1,1);			
			\draw (1.1,1.2) -- (0.9,0); 
			\draw (2,3.2) -- (1.8,2);
			\draw[dashed] (1.8,2.2) -- (2,1);
			\node at (2.1,1.8) {\small{$A_3$}};	
			\draw (2,1.2) -- (1.8,0); 
			\draw[thick] (0,1.1) -- (0.3,1.1) to[out=0,in=180] (0.9,1.6) -- (2.1,1.6);
			\node at (1.4,1.8) {\small{$C$}};
			\draw (0.1,1.1) circle (0.08);
			\filldraw (1.9,1.6) circle (0.06);
			\draw[->,decorate,decoration={snake,amplitude=1.3pt}] (3.35,1.6) -- (2.65,1.6);
			\node at (3,1.8) {\scriptsize{blow up}};
			\node at (3,1.4) {\scriptsize{at $\bullet$}}; 
			\draw[->,decorate,decoration={snake,amplitude=1.3pt}] (-1.35,1.6) -- (-0.65,1.6);
			\node at (-1,1.8) {\scriptsize{blow up}};
			\node at (-1,1.4) {\scriptsize{at $\circ$}}; 
			%
		\end{scope}
		\begin{scope}[shift={(-3.9,0)}]
			\draw (0,3.1) -- (2.1,3.1);
			\draw[dashed] (0.2,3.2) -- (0,2.2);
			\draw (0,2.4) -- (0.2,1.4);
			\node at (-0.2,1.9) {\small{$-3$}};
			\draw (0.2,1.6) -- (0,0.6); 
			\draw[dashed] (0,1.1) -- (0.3,1.1) to[out=0,in=-120] (0.8,1.8);
			\draw (0,0.8) -- (0.2,-0.2);
			\draw (1.1,3.2) -- (0.9,2);
			\draw[dashed] (0.9,2.2) -- (1.1,1);			
			\draw (1.1,1.2) -- (0.9,0); 
			\draw (2,3.2) -- (1.8,2);
			\draw[dashed] (1.8,2.2) -- (2,1);
			\node at (2.1,1.8) {\small{$A$}};
			\draw (2,1.2) -- (1.8,0); 
			\draw[thick] (0.6,1.6) -- (2.1,1.6);
			\node at (1.4,1.8) {\small{$-3$}};
			\node at (-0.7,1.6) {\Large{$\cong$}};
			%
		\end{scope}
		\begin{scope}[shift={(-7.2,0)}]
			\draw[thick, densely dashed] (0,3.1) -- (2.2,3.1);
			\node at (0.6,2.9) {\small{$A$}};
			\node at (0.6,3.3) {\small{$-1$}};
			\draw (0.2,3.2) -- (0,2);
			\draw (0,2.2) -- (0.2,1);
			\draw (0.2,1.2) -- (0,0);
			\draw[dashed] (0,1.4) -- (0.3,1.4) to[out=0,in=-90] (0.6,1.7);
			\draw (1.1,3.2) -- (0.9,2);
			\draw[dashed] (0.9,2.2) -- (1.1,1);			
			\draw (1.1,1.2) -- (0.9,0); 
			\draw (2,3.2) -- (1.8,2.2);
			\node at (1.6,2.65) {\small{$-3$}};
			\draw[dashed] (1.8,2.4) -- (2,1.4);
			\draw (2,1.6) -- (1.8,0.6);
			\draw (1.8,0.8) -- (2,-0.2);
			\draw[thick] (0.4,1.6) -- (1.2,1.6) to[out=0,in=180]  (1.8,1.1) --  (2,1.1);
			\node at (1.4,1.7) {\small{$-3$}};
		\end{scope}
		\begin{scope}[shift={(4.1,0)}]
		 	\draw (0,3.1) -- (2.1,3.1);
		 	\draw[dashed] (0.2,3.2) -- (0,2.2);
		 	\draw (0,2.6) -- (0.2,1.8);
		 	\draw (0.2,2) -- (0,1.2);
		 	\node at (-0.2,1.6) {\small{$-3$}};
		 	\draw[dashed] (0,1.4) -- (0.2,0.6);
		 	\node at (-0.2,1) {\small{$A$}};
		 	\draw (0.2,0.8) -- (0,0);
		 	\draw (1.1,3.2) -- (0.9,2);
		 	\draw[dashed] (0.9,2.2) -- (1.1,1);			
		 	\draw (1.1,1.2) -- (0.9,0); 
		 	\draw (2,3.2) -- (1.8,2);
		 	\draw (1.8,2.2) -- (2,1);
		 	\draw[dashed] (2,1.4) -- (1.7,1.4) to[out=180,in=-90] (1.4,1.7);
		 	\draw (2,1.2) -- (1.8,0); 
		 	\draw[thick] (0,1) -- (0.3,1) to[out=0,in=180] (0.8,1.6) -- (1.6,1.6);
		 	\node at (1.4,1.8) {\small{$-3$}}; 
		 	\node at (2.55,1.4) {\Large{$\cong$}};
		 	%
		 \end{scope}
	 	 \begin{scope}[shift={(7.2,0)}]
	 	 	\draw[dashed, thick] (0,3.1) -- (2.2,3.1);
	 	 	\node at (0.6,2.9) {\small{$A$}};
	 	 	\node at (0.6,3.3) {\small{$-1$}};
	 	 	\draw (0.2,3.2) -- (0,2.4);
	 	 	\node at (-0.2,2.85) {\small{$-3$}};
	 	 	\draw (0,2.6) -- (0.2,1.8);
	 	 	\draw[dashed] (0.2,2) -- (0,1.2);
	 	 	\draw (0,1.4) -- (0.2,0.6);
	 	 	\node at (-0.25,0.95) {\small{$-3$}};
	 	 	\draw (0.2,0.8) -- (0,0);
	 	 	\draw (1.1,3.2) -- (0.9,2);
	 	 	\draw[dashed] (0.9,2.2) -- (1.1,1);			
	 	 	\draw (1.1,1.2) -- (0.9,0); 
	 	 	\draw (2,3.2) -- (1.8,2.2);
	 	 	\node at (1.6,2.7) {\small{$-3$}};
	 	 	\draw[dashed] (1.8,2.4) -- (2,1.4);
	 	 	\draw (2,1.6) -- (1.8,0.6);
	 	 	\draw (1.8,0.8) -- (2,-0.2);
	 	 	\draw[thick] (0,1) -- (0.3,1) to[out=0,in=180] (0.9,1.6) -- (1.2,1.6) to[out=0,in=180] (1.8,1.1) -- (2,1.1);
	 	 \end{scope}	
	\end{tikzpicture}
	\caption{Remark \ref{rem:ht=3}: minimal log resolution of the surface $\bar{Y}$ from Example \ref{ex:ht=3} (in the middle), and two log surfaces $(X\am,D\am)$ from \cite[Proposition 3.2(5)]{PaPe_MT}, each equipped with two $\P^1$-fibrations: one from loc.\ cit.\ and the other pulled back from $Y$.}\vspace{-1em}
	\label{fig:ht=3-remark}
\end{figure}
\end{example} 
	Proposition \ref{prop:non-log-terminal}\ref{item:non-lt-swap} asserts that every non-lt del Pezzo surface $\bar{X}$ of rank one, height at least $3$, and with no descendant with elliptic boundary swaps vertically to the surface $\bar{Y}$ constructed in Example \ref{ex:ht=3} above. In fact, the proof of Proposition  \ref{prop:non-log-terminal} gives the following additional constraints on those swaps.

	\begin{remark}[Additional restrictions in \ref{prop:non-log-terminal}\ref{item:non-lt-swap}, see Figure \ref{fig:ht=3-remark}]\label{rem:ht=3}
		Let $\bar{X}$ be a del Pezzo surface as in Proposition \ref{prop:non-log-terminal}\ref{item:non-lt-swap}. Its minimal log resolution $(X,D)$ admits a morphism $X\to Y$ onto a minimal resolution of the surface from Example \ref{ex:ht=3}. In the proof of Proposition \ref{prop:non-lt-more}, \hyperref[case:5]{Case (5)}, we have constructed this morphism as a composition of the almost minimalization $\psi\colon (X,D)\to (X\am,D\am)$ with some further contraction $\tau\colon X\am \to Y$. The possible log surfaces $(X\am,D\am)$ are described in \cite[Proposition 3.2(5)]{PaPe_MT}, and shown in the left and right of Figure \ref{fig:ht=3-remark}. In those figures, $D\am$ is the sum of solid lines and the $(-1)$-curve $A$, cf.\ \cite[Figures 4(b) and 4(c)]{PaPe_MT}. The morphism $\tau$ is the blowup at the point \enquote{$\circ$} or two points \enquote{$\bullet$}. 
		
		More precisely, the morphism $\tau$ is constructed as follows. We keep notation from Example \ref{ex:ht=3}, let $A_i$ be the $(-1)$-curve over $p_i$, let $C$ be the proper transform of $\cc$, and let $G$ be the first exceptional curve over $p_0$, see Figure \ref{fig:ht=3-remark}. Now by construction, $\tau$ is a blowup at either:
		\begin{enumerate}
			\item[($\circ$)] \phantomsection\label{item:5-1}  $A_1\cap C$ : this is case (5.1) in  \cite[Proposition 3.2]{PaPe_MT}, shown in the left part of Figure \ref{fig:ht=3-remark}, or
			\item[($\bullet$)] \phantomsection\label{item:5-2}   $A_2\cap C$ and $A_0\cap G$: this is case (5.2) loc.\ cit., shown in the right part of Figure \ref{fig:ht=3-remark}.
		\end{enumerate}
		
		The map $\psi^{-1}$ has a unique base point, call it $p$. Since $D$ contains no $(-1)$-curves, the point $p$ lies on the $(-1)$-curve $A\subseteq D\am$. We remark that $p$ can be either a smooth point of $D\am$, as in Example \ref{ex:ht=3-swap}, or a singular one, as in all cases of Proposition \ref{prop:ht=3}.
		
		In the following discussion it will be more convenient to use the $\P^1$-fibration of $X\am$ constructed in \cite[Proposition 3.2(5)]{PaPe_MT}. 
		It is induced by the pencil of conics tangent to $\cc$ at $p_2$, $p_3$ in case \hyperref[item:5-1]{($\circ$)}, and at $p_1$, $p_3$ in case \hyperref[item:5-2]{($\bullet$)}. Its degenerate fibers are shown in the leftmost and rightmost parts of Figure \ref{fig:ht=3-remark}.
	\end{remark}
	
 Surfaces $\bar{X}$ as in Proposition \ref{prop:non-log-terminal}\ref{item:non-lt-swap} can be log canonical or not. The log canonical ones are classified in Proposition \ref{prop:ht=3} below, now we construct an example of the non-log canonical one.

\begin{example}[Non--log canonical surface in case \ref{prop:non-log-terminal}\ref{item:non-lt-swap}, see Figure \ref{fig:ht=3-swap}]\label{ex:ht=3-swap}
	
	Let $(X\am,D\am)$ be as in case \hyperref[item:5-1]{($\circ$)} of Remark \ref{rem:ht=3}, and let $A\subseteq D\am$ be the $(-1)$-curve. Choose $p\in A\setminus (D\am-A)$ and blow up over $p$ in such a way that the exceptional divisor is a chain $[2,1,3,(2)_{k-3}]$ meeting the proper transform $H$ of $A$ in the first tip.

	Denote the resulting surface by $X$, and as usual let $D$ be the reduced total transform of $D\am$ without the $(-1)$-curve. The connected components of $D$ are chains $[3,2,2]$, $[2]$, $[3,(2)_{k-3}]$ and a tree $D_0$, with one branching component $H=[k]$ and maximal twigs $[2,2,2]$, $[3]$, $[2]$, $[2]$, see Figure \ref{fig:ht=3-swap}. The tree $D_0$ contracts to a rational singularity, with fundamental cycle $D_0$ if $k\geq 4$ and $D_0+H+T$ if $k=3$, where $T$ is a chain $[2,2]$ meeting $H$. 
	Thus contracting $D$ we get a normal, non--log canonical surface $\bar{X}$ of rank one.
	
	Consider the pullback to $X$ of the $\P^1$-fibration of $X\am$ constructed in Remark \ref{rem:ht=3}, see the left part of Figure \ref{fig:ht=3-remark} or Figure \ref{fig:ht=3-swap} below. Then $D\hor$ consists of a $1$-section $H$ and a $2$-section, call it $H_2$, which is the middle component of the chain $[3,2,2]$. Using \cite[3.2.3,  3.2.1]{Flips_and_abundance} we compute that $\cf(H)=1+\frac{5}{12k-20}$ and $\cf(H_2)=\frac{2}{7}$, which substituted to the inequality \eqref{eq:ld_phi_H} shows that $\bar{X}$ is del Pezzo if and only if $k\geq 6$.
\smallskip
	
	We have $\height(\bar{X})\leq 3$. We claim that the equality holds. 	Suppose  $\height(\bar{X})=1$. Since $D$ has 4 connected components, by Lemma \ref{lem:ht=1_basics}\ref{item:ht=1_Sigma} the witnessing $\P^1$-fibration has at least 3 degenerate fibers. By Lemma \ref{lem:ht=1_basics}\ref{item:ht=1_nu} $D\hor=H$. The degenerate fiber containing the twig $[3]$ has exactly one $(-1)$-curve, so it is supported on a chain $[3,1,2,2]$, which is impossible as $D$ has no connected component of type $[2,2]$. The fact that $\height(\bar{X})\neq 2$ can be shown directly, too, but one can also use Theorem \ref{thm:ht=1,2}. Indeed, since $H$ is the unique branching component of $D$ and $\beta_{D}(H)=4$, we see that $\bar{X}$ does not swap vertically to any surface from Examples \ref{ex:ht=2}, \ref{ex:ht=2_meeting}; and since it has only one canonical singularity, of type $[2]$, it cannot swap vertically to the one from Example \ref{ex:ht=2_twisted_cha=2}, either. 
	
	We also note that $\bar{X}$ has no descendant with elliptic boundary. Indeed, otherwise by Theorem \ref{thm:GK} at least $6$ components of $D$ would lie in connected components of $D$ which are $(-2)$-chains or forks, while there is only one such. As a consequence, the surface $\bar{X}$ is covered only by case \ref{item:non-lt-swap} of Proposition \ref{prop:non-log-terminal}.
	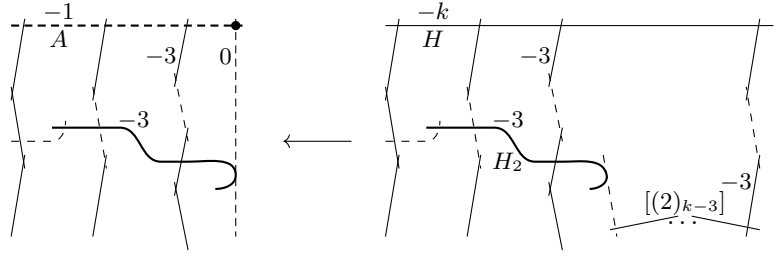
\begin{figure}[htbp]\vspace{-1em}
	\begin{tikzpicture}[scale=0.9]
		\begin{scope}
			\begin{scope}
				\draw[thick, densely dashed] (0,3.1) -- (3.4,3.1);
				\node at (0.7,2.9) {\small{$A$}};
				\node at (0.7,3.3) {\small{$-1$}};
				\filldraw (3.3,3.1) circle (0.06);
				\draw[densely dashed] (3.3,3.2) -- (3.3,0);
				\node at (3.15,2.65) {\small{$0$}}; 
				\draw (0.2,3.2) -- (0,2);
				\draw (0,2.2) -- (0.2,1);
				\draw (0.2,1.2) -- (0,0);
				\draw[dashed] (0,1.4) -- (0.6,1.4) to[out=0,in=-90] (0.8,1.7);
				\draw (1.4,3.2) -- (1.2,2);
				\draw[dashed] (1.2,2.2) -- (1.4,1);			
				\draw (1.4,1.2) -- (1.2,0); 
				\draw (2.6,3.2) -- (2.4,2.2);
				\node at (2.2,2.65) {\small{$-3$}};
				\draw[dashed] (2.4,2.4) -- (2.6,1.4);
				\draw (2.6,1.6) -- (2.4,0.6);
				\draw (2.4,0.8) -- (2.6,-0.2);
				\draw[thick] (0.6,1.6) -- (1.6,1.6) to[out=0,in=180]  (2.2,1.1) --  (2.6,1.1) to[out=0,in=90] (3.3,0.9) to[out=-90,in=0] (3,0.7);
				\node at (1.8,1.7) {\small{$-3$}};
			\end{scope}
			\draw[<-] (4,1.4) -- (5,1.4);
		\end{scope}
		\begin{scope}[shift={(5.5,0)}]
			\draw (0,3.1) -- (5.7,3.1);
			\node at (0.7,2.9) {\small{$H$}};
			\node at (0.7,3.3) {\small{$-k$}};
			\draw (0.2,3.2) -- (0,2);
			\draw (0,2.2) -- (0.2,1);
			\draw (0.2,1.2) -- (0,0);
			\draw[dashed] (0,1.4) -- (0.6,1.4) to[out=0,in=-90] (0.8,1.7);
			\draw (1.4,3.2) -- (1.2,2);
			\draw[dashed] (1.2,2.2) -- (1.4,1);			
			\draw (1.4,1.2) -- (1.2,0); 
			\draw (2.6,3.2) -- (2.4,2.2);
			\node at (2.2,2.65) {\small{$-3$}};
			\draw[dashed] (2.4,2.4) -- (2.6,1.4);
			\draw (2.6,1.6) -- (2.4,0.6);
			\draw (2.4,0.8) -- (2.6,-0.2);
			\draw[dashed] (3.2,1.2) -- (3.4,0);
			\draw (3.3,0.1) -- (4.3,0.3);
			\node at (4.4,0.2) {$\dots$};
			\node at (4.4,0.5) {\small{$[(2)_{k-3}]$}};
			\draw (4.5,0.3) -- (5.5,0.1);
			\draw (5.3,0) -- (5.5,1.2);
			\node at (5.15,0.8) {\small{$-3$}};
			\draw[dashed] (5.5,1) -- (5.3,2.2);
			\draw (5.3,2) -- (5.5,3.2);
			\draw[thick] (0.6,1.6) -- (1.6,1.6) to[out=0,in=180]  (2.2,1.1) --  (2.6,1.1) to[out=0,in=100] (3.25,0.9) to[out=-80,in=0] (3,0.7);
			\node at (1.8,1.7) {\small{$-3$}};
			\node at (1.8,1.1) {\small{$H_2$}};
		\end{scope}
	\end{tikzpicture}
	\caption{Example \ref{ex:ht=3-swap}: a non--log canonical del Pezzo surface as in Proposition  \ref{prop:non-log-terminal}\ref{item:non-lt-swap}.}
\label{fig:ht=3-swap}
\end{figure}
\end{example} 
\vspace{-1em}
We end this section with the classification of log canonical surfaces as in  Proposition \ref{prop:non-log-terminal}\ref{item:non-lt-swap}. We list their singularity types together with the structure of a witnessing $\P^1$-fibration obtained by pulling back the one from Remark \ref{rem:ht=3}. We use Notation \ref{not:fibrations}, and furthermore underline the numbers corresponding to the $2$-sections.

\begin{proposition}[Log canonical surfaces in \ref{prop:non-log-terminal}\ref{item:non-lt-swap}]
	\label{prop:ht=3}
	Let $\bar{X}$ be a del Pezzo surface of rank one which is log canonical, but not log terminal. Assume that $\height(\bar{X})\geq 3$, and $\bar{X}$ has no descendant with elliptic boundary. Then $\cha\kk=2$, $\height(\bar{X})=3$, and the surface $\bar{X}$, together with some witnessing $\P^1$-fibration of the minimal log resolution, is of one of the following types, for some $k\geq 2$ and some admissible chain $V$.
\begin{longlist}
	\item\label{item:6-1-6} $\langle \bs{k};\ldec{0}V\trp,[2]\dec{2},[3]\dec{3}\rangle+\ldec{0}V^{*}*[(2)_{k-1},3,2\dec{1},2]+[\ub{3}\dec{1,2},2\dec{3},2]+[2]\dec{2}$, $d(V)=6$,
	\item\label{item:6-2-6} $\langle \bs{k};\ldec{0}V\trp,[2]\dec{2},[3]\dec{3}\rangle +\ldec{0}V^{*}*[(2)_{k-1},4,2]\dec{1}+[2,3\dec{1},\ub{2}\dec{2},2\dec{3},2]+[2]\dec{2}$, $d(V)=6$,
		\item\label{item:6-1-4} $\langle \bs{k};\ldec{0}V\trp,[2]\dec{2},[2,2\dec{3},2]\rangle +\ldec{0}V^{*}*[(2)_{k-1},4]\dec{1}+[\ub{3}\dec{2,3},2\dec{1},2]+[2]\dec{2}$, $d(V)=4$.
\end{longlist}
	Moreover, each of the above singularity types is realized by exactly one surface $\bar{X}$, up to an isomorphism. Its minimal log resolution $(X,D)$ satisfies $h^{i}(\lts{X}{D})=0$ for $i=0,1$ and $h^{2}(\lts{X}{D})=1$.
\end{proposition}
\begin{proof}
	Since $\bar{X}$ is non--log terminal, $\height(\bar{X})\geq 3$ and $\bar{X}$ has no descendant with elliptic boundary, it is as in Proposition \ref{prop:non-log-terminal}\ref{item:non-lt-swap}. Remark \ref{rem:ht=3} gives a morphism $\psi\colon (X,D)\to (X\am,D\am)$ onto one of the two log surfaces from \cite[Proposition 3.2(5)]{PaPe_MT}, see Figure \ref{fig:ht=3-remark}; such that $\psi^{-1}$ has a unique base point, lying on the $(-1)$-curve $A\subseteq D\am$. Let $D_0$ be the connected component of $D$ which contracts to the non-lt singularity $r_0\in \bar{X}$. The above  description of $D\am$ shows that the only branching component of $D_0$ is the proper transform of $A$, and it meets at least two twigs of $D_0$ which are not of type $[2]$. Since $r_0$ is log canonical, but not log terminal, it follows that $D_0$  is a log canonical fork. Two of its maximal twigs are among  $\{[2],[3],[2,2,2]\}$ in case \hyperref[item:5-1]{($\circ$)} and among $\{[2],[3],[2,3]\}$ in case \hyperref[item:5-2]{($\bullet$)}. By Lemma \ref{lem:admissible_forks} all three maximal twigs of $D_0$ are either $\{[2],[3],V\}$ with $d(V)=6$, or, in case \hyperref[item:5-1]{($\circ$)}, they are $\{[2],[2,2,2],V\}$ with $d(V)=4$. Together with the pullback of a $\P^{1}$-fibration from \cite[Proposition 3.2(5)]{PaPe_MT} we get the above list.
	\smallskip
	
	We claim that each of the resulting surfaces $\bar{X}$ is del Pezzo. 
	Let $H_{1}$, $H_{2}$ be the $1$- and $2$-section in $D$. Since $H_1$ is the branching component of the log canonical fork, we have $\cf(H_1)=1$. An easy computation, see \cite[3.2.1]{Flips_and_abundance}, shows that $\cf(H_2)=\frac{3}{7}$. Now the claim follows from the inequality \eqref{eq:ld_phi_H}.
	
	Next, we claim that each of the singularity types $\cS$ in \ref{item:6-1-6}--\ref{item:6-1-4} is realized by exactly one del Pezzo surface $\bar{X}$, up to an isomorphism. From the classification in Theorems \ref{thm:ht=1,2} and \ref{thm:GK}, it follows that every such $\bar{X}$ satisfies the assumptions of Proposition \ref{prop:ht=3}, i.e.\ has $\height(\bar{X})\geq 3$ and no descendant with elliptic boundary. Indeed, if $\height(\bar{X})\leq 2$ then by Theorem \ref{thm:ht=1,2}\ref{part:types} the singularity type $\cS$ should be listed in Tables \ref{table:ht=1}--\ref{table:ht=2_char=2}, but we check directly that it is not the case. Similarly, if $\bar{X}$ had a descendant with elliptic boundary then by Theorem \ref{thm:GK}\ref{item:GK-Y} we would have $\cS=\cS_{\textnormal{can}}+\cT$, for some canonical singularity type $\cS_{\textnormal{can}}$ with $\#\cS_{\textnormal{can}}\geq 6$, which is false.
	
	Therefore, the minimal log resolution $(X,D)$ of $\bar{X}$ has a $\P^1$-fibration as in \ref{item:6-1-6}--\ref{item:6-1-4} above. Let $\check{D}$ be the sum of $D$ and vertical $(-1)$-curves. Contracting $(-1)$-curves in $\check{D}$ and its images we get an inner birational morphism $(X,\check{D})\to (Z,D_Z)$ onto the minimal log resolution of the configuration $\pp\de \ll_1+\ll_2+\ll_3+\cc$ from Example \ref{ex:ht=3}. 
	Since $\pp$ is unique up to a projective equivalence, the log surface $(Z,D_Z)$ is unique up to an isomorphism, see Lemma \ref{lem:7A1-family}\ref{item:7A1-family-R}, case $\nu=3$. Hence by Lemma \ref{lem:inner}, the same is true for $(X,D)$. This proves the uniqueness of $\bar{X}$. Moreover, by Lemma \ref{lem:blowup-hi}\ref{item:blowup-h2},\ref{item:blowup-hi-inner} we have $h^{i}(\lts{X}{D})=h^{i}(\lts{Z}{D_Z})$, which equals $0$ for $i=0,1$ and $1$ for $i=2$ by Lemma \ref{lem:7A1-family}\ref{item:7A1-family-R},\ref{item:7A1-family-h2}, case $\nu=3$. This proves the last statement.
\end{proof}

\clearpage
\section{Del Pezzo surfaces of height one}\label{sec:ht=1}

We now move on to the proof of Theorem \ref{thm:ht=1,2}. In this section, we settle the height one case. While it is relatively easy to classify the possible singularity types, see Section \ref{sec:ht=1_types}, computing the number of isomorphism classes of surfaces of each type (or moduli dimensions) takes more effort, see Section \ref{sec:ht=1_uniqueness}. We begin with an elementary lemma. 

\begin{lemma}\label{lem:ht=1_basics}
	Let $(X,D)$ be a minimal log resolution of a del Pezzo surface $\bar{X}$ of rank one, with a fixed $\P^{1}$-fibration of height one with respect to $D$. 
	Let $H\de D\hor$ be the $1$-section.
	\begin{enumerate}
		\item\label{item:ht=1_Sigma} Every degenerate fiber $F$ has a unique $(-1)$-curve, say $L_{F}$; and $F\redd-L_{F}=F\redd\wedge D\vert$. 
		\item \label{item:ht=1_nu} Every degenerate fiber meets $H$ in a tip of $D\vert$. In particular, there are exactly $\beta_{D}(H)$ degenerate fibers.
		\item \label{item:ht=1_chains} All connected components of $D$ not containing $H$ are admissible chains. In particular, $\bar{X}$ has at most one non--log terminal singularity.	
		\item \label{item:ht=1_swap} The log surface $(X,D)$ swaps vertically to a log surface $(Y,D_Y)$ from Example \ref{ex:ht=1} below, see Figure \ref{fig:ht=1-intro}.
	\end{enumerate}
\end{lemma}
\begin{proof}
	\ref{item:ht=1_Sigma} This follows from Lemma \ref{lem:delPezzo_fibrations}\ref{item:Sigma},\ref{item:-1_curves}. 
	
	\ref{item:ht=1_nu} Let $C_{F}$ be the component of $F$ meeting $H$, so $C_F$ has multiplicity one in $F$. Since by \ref{item:ht=1_Sigma} the fiber $F$ has only one $(-1)$-curve, $L_{F}$, Lemma \ref{lem:degenerate_fibers}\ref{item:adjoint_chain} implies that $C_F$ is a tip of $F\redd$ and $C_F\neq L_{F}$. By \ref{item:ht=1_Sigma} we get that $C_{F}\subseteq D\vert$. It follows that the branching number  $\beta_{D}(H)=H\cdot D\vert$ equals  the number of degenerate fibers. 
	
	\ref{item:ht=1_chains} Let $V$ be a connected component of $D$ not containing $H$. Then $V$ is contained in some fiber $F$. Part \ref{item:ht=1_Sigma} shows that $F$ has a unique $(-1)$-curve $L_{F}$. Since $L_{F}$ is non-branching in $F$, the divisor $F\redd\wedge D\vert=F\redd-L_{F}$ has two connected components, namely $V$ and $V'\de F\redd\wedge D\vert-V$. Since $V'$ meets $H$, it contains a component of multiplicity one in $F$. Thus $V$ is a chain by Lemma \ref{lem:degenerate_fibers}\ref{item:not_columnar}.
	
	\ref{item:ht=1_swap} Put $g=p_{a}(H)$, $e=-H^2$. Let $F_{1},\dots,F_{\nu}$ be all degenerate fibers. By  \ref{item:ht=1_Sigma} each $F_i$ has a unique $(-1)$-curve, call it $L_i$. By \ref{item:ht=1_nu} we have $\nu=\beta_{D}(H)$, so $e>2g-2+\nu$ by Lemma \ref{lem:delPezzo_fibrations}\ref{item:one-elliptic}. For each $i$, contract $L_i$ and $(-1)$-curves in the subsequent images of $D$ until the image of $F_i$ is supported on a chain $[2,1,2]$. Denote the resulting morphism  by $\phi\colon X\to Y$, and define $D_Y$ as $\phi_{*}D$ minus all $(-1)$-curves. Part \ref{item:ht=1_nu} implies that $\phi$ is an isomorphism near $H$. It follows that $(X,D)\sqto (Y,D_Y)$ is a vertical swap, see Definition \ref{def:vertical_swap}\ref{item:def-swap-elementary}.
	
	The surface $Y$ is a minimal ruled surface over a smooth curve $B$, with a section $H_Y$. Since $H_Y\cong H$ has genus $g$, so does $B$. By \cite[Proposition V.2.2]{Hartshorne_AG}, we have $Y=\P(\cE)$ for some rank two vector bundle $\cE$ on $B$. Since $H_Y^2<0$, by Proposition V.2.9.\ loc.\ cit.\ $-\deg\bigwedge^{2}\cE=e$. Since $e>2g-2+\nu\geq 2g-2$, Theorem V.2.12(b) loc.\ cit. implies that $\cE$ is decomposable. Thus we can take $\cE=\cO_{B}\oplus \cL^{*}$ for some line bundle $\cL$ on $B$ of degree~$e$. Therefore, $(Y,D_Y)$ is the log surface constructed in Example \ref{ex:ht=1}, as claimed.
\end{proof}

\begin{example}[Primitive log surfaces of height one, see Figure \ref{fig:ht=1-intro}]\label{ex:ht=1}
	Fix an integer $\nu\geq 0$. Let $B$ be a smooth curve of genus $g$, and let $\cL$ be an ample line bundle on $B$, of degree $e>2g-2+\nu$. Let $S=\P(\cO_{B}\oplus \cL^{*})$ be a minimal ruled surface over $B$. Let $D_{S}$ be the section corresponding to the surjection $\cO_{B}\oplus \cL^{*}\to \cL^{*}$, see \cite[V.2.11.3]{Hartshorne_AG}. Then $D_S$ is a smooth curve of genus $g$ and $D_S^2=-e$. The log surface $(S,D_S)$ is the minimal log resolution of the cone $\bar{S}\de \Proj(\bigoplus_{k\geq 0}H^{0}(B,\cL^{\otimes k}))$ over $(B,\cL)$, see \cite[\sec 14.38]{Badescu}. If $g=0$ then $S=\F_{e}$ and $\bar{S}$ is the cone over a rational normal curve. If $g=1$ we call $\bar{S}$ an \emph{elliptic cone}.
	
	Choose $\nu$ fibers $F_1,\dots,F_{\nu}$ and a point $p_{i}\in F_{i}\setminus D_S$ for each $i\in \{1,\dots,\nu\}$. Blow up at each $p_i$ and its infinitely near point on the proper transform of $F_i$. Denote the resulting morphism by $\tau\colon Y\to S$. Let $A_i$ be the $(-1)$-curve over $p_i$ and let $D_Y=\tau^{*}(D_S+\sum_{i}F_i)\redd-\sum_{i}A_i$.
\end{example}

\begin{remark}\label{rem:ht=1-primitive}
	The log surface $(Y,D_Y)$ constructed in Example \ref{ex:ht=1} is primitive, in the sense of Definition \ref{def:vertical_swap}\ref{item:def-swap-basic}. 
\end{remark}
\begin{proof}	
	Suppose the contrary. Then $Y$ contains a $(-1)$-curve $A$ meeting $D_Y$ normally, in exactly one $(-2)$-curve, say $C$. In particular, $Y$ is not minimal, so $\nu\geq 1$. Suppose $C$ lies in some fiber $F$. Let $L$ be the $(-1)$-curve in $F$ and let $\mu=A\cdot L$. We have $A\cdot F=2\mu+1$. If $F'$ is another degenerate fiber then $A$ meets $F'$ only in the $(-1)$-curve, so $A\cdot F'$ is even, which is impossible. Thus $\nu=1$. Let $\psi\colon Y\to Z$ be the contraction of $F\redd-C$, and let $H=\psi((D_Y)\hor)$, $V=\psi(C)$, $\bar{A}=\psi(A)$. Since $\NS(Z)$ is generated by the classes of $V$ and $H$, the equality $((2\mu+1)H-\bar{A})\cdot V=0$ implies that $((2\mu+1)H-\bar{A})^2=0$. Since $H^2\leq 1-e\leq -1$, $\bar{A}\cdot H\geq 0$ and $\bar{A}^2=-1+2\mu^2$, we get $0=((2\mu+1)H-\bar{A})^2\leq -(2\mu+1)^2-1+2\mu^2=-2\mu^2-4\mu-2\leq -2$, a contradiction. 
	
	Thus $C$ is horizontal, so it is the proper transform of the section $D_S$. It follows that $(e,g)=(2,0)$. Since $A$ is disjoint from $(D_Y)\vert$, we have $A\cdot F=2\mu$ for a fiber $F$, where $\mu$ is the intersection number of $A$ with each vertical $(-1)$-curve. Let $\psi\colon Y\to \F_{2}$ be the contraction of all vertical curves not meeting $(D_Y)\hor$. Then $\psi(A)^2=-1+2\nu\mu^2$, so $\psi(A)$ is a curve of odd self-intersection number on $\F_2$; a contradiction.
\end{proof}

\subsection{Singularity types}\label{sec:ht=1_types}

We now list all singularity types of rational log canonical del Pezzo surfaces of height one, together with the structure of some witnessing $\P^1$-fibration.  Together with Lemma \ref{lem:ht=1_basics}\ref{item:ht=1_swap} above, this implies Theorem \ref{thm:ht=1,2}\ref{item:ht=1}. This way, we get a recipe to reconstruct such surfaces from $\F_{m}$. Case $\height=1$ of Proposition \ref{prop:moduli}, which we  prove in Corollary \ref{cor:moduli-ht=1} below, asserts that this process of reconstruction is unique, with exceptions listed in Table \ref{table:exceptions}.
\smallskip

To describe the witnessing $\P^1$-fibration of the minimal resolution we use decorations introduced in Notation \ref{not:fibrations}. The notation for singularity types is summarized in Section \ref{sec:notation}.

\begin{lemma}[Case $\height=1$: list of types, see Table \ref{table:ht=1}]\label{lem:ht=1_types}
	Let $\bar{X}$ be a log canonical del Pezzo surface of rank $1$ and height $1$. Then either $\bar{X}$ is an elliptic cone, or $\bar{X}$ is rational and the singularity type of $\bar{X}$, together with the structure of a $\P^1$-fibration of $(X,D)$ of height one, is one of the following.
	\begin{longlist}
		\item\label{item:chains_columnar_T1=0,T2=0} $[\bs{m}]$, 
		\item\label{item:chains_columnar_T2=0} $\ldec{1}[T^{*},\bs{m}]+T\dec{1}$, 
		\item\label{item:chains_columnar} $\ldec{1}[T_{1}^{*},\bs{m},T_{2}^{*}]\dec{2}+T_{1}\dec{1}+\ldec{2}T_{2}$, 
		\item\label{item:chains_both_T1=0}  $[\bs{m},T^{*},r\dec{1},T]+[(2)_{r-2}]\dec{1}$, 
		\item\label{item:chains_both}  $\ldec{1}[T_{1}^{*},\bs{m},T_{2}^{*},r\dec{2},T_{2}]+T_{1}\dec{1}+[(2)_{r-2}]\dec{2}$, 
		\item\label{item:chains_not-columnar} $[T_{1},r_{1}\dec{1},T_{1}^{*},\bs{m},T_{2}^{*},r_{2}\dec{2},T_{2}]+[(2)_{r_{1}-2}]\dec{1}+[(2)_{r_{2}-2}]\dec{2}$,
		\medskip
		
		\item\label{item:beta=3_columnar} $\langle  \bs{m},\ldec{1}T_{1},\ldec{2}T_{2}, \ldec{3}T_3\rangle$+ $(T_{1}^{*})\dec{1}+(T_{2}^{*})\dec{2}+(T_{3}^{*})\dec{3}$, where 
		 the fork is admissible or log canonical,
		\item\label{item:beta=3_other} $\langle \bs{m};\ldec{1}[2],\ldec{2}[2],[T^{*},r\dec{3},T] \rangle+[2]\dec{1}+[2]\dec{2}+[(2)_{r-2}]\dec{3}$, 
		\item\label{item:beta=3_other_2} $\langle \bs{m}; \ldec{1}[2],\ldec{2}T,[2,2\dec{3},2]\rangle+[2]\dec{1}+(T^{*})\dec{2}$, where $T\in \{[3],[2,2],[2,2,2]\}$, and if $T=[2,2,2]$ then $m\geq 3$,  
		\item\label{item:beta=3_other_3} $\langle \bs{m}; \ldec{1}[2],[2,2\dec{2},2],[2,2\dec{3},2]\rangle+[2]\dec{1}$, $m\geq 3$,
		\medskip
		
		\item\label{item:T2*=[2,2]_b>2} $\langle b;\ldec{1}[2], T^{*},[\bs{2},T]\rangle+[(2)_{b-3},3]\dec{1}$, $T\in \{[3],[2,2]\}$ 
		\item\label{item:T1=0,T=T2=[2]_b>2} $\langle b; \ldec{1}[2],[2],[\bs{m},2]\rangle+[(2)_{b-3},3]\dec{1}$ 
		\item\label{item:TH=[3,2]_T=[3]_b>2} $\langle b;\ldec{1}T,[2], [\bs{3},2]\rangle+([(2)_{b-2}]*T^{*})\dec{1}$, $T\in \{[3],[2,2]\}$ 
		\item\label{item:TH=[2,2]_b>2} $\langle b;\ldec{1}T,[2], [\bs{2},2]\rangle+([(2)_{b-2}]*T^{*})\dec{1}$, where $d(T)\in \{3,4,5,6\}$ and $b\geq 3$ if $T=[(2)_{5}]$
\medskip

		\item\label{item:TH=[3,2]_tip} $\langle 3;[2,2]\dec{1},[2],[\bs{3},2]\rangle$,
		\item\label{item:TH=[2,2]_tip} $\langle b;[(2)_{b-1}]\dec{1},[2],[\bs{2},2]\rangle$, where $b\in \{3,4,5,6\}$, 
		\item\label{item:TH=[2,2]_tip_[3,2]} $\langle 3;[2,3]\dec{1},[2],[\bs{2},2]\rangle+[2]\dec{1}$,
		\medskip
			
		\item\label{item:TH_long_columnar_b>2} $\langle b;[2],[2]\dec{1},\ldec{2}[T,\bs{m},2]\rangle+[(2)_{b-3},3]\dec{1}+(T^{*})\dec{2}$ 
		\item\label{item:TH_long_other_b>2} $\langle b; [2],[2]\dec{1}, [T,r\dec{2},T^{*},\bs{m},2]\rangle+[(2)_{b-3},3]\dec{1}+[(2)_{r-2}]\dec{2}$, 
		\item\label{item:F'_fork_T=[3]_b>2} $\langle b;[2],[3]\dec{1},[2,2\dec{2},2,\bs{2},2]\rangle+[(2)_{b-3},3,2]\dec{1}$, 
		\item\label{item:F'_fork_T=[2,2]} $\langle 3,[2],[2,2]\dec{1},[2,2\dec{2},2,\bs{2},2]\rangle$,
		\medskip
		
		\item\label{item:T2=[2,2]_b>2} $\langle b;[3],\ldec{1}[2],\ldec{2}[(2)_{r-1},\bs{2},2,2]\rangle+[(2)_{b-3},3]\dec{1}+[r]\dec{2}$, 
		$r\in \{2,3\}$, 
		\item\label{item:T2=[2]_b>2} $\langle b;[2],\ldec{1}T,\ldec{2}[(2)_{r-1},\bs{2},2]\rangle+([(2)_{b-2}]*T^{*})\dec{1}+[r]\dec{2}$, $r\leq 4$, $T\in \{[3],[2,2]\}$, and $b\geq 3$ if $T=[2,2]$, $r=4$, 
		\item\label{item:T2=[2],T=[2,2],tip} $\langle 3;[2],[2,2]\dec{1},\ldec{2}[(2)_{r-1},\bs{2},2]\rangle+[r]\dec{2}$, $r\in \{2,3,4\}$, 
		\medskip
		
		\item\label{item:ht=1_bench} $\ldecb{1}{2}\lbr\bs{m}\rbr\decb{3}{4}+[2]\dec{1}+[2]\dec{2}+[2]\dec{3}+[2]\dec{4}$, $m\geq 3$,
		\item\label{item:ht=1_bench-big_b>2}
		$\ldec{1}\lbr b,2, \bs{m}\rbr\decb{2}{3}+[(2)_{b-3},3]\dec{1}+[2]\dec{2}+[2]\dec{3}$ where $b\geq 3$ if $m=2$,
		\item\label{item:ht=1_bench_two_a,b>2}
		$\ldec{1}\lbr a,2, \bs{m},2,b\rbr\dec{2}+[(2)_{a-3},3]\dec{1}+[(2)_{b-3},3]\dec{2}$, where $b\geq a$ and $b\geq 3$ if $m=2$,
	\end{longlist}
	for some $m,r,r_1,r_2,a,b\geq 2$ and admissible chains $T,T_{1},T_{2}$. Recall that we use a symbol \enquote{$(2)_{-1}$}, defined in formula \eqref{eq:convention_2-1}, and as a result some chains in the above list become empty (and as such should be ignored).  A list without this simplifying convention is given in Table \ref{table:ht=1}.
\end{lemma}

\setcounter{claim}{0}
\begin{proof}
	Let $D_H$ be the connected component of $D$ containing $H\de D\hor$. Assume first that $H$ is non-rational. By the classification of log canonical singularities, see Section \ref{sec:singularities}, $H$ is a smooth curve of genus one, and $H=D_H$. Lemma \ref{lem:ht=1_basics}\ref{item:ht=1_nu} implies that there are no degenerate fibers, so by Lemma \ref{lem:ht=1_basics}\ref{item:ht=1_swap}, $\bar{X}$ is an elliptic cone.
	
	Thus we can assume that $H$ is rational, so $D_H$ is a rational tree. Since $\bar{X}$ is log canonical, $D_{H}$ can be a chain, a fork or a bench. We consider each of those cases separately.
	\begin{casesp}
	\litem{$D_H$ is a chain}\label{case:D-chain} Let $F$ be a degenerate fiber. By Lemma \ref{lem:degenerate_fibers}\ref{item:adjoint_chain} either $F$ is columnar, or $F$ becomes columnar after the contraction of some twig $[(2)_{r-2},1]$ of $D+L_F$, $r\geq 2$, whose last tip is $L_{F}$. In the latter case $F\redd=\langle r;[(2)_{r-2},1],T\trp,T^{*}\rangle$, hence $F\redd\wedge D\vert=[(2)_{r-2}]+[T,r,T^{*}]$. Therefore, $D$ is as in 
	\ref{item:chains_columnar_T1=0,T2=0}--\ref{item:chains_not-columnar}.
	\litem{$D_{H}$ is a fork} Let $B$ be the branching component of $D_{H}$.
	
	Assume first that $B=H$. Then the connected components of $D\vert$ meeting $H$ are chains. Lemma \ref{lem:degenerate_fibers} implies that in fact all connected components of $D\vert$ are chains. As in case \ref{case:D-chain}, we infer that every fiber $F$ is either columnar, or satisfies $F\redd\wedge D\vert=[T,r,T^{*}]+[(2)_{r-2}]$. Let $c$ be the number of columnar fibers. By Lemma \ref{lem:admissible_forks} we have $c\leq 2$. If $c=0$ then $D$ is as in \ref{item:beta=3_columnar}. If $c=1$ then either $D_{H}=\langle \bs{m};[T^{*},r,T],[2],[2]\rangle$, so $D$ is as in  \ref{item:beta=3_other}, or by Lemma \ref{lem:admissible_forks}\ref{item:long-twig} we have $[T^{*},r,T]=[2,2,2]$, so $D$ is as in \ref{item:beta=3_other_2}. Eventually, if $c=2$ then $D_{H}$ has two twigs of length at least  three, so these twigs are of type $[2,2,2]$, and $D$ is as in \ref{item:beta=3_other_3}.
	\smallskip
	
	Assume now that $B\neq H$, so $B$ is contained in some fiber $F$. Let $T_{H}$ be the twig of $D_{H}$ containing $H$, so $T_{H}=[T_{1},\bs{m},T_{2}]$ for some $T_{1}, T_{2}$, possibly empty. Contract $L_{F}$ and subsequent $(-1)$-curves in the images of $F$ until it becomes a chain, and denote this morphism by $\sigma$. The exceptional curves of $\sigma$ have multiplicities at least $2$ in $F$, so $\sigma$ is an isomorphism near $H$. We infer that $T_{2}\neq 0$ and $\sigma_{*}F\redd=[T_{2},1,T_{2}^{*}]$. Thus $D_{H}=\langle b;T,(T_{2}^{*})\trp,T_{H}\rangle$ for some chain $T\subseteq F$ meeting $L_{F}$.
	
	\begin{claim*}\label{cl:ht=1_V}
		Put $V=F\redd-L_{F}-(F\redd\wedge D_{H})$. Then $V$ is a chain, and the following hold.
		\begin{enumerate}
			\item\label{item:ht=1_L-hits-in-tip} If $L_{F}$ meets $\ftip{T}$ then $[V,L_{F},T]=[(2)_{b-2}]*[T^*,1,T]$ and $d(T)\leq 6$;
			\item\label{item:ht=1_L-hits-near-B} If $L_{F}$ does not meet $\ftip{T}$ then $\#T\geq 2$, $L_{F}$ meets $\ltip{T}$, and either $[T,L_{F},V]=[2,3,1,2]$, $b=3$ or $[T,L_{F},V]=[(2)_{b-1},1]$, $b\in \{3,4,5,6\}$.
		\end{enumerate}
	\end{claim*}
	\begin{proof}
		By Lemma \ref{lem:ht=1_basics}\ref{item:ht=1_chains}, $V$ is a chain. Since $\#T_{H}\geq \#(H+T_{2})\geq 2$, Lemma \ref{lem:admissible_forks} gives $d(T)\leq 6$. 
		
		\ref{item:ht=1_L-hits-in-tip} Assume that $L_{F}$ meets $\ftip{T}$. Then the chain $[T,L_{F},V]$ contracts to a smooth point, so $V=T^{*}*[(2)_{k}]$ for some $k\geq -1$. After this contraction, $B$ becomes a $(-1)$-curve, so $-1=-b+k+1$, i.e.\ $k=b-2$, as claimed.
		
		\ref{item:ht=1_L-hits-near-B} Assume that $L_{F}$ does not meet $\ftip{T}$. Since $L_{F}$ meets $T$, we have $\#T\geq 2$. Since $d(T)\leq 6$, we have $T=[2,3], [3,2]$, or $[(2)_{k}]$ for some $k\in \{2,3,4,5\}$, see Lemma \ref{lem:admissible_forks}\ref{item:d=3}. Because $L_{F}+T+V$ contracts to a smooth point, we see that $L_{F}$ meets the last tip of $T$ and either $T=[2,3]$, $V=[2]$ or $T=[(2)_{k}]$, $V=0$. After this contraction $B$ becomes a $(-1)$-curve, so $b=3$ or $k+1$; respectively.
	\end{proof}
	
	\begin{casesp}
	\litem{$T_{1}=0$} Then $F$ is the unique degenerate fiber. 
	Assume $T_{2}\neq [2]$, so $d(T_H)>d(T_2)\geq 3$. Since $d(T_{2}^{*})=d(T_2)\geq 3$ and  $d(T)^{-1}+d(T_2^{*})^{-1}+d(T_{H})^{-1}\geq 1$, we get $T=[2]$ and $d(T_{2}^{*})=3$, so $T_{2}^{*}=[2,2]$ or $[3]$. Claim~\ref{item:ht=1_L-hits-in-tip} implies that $D$ is as in \ref{item:T2*=[2,2]_b>2}. 
	
	Next, assume that $T_{2}=[2]$, $T=[2]$. Then Claim~\ref{item:ht=1_L-hits-in-tip} implies that $V=[(2)_{b-3},3]$, so $D$ is as in \ref{item:T1=0,T=T2=[2]_b>2}. 
	
	Eventually, assume $T_2=[2]$, $T\neq [2]$. Since  $T_{H}=[\bs{m},2]$, Lemma \ref{lem:admissible_forks} gives  $m\in \{2,3\}$. 
	
	Consider the case $m=3$. Then $d(T)=3$, so $T=[3]$ or $[2,2]$. If $T=[3]$ then by Claim~\ref{item:ht=1_L-hits-in-tip} we have $V=[(2)_{b-3},3,2]$, so  $D$ is as in \ref{item:TH=[3,2]_T=[3]_b>2}. If $T=[2,2]$ then by  Claim~\ref{item:ht=1_L-hits-in-tip} and \ref{item:ht=1_L-hits-near-B} $D$ is as in \ref{item:TH=[3,2]_T=[3]_b>2} if $L_{F}$ meets $\ftip{T}$ and as in \ref{item:TH=[3,2]_tip} if $L_{F}$ meets $\ltip{T}$. 
	
	Consider the case $m=2$. Then $d(T)\leq 6$. If $L_{F}$ meets $\ftip{T}$, by Claim~\ref{item:ht=1_L-hits-in-tip}, $D$ is as in \ref{item:TH=[2,2]_b>2}. If $L_{F}$ meets $\ltip{T}$, Claim~\ref{item:ht=1_L-hits-near-B} implies that either $T=[(2)_{b-1}]$, $b\leq 6$, and $D$ is as in \ref{item:TH=[2,2]_tip}; or $T=[2,3]$ and $D$ is as in \ref{item:TH=[2,2]_tip_[3,2]}.
	
	\litem{$T_{1}\neq 0$} Let $F'\neq F$ be the degenerate fiber containing $T_1$. The $1$-section $H$ meets $\ltip{T_{1}}$, which therefore has multiplicity one in $F'$. Hence by Lemma \ref{lem:degenerate_fibers} either $F'\redd=[T_{1}^{*},1,T_{1}]$, i.e.\ $F'$ is columnar, or $F'\redd=\langle r;[(2)_{r-2},1],U\trp,U^{*}\rangle$, where $T_{1}=[U^{*},r,U]$ for some $r\geq 2$.
	
	If $T_{2}^{*}=[2]$, $T=[2]$ then by Claim~\ref{item:ht=1_L-hits-in-tip} $D$ is as in \ref{item:TH_long_columnar_b>2} if $F'$ is columnar and as in \ref{item:TH_long_other_b>2} if it is not. 
	Assume $T_{2}^{*}\neq [2]$ or $T\neq [2]$. Since $d(T_H)^{-1}+d(T_{2}^{*})+d(T)\geq 1$, we have $d(T_{H})\leq 6$. Since $\#T_{H}=\#T_{1}+1+\#T_{2}\geq 3$ we get $T_{H}=[(2)_{k}]$ for some $k\in \{3,4,5\}$. Moreover, since $d(T_{H})\geq 4$, we have $\{d(T_2),d(T)\}=\{2,3\}$.
	
	Consider the case when $F'$ is not columnar. Then $k=5$, $T_{2}=[2]$, so since by assumption $T\neq [2]$, we get $d(T)=3$. If $T=[3]$ then by Claim~\ref{item:ht=1_L-hits-in-tip} $D$ is as in \ref{item:F'_fork_T=[3]_b>2}. If $T=[2,2]$ then since $D_{H}$ cannot consist of $(-2)$-curves, we have $b\geq 3$, so the Claim shows that $L_{F}$ meets $\ltip{T}$ and $D$ is as in \ref{item:F'_fork_T=[2,2]}.

	Consider now the case when $F'$ is columnar, so $F'=[r,1,(2)_{r}]$, where $r=\#T_{1}=k-1-\#T_{2}$. Recall that $d(T_{2})\leq 3$, so $T_2=[2,2]$ or $[2]$. If $T_{2}=[2,2]$ then $T=[2]$ and Claim~\ref{item:ht=1_L-hits-in-tip} implies that $D$ is as in \ref{item:T2=[2,2]_b>2}. Assume $T_{2}=[2]$. Then $d(T)=3$, so $T=[3]$ or $[2,2]$. If $L_{F}$ meets $\ftip{T}$ then by Claim~\ref{item:ht=1_L-hits-in-tip} $D$ is as in \ref{item:T2=[2]_b>2}. Otherwise,  $T=[2,2]$, so by Claim~\ref{item:ht=1_L-hits-near-B} $b=3$ and $D$ is as in \ref{item:T2=[2],T=[2,2],tip}.
	\end{casesp}
	\litem{$D_{H}$ is a bench} Assume first that $D_{H}-H$ has no branching components. Then $D_{H}=\lbr \bs{m}\rbr$, and $H$ meets four disjoint $(-2)$-curves. Lemmas \ref{lem:ht=1_basics}\ref{item:ht=1_Sigma} and \ref{lem:degenerate_fibers}\ref{item:adjoint_chain} imply that each of those $(-2)$-curves lies in a fiber $[2,1,2]$, and $D$ is as in \ref{item:ht=1_bench}. 
	
	Let now $V$ be a connected component of $D_{H}-H$ with a branching component. By Lemma \ref{lem:degenerate_fibers}\ref{item:not_columnar}, $V$ meets $H$ in a tip. Hence $V$ is a fork of type $\langle b,[2],[2],T_{V}\rangle$ and meets $H$ in $\ftip{T_{V}}$. Let $F$ be a fiber containing $V$. Contract $L_{F}$ and subsequent $(-1)$-curves in $F$ until the image of $V$ contains a $(-1)$-curve. Denote by $H'$, $V'$, $F'$ the images of $H$, $V$ and $F$, respectively, and let $L'$ be the $(-1)$-curve in $V'$.
	
	Suppose that $L'$ is not a tip of $V'$. Since $L'$ is non-branching in $F'\redd$, we get $F'\redd=V'$. Now $V'-L'$ has two connected components: one of them is not a chain, and the other meets the $1$-section $H'$, so contains a component of multiplicity one. This is a contradiction with Lemma \ref{lem:degenerate_fibers}\ref{item:not_columnar}.
	
	Thus $L'$ is a tip of $V'$. Since $L'$ is the unique $(-1)$-curve in $F'$, it has multiplicity at least two in $F'$, so it does not meet $H'$. It follows that $L'$ is the image of one of the short $(-2)$-twigs of $V$. Thus $F\redd=\langle b,[U,1,2],[2],T_{V}\rangle$, where $U$ is a chain such that $[U,1,2]$ contracts to a smooth point, so $U=[(2)_{k-1},3]$ for some $k\geq 0$. After the contraction of $[U,1,2]$, the image of $F\redd$ is a chain $[T_{V},b-k-1,2]$, so $T_{V}=[2]^{*}=[2]$ and $b-k-1=1$, i.e.\ $k=b-2$. We conclude that if $D_{H}-H$ has two branching components then $D$ is as in \ref{item:ht=1_bench_two_a,b>2}. Otherwise, $H$ meets one fiber of the above type, and two twigs of type $[2]$. Like before, using Lemmas \ref{lem:ht=1_basics}\ref{item:ht=1_Sigma} and \ref{lem:degenerate_fibers}\ref{item:adjoint_chain} we conclude that those twigs lie in fibers $[2,1,2]$, so $D$ is as in \ref{item:ht=1_bench-big_b>2}.
	\qedhere\end{casesp}
\end{proof}

\begin{remark}[Canonical del Pezzo surfaces of rank one and height $1$, see Figure \ref{fig:ht=1}]\label{rem:canonical_ht=1}
	For future reference we point out all canonical types on the list in Lemma \ref{lem:ht=1_types}. Since for a canonical type all components have self intersection number $-2$, to get one we need to fix the value of all parameters as $2$; and types of all admissible chains as $[2,\dots,2]$. Note that if $T=[2,\dots,2]$ then the same holds for $T^{*}$ if and only if $T=[2]$. As a consequence, the canonical types in Lemma \ref{lem:ht=1_types} are precisely those in Figure \ref{fig:ht=1}. The vertically primitive ones are those of types $\rA_1$, $\rA_1+\rA_2$, $2\rA_1+\rA_3$ and $3\rA_1+\rD_4$, see Figures \ref{fig:A1}, \ref{fig:A1+A2}, \ref{fig:2A1+A3} and \ref{fig:3A1+D4}. Their minimal log resolutions are constructed in Example \ref{ex:ht=1} for $g=0$, $e=2$ and $\nu\leq 3$; so they are primitive by Remark \ref{rem:ht=1-primitive}.
\end{remark}	
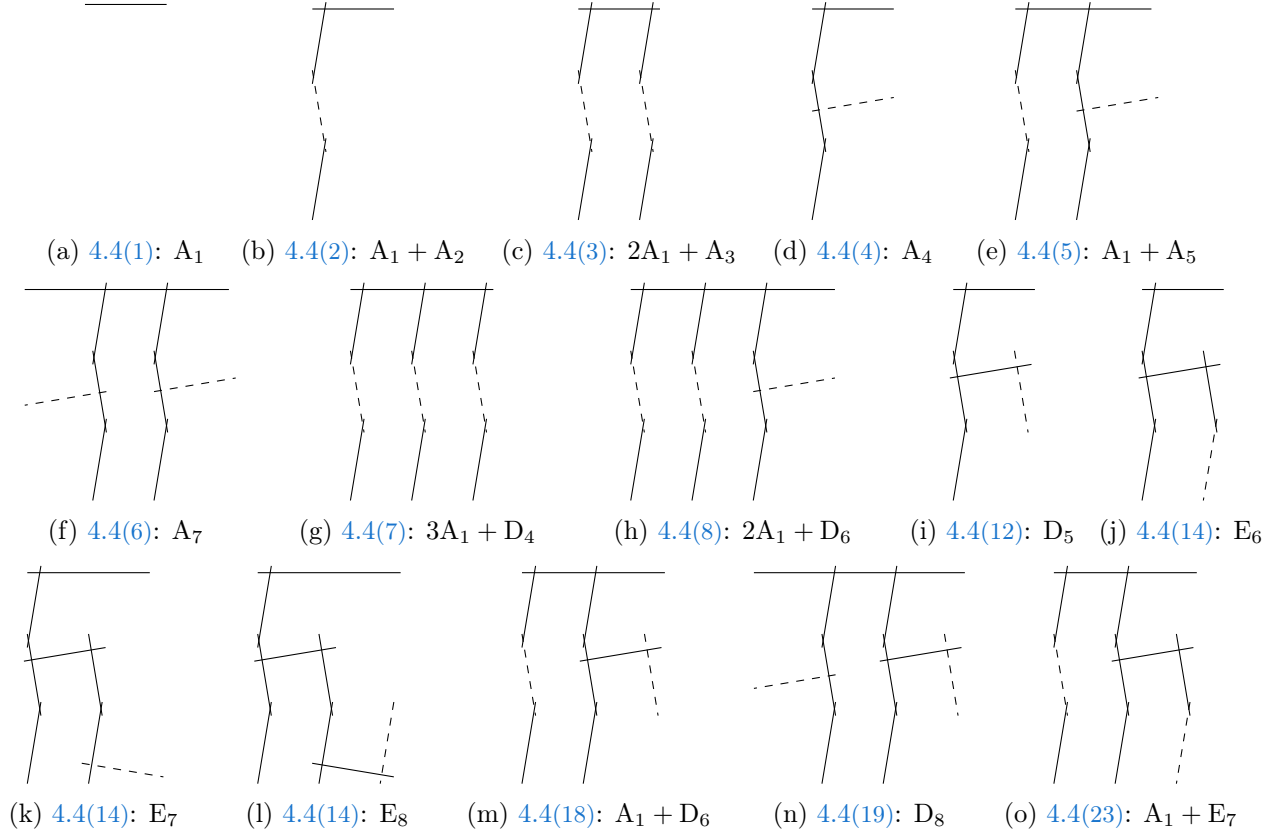
\begin{figure}[htbp]
	\subcaptionbox{\ref{lem:ht=1_types}\ref{item:chains_columnar_T1=0,T2=0}: $\rA_1$\label{fig:A1}}
	[.14\linewidth]
	{\raisebox{2.85cm}{
			\begin{tikzpicture}[scale=0.9]
				\draw (0,3) -- (1.2,3);
			\end{tikzpicture}
		}
	}
	\subcaptionbox{\ref{lem:ht=1_types}\ref{item:chains_columnar_T2=0}: $\rA_1+\rA_2$\label{fig:A1+A2}}
	[.2\linewidth]
	{
		\begin{tikzpicture}[scale=0.9]
			\draw (0,3) -- (1.2,3);
			\draw (0.2,3.1) -- (0,1.9);
			\draw[dashed] (0,2.1) -- (0.2,0.9);
			\draw (0.2,1.1) -- (0,-0.1);
		\end{tikzpicture}
	}
	\subcaptionbox{\ref{lem:ht=1_types}\ref{item:chains_columnar}: $2\rA_1+\rA_3$\label{fig:2A1+A3}}
	[.2\linewidth]{
		\begin{tikzpicture}[scale=0.9]
			\draw (0,3) -- (1.2,3);
			\draw (0.2,3.1) -- (0,1.9);
			\draw[dashed] (0,2.1) -- (0.2,0.9);
			\draw (0.2,1.1) -- (0,-0.1);
			\draw (1.1,3.1) -- (0.9,1.9);
			\draw[dashed] (0.9,2.1) -- (1.1,0.9);
			\draw (1.1,1.1) -- (0.9,-0.1);		
		\end{tikzpicture}	
	}
	\subcaptionbox{\ref{lem:ht=1_types}\ref{item:chains_both_T1=0}: $\rA_4$}
	[.15\linewidth]
	{
		\begin{tikzpicture}[scale=0.9]
			\draw (0,3) -- (1.2,3);
			\draw (0.2,3.1) -- (0,1.9);
			\draw (0,2.1) -- (0.2,0.9);
			\draw (0.2,1.1) -- (0,-0.1);
			\draw[dashed] (0,1.5) -- (1.2,1.7);
		\end{tikzpicture}
	}	
	\subcaptionbox{\ref{lem:ht=1_types}\ref{item:chains_both}: $\rA_1+\rA_5$}
	[.2\linewidth]{
		\begin{tikzpicture}[scale=0.9]
			\draw (0,3) -- (2,3);
			\draw (0.2,3.1) -- (0,1.9);
			\draw[dashed] (0,2.1) -- (0.2,0.9);
			\draw (0.2,1.1) -- (0,-0.1);
			\draw (1.1,3.1) -- (0.9,1.9);
			\draw (0.9,2.1) -- (1.1,0.9);
			\draw (1.1,1.1) -- (0.9,-0.1);
			\draw[dashed] (0.9,1.5) -- (2.1,1.7);		
		\end{tikzpicture}	
	}
\medskip

	\subcaptionbox{\ref{lem:ht=1_types}\ref{item:chains_not-columnar}: $\rA_7$ \label{fig:A7}}
	[.22\linewidth]{
		\begin{tikzpicture}[scale=0.9]
			\draw (-1,3) -- (2,3);
			\draw (0.2,3.1) -- (0,1.9);
			\draw (0,2.1) -- (0.2,0.9);
			\draw (0.2,1.1) -- (0,-0.1);
			\draw (1.1,3.1) -- (0.9,1.9);
			\draw (0.9,2.1) -- (1.1,0.9);
			\draw (1.1,1.1) -- (0.9,-0.1);
			\draw[dashed] (0.9,1.5) -- (2.1,1.7);	
			\draw[dashed] (0.2,1.5) -- (-1,1.3);	
		\end{tikzpicture}	
	}
	\subcaptionbox{\ref{lem:ht=1_types}\ref{item:beta=3_columnar}: $3\rA_1+\rD_4$ \label{fig:3A1+D4}}
	[.22\linewidth]{
		\begin{tikzpicture}[scale=0.9]
			\draw (0,3) -- (2.1,3);
			\draw (0.2,3.1) -- (0,1.9);
			\draw[dashed] (0,2.1) -- (0.2,0.9);
			\draw (0.2,1.1) -- (0,-0.1);
			\draw (1.1,3.1) -- (0.9,1.9);
			\draw[dashed] (0.9,2.1) -- (1.1,0.9);
			\draw (1.1,1.1) -- (0.9,-0.1);
			\draw (2,3.1) -- (1.8,1.9);
			\draw[dashed] (1.8,2.1) -- (2,0.9);
			\draw (2,1.1) -- (1.8,-0.1);		
		\end{tikzpicture}	
	}
	\subcaptionbox{\ref{lem:ht=1_types}\ref{item:beta=3_other}: $2\rA_1+\rD_6$\label{fig:2A1+D6}}
	[.25\linewidth]{
		\begin{tikzpicture}[scale=0.9]
			\draw (0,3) -- (3,3);
			\draw (0.2,3.1) -- (0,1.9);
			\draw[dashed] (0,2.1) -- (0.2,0.9);
			\draw (0.2,1.1) -- (0,-0.1);
			\draw (1.1,3.1) -- (0.9,1.9);
			\draw[dashed] (0.9,2.1) -- (1.1,0.9);
			\draw (1.1,1.1) -- (0.9,-0.1);
			\draw (2,3.1) -- (1.8,1.9);
			\draw (1.8,2.1) -- (2,0.9);
			\draw (2,1.1) -- (1.8,-0.1);
			\draw[dashed] (1.8,1.5) -- (3,1.7);		
		\end{tikzpicture}	
	}
	\subcaptionbox{\ref{lem:ht=1_types}\ref{item:T1=0,T=T2=[2]_b>2}: $\rD_5$}
	[.14\linewidth]
	{
		\begin{tikzpicture}[scale=0.9]
			\draw (0,3) -- (1.2,3);
			\draw (0.2,3.1) -- (0,1.9);
			\draw (0,2.1) -- (0.2,0.9);
			\draw (0.2,1.1) -- (0,-0.1);
			\draw (-0.05,1.7) -- (1.15,1.9);
			\draw[dashed] (0.9,2.1) -- (1.1,0.9);
		\end{tikzpicture}
	}
	\subcaptionbox{\ref{lem:ht=1_types}\ref{item:TH=[2,2]_b>2}: $\rE_6$\label{fig:E6}}
	[.14\linewidth]
	{
		\begin{tikzpicture}[scale=0.9]
			\draw (0,3) -- (1.2,3);
			\draw (0.2,3.1) -- (0,1.9);
			\draw (0,2.1) -- (0.2,0.9);
			\draw (0.2,1.1) -- (0,-0.1);
			\draw (-0.05,1.7) -- (1.15,1.9);
			\draw (0.9,2.1) -- (1.1,0.9);
			\draw[dashed] (1.1,1.1) -- (0.9,-0.1);
		\end{tikzpicture}
	}
\medskip

	\subcaptionbox{\ref{lem:ht=1_types}\ref{item:TH=[2,2]_b>2}: $\rE_7$}
	[.15\linewidth]
	{
		\begin{tikzpicture}[scale=0.9]
			\draw (0,3) -- (1.8,3);
			\draw (0.2,3.1) -- (0,1.9);
			\draw (0,2.1) -- (0.2,0.9);
			\draw (0.2,1.1) -- (0,-0.1);
			\draw (-0.05,1.7) -- (1.15,1.9);
			\draw (0.9,2.1) -- (1.1,0.9);
			\draw (1.1,1.1) -- (0.9,-0.1);
			\draw[dashed] (0.8,0.2) -- (2,0);
		\end{tikzpicture}
	}	
	\subcaptionbox{\ref{lem:ht=1_types}\ref{item:TH=[2,2]_b>2}: $\rE_8$\label{fig:E8}}
	[.2\linewidth]
	{
		\begin{tikzpicture}[scale=0.9]
			\draw (0,3) -- (2.1,3);
			\draw (0.2,3.1) -- (0,1.9);
			\draw (0,2.1) -- (0.2,0.9);
			\draw (0.2,1.1) -- (0,-0.1);
			\draw (-0.05,1.7) -- (1.15,1.9);
			\draw (0.9,2.1) -- (1.1,0.9);
			\draw (1.1,1.1) -- (0.9,-0.1);
			\draw (0.8,0.2) -- (2,0);
			\draw[dashed] (2,1.1) -- (1.8,-0.1);
		\end{tikzpicture}
	}
	\subcaptionbox{\ref{lem:ht=1_types}\ref{item:TH_long_columnar_b>2}: $\rA_1+\rD_6$ \label{fig:A1+D6}}
	[.2\linewidth]
	{
		\begin{tikzpicture}[scale=0.9]
			\draw (0,3) -- (2.1,3);
			\draw (0.2,3.1) -- (0,1.9);
			\draw[dashed] (0,2.1) -- (0.2,0.9);
			\draw (0.2,1.1) -- (0,-0.1);
			\draw (1.1,3.1) -- (0.9,1.9);
			\draw (0.9,2.1) -- (1.1,0.9);
			\draw (1.1,1.1) -- (0.9,-0.1);
			\draw (0.85,1.7) -- (2.05,1.9);
			\draw[dashed] (1.8,2.1) -- (2,0.9);
		\end{tikzpicture}
	}	
	\subcaptionbox{\ref{lem:ht=1_types}\ref{item:TH_long_other_b>2}: $\rD_8$}
	[.2\linewidth]
	{
		\begin{tikzpicture}[scale=0.9]
			\draw (-1,3) -- (2.1,3);
			\draw (0.2,3.1) -- (0,1.9);
			\draw (0,2.1) -- (0.2,0.9);
			\draw[dashed] (0.2,1.5) -- (-1,1.3);
			\draw (0.2,1.1) -- (0,-0.1);
			\draw (1.1,3.1) -- (0.9,1.9);
			\draw (0.9,2.1) -- (1.1,0.9);
			\draw (1.1,1.1) -- (0.9,-0.1);
			\draw (0.85,1.7) -- (2.05,1.9);
			\draw[dashed] (1.8,2.1) -- (2,0.9);
		\end{tikzpicture}
	}
	\subcaptionbox{\ref{lem:ht=1_types}\ref{item:T2=[2]_b>2}: $\rA_1+\rE_7$ \label{fig:A1+E7}}
	[.2\linewidth]
	{
		\begin{tikzpicture}[scale=0.9]
			\draw (0,3) -- (2.1,3);
			\draw (0.2,3.1) -- (0,1.9);
			\draw[dashed] (0,2.1) -- (0.2,0.9);
			\draw (0.2,1.1) -- (0,-0.1);
			\draw (1.1,3.1) -- (0.9,1.9);
			\draw (0.9,2.1) -- (1.1,0.9);
			\draw (1.1,1.1) -- (0.9,-0.1);
			\draw (0.85,1.7) -- (2.05,1.9);
			\draw (1.8,2.1) -- (2,0.9);
			\draw[dashed] (2,1.1) -- (1.8,-0.1);
		\end{tikzpicture}
	}	
	\vspace{-0.5em}
	\caption{Canonical del Pezzo surfaces of rank one and height one, see Remark \ref{rem:canonical_ht=1}.}
	\label{fig:ht=1}
\end{figure} 

\subsection{Computation of moduli}\label{sec:ht=1_uniqueness}

Lemma \ref{lem:ht=1_types} allows to construct each rational log canonical del Pezzo surface of rank one and height one (o rather its  minimal log resolution) by a specific sequence of blowups over some Hirzebruch surface. In this section we analyze to which extent this process is unique, thus proving Proposition \ref{prop:moduli} for surfaces of height one, see Corollary \ref{cor:moduli-ht=1}. Once this is done, the proof of case $\height(\bar{X})=1$ of Theorem \ref{thm:ht=1,2} will be complete. 

We begin with a brief sketch of the proof. Our goal is to put all log surfaces $(X,D)$ from Lemma \ref{lem:ht=1_types} with a fixed graph of $D$ in a smooth family, and to count the number of its non-isomorphic fibers. 
 
 Let $\check{D}$ be the sum of $D$ and $(-1)$-curves provided by Lemma \ref{lem:ht=1_types}. Lemma \ref{lem:adding-1} allows to replace $D$ by $\check{D}$ and to study the set $\cP_{+}(\check{\cS})$, where $\check{\cS}$ is the combinatorial type of $(X,\check{D})$, see Section \ref{sec:moduli}. 
 
 Now, we simplify $\check{\cS}$ by inner blowdowns. More precisely, let $(X,\check{D})\to (Z,D_Z)$ be a vertical inner snc-minimalization, i.e. the contraction of all $(-1)$-curves in $\check{D}\vert$ and its images which are not tips of the images of $\check{D}$. Let $\cZ$ be the combinatorial type of $(Z,D_Z)$: it is as  in Lemma~\ref{lem:ht=1_reduction}. The center of each blowdown is uniquely determined by the combinatorial type $\check{\cS}$, so it is enough to study $\cP_{+}(\cZ)$, cf.\ Lemma \ref{lem:inner}. 
 
 To further simplify $(Z,D_Z)$, we need to perform an outer blowdown $(Z,D_Z)\to (Y,D_Y)$, centered somewhere on the curve $G^{\circ}\de  G\setminus (D_Y-G)$, where $G$ is some component of $D_Y$. At this point we need to understand the action of $\Aut(Y,D_Y)$ on $G^{\circ}$, see Lemma \ref{lem:outer}. This is done case by case in Lemma \ref{lem:ht=1_uniqueness}.
\smallskip

We note that the proof of Propositions \ref{prop:moduli} in case $\height(\bar{X})=2$ is based on the same strategy, but requires substantially less computations, as surfaces of bigger height turn out to be more rigid. 
\smallskip

We now proceed with the proof, following the strategy outlined above. First, we list possible combinatorial types $\cZ$, together with some additional ones which will be useful either for induction or in the next section.

\begin{lemma}[Vertical inner snc-minimalization, see Table \ref{table:ht=1_exceptions} and Figure \ref{fig:basic_ht=1}]\label{lem:ht=1_reduction}
	Let $\bar{X}$ be a rational log canonical del Pezzo surface of rank one. Let $(X,D)$ be its minimal log resolution, equipped with a $\P^{1}$-fibration as in Lemma \ref{lem:ht=1_types}. Let $\check{D}$ be the sum of $D$ and all vertical $(-1)$-curves. Let $(X,\check{D})\to (Z,D_Z)$ be a vertical inner snc-minimalization, i.e.\ a contraction of a maximal number of vertical $(-1)$-curves which are not tips of the subsequent images of $\check{D}$. Then the weighted graph of $D_Z$ is as in one of the cases \ref{item:uniq_easy}--\ref{item:uniq-bench} described below and shown in Figure \ref{fig:basic_ht=1}, with additional restrictions listed in Table \ref{table:ht=1_exceptions}. 

In each of these cases, we have $D_Z=V_Z+\sum_{j=1}^{v}F_{j}+\sum_{j=1}^{h}H_j$, where, unless stated otherwise: 
\begin{itemize}
	\item $V_Z$ is the support of all degenerate fibers,
	\item  $F_j$ for each $j\in \{1,\dots, v\}$ is a nondegenerate fiber,
	\item $h\in \{1,2\}$, and for each $j\in \{1,\dots, h\}$, $H_j$ is a $1$-section,
	\item if $h=2$ then the $1$-sections $H_1$, $H_2$ are disjoint and meet distinct components of $V_{Z}$, 
	\item $H_{1}=[m]$ for some $m\geq 0$; in fact $m=\tilde{m}+v$, where $\tilde{m}$ is the self-intersection number of $D\hor$.
\end{itemize}
and, depending on the type of $(X,\check{D})$ as specified in Table \ref{table:ht=1_exceptions},  one of the following holds, see Figure  \ref{fig:basic_ht=1}:
	\begin{longlist}
		\item\label{item:uniq_easy} $V_Z=0$; and $v\leq 4$
		\item\label{item:uniq_n=3} $V_Z=\langle 2;[1],T^*,T\trp\rangle$ for some admissible $T$, or $V_Z=[1,1]$ where each $H_j$ meets $V_Z\cp{1}$; and $v\leq 2$
		\item\label{item:uniq_n=2} $V_Z$ is a sum of two fibers as in \ref{item:uniq_n=3}; and 
		$v=0$, $m\geq 2$
		\item\label{item:uniq_dumb} $V_Z=[1]+\langle k;T^{*},T\trp,[(2)_{k-2},\underline{2}]\rangle$ for some $k\geq 3$ and admissible $T$, where $[1]$ meets the underlined curve; or $V_Z=\langle 2;[k],[1],[(2)_{k-2}]\rangle$ for some $k\geq 3$, where each $H_j$ meets the $(-k)$-curve; and $v\leq 1$,
		\item\label{item:uniq_n=2_v=1} $V_{Z}=\langle 2;[2],[2],[1]\rangle+\langle 2;[2],[2],[1]\rangle$; and $h=v=m=1$, 
		\item\label{item:uniq_fork-lc} $V_{Z}=\langle 2;[2],[2],[1]\rangle + \langle 3;[2,\uline{2}],[2],[2]\rangle+[1]$, where $[1]$ meets the underlined curve; and $h=1$, $v=0$, $m=2$,
		\item\label{item:uniq_A1E7} $V_Z=\langle 2;[2],[1,2,2],[2]\rangle$ and $v=h=m=1$,
		\item \label{item:uniq_2} $V_Z=\langle  2;T^{*},T\trp,[1,2]\rangle$ for some admissible $T$, or $V_Z=[1,2,1]$ where each $H_j$ meets $V_Z\cp{1}$; and $v=1$, $h=2$,
		\item \label{item:uniq_-2-chain} $V_Z=\langle  2;[(2)_{d-1}],[d],[1,2]\rangle$ for some $d\geq 2$, where $H_1$ meets a $(-2)$-tip of $V_Z$, and $v\leq 2$, $h=1$, $m\geq 1$.
		\item\label{item:uniq_fork} $D_Z=\langle  2;[\bs{m},(2)_{d-1}],[d],[1,2]\rangle+\langle 2;T^{*},T\trp,[1]\rangle$ for some $d\geq 2$ and admissible $T$,
		\item\label{item:uniq_3} $D_Z=\langle 2;[\bs{2},3],[2,2],[1,2]\rangle$,
		\item\label{item:uniq_5} $D_Z=\langle 2;[\bs{3},2],[2],[1,2,2]\rangle$,		
		\item\label{item:uniq_Ek} $D_Z=\langle 2;[\bs{2},2],[2],[1,(2)_{k}]\rangle$, $k\in \{2,3,4\}$,
		\item\label{item:uniq-bench} $V_{Z}=\langle 2;[2],[2],[1,2]\rangle+\langle 2;[2],[2],[1,2]\rangle$, $h=1$, $v=0$, $m\geq 3$.
	\end{longlist}
	where the bold numbers in \ref{item:uniq_fork}--\ref{item:uniq_Ek} refer to the unique horizontal component $H_1$ of $D_Z$, cf.\ Notation \ref{not:fibrations}\ref{not:bold}.
\end{lemma}
\begin{proof}
	Straightforward from the description of $(X,\check{D})$ in Lemma \ref{lem:ht=1_types}. 
\end{proof}

\begin{figure}
	\subcaptionbox{\ref{lem:ht=1_reduction}\ref{item:uniq_easy}\label{fig:uniq_easy}}[.24\linewidth]{
		\begin{tikzpicture}
			\draw (-0.6,2) -- (2.8,2);
			\node at (-0.4,2.2) {\small{$H_1$}};
			\node at (-0.4,1.8) {\small{$-m$}};
			\draw[densely dotted] (-0.6,0) -- (2.8,0);
			\node at (-0.4,0.2) {\small{$H_2$}};
			\node at (-0.4,-0.2) {\small{$m$}};
			\draw[densely dotted] (0.2,2.1) -- (0.2,-0.1);
			\node at (0.4,1) {\small{$F_1$}};
			\draw[densely dotted] (1,2.1) -- (1,-0.1);
			\node at (1.2,1) {\small{$F_2$}};
			\draw[densely dotted] (1.8,2.1) -- (1.8,-0.1);
			\node at (2,1) {\small{$F_3$}};
			\draw[densely dotted] (2.6,2.1) -- (2.6,-0.1);
			\node at (2.8,1) {\small{$F_4$}};
		\end{tikzpicture}	
	}
	\subcaptionbox{\ref{lem:ht=1_reduction}\ref{item:uniq_n=3}\label{fig:uniq_n=3}}[.22\linewidth]{
		\begin{tikzpicture}
			\draw[densely dotted] (-0.1,0) -- (2.6,0);
			\node at (1,-0.2) {\small{$m-1$}};
			\node at (1,0.2) {\small{$H_2$}};
			\draw (0,-0.2) -- (0.2,0.8);
			\node at (0.1,1) {$\vdots$};
			\draw (0.2,1) -- (0,2);
			\draw [decorate, decoration = {calligraphic brace}, thick] (-0.2,-0.2) --  (-0.2,1.9);
			\node[rotate=-90] at (-0.5,0.8) {\small{$T$}}; 
			\draw (0,1.8) -- (0.2,3);
			\node at (-0.2,2.5) {\small{$-2$}};
			\draw (-0.1,2.2) -- (1.1,2.4);
			\node at (0.6,2.5) {\small{$-1$}};
			\draw (0.2,2.8) -- (0,3.8);
			\node at (0.1,4) {$\vdots$};
			\draw (0,4) -- (0.2,5);
			\draw [decorate, decoration = {calligraphic brace}, thick] (-0.2,2.9) --  (-0.2,5);
			\node[rotate=-90] at (-0.5,3.85) {\small{$T^{*}$}};
			\draw (-0.1,4.8) -- (2.6,4.8);
			\node at (1,4.6) {\small{$-m$}};
			\node at (1,5) {\small{$H_1$}};
			\draw[densely dotted] (1.7,-0.2) -- (1.7,5);
			\node at (1.95,2.5) {\small{$F_1$}};
			\draw[densely dotted] (2.5,-0.2) -- (2.5,5);
			\node at (2.75,2.5) {\small{$F_2$}};
		\end{tikzpicture}	
	}
	\subcaptionbox{\ref{lem:ht=1_reduction}\ref{item:uniq_n=2}, $m\geq 2$\label{fig:uniq_n=2}}[.26\linewidth]{
		\begin{tikzpicture}
			\draw[densely dotted] (-0.1,0) -- (2.2,0);
			\node at (1,-0.2) {\small{$m-2$}};
			\node at (1,0.2) {\small{$H_2$}};
			\draw (0,-0.2) -- (0.2,0.8);
			\node at (0.1,1) {$\vdots$};
			\draw (0.2,1) -- (0,2);
			\draw [decorate, decoration = {calligraphic brace}, thick] (-0.2,-0.2) --  (-0.2,1.9);
			\node[rotate=-90] at (-0.5,0.8) {\small{$T_1$}}; 
			\draw (0,1.8) -- (0.2,3);
			\node at (-0.2,2.5) {\small{$-2$}};
			\draw (-0.1,2.2) -- (1.1,2.4);
			\node at (0.6,2.5) {\small{$-1$}};
			\draw (0.2,2.8) -- (0,3.8);
			\node at (0.1,4) {$\vdots$};
			\draw (0,4) -- (0.2,5);
			\draw [decorate, decoration = {calligraphic brace}, thick] (-0.2,2.9) --  (-0.2,5);
			\node[rotate=-90] at (-0.5,3.85) {\small{$T_1^{*}$}};
			\draw (2,-0.2) -- (2.2,0.8);
			\node at (2.1,1) {$\vdots$};
			\draw (2.2,1) -- (2,2);
			\draw [decorate, decoration = {calligraphic brace}, thick] (2.4,1.9) --  (2.4,-0.2);
			\node[rotate=90] at (2.7,0.8) {\small{$T_2$}}; 
			\draw (2,1.8) -- (2.2,3);
			\node at (1.8,2.5) {\small{$-2$}};
			\draw (1.9,2.2) -- (3.1,2.4);
			\node at (2.6,2.5) {\small{$-1$}};
			\draw (2.2,2.8) -- (2,3.8);
			\node at (2.1,4) {$\vdots$};
			\draw (2,4) -- (2.2,5);
			\draw [decorate, decoration = {calligraphic brace}, thick] (2.4,5) -- (2.4,2.9);
			\node[rotate=90] at (2.7,3.85) {\small{$T_2^{*}$}};
			\draw (0,4.8) -- (2.3,4.8);
			\node at (1,4.6) {\small{$-m$}};
			\node at (1,5) {\small{$H_1$}};
		\end{tikzpicture}	
	}
	\subcaptionbox{\ref{lem:ht=1_reduction}\ref{item:uniq_dumb}, $k\geq 3$\label{fig:uniq_dumb}}[.22\linewidth]{
		\begin{tikzpicture}
			\draw[densely dotted] (-0.1,0) -- (2.9,0);
			\node at (1.2,-0.2) {\small{$m-1$}};
			\node at (1.2,0.2) {\small{$H_2$}};
			\draw (0,-0.2) -- (0.2,0.8);
			\node at (0.1,1) {$\vdots$};
			\draw (0.2,1) -- (0,2);
			\draw [decorate, decoration = {calligraphic brace}, thick] (-0.2,-0.2) --  (-0.2,1.9);
			\node[rotate=-90] at (-0.5,0.8) {\small{$T$}}; 
			\draw (0,1.8) -- (0.2,3);
			\node at (-0.2,2.5) {\small{$-k$}};
			\draw (-0.1,2.2) -- (1.1,2.4);
			\node at (1.2,2.2) {$\cdots$};
			\node at (1.2,2.6) {\small{$[(2)_{k-1}]$}};
			\draw (1.3,2.4) -- (2.5,2.2);
			\draw (0.6,2.4) -- (0.8,1.2);
			\node at (0.5,1.6) {\small{$-1$}};
			\draw (0.2,2.8) -- (0,3.8);
			\node at (0.1,4) {$\vdots$};
			\draw (0,4) -- (0.2,5);
			\draw [decorate, decoration = {calligraphic brace}, thick] (-0.2,2.9) --  (-0.2,5);
			\node[rotate=-90] at (-0.5,3.85) {\small{$T^{*}$}};
			\draw (-0.1,4.8) -- (2.9,4.8);
			\node at (1.2,4.6) {\small{$-m$}};
			\node at (1.2,5) {\small{$H_1$}};
			\draw[densely dotted] (2.8,-0.2) -- (2.8,5);
			\node at (3.05,2.5) {\small{$F_1$}};
		\end{tikzpicture}	
	}
	\medskip
	
	\subcaptionbox{\ref{lem:ht=1_reduction}\ref{item:uniq_n=2_v=1}\label{fig:uniq_n=2_v=1}}[.35\linewidth]{
		\begin{tikzpicture}
			\draw (0,3) -- (4,3);
			\node at (1,3.2) {\small{$H_1$}};
			\node at (1,2.8) {\small{$-1$}};
			\draw (0,0) -- (0.2,1.2);
			\node at (-0.2,0.4) {\small{$-2$}};
			\draw (0.2,1) -- (0,2.2);
			\node at (-0.2,1.6) {\small{$-2$}};
			\draw (0,1.4) -- (1.2,1.6);
			\node at (0.6,1.7) {\small{$-1$}};
			\draw (0,2) -- (0.2,3.2);
			\node at (-0.2,2.6) {\small{$-2$}};
			\draw (2,0) -- (2.2,1.2);
			\node at (1.8,0.4) {\small{$-2$}};
			\draw (2.2,1) -- (2,2.2);
			\node at (1.8,1.6) {\small{$-2$}};
			\draw (2,1.4) -- (3.2,1.6);
			\node at (2.6,1.7) {\small{$-1$}};
			\draw (2,2) -- (2.2,3.2);
			\node at (1.8,2.6) {\small{$-2$}};
			\draw (3.8,0) -- (3.8,3.2);
			\node at (4.05,1.5) {\small{$F_1$}};
		\end{tikzpicture}
	}
	\subcaptionbox{\ref{lem:ht=1_reduction}\ref{item:uniq_fork-lc}\label{fig:uniq_fork-lc}}[.35\linewidth]{
		\begin{tikzpicture}
			\draw (0,3) -- (3.4,3);
			\node at (1.7,3.2) {\small{$H_1$}};
			\node at (1.7,2.8) {\small{$-2$}};
			\begin{scope}[shift={(3,0)}]
			\draw (0,0) -- (0.2,1.2);
			\node at (-0.2,0.4) {\small{$-2$}};
			\draw (0.2,1) -- (0,2.2);
			\node at (-0.2,1.6) {\small{$-2$}};
			\draw (0,1.4) -- (1.2,1.6);
			\node at (0.6,1.7) {\small{$-1$}};
			\draw (0,2) -- (0.2,3.2);
			\node at (-0.2,2.6) {\small{$-2$}};
			\end{scope}
			\begin{scope}[shift={(-2,0)}]
			\draw (2,0) -- (2.2,1.2);
			\node at (1.8,0.4) {\small{$-2$}};
			\draw (2.2,1) -- (2,2.2);
			\node at (1.8,1.6) {\small{$-3$}};
			\draw (2,1.4) -- (3.2,1.6);
			\node at (2.6,1.7) {\small{$-2$}};
			\draw (2.7,1.6) -- (2.9,0.4);
			\node at (2.6,0.7) {\small{$-1$}};
			\draw (3,1.6) -- (4.2,1.4);
			\node at (3.7,1.7) {\small{$-2$}};
			\draw (2,2) -- (2.2,3.2);
			\node at (1.8,2.6) {\small{$-2$}};
			\end{scope}
		\end{tikzpicture}
	}
	\subcaptionbox{\ref{lem:ht=1_reduction}\ref{item:uniq_A1E7}\label{fig:uniq_A1E7}}[.25\linewidth]{
		\begin{tikzpicture}
			\draw (0,3) -- (3,3);
			\node at (1.6,3.2) {\small{$H_1$}};
			\node at (1.6,2.8) {\small{$-1$}};
			\draw (0,0) -- (0.2,1.2);
			\node at (-0.2,0.4) {\small{$-2$}};
			\draw (0.2,1) -- (0,2.2);
			\node at (-0.2,1.6) {\small{$-2$}};
			\draw (0,1.4) -- (1.2,1.6);
			\node at (0.6,1.7) {\small{$-2$}};
			\draw (1,1.6) -- (2.2,1.4);
			\node at (1.6,1.7) {\small{$-2$}};
			\draw (2.1,1.6) -- (1.9,0.4);
			\node at (1.7,0.8) {\small{$-1$}};
			\draw (0,2) -- (0.2,3.2);
			\node at (-0.2,2.6) {\small{$-2$}};
			\draw (2.8,0) -- (2.8,3.2);
			\node at (3.05,1.7) {\small{$F_1$}};
		\end{tikzpicture}
	}
	\medskip

	\subcaptionbox{\ref{lem:ht=1_reduction}\ref{item:uniq_2}\label{fig:uniq_2}}[.3\linewidth]{
		\begin{tikzpicture}
			\draw (-0.1,0) -- (2.9,0);
			\node at (1.6,-0.2) {\small{$m-1$}};
			\node at (1.6,0.2) {\small{$H_2$}};
			\draw (0,-0.2) -- (0.2,0.8);
			\node at (0.1,1) {$\vdots$};
			\draw (0.2,1) -- (0,2);
			\draw [decorate, decoration = {calligraphic brace}, thick] (-0.2,-0.2) --  (-0.2,1.9);
			\node[rotate=-90] at (-0.5,0.8) {\small{$T$}}; 
			\draw (0,1.8) -- (0.2,3);
			\node at (-0.2,2.5) {\small{$-2$}};
			\draw (-0.1,2.2) -- (1.1,2.4);
			\node at (0.5,2.5) {\small{$-2$}};
			\draw (0.9,2.4) -- (2.1,2.2);
			\node at (1.6,2.5) {\small{$-1$}};
			\draw (0.2,2.8) -- (0,3.8);
			\node at (0.1,4) {$\vdots$};
			\draw (0,4) -- (0.2,5);
			\draw [decorate, decoration = {calligraphic brace}, thick] (-0.2,2.9) --  (-0.2,5);
			\node[rotate=-90] at (-0.5,3.85) {\small{$T^{*}$}};
			\draw (-0.1,4.8) -- (2.9,4.8);
			\node at (1.6,4.6) {\small{$-m$}};
			\node at (1.6,5) {\small{$H_1$}};
			\draw (2.7,-0.2) -- (2.7,5);
			\node at (2.95,2.5) {\small{$F_1$}};
		\end{tikzpicture}	
	}
	\subcaptionbox{\ref{lem:ht=1_reduction}\ref{item:uniq_-2-chain}, $m\geq 1$, $d\geq 2$\label{fig:uniq_-2-chain}}[.36\linewidth]{
		\raisebox{3em}{
			\begin{tikzpicture}
				\draw (0.2,1) -- (0,2);
				\node at (-0.2,1.4) {\small{$-d$}};
				\draw (0,1.8) -- (0.2,3);
				\node at (-0.2,2.5) {\small{$-2$}};
				\draw (-0.1,2.2) -- (1.1,2.4);
				\node at (0.5,2.5) {\small{$-2$}};
				\draw (0.9,2.4) -- (2.1,2.2);
				\node at (1.6,2.5) {\small{$-1$}};
				\draw (0.2,2.8) -- (0,3.8);
				\node at (0.1,4) {$\vdots$};
				\draw (0,4) -- (0.2,5);
				\node[rotate=-90] at (-0.3,3.8) {\small{$[(2)_{d-1}]$}};
				\draw (-0.1,4.8) -- (3.6,4.8);
				\node at (1.6,4.6) {\small{$-m$}};
				\node at (1.6,5) {\small{$H_1$}};
				\draw[densely dotted] (2.5,1) -- (2.5,5);
				\node at (2.75,2.5) {\small{$F_1$}};
				\draw[densely dotted] (3.5,1) -- (3.5,5);
				\node at (3.75,2.5) {\small{$F_2$}};
			\end{tikzpicture}
		}	
	}
	\subcaptionbox{\ref{lem:ht=1_reduction}\ref{item:uniq_fork}, $d,m\geq 2$\label{fig:uniq_fork}}[.3\linewidth]{
		\raisebox{.5em}{
			\begin{tikzpicture}
				\draw (0.2,1) -- (0,2);
				\node at (-0.2,1.4) {\small{$-d$}};
				\draw (0,1.8) -- (0.2,3);
				\node at (-0.2,2.5) {\small{$-2$}};
				\draw (-0.1,2.2) -- (1.1,2.4);
				\node at (0.5,2.5) {\small{$-2$}};
				\draw (0.9,2.4) -- (2.1,2.2);
				\node at (1.6,2.5) {\small{$-1$}};
				\draw (0.2,2.8) -- (0,3.8);
				\node at (0.1,4) {$\vdots$};
				\draw (0,4) -- (0.2,5);
				\node[rotate=-90] at (-0.3,3.8) {\small{$[(2)_{d-1}]$}};
				\draw (-0.1,4.8) -- (3.3,4.8);
				\node at (1.6,4.6) {\small{$-m$}};
				\node at (1.6,5) {\small{$H_1$}};
				\draw (3,-0.2) -- (3.2,0.8);
				\node at (3.1,1) {$\vdots$};
				\draw (3.2,1) -- (3,2);
				\draw [decorate, decoration = {calligraphic brace}, thick] (3.4,1.9) --  (3.4,-0.2);
				\node[rotate=90] at (3.7,0.8) {\small{$T$}}; 
				\draw (3,1.8) -- (3.2,3);
				\node at (2.8,2.5) {\small{$-2$}};
				\draw (2.9,2.2) -- (4.1,2.4);
				\node at (3.6,2.5) {\small{$-1$}};
				\draw (3.2,2.8) -- (3,3.8);
				\node at (3.1,4) {$\vdots$};
				\draw (3,4) -- (3.2,5);
				\draw [decorate, decoration = {calligraphic brace}, thick] (3.4,5) -- (3.4,2.9);
				\node[rotate=90] at (3.7,3.85) {\small{$T^{*}$}};
			\end{tikzpicture}
		}	
	}
	\medskip
	
	\subcaptionbox{\ref{lem:ht=1_reduction}\ref{item:uniq_3}\label{fig:uniq_3}}[.22\linewidth]{
		\begin{tikzpicture}
			\draw (0,3) -- (2.2,3);
			\node at (1.6,3.2) {\small{$H_1$}};
			\node at (1.6,2.8) {\small{$-2$}};
			\draw (0,0.2) -- (0.2,-1);
			\node at (-0.2,-0.6) {\small{$-2$}};
			\draw (0,0) -- (0.2,1.2);
			\node at (-0.2,0.6) {\small{$-2$}};
			\draw (0.2,1) -- (0,2.2);
			\node at (-0.2,1.6) {\small{$-2$}};
			\draw (0,1.4) -- (1.2,1.6);
			\node at (0.6,1.7) {\small{$-2$}};
			\draw (1,1.6) -- (2.2,1.4);
			\node at (1.6,1.7) {\small{$-1$}};
			\draw (0,2) -- (0.2,3.2);
			\node at (-0.2,2.6) {\small{$-3$}};
		\end{tikzpicture}
	}
	\subcaptionbox{\ref{lem:ht=1_reduction}\ref{item:uniq_5}\label{fig:uniq_5}}[.22\linewidth]{
		\begin{tikzpicture}
			\draw (0,3) -- (2.2,3);
			\node at (1.6,3.2) {\small{$H_1$}};
			\node at (1.6,2.8) {\small{$-3$}};
			\draw (0,0) -- (0.2,1.2);
			\node at (-0.2,0.4) {\small{$-2$}};
			\draw (0.2,1) -- (0,2.2);
			\node at (-0.2,1.6) {\small{$-2$}};
			\draw (0,1.4) -- (1.2,1.6);
			\node at (0.6,1.7) {\small{$-2$}};
			\draw (1,1.6) -- (2.2,1.4);
			\node at (1.6,1.7) {\small{$-2$}};
			\draw (2.1,1.6) -- (1.9,0.4);
			\node at (1.7,0.8) {\small{$-1$}};
			\draw (0,2) -- (0.2,3.2);
			\node at (-0.2,2.6) {\small{$-2$}};
		\end{tikzpicture}
	}
	\subcaptionbox{\ref{lem:ht=1_reduction}\ref{item:uniq_Ek}, $k\in \{2,3,4\}$\label{fig:uniq_Ek}}[.25\linewidth]{
		\begin{tikzpicture}
			\draw (0,3) -- (2.6,3);
			\node at (1.6,3.2) {\small{$H_1$}};
			\node at (1.6,2.8) {\small{$-2$}};
			\draw (0,0) -- (0.2,1.2);
			\node at (-0.2,0.4) {\small{$-2$}};
			\draw (0.2,1) -- (0,2.2);
			\node at (-0.2,1.6) {\small{$-2$}};
			\draw (0,1.4) -- (1.2,1.6);
			\node at (1.3,1.8) {\small{$[(2)_{k}]$}};
			\node at (1.3,1.5) {$\cdots$};
			\draw (1.4,1.6) -- (2.6,1.4);
			\draw (2.5,1.6) -- (2.3,0.4);
			\node at (2.1,0.8) {\small{$-1$}};
			\draw (0,2) -- (0.2,3.2);
			\node at (-0.2,2.6) {\small{$-2$}};
		\end{tikzpicture}
	}
	\subcaptionbox{\ref{lem:ht=1_reduction}\ref{item:uniq-bench}, $m\geq 3$\label{fig:uniq-bench}}[.25\linewidth]{
		\begin{tikzpicture}
			\draw (0,3) -- (2.4,3);
			\node at (1,3.2) {\small{$H_1$}};
			\node at (1,2.8) {\small{$-m$}};
			\draw (0,0) -- (0.2,1.2);
			\node at (-0.2,0.4) {\small{$-2$}};
			\draw (0.2,1) -- (0,2.2);
			\node at (-0.2,1.6) {\small{$-2$}};
			\draw (0,1.4) -- (1.2,1.6);
			\node at (0.6,1.7) {\small{$-2$}};
			\draw (1.1,1.7) -- (1.3,0.5);
			\node at (1,0.9) {\small{$-1$}};
			\draw (0,2) -- (0.2,3.2);
			\node at (-0.2,2.6) {\small{$-2$}};
			\draw (2,0) -- (2.2,1.2);
			\node at (1.8,0.4) {\small{$-2$}};
			\draw (2.2,1) -- (2,2.2);
			\node at (1.8,1.6) {\small{$-2$}};
			\draw (2,1.4) -- (3.2,1.6);
			\node at (2.6,1.7) {\small{$-2$}};
			\draw (3.1,1.7) -- (3.3,0.5);
			\node at (3,0.9) {\small{$-1$}};
			\draw (2,2) -- (2.2,3.2);
			\node at (1.8,2.6) {\small{$-2$}};
		\end{tikzpicture}
	}
	\caption{Auxiliary pairs $(Z,D_Z)$ from Lemma \ref{lem:ht=1_reduction}, cf.\ Table \ref{table:ht=1_exceptions}. Dotted lines correspond to curves which may or may not be components of $D_Z$ (Lemma \ref{lem:ht=1_uniqueness} covers all such possibilities).}
	\label{fig:basic_ht=1}
\end{figure}
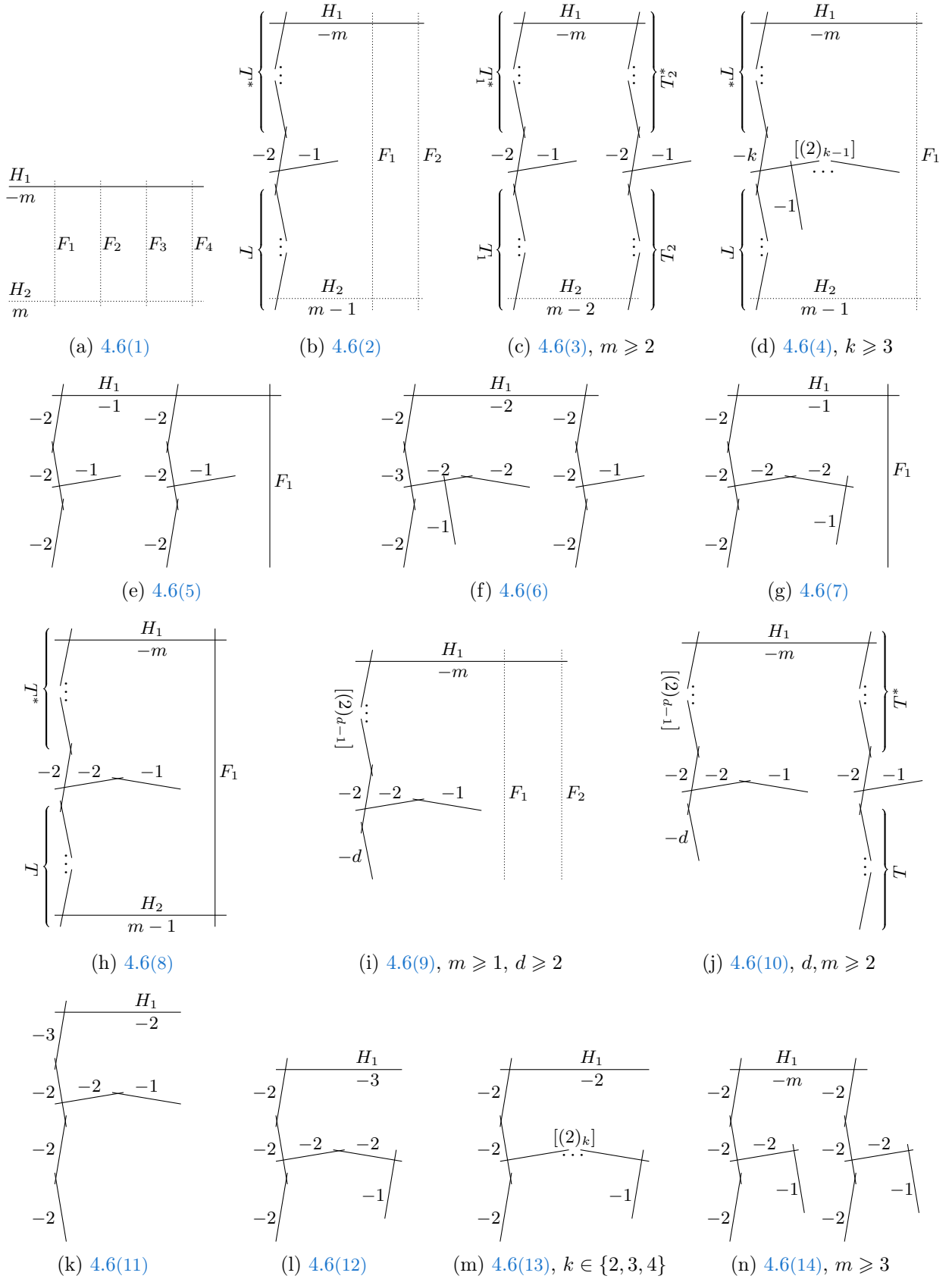

\begin{table}[h]
			\begin{tabular}{l|c|rl|l}	
				$(Z,D_Z)$ from \ref{lem:ht=1_reduction} & $\#\cP_{+}(\cZ)$
				& \multicolumn{2}{c|}{corresponding $(X,\check{D})$ from \ref{lem:ht=1_types}, if exists} & \multicolumn{1}{c}{exception in \ref{lem:ht=1_uniqueness}\ref{item:ht=1_uniqueness-exceptions}}
				\\ \hline \hline
				\ref{item:uniq_easy} $v\leq 3$ &
				$1$ & 
				\ref{item:chains_columnar_T1=0,T2=0},
				\ref{item:chains_columnar_T2=0},
				\ref{item:chains_columnar},
				\ref{item:beta=3_columnar} & 
				&\\
				\ref{item:uniq_easy} $v=4$ &
				$\cM^1$ &
				\ref{item:ht=1_bench} &
				&\\ \hline 
				\multirow{3}{*}{\ref{item:uniq_n=3} $v\leq 1$} &
				\multirow{3}{*}{1} & 	
				\ref{item:chains_both_T1=0}, 
				\ref{item:chains_both} 
				& &\\
				&&
				\ref{item:T2*=[2,2]_b>2},  
				\ref{item:T1=0,T=T2=[2]_b>2},  
				\ref{item:TH=[3,2]_T=[3]_b>2},
				\ref{item:TH=[2,2]_b>2}  & $b\geq 3$ &\\
				&& \ref{item:TH_long_columnar_b>2}, 
				\ref{item:T2=[2,2]_b>2}, 
				\ref{item:T2=[2]_b>2} 
				 & $b\geq 3$ 
				& \\
				\hline 
				\multirow{2}{*}{\ref{item:uniq_n=3} $v=2$} & \multirow{2}{*}{1} & 
				\ref{item:beta=3_other},
				\ref{item:beta=3_other_2} && 
					\multirow{2}{*}{$\cha\kk|d(T)$} \\
				&& 	\ref{item:ht=1_bench-big_b>2} & $b\geq 3$ &
				 \\
				\hline 
				\multirow{3}{*}{\ref{item:uniq_n=2}}
				& \multirow{3}{*}{1} &
				\ref{item:chains_not-columnar} & &
				\multirow{3}{*}{$\cha\kk|d([T_{1}^{*},m,T_{2}^{*}])$}
				\\
				&& \ref{item:TH_long_other_b>2}, \ref{item:F'_fork_T=[3]_b>2} & $b\geq 3$ &\\
				&& \ref{item:ht=1_bench_two_a,b>2} & $a,b\geq 3$
				&  \\ \hline 
				\ref{item:uniq_dumb}
				& 1  &
				\ref{item:TH=[3,2]_tip},
				\ref{item:TH=[2,2]_tip}, 
				\ref{item:TH=[2,2]_tip_[3,2]},  
				\ref{item:T2=[2],T=[2,2],tip} & 
				&\\ \hline
				\ref{item:uniq_n=2_v=1}
				& $\cM^1$
				& \ref {item:beta=3_other_3} & 
				&
				$\cha\kk=2$
				\\ \hline
				\ref{item:uniq_fork-lc} 
				& $\cM^1$& 
				\ref{item:F'_fork_T=[2,2]} & 
				&\\
				\ref{item:uniq_A1E7}
				& $2$  & 
				\ref{item:T2=[2]_b>2} & $b=2$, $T=[2,2]$ &\\
				%
				\ref{item:uniq_2}
				& $2$ & \multicolumn{2}{c|}{\small{none (case \ref{lem:ht=1_reduction}\ref{item:uniq_2} is auxiliary)}} &\\
				%
				\ref{item:uniq_3} 
				& $2$ &
				\ref{item:T2*=[2,2]_b>2} & $b=2$, $T=[3]$ & \\ \hline
			\end{tabular}
			\vspace{1em}
			
			\begin{tabular}{l|cc|rl|c|l}
				\multirow{2}{*}{$(Z,D_Z)$ from \ref{lem:ht=1_reduction}} & \multicolumn{2}{c|}{$\#\cP_{+}(\cZ)$ if}
				&  \multicolumn{3}{c|}{corresponding $(X,\check{D})$ from \ref{lem:ht=1_types}, if exists} & \multicolumn{1}{c}{\multirow{2}{*}{exception in \ref{lem:ht=1_uniqueness}\ref{item:ht=1_uniqueness-exceptions}}} \\
				& $\cha \kk\nmid d$ & $\cha\kk | d$ & \multicolumn{2}{c|}{ } & $d$ & \\ \hline \hline
				\multirow{3}{*}{\ref{item:uniq_-2-chain} $v\leq 1$}
				& 
				\multirow{4}{*}{$1$}
				& 
				\multirow{4}{*}{$2$} &  
				\ref{item:T1=0,T=T2=[2]_b>2}, \ref{item:TH_long_columnar_b>2} & $b=2$ & \multirow{2}{*}{$2$} & \\ 
				&&& \ref{item:TH=[3,2]_T=[3]_b>2}, \ref{item:TH=[2,2]_b>2}, \ref{item:T2=[2]_b>2} & $b=2$, $T\in \{[k],[2,3]\}$ &&\\ 
				\cline{4-7}
				&& &
				\ref{item:T2*=[2,2]_b>2} &  $b=2$, $T=[2,2]$    & \multirow{2}{*}{$3$} &\\
				&&& \ref{item:T2=[2,2]_b>2} & $b=2$ &
				& \\ \hline 
				\ref{item:uniq_-2-chain} $v=2$ 
				& 
				$1$
				& 
				$\cM^1$
				&
				\ref{item:ht=1_bench-big_b>2} & $b=2$ & $2$ &
				$\cha\kk|d$ \\ \hline 
				%
				\multirow{2}{*}{\ref{item:uniq_fork}}
				& 
				\multirow{2}{*}{$1$}
				& 
				\multirow{2}{*}{$\cM^1$}
				&
				\ref{item:TH_long_other_b>2}, \ref{item:F'_fork_T=[3]_b>2} &  $b=2$ &\multirow{2}{*}{$2$} & \multirow{2}{*}{$\cha\kk |d([(2)_{d-1},m,T^{*}])$} \\
				&&& \ref{item:ht=1_bench_two_a,b>2} & $a=2$, $b\geq 3$  && \\ \hline
			\end{tabular}
			\vspace{1em}
			
			\begin{tabular}{l|cccc|l|l}
				\multirow{2}{*}{$(Z,D_Z)$ from \ref{lem:ht=1_reduction}} & \multicolumn{4}{c|}{$\#\cP_{+}(\cZ)$ if $\cha \kk$} 
				& \multicolumn{1}{c|}{corresponding $(X,\check{D})$ from \ref{lem:ht=1_types},} & \multicolumn{1}{c}{\multirow{2}{*}{exception in \ref{lem:ht=1_uniqueness}\ref{item:ht=1_uniqueness-exceptions}}} \\
				& $\neq 2,3,5$ & $=5$ & $=3$ & $=2$ & \multicolumn{1}{c|}{if exists} &\\ \hline \hline
				\ref{item:uniq_5} 
				& 1 & 2 & 1 & 2 &
				\ref{item:TH=[3,2]_T=[3]_b>2}, $b=2$, $T=[2,2]$ &\\
				%
				\ref{item:uniq_Ek} $k=2$
				& 1 & 1 & 2 & 2 
				& \ref{item:TH=[2,2]_b>2}, $b=2$, $T\in\{[2,2],[3,2]\}$  & \\
				%
				\ref{item:uniq_Ek} $k=3$
				& 1 & 1 & 2 & 3 
				& \ref{item:TH=[2,2]_b>2}, $b=2$, $T=[2,2,2]$ &\\	
				\ref{item:uniq_Ek} $k=4$ 
				& 2 & 2 & 3 & 3 
				& \ref{item:TH=[2,2]_b>2}, $b=2$, $T=[2,2,2,2]$ & $\cha\kk=2$\\	
				\ref{item:uniq-bench} 
				& 1 & 1 & 1 &  $\cM^2$ 
				& \ref{item:ht=1_bench_two_a,b>2}, $a=b=2$ & $\cha\kk|2(m-1)$ \\ \hline 	
			\end{tabular} \vspace{-0.5em}
	\caption{Lemma \ref{lem:ht=1_reduction}: combinatorial types $\cZ$ of log surfaces obtained from the ones in Lemma \ref{lem:ht=1_types} by a vertical inner snc-minimalization, with the number $\#\cP_{+}(\cZ)$ computed in Lemma \ref{lem:ht=1_uniqueness}.}
	\label{table:ht=1_exceptions}\vspace{-0.5em}
\end{table}

\begin{table}[h]
	\begin{tabular}{l|c|c|c}
		$(Z,D_Z)$ from \ref{lem:ht=1_reduction} & $\cha\kk$ divides & $\#\cP_{+}(\cZ)$ & $h^{1}(\lts{Z}{D_Z})$ \\ \hline
		\ref{item:uniq_n=3} $v=2$ &  $d(T)$ & $1$ & $1$ \\
		\ref{item:uniq_n=2} & $d([T_{1}^{*},m,T_{2}^{*}])$ & $1$ & $1$ \\
		\ref{item:uniq_n=2_v=1} & $2$ & $\cM^{1}$ & $2$ \\
		\ref{item:uniq_-2-chain} $v=2$ & $d$ &  $\cM^{1}$ & $1,2$ \\\hline 
		\ref{item:uniq_fork} $\cha\kk\nmid d$ & \multirow{2}{*}{$d([(2)_{d-1},m,T^{*}])$} & $1$ & $1$ \\
		\ref{item:uniq_fork} $\cha\kk\ |\ d$ & & $\cM^{1}$ & $1,2$ \\ \hline 
		\ref{item:uniq_Ek} $k=4$ & 2 & $3$ & $0,1,3$ \\ \hline 
		\ref{item:uniq-bench} $\cha\kk \neq 2$ & \multirow{2}{*}{$2(m-1)$} & $1$ & $1$  \\
		\ref{item:uniq-bench} $\cha\kk = 2$ & & $\cM^{2}$ & $2,3$ 
	\end{tabular} \vspace{-0.5em}
	\caption{Lemma \ref{lem:ht=1_uniqueness}\ref{item:ht=1_uniqueness-exceptions}: cases when the family is not $h^{1}$-stratified or universal.}
	\label{table:exceptions-to-exceptions}\vspace{-1em}
\end{table}	

We now state the main technical lemma, which is a counterpart of Proposition \ref{prop:moduli} for log surfaces $(Z,D_{Z})$. We prove it at the end of this section. For the relevant definitions see Section \ref{sec:moduli}.

\begin{lemma}[Moduli computation]\label{lem:ht=1_uniqueness}
	Let $\cZ$ be the combinatorial type of a log surface $(Z,D_{Z})$ from Lemma \ref{lem:ht=1_reduction}. 
	\begin{enumerate}
		\item\label{item:ht=1_uniqueness-finite} If $\#\cP_{+}(\cZ)<\infty$ then $\cP_{+}(\cZ)$ is represented by a stratified family; and $\#\cP_{+}(\cZ)$ is listed in Table \ref{table:ht=1_exceptions}. 
		\item\label{item:ht=1_uniqueness-infinite} If $\#\cP_{+}(\cZ)=\infty$ then $\cP_{+}(\cZ)$ is represented by an $\Aut(\cZ)$-faithful family, whose base $B$ and symmetry group is listed in Table \ref{table:bases-primitive}. The corresponding entry in Table \ref{table:ht=1_exceptions} is $\cM^{\dim B}$, see Notation \ref{not:Md}.
		\item\label{item:ht=1_uniqueness-exceptions} 
			The family from \ref{item:ht=1_uniqueness-finite}  or \ref{item:ht=1_uniqueness-infinite} can be chosen  $h^{1}$-stratified or universal, respectively, unless $\cZ$ is as in Table \ref{table:exceptions-to-exceptions}.	
		\end{enumerate}
\end{lemma}	

\begin{corollary}[Case $\height=1$ of Propositions \ref{prop:moduli}, \ref{prop:moduli-hi}]
	\label{cor:moduli-ht=1}
	The set of isomorphism classes of del Pezzo surfaces of rank 1, height 1 and fixed singularity type is represented by a family with properties listed in Propositions \ref{prop:moduli},~\ref{prop:moduli-hi}.
\end{corollary}
\begin{proof}
	Fix a singularity type $\cS$. 
	Let $\Phtt(\cS)$ be the set of isomorphism classes of del Pezzo surfaces of rank one, height one and singularity type $\cS$, and let $\Phttres(\cS)\subseteq \cP(\cS)$ be the set of minimal log  resolutions of surfaces in $\Phtt(\cS)$.  On each  $(X,D)\in \Phttres(\cS)$, consider a $\P^{1}$-fibration given by Lemma  \ref{lem:ht=1_types}. Let $\check{D}$ be the sum of $D$ and all vertical $(-1)$-curves, and let $\check{\cS}$ be the combinatorial type of $(X,\check{D})$. Since type $\cS$ appears exactly once in the list of Lemma  \ref{lem:ht=1_types}, it uniquely determines the extended type $\check{\cS}$, and $\Phttres(\cS)$ is the image of $\cP(\check{\cS})$ by the natural restriction map $\cP(\check{\cS})\to \cP(\cS)$. Let $\cZ$ be the corresponding type from Lemma \ref{lem:ht=1_uniqueness}, see Table \ref{table:ht=1_exceptions}. 
	
	By Lemma \ref{lem:inner} we have $\#\cP_{+}(\check{\cS})=\#\cP_{+}(\cZ)$.  	Assume $\#\cP_{+}(\cZ)<\infty$. By Lemma \ref{lem:ht=1_uniqueness}\ref{item:ht=1_uniqueness-finite}  $\cP_{+}(\cZ)$ is represented by a stratified family, which is $h^{1}$-stratified unless $\cZ$ is as in Table \ref{table:exceptions-to-exceptions}. By Lemma \ref{lem:inner} the same is true for $\cP_{+}(\check{\cS})$, with the corresponding exceptional types $\cS$ listed in Table \ref{table:exceptions-to-moduli}. By Lemma \ref{lem:adding-1}\ref{item:adding-1-h1}, the same holds for $\Phttres(\cS)$, and $\#\Phttres(\cS)=\#\cP_{+}(\check{\cS})=\#\cP_{+}(\cZ)$. If $\#\cP_{+}(\cZ)=1$, this means that  $\#\Phttres(\cS)=1$, and $(X,D)\in\Phttres(\cS)$ has $h^{1}(\lts{X}{D})=0$ (or $1$ for the exceptions). Otherwise, the number  $\#\Phttres(\cS)=\#\Phtt(\cS)$ (and $h^{1}(\lts{X}{D})$ in the exceptional case $\cS=\rE_8$, $\cha\kk=2$) are  listed in Tables \ref{table:exceptions} and \ref{table:exceptions-to-moduli}, as needed.
	
	Assume that $\#\cP_{+}(\cZ)$ is infinite. By Lemma \ref{lem:ht=1_uniqueness}\ref{item:ht=1_uniqueness-infinite}  $\cP_{+}(\cZ)$ is represented by an $\Aut(\cZ)$-faithful family whose base and symmetry group is listed in Table \ref{table:bases-primitive}. This family is universal unless $\cZ$ is as in Table \ref{table:exceptions-to-exceptions}.
	
	Looking at the list in Lemma \ref{lem:ht=1_types} we check directly that 
	\begin{equation}\label{eq:preserving-Dhor}
		\mbox{every automorphism of the weighted graph of }D\mbox{ preserves }D\hor. 	
	\end{equation}
	Observation \eqref{eq:preserving-Dhor} implies that 
	is $\Aut(\check{\cS})$-equivariant. To see this, write $\psi=\psi_{2}\circ \psi_{1}$ for some morphisms $\psi_{1}$, $\psi_{2}$, let $L$ be a $(-1)$-curve in $\Exc\psi_{2}$, and let $L'$ be a component of $(\psi_{1})_{*}\check{D}$ corresponding to a vertex in the same $\Aut(\check{\cS})$-orbit. Then $L$ and $L'$ are vertical $(-1)$-curves with of branching number $2$ in $(\psi_{1})_{*}\check{D}$. Since our $\P^{1}$-fibration has height one, $L$ and $L'$ are disjoint, so $\psi_2$ contracts them both, as needed.
	
	Thus $\Aut(\check{\cS})$ gets identified with a subgroup of $\Aut(\cZ)$, see Section \ref{sec:moduli}.
	Hence by Lemma \ref{lem:inner} the set $\cP_{+}(\check{\cS})$ is represented by an $\Aut(\check{\cS})$-faithful family with the same base and symmetry group; universal unless $\cZ$ is as in Table \ref{table:exceptions-to-exceptions}. Observation \eqref{eq:preserving-Dhor} and Lemma \ref{lem:adding-1-criterion} imply that this family satisfies condition \eqref{eq:E-invariant}, so by Lemma \ref{lem:adding-1}\ref{item:adding-1-faithful} it restricts to an almost faithful family representing $\Phttres(\cS)$, as needed. 
	
	By Lemma \ref{lem:blowup-hi}\ref{item:blowup-hi-inner} we have $h^i=h^i(\lts{Z}{D_Z})$. Let $\tau\colon (Z,D_Z)\to(\F_{m},B)$ be the contraction of all vertical curves which are disjoint from $H_1$. Then $B$ is a sum of the  negative section and $\nu$ fibers, where $\nu$ is the number of degenerate fibers in $\check{D}$. Put $\bar{h}^i=h^i(\lts{\F_{m}}{B})$, $\bar{\chi}=\chi(\lts{\F_{m}}{B})$. A direct computation as in \cite[Lemma 2.5(2)]{FZ-deformations} gives $\bar{h}^2=0$ and $\bar{\chi}=5+m-\nu$. By Lemma \ref{lem:blowup-hi}, $h^2=\bar{h}^2=0$ and $\chi=\bar{\chi}-\epsilon$, where $\epsilon$ is the number of outer blowups in the decomposition of $\tau$. Computing $\epsilon$ shows that $\chi$ is as in Table \ref{table:ht=1}.
\end{proof}

\begin{table}[h]
{\renewcommand{\arraystretch}{1.1}		\begin{tabular}{r||c|c|c|c|c|c}
		$(Z,D_Z)$ from \ref{lem:ht=1_reduction} &
		\ref{item:uniq_easy} $v=4$ & 
		\ref{item:uniq_n=2_v=1} & 
		\ref{item:uniq_fork-lc} &
		\ref{item:uniq_-2-chain} $v=2$, $\cha\kk|d$ &
		\ref{item:uniq_fork} $\cha\kk|d$ &
		\ref{item:uniq-bench} $\cha\kk=2$ \\ \hline
		base & 
		$\P^1\setminus\{0,1,\infty\}$  & 
		\multicolumn{2}{c|}{$\A^1_{*}$} & 
		$\A^1$ & 
		$\A^1$ & 
		$\A^2$ \\ \hline
		symmetry group &
		$S_3$ &
		$\Z/2$ &
		$\{\id\}$ & 
		$\Z/\rr(d)$ &
		$\Z/\rr(d([(2)_{d-1},m,T^{*}]))$ & 
		$\Z/2$ 
	\end{tabular}\vspace{1em}
	
	\begin{tabular}{r||c|c|c|c|c|c|c}
		$\cha \kk$ & \multicolumn{3}{c|}{any} & \multicolumn{4}{c}{$2$} \\ \hline 
		singularity type of $\bar{X}$ &
		\ref {item:beta=3_other_3} &  
		\ref{item:F'_fork_T=[2,2]} & 
		\ref{item:ht=1_bench} &
		\ref{item:TH_long_other_b>2} &
		\ref{item:F'_fork_T=[3]_b>2}, 
		\ref{item:ht=1_bench-big_b>2} &
		\multicolumn{2}{c}{\ref{item:ht=1_bench_two_a,b>2}}
		\\
		from Lemma \ref{lem:ht=1_types}  &&&& $b=2 $ & $b=2$ & $a=2$, $b\geq 3$ & $a=b=2$ 
		\\ \hline
		base &
		\multicolumn{2}{c|}{$\A^{1}_{*}$} &
		$\P^1\setminus \{0,1,\infty\}$ & 
		\multicolumn{3}{c|}{$\A^1$} &
		$\A^{2}$ \\ \hline
		symmetry group & 
		$\Z/2$ &
		$\{\id\}$ & 
		$S_3$ &
		$\Z/\rr(d[2,m,T^{*}])$ & 
		$\{\id\}$ &
		$\Z/\rr(m-1)$ &
		$\Z/2$  
	\end{tabular}}\vspace{-0.5em}
	\caption{Bases and symmetry groups of almost faithful families constructed in Lemma  \ref{lem:ht=1_uniqueness}\ref{item:ht=1_uniqueness-infinite}. For $n\in \Z$ we define $\rr(n)\in \Z$ as the quotient of $n$ by the highest power of $\cha\kk$ dividing $n$.}
	\label{table:bases-primitive}
\end{table} \vspace{-1em}

\begin{remark}[Comparison of Lemma \ref{lem:ht=1_uniqueness} with some results in the literature]
	To give a self-contained treatment, we present a complete proof of Lemma \ref{lem:ht=1_uniqueness} below. However, explicit computations in various cases were already made in literature. For example, \cite[Example 4.16]{FZ-deformations} treats the case \ref{lem:ht=1_reduction}\ref{item:uniq_n=2_v=1}; and a straightforward generalization of the computations from loc.\ cit.\ allows one to treat cases \ref{lem:ht=1_reduction}\ref{item:uniq_n=3} and \ref{lem:ht=1_reduction}\ref{item:uniq_n=2}, too. 
	
	Case \ref{lem:ht=1_reduction}\ref{item:uniq_Ek} essentially classifies del Pezzo surfaces of rank one and type $\rE_{k+4}$ (together with a specific $(-1)$-curve on the minimal resolution, whose existence is known, see  \cite[Theorem 3.4]{Ye}). Thus Lemma \ref{lem:ht=1_uniqueness} in this case can be inferred e.g.\ from explicit equations in weighted projective spaces given in \cite{Ye,Kawakami_Nagaoka_canonical_dP-in-char>0}.
	
	Recall that the proof of Lemma \ref{lem:ht=1_uniqueness} relies on describing the action of $\Aut(Y,D_Y)$ on $G^{\circ}$, where $(Y,D_Y)$ is obtained from $(Z,D_Z)$ by an inner blowdown, and the curve $G^{\circ}$ is the locus of its possible centers. Thus we study specific automorphisms of the affine surface $S_Y=Y\setminus (D_Y-G)$. The entire group $\Aut(S_Y)$ is much more complicated than $\Aut(Y,D_Y)$, but in some cases it is well described in the literature. For example, in case \ref{lem:ht=1_reduction}\ref{item:uniq_-2-chain} with  $v=1$, see Figure \ref{fig:uniq_-2-chain}, the log surface $(Y,D_Y)$ is as in Figure \ref{fig:uniq_n=3} (with $h=v=1$, $T=[d]$), and $S_Y$ is a Gizatullin surface. If $\cha\kk | d$, \cite{DG-example} shows that $\Aut(S_Y)$ has a fixed point, which is the reason why in this case we get $\#\cP_{+}(\cZ)=2$, cf.\ Lemma \ref{lem:outer}\ref{item:outer-fixed-point}. If $\cha\kk=0$, $\Aut(S_Y)$ is usually transitive \cite{Kovalenko_Gizatullin-surfaces}, and indeed we get $\#\cP_{+}(\cZ)=1$, cf.\  Lemma \ref{lem:outer}\ref{item:outer-transitive}. Of course, in the proof below we do not use the description of $\Aut(S_Y)$, but work directly with its well-behaved subgroup $\Aut(Y,D_Y)$. 
\end{remark}

\begin{proof}[Proof of Lemma \ref{lem:ht=1_uniqueness}]
	We study each case of Lemma \ref{lem:ht=1_reduction} separately. 
	\begin{casesp*} 
	\litem{\ref{lem:ht=1_reduction}\ref{item:uniq_easy}} In this case, $Z$ is a Hirzebruch surface $\F_{m}$ for some $m\geq 0$, and $D_{Z}$ is a union of $v\leq 4$ fibers, a negative section $H_1=[m]$, and possibly a section $H_2=[-m]$ disjoint from $H_1$. By Lemma \ref{lem:h1}\ref{item:h1-Fm} we have 
	$h^{1}(\lts{Z}{D_Z})=0$ if $v\leq 3$, and $h^{1}(\lts{Z}{D_Z})=1$ if $v=4$.
	
	Assume $h=2$. Performing elementary transformations on a fiber in $D_{Z}$ and applying Lemma \ref{lem:inner} we can assume $m=0$. If $v=3$ then $\#\cP_{+}(\cZ)=1$, so $\cP_{+}(\cZ)$ is represented by an $h^{1}$-stratified family over a point, as claimed in \ref{item:ht=1_uniqueness-finite}. If $v=4$, Examples \ref{ex:4-points} and \ref{ex:4-points-Aut} show that $\cP_{+}(\cZ)$ is represented by a universal $\Aut(\cZ)$-faithful family, with one-dimensional base parametrizing the choice of the fourth fiber, as claimed in \ref{item:ht=1_uniqueness-infinite} and \ref{item:ht=1_uniqueness-exceptions}.
	
	Assume $h=1$. If $v\leq 3$, case $h=2$ above implies that $\#\cP_{+}(\cZ)=1$, as needed. Assume $v=4$. Choose a section $H\not \subseteq D_Z$ disjoint from $H_1$, and let $\cZ'$ be the combinatorial type of $(Z,D_Z+H)$. An explicit description of $\Aut(\F_{m})$ in \cite[\sec 6.1]{Blanc_Lukecin} shows that $\Aut(Z,D_Z)$ acts transitively on the linear system $|H|$ (and on $|H|\setminus \{H_1\}$ in case $m=0$), so the restriction map $\cP_{+}(\cZ')\to \cP_{+}(\cZ)$ is bijective. Thus the universal, $\Aut(\cZ')$-faithful family representing $\cP_{+}(\cZ')$ constructed in case $h=2$ above restricts to a faithful family representing $\cP_{+}(\cZ)$. Since $\Aut(\cZ)\subseteq \Aut(\cZ')$, this family is $\Aut(\cZ)$-faithful. The fact that it is universal follows from \cite[Proposition 1.7(1)]{FZ-deformations} and from the fact that the dimensions $h^{1}(\lts{Z}{(D_Z+H)})$ and $h^{1}(\lts{Z}{D_Z})$ are equal.
	\litem{\ref{lem:ht=1_reduction}\ref{item:uniq_n=3}} 
	Let $\sigma\colon (Z,D_Z)\to (Y,D_Y)$ be the contraction of the $(-1)$-tip of $D_Z$, and let $G\subseteq D_Y$ be the component containing its image. Put $G^{\circ}\de G\setminus (D_Y-G)$. 
	Assume $h=2$, $v\geq 1$. By Lemma  \ref{lem:inner} performing elementary transformations on $F_1$ we can further assume that $m=1$. Contracting all vertical curves in $D_Y$ which do not meet the image of $H_1$ we get an inner birational  morphism $\tau\colon (Y,D_Y)\to (\F_{1},\hat{\pp})$ where $(\F_{1},\hat{\pp})$ is as in case \ref{lem:ht=1_reduction}\ref{item:uniq_easy} above, with $v\leq 3$; so $(\F_1,\hat{\pp})$ is unique up to an isomorphism and has $h^{1}(\lts{\F_1}{\hat{\pp}})=0$. In other words, if $\cY$ and $\cF$ are the combinatorial types of $(Y,D_Y)$ and $(\F_{1},\hat{\pp})$, then $\cP_{+}(\cF)$ is represented by an $h^{1}$-stratified family over a point. By Lemma \ref{lem:inner} the same is true for $\cY$. 
	
	We claim that $\Aut(Y,D_Y)$ acts transitively on $G^{\circ}$. Let $\phi\colon(Y,D_Y)\to (\P^2,\pp)$ be a composition of $\tau$ with the contraction of the negative section $\Sec_1$. Then $\pp$ is a sum of lines $\ll_{0}$, $\ll_j\de\phi(\sigma(F_j))$, $j\in \{1,\dots,v\}$ which pass through the image of $H_1$; and a line $\ll\de \phi(\sigma(H_{2}))$ which does not. Fix coordinates $[x:y:z]$ on $\P^2$ so that $\ll=\{x=0\}$, $\ll_0=\{y=0\}$, $\ll_1=\{z=0\}$, and consider a $\G_{m}$-action on $(\P^2,\pp)$ given by $\alpha_{\lambda}\colon[x:y:z]\mapsto[\lambda x:y:z]:\lambda \in \kk^{*}$. Put $\cc_{\lambda}=\{\lambda x^{d}=y^{e}\}$, where $d=d(T)$, $e=d(T^{*}-\ftip{T^{*}})$. Then $\kk^{*}\ni \lambda\mapsto G^{\circ}\cap \tau^{-1}_{*}\cc_{\lambda}\in G^{\circ}$ parametrizes $G^{\circ}$. Hence the transitive action of $\Aut(\P^2,\pp)$ on $\{\cc_{\lambda}\}_{\lambda\in \kk^{*}}$, given by $\alpha_{\lambda}(\cc_{1})=\cc_{\lambda^{d}}$  lifts to a transitive action of $\Aut(Y,D_Y)$ on $G^{\circ}$; as needed. 
	
	The group $\Aut(\cZ)$ is trivial if $v\leq 1$, and if $v=2$ it is generated by the involution interchanging the fibers. The latter is realized by some $\alpha\in \Aut(Z,D_Z)$: to see this, choose the coordinates above so that $\ll_{2}=\{y=z\}$ and note that an automorphism of $(\P^2,\pp)$ given by $[x:y:z]\mapsto [\zeta x:y:y-z]$ with $\zeta^{d}=(-1)^{d-e}$ maps $(\ll_1,\ll_2)$ to $(\ll_2,\ll_1)$ and acts trivially on $G^{\circ}$, hence lifts to the required $\alpha$. Therefore, Lemma \ref{lem:outer}\ref{item:outer-transitive} implies that $\cP_{+}(\cZ)$ is represented by an $\Aut(\cZ)$-faithful family over a point, in particular $\#\cP_{+}(\cZ)=1$.

	 If $\cha\kk\nmid d$ then a direct computation shows that the derivative of the above $\G_{m}$ action on $G^{\circ}$ is nonzero, hence by Lemma \ref{lem:outer}\ref{item:outer-h1} we have $h^{1}(\lts{Z}{D_Z})=h^{1}(\lts{Y}{D_Y})=0$, as needed. If  $v\leq 1$ we have another $\G_{m}$-action on $(\P^2,\pp)$ which lifts to a transitive action on $G^{\circ}$, namely $[x:y:z]\mapsto [x:y:\lambda z]$, $\lambda \in \kk^{*}$. Its derivative is nonzero if $\cha\kk\nmid e$. Since $\gcd(d,e)=1$, we conclude as before that  $h^{1}(\lts{Z}{D_Z})=0$. 
	
	Thus we have shown that $\#\cP_{+}(\cZ)=1$ and $h^{1}(\lts{Z}{D_Z})=0$ if $\cha\kk\nmid d$ or $v\leq 1$. The same holds if $h=1$: indeed, after removing $H_2$ the numbers $\#\cP_{+}(\cZ)$ and $h^{1}(\lts{Z}{D_Z})$ can only decrease, see Lemma \ref{lem:h1}\ref{item:h1_positive_curve} for the latter. It remains to show that in case $\cha\kk |d$, $v=2$, see Table \ref{table:exceptions-to-exceptions}, we have $h^{1}(\lts{Z}{D_Z})=1$. By Lemma \ref{lem:outer}\ref{item:outer-h1-grows} we need to check that every vector field $\xi$ on $Y$ tangent to $D_Y$ vanishes along $G^{\circ}$. 
	
	As before, let $\tau\colon Y\to \F_{m}$ be the contraction of vertical curves which do not meet the image of $H_1$. The image $\bar{\xi}$ of $\xi$ is a vector field on $\F_{m}$ which is tangent to three fibers, hence it is vertical. Choose a local coordinate system $(x,y)$ around the center $p$ of $\tau$ such that $x$ is the restriction of the $\P^1$-fibration of $\F_m$. Since $\bar{\xi}$ is vertical and vanishes at $p$, we can assume that $\xi=x\frac{\d}{\d y}$ or $-y\frac{\d}{\d y}$. After a blowup $(x,y)\mapsto (xy,y)$, vector fields $x\frac{\d}{\d x}$ and $ax\frac{\d}{\d x}-by\frac{\d}{\d y}$, $a,b\in \kk$ lift to $x\frac{\d}{\d y}$ and $(a+b)x\frac{\d}{\d x}-by\frac{\d}{\d y}$. Thus by induction on the number of blowups, we conclude that, in some local coordinates $(x,y)$ around $G\cap \sigma_{*}T$ where $\sigma_{*}T=\{x=0\}$, $G=\{y=0\}$, we have $\xi=x\frac{\d}{\d y}$ or $dx\frac{\d}{\d x}-ey\frac{\d}{\d y}=-ey\frac{\d}{\d y}$. In any case, $\xi$ vanishes along $G$, as needed.

	\litem{\ref{lem:ht=1_reduction}\ref{item:uniq_n=2}} Assume $h=2$. Then $H_2=[2-m]$. Indeed, denoting by $\phi\colon Z\to \F_{m}$ the contraction of all vertical curves which are disjoint from $H_1$, we see that $\phi(H_1)$ is the negative section and $\phi(H_2)\cap \phi(H_1)=\emptyset$, so $\phi(H_2)=[-m]$. Thus $H_{2}=[2-m]$, as needed.
	
	We blow up $m-2\geq 0$ times at some point of $H_2\cap V_Z$ and its infinitely near points on the proper transforms of $H_2$, so that the latter becomes a $0$-curve, and perform elementary transformations on it until the total transform of $D_Z$ contains a chain $[0,0]$. Via further elementary transformations, we \enquote{move $[0,0]$ along $D_Z$} to get an inner birational map $(Z,D_Z)\map (Z',D')$ such that a $(-1)$-tip of $D'$ meets a $0$-curve, cf.\ \cite[Examples 2.11]{FKZ-weighted-graphs}. We get $D'=L_1+L_2+R$, where  $R=[\uline{0}]+[0,2,T_{1}]*[(2)_{m-1}]*[T_2,\uline{2},T_{2}^{*},m,T_1^{*}]$ is circular, $L_j=[1]$, $L_j\cdot R=1$, $j=1,2$. The underlined numbers refer to the components $G_{1}'=[0]$, $G_2'=[2]$ meeting $L_1$ and $L_2$, respectively (in the degenerate case when some $V_j$ is a chain, the corresponding chain $T_j$ is empty). Let $U\neq G_1'\subseteq R$ be the other $0$-curve, and let $H$ be the component of $D'-G_{1}'$ meeting it. Then $|U|$ induces a $\P^1$-fibration of $Z'$, whose degenerate fibers are supported on $D'-(G_{1}'+U+H)$ and, by Lemma \ref{lem:fibrations-Sigma-chi}, on $L_1'+W=[1,1]$ for a unique $(-1)$-curve $W$ such that $W\cdot D'=W\cdot H=1$. Let $\cZ'$ and $\cW$ be the combinatorial types of $(Z',D'_Z)$ and $(Z',D_{Z}'+W)$. By Lemma \ref{lem:inner} we have $\#\cP_{+}(\cZ)=\#\cP_{+}(\cZ')$. Since on each log surface $(Z',D_{Z}')\in \cP(\cZ')$ the curve $W$ is unique, we have $\#\cP_{+}(\cZ')=\#\cP_{+}(\cW)$. Similarly, $h^{1}(\lts{Z}{D_Z})=h^{1}(\lts{Z'}{D'})=h^{1}(\lts{Z'}{(D'+W)})$ by Lemmas \ref{lem:blowup-hi}\ref{item:blowup-hi-inner} and \ref{lem:h1}\ref{item:h1-1_curve}. Contraction of $W$ is an inner  blowdown onto a log surface as in case~\ref{lem:ht=1_reduction}\ref{item:uniq_n=3} above, so by Lemma \ref{lem:inner}  we get $\#\cP_{+}(\cW)=1$, and by Lemma \ref{lem:blowup-hi}\ref{item:blowup-hi-inner} we get $h^{1}(\lts{Z'}{(D'+W)})=0$ if $\cha\kk \nmid d\de d([T_{1}^{*},m,T_{2}^{*}])$ and $1$ otherwise. This ends the proof in case $h=2$. 
	
	Assume $h=1$. From case $h=2$ we conclude that $\#\cP_{+}(\cZ)=1$ and, using Lemma \ref{lem:h1}\ref{item:h1_positive_curve}, $h^{1}(\lts{Z}{D_Z})=0$ if $\cha\kk \nmid d$, and $h^{1}(\lts{Z}{D_Z})\leq 1$ if $\cha\kk | d$. It remains to check that the latter inequality is an equality. Let $(Z,D_{Z})\to (Y,D_Y)$ be the contraction of a $(-1)$-tip of the first fiber, let $G\subseteq D_Y$ be the new $(-1)$-curve. By Lemma \ref{lem:blowup-hi}\ref{item:blowup-hi-outer} we have $h^{1}(\lts{Z}{D_Z})\geq h^{1}(\lts{Y}{D_Y})$, with strict inequality if every vector field on $Y$ tangent to $D_Y$ vanishes along $G$. Thus if $h^{1}(\lts{Y}{D_Y})=1$ then the result follows. Assume $h^{1}(\lts{Y}{D_Y})=0$. Case \ref{lem:ht=1_reduction}\ref{item:uniq_n=3} above gives $\cha\kk\nmid d(T_1)$, so by symmetry $\cha\kk\nmid d(T_2)$. Let $(Y,D_Y)\to (\bar{Y},D_{\bar{Y}})$ be the contraction of the $(-1)$-tip of $D_Y$, and let $(\bar{Y},D_{\bar{Y}})\to (\F_{m},B)$ be the inner morphism contracting both degenerate fibers to $0$-curves. Since $\cha\kk \nmid d(T_1),d(T_2)$, the local computation at the end of case \ref{lem:ht=1_reduction}\ref{item:uniq_n=3} shows that the lift of every vertical vector field on $\F_{m}$ either does not vanish on some $(-1)$-curve in $D_{\bar{Y}}$, hence does not lift to $Z$ (case $-y\frac{\d}{\d y}$); or lifts to one vanishing on $G$ (case $x\frac{\d}{\d y}$). Since every vector field on $\F_m$ tangent to $B$ is a linear combination of a vertical one and one tangent to some $H_2$, cf.\ \cite[Lemma 2.5(2)]{FZ-deformations}, case $h=2$ above shows that all vector fields on $(Y,D_Y)$ vanish along $G$, as needed.
	
	\litem{\ref{lem:ht=1_reduction}\ref{item:uniq_dumb}} We can assume $h=2$, $v=1$. Write $V_Z=L+B+T_0+T_1+T_2$, where $L=[1]$, $\beta_{V_Z}(B)=3$, and for $j\in \{1,2\}$, $T_j$ is the maximal twig of $V_Z$ meeting $H_j$. As in case \ref{lem:ht=1_reduction}\ref{item:uniq_n=3}, after a sequence of elementary transformations on $F_1$ we can assume $H_1=[0]$. Like in case \ref{lem:ht=1_reduction}\ref{item:uniq_n=2}, we \enquote{move} the subchain $F_1+H_1=[0,0]$ to $B+\ltip{T_1}$, that is, we perform elementary transformations $(Z,D_Z)\map (Z', D')$ such that the proper transforms $B'$, $T'$ of $B$, $\ltip{T_1}$ are $0$-curves. Now $|T'|$ induces a $\P^1$-fibration such that $D'\hor=B'+H$ consists of two disjoint $1$-sections; and $D'\vert=T'+F'+R$, where $T',F'$ are fibers, and $R$ is the image of $L+T_0$. By Lemma \ref{lem:fibrations-Sigma-chi}, the fiber $F_R$ containing $R$ has exactly one component off $D'$, say $W$, and $F',F_R$ are the only degenerate fibers. Hence $(F_R)\redd=R+W$ and $W\cdot D'=W\cdot H=1$. We have $R=[1,(2)_{k-1}]$ for some $k\geq 3$, and $R\cp{2}$ meets $B'$, so it has multiplicity one in $F_{R}$. Therefore, $F_{R}=[1,(2)_{k-1},1]$ since $k\geq 3$. 
	
	As in case \ref{lem:ht=1_reduction}\ref{item:uniq_n=2} above, it is enough to study the log surface $(Z',D'+W)$. An inner morphism contracting a subchain $[1,(2)_{k-2}]$ of $F_R$ containing $W$ reduces the proof to case \ref{lem:ht=1_reduction}\ref{item:uniq_n=3} (with $T=0$), settled above.
	
	\litem{\ref{lem:ht=1_reduction}\ref{item:uniq_n=2_v=1}} Let $A_1,A_2$ be the $(-1)$-curves in $V_Z$, and let $\sigma\colon (Z,D_{Z})\to (Y,D_{Y})$ be the contraction of $A_2$. Let $A_1'=\sigma(A_1)$, let $G_2$ be the component of $D_Y$ containing the point $\sigma(A_2)$, and let $G_{2}^{\circ}=G_{2}\setminus (D_Y-G_2)$. Case \ref{lem:ht=1_reduction}\ref{item:uniq_n=3} shows that the combinatorial type $\cY$ of $(Y,D_{Y})$ satisfies $\#\cP_{+}(\cY)=1$, and $h^{1}(\lts{Y}{D_Y})$ equals $0$ if $\cha\kk\neq 2$ and $1$ otherwise. Moreover, in case $\cha\kk=2$ all vector fields on $Y$ tangent to $D_Y$ vanish along $G^{\circ}_2$: to see this, contract the fiber containing $A_1'$ to a $0$-curve and apply the computation from the end of~\ref{lem:ht=1_reduction}\ref{item:uniq_n=3}. 

	Thus $\cP_{+}(\cY)$ is represented by a faithful family over a point, universal if $\cha\kk\neq 2$. Lemma \ref{lem:outer}\ref{item:outer-trivial} shows that to construct a faithful family over $G_{2}^{\circ}$ representing $\cP_{+}(\cZ)$ it is enough to show that the group $\Aut(Y,D_Y)$ acts trivially on $G_{2}^{\circ}\cong \A^{1}_{*}$. Moreover, the resulting family will be universal if $\cha\kk\neq 2$; and in case $\cha\kk=2$ we will get $h^{1}(\lts{Z}{D_Z})=1$. Since $G_{2}\cong \P^1$ and $\Aut(Y,D_Y)$ fixes each of the two points $G_{2}\setminus G_{2}^{\circ}$ it is enough to show that  $\Aut(Y,D_Y)$ has a fixed point on $G_{2}^{\circ}$.
	
	Let $\upsilon\colon (Y,D_Y)\to (Z',D_{Z}')$ be the contraction of $\sigma(A_1)$, let $V_{i}\subseteq D'_{Z}$ be the vertical $(-1)$-curve and let $V_{i}^{\circ}=V_{i}\setminus (D_{Z}'-V_i)$. Let $\phi\colon Z'\to \P^2$ be the contraction of the images of $H_1$ and all vertical curves which are disjoint from it. Let $p,p_1,p_2$ be the images of $H_1,A_1,A_2$, let $\ll_{12}$ be the line joining $p_1$ with $p_2$, and for $i\in \{1,2\}$ let $\ll_i$ be the line joining $p_i$ with $p$, so $\ll_1+\ll_2$ is the image of $V_Z$. Let $\{\cc_{z}\}_{z\in \P^1}$ be the pencil of conics tangent to $\ll_i$ at $p_1,p_2$, with degenerate members $\cc_0\de \ll_1+\ll_2$ and $\cc_{\infty}=2\ll_{12}$. Write $\{q^{z}_{i}\}=\phi^{-1}_{*}\cc_{z}\cap V_{i}$, so $\A^{1}_{*}\ni z\mapsto q^{z}_{i}\in V_{i}^{\circ}$ is a parametrization of $V_{i}^{\circ}$. Say that $q_1^{1}=\upsilon(\sigma(A_1))$. Then $(\phi\circ\upsilon)^{-1}_{*}\cc_1$ is fixed by $\Aut(Y,D_Y)$; so its common point with $G_2^{\circ}$, namely, $\upsilon^{-1}(q_{1}^{2})$, is fixed by $\Aut(Y,D_Y)$, as needed.
		
	It remains to show that the above faithful family is $\Aut(\cZ)$-faithful. The group $\Aut(\cZ)\cong \Z/2$ interchanges the degenerate fibers, see  Figure \ref{fig:uniq_n=2_v=1}. We need to endow our family with a $\Z/2$-action such that the fibers $(Z_{b},D_{Z,b})$ obtained from $(Y,D_Y)$ by blowing up points $b\in G_{2}^{\circ}$ which lie in the same orbit are isomorphic.

	It is easy to see that $(Z',D_{Z}')$ admits a $\Z/2$-action interchanging the degenerate fibers; and a $\G_{m}$-action which is transitive on $V_{1}^{\circ}$: indeed, otherwise we would have $\#\cP_{+}(\cY)>1$ by Lemma \ref{lem:outer}, which is false. Write the $\G_{m}$-action as $\lambda\cdot q_{z}^{1}=q_{\lambda z}^{1}$. Since $\Aut(Z',D_{Z}')$ fixes the pencil $\{\phi^{-1}_{*}\cc_{z}\}$, we have $\lambda\cdot q_{z}^{2}=q_{\lambda z}^{2}$, too. 
	 
	 Parametrize $G_{2}^{\circ}\cong V_{2}^{\circ}$ in such a way that $(Z_{b},D_{Z,b})$ is obtained from $(Z',D_{Z}')$ by blowing up $q_{1}^{1}$ and $q_{b}^{2}$. Applying the above $\Z/2$- and $\G_{m}$-actions, we see that $(Z_{b},D_{Z,b})$ is isomorphic to the log surface obtained by blowing up $q_{b}^1$ and $q_1^2$; and to the one obtained by blowing up $q_{1}^{1}$ and $q_{1/b}^{2}$. Thus $(Z_{b},D_{Z,b})\cong  (Z_{1/b},D_{Z,1/b})$, and the required $\Z/2$-action lifts the automorphism of $Z'$ mapping $(q_{1}^{1},q^{2}_{b})$ to $(q^{2}_{1},q^{1}_{1/b})$. In particular, the symmetry group of our family is $\Z/2$, acting on the base $G_{2}^{\circ}\cong \A^{1}_{*}$ is by $b\mapsto \frac{1}{b}$.

	\litem{\ref{lem:ht=1_reduction}\ref{item:uniq_fork-lc}}
		Let $F_{1,Z},F_{2,Z}$ be the fibers in $D_Z$, where $(F_2)\redd=\langle 2;[2],[2],[1]\rangle$, see Figure \ref{fig:uniq_fork-lc}. Let $T_{i,Z}$ be the $(-2)$-tip of multiplicity $1$ of $F_{i,Z}$ which does not meet $H_1$, and let $U_{i,Z}$ be the $(-1)$-curve in $F_{i,Z}$. Let $\sigma\colon (Z,D_Z)\to (Y,D_Y)$ be the contraction of $U_{2,Z}$, let $F_{i}=\sigma_{*}F_{i,Z}$, $T_{i}=\sigma(T_{i,Z})$, $U_1=\sigma_{*}(U_1,Z)$; let $G$ be the $(-1)$-curve in $F_2$ and let $G^{\circ}=G\setminus (D_Y-G)\cong \A^1_{*}$. 
		
		The contraction of $F_2-T_2$ is an inner morphism from $(Y,D_Y)$ onto a log surface $(Y',D_{Y}')$ as in case \ref{lem:ht=1_reduction}\ref{item:uniq_dumb}. We have shown that $h^{1}(\lts{Y'}{D_{Y}'})=0$ and the combinatorial type $\cY'$ of $(Y',D_{Y}')$ has $\#\cP_{+}(\cY')=1$, so $\cP_{+}(\cY')$ is represented by a universal faithful family over a point. By Lemma \ref{lem:inner}, the same is true for $(Y,D_Y)$. Note that the group $\Aut(\cZ)$ is trivial, see Figure \ref{fig:uniq_fork-lc}. Thus to conclude that   \ref{lem:ht=1_uniqueness}\ref{item:ht=1_uniqueness-infinite} holds in this case it is enough, by Lemma \ref{lem:outer}\ref{item:outer-trivial}, to prove that the action of $\Aut(Y,D_Y)$ on $G^{\circ}\cong \A^{1}_{*}$ is trivial. 
		
		Let $A$ be the restriction of $\Aut(Y,D_Y)$ to $G^{\circ}$, so $A\leq \G_m$.  The action of $A$ on $T_1\setminus (D_Y-T_1)\cong \A^1$ has a fixed point, say $p_1$. Let $\tau\colon Y\to \P^1\times \P^1$ be the contraction of $(D_Y)\vert-T_1-T_2$. Let $H'$ be the horizontal line through $\tau(p_1)$. The action of $A$ on $\P^1\times \P^1$ fixes $H'$, so $A\leq \Aut(Y,D_Y+H)$, where $H=\tau^{-1}_{*}H'=[0]$. 
		
		Now, we argue like in case \ref{lem:ht=1_reduction}\ref{item:uniq_dumb}. Let $B$ be the branching component of $(F_1)\redd-U_1$. Perform two elementary transformations on $H$ so that the proper transform of $T_1$, say $\hat{T}_1$, becomes a $0$-curve, and denote this map by $\phi\colon (Y,D_{Y})\map (\hat{Y},\hat{D})$. Consider a $\P^1$-fibration induced by $|\hat{T}_{1}|$. Then $\hat{D}\hor$ consists of two $1$-sections, namely $\phi_{*}B$ and a $0$-curve, say $E$. The divisor $\hat{D}\vert$ has three connected components, namely: a smooth fiber $\hat{T}_{1}$, a chain $C'=[2,2,2,1,4]$ supporting a degenerate fiber, and a chain $C=[1,2,2]\subseteq \phi_{*}F_1$, such that $\phi^{-1}$ is an isomorphism near $C$. Lemma \ref{lem:fibrations-Sigma-chi} implies that the unique degenerate fiber not contained in $\phi_{*}D_Y$ is the one containing $C$, which is a chain $C+L=[1,2,2,1]$. We have $L\cdot \hat{D}=2$, and $L$ meets $E$ and the $(-2)$-tip of $\hat{D}$.
		
		Perform elementary transformations on $\hat{T}_{1}$ in such a way that the image of $B$ becomes a $(+1)$-curve, and contract all components of the image of $\hat{D}+L$ which are disjoint from it. Denote the resulting map by $\psi\colon (Y,D_Y)\map (\P^2,\pp)$. Then $\pp$ is the sum of four lines: the image of $B$, call it $\ll$, and three others, call them $\ll_1,\ll_2,\ll_3$, which meet at one point away from $\ll$: these contain the images of $C$, $C'$ and $\hat{T}_{1}$, respectively. The image of $U_1$ is a point $q\in \ll_1\setminus (\pp-\ll_1)$. We have an injective homomorphism $A\into \Aut(\P^2,\pp,q)$ whose image fixes each component of $\pp$, so $A=\{\id\}$, as needed.

	\litem{\ref{lem:ht=1_reduction}\ref{item:uniq_A1E7}} We claim that $\#\cP_{+}(\cZ)=2$. Let $T_{1}$, $T_2$ be the $(-1)$- and $(-2)$-tip of $D_Z$, and let $\phi\colon Z\to Z'$ be the contraction of $D_Z-T_1-T_2-F_1=[1,2,2,2,2]$, see Figure \ref{fig:uniq_A1E7}. Then $\rho(Z')=2$ and $\phi(T_2)$ is a $(+1)$-curve, so $Z'\cong \F_{1}$ and $\phi(T_2)$ is disjoint from the negative section $\Sec_1=[1]$. We have $\phi(T_1)=[0]$, so $\phi(T_1)$ is a fiber, and $\phi(F_1)$ is a $(+5)$-curve meeting $\phi(T_1)$ and $\phi(T_2)$ only at their unique common point, with multiplicities $1$ and $3$, respectively; so $\phi(F_1)\equiv 2\phi(T_1)+\phi(T_2)$. Thus $\Sec_1\cdot \phi(F_1)=2$, $\Sec_1\cdot \phi(T_1)=1$, and $\Sec_1\cdot \phi(T_2)=0$. Let $\tau\colon (Z,D_Z)\to (\P^2,\pp)$ be a composition of $\phi$ with the contraction of $\Sec_1$. Then $\tau(F_1)$ is a rational cubic, $\tau(T_{2})$ is its inflectional tangent line, and $\tau(T_1)$ is a line joining $\tau(F_1)\cap \tau(T_2)$ with $\Sing \tau(F_1)$. Up to an isomorphism there are exactly two such pairs $(\P^2,\pp)$, one with $\tau(F_1)$ nodal and the other with $\tau(F_1)$ cuspidal, see \cite[Lemma 5.5]{PaPe_MT}. Since each blowup in a decomposition of $\tau$ is uniquely determined by $\pp$, by the universal property of blowing up we get exactly two log surfaces $(Z,D_Z)$, cf.\  Lemma \ref{lem:inner} and its proof. Since $\Aut(\cZ)$ is trivial, this means that $\#\cP_{+}(\cZ)=2$, as needed.
	
	We claim that $\cP_{+}(\cZ)$ is represented by a stratified family. Let $(Z,D_Z)\toin{\sigma_1} (Y_{1},D_{1})\toin{\sigma_2} (Y_2,D_2)$  be the contractions of the $(-1)$-tips of $D_Z$ and $D_1$; let $G_i$ be the $(-1)$-tip of $D_i$, let $G^{\circ}_{i}=G_{i}\setminus (D_{i}-G_i)\cong \A^1$, and let $\cY_{i}$ be the combinatorial type of $(Y_i,D_i)$. Now $\cY_{2}$ is as in case~\ref{lem:ht=1_reduction}\ref{item:uniq_n=3} above, so  $\#\cP_{+}(\cY_2)=1$.
	 
	Since $\#\cP_{+}(\cZ)<\infty$, Lemma \ref{lem:outer}\ref{item:outer-trivial} implies that for both $i\in \{1,2\}$, the action of $\Aut(Y_i,D_i)$ on $G_{i}^{\circ}$ has an open orbit; so applying Lemma \ref{lem:outer}\ref{item:outer-open} twice we get that $\cZ$ is represented by a stratified family, as needed. 
	
	We remark that, since $\#\cP_{+}(\cZ)=2=\#\cP_{+}(\cY_2)+1$, Lemma \ref{lem:outer}\ref{item:outer-open} implies that the action of $\Aut(Y_{i},D_{i})$ on $G_{i}^{\circ}$ is transitive for one $i\in \{1,2\}$, and has a fixed point for the other. The proof in case \ref{lem:ht=1_reduction}\ref{item:uniq_-2-chain} below will show that this action is transitive for $i=2$ if $\cha\kk\neq 2$, and for $i=1$ if $\cha\kk=2$.
	
	To conclude that this family is $h^{1}$-stratified, by Lemma \ref{lem:outer}\ref{item:outer-h1-stratified} it is enough to show that for $(Z,D_Z)$ lying over the open stratum, i.e.\ in case when $\tau(F_1)$ is nodal, we have $h^1\de h^{1}(\lts{Z}{D_Z})=0$. Let $E=[1]$ be the proper transform of the negative section of $\F_1$, and let $L\subseteq Z$ be the proper transform of a line tangent to $\Sing \tau(F_1)$. Since $\tau(F_1)$ is nodal, the sum $D_{Z}+E$ is snc, so Lemma \ref{lem:h1}\ref{item:h1-1_curve} gives  $h^{1}=h^{1}(\lts{Z}{(D_Z+E)})$. Let $\upsilon\colon W\to Z$ be a blowup at $E\cap L$, and let $D_W=(\upsilon^{*}D_{Z})\redd$, $E_{W}=\upsilon^{-1}_{*}E=[2]$, $L_{W}=\upsilon^{-1}_{*}L=[1]$.  By Lemmas \ref{lem:blowup-hi}\ref{item:blowup-hi-inner} and \ref{lem:h1}\ref{item:h1-1_curve} we have $h^{1}=h^{1}(\lts{W}{(D_W+E_W)})=h^{1}(\lts{W}{(D_W+E_W+L_W)}$. Contracting non-branching $(-1)$-curves in $D_{W}+E_W+L_{W}$ and its images yields an inner morphism $(W,D_{W}+L_W)\to (\P^2,\pp)$, where $\pp$ is a sum of four general lines, so $h^{1}=h^{1}(\lts{\P^2}{\pp})=0$ by Lemma \ref{lem:h1}\ref{item:h1-P2}, as claimed.
	\litem{\ref{lem:ht=1_reduction}\ref{item:uniq_2}} 
	Let $\sigma\colon (Z,D_Z)\to (Y,D_Y)$ be the contraction of the $(-1)$-tip of $D_Z$, let $G$ be the $(-1)$-tip of $D_Y$,  let $G^{\circ}=G\setminus (D_Y-G)\cong \A^{1}$, and let $\cY$ be the combinatorial type of $(Y,D_Y)$. Case \ref{lem:ht=1_reduction}\ref{item:uniq_n=3}, $v=1$ shows that $\cP_{+}(\cY)$ is represented by an $h^{1}$-stratified family over a point. By Lemma \ref{lem:outer}\ref{item:outer-fixed-point},\ref{item:outer-h1}, 
	it is enough to show that the action of $\Aut(Y,D_Y)$ on $G^{\circ}\cong \A^{1}$ has two orbits, and the derivative of this action is nonzero.
	
	We argue like in \ref{lem:ht=1_reduction}\ref{item:uniq_n=2} above. We assume that $V_Z$ is a fork, the degenerate case $V_Z=[1,2,1]$ is analogous. 
	
	Say that $H_1$ meets $T^{*}$. Note that $H_2=[1-m]$: indeed, after the contraction of all vertical curves disjoint from $H_1$, the image of $H_2$ is a member of $|\sigma_{m}|$. Replacing $H_1$ with $H_2$ if needed, we can assume $m\geq 1$. 
	
	Let $\phi_{1}\colon Y'\to Y$ be the composition of blowups over $\sigma(H_2)\cap \sigma(F_1)$ such that 
	\begin{equation*}
		(\phi_{1}^{*}D_Y)\redd=[1,2]+[T,\uline{1},T^{*},\bs{1},T]*[(2)_{m-1},\uline{1},\uline{m},T^{*}].
	\end{equation*}
	Here, the first chain is the proper transform of $G+B$, where $\beta_{D_Y}(B)=3$; the second chain meets $(\phi_{1})^{-1}_{*}B$; and the subsequent underlined components are the proper transforms of $H_2$, $F_1$ and $H_{1}$, see Figure \ref{fig:phi}.
	
	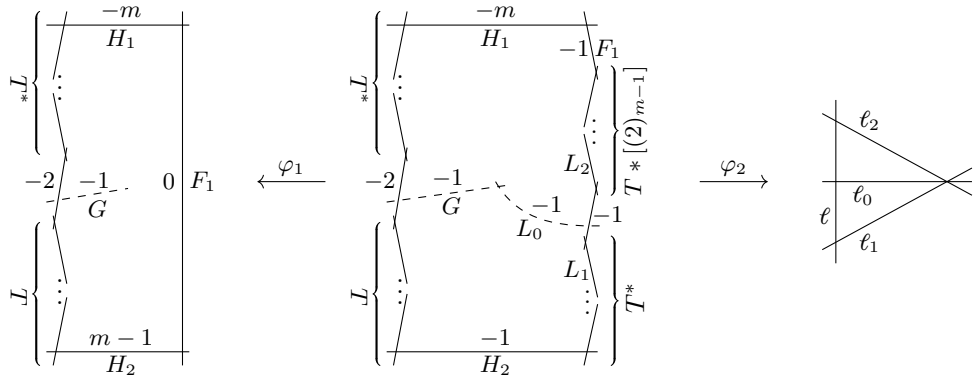
\begin{figure}[ht]
		\begin{tikzpicture}[scale=0.9]
			\begin{scope}
				\draw (-0.1,0) -- (2,0);
				\node at (1,0.2) {\small{$m-1$}};
				\node at (1,-0.2) {\small{$H_2$}};
				\draw (0,-0.2) -- (0.2,0.8);
				\node at (0.1,1) {$\vdots$};
				\draw (0.2,1) -- (0,2);
				\draw [decorate, decoration = {calligraphic brace}, thick] (-0.2,-0.2) --  (-0.2,1.9);
				\node[rotate=-90] at (-0.5,0.8) {\small{$T$}}; 
				\draw (0,1.8) -- (0.2,3);
				\node at (-0.2,2.5) {\small{$-2$}};
				\draw[dashed] (-0.1,2.2) -- (1.1,2.4);
				\node at (0.6,2.5) {\small{$-1$}};
				\node at (0.65,2.1) {\small{$G$}};
				\draw (0.2,2.8) -- (0,3.8);
				\node at (0.1,4) {$\vdots$};
				\draw (0,4) -- (0.2,5);
				\draw [decorate, decoration = {calligraphic brace}, thick] (-0.2,2.9) --  (-0.2,5);
				\node[rotate=-90] at (-0.5,3.85) {\small{$T^{*}$}};
				\draw (-0.1,4.8) -- (2,4.8);
				\node at (1,5) {\small{$-m$}};
				\node at (1,4.6) {\small{$H_1$}};
				\draw (1.9,-0.2) -- (1.9,5);
				\node at (1.7,2.5) {\small{$0$}};
				\node at (2.2,2.5) {\small{$F_1$}};
				\draw[<-] (3,2.5) -- (4,2.5);
				\node at (3.5,2.7) {\small{$\phi_1$}};
			\end{scope}
			\begin{scope}[shift={(5,0)}]
				\draw (-0.1,0) -- (3,0);
				\node at (1.5,0.2) {\small{$-1$}};
				\node at (1.5,-0.2) {\small{$H_2$}};
				\draw (0,-0.2) -- (0.2,0.85);
				\node at (0.1,1) {$\vdots$};
				\draw (0.2,1) -- (0,2);
				\draw [decorate, decoration = {calligraphic brace}, thick] (-0.2,-0.2) --  (-0.2,1.9);
				\node[rotate=-90] at (-0.5,0.8) {\small{$T$}}; 
				\draw (0,1.8) -- (0.2,3);
				\node at (-0.2,2.5) {\small{$-2$}};
				\draw[dashed] (-0.1,2.2) -- (1.7,2.45);
				\node at (0.8,2.55) {\small{$-1$}};
				\node at (0.85,2.15) {\small{$G$}};
				\draw (0.2,2.8) -- (0,3.8);
				\node at (0.1,4) {$\vdots$};
				\draw (0,4) -- (0.2,5);
				\draw [decorate, decoration = {calligraphic brace}, thick] (-0.2,2.9) --  (-0.2,5);
				\node[rotate=-90] at (-0.5,3.85) {\small{$T^{*}$}};
				\draw (-0.1,4.8) -- (3,4.8);
				\node at (1.5,5) {\small{$-m$}};
				\node at (1.5,4.6) {\small{$H_1$}};
			\begin{scope}[shift={(1,0)}]	
				\draw (1.8,-0.2) -- (2,0.7);
				\node at (1.85,0.85) {\small{$\vdots$}};
				\draw (2,0.8) -- (1.8,1.7);
				\draw [decorate, decoration = {calligraphic brace}, thick] (2.2,1.7) --  (2.2,-0.2);
				\node[rotate=90] at (2.55,0.8) {$T^{*}$};
				\node at (1.7,1.2) {\small{$L_1$}};
				\draw (1.8,1.5) -- (2,2.5);
				\node at (2.15,2) {\small{$-1$}};
				\draw (2,2.3) -- (1.8,3.2);
				\node at (1.7,2.75) {\small{$L_2$}};
				\node at (1.95,3.35) {\small{$\vdots$}};
				\draw (1.8,3.3) -- (2,4.2);
				\draw [decorate, decoration = {calligraphic brace}, thick] (2.2,4.2) --  (2.2,2.3);
				\node[rotate=90] at (2.55,3.25) {\small{$T*[(2)_{m-1}]$}};
				\draw (2,4) -- (1.8,5);
				\node at (1.65,4.45) {\small{$-1$}};
				\node at (2.15,4.45) {\small{$F_1$}};
				\draw[dashed] (0.5,2.5) to[out=-60,in=180] (1.9,1.85)-- (2.1,1.85);
				\node at (1.25,2.15) {\small{$-1$}};
				\node at (1,1.8) {\small{$L_0$}};
				\draw[->] (3.5,2.5) -- (4.5,2.5);
				\node at (4,2.7) {\small{$\phi_2$}};
			\end{scope}
			\end{scope}
			\begin{scope}[shift={(11.5,2.5)}]
				\draw (0,-1.2) -- (0,1.2);
				\node at (-0.15,-0.5) {\small{$\ll$}};
				\draw (-0.2,-1) -- (2,0.2);
				\node at (0.5,0.9) {\small{$\ll_2$}};
				\draw (-0.2,1) -- (2,-0.2);
				\node at (0.4,-0.2) {\small{$\ll_0$}};
				\draw (-0.2,0) -- (2,0);
				\node at (0.5,-0.9) {\small{$\ll_1$}};
			\end{scope}
		\end{tikzpicture}\vspace{-1em}
		\caption{The map $\phi\colon Y\map \P^2$ constructed in case \ref{lem:ht=1_reduction}\ref{item:uniq_2}.}
		\label{fig:phi}
	\end{figure}

	Let  $L_1$, $L_2$ be the components of $T^{*}$ and $T$ meeting the exceptional $(-1)$-curve of $\phi_{1}$, denoted above by \enquote{$\bs{1}$}. Let $\phi_2\colon Y'\to \P^2$ be the contraction of $(\phi_1^{*}D_Y)\redd-(\phi_{1}^{-1})_{*}B-L_1-L_2$. Put $\phi\de\phi_2\circ \phi_1^{-1}$. Then $\ll_{j}\de \phi_{2}(L_{j})$, $j\in \{1,2\}$ and $\ll\de \phi_{*}B$ are non-concurrent lines. Write $\{q\}=\ll_1\cap \ll_2$, $\{q_j\}=\ll_j\cap \ll$, $\{q_0\}=\phi(G)\subseteq \ll\setminus \{q_1,q_2\}$, let $\ll_0$ be the line joining $q_0$ with $q$; and $\pp=\ll+\sum_{j=0}^{2}\ll_{j}$. Fix coordinates on $\P^2$ so that $\ll=\{x=0\}$, $\ll_1=\{y=0\}$, $\ll_2=\{z=0\}$, $\ll_0=\{y=z\}$. Now $\Aut(Y,D_Y)=\Aut(\P^2,\pp)=\{[x:y:z]\mapsto [\lambda x:y:z]: \lambda\in \kk^{*}\}\cong \G_{m}$ has two orbits on $G^{\circ}$, namely $\phi^{-1}_{*}\ll_0\cap G^{\circ}$ and the open one, and the vector field $x\frac{\d}{\d x}$ on $\P^2$ lifts to a vector field on $(Y,D_Y)$ which does not vanish along that open orbit, as needed.
	
	\litem{\ref{lem:ht=1_reduction}\ref{item:uniq_-2-chain}} As in case \ref{lem:ht=1_reduction}\ref{item:uniq_2} above, let $\sigma\colon (Z,D_Z)\to (Y,D_Y)$ be the contraction of the $(-1)$-tip of $V_Z$, let $G\subseteq D_Y$ be the $(-1)$-tip of $V_Y\de \sigma_{*}V_Z$, and let $G^{\circ}=G\setminus (D_Y-G)\cong \A^{1}$. The combinatorial type $\cY$ of $(Y,D_Y)$ is as in case~\ref{lem:ht=1_reduction}\ref{item:uniq_n=3}, so $\#\cP_{+}(\cY)=1$, and in case $v\leq 1$ or $\cha\kk\nmid d$ we have $h^{1}(\lts{Y}{D_Y})=0$. 
	
	Let $T=[d]$ be the maximal twig of $V_Y$ which does not meet $H_Y\de \sigma_{*}H_1$. Put $F=\sigma_{*}F_1$ if $v\geq 1$, otherwise let $F\subseteq Y$ be some non-degenerate fiber. Fix a section $H\subseteq Y$ whose image under the contraction $Y\to \F_{m-1}$ of $V_Y-T$ is disjoint from the negative section. Then $H\cdot (V+H_Y)=H\cdot T=1$ and $H=[1-m]$.
	\begin{casesp*}
	\litem{$\cha\kk\nmid d$} By Lemma \ref{lem:outer}\ref{item:outer-transitive},\ref{item:outer-h1-stays} it is enough to check that $\Aut(Y,D_Y)$ acts transitively on $G^{\circ}$, and the derivative of this action is nonzero.
	
	For this we can assume $v=2$. If $m=1$ blow up at $F_1\cap H_{Y}$, otherwise blow up $m-2$ times at $H\cap F_1$ and its infinitely near points on the proper transforms of $H$. Denote this morphism by $\tau_1\colon Y''\to Y$. Let $U$ be the exceptional $(-1)$-curve of $\tau_{1}$; and let $\tau_2\colon Y''\to \P^2$ be the contraction of $(\tau_1^{*}D_Y)\redd-(\tau_1)^{-1}_{*}(T+F_2)-U$: if $m=1$, this divisor is a sum of disjoint chains $[1]+[(2)_{d+1},1]$; otherwise it is a chain $[(2)_{m-3},1,m,(2)_{d},1]$. Put $\tau=\tau_{2}\circ \tau_{1}^{-1}$, $\ll_j=\tau_{*}F_j$, $\ll_{T}=\tau_{*}T$, $\pp=\tau_{*}D_Y=\ll_1+\ll_2+\ll_{T}$. Then $\ll_1,\ll_2,\ll_{T}$ are lines meeting at a point $q$. Choose coordinates on $\P^2$ so that $\ll_1=\{x=0\}$, $\ll_{T}=\{y=0\}$ and put $\cc=\{x^{d+1}=yz^{d}\}$. Then $q=[0:0:1]$ is  smooth point of $\cc$, and we have $(\cc\cdot \ll_{T})_{q}=d+1$. Let $B\subseteq D_Y$ be the $(-2)$-curve meeting $G$. Write $\tau=\upsilon\circ \upsilon'$, where $\upsilon'$ is the contraction of $G$. Then $\upsilon^{-1}_{*}\cc$ meets $(\upsilon^{*}\ll_{T})\redd$ only in $\upsilon'(B)$, once. Since $\#\cP_{+}(\cY)=1$, we can assume that the point $\upsilon'(G)$ lies on $\upsilon^{-1}_{*}\cc$, so $\tau^{-1}_{*}\cc$ meets $G^{\circ}$. 
	
	For $\lambda\in \kk$ define $\alpha_{\lambda}\in \Aut(\P^2,\pp)$ by $\alpha_{\lambda}[x:y:z]=[x:y:z+\lambda x]$. For $\lambda\neq 0$ we have $(\alpha_{\lambda}(\cc)\cdot \cc)_{q}=d+2$. To see this, parametrize $\alpha_{\lambda}(\cc)$ near $q$ by $t\mapsto [t:t^{d+1}:1+\lambda t]$. Substituting this parametrization to the equation $x^{d+1}-yz^{d}$ of $\cc$ we get
	\begin{equation*}
		t^{d+1}-t^{d+1}(1+\lambda t)^{d}=d\lambda t^{d+2}+\hot,
	\end{equation*}  
	where \enquote{$\hot$} stands for \enquote{higher order terms}. Since $\cha\kk\nmid d$, the lowest exponent is $d+2$, as claimed. 
	
	In particular,  $(\alpha_{\lambda}(\cc)\cdot \cc)_{q}>d+1=(\ll_{T}\cdot \cc)_{q}$. Hence by our construction $\alpha_{\lambda}$ lifts to an element of $\Aut(Y,D_Y)$. In turn, since $(\alpha_{\lambda}(\cc)\cdot \cc)_{q}=d+2$, this lift does not fix $\tau^{-1}_{*}\cc \cap G^{\circ}$. We conclude that $\{\alpha_{\lambda}:\lambda\in \kk\}\cong \G_{a}$ lifts to a subgroup of $\Aut(Y,D_Y)$ acting transitively on $G^{\circ}\cong \A^{1}$, as needed. 
	
	A direct computation shows that the derivative of this action is nonzero. Indeed, in local coordinates $(\check{x},\check{y})=(\frac{x}{z},\frac{y}{z})$, the $\G_{a}$-action is given by $\lambda\cdot (\frac{\check{x}}{1+\lambda \check{x}},\frac{\check{y}}{1+\lambda\check{x}})$, so its derivative at $\lambda=0$ maps the unit vector field $\frac{\d}{\d \lambda}$ to $\xi\de -\check{x}^{2}\frac{\d}{\d\check{x}}-\check{x}\check{y}\frac{\d}{\d \check{y}}$. At a point of $G^{\circ}$ choose local coordinates $(x,y)$ on $Y$, with $G^{\circ}=\{x=0\}$, and write $\tau\colon Y\to \P^2$ as a composition of an outer blowup $(x,y)\mapsto (xy+1)$ and $d+1$ blowups $(x,y)\mapsto (x,xy)$, so $\tau(x,y)=(x,x^{d+1}(xy+1))$. Multiplying $\xi$ by the inverse of its jacobian matrix, we conclude that $\xi$ lifts to $-x^2\frac{\d}{\d x}+(d+1)xy\frac{\d}{\d y}+d\frac{\d}{\d y}$, which does not vanish along $G^{\circ}=\{x=0\}$ because $\cha\kk\nmid d$.

\litem{$\cha \kk | d$} We claim that the  action of $\Aut(Y,D_Y)$ on $G^{\circ}$ has a fixed point.  

Put $\tilde{D}_Y= V_Y+H_Y+H+F$. Then $(Y,\tilde{D}_Y)$ is the same as the log surface $(Y,D_Y)$ considered in case \ref{lem:ht=1_reduction}\ref{item:uniq_2} above. Let $\phi\colon (Y,\tilde{D}_Y)\map (\P^2,\pp)$ be the map constructed there, see Figure \ref{fig:phi}. The curve $L_0\de \phi^{-1}_{*}\ll_0$ has a unique cusp at $H\cap F$. Its minimal log resolution is  $\phi_{1}\colon (Y',L_{0}'+R)\to (Y,L_0)$, where $R=[(2)_{d-1},1,d+1,(2)_{m-2}]\subseteq D'$. Fix $\alpha\in \Aut(Y,D_Y)$ and put $\bar{L}_0=\alpha(L_0)$. Then $\bar{L}_0$ has a cusp of the same type. The curve  $\bar{L}_0'\de (\phi_{1}^{-1})_{*}\bar{L}_0$ meets $R$ in at most one point, and if it does then either $\bar{L}_0'\cdot R=\bar{L}_0'\cdot R\cp{d}=1$, or $\bar{L}_0'\cdot R=\bar{L}_0'\cdot R\cp{i}=d$ for some $i>d$; where $R\cp{j}$ is the $j$-th component of the chain $R$. In any case, it follows that at the point $\ll_1\cap \ll_2$, the lines $\ll_{1}$ and $\ll_{2}$ meet every branch of $\bar{\ll}_0\de \phi_{*}\bar{L}_0$ with multiplicity divisible by $d$. As in case \ref{lem:ht=1_reduction}\ref{item:uniq_2} above, we choose coordinates on $\P^2$ so that $\ll_{1}=\{y=0\}$, $\ll_{2}=\{z=0\}$, $\ll_0=\{y=z\}$. Then the curve $\bar{\ll}_0$ can be parametrized by $[s:t]\mapsto [f:g_1^d:g_2^d]$ for some $f,g_{1},g_2\in \kk[s,t]$. Since $\cha\kk|d$, we get $\cha \kk|(\ll_0\cdot \bar{\ll}_0)_{r}$ for every $r\in \ll_0\cap \bar{\ll}_0$. In particular, $\cha \kk$ divides the intersection multiplicity of $\ll_0$ and $\bar{\ll}_0$ at the point $q_0=\phi(G)$, so $(\bar{\ll}_0\cdot \ll_0)_{q_{0}}\geq 2$. Hence $\bar{L}_0$ meets $L_0$ at $\{p\}\de L_0\cap G$, which is the required fixed point. 

Consider the case $v\leq 1$. Since $\#\cP_{+}(\cY)=1$ and $h^{1}(\lts{Y}{D_Y})=0$, Lemma \ref{lem:outer} and the above claim imply that $\cP_{+}(\cZ)$ is represented by an $h^1$-stratified or universal family over $G^{\circ}$. On the other hand, for some section $H'\subseteq Z$ the combinatorial type $\cZ'$ of $(Z,D_Z+H')$ is as in case~\ref{lem:ht=1_reduction}\ref{item:uniq_2} above, so we have a surjective restriction map $\cP_{+}(\cZ')\to \cP_{+}(\cZ)$. Thus $\#\cP_{+}(\cZ)\leq \#\cP_{+}(\cZ')=2$, so our family if $h^1$-stratified, as claimed. 

Consider the case $v=2$. We claim that $\cP_{+}(\cZ)$ is represented by an $\Aut(\cZ)$-faithful family over $G^{\circ}$. Since $\#\cP_{+}(\cY)=1$ and $\Aut(\cZ)=\Aut(\cY)$, by Lemma \ref{lem:outer}\ref{item:outer-trivial} it is enough to prove that the action of $\Aut(Y,D_Y)$ on $G^{\circ}\cong \A^{1}$ has finite orbits. Suppose the contrary. Since $\Aut(Y,D_Y)$ has a fixed point on $G^{\circ}$, it has a subgroup $A\cong \G_{m}$ with an open orbit on $G^{\circ}$. Since $\Aut(Y,D_Y)$ fixes three fibers in $D_Y$, it acts trivially on the base, i.e.\ sends every fiber to itself. Thus $A\cong \G_{m}$ has a fixed point on $T\setminus D_{Y}\cong \A^{1}$, and on $F'\setminus H_1\cong \A^1$ for any nondegenerate fiber $F'$. It follows that $A$ fixes some section $H$ as above, so the map $\alpha\mapsto \phi\circ \alpha\circ \phi^{-1}$ defines an embedding $\iota\colon A\into \Aut(\P^{2},\pp)$. In the above coordinates we have $\Aut(\P^2,\pp)=\{[x:y:z]=[x:y:\lambda z]: \lambda \in \kk^{*}\}\cong \G_{m}$. On the other hand, $\iota(A)$ fixes $\phi(F_2)$, which is a curve of degree $d$, satisfying $\phi(F_2)\cap \ll_{j}=\ll\cap \ll_j$ for $j\in \{1,2\}$. Thus $\phi(F_2)=\{zy^{d-1}=c x^{d}\}$ for some $c \in \kk^{*}$. The only element of $\Aut(\P^{2},\pp)$ fixing this curve is the identity, so $A=\{\id\}$, a contradiction. Thus $\cP_{+}(\cZ)$ is represented by an $\Aut(\cZ)$-faithful family over $G^{\circ}$.

The symmetry group of the above family is the image of $\Aut(Y,D_Y)$ in $\Aut(G^{\circ})$. A direct computation shows that $\Aut(Y,D_Y)$ is generated by the automorphisms interchanging the non-degenerate fibers in $D_Y$ constructed in case~\ref{lem:ht=1_reduction}\ref{item:uniq_n=3}, which act on $G^{\circ}\cong \A^1$ as multiplication by $\zeta$, where $\zeta^d=-1$; as needed.

To conclude, we need to show that $h^{1}(\lts{Z}{D_Z})$ equals $2$ for the central fiber of this family, and $1$ for the remaining ones. By Lemma \ref{lem:blowup-hi}\ref{item:blowup-hi-outer}, it is enough to prove that there is a unique point $p\in G^{\circ}$ such that every vector field $\xi\in H^{0}(\lts{Y}{D_Y})$ vanishes at $p$. Let $\tau\colon Y\to \F_{m-1}$ be the contraction of $V_{Y}-T$. It can be written as $\tau(x,y)=((xy+1)y^d,y)$, where $(x,y)$ are local coordinates at some $q\in G^{\circ}$, with $G^{\circ}=\{y=0\}$, and at the target we choose local coordinates $(\check{x},\check{y})$ such that $\check{x}$ is the restriction of the $\P^1$-fibration of $\F_{m-1}$, and $\{\check{y}=0\}$ is the equation of the negative section. The vector field $\tau_{*}\xi$ is tangent to three fibers of $\F_{m-1}$, so it is vertical, and therefore a multiple of $\check{x}\frac{\d}{\d\check{y}}$ or $\check{y}\frac{\d}{\d\check{y}}$. Multiplying these vector fields by the inverse of the jacobian matrix of $\tau$ we see that in the first case $\xi$ vanishes along $G^{\circ}$, and in the second case it vanishes at $q$. Furthermore, if $\tau_{*}\xi=\check{y}\frac{\d}{\d\check{y}}$ then $\xi=-x\frac{\d}{\d x}+y\frac{\d}{\d y}$ does not vanish along $G^{\circ}\setminus \{q\}$, as needed.
	\end{casesp*}
	
	\litem{\ref{lem:ht=1_reduction}\ref{item:uniq_fork}} Let $V_{Z,1}=\langle 2;[(2)_{d-1}],[d],[1,2]\rangle$, $V_{Z,2}=\langle 2;T^{*},T\trp,[1]\rangle$ be the connected components of $V_{Z}$, they both meet $H_1=[m]$ in the first tip of the first twig. Let $\sigma\colon (Z,D_Z)\to (Y,D_Y)$ be the contraction of the $(-1)$-tip of $V_{Z,1}$, let $H=\sigma_{*}H_{1}$, $V_{i}=\sigma_{*}V_{Z,i}$, let $B_i$ be the branching component of $V_i$, let $G$ be the $(-1)$-tip of $V_{1}$ and let $G^{\circ}=G\setminus (D_Y-G)\cong \A^{1}$, see Figure \ref{fig:uniq_fork}. The log surface $(Y,D_Y)$ is as in case~\ref{lem:ht=1_reduction}\ref{item:uniq_n=2}.
	
	Let $\upsilon \colon (Y,D_Y)\to (Y_1,D_1)$ be the contraction of the $(-1)$-tip of $V_2$, and let $\tau\colon (Y_{1},D_{1})\to (Y_{2},D_{2})$ be the contraction of all components of $\upsilon_{*}V_2$ except for the tip of $D_1$. The log surface $(Y_2,D_2)$ is as in case~\ref{lem:ht=1_reduction}\ref{item:uniq_n=3}, $v=1$, see Figure \ref{fig:uniq_n=3}. In case~\ref{lem:ht=1_reduction}\ref{item:uniq_-2-chain} above we have shown that the action of $\Aut(Y_2,D_2)$ on the image of $G^{\circ}$ has a fixed point if $\cha\kk | d$ and is transitive with nonvanishing derivative if $\cha\kk\nmid d$. Since the morphism $\tau$ is inner and $\Bs\tau^{-1}$ is fixed by $\Aut(Y_2,D_2)$, every automorphism of $(Y_2,D_2)$ lifts to an automorphism of $(Y_1,D_1)$. 
		
	Consider the case $\cha\kk\nmid d$. Since $\upsilon(G^{\circ})\cong \A^1$, there is a subgroup  $\G_a\leq \Aut(Y_1,D_1)$ acting transitively with nonvanishing derivative on $\upsilon(G^{\circ})\cong \A^1$. This subgroup fixes the two points $\upsilon(B_2)\cap (D_1-\upsilon(B_2))$, so it fixes $\upsilon(B_2)$ pointwise, so its action lifts to a transitive action on $G^{\circ}$, with nonvanishing derivative. Since $\#\cP_{+}(\cY)=1$, by Lemma \ref{lem:outer}\ref{item:outer-transitive} we have $\#\cP_{+}(\cZ)=1$ and $h^{1}(\lts{Z}{D_Z})=h^{1}(\lts{Y}{D_Y})$, as needed.
	
	Consider the case $\cha\kk |d$. Put $c=d([(2)_{d-1},m,T^{*}])$. To prove that $\cP_{+}(\cZ)$ is represented by a faithful family over $G^{\circ}$, by Lemma \ref{lem:outer}\ref{item:outer-trivial} it is enough to prove that the restriction $A$ of $\Aut(Y,D_Y)$ to $G^{\circ}$ is finite (the group $A$ will then be the symmetry group of the resulting family). Since $A$ has a fixed point on $G^{\circ}$, we have $A\leq \G_{m}$. Any subgroup of $\G_m$ acting on $(Y,D_Y)$ fixes the proper transform of some section $H_2\subseteq Y$ as in \ref{lem:ht=1_reduction}\ref{item:uniq_n=2}. Let $(Y,D_Y+H_2)\map (Z',D')$ be the birational map constructed in \ref{lem:ht=1_reduction}\ref{item:uniq_n=2}, and let $(Z',D')\to (\P^1\times \P^1,B)$ be the contraction of all vertical curves not meeting the section $G_{1}'=[0]$ of the $\P^1$-fibration constructed there. Now $B$ consists of two horizontal and three vertical lines, say $\bar{H}_0,\bar{H}_\infty$ and $\bar{V}_0,\bar{V}_1,\bar{V}_{\infty}$, where $\bar{V}_1$ is the proper transform of $G^{\circ}$, and $\{p\}\de \bar{H}_0\cap \bar{V}_0$ is the image of $L_2$. Fix coordinates $(x,y)$ on $\P^1\times \P^1$ such that $\bar{H}_{z}=\{x=z\}$, $\bar{V}_{z}=\{y=z\}$. Every automorphism of $\P^1\times \P^1$ preserving $B$ componentwise is of the form $\alpha_{\lambda}\colon (x,y)\mapsto (\lambda x,y)$, for $\lambda\in \kk^{*}$. Such $\alpha_{\lambda}$ lifts to an element of $\in \Aut(Y_1,D_1)$, and a direct computation shows that the latter acts on the $(-1)$-curve $L$ in the image of $V_2$ as a multiplication by $\lambda^{c}$. Thus $\alpha_{\lambda}$ lifts to an element of $\Aut(Y,D_Y)$ if and only if $\lambda^{c}=1$, and the latter acts on $G^{\circ}$ as multiplication by $\lambda$. Therefore, $A$ is the finite group of $c$-th roots of unity, as needed. 
	
	Case \ref{lem:ht=1_reduction}\ref{item:uniq_n=2} gives $\#\cP_{+}(\cY)=1$, so the above claim and Lemma \ref{lem:outer}\ref{item:outer-trivial} imply that $\cP_{+}(\cZ)$ is represented by a faithful family over $G^{\circ}$. If $\cha\kk\nmid c$ then $h^{1}(\lts{Y}{D_Y})=0$, so this family is universal.
		
	Assume $\cha\kk |c$, so $h^{1}(\lts{Y}{D_Y})=1$. We have $c=d\cdot d([m,T^{*}])-(d-1)\cdot d(T)$, so $\cha\kk|d(T)$. Let $G_{2}^{\circ}$ be the image of $G^{\circ}$ on $Y_2$. 
	At the end of case~\ref{lem:ht=1_reduction}\ref{item:uniq_-2-chain} we have shown that all vertical vector fields on $(Y_2,D_2)$ vanish at some $q\in G^{\circ}$, and there is one nonvanishing along $G_2^{\circ}\setminus \{q\}$. Moreover, since $\cha\kk|d(T)$, the computation in case \ref{lem:ht=1_reduction}\ref{item:uniq_n=3} shows that these vector fields lift to $(Y,D_Y)$. We claim that every $\xi\in H^{0}(\lts{Y}{D_Y})$ is vertical, so vanishes at $q$. Indeed, the image of $\xi$ by the contraction $Y\to \F_{m-2}$ is a linear combination of a vertical vector field and $\check{x}\frac{\d}{\d \check{x}}$, where as in case~\ref{lem:ht=1_reduction}\ref{item:uniq_-2-chain} $\check{x}$ denotes a parameter on the base of the $\P^1$-fibration $\F_{m-2}\to \P^1$, such that the images $V_1$, $V_2$ lie over $\{\check{x}=0\}$ and $\{\check{x}=\infty\}$, see \cite[Lemma 2.5]{FZ-deformations}. A direct computation as in \ref{lem:ht=1_reduction}\ref{item:uniq_-2-chain} shows that $\check{x}\frac{\d}{\d \check{x}}$ does not lift to $Y$, which proves the claim.
	
	Now Lemma \ref{lem:blowup-hi}\ref{item:blowup-hi-outer} shows that $h^{1}(\lts{Z}{D_Z})=h^{1}(\lts{Y}{D_Y})+\epsilon=1+\epsilon$, where $\epsilon=1$ if the blowup $\sigma\colon (Z,D_Z)\to (Y,D_Y)$ is centered at $q$ and $0$ otherwise; as needed.

	\litem{\ref{lem:ht=1_reduction}\ref{item:uniq_3}} As usual, let $\sigma\colon (Z,D_Z)\to (Y,D_Y)$ be the contraction of the $(-1)$-tip of $D_Z$, let $G$ be the $(-1)$-tip of $D_Y$, and let $G^{\circ}=G\setminus (D_Y-G)\cong \A^1$. The combinatorial type $\cY$ of $(Y,D_Y)$ is as in case~\ref{lem:ht=1_reduction}\ref{item:uniq_n=3}, $v=0$, see Figure \ref{fig:uniq_n=3}, so $\#\cP_{+}(\cY)=1$ and $\cP_{+}(\cY)$ is represented by an $h^{1}$-stratified family (over a point). In case \ref{lem:ht=1_reduction}\ref{item:uniq_2}, we have seen that $\Aut(Y,D_Y)$ acts on $G^{\circ}$ with an open orbit, and the derivative of this action along the open orbit does not vanish. By Lemma \ref{lem:outer}\ref{item:outer-fixed-point},\ref{item:outer-h1-stays} it remains to prove that it has a fixed point. 
	
	Let $T$ be the $(-2)$-tip of $\sigma(V_Z)$, and let $\phi\colon (Y,D_Y)\to (\P^2,\ll)$ be the contraction of $D_{Y}-T=\langle 2;[3,2],[2],[1]\rangle$, and let $\{q\}=\Bs\phi^{-1}$. Choose coordinates on $\P^{2}$ so that $\ll=\{x=0\}$ and $q=[0:0:1]$. Put $\cc=\{y^{5}=x^{2}z^3\}$. Since $\#\cP_{+}(\cY)=1$, we can assume that $\phi^{-1}_{*}\cc$ meets $G^{\circ}$. We have an injective group homomorphism $\Aut(Y,D_Y)\ni \alpha\mapsto \phi\circ\alpha\circ\phi^{-1}\in \Aut(\P^2,\ll,q)$, whose image consists of those automorphisms $\upsilon\in \Aut(\P^2,\ll,q)$ such that $\phi^{-1}_{*}\upsilon(\cc)$ meets $G^{\circ}$, too, i.e.\ $(\upsilon(\cc)\cdot \cc)_{q}\geq 11$. 	
	Take any $\upsilon\in \Aut(\P^{2},\ll,q)$, so $\upsilon[x:y:z]=[x:ax+by:cx+dy+ez]$ for some $b,e\in \kk^{*}$, $a,c,d\in \kk$. We can parametrize $\upsilon(\cc)$ near $q$ by $t\mapsto [t^{5}:at^{5}+bt^{2}:ct^{5}+dt^{2}+e]$. Substituting this to the equation of $\cc$ gives $(at^{5}+bt^{2})^{5}-t^{10}(ct^{5}+dt^{2}+e)^3$,  which has coefficient zero near $t^{11}$. Therefore, whenever $(\upsilon(\cc)\cdot \cc)_{q}\geq 11$, we have in fact $(\upsilon(\cc)\cdot \cc)_{q}\geq 12$, that is, $\phi^{-1}_{*}\cc\cap G^{\circ}=\phi^{-1}_{*}\upsilon(\cc)\cap G^{\circ}$. Hence $\phi^{-1}_{*}\cc\cap G^{\circ}$ is the required fixed point of the action of $\Aut(Y,D_Y)$ on $G^{\circ}$.
	
	\litem{\ref{lem:ht=1_reduction}\ref{item:uniq_5}} Let $(Z,D_Z)\overset{\sigma}{\to}(Y,D_Y) \to (\bar{Y},D_{\bar{Y}})$ be the contractions of the $(-1)$-tips of $D_Z$ and $D_Y$. Let $G,\bar{G}$ be the $(-1)$-tips of $D_Y$, $\bar{D}_Y$ and let $G^{\circ}=D_Y\setminus (D_Y-G)\cong \A^1$, $\bar{G}^{\circ}=\bar{G}\setminus (D_{\bar{Y}}-\bar{G})\cong \A^1$. Let $\cY$, $\bar{\cY}$ be the combinatorial types of $(Y,D_Y)$ and $(\bar{Y},D_{\bar{Y}})$, respectively: they are as in cases  \ref{lem:ht=1_reduction}\ref{item:uniq_-2-chain} and \ref{lem:ht=1_reduction}\ref{item:uniq_n=3}, $v=0$. Thus $\#\cP_{+}(\bar{\cY})=1$, and $\cP_{+}(\cY)$ is represented by an $h^{1}$-stratified family, say $f$, over a point if $\cha\kk\neq 2$, and over $\bar{G}^{\circ}$ if $\cha\kk= 2$. By Lemma \ref{lem:outer}\ref{item:outer-open} it remains to show that $\Aut(f)$ acts on $G^{\circ}$ transitively if $\cha\kk\neq 5$, with exactly two orbits if $\cha\kk=5$, and the derivative of this action along the open orbit does not vanish. 
	
	Let $T$ be the $(-2)$-tip of $D_Y$, and let $\phi\colon (Y,D_Y)\to (\F_{2},\pp)$ be the contraction of $\sigma(V_Z)-T$. Then $\pp$ is the sum of the negative section $\Sec_{2}$ and a fiber $F_T\de \phi(T)$. Let $\chi$ be a smooth germ at $\{q\}\de \Sec_{2}\cap F_{T}$ such that $\phi^{-1}_{*}\chi$ meets $G^{\circ}$. Fix coordinates on $\F_{2}=\{([x:y:z];[u:v])\in \P^2\times \P^1: yv^{2}=zu^{2}\}$ so that $\Sec_{2}=\{y=z=0\}$; $F_T=\{u=0\}$. At $q$, we have local coordinates $(\check{z},\check{u})\de (\frac{z}{x},\frac{u}{v})$. Since $\#\cP_{+}(\bar{\cY})=1$, we can assume that $\chi$ is parametrized in those coordinates by $t\mapsto (t,t^{2}+\lambda t^{3})$. The number $\lambda \in \kk$ specifies the point where the proper transform of $\chi$ meets $\bar{G}^{\circ}$. Moreover, if $\cha\kk\neq 2$ then $\#\cP_{+}(\cY)=1$, too, so we can assume $\lambda=0$.
	
	We have an injective group homomorphism $\iota\colon \Aut(Y,D_Y)\ni \beta\mapsto \phi\circ\beta\circ \phi^{-1}\in \Aut(\F_{2},\pp)$, and $\alpha \in \Aut(\F_{2},\pp)$ lies in its image if and only if $(\alpha(\chi)\cdot \chi)_{q}\geq 4$. In this case, $\iota^{-1}(\alpha)$ fixes the point $\{p\}\de \phi^{-1}_{*}\chi\cap G^{\circ}$ if and only if $(\alpha(\chi)\cdot \chi)_{q}\geq 5$. Near $F_T$ we can express any $\alpha\in \Aut(\F_{2},\pp)$ as
	\begin{equation*}
		\alpha([x:y:z],[u:1])=([x+z(a_0+a_1u+a_2u^2):z(au)^2:z(cu+d)^{2}],[au:cu+d]),
	\end{equation*}
	for some $a_0,a_1,a_2,c\in \kk$, $a,d\in \kk^{*}$, see \cite[\sec 6.1]{Blanc_Lukecin}. In local coordinates $(\check{z},\check{u})$ at $q$, the germ $\alpha(\chi)$ is given by 
	\begin{equation*}
		t\mapsto \left(
		\frac{t(d+c(t^{2}+\lambda t^{3}))^{2}}{1+t(a_0+a_1(t^{2}+\lambda t^{3})+a_2(t^{2}+\lambda t^{3})^{2})},\ 
		\frac{a(t^{2}+\lambda t^{3})}{d+c(t^{2}+\lambda t^{3})}
		\right).
	\end{equation*}
	Substituting this to the equation $\check{u}-\check{z}^{2}-\lambda \check{z}^{3}$ of $\chi$ and clearing denominators, we get, modulo $t^{5}$:
	\begin{equation*}
		\begin{split}
			& a(t^{2}+\lambda t^{3})(1+a_0t)^{3}
			-t^{2}(d+c(t^{2}+\lambda t^{3}))^{5}\cdot
			(1+a_0t+
			\lambda d^{2}t)=\\
			= &		(a-d^{5})t^{2}+
			(3aa_{0}+a\lambda-a_0d^{5}-\lambda d^{7})t^{3}+
			(3aa_0^2+3a\lambda a_0-5cd^{4})t^{4} 
		\end{split}
	\end{equation*}
	hence $\alpha=\iota(\tilde{\alpha})$ for some $\tilde{\alpha}\in \Aut(Y,D_Y)$ if and only if $a=d^{5}$ and $a_{0}(3a-d^{5})=\lambda(d^{7}-a)$, and under these conditions, $\tilde{\alpha}(p)=p$ if and only if $3aa_0(a_0+\lambda)=5cd^{4}$. 
	
	Consider the case $\cha\kk\neq 5$. Taking $a,d=1$, $a_0,a_1,a_2=0$ we get a subgroup of $\Aut(\F_2,\pp)$ consisting of 
	\begin{equation*}
		([x:y:z],[u:v])\mapsto ([xv^{2}:zu^2:z(cu+v)^{2}],[u:cu+v]),\quad c\in \kk,
	\end{equation*}
	which lifts to $\Aut(Y,D_Y)$, does not fix $p$, and 
	acts transitively on $G^{\circ}$. 
	Moreover, it does not depend on $\lambda$, so it comes from the action of $\Aut(f)$, as needed.
	
	To compute the derivative, we first write this action in local coordinates $(\check{z},\check{u})$ as $c\cdot (\check{z},\check{u})=(\check{z}(c\check{u}+1)^2,\frac{\check{u}}{c\check{u}+1})$. Its derivative at $c=0$ maps $\frac{\d}{\d c}$ to $\xi\de 2\check{z}\check{u}\frac{\d}{\d\check{z}}-\check{u}^2\frac{\d}{\d\check{u}}$. We need to compute the lift of $\xi$ to $Y$. To this end, we first perform two blowups at $(0,0)$, and write each one as $(z,u)\mapsto (z,zu)$. Next, we perform an outer blowup $(z,zu+1)$ which yields $\bar{Y}$, and to get $Y$ we write $(z,u)\mapsto(z,zu)$ once more. We conclude that in local coordinates $(\check{z},\check{u})$ on $\F_2$ and some $(z,u)$ on $Y$, with $G^{\circ}=\{z=0\}$ we have $\phi(z,u)=(z,z^2(uz^2+1))$. Multiplying $\xi=2z^3(uz^2+1)\frac{\d}{\d\check{z}}-z^2(uz^2+1)\frac{\d}{\d\check{u}}$ by the inverse of its jacobian matrix, we conclude that $\xi$ lifts to $(uz^2+1)\cdot (2z^3\frac{\d}{\d z}-(9uz^2+5)\frac{\d}{\d u})$, which does not vanish along $G^{\circ}=\{z=0\}$ because $\cha\kk\neq 5$.
\smallskip 
	
	Consider the case $\cha\kk=5$. Then $\lambda=0$ by assumption. Now $\alpha\in \Aut(\F_2,\pp)$ lifts to $\Aut(Y,D_Y)$ if and only if $a=d^{5}$ and $a_0=0$, so the condition $3aa_0(a_0+\lambda)=5cd^{4}$ holds automatically, i.e.\  $p$ is a fixed point of $\Aut(Y,D_Y)$. On the other hand, taking $a_{0},a_{1},a_{2},c=0$ and $a=d^{5}$ we get a subgroup consisting of 
	\begin{equation*}
		\alpha_{d}\colon ([x:y:z],[u:v])\mapsto ([xv^2:z(d^{5}u)^2:z(dv)^{2}],[d^{4}u:v]),\quad d\in \kk^{*},
	\end{equation*}
	which lifts to $\Aut(Y,D_{Y})$ and acts transitively on $G^{\circ}\setminus \{p\}$. To see the last assertion, take a germ $\eta$ at $q$ such that $\phi^{-1}_{*}\eta$ meets $G^{\circ}\setminus \{p\}$, say $t\mapsto (t,t^{2}+t^{4})$ in coordinates $(\check{z},\check{u})$. In these coordinates, $\alpha_{d}(\eta)$ is parametrized by $t\mapsto (d^2 t,d^{4}(t^{2}+t^{4}))$, which substituted to the equation of $\eta$ gives $d^{4}(t^{2}+t^{4})-d^{4}t^{2}-d^{8}t^{4}=d^4(1-d^4) t^{4}$, so $(\alpha_{b}(\eta)\cdot \eta)_{q}=4$ unless $d^4=1$. It follows that the above subgroup acts transitively on $G^{\circ}\setminus \{p\}$, as needed. 
	
	A direct computation as before shows that the derivative of this action does not vanish along $G^{\circ}\setminus \{p\}$. Indeed, in local coordinates on $\F_2$ this action is given by $d\cdot (\check{z},\check{u})=(d^2\check{z},d^{4}\check{u})$, so its derivative at $d=1$ is $\xi\de 2\check{z}\frac{\d}{\d\check{z}}+4\check{u}\frac{\d}{\d\check{u}}$. As before, writing $\phi(z,u)=(z,z^2(uz^2+1))$ in some local coordinates $(z,u)$ with $G^{\circ}=\{z=0\}$, we conclude that $\xi$ lifts to $2z\frac{\d}{\d z}-4u\frac{\d}{\d u}$, so does not vanish along $G^{\circ}$, as claimed.
	
	\litem{\ref{lem:ht=1_reduction}\ref{item:uniq_Ek}} As before, let $(Z,D_Z)\to (Y,D_Y)\to (\bar{Y},D_{\bar{Y}})$ be the contraction of the $(-1)$-tips of $D_Z$ and $D_Y$, let $G,\bar{G}$ be the $(-1)$-tips of $D_Y$ and $D_{\bar{Y}}$, let $G^{\circ}=G\setminus (D_Y-G)$, $\bar{G}^{\circ}=D_{\bar{Y}}\setminus (\bar{G}-D_{\bar{Y}})$, and let $\cY$, $\bar{\cY}$ be the combinatorial types of $(Y,D_Y)$ and $(\bar{Y},D_{\bar{Y}})$.
	
	Let $T$ be the $(-2)$-tip of $D_{Z}$ contained in $V_Z$, and let $\phi\colon Z\to \P^2$ be the contraction of $D_{Z}-T$. Then $\ll\de \phi(T)$ is a line, and $\{q\}\de (\Bs \phi^{-1})\redd\subseteq \ll$. Fix coordinates on $\P^2$ so that $\ll=\{x=0\}$, $q=[0:0:1]$.  Let $\chi_{p}$ be a germ at $p\in G^{\circ}$ meeting $G^{\circ}$ normally. In local coordinates $(\check{x},\check{y})\de (\frac{x}{z},\frac{y}{z})$ at $q$, we write $\chi\de \phi_{*}\chi_{p}$ as $(t^{3}+\hot,t)$, with higher order terms chosen for each particular case. Any $\alpha\in \Aut(\P^{2},\ll,q)$ can be written as
	\begin{equation}\label{eq:LTr}
		\alpha\colon [x:y:z]\mapsto [x:ax+by:cx+dy+ez],\quad a,c,d\in \kk,\ b,e\in \kk^{*}.
	\end{equation}	
	As in the previous cases, we say that $\alpha\in \Aut(\P^2,\ll,q)$ \emph{lifts to $(Y,D_Y)$} if it is in the image of the injective homomorphism $\Aut(Y,D_Y)\ni \upsilon\mapsto \phi\circ \upsilon\circ \phi^{-1}\in  \Aut(\P^{2},\ll,q)$. This happens if and only if $(\alpha(\chi)\cdot \chi)_{q}\geq k+3$. The induced automorphism of $(Y,D_Y)$, denoted by the same letter, fixes $p$ if and only if $(\alpha(\chi)\cdot \chi)_{q}\geq k+4$.
	
	The proof in subsequent cases $k=2,3,4$ boils down to a computation analogous to the one in case \ref{lem:ht=1_reduction}\ref{item:uniq_5}. 
	
	\begin{casesp*}
	\litem{$k=2$}
	In this case, type $\bar{\cY}$ is as in case~\ref{lem:ht=1_reduction}\ref{item:uniq_n=3}, $v=0$, see Figure \ref{fig:uniq_n=3}, so $\#\cP_{+}(\bar{\cY})=1$. Therefore, we can take $\chi=(t^{3}+\lambda t^{4},t)$, where $\lambda\in \kk$ parametrizes $\bar{G}^{\circ}$. Since $\cY$ is as in case~\ref{lem:ht=1_reduction}\ref{item:uniq_-2-chain} with $v=0$, $\cP_{+}(\cY)$ is represented by an $h^{1}$-stratified family $f$. Moreover, if $\cha\kk\neq 2$ then $\#\cP_{+}(\bar{\cY})=1$, we can take $\lambda=0$.

	By Lemma \ref{lem:outer}\ref{item:outer-open} it remains to prove that $\Aut(f)$ acts transitively on $G^{\circ}$ if $\cha\kk\neq 3$,  with exactly two orbits if $\cha\kk=3$, and in each case the derivative of this action does not vanish along the open orbit. 
	\smallskip
	
	Assume $\cha\kk\neq 3$. Then the subgroup of $\Aut(\P^2,\ll,q)$ consisting of 
	\begin{equation}\label{eq:Ga_alpha}
		\alpha_{a}\colon [x:y:z]\mapsto [x:y+ax:z],\quad a\in \kk
	\end{equation}	
	lifts to a subgroup $\G_{a}\leq \Aut(Y,D_{Y})$ acting transitively on $G^{\circ}\cong \A^{1}$, with nonvanishing derivative. Indeed, substituting $\alpha_{a}(\chi)=(t^{3}+\lambda t^{4},t+at^{3}+a\lambda t^{4})$ to the equation $\check{y}^{3}(1+\lambda \check{y})-\check{x}$ of $\chi$, we get:
	\begin{equation*}
		t^{3}(1+at^{2}+a\lambda t^{3})^{3}(1+\lambda t+a\lambda t^{3}+a\lambda^{2}t^{4})-(t^{3}+\lambda t^{4})=3at^{5}+\hot
	\end{equation*}
	so $(\alpha_{a}(\chi)\cdot \chi)_{q}=5$, and $a\mapsto \alpha_{a}(p)$ parametrizes $G^{\circ}$, so the subgroup  \eqref{eq:Ga_alpha} acts transitively on $G^{\circ}$. In fact, since \eqref{eq:Ga_alpha} does not depend on $\lambda$, it lifts to a subgroup of $\Aut(f)$, as needed.
	
	Nonvanishing of the derivative follows similarly. The derivative of \eqref{eq:Ga_alpha} at $a=0$ maps $\frac{\d}{\d a}$ to $\check{x}\frac{\d}{\d\check{y}}$. Decomposing $\phi$ as two outer blowups $(x,y)\mapsto (xy+\lambda,y)$, $(x,y)\mapsto (xy+1,y)$ and three inner ones $(x,y)\mapsto (xy,y)$, we get $(\check{x},\check{y})=\phi(x,y)=(xy^5+\lambda y^4+y^3,y)$, where $(x,y)$ are some local coordinates on $Y$ with $G^{\circ}=\{y=0\}$. To compute the required derivative, multiply $\check{x}\frac{\d}{\d\check{y}}$ by the inverse of the jacobian matrix of $\phi$. The resulting vector field $-(xy^2+\lambda y+1)(5xy^2+4\lambda y+3)\frac{\d}{\d x}+(xy^5+\lambda y^4+y^3)\frac{\d}{\d y}$ does not vanish on $G^{\circ}$ since $\cha\kk\neq 3$.
	\smallskip
	
	Assume $\cha\kk=3$. Then $\lambda=0$. Fix $\alpha\in \Aut(\P^2,\ll,q)$; so $\alpha$ is given by formula \eqref{eq:LTr}. We now substitute $\alpha(\chi)=(t^{3}(ct^{3}+dt+e)^{-1},(at^{3}+bt)(ct^{3}+dt+e)^{-1})$ to the equation $\check{y}^{3}-\check{x}$ of $\chi$. Clearing denominators, we get
	\begin{equation*}
		t^{3}((at^{2}+b)^{3}-(ct^{3}+dt+e)^{2})=(b^3-e^2)t^{3}-2det^4-d^2t^5+\hot
	\end{equation*}
	Recall that  $\alpha$ lifts to $(Y,D_Y)$ if and only if the coefficients near $t^0,\dots,t^4$ above vanish, i.e.\ if  $b^3=e^2$ and $d=0$. But under these conditions, the coefficient near $t^5$ vanishes, too, so $\alpha$ fixes $p$. We conclude that $p$ is a fixed point of $\Aut(Y,D_Y)$. On the other hand,  the subgroup of $\Aut(\P^2,\ll,q)$ consisting of 
	\begin{equation}\label{eq:GM}
		\epsilon_{a}\colon [x:y:z]\mapsto [a^{3}x:ay:z], \quad a\in \kk^{*}.
	\end{equation}	
	lifts to a subgroup $\G_{m}\leq \Aut(Y,D_{Y})$ acting transitively on $G^{\circ}\setminus \{p\}\cong \A^{1}_{*}$. To see this, take a germ $\eta'$ meeting $G^{\circ}\setminus \{p\}$ normally, say $\eta'\de\phi^{-1}_{*}\eta$ for $\eta=(t^3+t^5,t)$. Substituting $\epsilon_{a}(\eta)=(a^{3}t^{3}+a^{3}t^{5},at)$ to the equation $\check{y}^{3}+\check{y}^{5}-\check{x}$ of $\eta$, we get
	$
	(a^{5}-a^{3})t^{5}
	$, 
	so $\kk^{*}\ni a\mapsto \epsilon_{a}(\eta')\cap G^{\circ}\in G^{\circ}\setminus \{p\}$ is surjective, as claimed. 
	
	Again, a direct computation shows that the derivative of this action does not vanish. Indeed, the derivative of \eqref{eq:GM} at $a=1$ is $\check{y}\frac{\d}{\d\check{y}}$, and since $\lambda=0$, the local expression of $\phi$ chosen above is $\phi(x,y)=(xy^5+y^3)$, where $(x,y)$ are some local coordinates at $p$, with $G^{\circ}=\{y=0\}$. Multiplying $\check{y}\frac{\d}{\d\check{y}}$ by the inverse of the jacobian matrix of $\phi$ we conclude that our derivative is $(xy^2+1)x\frac{\d}{\d y}$, which does not vanish on $G^{\circ}\setminus \{p\}$.
	
	\litem{$k=3$} Now $\cY$ is as in Subcase $k=2$ above; and $\bar{\cY}$ is as in case \ref{lem:ht=1_reduction}\ref{item:uniq_-2-chain}. 
	
	Assume $\cha\kk\neq 2$. Since $\#\cP_{+}(\bar{\cY})=1$, $\cP_{+}(\cY)$ is represented by a family $f$ over $\bar{G}^{\circ}$, so we can take  $\chi=(t^{3}+\lambda t^{5},t)$ for some $\lambda\in \kk$. By the previous Subcase and Lemma \ref{lem:outer}\ref{item:outer-open}, it remains to prove that $\Aut(f)$ acts transitively on each $G^{\circ}$, with nonzero derivative.  The subgroup  of $\Aut(\P^2,\ll,q)$ consisting of 
	\begin{equation}\label{eq:Ga_beta}
		\beta_{a}\colon [x:y:z]\mapsto [x:y:z+ax],\quad a\in \kk
	\end{equation}	
	lifts to a subgroup of $\Aut(f)$ acting transitively on $G^{\circ}$. Indeed, substituting $\beta_{a}(\chi)=((t^{3}+\lambda t^{5})(1+at^{3}+a\lambda t^{5})^{-1},t(1+at^{3}+a\lambda t^{5})^{-1})$ to the equation $\check{y}^{3}(1+\lambda \check{y}^{2})-\check{x}$ of $\chi$ we get
	\begin{equation*}
		t^{3}((1+at^{3}+a\lambda t^{5})^{2}+\lambda t^{2})-t^{3}(1+\lambda t^{2})(1+at^{3}+a\lambda t^{5})^{4}=-2at^{6}+\hot
	\end{equation*}
	so $a\mapsto \beta_{a}(p)$ parametrizes $G^{\circ}$. As before, we check directly that the derivative of this action does not vanish. To compute it, we multiply the derivative of \eqref{eq:Ga_beta}, namely $-\check{x}^2\frac{\d}{\d\check{x}}-\check{x}\check{y}\frac{\d}{\d\check{y}}$, by the inverse of the jacobian matrix of $\phi(x,y)=(xy^6+\lambda y^5+y^3,y)$, where $(x,y)$ are some local coordinates on $Y$ with $G^{\circ}=\{y=0\}$. Substituting $y=0$ to the resulting vector field gives $2\frac{\d}{\d x}$, which does not vanish along $G^{\circ}$, as needed.
	\smallskip
	
	Assume $\cha\kk=2$. In the previous Subcase, we have inferred from case \ref{lem:ht=1_reduction}\ref{item:uniq_-2-chain} that $\cP_{+}(\bar{\cY})$ is parametrized by germs $\chi=(t^3+\lambda t^4,t)$ and we have shown that the same is true for $\cP_{+}(\cY)$. Thus we can take $\chi=(t^3+\lambda t^4,t)$.
	
	We now study two particular values: $\lambda=1$ and $\lambda=0$. We claim that $\Aut(Y,D_Y)$ acts transitively on $G^{\circ}$ in the first case, and has one fixed point in the second case; moreover, the derivative of either action does not vanish along the open orbit. Once this is shown, we will see that $\lambda=1,0$ correspond to the general and special fiber of the $h^{1}$-stratified family representing $\cP_{+}(\cY)$, and our result will follow from Lemma \ref{lem:outer}\ref{item:outer-fixed-point}.
	
	Consider the case $\lambda=1$. Then a transitive $\G_{a}$-action on $G^{\circ}$ is given by
	\begin{equation}\label{eq:Ga_gamma}
		\gamma_{a}\colon [x:y:z]\mapsto [x:y+a(1+a)x:z+ay],\quad a\in \kk.
	\end{equation}	
	Indeed, substituting $\gamma_{a}(\chi)=(\frac{t^{3}(1+t)}{1+at},\frac{t+a(1+a)t^{3}(1+t)}{1+at})$ to the equation $\check{y}^3+\check{y}^4+\check{x}$ of $\chi$, we get
	\begin{equation*}
		t^{3}(1+a(1+a)t^2(1+t))^{3}(1+at+t+a(1+a)t^{3}(1+t))+t^{3}(1+t)(1+at)^{3}=
		(a+a^2)t^{6}+\hot
	\end{equation*}	
	so $a\mapsto \gamma_{a}(p)$ is surjective, as needed. To compute the derivative, we multiply the derivative of \eqref{eq:Ga_gamma}, namely $\check{x}\check{y}\frac{\d}{\d\check{x}}+(\check{x}+\check{y}^2)\frac{\d}{\d\check{y}}$ by the inverse of the jacobian matrix of $\phi(x,y)=(xy^6+y^4+y^3,y)$, and restrict it to $G^{\circ}=\{y=0\}$. The resulting vector field $\frac{\d}{\d x}$ does not vanish, as needed.
	
	Consider the case $\lambda=0$. Then for any $\alpha\in \Aut(\P^2,\ll,q)$ as in \eqref{eq:LTr},  substituting $\alpha(\chi)=(\frac{t^{3}}{ct^{3}+dt+e},\frac{at^{3}+bt}{ct^{3}+dt+e})$ to the equation $\check{y}^{3}+\check{x}$ of $\chi$ gives
	\begin{equation*}
		(at^{3}+bt)^{3}+t^{3}(ct^{3}+dt+e)^{2}=(b^3+e^2)t^{3}+(ab^{2}+d^2)t^5+a^{2}bt^{7}+\hot
	\end{equation*}
	Therefore, $\alpha$ lifts to $(Y,D_Y)$ if and only if $b^3=e^2$ and $ab^2=d^2$; and under these conditions $\alpha$ fixes $\{p\}$. The subgroup \eqref{eq:GM} acts transitively on $G^{\circ}\setminus \{p\}$. To see this, fix a germ $\eta$ such that $\phi^{-1}_{*}\eta$ meets $G^{\circ}$ transversally off $p$, say $\eta=(t^3+t^6,t)$. Then substituting $\epsilon_{a}(\eta)=(a^{3}t^{3}(1+t^3),at)$ to the equation $\check{y}^3(1+\check{y}^3)+\check{x}$ of $\eta$ gives
	\begin{equation*}
		a^3t^3(1+a^3t^3)+a^3t^3(1+t^3)=(a^3+a^6)t^6,
	\end{equation*}
	so $\kk^{*}\ni a\mapsto \phi^{-1}_{*}\epsilon_{a}(\eta)\cap G^{\circ}\in G^{\circ}\setminus \{p\}$ is indeed surjective. Its derivative is obtained by multiplying the derivative of \eqref{eq:GM}, namely $\check{x}\frac{\d}{\d\check{x}}+\check{y}\frac{\d}{\d\check{y}}$, by the inverse of the jacobian matrix of $\phi(x,y)=(xy^6+y^4,y^3)$. Restricted to $G^{\circ}=\{y=0\}$, the resulting vector field is $x\frac{\d}{\d x}$, so it does not vanish on $G^{\circ}\setminus \{p\}$, as needed.
	
	\litem{$k=4$} Now $\cY$ and $\bar{\cY}$ are as in Subcases $k=3$ and $k=2$ above. Recall that $\cP_{+}(\cY)$ is represented by an $h^{1}$-stratified family whose base is parametrized by some family of germs $\chi$. Let $(Y_{\chi},D_{Y,\chi})\in \cP_{+}(\cY)$ be the fiber corresponding to the germ $\chi$, and as usual let $G^{\circ}_{\chi}=G_{\chi}\setminus (D_{Y,\chi}-G_{\chi})$, where $G_{\chi}$ is the $(-1)$-tip of $D_{Y,\chi}$. Let $\cA(\chi)$ be the action of $\Aut(Y_{\chi},D_{Y,\chi})$ on $G^{\circ}_{\chi}$. We claim that it is enough to prove the following. 
	\begin{enumerate-alt}
		\item\label{item:E8-char_neq_2,3} If $\cha\kk\neq 2,3$ then $\cA(t^3,t)$ has exactly two orbits.
		\item\label{item:E8-char=3} If $\cha\kk=3$ then $\cA(t^3+t^5,t)$ is transitive, and $\cA(t^3,t)$ has exactly two orbits.
		\item\label{item:E8-char=2} If $\cha\kk=2$ then $\cA(t^3,t)$, $\cA(t^3+t^4,t)$, and $\cA(t^3+t^6,t)$ are transitive.
		\item\label{item:E8-derivative} The derivative of each of the above actions does not vanish along the open orbit, unless $\cha\kk=2$ and $\chi=(t^3,t)$. In the latter case, all vector fields on $Y_{\chi}$ tangent to $D_{Y_{\chi}}$ vanish along $G_{\chi}^{\circ}$.
	\end{enumerate-alt}
	Indeed, assume that \ref{item:E8-char_neq_2,3}--\ref{item:E8-derivative} hold. If $\cha\kk\neq 2,3$ then $\#\cP_{+}(\cY)=1$, so \ref{lem:ht=1_uniqueness}\ref{item:ht=1_uniqueness-finite} follows from \ref{item:E8-char_neq_2,3},\ref{item:E8-derivative} and Lemma \ref{lem:outer}\ref{item:outer-fixed-point}. If $\cha\kk=3$ then $\#\cP_{+}(\cY)=2$, and \ref{item:E8-char=3} shows that the germs $(t^3+t^5,t)$ and $(t^3,t)$ correspond to, respectively, the general and special fibers of our $h^{1}$-stratified family; so again  \ref{lem:ht=1_uniqueness}\ref{item:ht=1_uniqueness-finite} follows from \ref{item:E8-char=3},\ref{item:E8-derivative} and Lemma \ref{lem:outer}\ref{item:outer-fixed-point}. Eventually, assume $\cha\kk=2$. In Subcase $k=1$ we have shown that the log surfaces corresponding to $(t^3,t)$ and $(t^3+t^4,t)$ are non-isomorphic; and in Subcase $k=2$ we have shown that the one corresponding to $(t^3+t^6,t)$ is not isomorphic to either of them; in fact it is the special fiber of our family. Thus in this case \ref{lem:ht=1_uniqueness}\ref{item:ht=1_uniqueness-finite} follows from \ref{item:E8-char=2},\ref{item:E8-derivative} and Lemma \ref{lem:outer}\ref{item:outer-transitive},\ref{item:outer-h1}.
	\smallskip
	
	Thus it remains to show \ref{item:E8-char_neq_2,3}--\ref{item:E8-derivative}. We begin with the case $\chi=(t^3,t)$. Fix any $\alpha\in \Aut(\P^2,\ll,q)$ as in \eqref{eq:LTr}. Substituting $\alpha(\chi)$ to the equation $\check{y}^3-\check{x}$ of $\chi$, we get 
	\begin{equation*}
		(at^{3}+bt)^{3}-t^{3}(ct^{3}+dt+e)^{2}=(b^3-e^2)t^{3}-2det^4+(3ab^2-d^2)t^5-2cet^6+(3a^2b-2cd)t^7+\hot.
	\end{equation*}	
	Such $\alpha$ lifts to $(Y_{\chi},D_{Y,\chi})$ if and only if coefficients near $t^0,\dots,t^6$ vanish. 
	
	Assume $\cha\kk\neq 2$. Then $\Aut(Y_{\chi},D_{Y,\chi})$ is generated by \eqref{eq:GM} and, in case $\cha\kk=3$, by \eqref{eq:Ga_alpha}. In any case, the coefficient near $t^7$ above vanishes if $\alpha$ lifts to $(Y_{\chi},D_{Y,\chi})$, so $\Aut(Y,D_Y)$ fixes $\{p\}$. The map $a\mapsto \phi^{-1}_{*}\epsilon_{a}(\eta)\cap G_{\chi}^{\circ}\setminus \{p\}$ for $\eta=(t^3+t^7,t)$ is surjective, so \eqref{eq:GM} acts transitively on $G_{\chi}^{\circ}\setminus \{p\}$, as claimed in \ref{item:E8-char_neq_2,3} and \ref{item:E8-char=3}. Note that this ends the proof of \ref{item:E8-char_neq_2,3} in case $\cha\kk\neq 2$, we will treat case $\cha\kk=2$ at the end. 
	
	To prove the corresponding claim in \ref{item:E8-derivative} we multiply, as before, the derivative $3\check{x} \frac{\d}{\d\check{x}}+\check{y}\frac{\d}{\d\check{y}}$ of \eqref{eq:GM} by the inverse of the jacobian matrix of $\phi(x,y)=(xy^7+y^3,y)$, where $(x,y)$ are coordinates at $p$ with $G^{\circ}=\{y=0\}$. The resulting vector field $-4x\frac{\d}{\d x}+y\frac{\d}{\d y}$ does not vanish along $G^{\circ}\setminus \{p\}$, as needed.
	\smallskip
		
	Assume now that $\cha\kk=3$, consider $\chi=(t^3+t^5,t)$ and the subgroup \eqref{eq:Ga_alpha}. Substituting $\alpha_{a}(\chi)=(t^3(1+t^2),t(1+at^2(1+t^2)))$ to the equation $\check{y}^3(1+\check{y}^2)-\check{x}$ of $\chi$, we get 
	\begin{equation*}
		t^3(1+a^3t^6(1+t^6))(1+t^2(1+at^2(1+t^2))^2)-t^3(1+t^2)=2at^7+\hot
	\end{equation*}
	so $a\mapsto \alpha_{a}(p)$ parametrizes $G_{\chi}^{\circ}$, hence \eqref{eq:Ga_alpha} acts transitively on $G_{\chi}^{\circ}$, which ends the proof of \ref{item:E8-char=3}. Again, we compute directly that the derivative of this action does not vanish, as claimed in \ref{item:E8-derivative}. To this end, we multiply the derivative $\check{x}\frac{\d}{\d\check{y}}$ of \eqref{eq:Ga_alpha} by the inverse of the jacobian of $\phi=(xy^7+y^5+y^3,y)$: the resulting vector field, restricted to $G^{\circ}=\{y=0\}$ equals $\frac{\d}{\d x}$, so does not vanish along $G^{\circ}$, as needed.
	
	Assume $\cha\kk=2$. Like before we compute that \eqref{eq:Ga_beta} gives a  transitive action of $\cA(t^3+t^4,t)$, and 
	\begin{equation}\label{eq:Ga_delta}
		\delta_{d}\colon
		[x:y:z]\mapsto [x:y+d^2x:z+dy],\quad d\in \kk,
	\end{equation}	
	gives a transitive action of $\cA(t^3+t^6,t)$; moreover, their derivatives do not vanish along $G^{\circ}$, as needed. In the remaining case $\chi=(t,t^3)$, the group \eqref{eq:Ga_delta} acts transitively on $G^{\circ}$, too, which ends the proof of \ref{item:E8-char_neq_2,3}--\ref{item:E8-char=2}.

To end the proof of \ref{item:E8-derivative}, we need to show that for $\chi=(t,t^3)$, every vector field $\xi$ on $Y$ tangent to $D_Y$ vanishes along $G^{\circ}$. The  image $\check{\xi}$ of $\xi$ on $\P^2$ vanishes at $(0,0)$ and is tangent to $\{x=0\}$, so in coordinates $(\check{x},\check{y})$ it can be written as $\check{x}(e+c\check{x}+d \check{y})\frac{\d}{\d\check{x}}+(a\check{x}+\check{y}(b+c \check{x}+d\check{y}))\frac{\d}{\d \check{y}}$ for some $a,b,c,d,e \in \kk$. 
	
To compute $\xi$, we multiply $\check{\xi}$ by the inverse of the jacobian matrix of $\phi(x,y)=(xy^{7}+y^{3},y)$, where as before $(x,y)$ are some local coordinates on $Y$ with $\{y=0\}$. The coefficient near $\frac{\d}{\d x}$ of the resulting vector field is $(x+y^{-4})(b+e+ay^2+axy^{7})$. Since $\check{\xi}$ lifts to a vector field $\xi$ on $Y$, we have $b+e=0$ and $a=0$, so the coefficient near $\frac{\d}{\d x}$ is zero. Thus $\xi$ vanishes along $G^{\circ}$, as claimed.
\end{casesp*}
	\litem{\ref{lem:ht=1_reduction}\ref{item:uniq-bench}} Let $F_{1}'$, $F_{2}'$ be the fibers contained in $D_Z$, and for $i\in \{1,2\}$ let $A_i$ and $T_i$ be the $(-1)$- and $(-2)$-tip of $(F_i')\redd$. Let $\sigma\colon (Z,D_{Z})\to (Y,D_{Y})$ be the contraction of $A_2$, let $G$ be the  $(-1)$-tip of $D_Y$ containing $\sigma(A_2)$, and let $G^{\circ}=G\setminus (D_{Y}-G)\cong \A^1$. Moreover, let $(Y,D_Y)\to (\bar{Y},D_{\bar{Y}})$ be the contraction of a $(-1)$-tip of $D_Y$ other than $G$, let $\bar{G}$ be the new $(-1)$-tip of $D_{\bar{Y}}$, and let $\bar{G}^{\circ}=\bar{G}\setminus (D_{\bar{Y}}-\bar{G})\cong \A^1$.
		
	The combinatorial type $\cY$ of $(Y,D_Y)$ is as in case \ref{lem:ht=1_reduction}\ref{item:uniq_fork} with $d=2$, $T=[2]$. We claim that $\Aut(Y,D_Y)$ acts on $G^{\circ}$ trivially if $\cha\kk=2$, and transitively with nonvanishing derivative if $\cha\kk\neq 2$. 

	Put $n=m-2\geq 1$. Let $\tau\colon (Y,D_Y)\to (\F_{n},\pp)$ be the contraction of $(D_{Y})\vert-T_1-T_2$. Then $\pp=\Sec_{n}+F_1+F_2$, where $F_i=\tau_{*}T_i$ is a fiber, and $\Sec_n=[n]$ is the negative section. Fix coordinates on $\F_{n}=\{([x:x_1:x_2],[u_1:u_2])\in \P^2\times \P^1: x_1u_1^n=x_2u_2^n\}$ so that  $F_i=\{u_i=0\}$, $\Sec_n=\{x_1=x_2=0\}$. 
		It follows from \cite[\sec 6.1]{Blanc_Lukecin} that $\Aut(\F_{n},\pp)=\{\alpha_{p,\lambda_1,\lambda_2}:\ p\in \kk[u_1,u_2]_n,\ \lambda_1,\lambda_2\in \kk^{*}\}$, where $\alpha_{p,\lambda_1,\lambda_2}$ is given on $\{u_1\neq 0\}$ by 
		\begin{equation*}
			\alpha_{p,\lambda_1,\lambda_2}\colon ([x:x_1:x_2],[u_1:u_2])\mapsto ([x+x_{2}\cdot p (1,\tfrac{u_2}{u_1}):\lambda_1^{n}x_{1}:\lambda_2^{n}x_{2}],[\lambda_2 u_1:\lambda_1 u_2]).
		\end{equation*}
		Write $\{q_i\}=F_i\cap \Sec_n$. Let $\chi_{i}$ be a germ  given in local coordinates $(\check{x}_i,\check{u}_i)\de (\frac{x_i}{x},\frac{u_i}{u_{2-i}})$ at $q_i$ by $t\mapsto (t,a_it^2+b_it^3)$ for some $a_i\in \kk^{*}$, $b_i\in \kk$. Fix $a_1,b_1,a_2$ so that the proper transform $\tau^{-1}_{*}\chi_{i}$ meets $D_Y$ normally, in a $(-1)$-tip of $\sigma_{*}F_i'$: then $\kk\ni b_{2}\mapsto \tau^{-1}_{*}\chi_{2}\cap G^{\circ}$ is a parametrization of $G^{\circ}$.  
		An $\alpha\in \Aut(\F_{n},\pp)$ lifts to $\Aut(Y,D_Y)$ if and only if $(\alpha(\chi_1)\cdot \chi_1)_{q_1}\geq 4$, $(\alpha(\chi_2)\cdot \chi_2)_{q_2}\geq 3$; and it fixes the point $\tau^{-1}_{*}\chi_{2}\cap G^{\circ}$ if and only if $(\alpha(\chi_2)\cdot \chi_2)_{q_2}\geq 4$. Let $s_{i}\de a_{i}t^{2}+b_{i}t^{3}$ be the $\check{u}_{i}$-th coordinate of $\chi_{i}$. Then in the local coordinates $(\check{x}_i,\check{u}_i)$ at $q_{i}$ we have 
		\begin{equation*}
			\alpha_{p,\lambda_1,\lambda_2}(\chi_1)=\left(\frac{\lambda_1^n\cdot t}{1+t\cdot p(s_{1},1)},\frac{\lambda_2}{\lambda_1}\cdot s_{1},\right),
			\quad
			\alpha_{p,\lambda_1,\lambda_2}(\chi_2)=\left(\frac{\lambda_2^n \cdot t}{1+t\cdot p(1,s_2)}, \frac{\lambda_1}{\lambda_2}\cdot  s_{2}\right).
		\end{equation*}
		Substituting these formulas to the equations $a_i\check{x}_i^2+b_i\check{x}_i^3-\check{u}_i$ of $\chi_{i}$, we compute that $(\alpha_{p,\lambda,\mu}(\chi_{i})\cdot \chi_{i})\geq 3$ if and only if $\lambda_i=\lambda_{2-i}^{2n+1}$, in which case $(\alpha_{p,\lambda}(\chi_{i})\cdot \chi_{i})\geq 4$ if and only if $2a_{i}p_{i}=(\lambda_{i}^n-1)b_i$, where $p_i$ is the coefficient of $p$ near $u_{2-i}^{n}$.  Thus $\alpha_{p,\lambda_1,\lambda_2}$ lifts to $\Aut(Y,D_Y)$ if and only if $\lambda_{1}^n=\lambda_2^n=1$ and $2p_{1}=0$; and such a lift acts trivially on $G^{\circ}$ if and only if $2p_2=0$. 
		Thus if $\cha\kk=2$ then the action of $\Aut(Y,D_Y)$ on $G^{\circ}$ is trivial. 
		
		Assume $\cha\kk\neq 2$. For $q\in \kk$ put $\alpha_q=\alpha_{qu_1^n,1,1}$.  The subgroup $\{\alpha_{q}:q\in \kk\}\cong \G_a$ lifts to $\Aut(Y,D_Y)$ and acts transitively on $G^{\circ}$. To compute its derivative, write $\alpha_q$ in local coordinates as $\alpha_q(\check{x}_2,\check{u}_2)=(\frac{\check{x}_{2}}{1+q\check{x}_2},\check{u}_2)$. Its derivative at $q=0$ maps $\frac{\d}{\d q}$ to $-\check{x}_{2}^2\frac{\d}{\d\check{x}_{2}}$. Multiplying this vector field by the inverse of the jacobian matrix of $\tau(x,u)=(x,ux^3+x^2)$, where like before $(x,u)$ are some local coordinates on $Y$ with $G^{\circ}=\{x=0\}$, we get $-x^2\frac{\d}{\d x}+(3xu+2)\frac{\d}{\d u}$, which does not vanish along $G^{\circ}$ since $\cha\kk\neq 2$, as claimed.
		\smallskip
		
		If $\cha\kk\neq 2$ then the above claim and Lemma \ref{lem:outer}\ref{item:outer-transitive} show that $\#\cP_{+}(\cZ)=\#\cP_{+}(\cY)=1$; and by Lemma \ref{lem:outer}\ref{item:outer-h1-stays} we get $h^{1}(\lts{Z}{D_Z})=h^{1}(\lts{Y}{D_Y})$, which equals $0$ if $\cha\kk\nmid m-1$ and $1$ otherwise.
		
		Assume  $\cha\kk=2$. Lemma \ref{lem:outer}\ref{item:outer-trivial} implies that $\#\cP_{+}(\cZ)$ is represented by a faithful family over $G^{\circ}\times \bar{G}^{\circ}$, which by symmetry of our construction is $\Aut(\cZ)$-faithful for $\Aut(\cZ)=\Z/2$ interchanging the factors of $G^{\circ}\times \bar{G}^{\circ}$.
		
		Put $\bar{G}_{1}^{\circ}=\bar{G}^{\circ}$ and let $\bar{G}_{2}^{\circ}$ be the image of $G^{\circ}$ on $\bar{Y}$. The computation in case \ref{lem:ht=1_reduction}\ref{item:uniq_-2-chain} (repeated for both fibers) shows that there are points $\bar{q}_{i}\in \bar{G}_i^{\circ}$ such that every vector field $\xi\in H^{0}(\lts{\bar{Y}}{D_{\bar{Y}}})$ vanishes at $\bar{q}_{1},\bar{q}_2$, and there is a vector field $\xi_{1}$ which does not vanish along $\bar{G}_{1}^{\circ}\setminus \{\bar{q}_{1}\}$. We check that any such $\xi_{1}$ does not vanish along  $\bar{G}_{2}^{\circ}\setminus \{\bar{q}_{2}\}$, either. Since $h^{1}(\lts{\bar{Y}}{D_{\bar{Y}}})=1$, Lemma \ref{lem:blowup-hi}\ref{item:blowup-hi-outer} implies that $h^{1}(\lts{Z}{D_Z})=3$ if $(Z,D_Z)$ is obtained by blowing up at $\bar{q}_{1}$, $\bar{q}_2$, and $h^{1}(\lts{Z}{D_Z})=2$ otherwise.
	\qedhere\end{casesp*}
\end{proof}

\clearpage
\section{Del Pezzo surfaces of height two}\label{sec:ht=2}

In this section, we complete the proof of Theorem \ref{thm:ht=1,2} by describing all del Pezzo surfaces of rank one and height $2$. The proof is organized into the following steps.
\begin{steps}
	\item\label{step:basic-ht=2} We construct vertically primitive surfaces of height $2$ from Theorem \ref{thm:ht=1,2}, see Figures \ref{fig:3A2-intro}--\ref{fig:KM_surface-intro}. 
	\item\label{step:swap-ht=2} We show that every del Pezzo surface $\bar{X}$ of rank $1$ and height $2$ swaps vertically to one of the surfaces constructed in Step \ref{step:basic-ht=2}. This is done in Lemmas \ref{lem:ht=2_swaps} and \ref{lem:ht=2_twisted-swaps}, by a suitable choice of a witnessing $\P^1$-fibration.
	\item 	\label{step:classifiation-ht=2} We reconstruct all possible sequences of vertical swaps from Step \ref{step:swap-ht=2}, thus classifying singularity types.
	\item \label{step:uniqueness-ht=2} We classify surfaces $\bar{X}$ of a given singularity type $\cS$, that is, we check to which extent the sequence reconstructed in Step \ref{step:classifiation-ht=2} is uniquely determined by $\cS$. This amounts to classifying vertically primitive surfaces from Step \ref{step:basic-ht=2}, and describing outer blowups made in Step \ref{step:classifiation-ht=2}, as outlined in Observation \ref{obs:blowups}.
\end{steps}

We note that the proof in case $\height=1$ followed a similar path: we have constructed primitive surfaces in Example \ref{ex:ht=1}, and classified singularity types in Lemma \ref{lem:ht=1_types}. The main technical part was the final Step \ref{step:uniqueness-ht=2}. 

Now, the situation is different. There are several classes of vertically primitive surfaces, including one specific to characteristic $2$. The latter is constructed from a pencil of lines tangent to a smooth planar conic, see \cite[\sec 9]{Keel-McKernan_rational_curves} or Example \ref{ex:ht=2_twisted_cha=2}: clearly, such a pencil exists if and only if $\cha\kk=2$. The remaining surfaces are constructed from simpler planar configurations, and their combinatorics do not depend on $\cha\kk$.

As a consequence, the reconstruction process in Step \ref{step:classifiation-ht=2} gives a longer list of singularity types: for example, from the primitive surfaces specific to characteristic $2$ we get complicated ones listed in Table \ref{table:ht=2_char=2}, see Theorem \ref{thm:ht=1,2}\ref{item:ht=2_char=2}. Nonetheless, it turns out that this process involves very few outer blowups, so Step \ref{step:uniqueness-ht=2} is easier, and essentially boils down to the known classification of vertically primitive surfaces from Step \ref{step:basic-ht=2}.

In Steps \ref{step:swap-ht=2}--\ref{step:uniqueness-ht=2}, we first study the case $\width=2$, i.e.\ when for some witnessing $\P^1$-fibration the horizontal part of the boundary consists of two $1$-sections. This is done in Section \ref{sec:ht=2_untwisted}. The remaining case $\width=1$, which contains examples specific to characteristic $2$, is settled  in Section \ref{sec:ht=2_twisted}.

\subsection{Vertically primitive surfaces of height $2$}\label{sec:ht=2_constructions}

We now construct vertically primitive surfaces 
appearing in Theorem \ref{thm:ht=1,2}\ref{item:ht=2_width=2},\ref{item:ht=2_width=1}. Their constructions 
are known; see e.g.\  \cite{MZ_canonical} for the canonical ones (Fig.\ \ref{fig:3A2-intro}--\ref{fig:2A3+2A1-intro}) and  \cite[\S 9]{Keel-McKernan_rational_curves} for the additional ones in characteristic $2$ (Fig.\ \ref{fig:KM_surface-intro}). Nonetheless, in our approach we need a self-contained presentation. 

\begin{notation}[A pencil of lines]\label{not:pencil_of_lines}
	Fix a point $p_0\in \P^2$, lines $\ll_{1},\dots,\ll_{\nu}$, passing through $p_0$, and a curve $\cc\subseteq \P^2$. Put $\pp=\cc+\ll_1+\dots+\ll_\nu$. For a birational morphism $\phi\colon Y\to \P^2$ such that $p_0\in \Bs\phi^{-1}$, we define $D_{Y}$ as $(\phi^{*}\pp)\redd$ with all $(-1)$-curves subtracted, and consider a $\P^1$-fibration of $Y$ induced by the pencil of lines through $p_0$. We denote the proper transforms of $\ll_i,\cc$ by $L_i,C$, respectively. For a point $p_i\in \P^2$, we denote by $A_i$ the $(-1)$-curve in the preimage of $p_i$. The surface $\bar{Y}$ from Theorem \ref{thm:ht=1,2} is obtained from $Y$ by contracting $D_Y$.
\end{notation}

\begin{example}[Del Pezzo surfaces of rank one and  types $3\rA_{2}$, $\rA_{1}+2\rA_{3}$, $\rA_{1}+\rA_{2}+\rA_{5}$ or $2\rD_{4}$, see Figure \ref{fig:ht=2_w=2_basic}]
	\label{ex:ht=2} 
We keep Notation \ref{not:pencil_of_lines}. Let $\cc$ be a line not passing through $p_{0}$. Write $\{p_{j}\}=\ll_{j}\cap \cc$. We now define a morphism $\phi\colon Y\to\P^2$ as in Notation \ref{not:pencil_of_lines}, such that $(Y,D_{Y})$ is as in Figure \ref{fig:3A2}--\ref{fig:2D4}; and the del Pezzo surface $\bar{Y}$ obtained by contracting $D_Y$ is of type $3\rA_2$, $\rA_1+2\rA_3$, $\rA_1+\rA_2+\rA_5$ or $\rD_4$. 
	
	\begin{parlist}
		\item\label{item:3A2_construction} For type $3\rA_{2}$: take $\nu=2$, and blow up three times at $p_0, p_{1}$ and their infinitely near points on the proper transforms of $\ll_2$ and $\cc$, respectively, see Figure \ref{fig:3A2}.
		\item\label{item:A1+2A3_construction} For type $\rA_{1}+2\rA_{3}$: take $\nu=3$, blow up at $p_{1}$; and at $p_0$, $p_2$, $p_3$ and their infinitely near points on the proper transforms of $\ll_{1}$, $\ll_{2}$, $\ll_{3}$, respectively, see Figure \ref{fig:2A3+A1}.
		\item\label{item:A1+A2+A5_construction} For type $\rA_{1}+\rA_{2}+\rA_{5}$: take $\nu=3$, choose a point $q_3\in \ll_{3}\setminus \cc$, blow up at $p_0$, $p_1$, $p_2$, $q_3$ and their infinitely near points on the proper transforms of $\ll_{1}$, $\cc$, $\ll_{2}$, $\ll_{3}$. Let $A_3$ be the $(-1)$-curve over $q_3$ (not $p_3$), see Fig.\  \ref{fig:A5+A2+A1}.
		\item\label{item:2D4_construction} For type $2\rD_{4}$: take $\nu=4$, blow up at $p_{4}$; twice at $p_{2}$, $p_{3}$ and three times at $p_{0}$; each time on the proper transforms of $\ll_{2}$, $\ll_{3}$ and $\ll_{1}$; respectively, see Figure \ref{fig:2D4}.
	\end{parlist}
	In cases \ref{item:3A2_construction}--\ref{item:A1+A2+A5_construction}, $\pp\subseteq \P^2$ is uniquely determined by four points $(p_0,p_1,p_2,q)$ in a general position (where $q=p_3$ or $q_3$), so the surface $\bar{Y}$ 
	is unique up to an isomorphism. In case \ref{item:2D4_construction}  
	the set of the isomorphism classes of surfaces $\bar{Y}$ is 
	parametrized by the choice of the fourth line $\ll_4$; see Lemma \ref{lem:ht=2,untwisted} for details.
	
	In Remark \ref{rem:primitive_ht=2} we will see that the surfaces in \ref{item:3A2_construction}--\ref{item:A1+A2+A5_construction} are not only vertically primitive, but in fact primitive. The one in  \ref{item:2D4_construction} is not: indeed, contracting $A_4$ we get a swap to a surface of type $3\rA_1+\rD_4$, see Figure \ref{fig:3A1+D4}.
\end{example}

\begin{figure}[ht]
	\subcaptionbox{$3\rA_{2}$ \label{fig:3A2}}[.12\linewidth]{\centering
		\begin{tikzpicture}[scale=1]
			\draw (0.1,3) -- (1.3,3);
			\draw[dashed] (0.2,3.1) -- (0,2.1);
			\node at (0.3,2.6) {\small{$L_1$}};
			\draw (0,2.3) -- (0.2,1.4);
			\draw (0.2,1.6) -- (0,0.7);
			\draw[dashed] (0,0.9) -- (0.2,-0.1);
			\node at (0.3,0.4) {\small{$A_1$}};
			\draw (1.2,3.1) -- (1,1.9);
			\draw[dashed] (1,2.1) -- (1.2,0.9);
			\node at (1.3,1.5) {\small{$A_0$}};
			\draw (1.2,1.1) -- (1,-0.1);
			\node at (1.3,0.5) {\small{$L_2$}};
			\draw (0.1,0.05) -- (1.1,0.05);
			\node at (0.7,0.2) {\small{$C$}};
		\end{tikzpicture}
	}
	\subcaptionbox{$\rA_{1}+2\rA_{3}$\label{fig:2A3+A1}}[.21\linewidth]{\centering
		\begin{tikzpicture}[scale=1]
			\draw (0.1,3) -- (2.3,3);
			\draw (-0.1,0) -- (2.1,0);
			\node at (1.6,0.2) {\small{$C$}};
			\draw[dashed] (0.2,3.1) -- (0,1.9);
			\node at (0.3,2.5) {\small{$A_0$}};
			\draw (0,2.1) -- (0.2,0.9);
			\node at (0.3,1.5) {\small{$L_1$}};
			\draw[dashed] (0.2,1.1) -- (0,-0.1);
			\node at (0.3,0.5) {\small{$A_1$}};
			\draw (1.2,3.1) -- (1,1.9);
			\node at (1.3,2.5) {\small{$L_2$}};
			\draw[dashed] (1,2.1) -- (1.2,0.9);
			\node at (1.3,1.5) {\small{$A_2$}};
			\draw (1.2,1.1) -- (1,-0.1);
			\draw (2.2,3.1) -- (2,1.9);
			\node at (2.3,2.5) {\small{$L_3$}};
			\draw[dashed] (2,2.1) -- (2.2,0.9);
			\node at (2.3,1.5) {\small{$A_3$}};
			\draw (2.2,1.1) -- (2,-0.1);			
		\end{tikzpicture}
	}
	\subcaptionbox{$\rA_1+\rA_2+\rA_5$ \label{fig:A5+A2+A1}}[.21\linewidth]{\centering
		\begin{tikzpicture}[scale=1]
			\draw (0.1,3) -- (2.4,3);
			\draw (0,0) -- (1,0) to[out=0,in=180] (2.3,2.8) -- (2.4,2.8);
			\node at (1.7,0.4) {\small{$C$}};
			\draw[dashed] (0.2,3.1) -- (0,2.1);
			\node at (0.3,2.6) {\small{$A_0$}};
			\draw (0,2.3) -- (0.2,1.4);
			\node at (0.3,1.9) {\small{$L_1$}};
			\draw (0.2,1.6) -- (0,0.7);
			\draw[dashed] (0,0.9) -- (0.2,-0.1);
			\node at (0.3,0.4) {\small{$A_1$}};
			\draw (1.2,3.1) -- (1,1.9);
			\node at (1.3,2.5) {\small{$L_2$}};
			\draw[dashed] (1,2.1) -- (1.2,0.9);
			\node at (1.3,1.5) {\small{$A_2$}};
			\draw (1.2,1.1) -- (1,-0.1);
			\draw (2.3,3.1) -- (2.1,1.9);
			\node at (2.4,2.5) {\small{$L_3$}};
			\draw[dashed] (2.1,2.1) -- (2.3,0.9);
			\node at (2.4,1.5) {\small{$A_{3}$}};
			\draw (2.3,1.1) -- (2.1,-0.1);			
		\end{tikzpicture}
	}	
	\subcaptionbox{$2\rD_4$ \label{fig:2D4}}[.28\linewidth]{\centering
		\begin{tikzpicture}[scale=1]
			\draw (0.1,3) -- (2.9,3);
			\draw (-0.1,0) -- (2.9,0);
			\node at (2.5,0.2) {\small{$C$}};
			\draw (0.2,3.1) -- (0,1.9);
			\draw[dashed] (0,2.1) -- (0.2,0.9);
			\node at (0.3,1.5) {\small{$A_0$}};
			\draw (0.2,1.1) -- (0,-0.1);
			\node at (0.3,0.5) {\small{$L_1$}};
			\draw (1.2,3.1) -- (1,1.9);
			\node at (1.3,2.5) {\small{$L_2$}};
			\draw[dashed] (1,2.1) -- (1.2,0.9);
			\node at (1.3,1.5) {\small{$A_2$}};
			\draw (1.2,1.1) -- (1,-0.1);
			\draw (2.2,3.1) -- (2,1.9);
			\node at (2.3,2.5) {\small{$L_3$}};
			\draw[dashed] (2,2.1) -- (2.2,0.9);
			\node at (2.3,1.5) {\small{$A_3$}};
			\draw (2.2,1.1) -- (2,-0.1);
			\draw[dashed] (2.8,3.1) -- (3.1,1.4);
			\node at (3.1,2.5) {\small{$L_4$}};	
			\draw[dashed] (2.8,-0.1) -- (3.1,1.6);
			\node at (3.1,0.5) {\small{$A_4$}};			
		\end{tikzpicture}			
	}	
	\subcaptionbox{$2\rA_4$ \label{fig:2A4}}[.15\linewidth]{\centering
		\begin{tikzpicture}[scale=1]
			\draw (0,3) -- (1.2,3) to[out=0,in=120] (2,1.4);
			\draw (0,0) -- (1,0) to[out=0,in=-120] (2,1.6);
			\node at (1.75,0.4) {\small{$C$}};
			\draw[dashed] (0.2,3.1) -- (0,2.4);
			\node at (0.3,2.75) {\small{$A_0$}};
			\draw (0,2.55) -- (0.2,1.95);
			\node at (0.3,2.3) {\small{$L_1$}};
			\draw (0.2,2.05) -- (0,1.45);
			\draw (0,1.55) -- (0.2,0.95);
			\draw (0.2,1.05) -- (0,0.45);
			\draw[dashed] (0,0.55) -- (0.2,-0.1);
			\node at (0.3,0.3) {\small{$A_1$}};
			\draw (1.1,3.1) -- (0.9,1.9);
			\node at (1.2,2.5) {\small{$L_2$}};
			\draw[dashed] (0.9,2.1) -- (1.1,0.9);
			\node at (1.2,1.5) {\small{$A_2$}};
			\draw (1.1,1.1) -- (0.9,-0.1);	
		\end{tikzpicture}
	}
	\caption{Examples \ref{ex:ht=2}, \ref{ex:ht=2_meeting}: vertically primitive surfaces in case $\height=2$, $\width=2$.}
	\label{fig:ht=2_w=2_basic}
\end{figure}
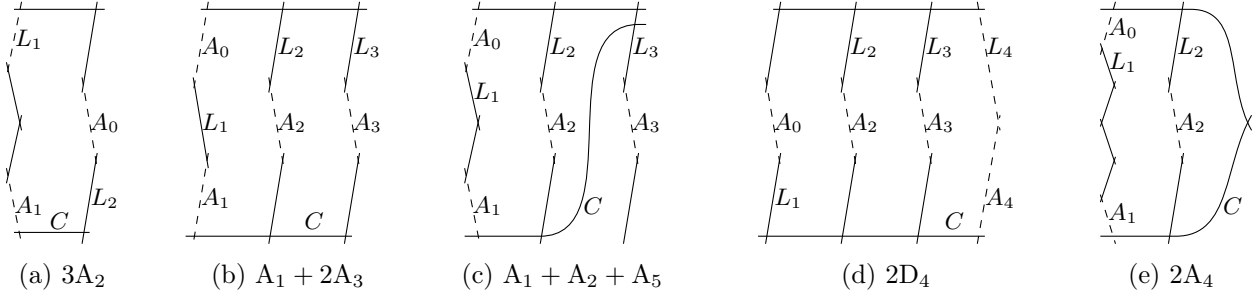

\begin{example}[Del Pezzo surface of rank one and type $2\rA_{4}$, see Figure \ref{fig:2A4}]\label{ex:ht=2_meeting}
	Let $\cc$ be a conic passing through $p_0$. Take $\nu=2$ and write $\{p_{j},p_0\}=\ll_{j}\cap \cc$. Blow up twice at $p_0$, twice at $p_2$, and four times at $p_1$, each time on the proper transform of $\ll_{1}$, $\ll_{2}$ and $\cc$; respectively. Then $\bar{Y}$ is of type $2\rA_{4}$, see Figure \ref{fig:2A4}. 
	As in Example \ref{ex:ht=2} above, we see that $\bar{Y}$ is unique up to an isomorphism. In Remark \ref{rem:primitive_ht=2} we will see that it is primitive. 
\end{example}

\begin{example}[Del Pezzo surfaces of rank one and types  $\rA_{3}+\rD_{5}$ or  $2\rA_{1}+2\rA_{3}$, see Fig.\  \ref{fig:ht=2_w=1-basic}]\label{ex:ht=2_twisted}
	Let $\cc$ be a conic, $p_{0}\not\in \cc$. Assume that $\ll_{2},\dots,\ll_{\nu}$ are tangent to $\cc$ and $\ll_{1}$ is not. Write $\{p_{1},p_{1}'\}=\ll_{1}\cap \cc$, $\{p_{j}\}=\ll_{j}\cap \cc$, $j\geq 2$. 
	
	\begin{parlist}
		\item\label{item:A3+D5_construction} For $\rA_{3}+\rD_{5}$: take $\nu=2$, blow up at $p_0$, twice over $p_1$ on the proper transforms of $\ll_1$, and five times over $p_{2}$, on the proper transforms of $\cc$, see Figure \ref{fig:D5+A3_twisted}.
		\item\label{item:2A1+2A3_construction} For $2\rA_{1}+2\rA_{3}$: take $\nu =3$ (so in this case $\cha \kk\neq 2$). Blow up at $p_0$, twice over $p_1$ on the proper transforms of $\ll_1$; and $j$ times over $p_{j}$, $j\in \{2,3\}$ on the proper transforms of $\cc$, see Figure \ref{fig:2A3+2A1}.
	\end{parlist}
	As before, projective uniqueness of each configuration shows that $\bar{Y}$ is unique up to an isomorphism. 
	
	In Remark \ref{rem:primitive_ht=2}, we will see that the surface from \ref{item:2A1+2A3_construction} is primitive. The one from \ref{item:A3+D5_construction} is not. Indeed, consider the $\P^1$-fibration of $Y$ given by the pencil of conics tangent to $\ll_i$ at $p_i$, $i\in \{1,2\}$. Let $A$ be the proper transform of the unique member of this pencil meeting $\cc$ at $p_1$ with multiplicity $3$. Then $A$ is a $(-1)$-curve, $A\cdot D_Y=1$ and $A$ meets the second component of the long twig of the fork in $D_Y$. Contracting $A$, wee see that $\bar{Y}$ swaps vertically, with respect to the new $\P^1$-fibration, to the surface of type $\rA_1+2\rA_3$ from Example \ref{ex:ht=2}\ref{item:A1+2A3_construction}, cf.\ Example \ref{ex:remaining_canonical}\ref{item:A1+2A3_tower}.

	Note that the existence of the above $\P^1$-fibration shows that $\width(\bar{Y})=2$, even though in Lemma \ref{lem:ht=2_twisted-swaps}\ref{item:ht=2_sep_nu=2} we will use $\bar{Y}$ as a vertically primitive model for log surfaces of height $2$ and width $1$.
\end{example}

\begin{figure}[ht]
	\subcaptionbox{$\rA_3+\rD_5$ \label{fig:D5+A3_twisted}}[.3\linewidth]{
		\begin{tikzpicture}[scale=1]
			\draw (0.1,3) -- (0.3,3) to[out=0,in=170] (1.2,1.56) to[out=-10,in=-135] (1.3,1.61);
			\draw (-0.1,0) -- (0.2,0) to[out=0,in=170] (1.2,1.56) to[out=-10,in=135] (1.3,1.49);
			\node at (0.8,0.5) {\small{$C$}};
			\draw (0.2,3.1) -- (0,1.9);
			\node at (0.3,2.5) {\small{$L_1$}};
			\draw[dashed] (0,2.1) -- (0.2,0.9);
			\node at (0.3,1.5) {\small{$A_1$}};
			\draw (0.2,1.1) -- (0,-0.1);
			\draw[dashed] (1,1.6) -- (2.1,1.4);
			\node at (1.65,1.65) {\small{$A_2$}};
			\draw (1.9,1.4) -- (3,1.6);
			\draw (2.8,1.6) -- (3.9,1.4);
			\draw (3.9,3.1) -- (3.7,1.9);
			\node at (4,2.5) {\small{$L_2$}};
			\draw (3.7,2.1) -- (3.9,0.9);
			\draw (3.9,1.1) -- (3.7,-0.1);
		\end{tikzpicture}
	}
	\subcaptionbox{$2\rA_1+2\rA_3$, $\cha\kk\neq 2$ \label{fig:2A3+2A1}}[.3\linewidth]{
		\begin{tikzpicture}[scale=1]
			\draw (0.2,3.1) -- (0,1.9);
			\node at (0.3,2.5) {\small{$L_2$}};
			\draw[dashed] (0,2.1) -- (0.2,0.9);
			\node at (-0.1,1.35) {\small{$A_2$}};
			\draw (0.2,1.1) -- (0,-0.1);
			\draw (1.2,3.1) -- (1,1.9);
			\node at (1.3,2.5) {\small{$L_1$}};
			\draw[dashed] (1,2.1) -- (1.2,0.9);
			\node at (1.3,1.5) {\small{$A_1$}};
			\draw (1.2,1.1) -- (1,-0.1);
			\draw[dashed] (1.7,1.6) -- (2.8,1.4);
			\node at (2.35,1.65) {\small{$A_3$}};
			\draw (2.8,3.1) -- (2.6,1.9);
			\node at (2.9,2.5) {\small{$L_3$}};
			\draw (2.6,2.1) -- (2.8,0.9);
			\draw (2.8,1.1) -- (2.6,-0.1);
			\draw (-0.05,1.55) to[out=0,in=180] (0.05,1.5) to[out=0,in=180] (1.1,2.9) -- (1.3,2.9) 
			to[out=0,in=170] (1.9,1.56) to[out=-10,in=-135] (2,1.61);
			\draw (-0.05,1.45) to[out=0,in=180] (0.05,1.5) to[out=0,in=180] (0.9,0.1) -- (1.1,0.1) 
			to[out=0,in=170] (1.9,1.56) to[out=-10,in=135] (2,1.51);
			\node at (1.6,0.5) {\small{$C$}};
		\end{tikzpicture}
	}
	\subcaptionbox{type \eqref{eq:KM_type}, $\cha\kk=2$ \label{fig:KM_surface}}[.3\linewidth]{
		\begin{tikzpicture}[scale=1]
			\draw[very thick] (0,1.2) -- (3.1,1.2);
			\node at (2.1,1.4) {\small{$C$}};
			\node at (2.1,1) {\small{$4-2\nu$}};
			\draw (0.2,2.5) -- (0,1.5);
			\node at (0.3,2) {\small{$L_1$}};
			\draw[dashed] (0,1.7) -- (0.2,0.7);
			\node at (0.3,1.4) {\small{$A_1$}};
			\draw (0.2,0.9) -- (0,-0.1);
			\draw (1.2,2.5) -- (1,1.5);
			\node at (1.3,2) {\small{$L_2$}};
			\draw[dashed] (1,1.7) -- (1.2,0.7);
			\node at (1.3,1.4) {\small{$A_2$}};
			\draw (1.2,0.9) -- (1,-0.1);
			\node at (2.1,2) {\Large{$\cdots$}};
			\draw (3,2.5) -- (2.8,1.5);
			\node at (3.1,2) {\small{$L_{\nu}$}};
			\draw[dashed] (2.8,1.7) -- (3,0.7);
			\node at (3.1,1.4) {\small{$A_{\nu}$}};
			\draw (3,0.9) -- (2.8,-0.1);
			\draw [decorate, decoration = {calligraphic brace}, thick] (0,2.6) -- (3.2,2.6);
			\node at (1.6,2.9) {\small{$\nu$ fibers}};			
		\end{tikzpicture}
	}
	\caption{Examples \ref{ex:ht=2_twisted}, \ref{ex:ht=2_twisted_cha=2}: vertically primitive surfaces in case $\height=2$, $\width=1$.}
	\label{fig:ht=2_w=1-basic}
\end{figure}
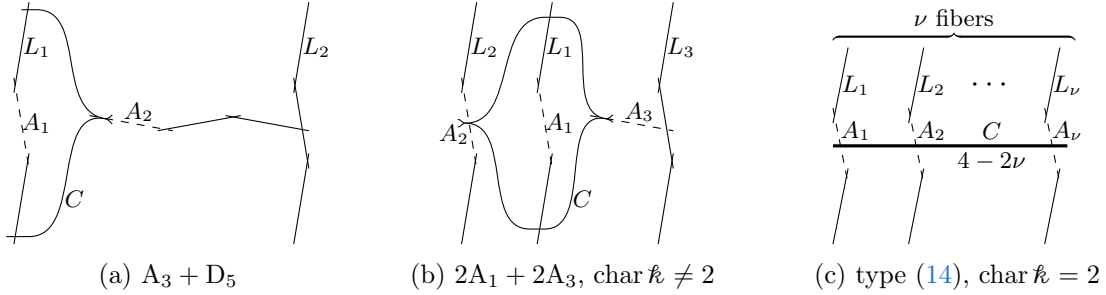
\vspace{-1em}
\begin{example}[{Additional class of primitive del Pezzo surfaces in case $\cha\kk=2$, cf.\ \cite[\S 9]{Keel-McKernan_rational_curves}}]\label{ex:ht=2_twisted_cha=2}
	Assume $\cha\kk=2$. Take $\nu\geq 3$, and let $\cc$ be a conic such that $(\ll_{j}\cdot \cc)_{p_{j}}=2$ for all $j$. Clearly, such a \emph{strange} conic exists if and only if $\cha\kk=2$. Blow up at $p_0$, and twice over each $p_{j}$, on the proper transforms of $\cc$, see Figure \ref{fig:KM_surface}. The singularity type of $\bar{Y}$, written together with the resulting $\P^1$-fibration using Notation \ref{not:fibrations}, is
	\begin{equation}\label{eq:KM_type}
		[\bs{2\nu-4}]\dec{1,\dots,\nu}+\sum_{j=1}^{\nu}([2]\dec{j}+[2]\dec{j}).
	\end{equation}
	
	For $\nu=3$ we get a canonical surface of type $7\rA_{1}$, unique up to an isomorphism. In general, the set of isomorphism classes of surfaces $\bar{Y}$ has moduli dimension $\nu-3$, with the representing family parametrized by the choice of $\ll_{4},\dots, \ll_{\nu}$; see Lemma \ref{lem:7A1-family}\ref{item:7A1-family-P} for details.
	
	We claim that $\bar{Y}$ is primitive. We argue as in Remark \ref{rem:ht=1-primitive}. Case $\nu=3$ is covered by Remark \ref{rem:primitive_ht=2}, so assume $\nu\geq 4$. Suppose  $\bar{Y}$ is not primitive. Then $Y$ contains a $(-1)$-curve $A$ meeting $D\vert$ once. 
	Every degenerate fiber $F$ is supported on a chain $[2,1,2]$, with a $(-1)$-curve of multiplicity $2$ and tips of multiplicity $1$. Hence $A\cdot F$ is odd for the fiber $F$ containing the point $A\cap D\vert$ and even for all the other degenerate fibers, a contradiction.
\end{example}

\begin{example}[See Figure \ref{fig:remaining_canonical}]\label{ex:remaining_canonical}
	We  now reconstruct remaining canonical del Pezzo surfaces of rank one and height $2$, and we show that the following hold, see Table \ref{table:canonical}.
	\begin{enumerate}
		\item \label{item:3A2_tower} Surfaces of types $\rA_{2}+\rA_{5}$, $\rA_{8}$, $\rA_{2}+\rE_{6}$ swap vertically to the one of type $3\rA_2$, constructed in \ref{ex:ht=2}\ref{item:3A2_construction}.
		\item \label{item:A1+2A3_tower} Surfaces of types $\rA_{3}+\rD_{5}$ $\rA_{1}+\rA_{7}$ swap vertically to the one of type $\rA_{1}+2\rA_{3}$, constructed in \ref{ex:ht=2}\ref{item:A1+2A3_construction}.
		\item \label{item:7A1_tower} Surfaces of type $4\rA_{1}+\rD_{4}$ swap vertically to the one of type $7\rA_{1}$, constructed in \ref{ex:ht=2_twisted_cha=2} for $\nu=3$. 
	\end{enumerate}
	Let $(Y,D_{Y})\to (\bar{Y},0)$ and $C,L_j,A_j\subseteq Y$ be as in one of the examples quoted above.
	\smallskip
	
	\ref{item:3A2_tower} Take $\bar{Y}$ of type $3\rA_2$, see Figure \ref{fig:3A2}. Choose points $q\in L_{1}\setminus D_{Y}$, $r\in A_{0}\setminus D_{Y}$, and some $q'$  infinitely near to $q$, not on the proper transform of $L_{1}$. Define a vertical swap $(X,D)\sqto (Y,D_Y)$ as a blowup either at $q$, or at $q$ and $r$, or at $q$ and $q'$; and let $X\to \bar{X}$ be the contraction of $D$. Then $\bar{X}$ is either of type $\rA_{2}+\rA_{5}$, or  $\rA_{8}$, or $\rA_{2}+\rE_{6}$, respectively, see Figures \ref{fig:A5+A2}, \ref{fig:A8}, or \ref{fig:E6+A2}. 
	Lemma \ref{lem:ht=1_uniqueness} in cases \ref{lem:ht=1_reduction}\ref{item:uniq_n=3} and \ref{item:uniq_2} implies that types  $\rA_{2}+\rA_{5}$ and $\rA_{8}$ are realized by a unique surface, and type $\rA_{2}+\rE_{6}$ by exactly two; moreover, the set $\cP(\rA_2+\rE_6)$ is represented by an $h^{1}$-stratified family. 
	
	One can also see this directly, as follows. Put $\bar{q}=\phi(q)\in \ll_1$.  Clearly, $(\P^2,\pp,\bar{q})$ is unique up to an isomorphism, hence so is $\bar{X}$ of type $\rA_{2}+\rA_{5}$. Fix coordinates on $\P^2$ such that $p_0=[0:1:0]$, $p_1=[0:0:1]$, $p_2=[1:0:0]$, $\bar{q}=[0:1:1]$. Then $\Aut(\P^2,\pp,\bar{q})=\{[x:y:z]\mapsto [x:y:\alpha z]: \alpha\in \kk^{*}\}\cong \G_{m}$ acts transitively on the set of irreducible members of the pencil $\{\lambda x^{3}=\mu y z^2\}_{[\lambda:\mu]\in \P^1}$: this action lifts to a transitive action on $A_1\setminus D_{Y}$. It follows that the surface $\bar{X}$ of type $\rA_8$ is unique up to an isomorphism. On the other hand, the action of $\Aut(\P^2,\pp,\bar{q})$ on the pencil of lines through $\bar{q}$ other than $\ll_1$ has two orbits: the open one, and the fixed line $\bar{\ll}$ joining $\bar{q}$ with $p_2$; and the derivative of this action does not vanish along the open orbit. Thus by Lemma \ref{lem:outer}\ref{item:outer-fixed-point} we have $\#\cP(\rA_2+\rE_6)=2$ and $\cP(\rA_{2}+\rE_{6})$ is represented by an $h^{1}$-stratified family, as claimed.

	We remark that by \cite[Proposition 1.5(b)]{PaPe_MT}, see Table \ref{table:canonical}, exactly one of the two surfaces of type $\rA_2+\rE_6$ has a singular member $\bar{T}\in |-K_{\bar{X}}|$ in its smooth locus. We claim that it is the general one if $\cha\kk\neq 3$, and the special one if $\cha\kk=3$. Here the \emph{general} and \emph{special} one is, respectively, the one corresponding to the open and closed orbit, i.e.\ to the general and special fiber of the $h^{1}$-stratified family.
	
	Indeed, let $T$ be the proper transform of $\bar{T}$ on $X$. We have $T\in |-K_{X}|$, so by adjunction $T$ meets each $(-1)$-curve once. It follows that the image of $T$ on $\P^2$, call it $\qq$, is a rational cubic such that $p_{0},p_{1},\bar{q}\in \qq\reg$; $(\qq\cdot \ll_{2})_{p_0}=(\qq\cdot \cc)_{p_1}=3$, and $\qq$ is tangent to $\bar{\ll}$ at $\bar{q}$ precisely in the special case. The Hurwitz formula or \cite[Lemma 5.5]{PaPe_MT} show that the latter holds exactly when $\cha\kk=3$, as claimed.
	\smallskip
	
	\ref{item:A1+2A3_tower} Take $\bar{Y}$ of type $\rA_{1}+2\rA_{3}$, see Figure \ref{fig:2A3+2A1}. Choose a point $q_{j}\in A_{j}^{\circ}\de A_{j}\setminus D_{Y}$, $j\in \{0,2\}$. A vertical swap $(X,D)\sqto (Y,D_Y)$ given by a blowup at $q_{0}$ or $q_{2}$, followed by the contraction of $D$, gives a surface $\bar{X}$ of type $\rA_{3}+\rD_{5}$ or $\rA_{1}+\rA_{7}$, respectively, see Figures \ref{fig:D5+A3}, \ref{fig:A7+A1}. Like before, Lemma \ref{lem:ht=1_uniqueness} in case \ref{lem:ht=1_reduction}\ref{item:uniq_n=3} implies that these surfaces are unique up to an isomorphism. To see this directly, note that the choice of $q_0,q_2$ is equivalent to a choice of a conic $\qq$ tangent to $\ll_{1}$ and $\cc$ at $p_0$, $p_2$; and the configuration $\pp+\qq$ is uniquely determined by four points $(p_0,p_1,p_2,\qq\cap \ll_3\setminus \{p_0\})$ in a general position.
	\smallskip
	
	\ref{item:7A1_tower} Assume $\cha\kk=2$ and take $\bar{Y}$ of type $7\rA_1$, see Figure \ref{fig:KM_surface}. Blowing up a point $q\in A_{1}\setminus D_{Y}$ gives a surface of type $4\rA_{1}+\rD_{4}$. Since the group $\Aut(Y,D_Y)$ acts on $A_{1}\setminus D_Y$ with finite orbits, Lemma \ref{lem:outer}\ref{item:outer-trivial} implies that the set of isomorphism classes of such surfaces has moduli dimension $1$, with the representing family parametrized by the choice of $q$, cf.\ \cite[Example 8.1(b)]{PaPe_MT} or Lemma \ref{lem:ht=2_twisted-inseparable}. 
\end{example}
\begin{figure}
	\subcaptionbox{$\rA_{2}+\rA_{5}$ \label{fig:A5+A2}}[.14\linewidth]{
		\begin{tikzpicture}[scale=0.9]
			\draw (0,3) -- (1.3,3);
			\draw[dashed] (-0.6,2.8) -- (0.2,2.6); 
			\draw (0.2,3.1) -- (0,2.1);
			\draw (0,2.3) -- (0.2,1.4);
			\draw (0.2,1.6) -- (0,0.7);
			\draw[dashed] (0,0.9) -- (0.2,-0.1);
			\draw (1.2,3.1) -- (1,1.9);
			\draw[dashed] (1,2.1) -- (1.2,0.9);
			\draw (1.2,1.1) -- (1,-0.1);
			\draw (0.1,0.05) -- (1.1,0.05);
		\end{tikzpicture}
	}
	\subcaptionbox{$\rA_8$ \label{fig:A8}}[.12\linewidth]{
		\begin{tikzpicture}[scale=0.9]
			\draw (0,3) -- (1.3,3);
			\draw[dashed] (-0.6,2.8) -- (0.2,2.6);
			\draw (0.2,3.1) -- (0,2.1);
			\draw (0,2.3) -- (0.2,1.4);
			\draw (0.2,1.6) -- (0,0.7);
			\draw[dashed] (0,0.9) -- (0.2,-0.1);
			\draw (1.2,3.1) -- (1,1.9);
			\draw (1,2.1) -- (1.2,0.9);
			\draw[dashed] (1,1.5) -- (1.8,1.7);
			\draw (1.2,1.1) -- (1,-0.1);
			\draw (0.1,0.05) -- (1.1,0.05);
		\end{tikzpicture}
	}	
	\subcaptionbox{$\rA_2+\rE_6$ \label{fig:E6+A2}}[.12\linewidth]{
		\begin{tikzpicture}[scale=0.9]
			\draw (0,3) -- (1.3,3);
			\draw (-0.6,2.8) -- (0.2,2.6);
			\draw[dashed] (-0.4,3.1) -- (-0.6,2.1);
			\draw (0.2,3.1) -- (0,2.1);
			\draw (0,2.3) -- (0.2,1.4);
			\draw (0.2,1.6) -- (0,0.7);
			\draw[dashed] (0,0.9) -- (0.2,-0.1);
			\draw (1.2,3.1) -- (1,1.9);
			\draw[dashed] (1,2.1) -- (1.2,0.9);
			\draw (1.2,1.1) -- (1,-0.1);
			\draw (0.1,0.05) -- (1.1,0.05);
		\end{tikzpicture}	
	}
	\subcaptionbox{$\rA_{3}+\rD_{5}$ \label{fig:D5+A3}}[.17\linewidth]{
		\begin{tikzpicture}[scale=0.9]
			\draw (0,3) -- (2.3,3);
			\draw (-0.1,0) -- (2.1,0);
			\draw[dashed] (-0.6,2.8) -- (0.2,2.6);
			\draw (0.2,3.1) -- (0,1.9);
			\draw (0,2.1) -- (0.2,0.9);
			\draw[dashed] (0.2,1.1) -- (0,-0.1);
			\draw (1.2,3.1) -- (1,1.9);
			\draw[dashed] (1,2.1) -- (1.2,0.9);
			\draw (1.2,1.1) -- (1,-0.1);
			\draw (2.2,3.1) -- (2,1.9);
			\draw[dashed] (2,2.1) -- (2.2,0.9);
			\draw (2.2,1.1) -- (2,-0.1);			
		\end{tikzpicture}	
	}
	\subcaptionbox{$\rA_1+\rA_7$ \label{fig:A7+A1}}[.18\linewidth]{
		\begin{tikzpicture}[scale=0.9]
			\draw (0.1,3) -- (2.3,3);
			\draw (-0.1,0) -- (2.1,0);
			\draw[dashed] (0.2,3.1) -- (0,1.9);
			\draw (0,2.1) -- (0.2,0.9);
			\draw[dashed] (0.2,1.1) -- (0,-0.1);
			\draw (1.2,3.1) -- (1,1.9);
			\draw[dashed] (1,2.1) -- (1.2,0.9);
			\draw (1.2,1.1) -- (1,-0.1);
			\draw (2.2,3.1) -- (2,1.9);
			\draw (2,2.1) -- (2.2,0.9);
			\draw[dashed] (2,1.5) -- (2.8,1.7);
			\draw (2.2,1.1) -- (2,-0.1);			
		\end{tikzpicture}
	}	
	\subcaptionbox{$4\rA_{1}+\rD_{4}$, $\cha\kk=2$ \label{fig:4A1+D4}}[.23\linewidth]{
		\begin{tikzpicture}[scale=0.9]
			\draw[very thick] (0,1.4) -- (2.3,1.4);
			\draw (0.2,3.1) -- (0,1.9);
			\draw[dashed] (0,2.1) -- (0.2,0.9);
			\draw (0.2,1.1) -- (0,-0.1);
			\draw (1.2,3.1) -- (1,1.9);
			\draw[dashed] (1,2.1) -- (1.2,0.9);
			\draw (1.2,1.1) -- (1,-0.1);
			\draw (2.2,3.1) -- (2,1.9);
			\draw (2,2.1) -- (2.2,0.9);
			\draw[dashed] (2,1.6) -- (2.8,1.8);
			\draw (2.2,1.1) -- (2,-0.1);			
		\end{tikzpicture}
	}	\vspace{-0.5em}
	\caption{Example \ref{ex:remaining_canonical}: canonical surfaces of height $2$ which are not vertically primitive.}
	\label{fig:remaining_canonical}\vspace{-0.5em}
\end{figure}
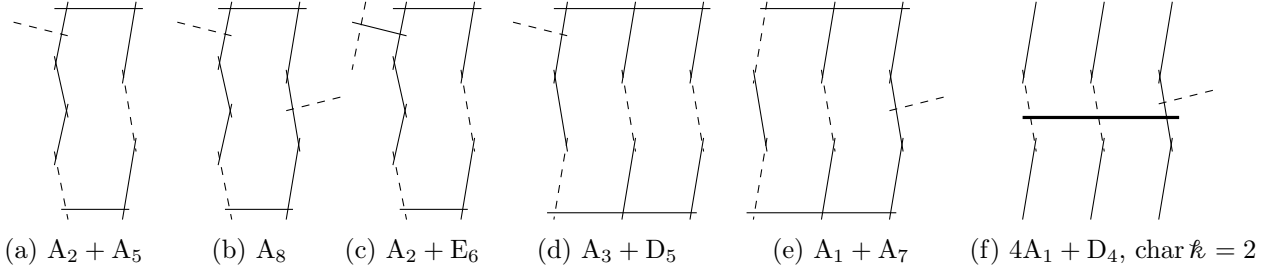

	\begin{remark}[Primitive canonical surfaces, see Table \ref{table:canonical}]\label{rem:primitive_ht=2}
		Recall that all singular, canonical del Pezzo surfaces of rank one are those listed in Table \ref{table:canonical}. For future reference, we note that the primitive ones are those of types: $\rA_1$, $\rA_1+\rA_2$, $2\rA_1+\rA_3$, $3\rA_1+\rD_4$ (see Remark \ref{rem:canonical_ht=1}); $3\rA_2$, $\rA_1+2\rA_3$, $\rA_1+\rA_2+\rA_5$, $2\rA_4$, $2\rA_1+2\rA_3$, $7\rA_1$ (see Figures \ref{fig:ht=2_w=2_basic} and \ref{fig:ht=2_w=1-basic}) and $4\rA_2$, $8\rA_1$,  see \cite[Examples 7.1, 8.1(c)]{PaPe_MT}.
		\end{remark}
		\begin{proof}
			We have seen in Remark \ref{rem:canonical_ht=1} and in the examples above that all other canonical surfaces are not primitive. It remains to prove that those listed in the statement are primitive. Suppose one of them, say $\bar{Y}$, is not. By Definition \ref{def:vertical_swap}\ref{item:def-swap-basic}, the minimal log resolution $(Y,D_Y)$ of $\bar{Y}$ contains a $(-1)$-curve $A$ such that $A\cdot D_Y=1$.

			To get a contradiction, one can use a description of all $(-1)$-curves on their minimal resolutions, given in \cite[Table 1]{BBD_canonical}. Nonetheless, we provide an easy direct argument. By Remark \ref{rem:canonical_ht=1} we can assume $\height(\bar{Y})\geq 2$. 
			
			Let $R$ be the connected component of $D_Y$ meeting $L$. Then $\#R\geq 4$, because $A+R$ is not non-positive definite. Recall that we have assumed $\height(\bar{Y})\geq 2$, so looking at Table \ref{table:canonical} we see that $\bar{Y}$ is of type $\rA_1+\rA_2+\rA_5$ or $2\rA_4$. In particular, $R$ is a chain, so since $A+R$ is not negative definite, $A$ does not meet a tip of $R$.

			Consider the $\P^1$-fibration of $Y$ shown in Figure \ref{fig:A5+A2+A1} or \ref{fig:2A4}. 
			In the first case, $R$ contains $(D_Y)\hor$; in the second case we can assume that the same holds by symmetry. 
			Let $F$ be a general fiber, let $\pi\colon Y\to \bar{Y}$ be the contraction of $D_Y$, and let $\bar{F}=\pi(F)$, $\bar{A}=\pi(A)$. Since $\bar{Y}$ is canonical, we have $\pi^{*}K_{\bar{Y}}=K_{Y}$, so by Noether's formula $K_{\bar{Y}}^2=K_{Y}^2=10-\rho(Y)=1$, and by adjunction $\bar{F}\cdot K_{\bar{Y}}=F\cdot K_{Y}=-2$. Since $\rho(\bar{Y})=1$, it follows that $\bar{F}=-2K_{\bar{Y}}$, so  
			$\bar{A}\cdot \bar{F}=-2\bar{A}\cdot K_{\bar{Y}}=-2A\cdot K_{Y}=2$. We compute that $\pi^{*}\bar{F}=F+2R-\ftip{R}-\ltip{R}$, so $\bar{A}\cdot \bar{F}=A\cdot F+2$, hence $A\cdot F=0$, i.e.\ $A$ is vertical, a contradiction, see 
			Figures \ref{fig:A5+A2+A1}, \ref{fig:2A4}.
		\end{proof}

\subsection{Case $\height(X,D)=\width(X,D)=2$}\label{sec:ht=2_untwisted}

Throughout this section, we assume that
\begin{equation}\label{eq:assumption_ht=2_width=2}
	\parbox{.9\textwidth}{
		$\bar{X}$ is a del Pezzo surface of rank one and height two; 
		$(X,D)\to (\bar{X},0)$ is the minimal log resolution; and
		$p\colon X\to \P^{1}$ is a  $\P^{1}$-fibration such that $D\hor$ consists of two $2$-sections.
	}
\end{equation}

\begin{lemma}[The structure of $p$]\label{lem:ht=2_basic}
	Let $H_{1},H_{2}$ be the $1$-sections in $D\hor$. Let $F_{1},\dots,F_{\nu}$ be the degenerate fibers.
	\begin{enumerate}
		\item\label{item:m,n>=2} We have $H_1=[n]$, $H_2=[m]$ for some $n,m\geq 2$. 
		\item\label{item:branching-sections} If both $H_{1}$ and $H_2$ are branching in $D$ then they lie in different connected components of $D$.
		\item\label{item:ht=2_sigma} The fiber $F_{1}$ has exactly two $(-1)$-curves, and each $F_{j}$ for $j\geq 1$ has exactly one.
		\item\label{item:ht=2_nu} Each fiber $F_{j}$ for $j\geq 2$ meets $D\hor$ in tips of $D\vert$. In particular, $\nu-1\leq \min\{\beta_{D}(H_{1}),\beta_{D}(H_{2})\}\leq 3$.
		\item\label{item:ht=2_s} Let $\phi$ be the contraction of $(-1)$-curves in $F_1$ and its images such that all $(-1)$-curves in $\phi_{*}F_1$ meet $\phi_{*}D\hor$, and $\phi$ is an isomorphism in some neighborhood of $D\hor$. Then $\phi_{*}F_1=[1,(2)_{s},1]$ for some $s\geq 0$.
	\end{enumerate}
\end{lemma}
\begin{proof}
	\ref{item:m,n>=2} The curves $H_1,H_2$ are rational by Lemma  \ref{lem:min_res}\ref{item:one-elliptic}. We have $H_{i}^{2}\leq -1$, because $D$ is negative definite, and $H_{i}^{2}\neq -1$ by Lemma \ref{lem:min_res}\ref{item:D-snc}.
	
	\ref{item:branching-sections} Suppose the contrary. Then $D$ contains a rational chain $T$ with $H_{1}$ and $H_{2}$ as tips, and four components $T_{i,j}\not\subseteq T$ such that $T_{i,j}$ meets $T$ only in $H_{i}$, for $i,j\in \{1,2\}$. Put $B=T+\sum_{i,j} T_{i,j}$. By \cite[3.2.5, 3.1.3]{Flips_and_abundance} we have $1\leq \cf_{B}(H_i)\leq \cf(H_i)$, so $2\leq \cf(H_1)+\cf(H_2)$, contrary to the inequality \eqref{eq:ld_phi_H}.
	
	\ref{item:ht=2_sigma} By Lemma \ref{lem:delPezzo_fibrations}\ref{item:-1_curves}, the $(-1)$-curves in $F_i$ are precisely those components of $F_i$ which do not lie in $D$. Thus \ref{item:ht=2_sigma} follows from Lemma \ref{lem:delPezzo_fibrations}\ref{item:Sigma}.
	
	\ref{item:ht=2_nu} Fix $j>1$. By \ref{item:ht=2_sigma}, the fiber $F_j$ contains a unique $(-1)$-curve. This $(-1)$-curve  has multiplicity at least $2$, so it does not meet $D\hor$. This proves the first inequality in \ref{item:ht=2_nu}, cf.\ Lemma \ref{lem:ht=1_basics}\ref{item:ht=1_nu}. The second one is obvious if $\bar{X}$ is log terminal; otherwise it follows from Proposition \ref{prop:non-lt-more}\ref{item:non-lt-B}.
	
	\ref{item:ht=2_s} If $\phi_{*}F_1=[0]$ then, since $\phi$ is an isomorphism near $D\hor$, the proper transform  $\phi^{-1}_{*}(\phi^{*}F_1)$ meets both components of $D\hor$, which is impossible by Lemma \ref{lem:delPezzo_fibrations}\ref{item:rivet}. Thus $\phi_{*}F_{1}$ contains a $(-1)$-curve. By the definition of $\phi$, every such curve  meets $\phi_{*}D\hor$, so it has multiplicity $1$ in $\phi_{*}F_{1}$. Hence $\phi_{*}F_{1}$ has at least $2$ such $(-1)$-curves, and part \ref{item:ht=2_sigma} shows that it has exactly two. The result follows,  since $\phi_{*}F_1$ contracts to a $0$-curve.
\end{proof}

Recall from \cite[7.2]{Fujita-noncomplete_surfaces} that a \emph{rivet} is a connected component of $D\vert$ meeting both $H_{1}$ and $H_{2}$.

\setcounter{claim}{0}
\begin{lemma}[Vertical swaps]\label{lem:ht=2_swaps}
	Let $(X,D)$ and $m,n,s,\nu,\phi$ be as in Lemma \ref{lem:ht=2_basic}. We can choose $p$ so that $\bar{X}$ swaps vertically to a canonical surface $\bar{Y}$ constructed in Examples \ref{ex:ht=2}--\ref{ex:ht=2_meeting}; and one of the following holds.
	\begin{enumerate}
		\item\label{item:ht=2_no-rivet} $H_{1}\cdot H_{2}=0$, 
		$s=m+n-\nu$, $\nu=2,3$ or $4$,  
		and $\bar{Y}$ is of type $3\rA_{2}$, $\rA_{1}+2\rA_{3}$ or $2\rD_{4}$, respectively, 
		\item\label{item:ht=2_rivet} $H_{1}\cdot H_{2}=0$, $s=m+n-2$, $\nu=3$, 
		and $\bar{Y}$ is of type $\rA_{1}+\rA_{2}+\rA_{5}$,
		\item\label{item:ht=2_meet} $H_{1}\cdot H_{2}=1$, $s=m+n$, $\nu=2$,  
		and $\bar{Y}$  is of type $2\rA_{4}$.
	\end{enumerate}
	Assume furthermore that $\bar{X}$ is log canonical. Then we can choose $p$ so that if case \ref{item:ht=2_no-rivet} holds and $R$ is a rivet then $R\subseteq F_2$, and either $\nu=3$ 
	or $\nu=2$, and the connected components of $D-R$ containing $H_1$, $H_2$ meet the same $(-1)$-curve in $F_1$.
\end{lemma}
\begin{proof}
	Let $R_0$ be a component of $D$ meeting both $H_{1}$ and $H_{2}$; put $R_0=0$ if there is no such. Since $D$ has no circular subdivisors, $R_0$ is unique. By Lemma \ref{lem:delPezzo_fibrations}\ref{item:rivet} $R_0\not\subseteq F_1$, so we can assume $R_0\subseteq F_2$. Put $\epsilon=\#R_0\in \{0,1\}$. 
	
	Let $\upsilon\colon X\to \F_{n}$ be the contraction of all vertical curves not meeting $H_{1}$. Then $\upsilon(H_1)$ is the negative section, and $\upsilon(H_{2})$ is a $1$-section such that $\upsilon(H_1)\cdot \upsilon(H_{2})=H_{1}\cdot H_{2}\in \{0,1\}$. Numerical properties of $\F_n$ yield $0=(\upsilon(H_2)-\upsilon(H_1))^2=\upsilon(H_2)^2-2(H_1\cdot H_2)-n$, so $\upsilon(H_{2})^{2}=n+2(H_{1}\cdot H_{2})$. Decomposing $\upsilon$ as a sequence of blowups, we see that those blowups which are not isomorphisms near the image of $H_2$ contract exactly $s+1$ components of $F_{1}$ and, if $\nu>1$, also  $1-\epsilon$ components of $F_2$ and one component of $F_j$ for each $j>2$. We get  $(s+1)+(1-\epsilon)+(\nu-2)=\upsilon(H_{2})^{2}-H_{2}^{2}=n+2(H_{1}\cdot H_{2})+m$, so
	\begin{equation}\label{eq:s}
		s=m+n-\nu+\epsilon+2H_{1}\cdot H_{2}.
	\end{equation}
	In particular, since $\nu \leq 4$ by  Lemma  \ref{lem:ht=2_basic}\ref{item:ht=2_nu}, we have  $s \geq m+n-4$.
	
	Let $\tau'$ be the contraction of $\Exc\phi$ and of all components of $F_{2}+\dots+F_{\nu}$ which do not meet $D\hor$. Then $\tau'_{*}H_{1}=[n]$, $\tau'_{*}H_{2}=[m]$. Since $s \geq m+n-4$, we can further contract $(-1)$-curves in $F_1$ and its images until the images of $H_{1}$ and $H_{2}$ are $(-2)$-curves. This way, we get a vertical swap $\tau\colon (X,D)\sqto (Y,D_Y)$. Put 
	$\bar{H}_{j}=\tau_{*}H_{j}$, 
	$\bar{F}_{j}=\tau_{*}(F_{j})\redd$. Then $\bar{H}_j=[2]$, $\bar{F}_{j}=[2,1,2]$ for $j>1$; and $\bar{F}_{1}=[1,(2)_{s'},1]$ for $s'=s-m-n+4$. By \eqref{eq:s}
	\begin{equation}\label{eq:s'}
		s'=4-\nu+\epsilon+2H_{1}\cdot H_{2}.
	\end{equation}
	Note that  $(Y,D_{Y})\to (\bar{Y},0)$ is a minimal log resolution of some canonical surface $\bar{Y}$. We have $\rho(\bar{Y})=\rho(\bar{X})=1$ and $\height(\bar{Y})\leq \height(\bar{X})=2$, so $\bar{Y}$ is del Pezzo by Lemma \ref{lem:delPezzo_criterion}. 
	We need to show that $\bar{Y}$ is as in Examples \ref{ex:ht=2}--\ref{ex:ht=2_meeting}.

	Assume $\nu\geq 3$. By Lemma \ref{lem:ht=2_basic}\ref{item:ht=1_nu} we have $\nu\leq 4$.  If $H_{1}\cdot H_{2}=1$ or $F_1$ contains a rivet then both $H_{1}$ and $H_{2}$ are branching in $D$, contrary to Lemma \ref{lem:ht=2_basic}\ref{item:branching-sections}. Hence $H_{1}\cdot H_{2}=0$, and a rivet, if exists, is contained in $F_2$. This way, we get \ref{item:branching-sections} if $\epsilon=1$ and \ref{item:ht=2_no-rivet} if $\epsilon=0$.

	Assume $\nu\leq 2$. Formula \eqref{eq:s'} gives $s'\geq 2$. Put $\bar{F}=\bar{H}_{1}+2\bar{F}_{1}\cp{1}+\bar{F}_{1}\cp{2}$, so $\bar{F}\redd=[2,1,2]$, and put $\tilde{F}=\tau^{*}\bar{F}$. The linear system  $|\tilde{F}|$ induces a $\P^{1}$-fibration $\tilde{p}\colon X\to \P^1$. The horizontal part of $D$ for $\tilde{p}$, call it $\tilde{D}\hor$, consists of disjoint $1$-sections. 
	Assume $H_{1}\cdot H_{2}=1$. If $\nu=2$ then \ref{item:ht=2_meet} holds, so assume $\nu=1$. Now, $\tilde{D}\hor$ consists of two disjoint $1$-sections: $H_{2}$ and $\tau^{-1}_{*}\bar{F}_{1}\cp{3}$. Thus replacing $p$ by $\tilde{p}$, we can assume $H_{1}\cdot H_{2}=0$. 
	
	We have $\nu=\#\tilde{D}\hor\geq \height(X,D)=2$, so $\nu=2$; and $\tilde{D}\hor$ consists of two $1$-sections. Since $s'\geq 2$, those sections are disjoint and do not meet the same component of $D$. Thus after replacing $p$ by $\tilde{p}$, we get case \ref{item:ht=2_no-rivet}.

	To prove the last statement, assume that $D$ has a rivet $R$. Since $D$ has no circular subdivisors, $R$ is unique. If $R\subseteq F_{1}$ then the $\P^1$-fibration $\tilde{p}$ defined above admits a rivet, too. This rivet lies in the fiber $\tilde{F}$, which has exactly one component off $D$. Thus replacing $p$ by $\tilde{p}$, we can assume $R\subseteq F_2$.
	
	Since $D$ has no circular subdivisors, the sections $H_1$ and $H_2$ lie in different connected components of $D-R$, call them $D_1$ and $D_2$. Put $V=\tau^{-1}_{*}(\bar{F}_1\wedge D_{Y})$. If $V\not\subseteq D_1+D_2$ then $\tilde{p}$ has no rivet, so we are done by replacing $p$ with $\tilde{p}$. Thus we can assume $V\subseteq D_1$. Let $\bar{A}$ be the tip of $\bar{F}_1$ which meets $\bar{H}_2$. Then since $F_1$ contains no rivet the preimage $\tau^{*}\bar{A}$ contains a $(-1)$-curve meeting both $D_1$ and $D_2$, as claimed.
\end{proof}

We are now ready to list singularity types of log canonical surfaces satisfying assumption \eqref{eq:assumption_ht=2_width=2}. 
We use notation summarized in Section \ref{sec:notation}. In particular, we encode our $\P^1$-fibration of $(X,D)$ using a decorated singularity type, introduced in Notation \ref{not:fibrations}. Singularity types without decorations are listed in Table \ref{table:ht=2_char-any}. 

\setcounter{claim}{0}
\begin{lemma}[Classification, see Table \ref{table:ht=2_char-any}]\label{lem:ht=2,untwisted}
	Let $\cS$ be a log canonical singularity type. Let $\Phtw(\cS)$ be the set of isomorphism classes of del Pezzo surfaces $\bar{X}$ of rank $1$, height $2$, width $2$, and singularity type $\cS$.
		
	Let $(X,D)$ be a minimal log resolution of a surface in $\Phtw(\cS)$, put $h^i= h^{i}(\lts{X}{D})$. Fix a $\P^1$-fibration of $X$ as in Lemma \ref{lem:ht=2_swaps}, so that $\bar{X}$ swaps vertically to a surface $\bar{Y}$ from Example \ref{ex:ht=2} or  \ref{ex:ht=2_meeting}. Let $\check{\cS}$ be the combinatorial type of the sum of $D$ and all vertical $(-1)$-curves. Then one of the following holds.
	\begin{enumerate}[itemsep=0.6em]
		\item\label{item:ht=2_3A2-unique}  $\bar{Y}$ is of type $3\rA_{2}$, see Example \ref{ex:ht=2}\ref{item:3A2_construction}; we have $\#\Phtw(\cS)=1$, and $\check{\cS}$ is one of the following:
		\begin{longlist}
			\item\label{item:tau=id_chains} $\ldec{2}[\lh{T_2},T^{*}]\dec{3}+\ldec{3}[T,\fh{T_1}]\dec{1}+\ldec{1}[T_1^{*},T_2^{*}]\dec{2}$,
			\item \label{item:V-chains_c=1} $\ldec{3}[T,\fh{T_1},r\dec{1},T_1^{*},T_{2}^{*}]\dec{2}+\ldec{2}[\lh{T_{2}},T^{*}]\dec{3}+[(2)_{r-2}]\dec{1}$, 
			\item \label{item:V-chains_c=1_rivet} $\ldec{2}[\lh{T_{2}},T^{*},r_2\dec{3},T,\fh{T_1},r_1\dec{1},T_1^{*},T_2^{*}]\dec{2}+[(2)_{r_1-2}]\dec{1}+[(2)_{r_2-2}]\dec{3}$,
			\smallskip
			
			\item \label{item:lc_rivet} $\langle r,\ldec{2}[2,2,2\dec{1},\bs{2},2],\ldec{2}[\bs{2},2],[2]\dec{3}\rangle +[(2)_{r-3},3]\dec{3}$, $r\geq 3$,
			\smallskip
			
			\item \label{item:[T_1,n]=[2,2]_r>2} 
			$\langle r; \ldec{3}[2,\bs{2}],\ldec{2}[(2)_{m}],[2]\dec{1}\rangle+\ldec{2}[\bs{m},2]\dec{3}+[(2)_{r-3},3]\dec{1}$, $m\in\{2,3,4,5\}$, $r\geq 3$,
			\item \label{item:[T_1,n]=[3,2]_r>2} 
			$\langle r; \ldec{3}[3,\bs{2}],\ldec{2}[2,2],[2]\dec{1}\rangle+\ldec{2}[\bs{2},2,2]\dec{3}+[(2)_{r-3},3]\dec{1}$,  $r\geq 3$,
			\item \label{item:lc_T=[2]} $\langle r,\ldec{3}[2,2,\bs{2}],\ldec{2}[2,2,2],[2]\dec{1}\rangle +\ldec{2}[\bs{3},3]\dec{3}+[(2)_{r-3},3]\dec{1}$, $r\geq 3$,
			\item \label{item:lc_T=[2,2]} $\langle r,\ldec{3}[2,\bs{2}],\ldec{2}[2,2],\ldec{1}[2,2]\rangle +\ldec{2}[\bs{2},2]\dec{3}+[(2)_{r-3},4]\dec{1}$, $r\geq 3$.
			\setcounter{foo}{\value{longlisti}}
		\end{longlist}
		\smallskip
		
		\noindent Moreover $h^{1}=1$ in case
			\ref{item:V-chains_c=1_rivet} with $\cha\kk|d([T,T_1])$ or \ref{item:lc_rivet} with $\cha\kk=3$; and $h^1=0$ otherwise.
		\myitem{(a')}\label{item:ht=2_3A2-2}  $\bar{Y}$ is of type $3\rA_{2}$, see Example \ref{ex:ht=2}\ref{item:3A2_construction}; we have $\#\Phtw(\cS)=2$, and $\check{\cS}$ is one of the following:
		\begin{longlist}\setcounter{longlisti}{\value{foo}}	
			\item \label{item:[T_1,n]=[2,2]_r=2} 
			$\langle 2; \ldec{3}[2,\bs{2}],\ldec{2}[(2)_{m}],[2]\dec{1}\rangle+\ldec{2}[\bs{m},2]\dec{3}$, $m\in\{2,3,4\}$ ,
			\item \label{item:[T_1,n]=[3,2]_r=2} 
			$\langle 2; \ldec{3}[3,\bs{2}],\ldec{2}[2,2],[2]\dec{1}\rangle+\ldec{2}[\bs{2},2,2]\dec{3}$.
			\setcounter{foo}{\value{longlisti}}
		\end{longlist}
		\smallskip
		
		\noindent Moreover, $\Phtw(\cS)$ is represented by an $h^{1}$-stratified family.
		\item\label{item:ht=2_A1+2A3}  $\bar{Y}$ is of type $\rA_{1}+2\rA_{3}$, see Example \ref{ex:ht=2}\ref{item:A1+2A3_construction}; we have $\#\Phtw(\cS)=1$, and $\check{\cS}$ is one of the following:
		\begin{longlist}\setcounter{longlisti}{\value{foo}}
			\item \label{item:c=2} $\ldec{3}[T_{1},\bs{n}\dec{1},T_{2}]\dec{4}+\ldec{4}[T_{2}^{*},\bs{m},T_{1}^{*}]\dec{3}+\ldec{1}[(2)_{m+n-3}]\dec{2}$, 
			\item\label{item:rivet_nu=3} $\ldec{3}[T_{2},\bs{m}\dec{1},T_{1},r\dec{4},T_{1}^{*},\bs{n}\dec{2},T_{2}^{*}]\dec{3}+\ldec{4}[(2)_{r-2}]+\ldec{1}[(2)_{m+n-3}]\dec{2}$, 
			\item\label{item:rivet_nu=3-fork-1}
			$\langle \bs{m};
			\ldec{3}[2,\bs{n}\dec{2},T_{1}^{*},r\dec{4},T_{1}],
			[2]\dec{1},[2]\dec{3}\rangle+
			\ldec{4}[(2)_{r-2}]+\ldec{1}[3,(2)_{m+n-4}]\dec{2}$,
			\item\label{item:rivet_nu=3-fork-2}
			$\langle \bs{m},
				\ldec{3}[2,\bs{2}\dec{2},2,2\dec{4},2]
				,\ldec{1}T\trp,[2]\dec{3}\rangle+\ldec{1}T^{*}*[(2)_{m-1}]\dec{2}$, $d(T)=3$, where $m\geq 3$ if $T=[2,2]$,
			\smallskip
			
			\item\label{item:tau=id_fork_chain} $\langle \bs{n},\ldec{1}T_{1},\ldec{2}T_{2}\trp,\ldec{3}T\trp \rangle+
			\ldec{1}[T_{1}^{*},\bs{m},T_{2}^{*}]\dec{2}+\ldec{3}T^{*}*[(2)_{n+m-3}]\dec{4}$
			\item\label{item:V-chains_c=2_[2]_T1=0}  
			$\langle \bs{n}; \ldec{2}T,\ldec{3}[2],\ldec{4}[2]\rangle+
			\ldec{3}[2,\bs{m}\dec{2},2]\dec{4}+[(2)_{r-2}]\dec{1}$, 
			where $T=[(2)_{m+n-3}]*[T^{*}_2,r\dec{1},T_2]$ or  $[(2)_{m+n-3},r]\dec{1}$, 
			\item\label{item:V-chains_c=2_[2,3]}  $\langle \bs{2}; \ldec{2}[2,3]\dec{1},\ldec{3}T,\ldec{4}[2]\rangle+\ldec{4}[2,\bs{2}\dec{2},T^{*}]\dec{3}+[2]\dec{1}$, $d(T)=3$, 
			\item\label{item:V-chains_c=2_-2_twig_long} 
			$\langle \bs{n}; \ldec{2}[(2)_{m+n-2}]\dec{1},\ldec{3}T,\ldec{4}[2]\rangle+\ldec{3}[2,\bs{m}\dec{2},T^{*}]\dec{4}$, where  $m+n\leq 7$, $d(T)=3$, and $m\leq 4$,
			\item\label{item:V-chains_c=2_-2_twig}  
			$\langle \bs{2}; \ldec{2}[2,2]\dec{1},\ldec{3}T,\ldec{4}[2]\rangle+\ldec{3}[2,\bs{2}\dec{2},T^{*}]\dec{4}$, $d(T)\in \{4,5\}$, 	
			\smallskip

			\item\label{item:tau=id_forks} $\langle \bs{n},\ldec{1}T_{1}\trp,\ldec{2}T_{2}\trp,\ldec{3}T_{3}\trp \rangle+\langle \bs{m};\ldec{1}T_{1}^{*},\ldec{2}T_{2}^{*},\ldec{4}T_{4}^{*}\rangle+\ldec{3}(T_{3}^{*}*[(2)_{n+m-3}]*T_{4})\dec{4}$ with one admissible fork,
			\item\label{item:V-chains_c=2_[2]}  
			$\langle \bs{n}; \ldec{2}T,\ldec{3}[2],\ldec{4}[2]\rangle+
			\langle \bs{m};\ldec{2}T_1\trp,\ldec{3}[2],\ldec{4}[2]\rangle+[(2)_{r-2}]\dec{1}$, \\ 
			where $T=T_1^{*}*[(2)_{m+n-3}]*[T^{*}_2,r\dec{1},T_2]$ or  $T_{1}^{*}*[(2)_{m+n-3},r]\dec{1}$, 
			\setcounter{foo}{\value{longlisti}}
		\end{longlist}
	\smallskip
	
	\noindent Moreover, $h^{1}=1$ in cases 
	\ref{item:rivet_nu=3}--\ref{item:rivet_nu=3-fork-2} with $\cha\kk |d(T_1)$ and 
	\ref{item:V-chains_c=2_[2]_T1=0}, \ref{item:V-chains_c=2_[2]} with $\cha\kk|d(T_2)$; and $h^1=0$ otherwise.	
		\item\label{item:ht=2_A1+A2+A5} $\bar{Y}$ is of type $\rA_{1}\! +\! \rA_{2}\! +\! \rA_{5}$, see Example \ref{ex:ht=2}\ref{item:A1+A2+A5_construction}; we have $\#\Phtw(\cS)=1$, $h^1=0$, and $\check{\cS}$ is one of the following: 
		\begin{longlist}\setcounter{longlisti}{\value{foo}}
			\item\label{item:R_0-non-branching} $\ldec{4}[T,\bs{n}\dec{1},r\dec{3},\bs{m},T^{*}]\dec{4}+[(2)_{r-1}]\dec{3}+\ldec{1}[(2)_{m+n-2}]\dec{2}$,
		\smallskip
			
			\item\label{item:R_0-lc} $\langle  r,\ldec{3}[2,2],\ldec{4}[2,\bs{2}]\dec{1},\ldec{4}[2,\bs{2}]\dec{2}\rangle+[(2)_{r-2},4]\dec{3}+\ldec{1}[2,2]\dec{2}, r\geq 3$,	
			\item\label{item:R_0_branching} $\langle r; \ldec{4}[2,\bs{2}]\dec{1},\ldec{4}[2,\bs{m}]\dec{2},[2]\dec{3}\rangle+[(2)_{r-2},3]\dec{3}+\ldec{1}[(2)_{m}]\dec{2}$, $m\in \{2,3\}$, $r\geq 2$, 		
			\setcounter{foo}{\value{longlisti}}
		\end{longlist}
		\item\label{item:ht=2_2A4} $\bar{Y}$ is of type $2\rA_{4}$, see Example \ref{ex:ht=2_meeting}; we have $\#\Phtw(\cS)=1$, $h^1=0$, and $\check{\cS}$ is one of the following:
		\begin{longlist}\setcounter{longlisti}{\value{foo}}
			\item\label{item:c=1_meeting} $\ldec{3}[T^{*},\bs{m}\dec{1},\bs{n}\dec{2},T]\dec{3}+\ldec{1}[(2)_{n+m}]\dec{2}$,
			\smallskip
			
			\item \label{item:2A4_[2]} $\langle \bs{n},[2]\dec{1},[2]\dec{3},\ldec{3}[2,\bs{m}]\dec{2}\rangle+\ldec{2}[(2)_{n+m-1},3]\dec{1}$,
			\item \label{item:2A4_m=3} $\langle \bs{n},\ldec{1}T,[2]\dec{3},\ldec{3}[2,\bs{3}]\dec{2}\rangle+\ldec{2}[(2)_{n+3}]*(T^{*})\dec{1}$, $d(T)=3$,
			\item \label{item:2A4_m=2} $\langle \bs{n},\ldec{1}T,[2]\dec{3},\ldec{3}[2,\bs{2}]\dec{2}\rangle+\ldec{2}[(2)_{n+2}]*(T^{*})\dec{1}$, $d(T)\in \{3,4,5,6\}$, if $T=[(2)_{5}]$ then $n\geq 3$, 
			\setcounter{foo}{\value{longlisti}}
		\end{longlist}
		\item\label{item:ht=2_2D4} $\bar{Y}$ is of type $2\rD_{4}$, see Example \ref{ex:ht=2}\ref{item:2D4_construction}; $\Phtw(\cS)$ has moduli dimension $1$, and $\check{\cS}$ is one of the following:
		\begin{longlist}\setcounter{longlisti}{\value{foo}}
			\item \label{item:c=3} $\langle \bs{n}\dec{1},\ldec{3}T_{1}\trp,\ldec{4}T_{2}\trp,\ldec{5}[2]\rangle+\langle \bs{m}\dec{2},\ldec{3}T_{1}^{*},\ldec{4}T_{2}^{*},\ldec{5}[2]\rangle+\ldec{1}[(2)_{m+n-4}]\dec{2}$, $d(T_{1})^{-1}+d(T_{2})^{-1}>\tfrac{1}{2}$,
			\item \label{item:ht=2_bench} $\ldecb{1}{3}\lbr \bs{n} \rbr\decb{4}{5}+\langle \bs{m}\dec{2},[2]\dec{3},[2]\dec{4},[2]\dec{5}\rangle+\ldec{1}[(2)_{m+n-5},3]\dec{2}$, $n\geq 3$, 
			\setcounter{foo}{\value{longlisti}}
		\end{longlist}
	\end{enumerate}
	where $r,r_1,r_2,n,m\geq 2$ are integers, and $T,T_1,T_2$ are admissible chains such all the above forks are admissible or log canonical. Moreover, we have $h^2=0$, and $h^0-h^1=2$ in case \ref{item:tau=id_chains}, $h^0-h^1=1$ in cases \ref{item:V-chains_c=1}, \ref{item:[T_1,n]=[2,2]_r>2}--\ref{item:lc_T=[2,2]}, \ref{item:c=2}, \ref{item:tau=id_fork_chain}, \ref{item:tau=id_forks}, and $h^0-h^1=0$ otherwise.
\end{lemma}
\begin{remark}[Redundancies in Lemma \ref{lem:ht=2,untwisted}]
		Interchanging the roles of chains $T,T_1,T_2$ in cases \ref{item:tau=id_chains}--\ref{item:V-chains_c=1_rivet} or \ref{item:c=2} can yield the same combinatorial type $\check{\cS}$, hence  isomorphic del Pezzo surfaces $\bar{X}$. For example, in case \ref{item:tau=id_chains} replacing $(T,T_1,T_2)$ with $(T_{1}\trp,T\trp,(T_{2}^{*})\trp)$ or $(T_{1}^{*},T_{2}^{*},T)$ we get the same $\check{\cS}$. In fact, in this case $D+\sum_{i}A_{i}$ is circular, and the above substitution corresponds to a reflection or a  rotation of its graph.
	
We also recall that by convention \eqref{eq:convention_2-1} type $[(2)_{0}]$ refers to a smooth point, and should be ignored. 
\end{remark}
\begin{proof}[Proof of Lemma \ref{lem:ht=2,untwisted}]
	We keep notation from Lemma \ref{lem:ht=2_basic} and choose $p$ as in Lemma \ref{lem:ht=2_swaps}. Then $D\hor=H_1+H_2$, $H_{1}=[n]$, $H_{2}=[m]$. Let $D_{j}$ be the connected component of $D$ containing $H_{j}$. First, we show that the structure of $p$ (hence the singularity type of $\bar{X}$) is as in \ref{item:ht=2_3A2-unique}--\ref{item:ht=2_2D4}. We study each case  \ref{lem:ht=2_swaps}\ref{item:ht=2_no-rivet}--\ref{lem:ht=2_swaps}\ref{item:ht=2_meet} separately.
	\begin{casesp*}
		\litem{\ref{lem:ht=2_swaps}\ref{item:ht=2_no-rivet}} Assume first that $D_1\neq D_2$. Then for $j\geq 2$, the fiber $F_{j}$  is columnar, say $(F_j)\redd=[T_{j-1},1,T_{j-1}^{*}]$. Assume $F_{1}\cap H_{j}\not\subseteq D$ for both $j\in \{1,2\}$. Then $F_{1}=[1,(2)_{m+n-\nu},1]$. If $\nu=2$ then $D$ is as in \ref{item:tau=id_chains} for $\#T_1=\#T_2=1$. If $\nu=3$ then $D$ is as in \ref{item:c=2}. If $\nu=4$ then $D_1,D_2$ are forks, and if neither is admissible then $\cf(H_1)+\cf(H_2)=2$, contrary to inequality \eqref{eq:ld_phi_H}. Thus $D$ is as in  \ref{item:c=3}.
	
	Assume now that $F_{1}\cap H_{1}\subseteq D$. Then $\nu=\beta_{D}(H_{1})\leq 4$. Consider the case $\nu=4$. Then $D_1$ is a bench, so $T_1,T_2,T_3=[2]$. It follows that $\beta_{D}(H_2)\geq 3$, and as before, inequality \eqref{eq:ld_phi_H} implies that $D_2$ is an admissible fork. In particular, $F_1$ meets $D_2$ in a $(-1)$-curve. Since $(F_1)\redd\wedge D_1=[2]$, we get $(F_1)\redd=[2,1,3,(2)_{s-1},1]$. Lemma \ref{lem:ht=2_swaps}\ref{item:ht=2_no-rivet} gives $s=m+n-5$, so $D$ is as in \ref{item:ht=2_bench}.
	
	Consider the case $\nu\leq 3$. Contract $(-1)$-curves in the subsequent images of $F_{1}$ until it becomes a chain  meeting the image of $D\hor$ in tips, and denote this morphism by $\tau$. Since $\tau_{*}F_1$ contracts to $[1,(2)_{s},1]$ in such a way that its tips  are not contracted, we have   $\tau_{*}F_1=[U_{1}',1,(U_{1}')^{*}]*[(2)_{n+m-\nu}]*[(U_{2}')^{*},1,U_{2}']$ for some chains $U_{1}'$, $U_{2}'$, possibly empty. Let $W_{j}'$ be the $(-1)$-curve meeting $U_{j}'$, and say that $\ftip{U_{1}'+W_1}$ meets $H_{1}$. The assumption $F_{1}\cap H_{1}\subseteq D$ gives $U_{1}'\neq 0$ or $W_{1}'\subseteq \tau_{*}D$. Since $\tau_{*}F_{1}$ is not a rivet, $W_{j}'\not\subseteq \tau_{*}D$ for some $j\in \{1,2\}$. Thus we can assume $W_{2}'\not\subseteq \tau_{*}D$, i.e.\ $\tau^{-1}$ has at most one base point and $\Bs\tau^{-1}\subseteq W_{1}'$. Put $U_{j}=\tau^{-1}_{*}U_{j}'$, $W_{j}=\tau^{-1}_{*}W_{j}'$ and $r=-W_{1}^{2}$. 
	
	Assume $\tau=\id$. Then $r=1$ and $U_1'\neq 0$. If $\nu=2$ then $D$ is as in \ref{item:tau=id_chains}.  If $\nu=3$ then $D$ is as in \ref{item:tau=id_fork_chain} if $U_{2}'=0$ and \ref{item:tau=id_forks} otherwise: in the latter case, inequality \eqref{eq:ld_phi_H} implies that one of the forks is admissible.
	
	Assume $\tau\neq\id$, so $r\geq 2$. Consider the case when $W_{1}$ is branching in $D_{1}$. Since $D_{1}$ has at most one branching component, we infer that $\nu=2$ and $D_{1}=\langle r,[T_{1},\bs{n},U_{1}],(U_{1}^{*}*[(2)_{n+m-2}]*U_{2}^{*})\trp,T\rangle$; $\Exc \tau=[(2)_{r-2}]*[T^{*},1,T]$.  Suppose $U_{1}\neq 0$. Then the second twig of $D_{1}$ is not a $(-2)$-twig; and has length at least $n+m-2\geq 2$, so by Lemma \ref{lem:admissible_forks}, the first twig of $D_{1}$, namely $[T_{1},n,U_1]$, has length at most two; a contradiction. Thus $U_{1}=0$, so $D_{1}=\langle r;T_1',T_2',T\rangle$; where $T_1'=[T_{1},\bs{n}]$, $T_2'=([(2)_{n+m-2}]*U_{2}^{*})\trp$; and $D_2=[U_{2},\bs{m},T_{1}^{*}]$. 
	We can assume $d(T_1')\geq d(T_2')$: indeed, replacing $p$ by the $\P^1$-fibration induced by $(F_1)\redd-W_1-T+H_2=[(2)_{m}]*[U_{2}^{*},1,U_{2},m]$ interchanges $T_1'$ with $T_{2}'$. Thus $d(T_{2}')\leq 4$ by Lemma \ref{lem:admissible_forks}. If the equality holds then $D_1=\langle r,[\bs{2},2,2],[2,2,2],[2]\rangle$, so $r\geq 3$, $n=2$, $m=3$, $U_2=0$ and $D$ is as in \ref{item:lc_T=[2]}. Assume $d(T_{2}')\leq 3$. Since $\#T_{2}'\geq 2$, we get $T_{2}'=[2,2]$, that is,  $m=n=2$, $U_{2}=0$. If $T\neq [2]$ then by Lemma \ref{lem:admissible_forks} $D_1=\langle r,[2,2],[2,2],[2,2]\rangle$ and $r\geq 3$, so $D$ is as in \ref{item:lc_T=[2,2]}. Assume $T=[2]$. Now $T_1'=[2,\bs{2}]$ or $[3,\bs{2}]$. In the first case; $D$ is as in \ref{item:[T_1,n]=[2,2]_r=2} if $r=2$ and \ref{item:[T_1,n]=[2,2]_r>2} otherwise. In the second case $D$ is as in \ref{item:[T_1,n]=[3,2]_r=2} or \ref{item:[T_1,n]=[3,2]_r>2}.
	
	Consider the case when $W_{1}\subseteq D$ is non-branching, so $\Exc\tau=[(2)_{r-2},1]$. If $\nu=2$ then putting  $T_{1}=[\bs{n},U_{1}]$, $T_2=[U_2,\bs{m}]$, we see that $D$ is as in \ref{item:V-chains_c=1}. 
	Assume $\nu=3$. Then $D_{1}=\langle \bs{n}; U,T_{1},T_{2}\rangle$, where  $U=U_{2}^{*}*[(2)_{n+m-3}]*[U_{1}^{*},r,U_{1}]$. If $T_{1}=T_{2}=[2]$ then $D$ is as in \ref{item:V-chains_c=2_[2]}. 
	Assume $T_{1}\neq [2]$. If $T_{2}\neq [2]$ then since $\#U\geq n+m-2\geq 2$, Lemma \ref{lem:admissible_forks} implies that $U=[2,2]$ and $n\geq 3$, which is impossible. Thus $T_2=[2]$.  If $U$ is a $(-2)$-twig then $r=2$, $U_{1}=U_{2}=0$, so $U=[(2)_{n+m-2}]$, and Lemma \ref{lem:admissible_forks} implies that $n+m-2\leq 5$, so $D$ is as in \ref{item:V-chains_c=2_-2_twig_long} if $d(T)=3$ and \ref{item:V-chains_c=2_-2_twig} otherwise. Assume that $U$ is not a $(-2)$-twig. Then $\#U=2$, so again $U_{1}=U_{2}=0$; and $U=[2,r]=[2,3]$. Thus $D$ is as in \ref{item:V-chains_c=2_[2,3]}. 
	
	Assume now that $D_1=D_2$, i.e.\ $D$ has a rivet, say $R$. Then by the second part of Lemma \ref{lem:ht=2_swaps} we can assume $R\subseteq F_2$, so $(F_{2})\redd=\langle r;T_2\trp,T_{2}^{*},[(2)_{r-2}]*[T^{*},1,T]\rangle$, where $\ltip{T_2}$ meets $H_1$. 
	
	Consider the case $\nu=3$. Then $(F_3)\redd=[T_3,1,T_3^{*}]$ is columnar. If $T\neq 0$ then $H_{1}$, $H_{2}$ lie in different twigs of $D_1$, both of length at least $3$, so by Lemma \ref{lem:admissible_forks} both these twigs have length $3$ and $r\geq 3$, which is impossible. Thus $T=0$. If $\phi=\id$ then $F_1=[1,(2)_{m+n-2},1]$ and $D$ is as in \ref{item:rivet_nu=3}. Assume $\phi\neq \id$. Then one of the sections, say $H_{1}$, is branching in $D_1$. Let $W$ be the twig of $D_{1}$ containing $(F_{2})\redd\wedge D_{1}$. Then $W=[T_{3}^{*},\bs{m},T_{2}^{*},r,T_{2}]$, so $\#W\geq 5$, and if $T_{3}\neq [2]$ then $W$ is not a $(-2)$-twig, contrary to Lemma \ref{lem:admissible_forks}. Thus $T_{3}=[2]$. Suppose that some blowup in the decomposition of $\phi^{-1}$ is centered at a smooth point of $\phi_{*}D$. Then $\#T_{1}=s+1=m+n-2\geq 2$, so by Lemma,  \ref{lem:admissible_forks} $D_{1}=\langle \bs{n};[2,2],[2,2,2,\bs{2},2],[2]\rangle$. In particular, $s=1$, $m=2$, and therefore $n=2$, so $D_{1}$ is not negative definite, a contradiction. Therefore, $F_{1}=[T_{1}\trp,1,T_{1}^{*}]*[(2)_{s},1]$. If $T_{1}=[2]$ then $D$ is as in \ref{item:rivet_nu=3-fork-1}. Assume $T_{1}\neq [2]$. Then by Lemma \ref{lem:admissible_forks} we have $W=[(2)_{5}]$ and $D$ is as in \ref{item:rivet_nu=3-fork-2}.

	Consider the case $\nu=2$. Define a vertical swap $\sigma\colon (X,D)\sqto (X',D')$ as the contraction of $(F_{2})\redd-R$. Then $D'$ becomes as in \ref{lem:ht=2_swaps}\ref{item:ht=2_no-rivet} with no rivet, $\nu=2$; and such that $D_1'\cap F_1'$, $D_2'\cap F_1'$ meet the same $(-1)$-curve in $F_1'$, where $F_{i}'=\sigma_{*}F_i$, and $D_{i}'$ is the connected component of $D'$ containing $\sigma_{*}H_{i}$. Hence $D'$ is as in   \ref{item:V-chains_c=1} or \ref{item:[T_1,n]=[2,2]_r>2}--\ref{item:[T_1,n]=[3,2]_r=2}. Let $A_3'$ be the $(-1)$-curve denoted there by decoration $\dec{3}$. Then $D_{1}$ is obtained from $D_{1}'+D_{2}'+A_{3}'$ by replacing $A_{3}'$ with a component $V$ and possibly adding a twig $W$ meeting $V$. In cases \ref{item:[T_1,n]=[2,2]_r>2}--\ref{item:[T_1,n]=[3,2]_r=2} we must have $W=0$ since $D_1$ has only one branching component, but even then $D_1$ is neither admissible nor log canonical; a contradiction. Thus $D'$ is as in \ref{item:V-chains_c=1}. If $W\neq 0$ then $D_1=\langle v;[T,\fh{T_1},r,T_1^{*},T_{2}^{*}]\trp, [\lh{T_2},T^{*}] ;W\rangle$, so  by Lemma \ref{lem:admissible_forks}\ref{item:long-twig} $D$ is as in \ref{item:lc_rivet}. If $W=0$ then $D_1$ is a chain, and $D$ is as in \ref{item:V-chains_c=1_rivet}.
	\litem{\ref{lem:ht=2_swaps}\ref{item:ht=2_rivet}}
	One of the fibers $F_2$, $F_3$ has a component $R_0\subseteq D$ meeting both $H_{1}$ and $H_{2}$. Say that $R_0\subseteq F_2$ and write $R_0=[r]$. Since $D$ has no circular subdivisors, the fiber $F_3$ is columnar, i.e.\ $(F_3)\redd=[T,1,T^{*}]$.
	
	Consider the case $\beta_{D}(R_{0})=2$, so $F_{2}=[r,1,(r)_{r-1}]$. If $(F_{1})\redd\wedge D\vert$ is disjoint from $D_1$ then $D$ is as in \ref{item:R_0-non-branching}. So we can assume $F_1\cap H_1\subseteq D$. Then $\beta_{D}(H_{1})=\nu=3$, so since $H_1,H_2$ lie in the same connected component of $D$, we have $\beta_{D}(H_{2})=2$, that is, $F_{1}\cap H_{2}\not\subseteq D$. It follows that $D_{1}=\langle  \bs{n};U',[T^{*},\bs{m},r],T\trp \rangle$, where $U'=[(2)_{m+n-2}]*[U^{*},r_1,U]\subseteq F_1$ for some $U$, possibly zero. The first two twigs have length at least $3$ each, so by Lemma \ref{lem:admissible_forks} they are both $(-2)$-twigs and $T=[2]$. Thus $n=m=r=2$, so $D_1$ is a $(-2)$-fork which is not negative definite; a contradiction.
	
	Consider the case $\beta_{D}(R_{0})=3$, so $F_2=[r,U,1,U^{*}]*[(2)_{r-1}]$ for some twig $U$ of $D_1$. The remaining two twigs of $D_1$ contain $H_{1}$ and $H_{2}$. In particular, $\beta_{D}(H_{j})=2$, $j\in \{1,2\}$, so  $(F_{1})\redd\wedge D\vert$ is disjoint from $D_{1}$. Now $D_{1}=\langle r;[T\trp,\bs{n}],[T^{*},\bs{m}];U\rangle$. By symmetry, we can assume $m\geq n$. If $U\neq [2]$ then, since the first two twigs have length at least $2$, by Lemma \ref{lem:admissible_forks} $m=n=2$, $T=[2]$ and $D$ is as in \ref{item:R_0-lc}. Assume $U=[2]$. Then by Lemma \ref{lem:admissible_forks} $T=[2]$, $n=2$ and $m\in \{2,3\}$. Hence $D$ is as in \ref{item:R_0_branching}.
	\litem{\ref{lem:ht=2_swaps}\ref{item:ht=2_meet}} We argue as in case~\ref{lem:ht=2_swaps}\ref{item:ht=2_rivet} above, $\beta_{D}(R_0)=2$. The fact that $D$ has no circular subdivisors implies that $F_2$ is columnar, i.e.\ $(F_{2})\redd=[T,1,T^{*}]$. If $(F_{1})\redd\wedge D\vert$ is disjoint from $D_1$, then $D$ is as in \ref{item:c=1_meeting}. Assume that $F_{1}\cap H_{1}\subseteq D\vert$. Then $F_{1}\cap H_{2}\not\subseteq D\vert$. 
	
	Assume further that $(F_{1})\redd$ is a chain. Then $(F_{1})\redd=[T_1,1,T_1^{*}]*[(2)_{n+m},1]$ and $D_1=\langle \bs{n},T_1\trp,T\trp,[T^{*},\bs{m}]\rangle$. We have $1\geq \frac{1}{d(T_1)}+\frac{1}{d(T)}+\frac{1}{d([T^{*},m])}\geq \frac{1}{2}+\frac{2}{d(T)}$, so $d(T)=2$, i.e.\ $T=[2]$. If $T_1=[2]$ then $D$ is as in \ref{item:2A4_[2]}. Assume $T_1\neq [2]$. If $m>2$ then $m=3$, $d(T_1)=3$ and $D$ is as in \ref{item:2A4_m=3}. If $m=2$ then $d(T_1)\leq 6$ and if the equality holds then $D_1$ is not a $(-2)$-fork, so $D$ is as in  \ref{item:2A4_m=2}. 
	
	Assume now that $(F_1)\redd$ is a fork. Then $D_{1}=\langle \bs{n};U',T^{*},[T,\bs{m}]\rangle$, where $U'=[(2)_{m+n}]*[U^{*},r_1,U]\subseteq F_1$ for some $U$, possibly empty. Hence $\#U'\geq m+n+1\geq 5$, $\#[T,\bs{m}]\geq 2$, so by Lemma \ref{lem:admissible_forks} both equalities hold, $n=m=2$ and $D_1$ is a $(-2)$-fork. Such a fork is not negative definite, a contradiction.
\end{casesp*}

To complete the proof, we compute the numbers $\#\Phtw(\cS)$, $h^0$, $h^1$, and construct the representing families. We argue as in the proof of Corollary \ref{cor:moduli-ht=1}. Let $\Phtwres(\cS)$ be the set of minimal log resolutions of surfaces in $\Phtw(\cS)$. For $(X,D)\in \Phtwres(\cS)$ let $\check{D}$ be the sum of $D$ and all vertical $(-1)$-curves, and let $\check{\cS}$ be the combinatorial type of $(X,\check{D})$. Lemma gives \ref{lem:h1}\ref{item:h1-1_curve} $h^i=h^{i}(\lts{X}{\check{D}})$. Since the singularity type $\cS$ appears exactly once on our list, it uniquely determines the type $\check{\cS}$, so $\Phtwres(\cS)$ is the image of $\cP_{+}(\check{\cS})$ in $\cP(\cS)$.
	
\begin{casesp*}
	\litem{\ref{item:ht=2_3A2-unique}, \ref{item:ht=2_A1+2A3}, \ref{item:ht=2_A1+A2+A5}} 
		Let 
		$(X,\check{D})\to (Z,D_Z)$ be a vertical inner snc-minimalization, and let $\cZ$ be the combinatorial type of $(Z,D_Z)$. By Lemma \ref{lem:inner} we have $\#\cP_{+}(\check{\cS})=\#\cP_{+}(\cZ)$ and $h^i=h^i(\lts{Z}{D_Z})$. These numbers are  computed in Lemma \ref{lem:ht=1_uniqueness}, see Table \ref{table:exceptions}. Indeed, the type $\cZ$ is as in:
		\begin{itemize}
		\item 
		\ref{lem:ht=1_reduction}\ref{item:uniq_easy} with $v=2$ in case \ref{item:tau=id_chains} and with $v=3$ in cases \ref{item:c=2}, \ref{item:tau=id_fork_chain}, \ref{item:tau=id_forks}; 
		\item 
		\ref{lem:ht=1_reduction}\ref{item:uniq_n=3} with $v=1$ in cases \ref{item:V-chains_c=1}, \ref{item:[T_1,n]=[2,2]_r>2}--\ref{item:lc_T=[2,2]}, 
		\item
		\ref{lem:ht=1_reduction}\ref{item:uniq_n=3} with $v=2$ in cases \ref{item:rivet_nu=3}--\ref{item:rivet_nu=3-fork-2}, \ref{item:V-chains_c=2_[2]_T1=0}--\ref{item:V-chains_c=2_-2_twig}, \ref{item:V-chains_c=2_[2]} and \ref{item:ht=2_A1+A2+A5}; the number $d(T)$ from \ref{lem:ht=1_reduction}\ref{item:uniq_n=3} equals $d(T_1)$ in cases~\ref{item:rivet_nu=3}--\ref{item:rivet_nu=3-fork-2}, $d(T_2)$ in cases \ref{item:V-chains_c=2_[2]_T1=0}, \ref{item:V-chains_c=2_[2]} and $1$ otherwise,
		\item 
		\ref{lem:ht=1_reduction}\ref{item:uniq_n=2} in cases \ref{item:V-chains_c=1_rivet}, \ref{item:lc_rivet}; the number $d(T_{1}^{*},m,T_{2}^{*})$ from \ref{lem:ht=1_reduction}\ref{item:uniq_n=2} equals $d([T,T_1])$ in case \ref{item:V-chains_c=1_rivet} and $3$ in case \ref{item:lc_rivet}.
	\end{itemize}

	Therefore, $\#\cP_{+}(\check{\cS})=1$ and $h^1$ is as in the statement. We have $h^2=0$ by Lemma \ref{lem:h1}\ref{item:h1-h2-fibration}, so $\chi\de \chi(\lts{Z}{D_Z})=h^0-h^1$. We have $\chi=\bar{\chi}-\epsilon$, where $\bar{\chi}=\chi(\lts{\F_{m}}{B})$, $\tau\colon (Z,D_Z)\to (\F_r,B)$ contracts all degenerate fibers to $0$-curves, and $\epsilon$ is the number of outer blowups within $\tau$. We have $\epsilon=0,1,2$ if $\cZ$ is as in \ref{lem:ht=1_reduction}\ref{item:uniq_easy}, \ref{lem:ht=1_reduction}\ref{item:uniq_n=3} and 	\ref{lem:ht=1_reduction}\ref{item:uniq_n=2}, respectively. To compute $\bar{\chi}$ note that $B$ is the sum of $2$ disjoint sections and $\nu$ fibers. Applying elementary transformations we can assume $r=0$, and we get $\bar{\chi}=4-\nu$. We conclude that $\chi$ equals $4-v$ in case \ref{lem:ht=1_reduction}\ref{item:uniq_easy}, $2-v$ in case \ref{lem:ht=1_reduction}\ref{item:uniq_n=3} and $0$ in case \ref{lem:ht=1_reduction}\ref{item:uniq_n=2}, as needed. 
	\litem{\ref{item:ht=2_3A2-2}} By Lemma \ref{lem:adding-1}\ref{item:adding-1-h1} it is enough to prove that $\#\cP_{+}(\check{\cS})=2$ and $\cP_{+}(\check{\cS})$ is represented by an $h^{1}$-stratified family. As before, let $(X,\check{D})\to (Z,D_Z)$ be a vertical inner snc-minimalization of $(X,\check{D})\in \cP_{+}(\check{\cS})$. Then $(Z,D_Z)$ is as in Lemma \ref{lem:ht=1_reduction}\ref{item:uniq_2}, so the result follows from Lemmas \ref{lem:ht=1_uniqueness} and \ref{lem:inner} as before.
	\litem{\ref{item:ht=2_2A4}} Let $X\to Y$ be the morphism onto the surface from Example \ref{ex:ht=2_meeting}, and let $Y\to \P^1\times \P^1$ be a further contraction of vertical curves such that the image $\pp$ of $\check{D}$ is snc, and consists of two fibers, one horizontal line and a diagonal. Let $\cZ$ be the combinatorial type of $(\P^1\times \P^1,\pp)$. We have $\#\cP_{+}(\cZ)=1$ and $h^{i}(\lts{\P^1\times \P^1}{\pp})=0$ by Lemma \ref{lem:h1}\ref{item:h1-diagonal}. 
	We check that the morphism $(X,\check{D})\to (\P^1\times \P^1,\pp)$ is inner, so by Lemmas \ref{lem:inner} and \ref{lem:adding-1}\ref{item:adding-1-h1} we get $\#\cP_{+}(\check{\cS})=\#\Phtwres(\cS)=1$ and $h^i=0$, as needed.
	\litem{\ref{item:ht=2_2D4}} The group $\Aut(\check{\cS})$ is a subgroup of $S_3\times \Z/2$, where $S_3$ permutes the fibers $F_2,F_3,F_4$ with one $(-1)$-curve each; and $\Z/2$ interchanges the sections $H_1,H_2$. Put $G=S_3\times \{\id\}\subseteq \Aut(\check{\cS})$. Let $\psi\colon (X,\check{D})\map (\P^1\times \P^1,\pp)$ be the contraction of all vertical curves disjoint from $H_1$, followed by some elementary transformations on the image of $F_1$. Clearly, $\psi$ is $G$-equivariant, and we check directly that $\psi$ is inner. The image $\pp$ of $\check{D}$ is the sum of two horizontal and four vertical fibers. Let $\cZ$ be the combinatorial type of $(\P^1\times \P^1,\pp)$. We have seen in Examples \ref{ex:4-points} and \ref{ex:4-points-Aut} that $\cP_{+}(\cZ)$ is represented by an $\Aut(\cZ)$-faithful universal family over $B=\P^{1}\setminus \{0,1,\infty\}$, parametrized by the cross-ratio of the images of $F_1,\dots, F_{4}$. By Lemma \ref{lem:inner} $\cP_{+}(\check{\cS})$ is represented by a universal $G$-faithful family over the same base, and $h^i=h^{i}(\lts{\P^1\times \P^1}{\pp})=1$ for $i=0,1$.
	
	It remains to prove that this family, viewed as one representing $\Phtwres(\cS)$, is almost faithful. Assume that we have an isomorphism $\phi\colon (X_{b_{1}},\check{D}_{b_{1}})\to (X_{b_{2}},\check{D}_{b_{2}})$ between two fibers. We check directly that any element of $\Aut(\cS)$ preserves the set of vertices corresponding to horizontal components, so $\phi$ maps the quadruple $\{F_{j,b_1}\cap H_{1,b_1}:j=1,2,3,4\}$ to $\{F_{j,b_2}\cap H_{i,b_2}:j=1,2,3,4\}$ for some $i\in \{1,2\}$. Hence their cross-ratios are equivalent by the $S_3$-action, which means that  $b_1$ and $b_2$ lie in the same $G$-orbit, as needed. 
\qedhere\end{casesp*}
\end{proof}

\subsection{Case $\height(X,D)=2$, $\width(X,D)=1$}\label{sec:ht=2_twisted}

To prove Theorem \ref{thm:ht=1,2}, it remains to work in the following setting.
\begin{equation}\label{eq:assumption_ht=2_width=1}
	\parbox{.9\textwidth}{
		$(X,D)\to (\bar{X},0)$ is the minimal log resolution of a del Pezzo surface $\bar{X}$ of rank one, such that $\height(\bar{X})=2$, and for all witnessing $\P^1$-fibrations, $D\hor$ is irreducible.
	}
\end{equation}

Fix a log surface $(X,D)$ satisfying assumption \eqref{eq:assumption_ht=2_width=1}, and a $\P^1$-fibration $p$ of $X$ witnessing the height.

\begin{lemma}[The structure of $p$]\label{lem:ht=2_twisted-models}
	Let $H$ be the $2$-section in $D$. Let $F_{1},\dots,F_{\nu}$ be the degenerate fibers.
	\begin{enumerate}
		\item \label{item:ht=2_rational} The curve $H$ is rational, and $\beta_{D}(H)\leq 3$.
		\item\label{item:ht=2_twisted_Sigma} Each $F_{j}$ has exactly one unique $(-1)$-curve, say $L_{j}$, and  $L_{j}=(F_{j})\redd-(F_{j})\redd\wedge D\vert$.
		\item\label{item:ht=2_twisted_columnar} We have $\#F_{j}\cap H=1$ for every $j\geq 2$, and either $\#F_{1}\cap H=1$ or $F_{1}$ is columnar.
		\item\label{item:ht=2_twisted_tangent} Assume $\#F_{j}\cap H=1$. Contract those $(-1)$-curves in $F_{j}$ and its images which do not meet the image of $H$, and let $F_{j}'$ be the image of $F_j$. Put $k_{j}=\#(F_{j}')\redd-1$. Then $(F_{j}')\redd=[2,1,2]$ or $\langle 2;[2],[2],[1,(2)_{k_{j}-3}]\rangle$.
		\item\label{item:ht=2_twisted_psi} There is a morphism $\psi\colon X\to \P^{2}$ such that $\cc\de \psi(H)$ is a conic and $\ll_{j}\de \psi_{*}F_{1}$, $j\in \{1,\dots, \nu\}$ are concurrent lines, meeting away from $\cc$. The line $\ll_j$ is tangent to $\cc$ if and only if $\# F_j\cap H=1$.
		\item\label{item:ht=2_twisted_H} Let $k_1,\dots,k_{\nu}$ be as in \ref{item:ht=2_twisted_tangent}, where we put $k_1=1$ if $F_1$ is columnar. Then $H^{2}=4-\sum_{j=1}^{\nu} k_{j}$.
	\end{enumerate}
\end{lemma}
\begin{proof}
	\ref{item:ht=2_rational} Since $\height(\bar{X})=2$, rationality of $H$ follows from Lemma \ref{lem:min_res}\ref{item:one-elliptic}. If $\beta_{D}(H)\geq 4$ then a direct computation shows that $\cf(H)\geq 1$, contrary to Lemma \ref{lem:delPezzo_criterion}.
	
	\ref{item:ht=2_twisted_Sigma} This follows from Lemma \ref{lem:delPezzo_fibrations}\ref{item:Sigma},\ref{item:-1_curves}. 
	
	\ref{item:ht=2_twisted_columnar} Assume $\#F_{1}\cap D=2$. Then $F_1$ meets $D$ in components of multiplicity one. By \ref{item:ht=2_twisted_Sigma}, they lie in $D\vert$, in fact in different connected components of $D\vert$ since otherwise $D$ would have a circular subdivisor. Lemma \ref{lem:degenerate_fibers} implies that $F_1$ is columnar. If $\#F_{2}\cap D=2$ then $\beta_{D}(H)\geq 4$ contrary to \ref{item:ht=2_rational}. Hence $\#F_{j}\cap D=1$ for $j\geq 2$. 
	
	\ref{item:ht=2_twisted_tangent} Part \ref{item:ht=2_twisted_Sigma} shows that $F_{j}'$ has a unique $(-1)$-curve, say $L_{j}'$. By the definition of $\phi^{j}$, the curve $L_{j}'$ meets $\phi^{j}(H)$, so it has multiplicity $2$ in $F_{j}'$. If $L_{j}'$ is a tip of $(F_{j}')\redd$, contract a maximal twig $[1,(2)_{k}]$ of $(F_{j}')\redd$ containing it and denote by $F_{j}''$ the image of $F_{j}'$; otherwise put $F_{j}''=F_{j}'$. Then the unique $(-1)$-curve in $F_{j}''$ has multiplicity two and is not a tip of $(F_{j}'')\redd$, so by Lemma \ref{lem:degenerate_fibers}, $(F_{j}'')\redd=[2,1,2]$. This proves the claim.
	
	\ref{item:ht=2_twisted_psi} By \ref{item:ht=2_twisted_columnar} and \ref{item:ht=2_twisted_psi}, there is a morphism $\phi\colon X\to \F_{n}$, for some $n\geq 0$,  such that $\phi_{*}H\cong H$ is a $2$-section, so $\phi_{*}H\in |bf+2\sigma_{n}|$ for some $b\geq 0$. Moreover, $0=p_{a}(H)=p_{a}(\phi_{*}H)=n+b-1$, so $n=1-b\leq 1$. We can choose the last blowdown in such a way that $n=1$, so $b=\Sec_{1}\cdot \phi_{*}H=0$, $(\phi_{*}H)^{2}=4$, and the composition of $\phi$ with the contraction of the negative section $\Sec_1$ is the required morphism $\psi$.
	
	\ref{item:ht=2_twisted_H} The morphism $\psi$ from \ref{item:ht=2_twisted_psi} factors as $\beta\circ\alpha$, where $\alpha$ is an isomorphism in a neighborhood of $H$, and $\beta$ is a composition of exactly $\sum_{j=1}^{\nu}k_{j}$ blowups, whose exceptional curves meet the images of $H$ once each. Part \ref{item:ht=2_twisted_H} follows since $\psi(H)^{2}=4$. 
\end{proof}

\begin{lemma}[Vertical swaps]\label{lem:ht=2_twisted-swaps}	
	Let $\bar{X}$ be as in \eqref{eq:assumption_ht=2_width=1}. Then $\bar{X}$ swaps vertically to a surface  $\bar{Y}$ from Example \ref{ex:ht=2_twisted} or \ref{ex:ht=2_twisted_cha=2}. With notation from Lemma \ref{lem:ht=2_twisted-models}, one of the following holds.
	\begin{enumerate}	
		\item\label{item:ht=2_sep_nu=2} $\nu=2$, $\#F_{1}\cap H=2$, $(F_{1})\redd\neq [2,1,2]$, $k_{2}=3-H^{2}\geq 5$; and $\bar{Y}$ is of type $\rA_{3}+\rD_{5}$, see Example \ref{ex:ht=2_twisted}\ref{item:A3+D5_construction},
		\item\label{item:ht=2_sep_nu=3} $\nu=3$, $\#F_{1}\cap H=2$, $k_{2}+k_{3}=3-H^{2}\geq 5$; $\cha\kk\neq 2$; and $\bar{Y}$ of type $2\rA_{1}+2\rA_{3}$, see Example \ref{ex:ht=2_twisted}\ref{item:2A1+2A3_construction},
		\item\label{item:ht=2_insep} $\nu\geq 3$, $\#F_{1}\cap H=1$; $\cha\kk=2$, and $\bar{Y}$ is of type \eqref{eq:KM_type}, see Example  \ref{ex:ht=2_twisted_cha=2}.
	\end{enumerate}
\end{lemma}
\begin{proof}
	Let $\psi$ be as in Lemma \ref{lem:ht=2_twisted-models}\ref{item:ht=2_twisted_psi}. Then $(\Bs \psi^{-1})\redd=\{p_{1},\dots, p_{\nu}\}$, where $p_{j}\in \ll_j\cap\cc$. In fact, $\{p_j\}=\ll_{j}\cap \cc$ if and only if $\#F_j\cap H=1$, i.e.\ $j\geq 2$ or $j=1$ and $F_1$ is not columnar.
	
	If $\nu=1$ then the pencil of lines through $p_{1}$ pulls back to a $\P^{1}$-fibration of $X$ such that $\#D\hor\leq 2$, and $D\hor$ consists of $1$-sections, contrary to assumption \eqref{eq:assumption_ht=2_width=1}. Thus $\nu\geq 2$.
	
	Consider the case $\nu=2$. If $(F_{1})\redd=[2,1,2]$ or $\#F_{1}\cap H=1$, then using the pencil of conics tangent to $\ll_{1},\ll_{2}$ at $p_{1}$, $p_{2}$ we get a contradiction with \eqref{eq:assumption_ht=2_width=1} as before; see  Example \ref{ex:ht=2_twisted}\ref{item:A3+D5_construction} for a particular case with $k_2=5$. Lemma \ref{lem:ht=2_twisted-models}\ref{item:ht=2_twisted_H} gives $k_{2}=3-H^{2}\geq 5$, and we conclude that 
	 \ref{item:ht=2_sep_nu=2} holds.
	
	Consider the case $\nu\geq 3$. Assume first that $\cha\kk\neq 2$. Then at most two lines through $p_0$ are tangent to $\cc$, so $\nu=3$ and $\#F_{1}\cap D=2$. As before, Lemma \ref{lem:ht=2_twisted-models}\ref{item:ht=2_twisted_H} gives $k_{2}+k_{3}=3-H^{2}\geq 5$, so, say, $k_{2}\geq 3$ and 
	\ref{item:ht=2_sep_nu=3} holds.
	
	Assume 
	$\cha\kk=2$. Since $p_0$ lies on two lines $\ll_{2}$, $\ll_{3}$ tangent to $\cc$, it lies on all of them, so $\#F_{j}\cap H=1$ for all $j$. Contracting all degenerate fibers to chains $[2,1,2]$ we get a vertical swap as required in \ref{item:ht=2_insep}.
\end{proof}

We now list all singularity types of log canonical surfaces $\bar{X}$ satisfying assumption \eqref{eq:assumption_ht=2_width=1}, and describe the witnessing $\P^1$-fibrations. 
In the list, we will frequently encounter fibers supported on rational forks with two twigs of type $[2]$, cf.\ Lemma \ref{lem:degenerate_fibers}\ref{item:not_columnar}. Thus to simplify our list, we introduce the following notation (we recall that all the remaining conventions used in our statements are summarized in Section \ref{sec:notation}).  

\begin{notation}[{Forks and chains with two twigs $[2]$}]\label{not:Dk,gfork}
	Let $T$ be an admissible chain, and let $r\geq 2$. We put 
	\begin{equation*}
		\gfork{[T,r]}\de \langle r;T,[2],[2]\rangle, \qquad \gfork{[r]}=[2,r,2].
	\end{equation*}
	We extend Notation \ref{not:fibrations} by putting $\gforkd{i}{[T,r]}=\langle r,\ldec{i}T,[2],[2]\rangle$, and $\gforkd{i}{[r]}=[2,r\dec{i},2]$. Moreover, we extend the notation for Dynkin type $\rD_{k}$ by putting $\rD_{3}=\rA_3$, $\rD_{2}=2\rA_1$, and extend Notation \ref{not:fibrations} to such types by
	\begin{equation*}
		\rD_{k}\dec{i}=\gforkd{i}{[(2)_{k-2}]} \mbox{ for }k\geq 3 \quad \mbox{and} \quad \rD_{2}\dec{i}=[2]\dec{i}+[2]\dec{i}.
	\end{equation*}
\end{notation}

First, we consider those cases in Lemma \ref{lem:ht=2_swaps} which can occur if $\cha\kk\neq 2$; thus completing the proof of Theorem \ref{thm:ht=1,2} and Propositions \ref{prop:moduli}, \ref{prop:moduli-hi} in this case.

\begin{lemma}[Classification in cases \ref{lem:ht=2_twisted-swaps}\ref{item:ht=2_sep_nu=2},\ref{item:ht=2_sep_nu=3}, see Tables \ref{table:ht=2_char-any}--\ref{table:ht=2_char-neq-2}.]\label{lem:ht=2_twisted-separable}
Let $(X,D)$ be as in \eqref{eq:assumption_ht=2_width=1}, i.e.\  $(X,D)$ is a minimal log resolution of a del Pezzo surface $\bar{X}$ of rank one, height $2$ and width $1$. Fix a $\P^1$-fibration of $X$ as in Lemma \ref{lem:ht=2_twisted-swaps}, and assume that $\bar{X}$ swaps vertically to a surface $\bar{Y}$ from Example \ref{ex:ht=2_twisted}.

Assume that $\bar{X}$ is log canonical.  Let $\cS$ be the singularity type of $\bar{X}$, and let $\check{\cS}$ be the combinatorial type of the sum of $D$ and all vertical $(-1)$-curves (which we write using Notations \ref{not:fibrations} and \ref{not:Dk,gfork}). Then $\cS$ determines $\bar{X}$ uniquely up to an isomorphism, $h^{i}(\lts{X}{D})=0$ for all $i$, and one of the following holds.
	\begin{enumerate}[itemsep=1em]
		\item\label{item:ht=2_A3+D5} The surface $\bar{Y}$ is  of type $\rA_{3}+\rD_{5}$, see Example \ref{ex:ht=2_twisted}\ref{item:A3+D5_construction}, and $\check{\cS}$ is one of the following:
		\begin{longlist}	
			\item\label{item:twisted_off_nu=2} $\ldec{1}[T,\bs{k}\dec{2},T^{*}]\dec{1}+\rD_{k+3}\dec{2}$, $T\neq [2]$, 
			\item\label{item:c=1_T=[2]_l=0} 
			$\langle \bs{k}\dec{3}; [3]\dec{1},\ldec{1}[2,2],[2]\dec{2} \rangle+\langle 2;\ldec{2}[3,(2)_{k+2}],[2],[2]\rangle$,
			\setcounter{foo}{\value{longlisti}}
		\end{longlist}
		\item \label{item:ht=2_2A1+2A3} We have $\cha\kk\neq 2$, the surface $\bar{Y}$ is of type  $2\rA_{1}+2\rA_{3}$, see Example \ref{ex:ht=2_twisted}\ref{item:2A1+2A3_construction}, and $\check{\cS}$ is one of the following:
		\begin{longlist}\setcounter{longlisti}{\value{foo}}
			\item\label{item:twisted_off_nu=3}
			$\ldec{1}[T,\bs{k+l-2}\dec{2,3},T^{*}]\dec{1}+\rD_{k+1}\dec{2}+\rD_{l}\dec{3}$, 
			\item\label{item:k2=2}
			$\langle \bs{k}\dec{3}; [2]\dec{1},[2]\dec{1},[2,l]\dec{2} \rangle+\ldec{2}[(2)_{l-2},3]+\rD_{k+1}\dec{3}$,
			\item\label{item:c=1_T1=[2]}
			$\langle \bs{k+l-2}\dec{3}; [2]\dec{1},[2]\dec{1},\ldec{2}T\trp \rangle+
			\gforkd{2}{T^{*}*[(2)_{k-1}]}+\rD_{l}\dec{3}$,
			\item\label{item:c=1_T=[2]} 
			$\langle \bs{k+l-2}\dec{3}; [3]\dec{1},\ldec{1}[2,2],[2]\dec{2} \rangle+
			\gforkd{2}{[3,(2)_{k-2}]} +\rD_{l}\dec{3}$,
			\setcounter{foo}{\value{longlisti}}
		\end{longlist}
	\end{enumerate}
	for some integers $k,l\geq 2$ and an admissible chain $T$.
\end{lemma} 
\begin{proof}
	We keep notation from Lemma \ref{lem:ht=2_twisted-models}. Let $D_{H}$ be the connected component of $D$ containing $H=[\bs{m}]$. Since $\#F_1\cap H=2$ by Lemma \ref{lem:ht=2_twisted-swaps}, Lemma \ref{lem:ht=2_twisted-models}\ref{item:ht=2_twisted_columnar} implies that $F_1$ is columnar, so  $(F_{1})\redd=[T,1,T^{*}]$. 
	
	If $L_{j}$ meets $H$ for all $j\geq 2$ then by Lemma \ref{lem:ht=2_twisted-models}\ref{item:ht=2_twisted_tangent},\ref{item:ht=2_twisted_H} $D$ is as in \ref{item:twisted_off_nu=2} if $\nu=2$ and   \ref{item:twisted_off_nu=3} if $\nu=3$. Assume $F_{2}$ meets $H$ in some component of $D\vert$. Then $\beta_{D}(H)\geq 3$, so $\beta_{D}(H)=3$ by Lemma \ref{lem:ht=2_twisted-models}\ref{item:ht=2_rational}. Suppose $D_H$ is not a fork. Then $D_H$ contains a chain whose tips are branching in $D_H$, and $H$ is one of those tips, so by \cite[3.1.3, 3.2.5]{Flips_and_abundance} we get $\cf(H)\geq 1$, a contradiction with Lemma \ref{lem:delPezzo_criterion}. Thus $D_{H}=\langle \bs{m};T\trp,T^{*},U\rangle$ for some admissible chain $U\subseteq F_2$; and we have either $\nu=2$, or $\nu=3$ and $L_{3}$ meets $H$. Note that the fork $D_{H}$ is admissible, since otherwise $\cf(H)\geq 1$, contrary to Lemma  \ref{lem:delPezzo_criterion}.
	
	Consider the case $k_{2}=2$. Then $\nu=3$, since otherwise $k_2\geq 5$ by Lemma  \ref{lem:ht=2_twisted-swaps}\ref{item:ht=2_sep_nu=2}. The morphism from Lemma \ref{lem:ht=2_twisted-models}\ref{item:ht=2_twisted_tangent} maps $(F_{2})\redd$ to a chain $[2,1,2]$, and does not contract $\ftip{U}$. It follows that $(F_2)\redd=[U,1,U^{*}]=[2,l,1,(2)_{l-2},3]$ for some $l\geq 2$. Since $\#U\geq 2$, admissibility of $D_{H}$, see Lemma \ref{lem:admissible_forks}\ref{item:has_-2}, implies that $T=[2]$ or $T^{*}=[2]$. Thus $T,T^{*}=[2]$ and $D$ is as in \ref{item:k2=2}.
	
	Consider the case $k_2\geq 3$. Then $(F_{2})\redd=\langle 2; [U,1,U^{*}]*[(2)_{k_{2}-3}],[2],[2]\rangle$. If $T=[2]$ then $\nu=3$ by Lemma \ref{lem:ht=2_twisted-swaps}, so $D$ is as in \ref{item:c=1_T1=[2]}. Assume $T\neq [2]$. Then interchanging $T$ with $T^{*}$ if necessary, we infer from  Lemma \ref{lem:admissible_forks}\ref{item:has_-2} that $U=[2]$, $T=[3]$, $T^{*}=[2,2]$, so $D$ is as in \ref{item:c=1_T=[2]_l=0} if $\nu=2$ and as in \ref{item:c=1_T=[2]} if $\nu=3$.
	\smallskip
	
	To see the uniqueness result, we argue, once again, as in the proof of Corollary \ref{cor:moduli-ht=1}. Let $\check{D}$ be the sum of $D$, the vertical $(-1)$-curves, and the $(-1)$-curve over the common point of point $p_0$ of the lines in Example   \ref{ex:ht=2_twisted}. Let $\psi\colon (X,\check{D})\to (Y,\check{D}_Y)$ be the morphism onto  the surface from Example   \ref{ex:ht=2_twisted}, and let $\check{\cS}$, $\check{\cS}_0$ be the combinatorial types of $(X,\check{D})$ and $(Y,\check{D}_Y)$. Applying Lemma \ref{lem:adding-1}\ref{item:adding-1-h1} to $B=\{\pt\}$ we see that it is enough to prove that $\#\cP_{+}(\check{\cS})=1$ and $h^{i}(\lts{X}{\check{D}})=0$. We check that the morphism $\psi$ is inner, so by Lemmas \ref{lem:inner} and \ref{lem:blowup-hi}\ref{item:blowup-hi-inner} it is enough to prove the corresponding statements for $(Y,\check{D}_Y)$. 
	
	Since each planar configuration used to construct $(Y,\check{D}_Y)$  in Example \ref{ex:ht=2_twisted} is unique up to a projective equivalence (even with a fixed order of components and singular points), we have $\#\cP_{+}(\check{\cS}_0)=1$.  It remains to prove that $h^{1}(\lts{Y}{\check{D}_Y})=0$. We keep the notation from Example \ref{ex:ht=2_twisted}.
	
	In case \ref{ex:ht=2_twisted}\ref{item:A3+D5_construction} put $D_{Y}'=D_{Y}+A_1+A_2+A_0+L$, where $L$ is the proper transform of a line joining $p_1$ with $p_2$. By Lemma \ref{lem:h1}\ref{item:h1-1_curve} we have $h^{1}(\lts{Y}{\check{D}_Y})=h^{1}(\lts{Y}{D_Y'})$. The contraction of $L+A_0+A_1+A_2+T$, where $T$ is the fork in $D_Y$ minus its branching component, is an inner morphism $(Y,D_Y')\to (\P^{2},B)$, where $B$ is a sum of four general lines. Thus $h^{i}(\lts{Y}{D'_Y})=0$ for all $i$ by Lemmas \ref{lem:blowup-hi}\ref{item:blowup-hi-inner} and \ref{lem:h1}\ref{item:h1-P2}. 
	
	In case \ref{ex:ht=2_twisted}\ref{item:2A1+2A3_construction}, let $(Y,\check{D}_Y)\to (Z,D_{Z}')$ be the contraction of $A_3$, and let $D_Z$ be the image of $D_Y+\sum_{i=0}^{3}A_{i}+L_2+L_3-C$, where $L_i$ is the line joining $p_1$ with $p_i$. Then $h^{1}(\lts{Y}{D_Y})=h^{1}(\lts{Z}{D_Z'})=h^{1}(\lts{Z}{D_Z})$ by Lemmas \ref{lem:blowup-hi}\ref{item:blowup-hi-inner} and \ref{lem:h1}\ref{item:h1-1_curve}. Contracting the image of $A_1+L_1+\phi^{-1}(p_2)+\phi^{-1}(p_3)$ we get an inner morphism $(Z,D_Z)\to (\P^1\times \P^1,B)$, where $B$ is the sum of three horizontal and three vertical lines, so $h^{i}(\lts{Z}{D_Z})=h^{i}(\lts{\P^1\times \P^1}{B})=0$ by Lemmas \ref{lem:blowup-hi}\ref{item:blowup-hi-inner} and  \ref{lem:h1}\ref{item:h1-grid}.
\end{proof}

To prove Theorem \ref{thm:ht=1,2} and Propositions \ref{prop:moduli}, \ref{prop:moduli-hi} it remains to classify, in case $\cha\kk=2$, del Pezzo surfaces of rank one swapping vertically to one of the surfaces from Example \ref{ex:ht=2_twisted_cha=2}. We begin by constructing representing families for the latter, in such a way that they can be lifted through the vertical swaps. This construction is analogous to the one in Examples \ref{ex:4-points}, \ref{ex:4-points-Aut}, where the family was parametrized by the choice of a fourth point on $\P^1$. Now, we consider $\nu \geq 3$ lines tangent to a given conic (recall that these exist if and only if $\cha\kk=2$), and construct a family parametrized by the choice of $\nu-3$ of them.

First, we recall the construction in Example \ref{ex:ht=2_twisted_cha=2} and introduce additional notation. We assume $\cha\kk=2$. Fix an integer $\nu\geq 3$. Let $\PKM(\nu)$ be the set of isomorphism classes of surfaces constructed in Example \ref{ex:ht=2_twisted_cha=2} for the given $\nu$, and let $\PKMres(\nu)$ be the set of isomorphism classes of their minimal log resolutions. Each $(Y,D_Y)\in \PKMres(\nu)$ is constructed in Example \ref{ex:ht=2_twisted_cha=2} as follows. Let $p_{1},\dots, p_{\nu}$ be distinct points on a smooth conic $\cc\subseteq \P^2$, and for $i\in \{1,\dots,\nu\}$ let $\ll_{i}$ be the line tangent to $\cc$ at $p_{i}$. Since $\cha\kk=2$, the lines $\ll_1,\dots,\ll_{\nu}$ meet at a point $p_{0}$. Put $\pp=\cc+\ll_1+\dots+\ll_{\nu}$. Now, let $\psi\colon \to \P^2$ as the minimal log resolution of $(\P^2,\pp)$, for $i\in \{0,\dots,\nu\}$ let $A_{i}$ be the $(-1)$-curve in $\psi^{-1}(p_{i})$, and let $D=(\psi^{*}\pp)\redd-\sum_{i=0}^{\nu}A_i$.

Write $\psi=\sigma\circ\psi'$, where $\sigma$ is the blowup at $\{p_0,p_{1}\}$, and put $\phi=\sigma'\circ \psi'\colon X\to \P^1\times \P^1$, where $\sigma'$ is the contraction of $\sigma^{-1}_{*}\ll_1$. Then $\phi_{*}D_Y=\bar{C}+\sum_{i=1}^{\nu}\bar{V}_{i}$, where each $\bar{V}_{i}$ is a vertical line, and $\bar{C}=\phi(C)$ is a $(+4)$-curve which is tangent to each vertical line with multiplicity $2$, and meets each horizontal line once. Write $\{q_i\}=\bar{V}_i\cap \bar{C}$, $V_{i}=\phi^{-1}_{*}\bar{V}_i$ and $\phi^{-1}(q_i)=A_i+W_i$, where $W_i=[2]$. Let $\check{D}=D_Y+\sum_{i=1}^{\nu}A_i$. Let $\check{\Gamma}$ be the weighted graph of $\check{D}$, let $c$ be its central vertex, let $a_1,\dots,a_{\nu}$ be the vertices of weight $(-1)$ adjacent to $c$, and for $i\in \{1,\dots,\nu\}$ let $v_i$, $w_i$ be the tips of weight $-2$ adjacent to $a_i$, see Figure \ref{fig:KM_alt}.

\begin{figure}[htbp] \vspace{-0.6em}
	\begin{tikzpicture}[scale=1.2]
		\begin{scope}
		\draw[very thick] (0,1.2) -- (3.1,1.2);
		\node at (2.1,1.4) {\small{$C$}};
		\node at (2.1,1) {\small{$4-2\nu$}};
		%
		\draw (0.2,2.5) -- (0,1.5);
		\node at (0.3,2) {\small{$V_1$}};
		\draw[dashed] (0,1.7) -- (0.2,0.7);
		\node at (0.3,1.4) {\small{$A_1$}};
		\draw (0.2,0.9) -- (0,-0.1);
		\node at (0.3,0.2) {\small{$W_1$}};
		\draw (1.2,2.5) -- (1,1.5);
		\node at (1.3,2) {\small{$V_2$}};
		\draw[dashed] (1,1.7) -- (1.2,0.7);
		\node at (1.3,1.4) {\small{$A_2$}};
		\draw (1.2,0.9) -- (1,-0.1);
		\node at (1.3,0.2) {\small{$W_2$}};
		\node at (2.1,2) {\Large{$\cdots$}};
		\draw (3,2.5) -- (2.8,1.5);
		\node at (3.1,2) {\small{$V_{\nu}$}};
		\draw[dashed] (2.8,1.7) -- (3,0.7);
		\node at (3.1,1.4) {\small{$A_{\nu}$}};
		\draw (3,0.9) -- (2.8,-0.1);
		\node at (3.1,0.2) {\small{$W_{\nu}$}};
		\draw[->] (-0.3,1.2) -- (-1.3,1.2);
		\node at (-0.7,1.4) {\small{$\phi$}};
		\end{scope}
		\begin{scope}[shift={(-4.8,0)}]
			\draw (0,2.5) -- (0,0);
			\node[right] at (0,0.2) {\small{$\bar{V}_1$}};
			\node[left] at (0,0.2) {\small{$0$}};
			\draw (1,2.5) -- (1,0);
			\node[right] at (1,2.3) {\small{$\bar{V}_2$}};
			\node[left] at (1,2.3) {\small{$0$}};
			\node at (2.1,2) {\Large{$\cdots$}};
			\draw (3,2.5) -- (3,0);
			\node[right] at (3,2.3) {\small{$\bar{V}_{\nu}$}};
			\node[left] at (3,2.3) {\small{$0$}};
			\draw[thick] (-0.6,2.2) -- (-0.3,2.2) to[out=0,in=90] (0,2) to[out=-90,in=90] (-0.3,1.7) to[out=-90,in=180] (-0.1,1.6); 
			\draw[thick] (0.1,1.6) -- (0.7,1.6) to[out=0,in=90] (1,1.4) to[out=-90,in=90] (0.7,1.1) to[out=-90,in=180] (0.9,1); 
			\draw[thick] (1.1,1) -- (2.7,1) to[out=0,in=90] (3,0.8) to[out=-90,in=90] (2.7,0.5) to[out=-90,in=180] (2.9,0.4);
			\draw[thick] (3.1,0.4) -- (3.5,0.4);
			\node at (2.1,1.2) {\small{$\bar{C}$}};
			\node at (2.1,0.8) {\small{$4$}};
		\end{scope}	
		\begin{scope}[shift={(6.3,0.2)}]
			\filldraw (0,2) circle (0.04); 
			\node[left] at (0,2.1) {\small{$4-2\nu$}};
			\node[right] at (0,2.1) {\small{$c$}};
			\draw (0,2) -- (-1.4,1.4); \filldraw (-1.4,1.4) circle (0.04);
			\node at (-1.65,1.35) {\small{$-1$}};
			\node at (-1.15,1.3) {\small{$a_1$}};
			\draw (0,2) -- (0,1.4); \filldraw (0,1.4) circle (0.04);
			\node at (-0.25,1.4) {\small{$-1$}};
			\node at (0.25,1.4) {\small{$a_2$}};
			\node at (0.8,0.8) {\Large{$\dots$}};
			\draw (0,2) -- (1.6,1.4); \filldraw (1.6,1.4) circle (0.04);
			\node at (1.3,1.3) {\small{$-1$}};
			\node at (1.85,1.4) {\small{$a_{\nu}$}};
			\draw (-1.4,1.4) -- (-1.7,0.4); \filldraw (-1.7,0.4) circle (0.04); 
			\node[below] at (-1.7,0.4) {\small{$v_1$}};
			\draw (-1.4,1.4) -- (-1.1,0.4); \filldraw (-1.1,0.4) circle (0.04);
			\node[below] at (-1.1,0.4) {\small{$w_1$}};
			\draw (0,1.4) -- (-0.3,0.4); \filldraw (-0.3,0.4) circle (0.04);
			\node[below] at (-0.3,0.4) {\small{$v_2$}};
			\draw (0,1.4) -- (0.3,0.4); \filldraw (0.3,0.4) circle (0.04);
			\node[below] at (0.3,0.4) {\small{$w_2$}};
			\draw (1.6,1.4) -- (1.3,0.4); \filldraw (1.3,0.4) circle (0.04);
			\node[below] at (1.3,0.4) {\small{$v_{\nu}$}};
			\draw (1.6,1.4) -- (1.9,0.4); \filldraw (1.9,0.4) circle (0.04);
			\node[below] at (1.9,0.4) {\small{$w_{\nu}$}};
			\draw[<->, black!80, densely dashed] (-1.7,0.1) to[out=-90,in=180] (-1.4,-0.2) to[out=0,in=-90] (-1.1,0.1);
			\node[black!80] at (-0.9,-0.15) {\small{$\Z/2$}};
			\draw[<->, black!80, densely dashed] (-0.3,0.1) to[out=-90,in=180] (0,-0.2) to[out=0,in=-90] (0.3,0.1);
			\node[black!80] at (0.5,-0.15) {\small{$\Z/2$}};
			\draw[black!30] (-1.9,1.55) rectangle (-0.9,0.13);
			\draw[<-,black!80, densely dashed] (-1.4,1.6) to[out=30,in=-150] (1.6,2.1);
			\draw[black!30] (-0.5,1.55) rectangle (0.5,0.13);
			\draw[<-,black!80, densely dashed] (0.2,1.6) to[out=45,in=-150] (1.6,2.1);
			\draw[black!30] (1.1,1.55) rectangle (2.1,0.13);
			\draw[<-,black!80, densely dashed] (1.6,1.6) to[out=90,in=-90] (1.6,2.1);
			\node[black!80] at (1.6,2.2) {\small{$S_{\nu}$ permutes}};
		\end{scope}
	\end{tikzpicture}\vspace{-0.8em}
	\caption{Constructing a surface in $\PKM(\nu)$ from $\nu$ fibers of $\P^1\times \P^1\to \P^1$, cf.\ Example \ref{ex:ht=2_twisted_cha=2}.}
	\label{fig:KM_alt}\vspace{-0.7em}
\end{figure}
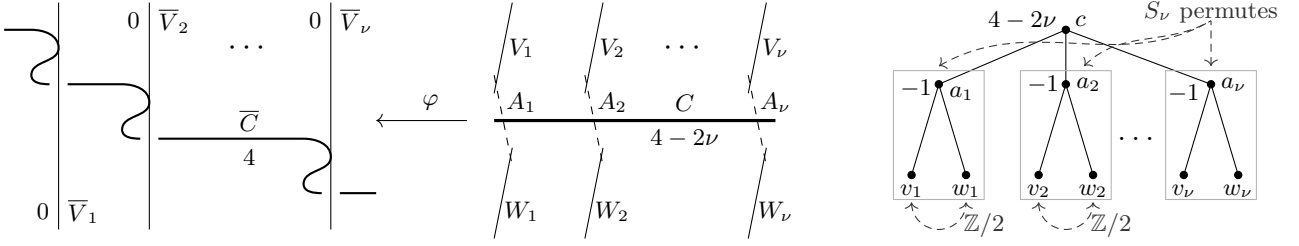

Let $\cS$ and $\check{\cS}$ be the combinatorial types of $(Y,D_Y)$ and $(Y,\check{D})$, respectively, so $\check{\cS}=(\check{\Gamma},2\nu+4)$. We consider the subgroup $\Z/2\rtimes S_{\nu}\subseteq \Aut(\check{\cS})$, where $S_{\nu}$ permutes the $\nu$ ordered triples $(v_i,a_i,w_i)$, and $\Z/2$ interchanges the tips $(v_1,v_2)$ with $(w_1,w_2)$ and leaves the rest of the graph intact.

Eventually, we define $\cR\subseteq\cP_{+}(\check{\cS})$ as the set of (equivalence classes of) triples $(Y,\check{D},\gamma)\in \cP_{+}(\check{\cS})$ such that $\gamma(v_i)$ is the vertex corresponding to $V_i$. By construction, the composition of forgetful maps $\cP_{+}(\check{\cS})\to \cP_{+}(\cS)\to \cP(\cS)$ maps $\cR$ onto $\PKMres(\nu)$. If $(\cY,\check{\cD})\to M$ is a smooth family representing $\cR$, then for $m\in M$, $i\in \{1,\dots,\nu\}$ we denote by $C_m$, $V_{i,m}$, $A_{i,m}$, $W_{i,m}$ the components of $\check{D}_{m}$ corresponding to $c,v_i,a_i,w_i$, respectively.

\begin{lemma}[Moduli in Example \ref{ex:ht=2_twisted_cha=2}]\label{lem:7A1-family} With the above notation, the following hold.
	\begin{enumerate}
		\item\label{item:7A1-family-R} The set $\cR$ is represented by a $(\Z/2\rtimes S_{\nu})$-faithful universal family $f\colon (\cY,\check{\cD})\to M$, with $\dim M=\nu-3$. 
		\item\label{item:7A1-family-iso} Fix $m_1,m_2\in M$. If the pairs $(C_{m_j},\{A_{i,m_j}\cap C_{m_j}\}_{i=1}^{\nu})$, $j\in \{1,2\}$ are isomorphic, then $m_1$, $m_2$ lie in the same orbit of the above $S_{\nu}$-action.
		\item\label{item:7A1-family-P} The set $\PKM(\nu)$ has moduli dimension $\nu-3$.
		\item\label{item:7A1-family-h2} Every $(Y,D_Y)\in \PKMres(\nu)$ satisfies $h^{0}(\lts{Y}{D_Y})=0$ and $h^{2}(\lts{Y}{D_Y})=\nu-2$.
	\end{enumerate}
\end{lemma}
\begin{proof}
	We will use the open part of the moduli space of genus zero curves with $\nu$ marked points. It is  constructed in \cite{Kapranov-moduli} as follows, cf.\ \cite[\sec III.2]{DO}. Fix a $\nu$-tuple  $\uline{x}=(x_{1},\dots,x_{\nu})$ of points  in $\P^{\nu-2}$ in general position. Let $H\subseteq \P^{\nu-2}$ be the union of hyperplanes spanned by $(\nu-1)$-tuples of points in $\uline{x}$, let $U$ be the blowup of $\P^{\nu-2}$ in $\uline{x}$ with the proper transform of $H$ removed, and let $E_{i}\subseteq U$ be the exceptional divisor over $x_i$. There is a $\P^{1}$-fibration $u\colon U\to M$ whose fibers are proper transforms of rational normal curves passing through $\uline{x}$. For $m\in M$ let $U_{m}$ be the fiber over $m$, and write $\{x_{i,m}\}=E_i\cap U_{m}$. It is shown in \cite[p.\ 243]{Kapranov-moduli} that for every ordered $\nu$-tuple of points $\uline{p}$ on $\P^{1}$, the pair $(\P^1,\uline{p})$ is isomorphic to $(U_m,\uline{x}_{m})$ for a unique $m\in M$. The $S_{\nu}$-action on $\P^{\nu-2}$ permuting the points in $\uline{x}$ induces an $S_{\nu}$-action on $U$ and $M$ which makes $u$ equivariant. 
	
	For any $m\in M$, the germ of $u$ at $m\in M$ is a semiuniversal deformation of $(U_m,Z_m)$, where $Z_m$ is the divisor $\sum_{i=1}^{\nu}x_{i,m}$. To see this, we note first that $h^{1}(\lts{\P^1}{Z_m})=\nu-3$, so it is enough to show that $u$ is versal at $m$. For any infinitesimal deformation $(\cP,\cZ)\to T$ of $(U_m,Z_m)$ we perform the above construction in a parametric way. More precisely, we argue as follows. Since $\P^1$ is rigid, we have $\cP=\P^{1}\times T$, and $\cZ$ consists of $\nu$ sections of the second projection. Let $\P^1\times T\to \P^{\nu-2}\times T$ be a family of Veronese embeddings, mapping $\cZ$ to $\{x_1,\dots,x_{\nu}\}\times T$ such that at the central fiber over $t_0\in T$, the point $(x_{i,m},t_0)\in U_{m}\times \{t_0\}$ is mapped to $(x_i,t_0)$. This embedding lifts to an embedding $\P^1\times T\to U\times T$, whose image lies in a neighborhood of $U_m\times T$. Composing it with the first projection and with the $\P^1$-fibration $u\colon U\to M$ we get a morphism $\alpha\colon T\to U$ such that our deformation is a pullback of $u$ along $\alpha$, as needed.
	\smallskip
	
	We go back to the proof of the lemma. Let $\bar{\Gamma}^{\circ}$ be a graph consisting of $\nu$ vertices of weight $0$ and no edges; let $\bar{\cS}^{\circ}=(4,\bar{\Gamma}^{\circ})$ and let $\bar{\cR}^{\circ}$ be a subset of $\cP_{+}(\bar{\cS}^{\circ})$ consisting of those triples whose underlying surface is $\P^1\times \P^1$. In other words, $\bar{\cR}^{\circ}$ is the set of isomorphism classes of $\nu$ ordered vertical lines on $\P^1\times \P^1$. Let $\bar{\cY}=U\times \P^1$, $\bar{\cV}_i=E_i\times \P^1$, $\bar{\cD}^{\circ}=\sum_{i}\bar{\cV}_i$, and let $\bar{f}\colon \bar{\cY}\to U\to M$ be the composition of the first projection and the $\P^1$-fibration $u$. Then $\bar{f}\colon (\bar{\cY},\bar{\cD}^{\circ})\to M$ is a universal, $S_{\nu}$-faithful family representing $\bar{\cR}^{\circ}$. 
	
	We now introduce a hypersurface $\bar{\cC}\subseteq \bar{\cY}$ corresponding to the family of conics, as follows. Write  $\{z_1,z_2,z_3\}=\{0,1,\infty\}\subseteq \P^1$, $\bar{\cH}_{j}\de U\times \{z_i\}$ for $i\in \{1,2,3\}$. For $m\in M$, let $\Delta_{m}\subseteq \bar{Y}_m\cong \P^1\times \P^1$ be the unique curve of type $(1,1)$ passing through $\bar{V}_{i,m}\cap \bar{H}_{i,m}$, $i\in \{1,2,3\}$. We claim that there is a hypersurface $\Delta$ in $\bar{\cY}$ such that $\bar{f}|_{\Delta}$ is a smooth family with fibers $\Delta_m$. If $\nu=3$ then the claim is clear. Otherwise let $\eta\colon \P^{\nu-2}\map  \P^1$ be the linear projection from the hyperplane spanned by $x_{4},\dots,x_{\nu}$, and choose the coordinates on the target $\P^1$ so that for $i\in \{1,2,3\}$ we have $\eta(x_i)=z_i$. Then $\eta$ extends to a morphism $U\to \P^1$ whose graph is the required hypersurface $\Delta$. Let now $\bar{\cC}\de (\id_{U},s)^{-1}(\Delta)$, where $s\colon \P^1\to \P^1$ is given by $s(z)=z^2$. Write $\cQ_{i}=\bar{\cV}_i\cap \bar{\cC}$. The restriction $\bar{f}|_{\bar{\cC}}$ is a smooth morphism whose fiber $\bar{C}_m$ is a curve of type $(2,1)$ on $\bar{Y}_m$, and since $\cha\kk=2$, it is tangent to $\bar{V}_{i,m}$ at $\{q_{i,m}\}\de \cQ_{i}\cap \bar{Y}_{m}$. For $i\in \{1,2,3\}$ we have $s(z_i)=z_i$, so $\{q_{i,m}\}=\bar{V}_{i,m}\cap \bar{H}_{i,m}$. Note that these properties determine the curve $\bar{C}_{m}\subseteq \bar{Y}_m$ uniquely; and for every curve $\bar{C}_m'$ of type $(2,1)$ on $\bar{Y}_m$ which is tangent to every vertical line there is an automorphism of $\bar{Y}_m$ mapping $\bar{C}'_{m}$ to $\bar{C}_m$. 
	
	Let $\bar{\cD}=\bar{\cD}^{\circ}+\bar{\cC}$. For every $i\in \{1,\dots,\nu\}$ blow up at $\cQ_i$ and at the intersection of the exceptional divisor with the proper transform of $\bar{\cC}$, and denote the resulting morphism by  $\Phi\colon \cY\to \bar{\cY}$. Let $\cV_{i}=\Phi^{-1}_{*}\bar{\cV}_i$, $\cH_{i}=\Phi^{-1}_{*}\bar{\cH}_i$, $\cC=\Phi^{-1}_{*}\bar{\cC}$; let $\cW_i$, $\cA_i$ be the first and second exceptional divisor over $\cQ_i$, and let $\check{\cD}=(\Phi^{*}\bar{\cD})\redd$, $\check{\cD}^{\circ}=(\Phi^{*}\bar{\cD}^{\circ})\redd$. Let $f\de \bar{f}\circ \Phi \colon (\cY,\check{\cD})\to M$ and let $f^{\circ}$ be the restriction of $f$ to $(\cY,\check{\cD}^{\circ})$. By the  definition of $\cR$ every log surface $(X,\check{D})$ in $\cR$ has a morphism $\phi$ mapping $(X,\check{D}-C)$ to a fiber $(\bar{Y}_m,\bar{D}^{\circ}_m)$ of $\bar{f}^{\circ}$, and the above uniqueness property of $\bar{C}_m$ shows that $\phi(C)=\bar{C}_m$. Hence $f$ is a family representing $\cR$. It is $S_{\nu}$-faithful since $\bar{f}$ is. 
	
	We now claim that $f$ has property \ref{item:7A1-family-iso}. To this end, assume that for two fibers $(Y_{m_j},\check{D}_{m_j})$, $j=1,2$ there is an isomorphism $\tau\colon C_{m_1}\to C_{m_2}$ preserving the distinguished $\nu$-tuples. Applying the $S_{\nu}$-action we can assume that $\tau(A_{i,m_1}\cap C_{m_1})=A_{i,m_2}\cap C_{m_2}$ for all $i\in \{1,\dots,\nu\}$. Hence $\tau$ descends to an isomorphism $\bar{C}_{m_1}\to \bar{C}_{m_2}$ mapping $q_{i,m_1}$ to $q_{i,m_2}$. This in turn yields an isomorphism $\Delta_{m_1}\to \Delta_{m_2}$ which after projecting to the $U$-coordinate yields an isomorphism $U_{m_1}\to U_{m_2}$, mapping $\uline{x}_{m_1}$ to $\uline{x}_{m_2}$. Hence $m_1=m_2$, which proves \ref{item:7A1-family-iso}.
	
	To prove that $f$ is $(\Z/2\rtimes S_{\nu})$-faithful, we need to construct an involution of $(\cY,\check{\cD})$ which maps fibers of $f$ to fibers of $f$, and on every such fiber it maps the $(-2)$-tips $V_{1,m},V_{2,m}$ to $W_{1,m},W_{2,m}$ and fixes the remaining components of $\check{D}_m$. To this end, let $\Phi'\colon (\cY,\check{\cD})\to (\bar{\cY}',\bar{\cD}')$ be the contraction of $\sum_{i=1}^{2}(\cA_i+\cV_i)+\sum_{i\geq 2}^{\nu}(\cA_i+\cW_i)$. 
	Then for each $m\in M$, the weighted graphs of $\bar{D}_{m}'$ and $\bar{D}_m$ are the same, in particular $\bar{Y}_{m}'\cong \P^1\times \P^1$, and $(\bar{Y}_{m}',\bar{D}_{m}')$ can be identified with some fiber $(\bar{Y}_{m'},\bar{D}_{m'})$ of $\bar{f}$. The curves $\Phi'(C_{m})$ and $\bar{C}_{m}'$ with distinguished $\nu$-tuples of points are isomorphic, so using property \ref{item:7A1-family-iso} we can assume that $m=m'$. This way, we can identify the family $(\bar{\cY}',\bar{\cD}')$ with $(\bar{\cY},\bar{\cD})$. The birational map $\Phi^{-1}\circ \Phi'$ extends to the required involution. 
	
	To end the proof of \ref{item:7A1-family-R} we need to check that $f$ is semiuniversal at every $m\in M$. Let $h\colon (\cY_T,\check{\cD}_T)\to T$ be a small deformation of $(Y_m,\check{D}_m)$, let $\cC_{T}$ be the component of $\check{\cD}_{T}$ corresponding to the vertex $c$, and let $\check{\cD}^{\circ}_{T}=\check{\cD}_{T}-\cC_{T}$. As in the proof of Lemma \ref{lem:inner}, we have a blowdown $\Phi_{T}\colon (\cY_{T},\check{\cD}_{T}^{\circ})\to (\bar{\cY}_{T},\bar{\cD}^{\circ}_{T})$ to a small deformation $\bar{h}^{\circ}\colon (\bar{\cY}_T,\bar{\cD}^{\circ}_{T})\to T$ of $(\bar{Y}_m,\bar{D}_{m}^{\circ})$. Since $\bar{f}^{\circ}$ is semiuniversal at $m$, we get a morphism $\alpha\colon T\to M$ such that $\bar{h}^{\circ}=\alpha^{*}\bar{f}^{\circ}$, with differential uniquely determined by $h$. The image of $\bar{\cC}$ on $\bar{\cY}_{T}\cong \P^1_{T}\times _{T} \P^1_{T}$ is a curve of type $(2,1)$ tangent to every vertical line, so composing $\Phi_{T}$ with an automorphism of $\P^{1}_{T}\times_{T}\P^1_{T}$ we can further assume that $\Phi_{T}(\cC_{T})=\alpha^{*}\bar{\cC}$. By the  universal property of blowing up we obtain that $h=\alpha^{*}f$, as required.
	\smallskip

	We have thus proved  \ref{item:7A1-family-R} and \ref{item:7A1-family-iso}. Now we prove \ref{item:7A1-family-P}. 
	Recall that the image of $\cR$ under the forgetful map $\cP_{+}(\check{\cS})\to \cP(\cS)$ equals $\PKMres(\nu)$, so the restriction of $f$ to $(\cY,\cD_{Y})$, where $\cD_{Y}=\check{\cD}-\sum_{i=1}^{\nu}\cA_{i}$, represents~$\PKMres(\nu)$. 
	
	We need to show that this restriction, call it $f_Y$, is almost universal. The fibers of $f_Y$ differ from those of $f$ by removing some $(-1)$-curves from the boundary. Since $f$ is semiuniversal at each $m\in M$, by \cite[Proposition 1.7(1)]{FZ-deformations} so is $f_Y$, cf.\ Lemma \ref{lem:h1}\ref{item:h1-1_curve}.	By Definitions \ref{def:moduli}\ref{item:def-family-faithful} and \ref{def:moduli-hi}\ref{item:def-universal} it remains to prove that whenever two fibers of $f_Y$ are isomorphic (as log surfaces), they are equivalent under the $S_{\nu}$-action constructed above. 
	
	If $\nu=3$ then $M$ has dimension $0$, so $\#\PKMres(\nu)=1$ and there is nothing to prove. Assume $\nu\geq 4$, and let $\tau\colon (X_{m_1},D_{Y,m_1})\to (X_{m_2},D_{Y,m_2})$ be an isomorphism between fibers of $f_Y$. The horizontal part of $D_{Y,m_i}$, namely the curve $C_{m_i}$, is the unique $(4-2\nu)$-curve in $D_{Y,m_i}$, hence $\tau(C_{m,1})=C_{m,2}$. Lemma  \ref{lem:adding-1-criterion} implies that $\tau(A_{i,m_1})=A_{\sigma(i),m_2}$ for some $\sigma\in S_{\nu}$. The claim now follows from property  \ref{item:7A1-family-iso}.
	\smallskip
	
	It remains to prove \ref{item:7A1-family-h2}. Fix $(Y,D_Y)\in \PKMres(\nu)$. Part \ref{item:7A1-family-P} implies that  $h^{1}(\lts{Y}{D_Y})=\nu-3$. The normal bundle $\cN_{C}$ to $C=[2\nu-4]$ in $Y$ satisfies  $h^{i}(\cN_{C})=h^{i}(\cO_{\P^1}(4-2\nu))=2\nu-5$ for $i=1$ and $0$ otherwise. Hence by Lemma \ref{lem:h1}\ref{item:h1_exact} it is enough to show that $h^{i}(\lts{Y}{(D_Y-C)})$ for $i=0,1,2$ equals $0,2\nu-6,0$, respectively.
	
	Put $B_Y=D_Y-C+\sum_{i=1}^{\nu}(A_i+H_i)$, where $H_{i}$ is the proper transform of the horizontal curve passing through $\{q_i\}=C\cap V_i$. Since $A_i$ and $H_i$ are $(-1)$-curves, by Lemma \ref{lem:h1}\ref{item:h1-1_curve}  $h^{i}(\lts{Y}{D_Y-C})=h^{i}(\lts{Y}{B_Y})$. By Lemma \ref{lem:blowup-hi}\ref{item:blowup-hi-inner},  $h^{i}(\lts{Y}{B_Y})=h^{i}(\lts{\P^1\times \P^1}{B})$, where $B\de \phi_{*}B_Y$ consists of $\nu$ horizontal and $\nu$ vertical lines. We compute directly that $h^{i}(\lts{\P^1\times \P^1}{B})$ equals $0,2\nu-6,0$ for $i=0,1,2$, as needed.	
\end{proof}

\begin{lemma}[Classification in case \ref{lem:ht=2_twisted-swaps}\ref{item:ht=2_insep}, $\cha\kk=2$, see Table \ref{table:ht=2_char=2}]\label{lem:ht=2_twisted-inseparable}
Let $(X,D)$ be as in \eqref{eq:assumption_ht=2_width=1}, i.e.\  $(X,D)$ is a minimal log resolution of a del Pezzo surface $\bar{X}$ of rank one, height $2$ and width $1$. Fix a $\P^1$-fibration of $X$ as in Lemma \ref{lem:ht=2_twisted-swaps}, and assume that $\bar{X}$ swaps vertically to a surface $\bar{Y}$ from Example \ref{ex:ht=2_twisted_cha=2} for some $\nu\geq 3$.

Assume $\bar{X}$ is log canonical. Let $\cS$ be the singularity type of $\bar{X}$, and let $\Phtinsep(\cS)$ be the set of isomorphism classes of surfaces $\bar{X}$ of type $\cS$ and above properties. Let $\check{\cS}$ be the combinatorial type of the sum of $D$ and all vertical $(-1)$-curves. Then $h^{0}(\lts{X}{D})=0$, $h^{2}(\lts{X}{D})=\nu-2$, and one of the following holds.
	\begin{enumerate}[itemsep=1em]
		\item \label{item:7A1} The set $\Phtinsep(\cS)$ has moduli dimension $\nu-3$, and $\check{\cS}$ is one of the following (see Notations \ref{not:fibrations}, \ref{not:Dk,gfork}):
	\begin{longlist}[leftmargin=1em]
		\item\label{item:beta=0} $\ldec{1,\dots,\nu}[\bs{k_{1}^{\geq}-4}]+\sum_{j=1}^{\nu}\rD_{k_{j}}\dec{j}$,
		\item\label{item:beta=1_k=2} $\ldec{2,\dots,\nu}[\bs{k_{2}^{\geq}-2},l\dec{1},2]+\ldec{1}[(2)_{l-2},3]+\sum_{j=2}^{\nu}\rD_{k_{j}}\dec{j}$,
		\item\label{item:beta=1_k>2}
		$\dec{2,\dots,\nu}[\bs{k_{1}^{\geq}-3},T]\dec{1}+
		\gforkd{1}{T^{*}*[(2)_{k_{1}-1}]}+\sum_{j=2}^{\nu}\rD_{k_{j}}\dec{j}$,
		\item\label{item:beta=2_k1=k2=2}
		$[2,l_{1}\dec{1},\ldec{3,\dots,\nu}\,\! \bs{k_{3}^{\geq}},l_2\dec{2},2]+\ldec{1}[(2)_{l_{1}-2},3]+\ldec{2}[(2)_{l_{2}-2},3]+\sum_{j=3}^{\nu}\rD_{k_{j}}\dec{j}$,
		\item\label{item:beta=2_k1=2,k2>2}
		$[2,l\dec{1},\ldec{3,\dots,\nu}\,\!\bs{k_{2}^{\geq}-1},T]\dec{2}+\ldec{1}[(2)_{l-2},3]+
		\gforkd{2}{T^{*}*[(2)_{k_2-1}]}+\sum_{j=3}^{\nu}\rD_{k_{j}}\dec{j}$,
		\item\label{item:beta=2_k1>2,k2>2}
		$\ldec{1}[T_{1}\trp,\ldec{3,\dots,\nu}\,\!\bs{k_{1}^{\geq}-2},T_{2}]\dec{2}+
		\gforkd{1}{T_{1}^{*}*[(2)_{k_{1}-1}]}+
		\gforkd{2}{T_{2}^{*}*[(2)_{k_{2}-1}]}
		+\sum_{j=3}^{\nu}\rD_{k_{j}}\dec{j}$,
		\item\label{item:beta=3_notD_k1=k2=2_V1=V2=[2,2]}
		$\langle \ldec{4,\dots,\nu}\,\!\bs{k_{3}^{\geq}+1};[2,2]\dec{1},[2,2]\dec{2},[2]\dec{3}\rangle+
		\gforkd{3}{[3,(2)_{k_3-2}]}+
		[3]\dec{1}+[3]\dec{2}+\sum_{j=4}^{\nu}\rD_{k_{j}}\dec{j}$,
		\item\label{item:beta=3_notD_k1=k2=2_V1=[3,2],V2=[2,2]}
		$\langle \ldec{4,\dots,\nu}\,\!\bs{k_{3}^{\geq}+1};[2,3]\dec{1},[2,2]\dec{2},[2]\dec{3}\rangle+
		\gforkd{3}{[3,(2)_{k_3-2}]}+
		[3,2]\dec{1}+[3]\dec{2}+\sum_{j=4}^{\nu}\rD_{k_{j}}\dec{j}$,
		\item\label{item:beta=3_U2=U3=[2]_k1=2}
		$\langle \ldec{4,\dots,\nu}\,\!\bs{k_{2}^{\geq}};[2,l]\dec{1},[2]\dec{2},[2]\dec{3}\rangle+
		\sum_{j=2}^{3}\gforkd{j}{[3,(2)_{k_j-2}]}+\ldec{1}[(2)_{l-2},3]+\sum_{j=4}^{\nu}\rD_{k_{j}}\dec{j}$,
		\item\label{item:beta=3_notD_k1=2_V1=[3,2]_k2>2}
		$\langle \ldec{4,\dots,\nu}\,\!\bs{k_{2}^{\geq}};[2,3]\dec{1},\ldec{2}T\trp,[2]\dec{3}\rangle+
		\gforkd{2}{T^{*}*[(2)_{k_{2}-1}]}+
		[3,2]\dec{1}+
		\gforkd{3}{[3,(2)_{k_3-2}]}+\sum_{j=4}^{\nu}\rD_{k_{j}}\dec{j}$, $d(T)=3$,
		\item\label{item:beta=3_notD_k1=2_V1=[2,2]_k2>2}
		$\langle \ldec{4,\dots,\nu}\,\!\bs{k_{2}^{\geq}};[2,2]\dec{1},\ldec{2}T\trp,[2]\dec{3}\rangle+
		\gforkd{2}{T^{*}*[(2)_{k_{2}-1}]}+
		[3]\dec{1}+
		\gforkd{3}{[3,(2)_{k_3-2}]}+\sum_{j=4}^{\nu}\rD_{k_{j}}\dec{j}$, $d(T)\in \{3,4,5\}$,
		\item\label{item:beta=3_notD_kj>2}
		$\langle \ldec{4,\dots,\nu}\,\! \bs{k_{1}^{\geq}-1};\ldec{1}T_{1}\trp,\ldec{2}T_2\trp,[2]\dec{3}\rangle+
		\sum_{j=1}^{2}\! \! 
		\gforkd{j}{T_{j}^{*}*[(2)_{k_{j}-1}]}+
		\gforkd{3}{[3,(2)_{k_3-2}]}+\sum_{j=4}^{\nu}\rD_{k_{j}}\dec{j}$,  $\frac{1}{d(T_{1})}+\frac{1}{d(T_{2})}>\tfrac{1}{2}$,
		\item\label{item:C_[2]_L2H=1}
		$\langle k;[2]\dec{1},[2],[\bs{k_{2}^{\geq}-2}]\dec{2,\dots,\nu}\rangle+[3,(2)_{k-2}]*[2]\dec{1}+\sum_{j=2}^{\nu}\rD_{k_{j}}\dec{j}$,
		\item\label{item:C_[2]_k2=2}
		$\langle k;[2]\dec{1},[2],[2,l\dec{2},\bs{k_{3}^{\geq}}]\dec{3,\dots,\nu}\rangle+[3,(2)_{k-2}]*[2]\dec{1}+
		\ldec{2}[(2)_{l-2},3]
		+\sum_{j=3}^{\nu}\rD_{k_{j}}\dec{j}$,
		\item\label{item:C_[2]_k2>2}
		$\langle k;[2]\dec{1},[2],\ldec{2}[T\trp,\bs{k_{2}^{\geq}-1}]\dec{3,\dots,\nu}\rangle+[3,(2)_{k-2}]*[2]\dec{1}+
		\gforkd{2}{T^{*}*[(2)_{k_2-2}]}
		+\sum_{j=3}^{\nu}\rD_{k_{j}}\dec{j}$,
		\item \label{item:C_columnar_nu=5} $\langle k;[2],\ldec{1}T\trp,[\bs{6}]\dec{2,3,4,5}\rangle+[3,(2)_{k-2}]*(T^{*})\dec{1}+\sum_{j=2}^{5} ([2]\dec{j}+[2]\dec{j})$, $d(T)=3$, $k\geq 3$,
		\item \label{item:C_columnar_nu=4_k2=3} $\langle k;[2],\ldec{1}T\trp,[\bs{5}]\dec{2,3,4}\rangle+[3,(2)_{k-2}]*(T^{*})\dec{1}+[2,2\dec{2},2]+\sum_{j=3}^{4}([2]\dec{j}+[2]\dec{j})$, $d(T)=3$,
		\item \label{item:C_columnar_nu=4_k2=2} $\langle k;[2],\ldec{1}T\trp,[\bs{4}]\dec{2,3,4}\rangle+[3,(2)_{k-2}]*(T^{*})\dec{1}+\sum_{j=2}^{4}([2]\dec{j}+[2]\dec{j})$, $d(T)\in\{3,4\}$,
		\item \label{item:C_columnar_nu=3_k2=3,k3=2} $\langle k;[2],\ldec{1}T\trp,[\bs{3}]\dec{2,3}\rangle+[3,(2)_{k-2}]*(T^{*})\dec{1}+[2,2\dec{2},2]+[2]\dec{3}+[2]\dec{3}$, $d(T)\in \{3,4,5,6\}$,
		\item \label{item:C_columnar_nu=3_k2=k3=2} $\langle k;[2],\ldec{1}T\trp,[\bs{2}]\dec{2,3}\rangle+[3,(2)_{k-2}]*(T^{*})\dec{1}+\sum_{j=2}^{3}([2]\dec{j}+[2]\dec{j})$, $T\neq [2]$,
		\item \label{item:C_columnar_nu=3_k2+k3>5} $\langle k;[2],\ldec{1}T\trp,[\bs{k_2+k_3-2}]\dec{2,3}\rangle+[3,(2)_{k-2}]*(T^{*})\dec{1}+\sum_{j=2}^{3}\rD_{k_j}\dec{j}$, $d(T)= 3$, $k_2+k_3\in \{6,7\}$
		\item \label{item:C_columnar_L2H=0_k2=2} $\langle k;[2],\ldec{1}T\trp,[2,2\dec{2},\bs{2}\dec{3}]\rangle+[3,(2)_{k-2}]*(T^{*})\dec{1}+[3]\dec{2}+[2]\dec{3}+[2]\dec{3}$, $d(T)\in\{3,4\}$ and $k\geq 3$ if $d(T)=4$,
		\item \label{item:C_columnar_L2H=0_k2=3} $\langle k;[2],\ldec{1}T\trp,[2\dec{2},\bs{3}\dec{3}]\rangle+[3,(2)_{k-2}]*(T^{*})\dec{1}+[2,3\dec{2},2]+[2]\dec{3}+[2]\dec{3}$, $d(T)=3$,
	\setcounter{foo}{\value{longlisti}}
	\end{longlist}
	\item\label{item:4A1+D4} The set $\Phtinsep(\cS)$ has moduli dimension $\nu-2$, and $\check{\cS}$ is one of the following 
	\begin{longlist}[leftmargin=1em]
	\setcounter{longlisti}{\value{foo}}
		\item\label{item:C_fork_L2H=1}
		$\langle k\dec{1};[2],[2],[\bs{k_{2}^{\geq}-2}]\dec{2,\dots,\nu}\rangle+[(2)_{k-2}]\dec{1}+\sum_{j=2}^{\nu}\rD_{k_{j}}\dec{j}$,
		\item\label{item:C_fork_k2=2}
		$\langle k\dec{1};[2],[2],[2,l\dec{2},\bs{k_{3}^{\geq}}]\dec{3,\dots,\nu}\rangle+[(2)_{k-2}]\dec{1}+
		\ldec{2}[(2)_{l-2},3]
		+\sum_{j=3}^{\nu}\rD_{k_{j}}\dec{j}$,
		\item\label{item:C_fork_k2>2}
		$\langle k\dec{1};[2],[2],\ldec{2}[T\trp,\bs{k_{2}^{\geq}-1}]\dec{3,\dots,\nu}\rangle+[(2)_{k-2}]\dec{1}+
		\gforkd{2}{T^{*}*[(2)_{k_2-2}]}
		+\sum_{j=3}^{\nu}\rD_{k_{j}}\dec{j}$,
		\item\label{item:twisted-bench} $\ldec{1}\lbr k \rbr\decb{2}{3}+[(2)_{k-3},3]\dec{1}+\sum_{j=2}^{3} ([2]\dec{j}+[2]\dec{j})$, $k\geq 3$; here $D\hor$ is a twig of the bench,
	\end{longlist}
	\end{enumerate}
	for some integers $k,l,l_1,l_2,k_{1},\dots,k_{\nu}\geq 2$; $\nu\geq 3$; $k^{\geq}_{i}\de \sum_{j=i}^{\nu}k_j$, and admissible chains $T,T_1,T_2$; where $\nu=5$ in case \ref{item:C_columnar_nu=5}, $\nu=4$ in cases \ref{item:C_columnar_nu=4_k2=3}--\ref{item:C_columnar_nu=4_k2=2}, $\nu=3$ in cases \ref{item:C_columnar_nu=3_k2=3,k3=2}--\ref{item:C_columnar_L2H=0_k2=3} and \ref{item:twisted-bench}.
\end{lemma}
\begin{proof}
	Let $D_{H}$ be the connected component of $D$ containing $H$. By Lemma \ref{lem:ht=2_twisted-models}\ref{item:ht=2_twisted_H} we have $H=[\bs{k_{1}^{\geq}-4}]$. 
	
	Put $\beta=\beta_{D}(H)$. We have $\beta\leq 3$ by Lemma \ref{lem:ht=2_twisted-models}\ref{item:ht=2_rational}. We order $F_{1},\dots, F_{\nu}$ in such a way that $L_{j}\cdot H=0$ for $j\leq \beta$ and $L_{j}\cdot H=1$ for $j>\beta$. 
	If $\beta=0$ then $D$ is as in \ref{item:beta=0}, so we can assume $\beta\geq 1$. For $j\in \{1,\dots, \beta\}$ let $C_{j}\subseteq D\vert$ be the component of $F_{j}$ meeting $H$. Clearly, $C_j$ has multiplicity $2$ in $F_j$. If $C_j$ is a tip of $D\vert$ then as in the proof of Lemma \ref{lem:ht=2_twisted-separable} we infer that  either $k_j=2$, $(F_{j})\redd=[2,l,1,(2)_{l-2},3]$ for some $l\geq 2$; and $C_j=(F_{j})\redd\cp{2}$; or $k_j\geq 3$ $(F_{j})\redd=\langle 2; [U,1,U^{*}]*[(2)_{k_{j}-3}],[2],[2]\rangle$ for some admissible chain $U$, and $C_j=\ftip{U}$. 
	
	\begin{casesp}
	\litem{$C_j$ is a tip of $D\vert$ for all $j\in \{1,\dots,\beta\}$} If $\beta=1$ then $D$ is as in \ref{item:beta=1_k=2} if $k_1=2$ and \ref{item:beta=1_k>2} if $k_1>2$ (in the latter case, we have replaced $k_{1}$ by $k_{1}-1\geq 2$). Similarly, if $\beta=2$ then $D$ is as in \ref{item:beta=2_k1=k2=2} if $k_1=k_2=2$, \ref{item:beta=2_k1=2,k2>2} if $k_1=2$, $k_2>2$ and \ref{item:beta=2_k1>2,k2>2} if $k_1,k_2>2$. Assume now $\beta=3$, and so $D_{H}$ is a fork with twigs $V_j\supseteq C_{j}$, $j\in \{1,2,3\}$. Note that if $k_j=2$ for some $j\in \{1,2,3\}$ then $V_{j}=[2,l_{j}]$ for some $l_j\geq 2$; in particular, $\#V_{j}=2$.
	
	Since $\cf(H)<1$ by Lemma \ref{lem:delPezzo_fibrations}, the fork $D_{H}$ is admissible. In particular, by Lemma \ref{lem:admissible_forks}\ref{item:has_-2} we have, say, $V_{3}=[2]$, so $k_{3}>2$.  Assume $k_1=k_2=2$. Then, say, $l_2=2$ and $l_1=2$ or $3$, so $D$ is as in \ref{item:beta=3_notD_k1=k2=2_V1=V2=[2,2]} or \ref{item:beta=3_notD_k1=k2=2_V1=[3,2],V2=[2,2]}, respectively  (where, as before, we have replaced $k_3$ by $k_3-1\geq 2$). Assume now $k_1=2$, $k_2,k_3\geq 3$. If $V_2=V_3=[2]$ then $D$ is as in \ref{item:beta=3_U2=U3=[2]_k1=2}. If $V_{2}\neq [2]$ then since $\#V_1\geq 2$, we have  $V_1=[2,3]$ or $[2,2]$, hence $D$ is as in \ref{item:beta=3_notD_k1=2_V1=[3,2]_k2>2} or \ref{item:beta=3_notD_k1=2_V1=[2,2]_k2>2}; respectively. Eventually, if $k_j>2$ for all $j$ then $D$ is as in \ref{item:beta=3_notD_kj>2}.
	
	\litem{$C_{1}=[k]$ is not a tip of $D\vert$} Then $C_1$ is branching in $D_{H}$. Assume $\beta_{D}(C_1)\geq 4$. Then $D_{H}$ is a bench $\lbr k \rbr$, so $k\geq 3$ and $-2=H^2=4-k_{1}^{\geq}$ by Lemma \ref{lem:ht=2_twisted-models}\ref{item:ht=2_twisted_H}, which implies that $\nu=3$ and $k_1=k_2=k_3=2$. By Lemma \ref{lem:ht=2_twisted-models}\ref{item:ht=2_twisted_tangent}, we have $(F_1)\redd=\langle k;[(2)_{k-2},3,1,2],[2],[2]\rangle$, so $D$ is as in \ref{item:twisted-bench}. 
	
	Thus we can assume $\beta_{D}(C_1)=3$. Since $\cf(H)<1$ by Lemma \ref{lem:delPezzo_criterion}, $H$ lies in a twig $T_{H}$ of $D_{H}$. It follows that $L_{j}\cdot H=1$ for $j\geq 3$, and either $L_{2}\cdot H=1$ or $C_2$ is a tip of $D\vert$; hence $F_2$ is as in case 1. The connected component of $(F_1)\redd-L_1$ containing $C_1$, call it $V$, is a chain. Since $C_1$ is not a tip of $F_1$, $V$ contains a component of multiplicity one in $F_1$. By Lemma \ref{lem:degenerate_fibers}, the other connected component $(F_1)\redd-L_1-V$ is a chain, too, so $k_1=2$, and either $(F_1)\redd=\langle k;[2],[2],[(2)_{k-2},1]\rangle$ or $(F_1)\redd=[2,k,T,1,T^{*}]*[(2)_{k-1},3]$. Consider the first case. If $L_2\cdot H=1$ then $D$ is as in \ref{item:C_fork_L2H=1}. If $L_2\cdot H=0$ then $D$ is as in \ref{item:C_fork_k2=2} if $k_2=2$ and \ref{item:C_fork_k2>2} if $k_2>2$. Consider the second case. If $T=[2]$ then as before we get $D$ as in \ref{item:C_[2]_L2H=1}--\ref{item:C_[2]_k2>2}. Assume $T\neq [2]$, so $d(T_{H})\leq 6$. Consider the case $L_2\cdot H=1$. Then $T_{H}=H=[\bs{k_{2}^{\geq}-2}]$, so $8\geq k_{2}^{\geq}\geq 2\nu-2$. It follows that $\nu\leq 5$. If $\nu=5$ then $k_j=2$ for all $j$, and $D$ is as in \ref{item:C_columnar_nu=5}. If $\nu=4$ then, say, $k_{2}\leq 3$, $k_{3}=k_4=2$, and $d(T)=3$, so $D$ is as in \ref{item:C_columnar_nu=4_k2=3} or \ref{item:C_columnar_nu=4_k2=2}. If $\nu=3$ then either $k_{2}=k_{3}=2$ and $D$ is as in \ref{item:C_columnar_nu=3_k2=k3=2}, or $k_2=3$, $k_3=2$ and $D$ is as in \ref{item:C_columnar_nu=3_k2=3,k3=2}; or $k_2+k_3\in \{6,7,8\}$ and $D$ is as in \ref{item:C_columnar_nu=3_k2+k3>5}. Eventually, consider the case $L_{2}\cdot H=0$. If $k_2=2$ then $\#T_{H}=3$, so by Lemma \ref{lem:admissible_forks} $T_H=[2,2,\bs{2}]$, and $D$ is as in \ref{item:C_columnar_L2H=0_k2=2}. Assume $k_2\geq 3$. Then $H^2\leq -3$, so $T_H$ is not a $(-2)$-twig; and $\#T_{H}\geq 2$; so by Lemma \ref{lem:admissible_forks} $T_{H}=[2,\bs{3}]$ and $D$ is as in \ref{item:C_columnar_L2H=0_k2=3}.
	\end{casesp}

Thus we got the list of singularity types $\cS$. The equalities $h^i(\lts{X}{D})=0,\nu-2$ for $i=0,2$ follow from Lemmas \ref{lem:7A1-family}\ref{item:7A1-family-h2} and \ref{lem:blowup-hi}. It remains to construct the almost universal families representing each $\Phtinsep(\cS)$. 

We argue as in the proof of Corollary \ref{cor:moduli-ht=1}. Let $\Phtinsepres(\cS)$ be the set of minimal log resolutions of surfaces in $\Phtinsep(\cS)$. For $(X,D)\in \Phtinsepres(\cS)$ let $\check{D}$ be the sum of $D$ and the vertical $(-1)$-curves found above, so $\check{\cS}$ is the combinatorial type of $(X,\check{D})$. As usual, since type $\cS$ appears exactly once on our list, it determines $\check{\cS}$ uniquely, i.e.\ $\Phtinsepres(\cS)$ is the image of $\cP(\check{\cS})$ in $\cP(\cS)$. Let $\psi\colon (X,D)\to (Y,\check{D}_Y)$ be the morphism onto the surface from Example \ref{ex:ht=2_twisted_cha=2}. As in  Figure \ref{fig:KM_surface} we denote by $L_j$ the proper transform of $\ll_j$, and by $A_j$ the $(-1)$-curve over $\ll_j\cap\cc$. Let $\check{\cS}_0$ be the combinatorial type of $(Y,\check{D}_Y)$, and let $\cR_0\subseteq \cP_{+}(\check{\cS}_0)$ be the subset from Lemma \ref{lem:7A1-family}\ref{item:7A1-family-R}, distinguished by choosing a $(-2)$-tip of each degenerate fiber in $\check{D}_Y$. By Lemma \ref{lem:7A1-family}, $\cR_0$ is represented by a universal $(\Z/2\rtimes S_{\nu})$-faithful family over a base $M$ of dimension $\nu-3$. Let $\cR$ be the preimage of $\cR_0$ in $\cP_{+}(\check{\cS})$, and let $G$ be the maximal subgroup of $\Z/2\rtimes S_{\nu}$ such that $\psi$ is $G$-invariant. 

Write $\psi=\psi_{\textnormal{out}}\circ \psi_{\textnormal{in}}$, where $\psi_{\textnormal{in}}$ is a vertical inner minimalization. In case \ref{item:7A1} we have $\psi_{\textnormal{out}}=\id$, so by Lemma \ref{lem:inner} $\cR$ is represented by a universal $G$-faithful family over the same base $M$ as $\cR_0$. Consider case \ref{item:4A1+D4}. Then $\psi_{\textnormal{out}}$ is a single outer blowup on $A_{1}^{\circ}\de A_{1}\setminus (\check{D}_Y-A_{1})$. We note that the only elements of $\Aut(Y,\check{D}_Y)$ which fix $A_1$ are the identity and the involution switching the two points $A_1\cap (D_Y)\vert$, which is the restriction of the $\Z/2$-action from Lemma \ref{lem:7A1-family}. Thus by Lemmas  \ref{lem:inner} and \ref{lem:outer}\ref{item:outer-trivial} $\cR$ is represented by a universal $G$-faithful family $f$ whose base $B$ admits a fibration $B\to M$ whose fibers are identifies with curves $A_1^{\circ}$ as above. 

Thus we get a universal $G$-faithful family representing $\cR$, of dimension $\nu-3$ in case \ref{item:7A1} and $\nu-2$ in case~\ref{item:4A1+D4}.

\begin{claim*}
	The restriction of the forgetful map $\cP_{+}(\check{\cS})\to \cP(\check{\cS})$ to $\cR$ is surjective. 
\end{claim*}
\begin{proof}
For $(X,\check{D})\in \cP(\check{\cS})$ we need to find $(X,\check{\cD},\gamma)\in \cP_{+}(\check{\cS})$ such that the triple $(Y,\check{D}_Y,\gamma_0)$, with $\gamma_0$ induced from $\gamma$ by the morphism $\psi$, is equivalent to some  $(Y,\check{D}_Y,\tilde{\gamma}_0)\in \cR_0$, see Section \ref{sec:moduli} for definitions.

Let $V_j^1,V_j^2$ be the two tips of multiplicity one in the fiber $F_j$. The isomorphism $\gamma$ is unique up to reordering those fibers $F_j$ which have the same weighted graph, and reordering $\{V_j^{1},V_j^{2}\}$ whenever the weighted graph of $F_j$ has an involution interchanging the corresponding vertices. In the latter case, call $F_j$ \emph{symmetric}. We check directly that there is at least one symmetric fiber. 

We need to show that, after possibly replacing $V_j^{1}$ with $V_j^{2}$ in some symmetric fibers, there is $\alpha\in \Aut(Y,\check{D}_Y)$ such that $\psi_{*}V_j^{1}=\alpha(L_j)$ for every $j\in \{1,\dots,\nu\}$.  
Let $J_{\alpha}^{i}=\{j: \psi_{*}V_j^{i}=\alpha(L_j)\}$, clearly $J_{\alpha}^{1}\sqcup J_{\alpha}^{2}=\{1,\dots,\nu\}$. Applying $\Z/2$-action  from Lemma \ref{lem:7A1-family} we see that for every pair of distinct indices $j,j'$ we can change $\alpha$ and move both $j$ and $j'$ between $J_{\alpha}^{2}$ and $J_{\alpha}^{1}$. Thus we can choose $\alpha$ so that either  $J_{\alpha}^{2}=\emptyset$, as needed, or $J_{\alpha}^{2}=\{j\}$ and $F_j$ is symmetric, in which case we are done after  replacing $V_{j}^{1}$ with $V_{j}^{2}$. 
\end{proof}

The above Claim implies that 
$f$ represents $\cP(\check{\cS})$. Recall that the image of $\cP(\check{\cS})$ in $\cP(\cS)$ equals $\Phtinsepres(\cS)$, so $f$ restricts to a family representing $\Phtinsepres(\cS)$. If $d=0$ this means that $\#\Phtinsepres(\cS)=1$, as needed, so we can assume $d\geq 1$. We need to prove that $f$, viewed as a family representing  $\Phtinsepres(\cS)$, is almost faithful. 

Let $(X_{i},\check{D}_{i},\gamma_{i})\in \cR$ be a fiber of $f$ over some $b_i\in B$, for $i\in \{1,2\}$. Let $E_{i}$ be the sum of vertical $(-1)$-curves in $\check{D}_{i}$, let $D_{i}=\check{D}_{i}-E_{i}$, $H_i=(D_i)\hor$, let $F_{j,i}$ be the fibers in $\check{D}_i$, ordered by $\gamma_i$; and let $F_i\subseteq X_i$ be the general fiber. Fix an isomorphism  $\phi\colon (X_{1},D_{1})\to (X_{2},D_{2})$.

\begin{casesp}
\litem{$\phi(H_{1})=H_2$} We argue as in the last part of the proof of Lemma \ref{lem:7A1-family}. By Lemma \ref{lem:adding-1-criterion} we have $\phi(\check{D}_1)=\check{D}_2$. Replacing $b_2$ by another element in its $G$-orbit we can reorder the fibers $F_{2,j}$ in such a way that $\phi(H_1\cap F_{1,j})=H_2\cap F_{2,j}$. Now $\phi$ descends to an isomorphism between fibers of $f_0$ which maps the image of $H_1$ to the image of $H_2$, and preserves their $\nu$ marked points. The property of $f_0$ stated in Lemma \ref{lem:7A1-family} shows that the images of $b_1,b_2$ in $M$ are the same. In case \ref{item:7A1} we have $B=M$, so the result follows. In case  \ref{item:4A1+D4} we get that $b_1,b_2$ lie in the same fiber of $B\to M$, see Lemma \ref{lem:outer}\ref{item:outer-trivial}. This fiber is identified with $A_{1}^{\circ}$, and this identification maps $b_i$ to the cross-ratio of the four common points of the proper transform of $A_1$ with the remaining part of $\check{D}_{i}$. These points are ordered by $\gamma^{i}$, and lie, respectively: on $H_i$, on the twig with a $(-1)$-curve, and on two twigs $[2]$ which are interchanged by an involution in $G$ coming from the $\Z/2$-action in Lemma \ref{lem:7A1-family}. Thus the cross-ratios $b_1$, $b_2$ lie in the same orbit of the $G$-action, as needed.  

\litem{$\phi(H_1)\neq H_2$} Now $\phi(H_1)$ is vertical, so $\phi(F_1)$ is horizontal. Since $\phi(H_1)$ is vertical, we have $\phi(F_1)\cdot F_2\geq \phi(F_1)\cdot \phi(H_1)=F_1\cdot H_1=2$. For every component $C$ of $D_2-H_2-\phi(H_1)$ we have $F_2\cdot C=0=\phi(F_1)\cdot C$ because $\phi^{-1}(C)$ is vertical. Since $K_{X_2}$ and components of $D_2$ generate the group $\NS_{\Q}(X_2)$, we conclude that $\phi(F_1)- F_2\equiv a\cdot (H_2-H_1)$ for some $a\in \Q$. Intersecting the above numerical equivalence with $F_2$ and $\phi(H_1)$ we get $2\leq \phi(F_1)\cdot F_2=2a$, so $a\geq 1$, and $2=a(H_2\cdot \phi(H_1)-H_1^2)\geq H_2\cdot \phi(H_1)-H_1^2\geq \sum_{j=1}^{\nu} k_{j}-4\geq 2\nu-4$ by Lemma \ref{lem:ht=2_twisted-models}\ref{item:ht=2_twisted_H}. Thus $\nu=3$ (so we are in case \ref{item:4A1+D4} by assumption $d\geq 1$), $k_{1}=k_{2}=k_{3}=2$; and the weighted graph of $D_1$ has an automorphism which does not preserve $H_1$. 
\smallskip

We check directly that these conditions can hold only in cases \ref{item:C_fork_L2H=1} and \ref{item:twisted-bench}. 
In these cases, the morphism $\psi\colon (X,\check{D})\to (X,\check{D}_Y)$ has only one exceptional $(-1)$-curve, call it $U$.

Recall that, since  we are in case \ref{item:4A1+D4} and $\nu=3$, the base $B$ of our family $f$ is the curve $A_1^{\circ}=A_1\setminus D_Y\cong \A^{1}_{**}$, see Figure \ref{fig:KM_surface}. 
Put $\hat{D}=D+U$, $\hat{D}_{Y}=\psi_{*}\hat{D}=D_Y+A_1$. Let $\hat{\cS}$, $\hat{\cS}_0$ be the combinatorial types of $(X,\hat{D})$ and $(Y,\hat{D}_Y)$; and let $\hat{\cR}$, $\hat{\cR}_0$ be the images of $\cR$ and $\cR_0$ in $\cP_{+}(\hat{\cS})$ and $\cP_{+}(\hat{\cS}_0)$, respectively. Case $\nu=3$ of Lemma \ref{lem:7A1-family}, or the proof above, shows that $\#\hat{\cR}_0=1$.
The restriction of $\Aut(Y,\hat{D}_Y)$ to $\Aut(A_{1}^{\circ})\cong S_3$ is surjective, see \cite[Lemma 8.2(e)]{PaPe_MT}. Applying Lemma \ref{lem:outer}\ref{item:outer-trivial} to a family representing $\hat{\cR}_0$ (whose base is a point), we get an $S_3$-faithful family over $A_1^{\circ}$ representing $\hat{\cR}$ (which is in fact the restriction of $f$). 

Let $(X_i,\hat{D}_i)$ be its fiber over $b_i\in A_{1}^{\circ}$, $i\in \{1,2\}$, let $U_i$ be the $(-1)$-curve in $\hat{D}_i$, let $D_i=\hat{D}_i-U_i$, and as before let $A_{1,i}$ be the proper transform of $A_{1}$ in $\hat{D}_i$. The point $b_i$ is determined by the cross-ratio of the quadruple $A_{1,i}\cap (\hat{D}_i-A_{1,i})$. Since $A_{1,i}$ is the unique component of $\hat{D}_i$ of branching number $4$, any isomorphism $(X_1,\hat{D}_1)\to (X_1,\hat{D}_2)$ maps $A_{1,1}$ to $A_{1,2}$, so the corresponding cross-ratios are equivalent by the $S_3$-action. Hence our $S_3$-faithful family representing $\hat{\cR}$ is an almost faithful family representing $\cP(\hat{\cS})$. 

To prove that it is almost faithful as a family representing $\Phtinsepres(\cS)$, which is the image of $\cP(\hat{\cS})$ in $\cP(\cS)$, we need to prove that every isomorphism $\phi\colon (X_1,D_1)\to (X_2,D_2)$ satisfies $\phi(U_1)=U_2$. 

Let $R_2$ be the connected component of $D_2$ containing $H_2$. The $(-1)$-curve $\phi(U_{1})$ meets $D_{2}$ only in $R_{2}$ and in the last tip of a connected component $T_2$ of $D_2$, where $T_2=[(2)_{k-2}]$ in case  \ref{item:C_fork_L2H=1} and $T_2=[(2)_{k-3},3]$ in case \ref{item:twisted-bench}. The fork $T_2+\phi(U_{1})+(R_2-H_2)$ supports a fiber of a $\P^{1}$-fibration of $X_2$, such that $(D_{2})\hor$ consists only of the $2$-section $H_2$. Since the components of $D_2$ and $K_{X_2}$ generate $\NS_{\Q}(X_2)$, we conclude that this $\P^1$-fibration of $X_2$ is the same as the original one (up to an automorphism of the base). Since $U_2$ is the unique vertical $(-1)$-curve which is disjoint from $D_2-R_2-T_2$, we have $U_2=\phi(U_{1})$, as claimed. 
\qedhere\end{casesp}
\end{proof}

The proof of Theorem  \ref{thm:ht=1,2} and Propositions \ref{prop:moduli}, \ref{prop:moduli-hi} is now complete. For the readers' convenience, we gather here the references to lemmas covering each case of those results.

\begin{proof}[Proof of Theorem  \ref{thm:ht=1,2} and Propositions \ref{prop:moduli}, \ref{prop:moduli-hi}]
	The assertions in Theorem \ref{thm:ht=1,2} about vertical swaps are proved in Lemma \ref{lem:ht=1_basics}\ref{item:ht=1_swap} in case $\height=1$, Lemma \ref{lem:ht=2_swaps} in case $\height=2$, $\width=2$ and Lemma \ref{lem:ht=2_twisted-swaps} in case $\height=2$, $\width=1$. 
	
	Let $\cS$ be a log canonical singularity type. If $\cS$ is non-rational it is a type of an elliptic cone. In this case, by Proposition \ref{prop:non-log-terminal} every surface in $\Pht(\cS)$ is of height one, hence by Lemma \ref{lem:ht=1_basics}\ref{item:ht=1_swap} swaps vertically to an elliptic cone and so, in fact, is an elliptic cone, as claimed. Assume that $\cS$ is rational. Let $\Phtt(\cS)$, $\Phtw(\cS)$, $\Phtsep(\cS)$ and $\Phtinsep(\cS)$ be the subsets of $\Pht(\cS)$ described in Lemmas \ref{lem:ht=1_types}, \ref{lem:ht=2,untwisted}, \ref{lem:ht=2_twisted-separable} and \ref{lem:ht=2_twisted-inseparable}, respectively; so $\Pht(\cS)$ is their union. We check directly that the lists of singularity types in those lemmas are disjoint, so for each $\cS$ exactly one of the above sets is nonempty, and therefore equal to $\Pht(\cS)$.
	
	In case $\Pht(\cS)=\Phtt(\cS)$, the list of possible types $\cS$ is given in Lemma \ref{lem:ht=1_types}, and the required properties of $\Pht(\cS)$ -- namely the number $\#\Pht(\cS)$ and the existence of representing families -- are proved in  Corollary \ref{cor:moduli-ht=1}. In the remaining cases they are proved in Lemmas  \ref{lem:ht=2,untwisted}, \ref{lem:ht=2_twisted-separable} and \ref{lem:ht=2_twisted-inseparable}, respectively. 
	
	The numbers $h^{1}(\lts{X}{D})$ for minimal log resolutions $(X,D)$ of surfaces in $\Pht(\cS)$ are computed in the quoted results, too. By Lemma \ref{lem:h1}\ref{item:h1-h2-fibration} the numbers $h^{2}(\lts{X}{D})$ are zero in all cases except when $\Pht(\cS)=\Phtinsep(\cS)$: in the latter case, they are computed in Lemma \ref{lem:ht=2_twisted-inseparable}.
\end{proof}

\begin{remark}[Bases and symmetry groups of almost faithful families in case $\height=2$]\label{rem:ht=2_bases}
	Let $\cS$ be a singularity type of a log canonical del Pezzo surface of rank one and height $2$ such that $\Pht(\cS)$ is infinite. Then Lemmas \ref{lem:ht=2,untwisted}, \ref{lem:ht=2_twisted-separable}, \ref{lem:ht=2_twisted-inseparable} imply that $\Pht(\cS)$ has moduli dimension $\nu-3+\epsilon$, where $\nu\geq 3$ is the number of degenerate fibers of the witnessing $\P^1$-fibration constructed in the above lemmas; we put $\epsilon=1$ in case \ref{lem:ht=2_twisted-inseparable}\ref{item:4A1+D4} and $\epsilon=0$ otherwise. Let $f\colon (\cX,\cD)\to B$ be the almost faithful family representing $\Pht(\cS)$ constructed above. We now discuss its base $B$ and symmetry group $G\leq \Aut(B)$.
	
	In each case, the family $f$ is constructed from a family $U_{\nu}\to M_{\nu}$ representing the combinatorial type of $\P^1\times \P^1$ with $\nu$ fibers, see Example \ref{ex:4-points} for the case $\nu=4$. The base $M_{\nu}$ is the (open part of the) moduli space of genus $0$ curves with $\nu$ marked points: explicitly, $M_{\nu}$ can be realized as the complement in $\P^{\nu-3}$ of the union of hyperplanes spanned by all $(\nu-2)$-tuples of fixed $(\nu-3)$ points in a general position, see \cite{Kapranov-moduli} for details. The symmetry group $H_{\nu}$ of the family $U_{\nu}\to M_{\nu}$ is the image in $\Aut(M_{\nu})$ of the group $S_{\nu}$ permuting the fibers: we have $H_{\nu}\cong S_{\nu}$ if $\nu\geq 5$; $H_{4}\cong S_{3}$, $M_4\cong \P^1\setminus \{0,1,\infty\}$, cf.\ Example \ref{ex:4-points-Aut}; and $H_3=\{\id\}$, $M_3=\{\pt\}$.
	
	In case $\epsilon=0$ the base $B$ of $f$ is still $M_{\nu}$, and its symmetry group $G$ is the image in $H_{\nu}$ of the product $\prod_{i}S_{k_i}\leq S_{\nu}$, where each factor permutes degenerate fibers of the same type (i.e.\ such that there is an automorphism of $\cS$ mapping one fiber to another). For instance, assume $\cha\kk\neq 2$. Then $\cS$ is as in Lemma \ref{lem:ht=2,untwisted}\ref{item:ht=2_2D4}, so $\epsilon=0$, $\nu=4$. Thus in this case $B=M_{4}\cong \P^1\setminus \{0,1,\infty\}$, and $G$ is isomorphic to $S_3$ in cases \ref{lem:ht=2,untwisted}\ref{item:ht=2_bench} and \ref{lem:ht=2,untwisted}\ref{item:c=3} with $T_1=T_2=[2]$, $\Z/2$ in case \ref{lem:ht=2,untwisted}\ref{item:c=3} with $T_1=T_2\neq [2]$ or $T_1\neq T_2=[2]$, and trivial otherwise.
	
	Assume $\epsilon=1$. Then $\bar{X}\in \Pht(\cS)$ swaps vertically to the surface $\bar{Y}$ from Example \ref{ex:ht=2_twisted_cha=2}, and this vertical swap factors through one outer blowup, centered at a point $p$ on a vertical $(-1)$-curve, say $p\in A_{1}^{\circ}\de A_1\setminus D_Y\cong M_4$, see Figure \ref{fig:KM_surface}. If $\cS=4\rA_1+\rD_4$ then $\nu=3$, $B=A_{1}^{\circ}$ and $G=\Aut(A_1^{\circ})\cong S_3$, see Example \ref{ex:remaining_canonical}\ref{item:7A1_tower}. In all other cases we have $B=M_{\nu}\times A_1^{\circ}$ and $G\cong \Sigma\times \Z/2$. The first factor $\Sigma\leq S_{\nu}$ is a product of symmetric groups, acting on $M_{\nu}$ as in case $\epsilon=0$: note that $\Sigma$ fixes the first degenerate fiber, so in fact $\Sigma\leq S_{\nu-1}$. The second factor $\Z/2$ acts on $A_1^{\circ}$ as the automorphism of $(Y,D_Y)$ switching the two points $A_1\cap (D_Y)\vert$, cf.\ Figure \ref{fig:KM_alt}.
\end{remark} 

\clearpage
\section{Del Pezzo surfaces admitting descendants with elliptic boundary} \label{sec:GK}

In this section we prove Theorem \ref{thm:GK}, that is, we describe log canonical del Pezzo surfaces of rank one and height at least $3$ admitting descendants with elliptic boundary. Recall from Definition \ref{def:GK} that a normal projective surface $\bar{X}$ \emph{has a descendant with elliptic boundary} if there is a morphism $\phi$ from the minimal log resolution of $\bar{X}$ onto a log surface $(\bar{Y},\bar{T})$ where  $\bar{T}\subseteq \bar{Y}\reg$ is a curve of arithmetic genus one. If $\bar{X}$ is a del Pezzo surface of rank one, further properties of $\phi$ are stated in Lemma \ref{lem:GK_intro}. We summarize these properties, and subsequent Notation \ref{not:GK}, in the following diagram.

\begin{equation}\label{eq:GK-diagram}
	\begin{tikzcd}[row sep=0, column sep = 0]
	 (X,D_{Y}+D_{T}) 
		\ar[r, "\sigma"', 
		"{\Exc\sigma=D_{T}+L-T,\ L=[1],\ L\not \subseteq D}"
		] 
		\ar[rd, "\phi", "\substack{T\de \phi^{-1}_{*}\bar{T}\\ D_T\de (\phi^{*}\bar{T})\redd-L}"', end anchor={[xshift=1ex,yshift=-1ex]north west}]
		\ar[d,"\textnormal{min. res.}"']
		&[12em]
		(Y, \sigma_{*}D_{Y}+\sigma_{*}T) 
		\ar[d, "\alpha"', 
		"\substack{
			\textnormal{min. res. of } \bar{Y} \\[0.3em]
			\Exc\alpha\ =\ \sigma_{*}D_Y  
			}"
		]   \\[3em] 
		\bar{X} &  (\bar{Y},\bar{T})
		\\[-0.5em]
		\scriptsize{\textnormal{del Pezzo of rank 1}} 
		 & 
		\scriptstyle{\bar{Y}\ -\ \textnormal{canonical del Pezzo of rank 1}} \\[-0.3em] & \scriptstyle{\bar{T}\subseteq \bar{Y}\reg,\ p_{a}(\bar{T})=1,\ \bar{T}\in |-K_{\bar{Y}}|}
	\end{tikzcd}
\end{equation}

\begin{proof}[Proof of Lemma \ref{lem:GK_intro}]
	Assume that $\bar{T}$ is smooth and $\bar{X}$ is del Pezzo. Then $T$ is a component of $D$ of genus $1$, so by Proposition \ref{prop:non-log-terminal} we have $\height(\bar{X})=1$. Thus $X$ admits a $\P^{1}$-fibration $X\to B$ such that $T$ is a $1$-section, so $B\cong T$ is a curve of genus $1$. If all fibers are non-degenerate then $\bar{X}$ is an elliptic cone, see Example \ref{ex:ht=1}, as needed. Suppose there is a degenerate fiber $F$. Then by  Lemma \ref{lem:ht=1_basics}\ref{item:ht=1_Sigma},\ref{item:ht=1_nu} $F$ meets $T$ in a connected component $C$ of $D\vert$ which contains no $(-1)$-curves, hence does not blow down to a smooth point. However, $\phi(C)$ is a point of $\bar{T}$, hence a smooth point of $\bar{Y}$. Thus $\phi^{-1}(\phi(C))$ contains a horizontal $(-1)$-curve. This is a contradiction, as the genus of each horizontal curve equals at least the genus of $B$, which is $1$.
	
	Assume now that $\bar{T}$ has a singular point $q$. Since $T$ is smooth, $q$ is a base point of $\phi^{-1}$, and since $q$ is a smooth point of $\bar{Y}$, its preimage contains a $(-1)$-curve $L$. Since $D$ contains no $(-1)$-curves, $L$ is a component of $L'\de \Exc\phi-(D-T)$. We have $\#L'=(\rho(X)-\rho(\bar{Y}))-(\#D-1)=\rho(\bar{X})-\rho(\bar{Y})+1=2-\rho(\bar{Y})\leq 1$, so we get $L'=L$, which proves \ref{item:GK-intro-L}, and $\rho(\bar{Y})=1$, which together with \cite[Lemma 2.11(b)]{PaPe_MT} proves \ref{item:GK-intro-Y}.
	
	Consider the case $\bar{Y}\cong \P^2$. Part \ref{item:GK-intro-Y} shows that $\bar{T}$ is a cubic. 	Part \ref{item:GK-intro-L} shows that $\phi$ factors through the blowup at $q$, so the pencil of lines through $q$ induces a $\P^1$-fibration of $X$ whose fiber $F$ satisfies $F\cdot D=F\cdot T=2$. Thus $\height(\bar{X})\leq 2$, as claimed in \ref{item:GK-intro-ht}. 
	
	Consider the case $\bar{Y}\not\cong \P^2$. Write $\phi=\alpha\circ\sigma$, where $\alpha\colon (Y,D_Y')\to (\bar{Y},0)$ is a minimal log resolution. Fix a witnessing $\P^1$-fibration of $(Y,D_Y')$. Let $F_Y$ be its non-degenerate fiber which does not pass through the singular point of $T_Y\de \alpha^{-1}_{*}\bar{T}$. By \ref{item:GK-intro-Y}, $F_Y\cdot T_Y=-F_Y\cdot K_Y=2$. Part \ref{item:GK-intro-L} shows that $\sigma$ is an isomorphism near $F\de \sigma^{*}F_Y$, so $F=[0]$, $F\cdot T=2$ and $F$ meets $D-T$ only in $D_Y\de \tau^{-1}_{*}D_Y'$. Thus $\height(\bar{X})\leq F\cdot D\leq F_Y\cdot (T_Y+D_Y')=2+\height(\bar{Y})$. Table \ref{table:canonical} shows that $\height(\bar{Y})\leq 4$, and $\height(\bar{Y})\leq 2$ if $\cha\kk\neq 2,3$, which ends the proof of \ref{item:GK-intro-ht}. 
	
	In any case, we have constructed a $\P^{1}$-fibration of $X$ such that $D\hor$ consists of a $2$-section $T$ and components with $\cf=0$. Lemma \ref{lem:delPezzo_criterion} shows that $\bar{X}$ is del Pezzo if and only if $\cf(T)<1$, as claimed in \ref{item:GK-intro-cf}.
\end{proof}

\begin{notation}\label{not:GK}
	Let $\bar{X}$ be a log canonical del Pezzo surface of rank one and singularity type $\cS$, admitting a descendant $(\bar{Y},\bar{T})$ with elliptic boundary such that $\bar{T}$ is singular. Let $\phi\colon (X,D+L)\to (\bar{Y},\bar{T})$ be the morphism from Definition \ref{def:GK}, where $L$ is the $(-1)$-curve from Lemma \ref{lem:GK_intro}\ref{item:GK-intro-L}. We call the curve $L$ an \emph{elliptic tie}.
		
	Write $\phi=\alpha\circ\sigma$, where $\alpha\colon Y\to \bar{Y}$ is the minimal resolution, see diagram \eqref{eq:GK-diagram}. Write $D=D_{Y}+D_T$, where $D_{Y}=\sigma^{-1}_{*}(\Exc\alpha)$: note that $\sigma$ is an isomorphism near $D_Y$.  Let $T\de \phi^{-1}_{*}\bar{T}\subseteq D_{T}$. Let $\cS_{Y}$ and $\cT$ be the weighted graphs of $D_Y$ and $D_T$, respectively. Hence $\cS=\cS_Y+\cT$, and $\cS_Y$ is the singularity type of $\bar{Y}$. Put $s\de \#\cS_Y$. By Noether's formula we have 
	\begin{equation}\label{eq:Noether}
		\bar{T}^2=K_{Y}^2=10-\rho(Y)=9-s.
	\end{equation}
\end{notation}

\subsection{Singularity types}

We now determine the possible singularity types $\cS$. 
Lemma \ref{lem:cuspidal_resolution} describes part $\cT$, 
obtained by blowing up over the singular point of $\bar{T}$. Assumption $s\geq 6$ is justified by subsequent Lemma \ref{lem:GK_exceptions}.

\begin{lemma}[Possible types $\cT$]\label{lem:cuspidal_resolution}
Let the notation be as in \ref{not:GK}.	Assume $s\geq 6$. 
Then the following hold.
	\begin{enumerate}
		\item\label{item:T_nodal} If $\bar{T}$ is nodal or $\#\cT\leq 2$ then $\cT$ is one of the following:
		\begin{longlist}
			\item\label{item:C_1} $[s-5]\adec{1,1}$ if $s\geq 7$,			
			\item\label{item:C_2} $\ldec{1}[V,s-5]*(V^{*})\adec{1}$.
		\setcounter{foo}{\value{longlisti}}
		\end{longlist}
		\item\label{item:T_cuspidal} If $\bar{T}$ is cuspidal and $\#\cT\geq 3$ then $\cT$ is one of the following:
		\begin{longlist}
		\setcounter{longlisti}{\value{foo}}
			\item\label{item:C_3} $[s-3]\adec{1}+[3]\adec{1}+[2]\adec{1}$,
			\item\label{item:C_chain_[2]} $[s-3,k\adec{1},3]+[3,(2)_{k-2}]\adec{1}$,
			\item\label{item:C_chain_[3]} $[s-3,k\adec{1},2]+[4,(2)_{k-2}]\adec{1}$,
			\item\label{item:C_chain_[rest]} $[3,k\adec{1},2]+[s-2,(2)_{k-2}]\adec{1}$,
			\item\label{item:C_[2]} $\langle k;[s-3],[3],V\adec{1}\rangle+[3,(2)_{k-2}]*(V^{*})\adec{1}$, where either $V=[2]$ or $s=6$ and $d(V)\leq 3$,
			\item\label{item:C_[3]} $\langle k;\ldec{1}V,[s-3],[2]\rangle+[4,(2)_{k-2}]*(V^{*})\adec{1}$, where $d(V)\leq 6,4,3$ if $s=6,7,8$, respectively,
			\item\label{item:C_[rest]} $\langle k; \ldec{1}V,[3],[2] \rangle + [s-2,(2)_{k-2}]*(V^{*})\adec{1}$, where $d(V)\leq 6$,
			\item\label{item:C_smooth} $\langle k\adec{1};[2],[3],[s-3]\rangle+[(2)_{k-2}]\adec{1}$, or $\langle 2\adec{1};[2],[3],[s-3]\rangle$ if $k=2$,
		\end{longlist}
	\end{enumerate}	
	where $k\geq 2$, $V$ is an admissible chain; $V^{*}$ is its adjoint, see formula \eqref{eq:adjoint_chain}; and the symbol $*$ is introduced in formula \eqref{eq:convention-Tono_star}. The numbers decorated by $\adec{1}$ (respectively, $\adec{1,1}$) refer to the components meeting the elliptic tie $L$ once (resp.\ twice). Decoration at an end of a chain refers to the corresponding tip, cf.\ Notation \ref{not:fibrations}.
\end{lemma}
\begin{proof}
	\ref{item:T_nodal}  Since $\cf(T)<1$ by Lemma \ref{lem:GK_intro}\ref{item:GK-intro-cf}, the connected component of $D_T$ containing $T$ is not circular. It follows that $D_{T}+L$ is circular, and \ref{item:T_nodal} holds.
	
	\ref{item:T_cuspidal} 
	Let $\theta$ be the minimal embedded resolution of the cusp of $\alpha^{-1}_{*}\bar{T}$. Then $(\theta^{*}\bar{T})\redd$ is a fork $F$ with a branching $(-1)$-curve $B$ and twigs $T_{0}=[s-3]$, $T_{2}=[2]$, $T_{3}=[3]$. If $\#\cT=3$ then $\sigma=\theta$ and \ref{item:C_3} holds. Assume $\#\cT\geq 4$, so $\sigma$ is a composition of $\theta$ with a sequence of blowups over some $p\in B$. 
	If all these blowups are centered on the proper transforms of $B$ then $\cT$ is as in \ref{item:C_chain_[2]}, \ref{item:C_chain_[3]}, \ref{item:C_chain_[rest]} if $p\in T_{2},T_{3},T_{0}$; respectively, or as in \ref{item:C_smooth} if $p\in F\reg$. Assume now that one of these blowups is not centered on a proper transform of $B$. If $p\in F\reg$ then the proper transform $B'$ of $B$ satisfies $\beta_{D_T}(B')\geq 4$ and meets a $(-3)$-curve, which is impossible. Hence $p\in T_{0},T_{2}$ or $T_{3}$, and $\beta_{D_T}(B')=3$. Moreover, $\cT$ equals $\langle k;T_{x},T_{y},V\rangle+T_{z}*[(2)_{k-1}]*V^{*}$ or $\langle k;T_{x},T_{y},T_{z}*[(2)_{k-1}]*[V^{*},l,V]\rangle+[(2)_{l-2}]$, where $V$ is an admissible chain, $k,l\geq 2$ and $\{x,y,z\}=\{0,2,3\}$. In the second case, at most one twig of the fork in $D_T$ is of type $[2]$, and the third one has length at least $3$ and is not a $(-2)$-twig, contrary to Lemma \ref{lem:admissible_forks}.  Hence the first case holds. If $z=2$ then $T_{x},T_{y}\neq [2]$, so by Lemma \ref{lem:admissible_forks} $V=[2]$ and $\cT$ is as in \ref{item:C_[2]}. If $z=3,0$ then Lemma \ref{lem:admissible_forks} shows that $\cT$ is as in   \ref{item:C_[3]}, \ref{item:C_[rest]}, respectively.
\end{proof}

We now focus on the part $\cS_Y$, that is, the possible singularity types of the canonical surface $\bar{Y}$. Recall that in Theorem \ref{thm:GK} we restrict our attention to the case $\height\geq 3$, not covered by Theorem \ref{thm:ht=1,2}. 

\begin{lemma}[$\height(\bar{X})\geq 3\implies \#\cS(\bar{Y})\geq 6$, and exceptions to the converse]\label{lem:GK_exceptions}
	We keep Notation \ref{not:GK}. If $\height(\bar{X})\geq 3$ then $s\geq 6$. Conversely, if $s\geq 6$ then either  $\height(\bar{X})\geq 3$, or one of the following holds.
	\begin{enumerate}
		\item\label{item:GK2_canonical} $\bar{X}$ is canonical, of type $\rA_1+\rA_7$, $\rA_1+\rA_2+\rA_5$, $2\rA_1+\rD_6$, $2\rA_1+2\rA_3$ or $4\rA_1+\rD_4$.
		\item\label{item:GK2_Ek} $\cS_Y=\rE_{s}$ and $\bar{T}$ is cuspidal or $\cT=[(2)_{l},s+l-5]$ for some $l\in \{0,1,2\}$.
		\item\label{item:A1E7} $\cS_Y\in \{\rA_2+\rE_6,\rA_3+\rD_5,\rA_1+\rE_7,2\rA_1+\rD_6,\rD_8\}$ and $\cT=[3]$.
		\item\label{item:A1A5} $\cS_Y\in \{\rA_1+\rA_5,\rA_1+\rD_6,\rA_1+\rE_7\}$, $\bar{T}$ is cuspidal and  $\cT=[s-3]+[3]+[2]$.
	\end{enumerate}
	Moreover, if $\bar{X}$ is of one of the above types then $\height(\bar{Y})\leq \height(\bar{X})\leq 2$. 
\end{lemma}
\begin{proof}
	By Lemma \ref{lem:GK_intro}\ref{item:GK-intro-Y} we have $\rho(\bar{Y})=1$, so $\bar{T}^2>0$. Suppose $\bar{T}^{2}\geq 4$. After $\bar{T}^{2}-3$ blowups over $\Sing \bar{T}$, the proper transform of $\bar{T}$ is a $0$-curve meeting the exceptional divisor at most twice. Its pullback on $(X,D)$ induces a $\P^1$-fibration of height at most 2. Thus if $\height(\bar{X})\geq 3$ then $\bar{T}^2\leq 3$, so by formula \eqref{eq:Noether} $s=9-\bar{T}^2\geq 6$.
	\smallskip
	
	Assume now that $s\geq 6$. Recall that types $\cS_Y$ and $\cT$ are listed in Table \ref{table:canonical} and Lemma \ref{lem:cuspidal_resolution}. Comparing them with Tables \ref{table:ht=1}--\ref{table:ht=2_char=2} we conclude that the inequality $\height(\bar{X})\leq 2$ can occur only in  cases \ref{item:GK2_canonical}--\ref{item:A1A5} above. We now construct a $\P^1$-fibration of height $\leq 2$ in each case. 
	Let $T_{Y}\in |-K_{Y}|$ be the proper transform of $\bar{T}$ on the minimal resolution $Y$ of $\bar{Y}$. By formula \eqref{eq:Noether} we have $T_Y^2=9-s$ and by adjunction $T_Y\cdot A=1$ for every $(-1)$-curve $A$ on $Y$. It follows from Lemma \ref{lem:GK_intro}\ref{item:GK-intro-L} that the map $\sigma\colon X\to Y$ is an isomorphism near $\sigma^{-1}_{*}A$.  
	
	\ref{item:GK2_canonical} Here we have $\height(\bar{Y})\leq \height(\bar{X})\leq 2$ by \cite[Proposition 6.1]{PaPe_MT}, see  Table \ref{table:canonical}.
	
	\ref{item:GK2_Ek} By \cite[Proposition 6.1]{PaPe_MT} we have $\height(\bar{Y})=1$, and $\bar{Y}$ is as in Remark \ref{rem:canonical_ht=1}, see Figure \ref{fig:ht=1}. Thus $Y$ contains a $(-1)$-curve $A$ meeting $\Exc\alpha$ once, in the last tip of its long $(-2)$-twig $G=[(2)_{s-4}]$. If $D_T=[(2)_{l},s+l-5]$ for some $l\in \{0,1,2\}$ then a subchain of $T+\sigma^{-1}_{*}(A+G)=[s+l-5,1,(2)_{s-4}]$ supports a fiber of a $\P^1$-fibration of height at most two, as needed. In fact, if $l=0$ then $\height(\bar{X})=1$ and $\bar{X}$ as in Lemma \ref{lem:ht=1_types}\ref{item:T2=[2]_b>2}; and if $l=1,2$ then $\height(\bar{X})=2$ and $\bar{X}$ is as in Lemma \ref{lem:ht=2,untwisted}\ref{item:[T_1,n]=[2,2]_r=2},\ref{item:tau=id_fork_chain}, respectively.
	
	Assume that $\bar{T}$ is cuspidal. If $\#\cT\leq 2$ then $D_T$ is as in the previous case, so we can assume $\#\cT\geq 3$. Then $\sigma=\gamma\circ\theta$ for some morphism $\theta$ such that $\#\Exc\theta=3$, so $T'\de \theta^{-1}_{*}T_{Y}=[s-3]$. Since $\bar{T}$ is cuspidal, we have $\beta_{D}((\gamma^{*}T')\redd)\leq 1$, cf.\ Lemma \ref{lem:cuspidal_resolution}\ref{item:T_cuspidal}. Now the pullback of $\theta^{-1}_{*}(T_{Y}+A+G)=[s-3,1,(2)_{s-4}]$ induces a $\P^1$-fibration of height at most two, as needed. More precisely, if $\#\cT=3$ then we get $\height(\bar{X})=1$ and $\bar{X}$ is as in Lemma \ref{lem:ht=1_types}\ref{item:beta=3_columnar}, and if $\#\cT\geq 4$ then $\height(\bar{X})=2$ and $\bar{X}$ is as in Lemma \ref{lem:ht=2,untwisted}\ref{item:tau=id_fork_chain}, 	\ref{item:V-chains_c=2_[2]_T1=0}--\ref{item:tau=id_forks}, or \ref{item:c=3}.
	
	\ref{item:A1E7} Consider the case $\cS_{Y}=\rA_2+\rE_6$. Figure \ref{fig:E6+A2} shows that $Y$ contains a $(-1)$-curve $A$ meeting $\Exc\alpha$ once, in the maximal twig $G=[2]$. Let $B$ be the branching component of $\Exc\alpha$. Now $T+\sigma^{-1}_{*}(A+G+B)=[3,1,2,2]$ supports a fiber of a fibration of height $2$. In fact, $\height(\bar{X})=2$ and $\bar{X}$ is as in Lemma  \ref{lem:ht=2,untwisted}\ref{item:R_0_branching}.
	
	In the remaining cases, we see from Figure \ref{fig:D5+A3} or \ref{fig:ht=1} that $Y$ contains a $(-1)$-curve $A$ meeting $\Exc\alpha$ once, in some tip of a twig $G=[2,2]$. Now $T+\sigma^{-1}_{*}(A+G)=[3,1,2,2]$ supports a fiber of a fibration of height $1$ or $2$. In fact, if $\cS_{Y}=\rA_1+\rE_7$ then $\height(\bar{X})=\height(\bar{Y})=1$ and $\bar{X}$ is as in Lemmas \ref{lem:ht=1_types}\ref{item:beta=3_other_2}, otherwise $\height(\bar{X})=2$ and $\bar{X}$ is as in Lemma \ref{lem:ht=2,untwisted}\ref{item:rivet_nu=3-fork-1}, \ref{lem:ht=2_twisted-separable}\ref{item:k2=2} or \ref{lem:ht=2_twisted-inseparable}\ref{item:C_fork_k2=2} if $\cS_Y$ equals $\rD_8,\rA_3+\rD_5$ or $2\rA_1+\rD_6$ respectively.
	
	\ref{item:A1A5} Like before, we see from Figure \ref{fig:ht=1} that there is a $(-1)$-curve $A\subseteq Y$ such that $A+\Exc\alpha$ contains a subchain $R=[(2)_{s-4},1]$ meeting the remaining part of $A+\Exc\alpha$ only in the second component, once. Now $\sigma^{-1}_{*}R+T=[(2)_{s-4},1,s-3]$ supports a fiber of the fibration of height $2$, as needed. In fact, $\height(\bar{X})=2$ and $\bar{X}$ is as in Lemma \ref{lem:ht=2_twisted-inseparable}\ref{item:beta=2_k1=k2=2},\ref{item:C_[2]_k2=2},\ref{item:C_columnar_L2H=0_k2=2} for $s=6,7,8$, respectively.
\end{proof}

\subsection{Computation of the height}

As in the previous section, we consider a del Pezzo surface $\bar{X}$ of rank one having a descendant $(\bar{Y},\bar{T})$ with elliptic boundary. The singularity type of $\bar{X}$ is $\cS_{Y}+\cT$, where $\cS_{Y}$ is the singularity type of $\bar{Y}$; $\cT$ is described as in Lemma \ref{lem:cuspidal_resolution}. Lemma \ref{lem:GK_intro}\ref{item:GK-intro-ht} shows that $\height(\bar{X})\leq \height(\bar{Y})+2$: in fact, pulling back a witnessing $\P^1$-fibration from the minimal log resolution of $\bar{Y}$ we get a $\P^1$-fibration of height $\height(\bar{Y})+2$ with respect to $D$. In this section we prove Theorem \ref{thm:GK}\ref{item:GK-ht}, which asserts that this inequality is in fact an equality, up to some  minor exceptions. More precisely, we prove the following statement. We keep Notation \ref{not:GK}. 

\begin{proposition}\label{prop:GK-ht-exceptions}
	Assume $\height(\bar{X})\geq 3$. If $\height(\bar{X})\neq \height(\bar{Y})+2$ then one of the following holds.
	\begin{enumerate}
		\item\label{item:GK-ht-exceptions-T=1} $\cS_{Y}\in \{\rA_1+\rA_2+\rA_5,\rA_8,\rA_1+\rA_7\}$ and $\#\cT=1$,
		\item\label{item:GK-ht-exceptions-T=3} $\cS_Y\in \{3\rA_2,\rA_2+\rA_5,\rA_2+\rE_6\}$  and $\cT=[s-3]+[3]+[2]$. In this case, we have $\cha\kk=3$.%
	\end{enumerate}
	Conversely, if $\bar{X}$ is of type $\cS_{Y}+\cT$ listed above then $\height(\bar{X})=3$ and $\height(\bar{Y})=2$.
\end{proposition}

Within the proof we keep Notation \ref{not:GK}. First, we construct the exceptional $\P^1$-fibrations of height $3$.
\begin{lemma}
	\label{lem:GK-ht-exceptions-ht=3}
	Assume that the singularity type $\cS_{Y}+\cT$ of $\bar{X}$ is as in Proposition \ref{prop:GK-ht-exceptions}.	Then $\height(\bar{X})=\height(\bar{Y})+1=3$.
\end{lemma}
\begin{proof}
	By \cite[Proposition 6.1]{PaPe_MT} we have $\height(\bar{Y})=2$, and the minimal resolution $Y$ of $\bar{Y}$ is constructed in Section \ref{sec:ht=2_constructions}, cf.\ Table \ref{table:canonical}. Since $-K_{\bar{Y}}=\bar{T}$, the adjunction formula shows that each $(-1)$-curve on $Y$ meets the image of $T$ once, so it pulls back to a $(-1)$-curve on $X$ meeting $T$ once and disjoint from the elliptic tie $L$. 
	Lemma \ref{lem:GK_exceptions} gives $\height(\bar{X})\geq 3$, so it is enough to find $\P^1$-fibrations of height $3$ in each case.
	
	\ref{item:GK-ht-exceptions-T=1} 
	Consider the case $\cS_{Y}=\rA_1+\rA_2+\rA_5$. Looking at Figure \ref{fig:A5+A2+A1}, we see that there is a $(-1)$-curve $A_0$ on $X$ meeting $T$ once and $D_Y$ twice, once in a tip of a connected component $V=[2,2]$ of $D_Y$, and once off $V$. Now $T+A_0+V=[3,1,2,2]$ supports a fiber of the required $\P^1$-fibration. 
	
	Consider the case $\cS_{Y}=\rA_8$. Looking at Figure \ref{fig:A8} we see that there are disjoint $(-1)$-curves $A_1,A_2$ on $X$ such that $A_i\cdot T=A_i\cdot D_Y=1$, $i=1,2$, and $A_1,A_2$ meet different components of $D_Y$. Let $V$ be the component of $D_Y$ meeting $A_2$, so $\beta_{D}(V)=2$. Now $A_1+T+A_2+V=[1,3,1,2]$ supports a fiber of the required $\P^1$-fibration.
	
	Eventually, consider the case $\cS_Y=\rA_1+\rA_7$. Looking at Figure \ref{fig:A7+A1} we see that there are $(-1)$-curves $A_1,A_2$ on $X$ such that $A_1\cdot T=A_2\cdot T=1$, $A_1\cdot D_Y=1$, $A_2\cdot D_Y=2$ and $A_2$ meets a connected component $V=[2]$ of $D+A_1$. As before, $A_1+T+A_2+V=[1,3,1,2]$ supports the required fiber.
	 
	\ref{item:GK-ht-exceptions-T=3} In this case $\bar{T}$ is cuspidal, so Table \ref{table:canonical} shows that $\cha\kk=3$. Looking at Figures \ref{fig:3A2}, \ref{fig:A5+A2} and \ref{fig:E6+A2} we see that there is a $(-1)$-curve $A$ on $X$ such that $D_{Y}+A$ has a subchain $V=[(2)_{s-4},1]$ meeting $D_Y$ once, in the third component. Now $V+T=[(2)_{s-4},1,s-3]$ supports the fiber of the required $\P^1$-fibration. 
\end{proof}

We now verify the completeness of the list in Proposition \ref{prop:GK-ht-exceptions}. We treat separately the cases when the graph of $D_{T}+L$ is circular (i.e.\ $\#\cT\leq 2$ or $\bar{T}$ is nodal) and when it is not, which lead to  \ref{prop:GK-ht-exceptions}\ref{item:GK-ht-exceptions-T=1} and \ref{prop:GK-ht-exceptions}\ref{item:GK-ht-exceptions-T=3}, respectively. In each case we argue by induction on $\#\cT$. We begin with a basic observation in the former case. 

\begin{lemma} \label{lem:GK-nodal-basics}
	Assume that $\#\cT\leq 2$ or $\bar{T}$ is nodal. Let $A\neq L$ be a $(-1)$-curve on $X$. Then the following hold.
	\begin{enumerate}
		\item \label{item:GK-nodal-K} We have $-K_{X}=D_{T}+L$. In particular, $A\cdot (D_{T}+L)=1$.
		\item \label{item:GK-nodal-F} If $F$ is a fiber of a $\P^1$-fibration of $X$ then $F\cdot (D_{T}+L)=2$. 
		\item \label{item:GK-nodal-neg-def} Assume $A\cdot T=1$. Then $D_Y+A$ is not negative (semi)-definite. In particular, the connected component of $D_Y+A$ containing $A$ is neither a chain $[1,2,\dots,2]$ nor a fork $\langle 2;[2],[2],[1,2,\dots,2]\rangle$.
		\item \label{item:GK-nodal-T=1} Assume $\#\cT=1$. If $A$ meets $L$ then $A\cdot D_Y\geq \height(\bar{Y})$.
	\end{enumerate}
\end{lemma} 
\begin{proof}
	\ref{item:GK-nodal-K}, \ref{item:GK-nodal-F} 
	Pulling back the linear equivalence $-K_{\bar{Y}}=\bar{T}$ we get the linear equivalence $-K_{X}=D_{T}+L$ from \ref{item:GK-nodal-K}. The remaining part of \ref{item:GK-nodal-K} and \ref{item:GK-nodal-F} follow by the adjunction formula. 
	
	\ref{item:GK-nodal-neg-def} Since $\rho(\bar{Y})=1$, we have $0<\phi(A)^2=(\phi^{*}\phi(A))^2$. Part \ref{item:GK-nodal-K} implies that $A$ is disjoint from $D_T-T+L$, so $\phi^{*}\phi(A)$ is supported on $D_{Y}+A$, hence the latter is not negative (semi)-definite. 
		
	\ref{item:GK-nodal-T=1} We have $A\cdot D_Y=\sigma(A)\cdot \sigma_{*}D_Y$ and $\sigma(A)=[0]$, so $|\sigma(A)|$ induces a $\P^1$-fibration of $Y$ of height $A\cdot D_{Y}$ with respect to $\sigma_{*}D_Y=\Exc\alpha$. Thus $A\cdot D_Y\geq \height(\bar{Y})$, as needed.
\end{proof}

\begin{lemma}
	\label{lem:GK-ht-T=1}
	Assume that $\#\cT=1$, $\height(\bar{Y})\geq 2$ and $\height(\bar{X})\neq  \height(\bar{Y})+2$. Then the following hold.
	\begin{enumerate}
		\item\label{item:T=1-L-horizontal} For every $\P^1$-fibration of $X$ of height at most $\height(\bar{Y})+1$ with respect to $D$ the elliptic tie $L$ is horizontal.
		\item\label{item:T=1-exceptions} We have $\height(\bar{X})\geq \height(\bar{Y})=2$ and $\bar{X}$ is as in Proposition \ref{prop:GK-ht-exceptions}\ref{item:GK-ht-exceptions-T=1} or as in Lemma \ref{lem:GK_exceptions}.
	\end{enumerate}
\end{lemma}
\begin{proof}
Lemma \ref{lem:GK_intro}\ref{item:GK-intro-ht} and the assumption $\height(\bar{X})\neq  \height(\bar{Y})+2$ give $\height(\bar{X})\leq \height(\bar{Y})+1$. Fix a $\P^1$-fibration of $X$ whose fiber $F$ satisfies $F\cdot D\leq \height(\bar{Y})+1$.

	To prove \ref{item:T=1-L-horizontal} suppose  $F\cdot L=0$. Then $F\cdot T=2$ by Lemma \ref{lem:GK-nodal-basics}\ref{item:GK-nodal-F} and $\sigma_{*}F$ is a fiber of a $\P^1$-fibration of $Y$, so $F\cdot D_{Y}\geq \height(\bar{Y})$. We get $F\cdot D=F\cdot D_Y+F\cdot T\geq \height(\bar{Y})+2$, a contradiction.

	It remains to prove part \ref{item:T=1-exceptions}. By Lemma \ref{lem:GK_exceptions} we can assume $F\cdot D\geq \height(\bar{X})\geq 3$. 
	If $\bar{X}$ is canonical then $\height(\bar{X})\leq 2$ or $\height(\bar{X})=4=\height(\bar{Y})+2$, see Table \ref{table:canonical}, contrary to the assumptions. Thus $-3\geq T^2=5-\#\cS_{Y}$ by \eqref{eq:Noether}, so $\#\cS_{Y}=8$. Part \ref{item:T=1-L-horizontal} and Lemma \ref{lem:GK-nodal-basics}\ref{item:GK-nodal-F} show that $L$ is a $1$- or a $2$-section. Let $F_{1},\dots, F_{\nu}$ be all degenerate fibers, and let 
	$k_{i}=\#(F_i)\redd-1$.
	\smallskip

	Consider the case $F\cdot L=1$. Then $F\cdot T=1$ by Lemma \ref{lem:GK-nodal-basics}\ref{item:GK-nodal-F}. By Lemma \ref{lem:GK-nodal-basics}\ref{item:GK-nodal-K}, every vertical $(-1)$-curve meets $T+L$ once, so it has multiplicity $1$ in $F_i$. Lemma \ref{lem:delPezzo_fibrations}\ref{item:-1_curves} implies that each $F_i$ consists of $(-1)$- and $(-2)$-curves, so $F_{i}=[1,(2)_{k_i-1},1]$. Let $L_i,T_i$ be the tips of $F$ meeting $L$ and $T$, respectively, and let $G_i=F_i-L_i-T_i$. We have $L_{i}\cdot D_{Y}\geq \height(\bar{Y})$ and $(T_{i}+G_i)\cdot (D_Y)\hor\geq 1$ by Lemma \ref{lem:GK-nodal-basics}\ref{item:GK-nodal-T=1},\ref{item:GK-nodal-neg-def}, so $\height(\bar{Y})+1\geq F_{i}\cdot D\hor=L_{i}\cdot (D_Y)\hor+(T_{i}+G_{i})\cdot (D_Y)\hor+F_{i}\cdot T\geq \height(\bar{Y})-L_{i}\cdot G_{i}+2$. We conclude that  $L_{i}\cdot (D_Y)\hor=\height(\bar{Y})-1$, $G_{i}\neq 0$ and $(T_{i}+G_{i})\cdot (D_Y)\hor=1$, so $(D_Y)\hor$ meets $G_{i}+T_{i}$ exactly once. We order the chain $F_i$ so that $L_i$ is its first tip. Then $(D_Y)\hor$ meets the $(l_i+1)$-th component of $F_i$ for a unique $l_{i}\in \{1,\dots, k_i\}$.  We have $\#(D_Y)\hor=\#D\hor-1=\nu$ by Lemma \ref{lem:delPezzo_fibrations}\ref{item:Sigma}.  Let $\tau\colon X\to \P^2$ be the contraction of $L+\sum_{i=1}^{\nu}(G_{i}+T_{i})$. 
	
	Suppose $\height(\bar{X})>3$. Then $\height(\bar{Y})>2$, so $\height(\bar{Y})=4$ and $\cS_{Y}\in \{8\rA_1,4\rA_2\}$, see Table \ref{table:canonical}. By Remark \ref{rem:primitive_ht=2} the surface $\bar{Y}$ is primitive, so each $(-1)$-curve $T_{i}$ meets $D_Y$ at least twice. Thus $T_{i}$ meets $(D_Y)\hor$, i.e.\ $l_i=k_{i}$, and therefore $G_{i}$ is a connected component of $D_Y$. If $\cS_{Y}=4\rA_2$ then $8=\sum_{i=1}^{\nu}\#G_i+\#(D_Y)\hor=3\nu$, so $\nu=\frac{8}{3}$, which is impossible. Thus $\cS_{Y}=8\rA_1$, and $8=\sum_{i=1}^{\nu}\#G_i+\#(D_Y)\hor=2\nu$, so $\nu=4$. Thus $(D_Y)\hor$ consists of four $1$-sections. Each of the four curves $T_{i}$ meets $(D_Y)\hor$ exactly once, so there is a component of $(D_Y)\hor$, call it $H$, which meets at most one $T_i$. Then $\tau(H)^2\leq -2+2=0$, a contradiction. 	
		
	Hence $\height(\bar{X})=3$. We have $\nu=\#(D_Y)\hor\leq 2$. If $\nu=1$ then $\tau((D_Y)\hor)$ is a conic, so $4=\tau((D_Y)\hor)^2=-2+l_{1}$. Thus $l_{1}=6$, and we get $\cS_{Y}=\rD_8$, so Table \ref{table:canonical} shows that  $\height(\bar{Y})=1$, a contradiction. Hence $\nu=2$. Now $(D_Y)\hor=H_{1}+H_2$, where each $H_{i}$ is a $1$-section, so $\tau(H_{i})$ is a line. If $H_{1}$ meets $L_{i}$ for both $i\in \{1,2\}$ then $\tau(H_{i})^2=-2$, which is impossible. Thus, say, $H_{i}$ meets $G_{i}+T_i$ for $i=1,2$, and $1=\tau(H_{i})^2=-2+l_i$, so $l_i=3$. We have $l_{i}\leq k_{i}$ and $k_1+k_2=\#D_Y=8$, so either $k_{1}=k_2=4$, or, say, $k_{1}=5$, $k_2=3$. Moreover, the map $\tau^{-1}$ is an isomorphism near the common point of the lines $\tau(H_i)$, so $H_{1}\cdot H_{2}=1$. It follows that either $\cS_{Y}=\rA_8$, as needed, or $\cS_{Y}=\rA_2+\rE_6$, in which case  Lemma \ref{lem:GK_exceptions} gives $\height(\bar{X})=2$, a contradiction.
		\smallskip
	
	Consider the case $F\cdot L=2$. By Lemma \ref{lem:GK-nodal-basics}\ref{item:GK-nodal-F} $T$ is vertical, so, say, $T\subseteq F_1$. For $i\geq 2$, the fiber $F_i$ consists of $(-1)$- and $(-2)$-curves, so $(F_i)\redd$ is a chain $[1,(2)_{k_i-1},1]$, $[2,1,2]$ or on a fork $\langle 2;[1,(2)_{k_i-3}],[2],[2]\rangle$. 
	 
	Since $T$ meets the $2$-section $L$ twice, it has multiplicity $1$ in $F_1$. Hence there is a birational morphism $\tau\colon X\to \F_{m}$ such that $\tau(T)$ is a fiber. Lemma \ref{lem:GK-ht-nodal}\ref{item:GK-nodal-K} implies that the exceptional $(-1)$-curve of every blowup in the decomposition of $\tau$ either lies in the image of $F_1$ and is disjoint from the image of $L$, or lies in the image of $F_i$ for some $i\geq 2$ and meets the image of $L$ once. Thus $\tau(L)$ is a smooth, rational $2$-section and $\tau(L)^2=L^2-\sum_{i=2}^{\nu}k_{i}$. Numerical properties of $\F_m$ imply that $\tau(L)^2=4$, so  $\sum_{i=2}^{\nu}k_{i}=\tau(L)^2-L^2=5$.
	 
	 Suppose the Lemma fails. Then $\cS_{Y}\in \{2\rA_4,4\rA_1+\rD_4,4\rA_2,8\rA_1\}$, see Table \ref{table:canonical}. Suppose that each $F_i$ for $i\geq 2$ has exactly one $(-1)$-curve; in particular $k_i\geq 2$.  If $\nu=2$ then $k_{2}=5$, so $D_{Y}$ has a subdivisor of type $\langle 2,[2],[2],[2,2]\rangle$, which is false for the above types $\cS_{Y}$. Hence $\nu=3$ and, say, $(k_{2},k_{3})=(2,3)$. Now $D_Y$ has a subchain $[2,2,2]$, so the above list gives $\cS_{Y}\in \{2\rA_4,4\rA_1+\rD_4\}$. In particular, $\height(\bar{Y})=2$, see Table \ref{table:canonical}. By Lemma \ref{lem:GK-nodal-basics}\ref{item:GK-nodal-T=1}, the $(-1)$-curve $A$ in $F_3$ satisfies $A\cdot D_Y\geq \height(\bar{Y})=2$, so it meets $D\hor$. Since $A$ has multiplicity $2$ in $F_3$, and $F_3\cdot D\geq \height(\bar{Y})+1=3$, we see that $D\hor$ consists of a $2$-section meeting $A$, and a $1$-section $H$ which meets $F_i$ in a $(-2)$-tip for $i=2,3$. Thus $(D_Y)\vert+H\subseteq D_Y$ contains a subchain $[(2)_{5}]$, a contradiction.
	 
	 Therefore, say, $F_2$ has two $(-1)$-curves, call them $A_1,A_2$. Then $F_2=[1,(2)_{k_2-1},1]$. Since $T$ is vertical, we have $A_{i}\cdot L=1$ and $A_{i}\cdot D_{Y}\geq \height(\bar{Y})$ by Lemma \ref{lem:GK-nodal-basics}\ref{item:GK-nodal-K},\ref{item:GK-nodal-T=1}. If $\cS_{Y}\in \{4\rA_2,8\rA_1\}$ then $\height(\bar{Y})=4$, so $5\geq F_2\cdot D\geq \sum_{i=1}^2(A_{i}\cdot D_{Y}-1)\geq 6$, a contradiction. Hence $\cS_{Y}\in \{2\rA_4,4\rA_1+\rD_4\}$ and $\height(\bar{Y})=2$. If $k_2=1$ then as before $3\geq F_{2}\cdot D= \sum_{i=1}^2A_{i}\cdot D_Y\geq 4$, which is false, so $k_2\geq 2$. By Lemma \ref{lem:delPezzo_fibrations}\ref{item:Sigma}, the number of $(-1)$-curves in $F_{1}$ is at most $\#D\hor-1\leq 2$, in particular $\#D\hor\geq 2$. Since $F_{1}$ consists of $(-1)$-curves off $D$, $(-2)$-curves in $D_Y$ and the $(-3)$-curve $T$, it is supported on a chain $[1,3,1,2]$ or $[3,1,2,2]$; and we have $\#D\hor=3$ or $\#D\hor\geq 2$, respectively. We conclude that $F_1$ contains a $(-1)$-curve $A$ of multiplicity $2$ or $3$, respectively, and $A\cdot D\hor=0$. Thus $A\cdot D_Y=1$. Since $A\cdot L=0$, we have $\sigma(A)=[1]$, so $\bar{Y}$ is not primitive. By Remark \ref{rem:primitive_ht=2} we have $\cS_{Y}\neq 2\rA_4$, so $\cS_{Y}=4\rA_1+\rD_4$. Let $G$ be the connected component of $D_Y$ of type $\rD_4$.
 
 	Suppose $(F_1)\redd=[3,1,2,2]$. The subchain $[2,2]$ of $(F_1)\redd$ is contained in $G$, so it meets $D\hor$ once. Since the components of this subchain have multiplicities $1$ and $2$ in $F_1$, it follows that $D\hor$ meets $T$ or  $A$, a contradiction.
 	
 	Thus $(F_1)\redd=[1,3,1,2]$, so Lemma \ref{lem:delPezzo_fibrations}\ref{item:Sigma} shows that $\#D\hor=3$, and each $F_i$ for $i\geq 2$ contains exactly one $(-1)$-curve. Since $D\hor$ consists of $1$-sections, it meets $F_1$ only in tips. By Lemma \ref{lem:GK-ht-nodal}\ref{item:GK-nodal-neg-def}, the $(-1)$-tip of $F_{1}$ meets $D\hor$, so the $(-2)$-tip of $F_{1}$ meets $D\hor$ at most twice. Thus the $(-2)$-tip of $F_1$ is either is a tip of $G$, or a connected component of $D_Y$. Hence the sum of $A$ and the connected component of $D_Y$ meeting it is a chain $[1,2]$ or a fork $\langle 2;[1,2],[2],[2]\rangle$, contrary to Lemma \ref{lem:GK-nodal-basics}\ref{item:GK-nodal-neg-def}.
\end{proof}

\begin{lemma}
	\label{lem:GK-ht-nodal}
	Assume that $\height(\bar{Y})\geq 2$, and either $\#\cT= 2$ or $\#\cT\geq 3$ and $\bar{T}$ is nodal. Then $\height(\bar{X})=\height(\bar{Y})+2$.
\end{lemma}
\begin{proof}
	Suppose the contrary. Then $\height(\bar{X})\leq \height(\bar{Y})+1$ by Lemma \ref{lem:GK_intro}\ref{item:GK-intro-ht}. Assume that the number $\#\cT\geq 2$ is minimal possible. Table \ref{table:canonical} shows that $\bar{X}$ is not canonical, so the elliptic tie $L$ meets exactly one $(-2)$-curve in $D$, call it $C$. Fix a $\P^1$-fibration of $X$ whose fiber $F$ satisfies $F\cdot D\leq \height(\bar{Y})+1$. 
	
\begin{claim}\label{cl:L-horizontal}
	The elliptic tie $L$ is horizontal.
\end{claim}	
\begin{proof} 
	Suppose the contrary. Let $\tau\colon X\to X'$ be the contraction of $L$, and let $D'=\tau_{*}(D-C)$, $F'=\tau_{*}F$. Then $(X,D)\sqto (X',D')$ is a swap onto a minimal resolution of another del Pezzo surface $\bar{X}'$ of rank one, with the same descendant $(\bar{Y},\bar{T})$. Let $\cT',L'$ etc.\ be as in Notation \ref{not:GK} for the surface $\bar{X}'$, so $L'=\tau(C)$. The linear system $|F'|$ induces a $\P^1$-fibration such that $F'\cdot D'=F\cdot D-F'\cdot L'\leq \height(\bar{Y})+1-F'\cdot L'$. In particular, $\height(\bar{X}')\leq \height(\bar{Y})+1$, so $\#\cT'=1$ by the minimality of $\#\cT$. By Lemma \ref{lem:GK-ht-T=1}, the elliptic tie $L'$ is horizontal and $\height(\bar{X}')\geq \height(\bar{Y})=2$, so the above inequality $F'\cdot D'\leq \height(\bar{Y})+1-F'\cdot L'$ gives $F'\cdot L'=1$ and $F'\cdot D'=2$. By Lemma \ref{lem:GK-nodal-basics}\ref{item:GK-nodal-F} we have $F'\cdot T'=1$. Let $F''\subseteq X'$ be a degenerate fiber. By Lemma \ref{lem:delPezzo_fibrations}\ref{item:-1_curves} $F''$ consists of $(-2)$-curves in $D'$ and $(-1)$-curves off $D'$. By Lemma \ref{lem:GK-nodal-basics}\ref{item:GK-nodal-K}, every such $(-1)$-curve meets $T'$ or $L'$, which are $1$-sections, so it has multiplicity $1$ in $F''$. Thus $F''=[1,2,\dots,2,1]$, $\ftip{F''}$ meets $L'$ and $\ltip{F''}$ meets $T'$. By Lemma \ref{lem:GK-nodal-basics}\ref{item:GK-nodal-T=1} $\ftip{F''}$ meets $(D'_{Y})\hor$, which is a $1$-section because $F''\cdot D'=2$. Hence $F''-\ftip{F''}=[2,\dots,2,1]$ meets $D-F''\wedge D$ only in $T$, a contradiction with Lemma \ref{lem:GK-nodal-basics}\ref{item:GK-nodal-neg-def}.
\end{proof}

	\begin{claim}\label{cl:AD-bound} 
		Assume that $A$ is a $(-1)$-curve meeting $L+D_{T}-T$. Then $A\cdot D_Y\geq \height(\bar{Y})$.
\end{claim}
\begin{proof}
	By Lemma \ref{lem:GK-nodal-basics}\ref{item:GK-nodal-K} we have $A\cdot (D_T+L)=1$. 
	Contracting $L$ and $(-1)$-curves in subsequent images of $D_{T}$ until the image of $A$ meets a $(-1)$-curve $L'$ in the image of $L+D_{T}$, we get a swap $(X,D)\sqto (X',D')$ onto a minimal log resolution of a del Pezzo surface $\bar{X}'$ with the same descendant $(\bar{Y},\bar{T})$. If $\#\cT'=1$ then the claim follows from Lemma \ref{lem:GK-nodal-basics}\ref{item:GK-nodal-T=1}. Assume $\#\cT'\geq 2$. The linear system $|L'+A'|$  induces a $\P^1$-fibration of $X'$ such that the elliptic tie $L'$ is vertical, so by Claim \ref{cl:L-horizontal} we have $\height(\bar{Y})+2\leq (A'+L')\cdot D'=A\cdot D_{Y}+2$, as needed.
\end{proof} 

\begin{claim}\label{cl:F}
	We have $\height(\bar{Y})=2$, and for any fiber $F$ containing a component of $D_T$, one of the following holds.
	\begin{enumerate}
		\item\label{item:F-with-T} $F=U+G_{1}+\dots G_{r}$, where $U=[V,T,V^{*}]$ or $U=T$, and each $G_i$ is a chain $[1,2,\dots,2]$, meeting $(D_Y)\hor$, whose first tip meets $T$ and the remaining components are contained in $D_Y$.
		\item\label{item:F-with-C} $F\redd=[2,1,2]$, $F\redd \wedge (D_T)\vert=C$, $(D_T)\hor$ is a $1$-section meeting $C$, and $(D_Y)\hor$ is a $2$-section meeting the $(-1)$-curve in $F$, once. 
	\end{enumerate} 	
\end{claim}	
\begin{proof}
	Since $D_Y$ consists of $(-2)$-curves, every $(-1)$-curve in $F$ meets $D_T$. Let $A_{1},\dots, A_{k}$ be the $(-1)$-curves in $F$ which do not meet $T$. Let $\mu_{i}$ be the multiplicity of $A_{i}$ in $F$, and let $\beta_{i}\de \beta_{F\redd}(A_i)-1\in \{0,1\}$. 
	Put $G\de F\redd-D_T\wedge F\redd$. Let $r\geq 0$ be the number of connected components of $G$ meeting $T$. Since $D_Y$ consists of $(-2)$-curves, every such connected component is of type $[1,2,\dots,2]$, so it meets $(D_Y)\hor$ by Lemma \ref{lem:GK-nodal-basics}\ref{item:GK-nodal-neg-def}. 
	
	Assume $k=0$. Then after the contraction of $G$ the image of $F$ is either a $0$-curve, or a chain meeting the image of  $(L+D_T)\hor$ in tips. Since $F\cdot (L+D_T)\hor=2$, those tips have multiplicity one, so the image of $F$ is of type $[0]$ or $[V,1,V^{*}]$ for some admissible chain $V$. Since the above contraction is an isomorphism in a neighborhood of $D_{T}-T$, we conclude that $F$ is as in \ref{item:F-with-T}. Suppose $\height(\bar{Y})\neq 2$.  Since by assumption $\height(\bar{Y})\geq 2$, Table \ref{table:canonical} shows that $\cS_{Y}\in \{4\rA_2,8\rA_1\}$, so every $(-1)$-curve in $F$ meets $T$ and a tip of a connected component of $D_Y$ which is a chain. This contradicts  Lemma \ref{lem:GK-nodal-basics}\ref{item:GK-nodal-neg-def}. Hence $\height(\bar{Y})=2$, as needed. 
	
	Assume $k\geq 1$. We have $F\cdot D\geq \sum_{i=1}^{k}\mu_{i} A_{i}\cdot (D_Y)\hor+F\cdot (D_T)\hor +r$. Claim \ref{cl:AD-bound} shows that for each $i=1,\dots,k$ we have $A_i\cdot (D_Y)\hor=A_{i}\cdot D_{Y}-\beta_{i}\geq \height(\bar{Y})-\beta_{i}$, so the inequality $F\cdot D\leq \height(\bar{Y})+1$ gives
	\begin{equation}\label{eq:F-bound}
		\height(\bar{Y})+1\geq \sum_{i=1}^{k}\mu_{i}(\height(\bar{Y})-\beta_{i})+F\cdot (D_T)\hor +r.
	\end{equation}
	Suppose $\beta_{1}=0$, i.e.\ $A_{1}$ is a tip of $F$. Since $\height(\bar{Y})\geq 2$, inequality \eqref{eq:F-bound} implies that $\mu_{1}=1$, so $F$ contains another $(-1)$-curve, call it $A'$. If $A'$ does not meet $T$, say $A'=A_2$, then inequality \eqref{eq:F-bound} implies that $\mu_{2}=1$ and $\beta_{2}=1$, which is impossible. Thus $A'$ meets $T$, so $r\geq 1$. Inequality \eqref{eq:F-bound} gives $r=1$, i.e.\ $A'$ is unique, and $F\cdot (D_T)\hor=0$, i.e.\ $F$ contains $D_T$. We conclude that $F=[1,2,\dots,2,T,1,(2)_{-T^2-2}]$, and $F_{T}-A_1$ meets $D\hor$ only in the last tip, call it $V$. The chain $L+F\redd-A=[1,T,1,(2)_{-T^2-2}]$ supports a fiber $F'$ of another $\P^1$-fibration of $X$. We have $F'\cdot D=3$, because $F'\redd$ meets $D-F'\redd\wedge D$ only once on $L$, once on $T$, and once on $V$; and $L,T,V$ have multiplicity $1$ in $F'$. Thus $F'\cdot D\leq \height(\bar{Y})+1$, and for this new $\P^1$-fibration the elliptic tie $L$ is vertical, a contradiction with Claim \ref{cl:L-horizontal}.
	
	We conclude that for all $i$ we have $\beta_{i}=1$, i.e.\ the $(-1)$-curve $A_i$ is not a tip of $F\redd$, and therefore $\mu_{i}\geq 2$. Inequality \eqref{eq:F-bound} implies that $\sum_{i}\mu_{i}\leq 3$, so $k=1$ and $\mu_1\leq 3$. 
	
	Suppose $F\cdot (D_T)\hor=0$. Then $D_{T}\subseteq F$. If $r=0$ then $A_1$ is the unique $(-1)$-curve in $F$, so the conditions $\beta_1=1$, $\mu_1\leq 3$ give $F\redd=[3,1,2,2]$ or $[2,1,2]$, contrary to the assumption $\#\cT\geq 2$. Thus $r\geq 1$, and inequality \eqref{eq:F-bound} gives $r=1$ and $\mu_1=2$. Moreover, the connected component of $F\redd-D_T$ meeting $T$, which is a chain $[1,2,\dots,2]$, meets $D\hor$ in a component of multiplicity $1$. We conclude that $F\redd=[2,1,3,2,\dots,2]*[T,1,2,\dots,2]$, so the tips of $D_{T}=[3,2,\dots,2]*T$ are not $(-2)$-curves. This is a contradiction, as one of those tips is $C$.
	
	Thus $F\cdot (D_T)\hor\geq 1$, so by Lemma \ref{lem:GK-nodal-basics}\ref{item:GK-nodal-F} and Claim \ref{cl:L-horizontal} the curves $(D_T)\hor$ and $L$ are $1$-sections. Inequality \eqref{eq:F-bound} gives $\height(\bar{Y})=2$, $\mu_{1}=2$ and $r=0$, so $F\redd=[2,1,2]$ and $F\redd\wedge D_T$ is a $(-2)$-curve meeting $L$, called $C$. Since $(D_Y)\hor$ meets $A_1$, we conclude that $(D_Y)\hor$ is a $2$-section, so \ref{item:F-with-C} holds, as needed.
\end{proof}

\begin{claim}\label{cl:not-vertical}
	The curves $L$ and $H_T\de (D_T)\hor$ are $1$-sections.
\end{claim}
\begin{proof}
	By Lemma \ref{lem:GK-nodal-basics}\ref{item:GK-nodal-F} any fiber $F$ satisfies $F\cdot D_{T}=2-F\cdot L\leq 1$ by Claim \ref{cl:L-horizontal}. If $F\cdot D_{T}=1$ then the claim follows, so assume $F\cdot D_{T}=0$, i.e.\ $D_T$ is contained in some fiber $F_{T}$. Such $F_{T}$ is as in Claim \ref{cl:F}\ref{item:F-with-T}. Since $\#D_{T}\geq 2$, we get $D_{T}=[V,T,V^{*}]$, so $D_{T}+L-T=[V\trp,1,(V^{*})\trp]=[V\trp,1,(V\trp)^{*}]$, see \cite[\sec 3.9]{Fujita-noncomplete_surfaces}, which blows down to a $0$-curve and not to a smooth point, a contradiction.
\end{proof}

\begin{claim}\label{cl:other-F}
	Let $F$ be a degenerate fiber which does not contain any component of $D_T$. Then $F=[1,2,\dots,2,1]$, $\ftip{F}$ meets $L$ and $(D_Y)\hor$, and $\ltip{F}$ meets $H_T$ and $T+(D_Y)\hor$.
\end{claim}	
\begin{proof}
	Let $A$ be a $(-1)$-curve in $F$. By Lemma \ref{lem:GK-nodal-basics}\ref{item:GK-nodal-K}, $A$ meets $D_T+L$. Since $L$ is horizontal and $F$ is disjoint from $(D_T)\vert$, the curve $A$ meets $(D_T+L)\hor$, which by Claim \ref{cl:not-vertical} consists of $1$-sections. Thus $A$ has multiplicity $1$ in $F$. Since $F\redd\wedge D$ consists of $(-2)$-curves, we get $F=[1,2,\dots,2,1]$. Since both $(-1)$-curves in $F$ meet $D_{T}+L$, we can order $F$ in such a way that $\ftip{F}$ meets $L$ and $\ltip{F}$ meets $H_T$. By Claim \ref{cl:AD-bound} if a tip of $F$ does not meet $T$ then it meets $(D_Y)\hor$, which ends the proof.
\end{proof}

\begin{claim}\label{cl:T-vertical}
	The curve $T$ is contained in a fiber $F_{T}$ as in Claim \ref{cl:F}\ref{item:F-with-T}.
\end{claim}
\begin{proof}
	Suppose the contrary. Then $T$ is horizontal, and $D_{T}-T$ is contained in a fiber $F$ as in Claim \ref{cl:F}\ref{item:F-with-C}, so $D_{T}=T+C$, $H_2\de (D_Y)\hor$ is a $2$-section, and $F\redd\wedge D_Y=[2]$. By Claim \ref{cl:other-F} any degenerate fiber $F'\neq F$ is a chain $[1,2,\dots,2,1]$, so by Lemma \ref{lem:delPezzo_fibrations}\ref{item:Sigma} there is exactly one such $F'$. By Claim \ref{cl:other-F} $A\de \ftip{F'}$ meets $L$ and $H_2$. If $A\cdot H_2=2$ then the chain $F'-A=[1,2,\dots,2]$ meets $T$ and is disjoint from $H_2$, contrary to Lemma \ref{lem:GK-nodal-basics}\ref{item:GK-nodal-neg-def}. Thus $H_2$ meets the first and $l$-th component of $F'$ for some $l\geq 2$. Let $\tau\colon X\to \P^2$ be the contraction of $L+(F\redd-C)+(F'-A)$. Then $\tau(H_2)$ is a conic, so $4=\tau(H_2)^2=-2+2+l$, hence $l=4$. We conclude that $\cS_{Y}\in\{\rA_1+\rA_5,\rA_1+\rD_6,\rA_1+\rE_7\}$, so $\height(\bar{Y})=1$, see Table \ref{table:canonical}, a contradiction.
\end{proof}

	Recall that by Claim \ref{cl:F} we have $\height(\bar{Y})=2$, so $F\cdot D=3$. By Claim \ref{cl:not-vertical},  $D\hor$ is a sum of a $1$-section $H_T$ and $(D_Y)\hor$, so $(D_Y)\hor$ is either a $2$-section or two $1$-sections. We claim that every component of $(D_Y)\vert$ meeting $(D_Y)\hor$ lies in the fiber $F_{T}$ from Claim \ref{cl:T-vertical}. Indeed, suppose that a component $V$ of $D_Y$ meeting $(D_Y)\hor$ lies in a fiber $F\neq F_T$. By Claim \ref{cl:F}, $F$ does not contain any component of $D_{T}$. Since $T$ is vertical, Claim \ref{cl:other-F} implies that both tips of $F$ meet $(D_Y)\hor$, so $2=F\cdot (D_Y)\hor \geq 2+V\cdot (D_Y)\hor\geq 3$, a contradiction. 
	
	Let $G_i=[2,\dots,2,1]$ for $i=1,\dots,r$ be the connected components of $(F_{T})\redd-D_T\wedge (F_T)\redd$, see Claim  \ref{cl:F}\ref{item:F-with-T}. Let $A_{i}$ be the $(-1)$-curve in $G_{i}$. Note that $G_{i}\cp{j}$ has multiplicity at least $j$ in $F_T$. 
	
	Suppose $r\geq 2$. Since each $G_{i}$ meets $(D_Y)\hor$, we conclude that $r=2$ and $(D_Y)\hor$ meets $G_{1}\cp{1},G_{2}\cp{1}$. Hence the sum $G_{1}+(D_Y)\hor+G_{2}-A_2$ or $G_{1}+(D_Y)\hor$ is a chain $[1,2,\dots, 2]$ meeting the remaining part of $A_{1}+D$ only in $A_1\cap T$. This is a contradiction with Lemma \ref{lem:GK-nodal-basics}\ref{item:GK-nodal-neg-def}.
	
	Thus $r=1$. Since $D_Y$ has no circular subdivisor, we see that either $(D_Y)\hor$ is a $2$-section meeting $G_{1}\cp{1}$ or $G_{1}\cp{2}$, or two $1$-sections meeting $G_{1}\cp{1}$. Thus $G_{1}+(D_Y)\hor$ is a chain $[1,2,\dots, 2]$ or a fork $\langle 2;[1,2,\dots,2],[2],[2]\rangle$ meeting the remaining part of $A_1+D$ only in $A_{1}\cap T$, which again contradicts Lemma \ref{lem:GK-nodal-basics}\ref{item:GK-nodal-neg-def}.
\end{proof}

\begin{lemma}
	\label{lem:GK-ht-T=3}
	Assume that $\#\cT=3$, $\bar{T}$ is cuspidal, $\height(\bar{Y})\geq 2$ and $\height(\bar{X})\neq \height(\bar{Y})+1$. Then the following hold.
	\begin{enumerate}
		\item\label{item:T=3-L-horizontal} For every $\P^1$-fibration of $X$ of height at most $\height(\bar{Y})+1$ with respect to $D$, the elliptic tie $L$ is horizontal.
		\item\label{item:T=3-exceptions} The surface $\bar{X}$ is as Proposition \ref{prop:GK-ht-exceptions}\ref{item:GK-ht-exceptions-T=3}, in particular, $\height(\bar{X})=\height(\bar{Y})+1=3$.
	\end{enumerate}
\end{lemma}
\begin{proof}
	The divisor $D_{T}$ is a disjoint sum of $T=[s-3]$, $T_{3}=[3]$, and $T_{2}=[2]$, see  \ref{lem:cuspidal_resolution}\ref{item:C_3}. By Lemmas \ref{lem:GK_exceptions} and  \ref{lem:GK_intro}\ref{item:GK-intro-ht} we have $s\geq 6$ and $\height(\bar{X})\leq \height(\bar{Y})+1$. Fix a $\P^1$-fibration of $X$ whose fiber $F$ satisfies $F\cdot D\leq \height(\bar{Y})+1$.
	
	\ref{item:T=3-L-horizontal} Suppose that $L$ is vertical. Let $\tau\colon X\to X'$ be the contraction of $L$ and let $D'=\tau_{*}(D-T_2)$. Then $(X',D')$ is a minimal log resolution of a del Pezzo surface $\bar{X}'$ of rank one with the same descendant $(\bar{Y},\bar{T})$, and the corresponding number $\#\cT'$ equals $\#\cT-1=2$. By Lemma \ref{lem:GK-ht-nodal} we have $\height(\bar{X}')=\height(\bar{Y})+2$. On the other hand, $|\tau_{*}F|$ induces a $\P^1$-fibration of $X'$ such that $\tau_{*}F \cdot D'=F\cdot (D-T_{2})\leq \height(\bar{Y})+1$, a contradiction. 
	
	\ref{item:T=3-exceptions} Pulling back the linear equivalence $-K_{\bar{Y}}=\bar{T}$ we get a linear equivalence $-K_{X}=T+T_2+T_3+2L$. The adjunction formula gives $2=F\cdot (T+T_2+T_3+2L)$, so since $F\cdot L\geq 1$ by \ref{item:T=3-L-horizontal}, we infer that $L$ is a $1$-section and each $T_i$ is vertical. Since $T_1,T_2,T_3$ meet the $1$-section $L$, they lie in different fibers, call them $F,F_2,F_3$. Since $F_3$ consists of some $(-1)$-curves off $D$, some $(-2)$-curves in $D_Y$, and a $(-3)$-curve $T_3$ of multiplicity $1$, it is supported on a chain $[2,2,1,3]$ or  $[2,1,3,1]$, or on a fork $\langle 3;[1],[1],[1]\rangle$. Let $A$ be a $(-1)$-curve in $F_3$. The image $\sigma(A)$ of $A$ on $Y$ is a $0$-curve, so $\height(\bar{Y})\leq \sigma(A)\cdot \Exc\alpha=A\cdot D_{Y}$. We get $F_3\cdot D\geq 3(\height(\bar{Y})-1)$ if $(F_3)\redd=[2,2,1,3]$, $F_3\cdot D\geq 2(\height(\bar{Y})-1)+\height(\bar{Y})$ if $(F_3)\redd=[2,1,3,1]$, and $F_3\cdot D\geq 3 \height(\bar{Y})$ if $(F_3)\redd=\langle 3;[1],[1],[1]\rangle$. Since $F_3\cdot D\leq \height(\bar{Y})+1$ and $\height(\bar{Y})\geq 2$, we conclude that $(F_3)\redd=[2,2,1,3]$, $\height(\bar{Y})=2$, $A\cdot D\hor=1$ and $(F_3)\redd-A$ is disjoint from $D\hor$. Thus $D\hor$ is a $3$-section. By Lemma \ref{lem:delPezzo_fibrations}\ref{item:Sigma} every degenerate fiber  has exactly one $(-1)$-curve, so $F\redd=[(2)_{s-2},1,s-3]$, $(F_2)\redd=[2,1,2]$. The $3$-section $D\hor$ meets exactly one component of $F$: indeed, otherwise the components of $F$ meeting $D\hor$ have multiplicity at most $2$ in $F$, so since $s-3\geq 3$, they lie in $D$ and therefore $D$ has a circular subdivisor, which is false. Thus $D\hor$ meets the unique component of multiplicity $3$ in $F$, which is the third component of $F\redd=[(2)_{s-2},1,s-3]$. Similarly, we conclude that $D\hor$ meets $F_2$ in its $(-1)$-curve, with multiplicity $2$, and once in its $(-2)$-tip $(F_2)\redd\wedge D_Y$. Therefore, $\cS_{Y}=3\rA_2,\rA_2+\rA_5,\rA_2+\rE_6$ if $s=6,7,8$, respectively, as needed.
\end{proof}

\begin{lemma}
	\label{lem:GK-ht-cuspidal}
	Assume that $\height(\bar{Y})\geq 2$, $\#\cT\geq 4$ and $\bar{T}$ is cuspidal. Then $\height(\bar{X})=\height(\bar{Y})+2$.
\end{lemma}
\begin{proof}
	Suppose the contrary. Then $\height(\bar{X})\leq \height(\bar{Y})+1$ by Lemma \ref{lem:GK_intro}\ref{item:GK-intro-ht}. Fix a $\P^1$-fibration of $X$ whose fiber $F$ satisfies $F\cdot D\leq \height(\bar{X})+1$. We assume that the number $\#\cT\geq 4$ is minimal possible, and exactly as in the proof of Claim \ref{cl:L-horizontal} of Lemma \ref{lem:GK-ht-nodal}, we conclude that the curve $L$ is horizontal. Indeed, suppose $L$ is vertical. Then a swap $\tau\colon (X,D)\sqto (X',D')$ of $L$ for a $(-2)$-curve $C\subseteq D_T$ meeting $L$ gives a minimal log resolution of a del Pezzo surface $\bar{X}'$ of rank one, with the same descendant $(\bar{Y},\bar{T})$, with  $\#\cT'=\#\cT-1$, and with a $\P^1$-fibration whose fiber $F'=\tau_{*}F$ satisfies  $F'\cdot D'+F'\cdot L'=F\cdot D\leq \height(\bar{Y})+1$. Thus $F'\cdot D'\leq \height(\bar{Y})+1$, so $\#\cT'=3$ by minimality of $\#\cT$, and Lemma \ref{lem:GK-ht-T=3} implies that $F'\cdot D'=\height(\bar{Y})+1$ and $F'\cdot L'\geq 1$, a contradiction. 
	
	Pulling back the linear equivalence $-K_{\bar{Y}}=\bar{T}$ we get a linear equivalence $-K_{X}=T+E$, where $E$ is an effective divisor supported on $D_{T}+L-T$. Since $L$ is horizontal and $2=-K_{X}\cdot F=F\cdot (T+E)$, the multiplicity of $L$ in $E$ equals at most $2$. Writing $\phi=\alpha'\circ \sigma'$, where $\alpha'$ is the minimal log resolution of $(\bar{Y},\bar{T})$, we conclude that  $\sigma'$ is a composition of outer blowups on the $(-1)$-curves in the subsequent total transforms of $\bar{T}$. Thus  $\cT$ is as in Lemma \ref{lem:cuspidal_resolution}\ref{item:C_smooth} for $k=2$. Moreover, the multiplicity of $L$ in $E$ is exactly $2$, so the equality $2=F\cdot(T+E)$ implies that $D_{T}$ is vertical and $L$ is a $1$-section. Since $D_T$ is connected, it is contained in some fiber $F$. Let $B$ be the branching $(-2)$-curve in $D_T$, see Lemma \ref{lem:cuspidal_resolution}\ref{item:C_smooth}. Clearly, $B$ is branching in $F\redd$, too, and $B$ has multiplicity $1$ in $F$ because it meets $L$. Hence $0=F\cdot B\geq \beta_{F\redd}(B)+B^2\geq 1$, a contradiction.
\end{proof}

\begin{proof}[Proof of Proposition \ref{prop:GK-ht-exceptions}] Assume $3\leq\height(\bar{X})\leq \height(\bar{Y})+1$. By Lemmas \ref{lem:GK-ht-nodal}, \ref{lem:GK-ht-cuspidal} we have $\#\cT=1$ or $\#\cT=3$ and $\bar{T}$ is cuspidal; so Lemmas \ref{lem:GK-ht-T=1}, \ref{lem:GK-ht-T=3} imply that the singularity type of $\bar{X}$ is as in \ref{prop:GK-ht-exceptions}\ref{item:GK-ht-exceptions-T=1} or \ref{prop:GK-ht-exceptions}\ref{item:GK-ht-exceptions-T=3}, respectively. Conversely, if $\bar{X}$ has one of these singularity types then $\height(\bar{X})=3$ by Lemma \ref{lem:GK-ht-exceptions-ht=3}, as needed.
\end{proof}

\subsection{Uniqueness of a descendant}

Lemmas \ref{lem:cuspidal_resolution} and \ref{lem:GK_exceptions} complete the proof of Theorem \ref{thm:GK}\ref{item:GK-Y},\ref{item:GK-types}, that is, they give a list of possible singularity types of $\bar{X}$. Theorem \ref{thm:GK}\ref{item:GK-ht}, which computes $\height(\bar{X})$, is proved in Proposition \ref{prop:GK-ht-exceptions}. The remaining part of Theorem~\ref{thm:GK} addresses the question whether the singularity type determines $\bar{X}$ uniquely, up to an isomorphism. As a first result in this direction we prove Theorem \ref{thm:GK}\ref{item:GK-unique}, which asserts that $\bar{X}$ determines its descendant uniquely, provided $\height(\bar{X})\geq 3$. More precisely, we prove the following statement. 

\begin{lemma}
	\label{lem:leash-is-unique}
	Let $\bar{X}$ be as in Notation \ref{not:GK}. Assume that $\height(\bar{X})\geq 3$. Then the following hold.
	\begin{enumerate}
		\item \label{item:leash-phi} The morphism $\phi$ from Definition \ref{def:GK} is unique up to an automorphism of the source and target.
		\item \label{item:leash-type} The singularity types of $\bar{Y}$ and $\bar{T}$ are uniquely determined by the singularity type of $\bar{X}$.
	\end{enumerate}
\end{lemma}
\begin{proof}
	\ref{item:leash-phi} Let $\phi'\colon (X,D+L')\to (\bar{Y}',\bar{T}')$ be another such morphism. As in Notation \ref{not:GK} write $T=(\phi')^{-1}_{*}\bar{T}'$ and  $D=D_{T}'+D_{Y}'$, where $D_{Y}'=(\phi')^{-1}(\Sing \bar{Y}')$, $D_{T}'+L'=(\phi')^{*}\bar{T}'\redd$. To prove the lemma, it is enough to find an automorphism $\iota\in \Aut(X,D)$ such that $\iota(L)=L'$ and $\iota(T)=T'$.
	
	\begin{casesp}
	\litem{$T=T'$}\label{case:T=T'} Let $\pi_{T}\colon X\to X_{T}$ be the contraction of $D-T$, in other words, $X_{T}\to \bar{X}$ is the extraction of $T$, see \cite[1.39]{Kollar_singularities_of_MMP} or Section \ref{sec:Lacini}.  
	Then $\rho(X_{T})=2$. The cone $\operatorname{NE}(X_T)$ has two extremal rays: one spanned by $\pi_{T}(T)$, which corresponds to the contraction $X_{T}\to \bar{X}$, and the other spanned by $\pi_{T}(L)$, corresponding to the contraction $X_{T}\to \bar{Y}$. But there is also another extremal contraction $X_{T}\to \bar{Y}'$, corresponding to $\pi_{T}(L')$. Since $\pi_{T}(L')$ contracts to a smooth point and $\pi_{T}(T)$ does not, we infer that  $\pi_{T}(L')$ spans the same extremal ray as $\pi_{T}(L)$. Since $\pi_{T}(L)^2<0$ we get $\pi_{T}(L)=\pi_{T}(L')$, and as a consequence $L=L'$, as needed.
	
	\litem{$L=L'$} Then $D_{T}=D_{T}'$ and $D_{Y}=D_{Y}'$, in particular $\bar{Y}$ and $\bar{Y}'$ have the same singularity type $\cS_Y$. Write $\phi=\phi_{1}\circ\phi_0$, $\phi'=\phi_{1}'\circ \phi_{0}$ such that $\#\Exc\phi_0$ is maximal possible. If $\phi_{0}=\phi$ then $T=T'$, as needed, so we can assume $\phi_0\neq \phi$, hence $\phi_{1},\phi_{1}'\neq \id$. Let $E$ and $E'$ be the unique $(-1)$-curve in $\Exc\phi_{1}$, and $\Exc\phi_{1}'$, respectively. Since $\#\Exc\phi_0$ is maximal, $E$ meets $E'$ and $E=\phi_{0}(T')$, $E'=\phi_{0}(T)$.  In particular, $-1=\phi_{0}(T)^2\leq \bar{T}^2-4\overset{\mbox{\tiny{\eqref{eq:Noether}}}}{=}5-s\leq -1$ by Lemma \ref{lem:GK_exceptions}, so $\phi_{1}$ (respectively, $\phi_{1}'$) is a single blowup at the singular point of $\bar{T}$ (respectively, $\bar{T}'$), and $s=6$. Moreover, if $\cS_Y=\rE_6$ then $\bar{T}$ is nodal by Lemma \ref{lem:GK_exceptions}\ref{item:GK2_Ek}. Thus by \cite[Proposition 1.5]{PaPe_MT} the singularity type $\cS_Y$ determines the log surface $(\bar{Y},\bar{T})$ uniquely, up to an isomorphism, see Table \ref{table:canonical}. We conclude that there is an isomorphism $\alpha\colon (\bar{Y}',\bar{T}')\to (\bar{Y},\bar{T})$ such that $\alpha\circ\phi'=\phi$; so $\iota\de \phi^{-1}\circ \alpha \circ \phi'$ extends to the required automorphism of $(X,D)$.
	
	\litem{$T\neq T'$, $L\neq L'$} Suppose $D_{T}=D_{T'}$. Then the description of $D_T+L$ in Lemma \ref{lem:cuspidal_resolution} shows that $L$ and $L'$ meet the same components of $D$. Since $\NS_{\Q}(X)$ is generated by $K_X$ and the components of $D$, the curves $L$ and $L'$ are numerically equivalent, hence equal because $L^2<0$; contrary to our assumptions. Thus $D_{T}\neq D_{T}'$. Let $R$ be the sum of common components of $D_T$ and $D_{T}'$, and let $B=D_{T}-R$. Then $B\subseteq D'_Y$, so each connected component of $B$ is a $(-2)$-chain or a $(-2)$-fork which is also a connected component of $D_{T}$.
	
	Lemma \ref{lem:cuspidal_resolution} shows that such connected components occur only if $R=0$ or if $D_T$ is as in  \ref{lem:cuspidal_resolution}\ref{item:C_3} or \ref{lem:cuspidal_resolution}\ref{item:C_smooth}. Consider the last two cases. Then $R=[s-3]+[3]$ or $\langle k;[s-3],[3],[2]\rangle$. Since $T\neq T'$, we get $s=6$ and $T=[3]$, $T'=[3]$. Now we check directly that after the contraction of $D-T$, the image of $L'$ can still be contracted; so as in case \ref{case:T=T'} above we get $L=L'$, as needed.
	
	Consider now the case $R=0$, so $\bar{X}$ is canonical. Since $\height(\bar{X})\geq 3$, by \cite[Proposition 6.1]{PaPe_MT} $\bar{X}$ is of type $4\rA_{2}$ or $8\rA_{1}$. Now $D_{T}$ and $D_{T}'$ are different connected components of $D$, and the $(-1)$-curve $L$ (respectively, $L'$) meets $D$ in both tips of $D_{T}$ (respectively, $D_{T}'$). This condition uniquely determines the numerical classes of $L$ and $L'$, hence it uniquely determines $L$ and $L'$. The result now follows from the fact that $\Aut(\bar{X})$ acts transitively on $\Sing \bar{X}$ by \cite[Lemmas 7.2(d), 8.4(d)]{PaPe_MT}, cf.\ \cite[Corollary 5.2 and Proposition 5.22(4)]{KN_Pathologies}.
	\end{casesp}
	
	\ref{item:leash-type} We claim that $\cS$ determines $\cS_Y$ uniquely. If $\bar{X}$ is canonical, then since $\height(\bar{X})\geq 3$ we have $\cS=4\rA_2$ or $8\rA_1$, see Table \ref{table:canonical}, so Lemma \ref{lem:cuspidal_resolution} yields $\cS_Y=3\rA_2$ or $7\rA_1$, respectively, as needed. Suppose $\bar{X}$ is not canonical and $\cS_{Y}$ is not uniquely determined by $\cS$. We have two decompositions $\cS=\cS_Y+\cT=\cS'_Y+\cT'$, where $\cS_Y\neq \cS'_Y$ are canonical singularity types, and $\cT\neq \cT'$ are as in Lemma \ref{lem:cuspidal_resolution}. It follows that, say, $\cT$ has a common connected component with $\cS_Y'$. As before, we infer from  Lemma \ref{lem:cuspidal_resolution} that $\cT=[s-3]+[3]+[2]$ or $\langle k,[2],[3],[s-3]\rangle +[(2)_{k-2}]$ for some $k\geq 3$; and applying Lemma \ref{lem:cuspidal_resolution} again we get $\cT=\cT'$, a contradiction.
	
	In turn, $\cS_Y$ uniquely determines the singularity type of $\bar{T}$. Indeed, Lemma \ref{lem:GK_exceptions} implies that $\#\cS_Y\geq 6$ and $\bar{T}$ is nodal if  $\cS_Y\in\{\rE_6,\rE_7,\rE_8\}$, so the claim follows from the classification in Table \ref{table:canonical}.
\end{proof}
%

\subsection{Computation of moduli}

Recall that $\Pdeb(\cS)$ denotes the set of isomorphisms classes of del Pezzo surfaces of rank one and type $\cS$, admitting descendants with elliptic boundary; and $\Pnode(\cS_Y)$, $\Pcusp(\cS_Y)$ denotes the set of isomorphism classes of log surfaces $(\bar{Y},\bar{T})$, where $\bar{Y}$ is a canonical surface of rank one and type $\cS_Y$; and $\bar{T}\subseteq \bar{Y}\reg$ is a nodal or cuspidal member of $|-K_{\bar{Y}}|$, respectively. We assume that type $\cS$ is such that $\Pdeb(\cS)$ contains a surface $\bar{X}$ with $\height(\bar{X})\geq 3$. By Lemma \ref{lem:GK_exceptions} the same inequality holds for all surfaces in $\Pdeb(\cS)$. Lemma \ref{lem:leash-is-unique} implies that we have a map from $\Pdeb(\cS)$ to $\Pnode(\cS_Y)$ or $\Pcusp(\cS_Y)$, associating to a surface $\bar{X}$ its descendant $(\bar{Y},\bar{T})$. In the first case the entry in Table \ref{table:canonical} corresponding to $\cS_Y$ is $\rN$; in the second it is $\rC$ or $\rC_d$, where $d$ is the moduli dimension of $\Pcusp(\cS_Y)$, see \cite[Proposition 1.5]{PaPe_MT} or Lemma \ref{lem:GK_hi-3} below. In all the other cases we put $d=0$.
\smallskip

With this information at hand, we can now compute the numbers $\#\Pdeb(\cS)$ and thus prove the remaining parts \ref{item:GK_uniqueness-nodal},\ref{item:GK_uniqueness-cuspidal} of Theorem \ref{thm:GK}. The proof is organized as follows. The case when $\bar{T}$ is nodal is settled in Lemma \ref{lem:nodal}, where we use the fact that $\#\Pnode(\cS_Y)=1$ to prove that $\#\Pdeb(\cS)=1$. If $\bar{T}$ is cuspidal (so $\cha\kk\in \{2,3,5\}$, see Table \ref{table:canonical} and Lemma \ref{lem:GK_exceptions}) then the proof splits in two cases, depending on whether $(X,D)$ dominates the minimal log resolution of $(\bar{Y},\bar{T})$ or not, i.e.\ whether $\#\cT\geq 3$ or $\#\cT\in \{1,2\}$. In the first case we take the almost universal family of dimension $d$ representing $\Pcusp(\cS_Y)$ and lift it to one representing $\Pdeb(\cS)$ following Observation \ref{obs:blowups}: this is done in Lemma \ref{lem:GK-h1}\ref{item:GK-h1-moduli}. In case $\#\cT\leq 2$ we still get an almost faithful family representing $\Pdeb(\cS)$ by blowing down the above one. Nonetheless, it can happen that $h^1(\lts{X}{D})>d$, so an almost \emph{universal} family may no longer exist, see Proposition \ref{prop:GK-small}. 
\smallskip

First we make an elementary observation which will frequently allow us to assume that $\bar{Y}$ is primitive.

\begin{lemma}[Simplifying $\bar{Y}$]\label{lem:GK-swap}
	Assume that $\bar{Y}$ swaps to a canonical surface $\bar{Y}'$ of type $\cS_Y'$ such that  $w\de T^2+s-\#\cS_{Y}'\leq -2$. Let $\cS'=\cS_Y'+\cT'$, where type $\cT'$ is obtained from $\cT$ by replacing the weight of a vertex corresponding to $T$ with $w$. Then $\bar{X}$ swaps to a surface $\bar{X}'\in \Pdeb(\cS')$. More precisely, there is an inner morphism $(X,D+L+A)\to (X',D'+L'+A')$ where $A$ and $A'$ are $(-1)$-curves, $(X',D')$ is a minimal log resolution of $\bar{X}'\in \Pdeb(\cS')$, and $L'$ is an elliptic tie, see Notation \ref{not:GK}. In particular, $h^i(\lts{X}{D})=h^i(\lts{X'}{D'})$.
\end{lemma}
\begin{proof}
	We can assume that the swap $\bar{Y}\sqto \bar{Y}'$ is elementary, i.e.\ it is a contraction of a $(-1)$-curve $A_Y$. Since $\sigma_{*}(D_Y)$ consists of $(-2)$-curves, by Definition \ref{def:vertical_swap}\ref{item:def-swap-elementary}  we have $A_Y\cdot \sigma_{*}(D_Y)=1$. By adjunction, we have $A_Y\cdot \sigma_{*}(T_Y)=A_{Y}\cdot (-K_{Y})=1$. Thus the proper transform $A$ of $A_Y$ on $X$ satisfies $A\cdot (D+L)=2$, and meets $D+L$ in $T$ and in a $(-2)$-curve in $D_Y$. The required morphism is given by the contraction of $A$.  
\end{proof}

\begin{lemma}[Uniqueness in the nodal case, see Theorem \ref{thm:GK}\ref{item:GK_uniqueness-nodal}]\label{lem:nodal} 
	Assume that $\bar{T}$ is nodal. Then $\#\Pdeb(\cS)=1$ and the minimal log resolution $(X,D)$ of the unique surface $\bar{X}\in \Pdeb(\cS)$ satisfies $h^{i}(\lts{X}{D})=0$ for all $i$.
\end{lemma}
\begin{proof}
	Let $\bar{Y}\sqto \bar{Y}'$ be a swap to a vertically primitive surface as in Remark \ref{rem:canonical_ht=1} or Example \ref{ex:remaining_canonical}. Since $\bar{T}$ is nodal, looking at Table \ref{table:canonical} we see that $\bar{Y}'$ is of one of the types: $\rA_1+\rA_2$, $2\rA_1+\rA_3$, $\rA_1+2\rA_3$ if $\cha\kk\neq 2$; $3\rA_2$ if $\cha\kk\neq 3$;  $2\rA_4$ if $\cha\kk\neq 5$; or $\rA_1+\rA_2+\rA_5$ if $\cha\kk\neq 2,3$. Let $\bar{T}'\subseteq \bar{Y}'$ be the image of $\bar{T}$, and let $(Y',B')$ be the minimal log resolution of $(\bar{Y}',\bar{T}')$. Consider the $\P^1$-fibrations of $Y'$ and $X$ induced by the ones in Figures \ref{fig:ht=1} or \ref{fig:ht=2_w=1-basic}, \ref{fig:remaining_canonical}, and let $\check{B}'$ be the sum of $B'$ and all vertical $(-1)$-curves. Those $(-1)$-curves are the proper transforms of those shown in Figures \ref{fig:ht=1}, \ref{fig:ht=2_w=1-basic}, \ref{fig:remaining_canonical}, the elliptic tie $L'$ over the node of $\bar{T}'$, and the proper transform of the fiber passing through the node, call it $V$. The divisor $\check{B}'$ is snc: to see this, we need to check that $V$ meets $\check{B}-V$ normally, which is obvious in all cases except $\cS_Y=2\rA_4$, when one needs to check that $V$ does not pass through the common point of the horizontal $(-2)$-curves. This follows from a direct computation using the construction of $(\bar{Y},\bar{T})$ in \cite[\sec 9]{PaPe_MT}, relying on the fact that $\cha\kk\neq 5$. 
	
	Let $\check{D}'$ be the reduced total transform of $\check{B}'$ on $X$: again, it is a sum of $D$ and $(-1)$-curves. We have an inner morphism $(X,\check{D})\to (Y',\check{B}')$, so by Lemma \ref{lem:blowup-hi}\ref{item:blowup-hi-inner}  and \ref{lem:adding-1}\ref{item:adding-1-faithful} it is enough to prove that $h^{i}(\lts{Y'}{\check{B}'})=0$ and that the combinatorial type $\cS'$ of $(Y',\check{B}')$ has $\#\cP_{+}(\check{\cS}')=1$. 

	Let $\phi\colon Y'\to Z$ be the contraction of $V$. Put $B_Z=\phi_{*}\check{B}'$ if $\height(\bar{Y})=1$ and $B_{Z}=\phi_{*}(\check{B}-L')-H'$, where $H'$ is one of the horizontal $(-1)$-curves in $\phi_{*}B$. Furthermore, in case $\height(\bar{Y})=1$ replace $Z$ with a blowup at one of the two common points of the proper transforms of $\bar{T}$ and $L$, and $B_Z$ by its reduced total transform. In any case, we conclude as before that it is enough to prove the required statements for $(Z,B_Z)$.
	
	Let $\tau\colon Z\to \P^2$ be the contraction of the horizontal $(-1)$-curve in $B_Z$ and all vertical curves disjoint from it. Then $\tau_{*}B_Z$ consists of a conic $\cc$ and lines $\ll_{1},\dots, \ll_{\nu}$ for some $\nu\in \{2,3\}$ meeting at a point off $\cc$, such that $\ll_2,\dots,\ll_{\nu}$ are tangent to $\cc$ and $\ll_1$ is not. Let $p$ be the base point of $\tau^{-1}$ in $\ll_1\cap \cc$, let $\ll$ be the line joining $p$ with $\ll_2\cap \cc$, and if $\nu=3$ let $\ll'$ be the line joining $\ll_2\cap \cc$ with $\ll_3\cap \cc$. The proper transforms of $\ll$ and (in case $\nu=3$) $\ll'$ are $(-1)$-curves meeting $B_Z$ normally, so as before we can add it to $B_Z$. In case $\nu=2$ we have an inner morphism $(Z,B_Z)\to (\P^2,\pp)$, where $\pp$ is a sum of four lines in a general position, so $(\P^2,\pp)$ is unique up to an isomorphism and $h^{i}(\lts{\P^2}{\pp})=0$ by Lemma \ref{lem:h1}\ref{item:h1-P2}, which ends the proof. In case $\nu=3$ we have an inner morphism $(Z,B_Z)\to (Z',B_{Z}')$ such that the proper transform $T'$ of $\bar{T}$ on $Z'$ is a $(-1)$-curve, so we can replace $(Z,B_Z)$ by $(Z',B_{Z}'-T')$. The latter admits an inner morphism onto $(\P^1\times \P^1,B)$, where $B$ is a sum of $3$ vertical and $3$ horizontal lines. Like before, the result follows since $(\P^1\times \P^1,B)$ is unique up to an isomorphism and $h^i(\lts{\P^1\times \P^1}{B})=0$ for all $i$ by Lemma \ref{lem:h1}\ref{item:h1-grid}. 
\end{proof}

From now on we assume that $\bar{T}$ is cuspidal, so $\cha\kk\in \{2,3,5\}$. The corresponding entry in Table \ref{table:canonical} is $\rC_{d}$ or $\rC$. In the latter case we put $d=0$. Moreover, we put $h^i\de h^i(\lts{X}{D})$ for a minimal log resolution $(X,D)$ of some fixed $\bar{X}\in \Pdeb(\cS)$. We begin with case $\#\cT=3$, i.e.\ when $(X,D+L)$ is the minimal log resolution of $(\bar{Y},\bar{T})$. We note that in case $d\geq 1$ the result essentially follows from \cite[Proposition 1.5(c)]{PaPe_MT}, but since the argument in case $d=0$ works similarly in case $d\geq 1$, too, we sketch a proof for any $d$.

\begin{lemma}[Moduli in case $\bar{T}$ cuspidal, $\#\cT=3$]\label{lem:GK_hi-3}
	Assume 
	$\#\cT=3$. Then $\Pdeb(\cS)$ has moduli dimension $d$. Moreover, we have $(h^0,h^1,h^2)=(0,d,d+1)$ for any $\bar{X}\in \Pdeb(\cS)$.
\end{lemma}
\begin{proof} 
%
	By Lemma \ref{lem:GK-swap} we can assume that $\bar{Y}$ is primitive. Assume $\cha \kk=2$, so $\cS_Y=2\rA_1+\rA_3, 3\rA_1+\rD_4, 7\rA_1$ or $8\rA_1$, and $d=0,1,2$ or $3$, respectively. Consider a $\P^1$-fibration of $X$ which is a pullback of the one of $Y$ shown in Figure \ref{fig:2A1+A3}, \ref{fig:3A1+D4}, \ref{fig:KM_surface} or \cite[Figure 6(b)]{PaPe_MT}, see Figure \ref{fig:8A1-GK} for the latter. Let $V$ be the proper transform of the fiber passing through the cusp of $\bar{Y}$, and let $\phi$ be the contraction of $V$. In cases $\cS_Y=2\rA_1+\rA_3$ and $3\rA_1+\rD_4$ we get this way an inner morphism $\phi\colon (X,D)\to (X',D')$ onto a log surface as in Lemma \ref{lem:7A1-family} for $\nu=d+1$, so the result follows from Lemma \ref{lem:inner}. Consider cases $\cS_Y=7\rA_1$ or $8\rA_1$. Put $B=D-(D_Y)\hor$. We have a composition of $d-1$ outer blowdowns $(X,B)\to (X',D')$ onto a log surface as in  Lemma \ref{lem:7A1-family} for $\nu=4$, so $h^i(\lts{X'}{D'})=0,1,2$ for $i=0,1,2$. By Lemma \ref{lem:blowup-hi}\ref{item:blowup-hi-outer} we get $h^{i}(\lts{X}{B})=0,d,2$ for $i=0,1,2$. By Lemma \ref{lem:h1}\ref{item:h1_exact} we get $h^1\leq d$ and $h^2=h^1+1$. Moreover, by  Lemma \ref{lem:outer}\ref{item:outer-trivial} we get an almost universal family representing the image of $\cP_{+}(\check{\cS})$ in $\cP_{+}(\check{\cS}')$, where $\check{\cS}$, $\check{\cS}'$ are the combinatorial types of $(X,D+L)$ and $(X,B+L)$, respectively. By construction, this family can be extended (by adding boundary components) to an almost faithful family representing $\Pdeb(\cS)$. It follows that $h^1\geq d$. Therefore, we have $h^1=d$ and the natural map between deformations of $(X,D)$ and $(X,B)$ is an isomorphism. Thus our family is almost universal, as needed. For an alternative construction of such families see \cite[Lemmas 8.2(a), 8.4(a)]{PaPe_MT}.
	\smallskip 
	
	Consider the case $\cha\kk=3$, $\cS_Y=3\rA_2$. We keep the notation from Example \ref{ex:ht=2}\ref{item:3A2_construction}, and denote the proper transforms on $X$ of the curves in Figure \ref{fig:3A2} by the same letters. The image of $T$ on $\P^2$ is a cuspidal cubic $\qq$  tangent to $\cc,\ll_2$ at $p_1,p_0$ with multiplicity $3$, so since $\cha\kk=3$, the line $\ll'$ joining $p_2$ with the cusp $q\in \qq$ is tangent to $\qq$ at $q$, too, see \cite[Lemma 5.5(b)]{PaPe_MT}. For $i\in \{0,1\}$ let $\ll_i'$ be the line joining $q$ with $p_i$, and let $L',L_i'$ be the proper transforms of $\ll',\ll_i'$ on $X$, they are $(-1)$-curves. Let $D'=D+L+L_1+L_1'-C$. Let $G_i$ for $i\in \{1,2\}$ be the connected components of $D_Y$ meeting $L_1$ and $A_i$. The contraction of $(L_0'+G_0)+(A_1+G_1)+(L'+L_2)$ is an inner morphism $(X,D')\to (\P^1\times \P^1,B)$, where $B$ is the sum of three horizontal and three vertical lines. Lemmas \ref{lem:h1}\ref{item:h1-grid} and \ref{lem:inner} show that $h^{i}(\lts{X}{(D-C)})=0$ for all $i$, and the combinatorial type $\cS'$ of $(X,D')$ satisfies $\#\cP_{+}(\cS')=1$, so $\#\Pdeb(\cS)=1$. Since $C=[2]$ we get $(h^0,h^1,h^2)=(0,0,1)$ by Lemma \ref{lem:h1}\ref{item:h1_exact}.
	
	In case $\cha\kk=3$, $\cS_Y=4\rA_2$ the first statement is proved in \cite[Lemma 7.2(b)]{PaPe_MT} and the second is elementary. For a direct argument we proceed as follows. Recall that $\bar{Y}$ has a descendant of type $3\rA_2$ with cuspidal elliptic boundary, and the corresponding $(-1)$-curve over the cusp, call it $L'$, has $L'\cdot D_Y=2$ and meets $D_Y$ in a node. By adjunction $L'\cdot T=1$, so $L'\cdot D=3$. Let $G,H$ be the components of $D_Y$ meeting $L'$, and let $(X,D-H)\to (X',D')$ be a contraction of $L'+G$. It is a composition of an inner and outer blowup, and the resulting log surface $(X',D')$ is a minimal log resolution of a del Pezzo surface as in the previous paragraph, i.e.\ in $\Pdeb(3\rA_2+[3]+[3]+[2])$. Thus $h^{i}(\lts{X'}{D'})=0,0,1$ for $i=0,1,2$, so Lemma \ref{lem:blowup-hi} implies that $h^i(\lts{X}{(D-H)})=0,1,1$, so $h^0=0$ and $h^1=h^2-1\geq 1$ by Lemma \ref{lem:h1}\ref{item:h1_exact}. Moreover, Lemma \ref{lem:outer}\ref{item:outer-trivial} implies that $\Pdeb(\cS)$ is represented by an almost faithful family over a curve, so as in cases $7\rA_1$, $8\rA_1$ above we conclude that $\Pdeb(\cS)$ has moduli dimension $h^1=1$, as needed.
	\smallskip
	
	It remains to consider case  $\cha\kk=5$, $\cS_Y=2\rA_4$. Again, we denote the curves in Figure \ref{fig:2A4} and their proper transforms on $X$ by the same letters. For $i=0,2$ let $L_i'=[1]$ be the proper transform of the line joining $p_i$ with the cusp of the image of $T$ on $\P^2$. Write $D_T=T+U+V$, where $U=[3]$, $V=[2]$. Let $G_T=[(2)_{4}]$, $G_U=[2,2]$ and $G_V=[2]$ be the connected components of $D_Y-C$. Then $T+A_1+G_T=[5,1,(2)_{4}]$ supports a fiber a $\P^1$-fibration such that $D\hor=C$ is a $5$-section, and $U+L_0'+G_U$ supports another degenerate fiber, see Figure \ref{fig:2A4-GK}. It follows from Lemma \ref{lem:fibrations-Sigma-chi} that the remaining degenerate fiber is supported on $V+A+G_V=[2,1,2]$ for some $(-1)$-curve $A$. Let $D'=D+L+A_2+L_2'-C$. The contraction of $(A_1+G_T)+(L_0'+G_U)+(A+G_V)$ is an inner morphism $(X,D')\to (\P^1\times \P^1,B)$, where $B$ is a sum of three horizontal and three vertical lines, so $h^i(\lts{X'}{D'})=0$ for $i=0,1,2$. We conclude exactly as in the case $\cS_Y=3\rA_2$. 
\end{proof}

\begin{figure}[htbp]\vspace{-0.5em}
		\subcaptionbox{$\cS_Y=8\rA_1$, $\cha\kk=2$, cf.\ \cite[Figure 6(b)]{PaPe_MT} \label{fig:8A1-GK}}[.45\linewidth]{	
			\begin{tikzpicture}[scale=0.8]
				\draw[thick] (-0.1,1.7) -- (2.3,1.7) to[out=0,in=-100] (3.05,2.3) to[out=80,in=180] (3.6,2.8)-- (5,2.8) to[out=0,in=-100] (5.15,2.9) to[out=80,in=-180] (5.3, 3) -- (5.6,3);
				\draw[thick] (-0.1,1.3) -- (2.3,1.3) to[out=0,in=100] (3.05,0.7) to[out=-80,in=180] (3.4,0.2) to[out=0,in=180] (4.8,2.5) -- (5,2.5) to[out=0,in=-100] (5.1,2.6) to[out=80,in=-180] (5.2, 2.7) -- (5.6,2.7);
				\draw[thick] (-0.1,1.5) -- (4,1.5);
				\draw[thick] (4.2,1.5) -- (4.6,1.5) to[out=0,in=180] (5,1) -- (5.6,1);
				\node at (3.5,1.7) {\small{$-5$}};
				\draw (0.2,3.1) -- (0,1.9);
				\draw[dashed] (0,2.1) -- (0.2,0.9);
				\draw (0.2,1.1) -- (0,-0.1);
				\draw (1.2,3.1) -- (1,1.9);
				\draw[dashed] (1,2.1) -- (1.2,0.9);
				\draw (1.2,1.1) -- (1,-0.1);
				\draw (2.2,3.1) -- (2,1.9);
				\draw[dashed] (2,2.1) -- (2.2,0.9);
				\draw (2.2,1.1) -- (2,-0.1);
				\draw[dashed] (3.2,3.1) -- (2.9,1.4);
				\draw[dashed] (2.9,1.6) -- (3.2,-0.1);
				\draw[dashed] (5.2,3.1) -- (5,2.1);
				\draw (5,2.3) -- (5.2,1.3);
				\node at (4.8,1.8) {\small{$-3$}};
				\draw[dashed] (5.2,1.5) -- (5,0.5);
				\draw (5,0.7) -- (5.2,-0.3);			
			\end{tikzpicture}	
		}
		\subcaptionbox{$\cS_Y=2\rA_4$, $\cha\kk=5$, cf.\ \cite[p.\ 12]{PaPe_MT} \label{fig:2A4-GK}}[.45\linewidth]{	
			\begin{tikzpicture}[scale=0.8]
			\draw[dashed] (-0.1,4.8) -- (3,4.8);
			\draw (0.2,5) -- (0,3.8);
			\node at (-0.2,4.4) {\small{$-5$}};
			\draw[dashed] (0,4) -- (0.2,3.2);
			\draw (0.2,3.4) -- (0,2.6);
			\draw (0,2.8) -- (0.2,2);
			\draw (0.2,2.2) -- (0,1.4);
			\draw (0,1.6) -- (0.2,0.8);
			\draw (1.4,5) -- (1.2,3.8);
			\node at (1,4.4) {\small{$-3$}};
			\draw[dashed] (1.2,4) -- (1.4,3);
			\draw (1.4,3.2) -- (1.2,2.2);
			\draw (1.2,2.4) -- (1.4,1.4);
			\draw (2.6,5) -- (2.4,3.8);
			\draw[dashed] (2.4,4) -- (2.6,2.8);
			\draw (2.6,3) -- (2.4,1.8);
			\draw[thick] (-0.2,3.6) -- (0.4,3.6) to[out=0,in=180] (0.8,3.1) -- (2.2,3.1) to[out=0,in=100] (2.57,2.9) to[out=-80,in=180] (2.7,2.7) -- (3,2.7);
			\end{tikzpicture} 	
		}\vspace{-0.5em}
	\caption{Minimal log resolutions of some log surfaces $(\bar{Y},\bar{T})$ with $\bar{T}$ cuspidal.}\vspace{-1em}
	\label{fig:GK-special}	
\end{figure}
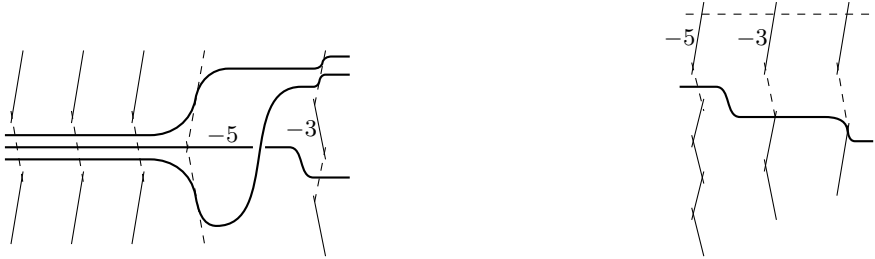

\begin{lemma}[Moduli in the cuspidal case, see Theorem \ref{thm:GK}\ref{item:GK_uniqueness-cuspidal}]\label{lem:GK-h1} 
	Let $t$ be as in Theorem \ref{thm:GK}\ref{item:GK-t}.
	\begin{enumerate}
		\item\label{item:GK-h1-family}	The set $\Pdeb(\cS)$ is represented by an almost faithful family of dimension $d+t$.
		\item\label{item:GK-h1-moduli}	Assume 
		$\#\cT\geq 3$. Then $\Pdeb(\cS)$ has moduli dimension $d+t$.
		\end{enumerate}
\end{lemma}
\begin{proof}
	Let $\check{\cS}$ be the combinatorial type of $(X,D+L)$. By Lemmas \ref{lem:leash-is-unique} and \ref{lem:adding-1}\ref{item:adding-1-faithful}, it is enough to prove the above properties for $\#\cP_{+}(\check{\cS})$. First, we prove part \ref{item:GK-h1-moduli}. Since $\#\cT\geq 3$, the morphism $\phi\colon (X,D+L)\to (\bar{Y},\bar{T})$ factors through $\psi\colon (X,D+L)\to (X',D'+L')$, where $(X',D'+L')$ is the minimal log resolution of $(\bar{Y},\bar{T})$, and $L'=[1]$. If $t=0$ then $\psi$ is inner, so the claim follows from Lemmas \ref{lem:inner} and \ref{lem:GK_hi-3}. 
	
	Assume $t=1$. Then $\psi=\gamma\circ \psi'$, where $\psi'\colon (X,D+L)\to(X'',D''+L'')$ is inner, and $\gamma\colon (X'',D''+L'')\to (X',D'+L')$ is a single outer blowup, centered on $L'$, which is  the branching component of the fork $\psi_{*}(D_{T}+L)=\langle 1;[2],[3],[s-3]\rangle$, see Lemma \ref{lem:cuspidal_resolution}\ref{item:C_smooth}. Since every vector field on $X'$ tangent to $D'+L'$ vanishes at three points of $L'$, it vanishes identically along $L'$. If $s\geq 7$ then $\Aut(X',D'+L')$ fixes each of the three points in $L'\cap D'$, so it acts trivially on $L'$. In this case, applying Lemma \ref{lem:outer}\ref{item:outer-trivial} we conclude that $\cP_{+}(\check{\cS})$ has moduli dimension $d+1$, as claimed. Assume $s=6$. Then $d=0$, see Table \ref{table:canonical}, so $\#\Pcusp(\cS_Y)=1$. Still, the restriction of $\Aut(X',D'+L')$ to $L'\setminus D'\cong \P^1\setminus \{3\mbox{ points}\}$ is a finite group (in fact, it is $\Z/2$ interchanging the common points of $L'$ with $(-3)$-curves in $D'$), so applying Lemmas \ref{lem:outer}\ref{item:outer-trivial}  and \ref{lem:inner} as before we conclude that  $\cP_{+}(\check{\cS})$ has moduli dimension $1$.
	
	To prove \ref{item:GK-h1-family} it remains to construct almost faithful families in case $\#\cT\leq 2$. To do this, we take one constructed above for $\#\cT=3$ and blow down divisors corresponding to $L$ and the $(-2)$-curve meeting it.
\end{proof}

\begin{remark}[Bases and symmetry groups in Lemma \ref{lem:GK-h1}\ref{item:GK-h1-family}, see Table \ref{table:symmetries}]\label{rem:GK-bases}
	Let $B_0,G_0$ be the base and symmetry group of an almost faithful family representing $\Pcusp(\cS_Y)$: if $d>0$ then $B_0,G_0$ are listed in \cite[Table 2]{PaPe_MT}, which we repeat in Table \ref{table:symmetries} below. If $d=0$ then $B_0=\{\pt\}$, $G_0=\{\id\}$. Let $f\colon (\cX,\cD)\to B$ be the almost faithful family representing $\Pdeb(\cS)$ constructed in Lemma \ref{lem:GK-h1}\ref{item:GK-h1-family}, and let $G$ be its symmetry group. The proof of Lemma \ref{lem:GK-h1}\ref{item:GK-h1-family} shows the following. If $t=0$ then $B=B_0$ and $G=G_0$. If $t=1$ then $B\cong B_0 \times L^{\circ}$, where $L^{\circ}\cong \P^1\setminus \{3\mbox{ points}\}$ parametrizes the center of the outer blowup, and either $\#\cS_Y\geq 7$ and  $G=G_0$ acts trivially on $L^{\circ}$, or we have $\#\cS_Y=6$, $B\cong L^{\circ}$ and $G\cong \Z/2$.
\end{remark} 

\begin{table}[htbp]
{\renewcommand{\arraystretch}{1.3}
	\begin{tabular}{r||c|c|c|c|c|c}
		$\cha\kk$ & 3 & \multicolumn{5}{c}{$2$} \\ \hline 
		type $\cS$ of $\bar{Y}$ & $4\rA_2$ & $3\rA_1+\rD_4$ & $2\rA_1+\rD_6$ & $7\rA_1$ & $4\rA_1+\rD_4$ & $8\rA_1$  \\
		\hline 
		base $B$ & $\P^1\setminus \P^1_{\F_3}$ & \multicolumn{2}{c|}{$\P^1\setminus \P^1_{\F_2}$} & \multicolumn{2}{c|}{$ \P^2\setminus \{\mbox{all } \F_2\mbox{-rational lines}\}$} & $3$-fold \cite[(8)]{PaPe_MT}  \\ \hline 
		symmetry group $G$	& $\PGL_{2}(\F_{3})$ 
		& $S_3$  & $\Z/2$ & $\PGL_{3}(\F_2)$ & $S_4$ & $\PGL_{3}(\F_2)$ 
	\end{tabular}
}
	\caption{Remark \ref{rem:GK-bases}: bases and symmetry groups of almost faithful families in Lemma \ref{lem:GK-h1}\ref{item:GK-h1-family}, case $t=0$, $d>0$. In case $d=t=0$ the base $B$ is a point and $G=\{\id\}$. In case $t=1$ the base is $B\times (\P^1\setminus \{3 \mbox{ points}\}$) and the symmetry group is $G$ if $\#\cS_Y\geq 7$ and $\Z/2$ if $\#\cS_Y=6$.}
	\label{table:symmetries}
\end{table}

To prove Theorem \ref{thm:GK} it remains to study the case when $\bar{T}$ is cuspidal and $\#\cT\leq 2$. The result is summarized  in Proposition \ref{prop:GK-small} below. Its proof relies on direct computations made in Lemmas \ref{lem:GK-small-moduli} and \ref{lem:GK-small-fails}.

\begin{proposition}[The special case $\#\cT\leq 2$ of Theorem \ref{thm:GK}\ref{item:GK_uniqueness-cuspidal}]\label{prop:GK-small}
	We keep Notation \ref{not:GK} and assume $\height(\bar{X})\geq 3$. Assume furthermore that $\bar{T}$ is cuspidal and $\#\cT\leq 2$. Put $h^i=h^i(\lts{X}{D})$. Then the following hold. 
	\begin{enumerate}
		\item\label{item:GK-small-h0} We have $h^0=0$ and $h^1=h^2\in \{d,d+1\}$.
		\item\label{item:GK-small-moduli} Assume that $\cha\kk=2$ or $\cha\kk=3$, $\#\cT=1$. Then $h^1=d$ and $\Pdeb(\cS)$ has moduli dimension $d$.
		\item\label{item:GK-small-fails} Assume that $\cha\kk=5$ or $\cha\kk=3$, $\#\cT=2$ and $\cS_Y\neq 4\rA_2$. Then $h^1=1=d+1$.  
	\end{enumerate}
\end{proposition}

\begin{remark}
	The only case missing from Proposition \ref{prop:GK-small}\ref{item:GK-small-moduli},\ref{item:GK-small-fails} is $\cha\kk=3$, $\cS=4\rA_2+[4,2]$. In this case, by Lemma \ref{lem:GK-h1}\ref{item:GK-h1-moduli} the set $\Pdeb(4\rA_2+[4,2])$ is represented by an almost faithful family of dimension $d=1$, so $h^1\in \{1,2\}$ by Proposition \ref{prop:GK-small}\ref{item:GK-small-h0}. We do not know which of these two possibilities holds.
\end{remark}

\begin{lemma}\label{lem:GK-small-moduli}
	Let $\cS$ be as in Proposition \ref{prop:GK-small}\ref{item:GK-small-moduli}. Then $h^0=0$ and $h^2\leq d$.
\end{lemma}
\begin{proof}
	We note that by Lemma \ref{lem:blowup-hi} for any birational morphism $(X_1,D_1)\to (X_2,D_2)$ between log smooth surfaces we have $h^i(\lts{X_1}{D_1})\leq h^i(\lts{X_2}{D_2})$ for $i=0,2$, 
	so we can replace $(X,D)$ by its images. 
	
	Assume first that $\cha\kk=2$. 
	Then $\bar{Y}$ swaps to a surface $\bar{Y}'$ of type $\rA_1+\rA_5$, $3\rA_1+\rD_4$, $7\rA_1$ or $8\rA_1$. In the first case, by Lemma \ref{lem:GK_exceptions} we have $\#\cT=2$, so by Lemma \ref{lem:GK-swap} $\bar{X}$ swaps to the surface of type $(\rA_1+\rA_5)+[2,2]$, and $h^i=h^i(\lts{Z}{D_Z})$, where $(Z,D_Z)$ is its minimal log resolution, see Figure \ref{fig:A5+A2+A1}. We keep the notation from Example \ref{ex:ht=2}\ref{item:A1+A2+A5_construction}. Let $V$ be the proper transform of the line joining $q$ with $p_2$, and let $\check{D}_Z$ be the sum of $D_Z+V$ and all $(-1)$-curves in Figure \ref{fig:A5+A2+A1}. We have an inner morphism $(Z,\check{D}_Z)\to (\P^1\times \P^1,B)$, where $B$ is a sum of three vertical and three horizontal lines, so $h^i=h^i(\lts{\P^1\times \P^1}{B})=0$, as needed. 
	
	In the remaining cases we argue as in the proof of Lemma \ref{lem:GK_hi-3}. By Lemma \ref{lem:GK-swap} we can assume that $\cS_Y=\in \{3\rA_1+\rD_4,7\rA_1,8\rA_1\}$. Consider a pullback to $X$ of the $\P^1$-fibration shown in Figures \ref{fig:3A1+D4}, \ref{fig:KM_surface} (for $\nu=3$) or \cite[Figure 6(b)]{PaPe_MT}, and let $V$ be the proper transform of the fiber passing through the cusp of $\sigma(T)$. 
	
	Assume $\cS_Y=3\rA_1+\rD_4$. Let $\phi\colon (X,D)\to (Z,D_Z+H)$ be the contraction of $V$, where $H=[1]$ is the image of $(D_Y)\hor$. Then $h^i=h^i(\lts{Z}{(D_Z+H)})=h^i(\lts{Z}{D_Z})$ by Lemma \ref{lem:h1}\ref{item:h1-1_curve}. Contracting the image of $D_T-T$ we get a minimal log resolution of a surface of type $7\rA_1$, so $h^0=0$ and $h^2=1$ by Lemma \ref{lem:7A1-family}\ref{item:7A1-family-h2}. 
	
	Assume $\cS_Y=7\rA_1$ or $8\rA_1$, so $d=2$ or $3$, respectively. Let $B=D-(D_Y)\hor$, $\tilde{h}^i=h^i(\lts{X}{B})$. Since $D-B$ is a sum of $d-1$ $(-2)$-curves, by Lemma \ref{lem:h1}\ref{item:h1_exact} we have $h^0=\tilde{h}^0$ and $h^2\leq \tilde{h}^2+d-1$.  The contraction of $V+D_T-T$ and, in case $\cS_Y=8\rA_1$, of a vertical $(-1)$-curve disjoint from $(D_Y)\vert$, cf.\ Figure \ref{fig:8A1-GK}, maps $(X,B)$ onto a minimal log resolution of the surface of type $7\rA_1$, so $\tilde{h}^0=0$, $\tilde{h}^2=1$ by Lemma \ref{lem:7A1-family}\ref{item:7A1-family-h2}, as needed. 
	
	\smallskip
	
	Assume now that $\cha\kk=3$, so $\#\cT=1$ by assumption \ref{prop:GK-small}\ref{item:GK-small-moduli}. In case $\cS_Y=4\rA_2$ let $H$ be a component of $D_Y$ and let $B=D-H$; in other cases put $B=D$. By Lemma \ref{lem:h1}\ref{item:h1_exact} it is enough to prove $h^i(\lts{X}{B})=0$ for $i=0,2$. We have a birational morphism $(X,B)\to (X',D')$, where $(X',D')$ is the minimal log resolution of a surface $\bar{Y}'$ of type $3\rA_2$ blown up in a point away from the boundary. Indeed, if $\cS_Y\neq 4\rA_2$ then $\bar{Y}$ swaps to $\bar{Y}'$ and we get the required morphism as in Lemma \ref{lem:GK-swap}. In turn, if $\cS_Y=4\rA_2$ then, since $\bar{Y}$ has a descendant with elliptic boundary, see \cite[Example 7.1]{PaPe_MT}, there is a $(-1)$-curve $L'$ meeting $D_Y$ only in $H\cap (D_Y-H)$, so the required morphism is the contraction of $L'$ and the component of $D_Y-H$ meeting it. 
	
	It remains to prove that $\tilde{h}^i\de h^i(\lts{X'}{D'})=0$ for $i=0,2$. Example \ref{ex:ht=2}\ref{item:3A2_construction} shows that there is an inner morphism $(X',D')\to (\F_1,\Delta')$, where $\Delta'$ is the sum of the negative section $\Sec_1$ and three sections disjoint from $\Sec_1$ meeting normally, so $\tilde{h}^i=h^i(\lts{\F_1}{\Delta'})$. The contraction of $\Sec_1$ maps $\Delta'$ to a triangle $\Delta\subseteq \P^2$, so by Lemma \ref{lem:blowup-hi} we have $\tilde{h}^i=h^i(\lts{\P^2}{\Delta})=0$ for $i>0$ and $\tilde{h}^0=h^0(\lts{\P^2}{\Delta})-2=0$, as needed.
\end{proof}

\begin{lemma}\label{lem:GK-small-fails}
	Let $\cS$ be as in Proposition \ref{prop:GK-small}\ref{item:GK-small-fails}. If $\cha\kk=5$ assume further that $\#\cT=1$. Then $h^2\geq 1$.
\end{lemma}
\begin{proof}
	Assume $\cha\kk=3$, so $\cS_{Y}\neq 4\rA_2$ and $\#\cT=2$. By Lemma \ref{lem:GK-swap} we can assume that 
	$\cS_{Y}=3\rA_2$, so $\cS=4\rA_2$. Denote the proper transforms on $X$ of the curves in Figure \ref{fig:3A2} by the same letters. Let $\check{D}=D+A_0+A_1+L_1$, so $h^2=h^2(\lts{X}{\check{D}})$. Let $G=D_T-T$. We have $2K_{X}+\check{D}=G+L$, so by Lemma \ref{lem:h1}\ref{item:h1-h2=h0}, $h^2=h^0(\lts{X}{\check{D}}\otimes \cO_{X}(G+L))$. We will construct a nonzero global section of the latter sheaf.
	
	Let $\tau\colon (X,\check{D})\to (\P^2,\pp)$ be the composition of $\sigma\colon X\to Y$ with the morphism from Example \ref{ex:ht=2}\ref{item:3A2_construction}. Then $\pp$ is a sum of the triangle $\ll_1+\ll_2+\cc$ and a cuspidal cubic $\qq$ tangent to $\ll_2$ and $\cc$ at their common points with $\ll_1$, with multiplicity $3$. Since $\cha\kk=3$, the cubic $\qq$ is tangent with multiplicity $3$ to every line passing through $\ll_2\cap \cc$, see cf.\ \cite[Lemma 5.5(b)]{PaPe_MT}. Moreover, every vector field on $\P^2$, which is tangent to those lines, is tangent to $\qq$, too. Let $\xi$ be such a vector field which vanishes along $\ll_1$: explicitly, choosing coordinates such that $\qq=\{y^3=x^2z\}$, $\cc=\{z=0\}$, $\ll_1=\{z=y\}$, $\ll_2=\{z=x\}$ we can take $\xi=(y-z)\frac{\d}{\d z}$. Write $\tau=\eta\circ \psi$, where $\eta$ is a composition of two blowups over the cusp of $\qq$, so $E\de \Exc\eta=\psi_{*}(G+L)$. A direct computation shows that $\xi$ lifts to a section of $\cT_{\psi(X)}(-\log \psi_{*}\check{D})\otimes \cO_{X}(E)$. Since the morphism $\psi$ is inner and is an isomorphism near $G+L$, $\xi$ further lifts to a global section of $\lts{X}{\check{D}}\otimes \cO_{X}(G+L)$, as needed.
\smallskip

	Assume now that $\cha\kk=5$, so $\cS=2\rA_4+[3]$. We proceed as in the previous case, this time using a pencil of lines coming from the $\P^1$-fibration constructed in the proof of Lemma \ref{lem:GK_hi-3}, see Figure \ref{fig:2A4-GK}.
	
	Let $(\tilde{X},\tilde{D})$ be the minimal log resolution of $(\bar{Y},\bar{T})$. We denote the proper transforms on $\tilde{X}$ of the curves on $X$ and those in Figure \ref{fig:2A4} by the same letters, so $\tilde{D}=D+V+W+H$, where $H=[1]$ meets $T=[5]$, $V=[3]$, $W=[2]$. The $(-1)$-curve $A_1$ meets $D$ in $T$, in $C$ and in a tip of a connected component $G_T$ of $D_Y$. The chain $T+A_1+G_T=[5,1,(2)_{4}]$ supports a fiber of a $\P^1$-fibration $f\colon \tilde{X}\to \P^1$. We have seen in the proof of Lemma \ref{lem:GK_hi-3} (or in  the proof of \cite[Corollary 1.4]{PaPe_MT}), that the horizontal part of $\tilde{D}$ consists of a $1$-section $H$ and a $5$-section $C$, and the degenerate fibers are supported on: $T+A_1+G_T=[5,1,(2)_{4}]$,  $V+A_V+G_V=[3,1,2,2]$ and $W+A_W+G_W=[2,1,2]$ for some $(-1)$-curves $A_V,A_W$. The Hurwitz formula shows that $f|_{C}\colon C\to \P^1$ is totally ramified. 
	Let $\tau\colon \tilde{X}\to \P^2$ be the contraction of $H+(A_1+G_T)+(A_V+G_V)+(A_W+G_W)$. It maps fibers to concurrent lines, and $C$ to a quintic tangent to each of those lines with multiplicity $5$. We can choose coordinates on $\P^2$ so that $\tau(T)=\{x=0\}$, $\tau(V)=\{x=z\}$, $\tau(W)=\{z=0\}$, $\tau(C)= \{y^5=xz^3(2z-x)\}$. Let $\xi$ be a vector field on $\P^2$ which is tangent to those lines and vanishes along $\tau(W)$, say $\xi=z\frac{\d}{\d y}$. It is tangent to $\tau(C)$, too, and a direct computation shows that it lifts to a global section of $\lts{\tilde{X}}{D'}\otimes \cO_{\tilde{X}}(H+A_1+G_T+A_V+G_V-A_W)\cong \lts{\tilde{X}}{D'}\otimes \cO_{\tilde{X}}(2K_{\tilde{X}}+D')$, where $D'=D+A_1$. Hence $h^{2}(\lts{\tilde{X}}{D'})\neq 0$ by Lemma \ref{lem:h1}\ref{item:h1-h2=h0}. We have $h^2=h^{2}(\lts{\tilde{X}}{D'})$ by Lemmas \ref{lem:h1}\ref{item:h1-1_curve}, and \ref{lem:blowup-hi}, so $h^2\neq 0$, as needed.
\end{proof}

\begin{proof}[Proof of Proposition \ref{prop:GK-small}]
	For $j\in \{1,2,3\}$ let $\cS_j=\cS_Y+\cT_j$ be a type as in  \ref{prop:GK-small} with $\#\cT_j=j$, i.e.\ $\cT_1=[s-5]$, $\cT_2=[2,s-4]$, $\cT_3=[s-3]+[3]+[2]$, cf.\ Lemma \ref{lem:cuspidal_resolution}. For $j\in \{1,2,3\}$ let $(X_j,D_j)$ be the minimal log resolution of $\bar{X}_j\in \Pdeb(\cS_j)$ such that the morphism $(X_3,D_3)\to (\bar{Y},\bar{T})$ factors through $(X_3,D_3+L_3)\to (X_2,D_2+L_2)\to (X_1,D_1+L_1)$, where $L_j$ is the elliptic tie. Let $h^{i}_{j}\de h^i(\lts{X_j}{D_j})$, let $T_j$ be the proper transform of $\bar{T}$ on $X_j$, and for $j\in \{2,3\}$ let $C_j$ for be the $(-2)$-curve in $D_j$ meeting $L_j$.
	
	By Lemma \ref{lem:GK_hi-3} we have $(h^0_3,h^1_3,h^2_3)=(0,d,d+1)$. Moreover, we have $h^0_1=0$. Indeed, if $\cS_Y\neq 2\rA_4$ then $h^0_1=0$ by Lemma  \ref{lem:GK-small-moduli}. In case $\cS_Y=2\rA_4$ we have $h^0_1\leq h^0(\lts{X_1}{(D_1-T_1)})$ by Lemma \ref{lem:h1}\ref{item:h1_exact}, and we see from Example \ref{ex:ht=2_meeting} that there is a birational morphism $X_1\to \P^1\times \P^1$ mapping $D_1-T_1$ to the sum of two vertical lines, one horizontal line and a diagonal, so $h^0_1=0$ by Lemmas \ref{lem:blowup-hi} and \ref{lem:h1}\ref{item:h1-grid}, as claimed.
	
	The contraction of $L_3$ is an inner blowup $(X_3,D_3-C_3)\to (X_2,D_2)$, so $h^{i}_2=h^{i}(\lts{X_3}{D_3-C_3})$. By Lemmas \ref{lem:GK_hi-3} and \ref{lem:h1}\ref{item:h1_exact} we get $h^{0}_2=0$ and $h^1_2=h^2_2\in \{d,d+1\}$. Moreover, if $h^1_2=d$ then the natural map between infinitesimal deformations of $(X_2,D_2)$ and $(X_2,D_2-C_2)$ is an isomorphism, so the almost faithful family representing $\Pdeb(\cS_2)$ obtained by blowing down one representing $\Pdeb(\cS_3)$ is in fact almost universal.

	In turn, the contraction of $L_2$ is an outer blowup $(X_2,D_2-C_2)\to (X_1,D_1)$, so since $h^0_1=0$, Lemma \ref{lem:blowup-hi} implies that $h^1_1=h^1(\lts{X_2}{D_2-C_2})-1$ and $h^2_1=h^2(\lts{X_2}{D_2-C_2})$. By Lemma \ref{lem:h1}\ref{item:h1_exact} we get $h^2_1=h^1_1$ and $h^j_1\leq h^j_2$ for $j=1,2$. On the other hand, by Lemma \ref{lem:GK-h1}\ref{item:GK-h1-moduli} the class $\Pdeb(\cS_1)$ is represented by an almost faithful family of dimension $d$, which in particular is non-trivial along any germ of a curve, so $h^1_1\geq d$; and as before we see that if $h^1_1=d$ then $\Pdeb(\cS_1)$ the above almost faithful family is almost universal.
	 
	This proves \ref{prop:GK-small}\ref{item:GK-small-h0} and shows that if $h^2_j=d$ for all $\bar{X}_j\in \Pdeb(\cS_j)$ then $\Pdeb(\cS_j)$ has moduli dimension $d$. By Lemma \ref{lem:GK-small-moduli} under the assumptions of \ref{prop:GK-small}\ref{item:GK-small-moduli} we have $h^2_j\leq d$, so by \ref{prop:GK-small}\ref{item:GK-small-h0} $h^1_j=d$ and thus  \ref{prop:GK-small}\ref{item:GK-small-moduli} holds. Under the  assumptions of Lemma \ref{lem:GK-small-fails} we have $h^2_j\geq 1=d+1$, so $h^1_j=h^2_j=1$ and \ref{prop:GK-small}\ref{item:GK-small-fails}  holds.
\end{proof}

\begin{proof}[Proof of Theorem \ref{thm:GK}]
	Part \ref{thm:GK}\ref{item:GK-unique} is proved in Lemma \ref{lem:leash-is-unique}\ref{item:leash-phi}. Part \ref{thm:GK}\ref{item:GK-Y} follows from Lemma \ref{lem:GK_exceptions}. The list of singularity types in \ref{thm:GK}\ref{item:GK-types} can be read off from Table \ref{table:canonical} (part $\cS(\bar{Y})$) and Lemma \ref{lem:cuspidal_resolution} (part $\cT$). Part \ref{thm:GK}\ref{item:GK_uniqueness-nodal} is proved in Lemma  \ref{lem:GK-h1}. Part \ref{thm:GK}\ref{item:GK_uniqueness-cuspidal} follows from Lemma \ref{lem:GK-h1} and Proposition \ref{prop:GK-small}. Part \ref{thm:GK}\ref{item:GK-ht} is proved in Proposition \ref{prop:GK-ht-exceptions}. Eventually, Lemma \ref{lem:GK_exceptions} gives a list of exceptions where $\#\cS(\bar{Y})\geq 6$ but $\height(\bar{X})\leq 2$. 
\end{proof}

\subsection{Primitive models}
	
As a last result of this section, we prove an analogue of Theorem \ref{thm:ht=1,2}\ref{part:swaps} announced in Remark \ref{rem:GK-cascades}. That is, we show that del Pezzo surfaces of rank one admitting descendants with elliptic boundary swap to particularly simple primitive surfaces, which are either listed in Theorem \ref{thm:ht=1,2}\ref{part:swaps} or in the following lemma.

\begin{lemma}[Primitivity criterion]\label{lem:GK_prim}
	Let $\bar{X}$ be a rational, non-canonical del Pezzo surface of rank one, having a descendant $(\bar{Y},\bar{T})$ with elliptic boundary. Let $\phi\colon (X,D)\to (\bar{Y},\bar{T})$ be the morphism from Definition \ref{def:GK}, 
	and let $\eta\colon (\tilde{Y},\tilde{D})\to (\bar{Y},\bar{T})$ be the minimal log resolution. Then $\bar{X}$ is primitive if the following hold.
	\begin{enumerate}
		\item\label{item:GK-prim_log-res} The birational map $\phi^{-1} \circ \eta\colon\tilde{Y}\map X$ is regular, i.e.\ $\#\cT=1$ if $\bar{T}$ is nodal and $\#\cT\leq 3$ if $\bar{T}$ is cuspidal.
		\item\label{item:GK-prim_Y-prim} The canonical surface $\bar{Y}$ is primitive.
		\item\label{item:GK-prim_Y-ht} Either $\height(\bar{Y})\geq 2$, or $\bar{X}$ is of type $3\rA_1+\rD_4+[3,2]$ and $\bar{T}$ is cuspidal.
	\end{enumerate}
	As a consequence, $\bar{X}$ is primitive if and only if its singularity type is one of the following:
	\begin{longlist}
		\item\label{item:GK-prim-cha-any} $2\rA_4+[3]$;
		\item\label{item:GK-prim-cha-gen} If $\cha\kk\not\in \{2,3\}$: $\rA_1+\rA_2+\rA_5+[3]$;
		\item\label{item:GK-prim-cha-5} If $\cha\kk=5$: $2\rA_4+[4,2]$, $2\rA_4+[5]+[3]+[2]$;
		\item\label{item:GK-prim-cha-3} If $\cha\kk=3$: 
		$3\rA_2+[3]+[3]+[2]$, 
		$4\rA_2+[3]$, $4\rA_2+[4,2]$, $4\rA_2+[5]+[3]+[2]$;
		\item\label{item:GK-prim-cha-2} If $\cha\kk=2$: 
		$3\rA_1+\rD_4+[3,2]$, 
	 	$7\rA_1+[3,2]$, $7\rA_1+[4]+[3]+[2]$, 
		$8\rA_1+[3]$, $8\rA_1+[4,2]$, $8\rA_1+[5]+[3]+[2]$.
	\end{longlist}
\end{lemma}
\begin{proof}
It is enough to prove that $\bar{X}$ satisfying conditions \ref{item:GK-prim_log-res}--\ref{item:GK-prim_Y-ht} is primitive; the subsequent list follows from the classification 
in Table \ref{table:canonical}. Suppose the contrary, and let $A\subseteq X$ be a $(-1)$-curve as in Definition \ref{def:vertical_swap}\ref{item:def-swap-elementary}. 

As 	in Notation \ref{not:GK} let $L$ be the elliptic tie, write $D=D_{T}+D_{Y}$ where $D_{T}=\phi^{*}\bar{T}\redd-L$ and $D_Y=\phi^{-1}(\Sing\bar{Y})$, let  $T=\phi^{-1}_{*}\bar{T}$ and let $\sigma\colon X\to Y$ be the morphism onto minimal resolution $Y$ of the canonical surface $\bar{Y}$. Note that the Picard rank of $\sigma$ equals $\#\cT$.

By condition \ref{item:GK-prim_log-res} we have $\#\cT=1$ if $\bar{T}$ is nodal and $\#\cT\leq 3$ if $\bar{T}$ is cuspidal. In particular, $-2\geq T^2\geq \bar{T}^2-6=3-\#\cS_Y$ by formula \eqref{eq:Noether}, so $\#\cS_Y\geq 5$. Pulling back the linear equivalence $K_{\bar{Y}}+\bar{T}=0$ we get $-K_{X}=D_{T}+\epsilon L$, where $\epsilon=1$ if $\#\cT\in \{1,2\}$ and $\epsilon=2$ if $\#\cT=3$. Adjunction formula gives 
\begin{equation}\label{eq:K-prim}
	1=A\cdot (-K_{X})=A\cdot D_{T}+\epsilon\,  A\cdot L.
\end{equation}

Consider the case $A\cdot L>0$. Then formula \eqref{eq:K-prim} gives $A\cdot L=1$, $A\cdot D_{T}=0$ and $\epsilon=1$, i.e.\  $\#\cT\in \{1,2\}$. Suppose $\#\cT=1$. Then $\sigma(A)$ is a $0$-curve. Since $A$ meets $D_Y$ once, the $\P^1$-fibration of $Y$ induced by $|\sigma(A)|$ has height one with respect to $\sigma_{*}D_Y$, hence $\height(\bar{Y})\leq 1$. Since $\#\cT=1$, this is a contradiction with \ref{item:GK-prim_Y-ht}. 

Thus $\#\cT=2$. Since $\bar{X}$ is not canonical, we have $T^2\leq -3$, so $\#\cS_Y\geq 7$ by formula \eqref{eq:Noether}. Let $\psi\colon X\to Z$ be the contraction of $L$, and let $D_{Z}=\psi_{*}(D_Y+T)$.Then $(Z,D_Z)$ is a minimal log resolution of a normal surface $\bar{Z}$ of rank one, again having $(\bar{Y},\bar{T})$ as a descendant with elliptic boundary. By Lemma \ref{lem:GK_intro}\ref{item:GK-intro-cf}, $\bar{Z}$ is del Pezzo. The curve $\psi(A)$ is a $0$-curve meeting $D_{Z}$ once in $\psi_{*}D_Y$ and once in $\psi_{*}T$, so $\height(\bar{Z})\leq 2$. Thus $\bar{Z}$ is one of the exceptional surfaces in Lemma \ref{lem:GK_exceptions}. Since $\bar{Y}$ is primitive by assumption \ref{item:GK-prim_Y-prim}, its singularity type is listed in Remark \ref{rem:primitive_ht=2}, see Table \ref{table:canonical}. We conclude that $\bar{Z}$ is of type $2\rA_1+2\rA_3$ or $4\rA_1+\rD_4$. Thus $\height(\bar{Z})=2$, see Table \ref{table:canonical}. Moreover, we have constructed a $\P^1$-fibration of $\bar{Z}$ such that $(D_Z)\hor$ consists of two $1$-sections, so  $\width(\bar{Z})=2$. This is a contradiction with Lemma \ref{lem:ht=2,untwisted}; cf.\ Theorem \ref{thm:ht=1,2}\ref{item:ht=2_width=1}.

Consider the case $A\cdot L=0$, so $A\cdot D_{T}=1$  by formula \eqref{eq:K-prim}. Then $\phi(A)$ is a smooth rational curve on $\phi(Y)$. Since $\#\cS_Y\geq 5>1$, \cite[Lemma 2.11(a)]{PaPe_MT} implies that $\phi(A)\not\subseteq \bar{Y}\reg$. Thus $A$ meets $D_Y$. It follows that the component of $D_T$ meeting $A$ is not a $(-2)$-curve. If $A$ meets $T$ then the image of $A$ on $Y$ is a $(-1)$-curve meeting $\sigma_{*}D_Y$ only once, contrary to \ref{item:GK-prim_Y-prim}. Therefore, $\#\cT=3$ and $A$ meets the $(-3)$-curve in $D_{T}$, so $A_Y$ is a $0$-curve. It follows that $\height(\bar{Y})=1$, and since $\#\cT=3$ we get a contradiction with condition \ref{item:GK-prim_Y-ht}.
\end{proof}

\begin{proposition}[Primitive models]\label{prop:GK_swaps}
	Let $\bar{X}$ be a del Pezzo surface of rank one having a descendant with elliptic boundary. Then the following hold.
	\begin{enumerate}
		\item \label{item:GK_swaps-models} The surface $\bar{X}$ swaps to a primitive del Pezzo surface $\bar{Z}$ of rank one such that one of the following holds. 
		\begin{enumerate}
			\item \label{item:GK_swaps-models-deb} $\bar{Z}$ has a descendant with elliptic boundary, hence is one of the surfaces in Lemma \ref{lem:GK_prim} above,
			\item \label{item:GK_swaps-models-canonical} $\bar{Z}$ is canonical,
			\item \label{item:GK_swaps-models-KM} $\cha\kk=2$ and $\bar{Z}$ is a surface of type $8\rA_1+[4]$ from Example \ref{ex:ht=2_twisted_cha=2}, case $\nu=4$.
		\end{enumerate}
		\item \label{item:GK_swaps-lc} Assume $\bar{X}$ is log canonical. If $\bar{X}$ is non-primitive then it admits an elementary swap to a log canonical del Pezzo surface of rank one which either has a descendant with elliptic boundary, or is as in \ref{item:GK_swaps-models-canonical}--\ref{item:GK_swaps-models-KM}. 
		\item \label{item:GK_swaps-L} Assume that $\bar{X}$ is not canonical, and does not satisfy condition \ref{lem:GK_prim}\ref{item:GK-prim_log-res}. Then the contraction of the $(-1)$-curve $L\subseteq \Exc\phi$ is an elementary swap. In particular, if $\bar{X}$ is log canonical then it admits an elementary vertical swap to a log canonical del Pezzo surface of rank one with the same descendant as $\bar{X}$.
	\end{enumerate}
\end{proposition}
\begin{proof}
	Clearly, any canonical del Pezzo surface of rank one swaps to a primitive canonical del Pezzo surface of rank one, so we can assume that $\bar{X}$ is not canonical. Assume that $\bar{X}$ does not satisfy the  condition \ref{lem:GK_prim}\ref{item:GK-prim_log-res}. Then the elliptic tie $L$ satisfies $L\cdot D\leq 2$ and $L\cdot C=1$ for some $(-2)$-curve $C\subseteq \Exc\phi$. If $L$ meets a $(-2)$-curve in $D-C$ then $D_{T}$ consists of $(-2)$-curves, so $\bar{X}$ is canonical, contrary to our assumption. Thus we can assume that $L$ does not meet any $(-2)$-curve in $D-C$. The contraction of $L$ is an elementary swap, as claimed in \ref{item:GK_swaps-L}.
	
	Contracting $L$ and repeating this process we eventually get a minimal log resolution of a del Pezzo surface of rank one, with the same descendant as $\bar{X}$, satisfying  \ref{lem:GK_prim}\ref{item:GK-prim_log-res}. Note that if $\bar{X}$ is log canonical then by Lemma \ref{lem:swap_lc}\ref{item:swap_lc_dP} all the intermediate log surfaces are minimal log resolutions of log canonical  del Pezzo surfaces of rank one, with the same descendant, as claimed in \ref{item:GK_swaps-lc} and \ref{item:GK_swaps-L}. In general, we only claim that so is the last one: this is clear as the eventual $D_T$ is a chain $[s-5]$, $[s-4,2]$ or a disjoint sum $[s-3]+[3]+[2]$.
	
	Therefore, we can assume that $\bar{X}$ satisfies \ref{lem:GK_prim}\ref{item:GK-prim_log-res}. By Lemma \ref{lem:GK-swap} we can assume that either $\bar{Y}$ is primitive, so \ref{lem:GK_prim}\ref{item:GK-prim_Y-prim} holds, or $T=[2]$, which since $\bar{X}$ is not canonical implies that $\bar{T}$ is cuspidal and $D_{T}=[2]+[2]+[3]$. 
	
	Assume that $\bar{X}$ it does not satisfy \ref{lem:GK_prim}\ref{item:GK-prim_Y-ht}. Then $\height(\bar{Y})=1$. Let $A_{Y}$ be the fiber of a witnessing $\P^1$-fibration of $Y$ passing through the singular point of $T_Y$. By adjunction, $A_Y\cdot T_Y=A_Y\cdot (-K_{Y})=2$. It follows that $A\de \sigma^{-1}_{*}A_Y$ satisfies $A\cdot (D+L)=2$, meets a $(-2)$-curve in $D_Y$ and the first exceptional curve of $\sigma$, call it $B$. Condition \ref{lem:GK_prim}\ref{item:GK-prim_log-res} implies that $-2\geq T^2\geq \bar{T}^2-6=3-\#\cS_Y$ by formula \eqref{eq:Noether}, so $\#\cS_Y\geq 5$. The equality holds if and only if $D_{T}=[2]+[2]+[3]$. In this case $B=[3]$, and contracting $A$ we get a swap to a canonical surface, as needed. Assume $\#\cS_Y\geq 6$, so $D_{T}\neq [2]+[2]+[3]$, and therefore by assumption \ref{lem:GK_prim}\ref{item:GK-prim_Y-prim} holds, i.e.\ $\bar{Y}$ is primitive. Since $\height(\bar{Y})=1$ and $\#\cS_Y\geq 6$, Remark \ref{rem:canonical_ht=1} shows that $\bar{Y}$ is of type $3\rA_1+\rD_4$. Then $D_{T}=[2]$, $[2,3]$ or $[2]+[3]+[4]$. The first two cases are excluded since $\bar{X}$ is not canonical and does not satisfy  \ref{lem:GK_prim}\ref{item:GK-prim_Y-ht}. Consider the latter case. Then $B=[3]$ and $T=[4]$. Taking the witnessing $\P^1$-fibration as in Figure \ref{fig:3A1+D4}, we see that $A$ meets the branching $(-2)$-curve in $D_Y$. Contracting $A$ we see that $\bar{X}$ swaps to the surface of type $8\rA_1+[4]$ from Example \ref{ex:ht=2_twisted_cha=2}, as needed. 
	
	As a consequence, we can assume that \ref{lem:GK_prim}\ref{item:GK-prim_Y-ht} holds. If $D_{T}=[2]+[2]+[3]$ then by formula \eqref{eq:Noether} we have $\#\cS_Y=9-\bar{T}^2=3-T^2=5$, so looking at Table \ref{table:canonical} we see that $\height(\bar{Y})=1$, contrary to \ref{lem:GK_prim}\ref{item:GK-prim_Y-ht}. Thus $\bar{X}$ satisfies all three conditions \ref{lem:GK_prim}\ref{item:GK-prim_log-res}--\ref{item:GK-prim_Y-ht}, which ends the proof.
\end{proof}

Eventually, we note the following partial results towards the statement discussed in Remark \ref{rem:Pht=P}.

\begin{remark}\label{rem:Pht=P-here}
	Let $\bar{X}_1,\bar{X}_2$ be non-isomorphic del Pezzo surfaces of rank one and the same log canonical singularity type $\cS$. 	Assume that for each $i=1,2$, the surface $\bar{X}_i$ has a descendant with elliptic boundary or satisfies  $\height(\bar{X}_i)\leq 2$. Then the following hold.
	\begin{enumerate}
		\item\label{item:ht-agree} We have $\height(\bar{X}_1)=\height(\bar{X}_2)$.
		\item\label{item:descendant-agree} Assume that $\bar{X}_1$ has a descendant with elliptic boundary. Then the same holds for $\bar{X}_2$. Moreover, for $i=1,2$ the descendant of $\bar{X}_i$ with elliptic boundary, call it $(\bar{Y}_i,\bar{T}_i)$, is unique up to an isomorphism, and the surfaces $\bar{Y}_1$, $\bar{Y}_2$ have the same singularity type.
		\item\label{item:stupid-example} 
		Nonetheless, there exist $\bar{X}_1,\bar{X}_2$ as in \ref{item:descendant-agree} such that $\bar{Y}_1\not\cong \bar{Y}_2$ and $\bar{T}_1\not\cong \bar{T}_2$.
	\end{enumerate}
\end{remark} 
\begin{proof}
	Assume that $\height(\bar{X}_1)\geq 3$, so $\bar{X}_1$ has a descendant $(\bar{Y}_1,\bar{T}_1)$ with elliptic boundary. Hence $\cS=\cS_Y+\cT$, where $\cS_{Y}$ is the singularity type of $\bar{Y}_1$. By Theorem \ref{thm:GK}\ref{item:GK-Y} we have $\#\cS_{Y}\geq 6$, and $\cS$ is not one of the types listed in Lemma \ref{lem:GK-ht-exceptions-ht=3}. We check directly that such $\cS$ is not listed in Tables \ref{table:ht=1}--\ref{table:ht=2_char=2}, so $\height(\bar{X}_2)\geq 3$. Hence, by assumption, $\bar{X}_2$ has a descendant $(\bar{Y}_2,\bar{T}_2)$ with elliptic boundary. By Theorem \ref{thm:GK}\ref{item:GK-unique}, the descendants  $(\bar{Y}_1,\bar{T}_1)$ and $(\bar{Y}_2,\bar{T}_2)$ are unique, and have the same singularity type, as claimed in \ref{item:descendant-agree}. In particular, we see from Table \ref{table:canonical} that $\height(\bar{Y}_1)=\height(\bar{Y}_2)$, so $\height(\bar{X}_1)=\height(\bar{X}_2)$ by Theorem \ref{thm:GK}\ref{item:GK-ht}, as claimed in \ref{item:ht-agree}.
	
	We may now assume that $\height(\bar{X}_i)\leq 2$ for $i=1,2$. Since $\#\Pht(\cS)\geq 2$, Proposition \ref{prop:moduli} shows that the singularity type $\cS$ is listed in Table \ref{table:exceptions}. Theorem \ref{thm:ht=1,2} implies that all surfaces in $\Pht(\cS)$ are of the same height (listed in Table \ref{table:exceptions}, too), so  $\height(\bar{X}_1)=\height(\bar{X}_2)$. This proves \ref{item:ht-agree}.
	
	For the proof of \ref{item:descendant-agree}, assume that $\bar{X}_1$ has a descendant with elliptic boundary. To describe the singularity type $\cS$,  we need a version of Lemma \ref{lem:cuspidal_resolution} without the assumption $s\geq 6$. Here the same argument yields the same list, with appropriate conditions ensuring that all chains and forks which appear on the list are admissible or log canonical (so, for instance, in \ref{lem:cuspidal_resolution}\ref{item:C_[2]}--\ref{item:C_[rest]} we allow $s=4,5$, with weaker restrictions on $V$). 
	
	Comparing this list of singularity types $\cS$ with Table \ref{table:exceptions} we conclude that the decomposition $\cS=\cS_{Y}+\cT$ as in Notation \ref{not:GK} is unique, $\cS_{Y}=\rD_5,\rE_6,\rE_7,\rE_8$, and $\bar{X}_i$ 
	is as in one of the following:
	\begin{enumerate-alt}
		\item \label{item:E+1} 
		Lemma 
		\ref{lem:ht=1_types}\ref{item:T2=[2]_b>2} with $b=2$, $r=s-4$, $T=[2,2]$, so $\cT=[s-4]\adec{1,1}$ 
		\item \label{item:E+2} 
		Lemma \ref{lem:ht=2,untwisted}\ref{item:[T_1,n]=[2,2]_r=2} with $m=s-4$, so $\cT=[s-4\adec{1},2\adec{1}]$
		\item  \label{item:t=1}
		Lemma 
		\ref{lem:ht=2,untwisted}\ref{item:c=3} with $n=2$, $k=s-3$, $T_2=[2,2]$, $T_3=[(2)_{s+2}]$, so $\cT=\langle m\adec{1},[s-3],[3],[2]\rangle +[(2)_{m-2}]\adec{1}$.
	\end{enumerate-alt}
To prove \ref{item:descendant-agree}, it is enough to find an elliptic tie, that is, a $(-1)$-curve $L$ on a minimal log resolution $(X,D)$ of $\bar{X}_i$, meeting $D$ in components decorated above by $\adec{1}$, cf.\ Lemma \ref{lem:cuspidal_resolution}. We note that such $L$, if exists, is unique. Indeed, if $L'$ is another elliptic tie then the weighted graphs of $D+L$ and $D+L'$ are the same, and since $K_{X}$ and the components of $D$ generate $\NS_{\Q}(X)$, we get $L\equiv L'$, hence $L=L'$, because $L^2<0$. It follows that the descendant $(\bar{Y}_i,\bar{T}_i)$ of $\bar{X}_i$ is unique, too. We now find $L$ in each of the above cases. We denote by $A_{1},A_2,\dots$ the $(-1)$-curves on $X$ corresponding to $\dec{1},\dec{2},\dots$ in the lemmas cited above.

\ref{item:E+1} Let $\tau\colon X\to \P^2$ be the contraction of the subchain $[1,(2)_{r+3}]$ of $D+A_1$ and $[1,2]$ of $D+A_2$. Then $\tau_{*}D$ is a sum of a line $\ll$ and a conic $\cc$ tangent to $\ll$ at the point $\tau(A_2)$. The required elliptic tie is the proper transform of a line $\ll'\neq \ll$ passing through $\tau(A_1)$ and its infinitely near point on the image of $A_2$. 

\ref{item:E+2} We have a vertical swap $\tau\colon (X_2,D_2)\sqto (Y,D_Y)$, where $(Y,D_Y)$ is as in Example \ref{ex:ht=2}\ref{item:3A2_construction}. Composing $\tau$ with the morphism $\phi$ from \ref{ex:ht=2}\ref{item:3A2_construction}, we get a morphism $\psi\colon X\to \P^2$ mapping $D$ to a triangle $\ll_1+\ll_2+\cc$. Put $L_{i}=\psi^{-1}_{*}\ll_i$, $C=\psi^{-1}_{*}\cc$. Then  $L_1$ is the branching component of $D$, and $C+L_2=[s-4,2]$ is a connected component of $D$. Let $p=\psi(A_1)\in \ll_{1}\setminus (\ll_2\cup \cc)$, and let $p'$ be the unique base point of $\psi^{-1}$ which is infinitely near to $p$. Let $\ll$ be the line passing through $p$ and $p'$. Then $\psi^{-1}_{*}\ll$ 
is the required elliptic tie. 

We note that for $s=6$, the pair of surfaces $\bar{X}_1,\bar{X}_2$ of type $\rA_2+\rE_6$ was constructed in Example \ref{ex:remaining_canonical}\ref{item:3A2_tower}. 
	
\ref{item:t=1} Here the required elliptic tie is $A_2$. In fact, in this case we get an almost universal family of del Pezzo surfaces of rank one and type $\cS$, all with the same descendant of type $\rE_{s}$ and cuspidal boundary.
\smallskip 

Eventually, we show how \ref{item:E+1} and \ref{item:E+2} provide examples proving \ref{item:stupid-example}. The elliptic boundary in \ref{item:E+1} (respectively, in \ref{item:E+2}) is cuspidal if $\ll'$ is tangent to $\cc$ (respectively, if $\ll$ passes through $\ll_1\cap \cc$). This gives examples with $\bar{T}_1\not\cong \bar{T}_2$. In case $s=8$, we get $\bar{Y}_1\not\cong \bar{Y}_2$, too. Indeed, suppose $\bar{Y}$ is a canonical surface of type $\rE_8$ whose smooth locus contains curves $\bar{T}_i$ with $p_{a}(\bar{T}_i)=1$, $i=1,2$, such that $\bar{T}_1$ is cuspidal and $\bar{T}_2$ is nodal. Let $\eta\colon \P^2\map \bar{Y}$ be the minimal resolution of $\bar{Y}$, followed by the contraction of the $(-1)$-curve from Figure \ref{fig:E8} and a chain $[(2)_{7}]$ meeting it. Then $(\eta^{*}\bar{T}_i)\redd=\ll+\qq_i$, where $\qq_i\cong\bar{T}_i$ is a cubic, $\ll$ is a line tangent to $\qq_i$ with multiplicity $3$ at an inflection point $p\in \qq_{i}\reg$, and $(\qq_1\cdot \qq_2)_{p}\geq 8$. If the equality holds then $\qq_1$ meets $\qq_2$ normally at some $q\neq p$, so $q=0$ in the group law on $\qq\reg$ with $p=0$, which is false. Hence $(\qq_1\cdot \ll)_{p}=9$, so $\qq_2$ is a member of a pencil of cubics generated by $\qq_1$ and $\ll$. In some coordinates $[x:y:z]$ on $\P^2$ we have $\qq_{2}=\{\lambda y^3=x^3-yz^2\}$ for some $\lambda\in \kk^{*}$, and a direct computation shows that $\qq_2$ is smooth if $\cha\kk\neq 2,3$ and cuspidal otherwise, a contradiction. 
\end{proof}

\clearpage

\section{Comparison with other classification results}\label{sec:comparison}

We now compare Theorem \ref{thm:ht=1,2} with some other results partially classifying (closed) del Pezzo surfaces of rank one. For convenience, throughout this section we take $\kk=\C$.

\subsection{Keel--McKernan's approach and Lacini's list}\label{sec:Lacini}

The main result of  \cite{Keel-McKernan_rational_curves} asserts that the smooth part of a del Pezzo surface of rank one is dominated by images of $\A^1$. The proof becomes easier if a given del Pezzo surface $\bar{X}$ admits a \emph{tiger}, see Definition 1.13 loc.\ cit. Assume that $\bar{X}$ is log terminal, otherwise it has a tiger by definition. Then a \emph{tiger} is an exceptional divisor $T$ such that $-(K_{X_T}+T)$ has nonnegative Kodaira dimension; where  $\pi_{T}\colon X_{T}\to \bar{X}$ is the extraction of $T$, see \cite[Lemma on p.\ 9]{Keel-McKernan_rational_curves}. Note that log terminality guarantees that such an extraction exist, see \cite[1.38]{Kollar_singularities_of_MMP}.

The major part of \cite{Keel-McKernan_rational_curves} concerns bounding the class of surfaces without tigers. Following this approach, in \cite[\S 6.1]{Lacini}, Lacini arranged all non-canonical del Pezzo surfaces in 24 series, assuming only  $\cha\kk\neq 2,3$. Most of these series consist of one or a few well described surfaces \emph{without} tigers. 

On the other hand, Remark \ref{rem:tiger} below shows that del Pezzo surfaces of height at most $2$, or admitting descendants with elliptic boundary, do have tigers. Hence from the point of view of \cite{Keel-McKernan_rational_curves,Lacini}, one does not need a thorough description of their geometry. In fact, various constructions of such surfaces, similar to reversing vertical swaps, are outlined in series LDP 20 -- 24 of \cite{Lacini}. The question of uniqueness of this process and the problem of classifying resulting singularity types are not addressed. In fact, addressing them requires an analysis along the lines of Lemmas \ref{lem:ht=1_uniqueness},  \ref{lem:ht=2,untwisted} and \ref{lem:ht=2_twisted-separable}.

To make this comparison more precise, we note that the definitions in loc.\ cit.\ imply that all surfaces in series LDP 20, 22, 24 are of height at most $2$. Series LDP 21 consists of one surface of type $2\rA_4+[2]+[3]+[5]$, which exists only if $\cha\kk=5$ and admits a descendant with elliptic boundary. Series LDP 23 is constructed from log surfaces $(S,C)$ listed in \sec 5 loc.\ cit, and unless the latter is as in Proposition 5.1(3)--(5) loc.\ cit, the resulting del Pezzo surface either has a descendant with elliptic boundary (this holds in cases 5.1(2), 5.10, 5.11(1),(4),(8) and 5.9(5) with $E\subseteq S\reg$) or is of height at most two (this holds in all the remaining cases). 

In our forthcoming articles we will see that all surfaces in LDP 1--19, and the ones in LDP 23 obtained from 5.1(3)--(5), are of height $3$ or $4$; and in most cases they do not admit descendants with elliptic boundary. 

\begin{remark}[Surfaces in this article have tigers]\label{rem:tiger}
	Let $\bar{X}$ be a del Pezzo surface of rank one. If $\height(\bar{X})\leq 2$, or if $\bar{X}$ admits a descendant with elliptic boundary, then $\bar{X}$ has a tiger.
\end{remark}
\begin{proof}
We can assume that $\bar{X}$ is log terminal. Let $\pi\colon X\to \bar{X}$ be a minimal resolution, and let  $D=\Exc\pi$.

Assume first that $\height(\bar{X})\leq 2$, and let $F$ be a fiber of a witnessing $\P^1$-fibration, so $F\cdot D\leq 2$. We have $\pi^{*}K_{\bar{X}}=K_{X}+C$, where the $\Q$-divisor $C\de \sum_{E} \cf(E)E$ is effective since $\pi$ is minimal, see Lemma \ref{lem:min_res}\ref{item:ld<=1}. Since $\bar{X}$ is del Pezzo of rank one, we have $K_{\bar{X}}=-\lambda\cdot \pi_{*}F$ for some $\lambda>0$, so $-K_{X}=\lambda \pi^{*}\pi_{*}F+C=V+H$ for some effective $\Q$-divisors $V$, $H$ such that $V$ is vertical and $\Supp H\subseteq D\hor$. We have $F\cdot H\redd\leq F\cdot D\leq 2=-F\cdot K_{X}=F\cdot H$, so some component $H_0$ of $H$ has coefficient at least one in $H$. Therefore $-(K_{X}+H_0)=V+(H-H_0)$ is effective, so $H_0$ is a tiger.

Assume now that $\bar{X}$ has a descendant $(\bar{Y},\bar{T})$ with elliptic boundary. Let $\phi\colon (X,D+L)\to (\bar{Y},\bar{T})$ be the morphism from Definition \ref{def:GK}, and let $T_{X}=\phi^{-1}_{*}\bar{T}\subseteq D$. We have factorizations $\pi=\pi_{T}\circ\rho_{T}$, where $\pi_{T}$ is the extraction of $T\de \rho_{T}(T_X)$, and $\phi=\phi'\circ \rho_{T}$, where $\phi'$ is a morphism of relative Picard rank one, i.e.\ contracts only one curve, call it $T^{\vee}$. We have $\phi'_{*}(K_{X_T}+T)=K_{\bar{Y}}+\bar{T}=0$, so $K_{X_{T}}+T=\lambda T^{\vee}$ for some $\lambda\in \Q$, and $\lambda T^{\vee}\cdot T=(K_{X_T}+T)\cdot T=\pi_{T}^{*}(K_{\bar{X}})\cdot T+(1-\cf(T))T^2<0$, so $\lambda<0$. Therefore, the $\Q$-divisor $-(K_{X_{T}}+T)$ is effective, so $T$ is a tiger. We remark that $T^{\vee}=\rho_{T}(L)$.
\end{proof}

\subsection{Miyanishi and Zhang's approach: curves of minimal anti-canonical degree}\label{sec:GurZha}

Another interesting way to study log terminal del Pezzo surfaces was proposed by Zhang in \cite{Zhang}, following the ideas of Miyanishi, Sugie and Tsunoda \cite{Miy_Su}, \cite{Miy_Tsu-opendP}. The key object in this approach is a curve $C\subseteq X$ of minimal anti-canonical degree $-K_{\bar{X}}\cdot \bar{C}$, where $\bar{C}$ is the image of $C$ on $\bar{X}$. Now, the arguments split in two cases, depending on whether the linear system  $|K_{X}+D+C|$ is empty or not, the first case being substantially more difficult. This strategy led Gurjar and Zhang \cite{GurZha_1} to the proof of finiteness of $\pi_{1}(\bar{X}\reg)$, but since there seem to be no natural constraints on the geometry of $(X,D)$ in case $|K_{X}+D+C|= \emptyset$, the classification following this approach would be difficult. Nonetheless, it led to a classification of particular classes of del Pezzo surfaces, which we now briefly review.

Under suitable assumptions, one can use the curve $C$ to construct an explicit $\P^1$-fibration of $X$, often of low height. This was done by Kojima in \cite[Theorem 1.1(2)]{Kojima_index-2} for del Pezzo surfaces of rank one and index two; and in \cite[Theorems 2.1(3), 3.1(2), 4.1(3)]{Kojima_Sing-1} for those with only one quotient singular point. 
Using this structure, loc.\ cit.\ provides a full classification in these cases, see Tables \ref{table:Sing=1} and \ref{table:index=2}. We note that Lemma \ref{lem:ht=1_types}\ref{item:TH=[2,2]_tip},\ref{item:F'_fork_T=[2,2]},\ref{item:ht=1_bench_two_a,b>2} gives examples of log canonical del Pezzo surfaces with exactly one singular point which is \emph{not} of quotient type, hence does not appear in loc.\ cit. We also remark that del Pezzo surfaces of rank one and exactly one singular point are of height at most $2$ by \cite[Corollary 1.3(a)]{PaPe_MT}, cf.\ \cite{Russell-ruled}.

Following a similar approach, Zhang \cite{Zhang_dP3} classified log terminal del Pezzo surfaces of rank one with exactly one non-canonical singularity which has multiplicity $3$ (called \emph{dP3-surfaces}). Here the $\P^1$-fibration constructed using $C$ has height at most four. A comparison with our classification, see Table \ref{table:dP3}, shows that sometimes it is not a witnessing one, and in fact most dP3-surfaces have $\height(\bar{X})\leq 2$. The remaining  ones have descendants with elliptic boundary.

A particular instance of the difficult case $|K+D+C|=\emptyset$, namely the one when $C$ meets at least two $(-2)$-curves (called case \emph{IIa}), was classified by Kojima and Takahashi in \cite[Theorem 4.1]{Kojima_Takahashi}, see Table \ref{table:IIa}. Here the linear system of $2C$ plus two $(-2)$-curves meeting it supports a $\P^1$-fibration of height two. Moreover, in Theorem 1.1 loc.\ cit.\ it is shown that a log canonical del Pezzo surface of rank one has at most $5$ singularities; surfaces realizing this maximum are classified in \cite{Kojima_supplement}: they are isomorphic to the ones from Lemma \ref{lem:ht=1_types}\ref{item:ht=1_bench}.

Recently, Belousov \cite{Belousov_4-sings} used this approach to classify log terminal del Pezzo surfaces of rank one and four singular points (recall from Remark \ref{rem:number-of-sing} that for $\kk=\C$ this is the maximal possible  number of singularities). He arranged the non-canonical ones in seven series (1)--(7), according to their singularity types. The ones in series (1)--(4) are of height at most $2$, and are isomorphic to the ones from Lemmas \ref{lem:ht=2_twisted-separable}\ref{item:c=1_T=[2]},  \ref{lem:ht=2_twisted-separable}\ref{item:c=1_T1=[2]}, \ref{lem:ht=1_types}\ref{item:beta=3_columnar},\ref{item:beta=3_other}, and \ref{lem:ht=2_twisted-separable}\ref{item:twisted_off_nu=3}, respectively, cf.\ Tables \ref{table:ht=1}, \ref{table:ht=2_char-neq-2}. The ones in series (5)--(7) admit descendants with nodal elliptic boundaries, whose underlying surfaces are of type $3\rA_2$, $\rA_1+2\rA_3$ and $\rA_1+\rA_2+\rA_5$, respectively.

\begin{remark}
	Del Pezzo surfaces of index two were classified using K3 surfaces by Alexeev and Nikulin \cite{Alexxev-Nikulin_delPezzo-index-2}, and, using yet another method, by Nakayama \cite{Nakayama_delPezzo-index-2}. The latter method works over any algebraically closed field $\kk$, and led to a classification in case of index three, too \cite{Fujita-Yasutake_delPezzo-index-3}. In this approach, to each $\bar{X}$ one associates a unique \emph{fundamental triplet}  \cite[\S 4]{Nakayama_delPezzo-index-2}, which can be explicitly described. While this correspondence endows each $\bar{X}$ with an interesting geometric structure, it does not immediately yield e.g. the singularity type or the height of $\bar{X}$. Therefore, a direct comparison of our classification with exhaustive lists of fundamental triplets from loc.\ cit.\ is beyond the scope of this article. 
\end{remark}

\subsection{Hwang's cascades of toric del Pezzo surfaces}\label{sec:Hwang}

In \cite{Hwang_cascades}, Hwang arranged all \emph{toric} del Pezzo surfaces in cascades. A \emph{cascade} is a sequence of toric surfaces such that minimal log resolutions of two subsequent elements are related by a toric blowup. The first element of such sequence, called a \emph{basic} surface, is explicitly described.

Note that since $\bar{X}$ is toric, either $\bar{X}=\P^2$ or $(X,D)$ admits a $\P^1$-fibration of height at most two, whose fibers are closures of the $\G_{m}$-orbits. It is known \cite{Orlik_Wagreich-Acta} that $D\hor$ consists of two disjoint sections, and each fiber is columnar, i.e.\ it is a chain meeting $D\hor$ in tips. Since $\bar{X}$ is toric, $(X,D)$ admits two such $\P^{1}$-fibrations. In \cite[Theorem 1.3]{Hwang_cascades}, this structure is used to prove that there are at most two degenerate fibers, and as a consequence, $(X,D)$ swaps vertically to a cone over a rational normal curve, see Example \ref{ex:ht=1}, or to the surface of type $3\rA_{2}$ from Example \ref{ex:ht=2}\ref{item:3A2_construction}. These vertical swaps correspond to toric blowups in loc.\ cit. It is further inferred in Theorem 1.4 loc.\ cit, that $\bar{X}$ is as in Lemma \ref{lem:ht=1_types}\ref{item:chains_columnar_T1=0,T2=0},\ref{item:chains_columnar_T2=0},\ref{item:chains_columnar} or \ref{lem:ht=2,untwisted}\ref{item:tau=id_chains}. Moreover, by Theorem 1.7 loc.\ cit, $\bar{X}$ is K\"ahler-Einstein if and only if it is as in \ref{lem:ht=2,untwisted}\ref{item:tau=id_chains}.
%

\begin{small}
	\begin{table}[ht]
		\begin{tabular}{rl|rl|l}
			\multicolumn{2}{c|}{\cite{Kojima_Sing-1}} & \multicolumn{2}{c|}{this article} & singularity \\ \hline \hline
			A & (1) & \ref{lem:ht=1_types}\ref{item:TH=[2,2]_tip} & $b=3$  & $\langle 3;[2,2],[2,2],[2]\rangle$ \\
			& (2) & \ref{lem:ht=1_types}\ref{item:TH=[2,2]_b>2} & $b=2$, $T=[2,2]$ & $\langle 2;[2,2],[2,2],[2]\rangle$ \\
			& (3) & \ref{lem:ht=1_types}\ref{item:T2*=[2,2]_b>2} & $b=2$, $T=[2,2]$ & $\langle 2;[2,2,2],[3],[2]\rangle$ \\
			& (4) & \ref{lem:ht=1_types}\ref{item:TH=[2,2]_b>2} & $b=2$, $T=[2,2,2]$ & $\langle 2;[2,2,2],[2,2],[2] \rangle$ \\
			& (5) & \ref{lem:ht=1_types}\ref{item:TH=[2,2]_tip} & $b=4$ & $\langle 4;[2,2,2],[2,2],[2]\rangle$ \\
			& (6) & \ref{lem:ht=1_types}\ref{item:T2*=[2,2]_b>2} & $b=2$, $T=[3]$ & $\langle 2;[2,3],[2,2],[2]\rangle$ \\
			& (7) & \ref{lem:ht=1_types}\ref{item:TH=[3,2]_tip}  && $\langle 3;[3,2],[2,2],[2]\rangle$ \\
			& (8) & \ref{lem:ht=1_types}\ref{item:TH=[3,2]_T=[3]_b>2} & $b=2$ & $\langle 2; [3,2],[2,2],[2]\rangle$ \\
			& (9) & \ref{lem:ht=1_types}\ref{item:TH=[2,2]_b>2} & $b=2$, $T=[2,2,2,2]$ & $\langle 2;[2,2,2,2],[2,2],[2]\rangle$ \\
			& (10) & \ref{lem:ht=1_types}\ref{item:TH=[2,2]_tip} & $b=5$  & $\langle 3;[2,2,2,2],[2,2],[2]\rangle$ \\
			B && \ref{lem:ht=1_types}\ref{item:T1=0,T=T2=[2]_b>2} & $b=2$ & $\langle 2; [m,2],[2],[2]\rangle$ \\
			&& \ref{lem:ht=1_types}\ref{item:TH_long_other_b>2} & $b=r=2$ & $\langle 2; [T,2,T^{*},m,2],[2],[2] \rangle$ \\
			C && \ref{lem:ht=1_types}\ref{item:chains_columnar_T1=0,T2=0}  && $[m]$ \\
			& (a) & \ref{lem:ht=1_types}\ref{item:chains_both_T1=0} & & $[m,T^{*},2,T]$ \\
			& (b) & \ref{lem:ht=1_types}\ref{item:chains_not-columnar} & $r_1=r_2=2$ &  $[T_{1},2,T_{1}^{*},m,T_{2}^{*},2,T_{2}]$ \\
			& (c) & \ref{lem:ht=2,untwisted}\ref{item:V-chains_c=1_rivet} & $r_1=r_2=2$ &  $[T_{2},T^{*},2,T,T_1,2,T_1^{*},T_2^{*}]$ \\
			\multicolumn{2}{c|}{--} & \ref{lem:ht=1_types}\ref{item:TH=[2,2]_tip} & $b=6$ & $\langle 6;[2,2,2,2,2],[2,2],[2]\rangle$ \\
			\multicolumn{2}{c|}{--} & \ref{lem:ht=1_types}\ref{item:F'_fork_T=[2,2]} &  &$\langle 3;[2,2,2,2,2],[2,2],[2]\rangle$\\
			\multicolumn{2}{c|}{--} & \ref{lem:ht=1_types}\ref{item:ht=1_bench_two_a,b>2} & $a=b=2$ & 
			$\lbr 2,2,m,2,2\rbr$ \\
			\multicolumn{2}{c|}{--} & \ref{lem:ht=1_types} & & elliptic cone
		\end{tabular}
		\vspace{-1em}
		\caption{\cite{Kojima_Sing-1}: del Pezzo surfaces of rank one with exactly one singular point.}
		\label{table:Sing=1}
	\end{table}
	\begin{table}[ht]
		\begin{tabular}{cc}
			\begin{tabular}{r|rl}
				\cite{Kojima_index-2} & \multicolumn{2}{c}{this article} \\ \hline\hline 
				(1) & \ref{lem:ht=1_types}\ref{item:chains_columnar_T1=0,T2=0} & $m=4$ \\
				(2) & \ref{lem:ht=1_types}\ref{item:beta=3_columnar} & $m=2$, $T_1=[2,2,2]$, $T_2=[2]$ \\
				(3) & \ref{lem:ht=2,untwisted}\ref{item:c=2} & $m,n=2$, $T_1,T_2=[3]$ \\
				(4) & \ref{lem:ht=1_types}\ref{item:chains_both_T1=0} & $m=3$, $r=2$, $T=[3]$ \\
				(5) & \ref{lem:ht=1_types}\ref{item:chains_not-columnar} & $m,r_1,r_2=2$, $T_1,T_2=[3]$ \\
				(6) & \ref{lem:ht=1_types}\ref{item:chains_both} & $m,r=2$, $T_1,T_2=[3]$ \\
				(7) & \ref{lem:ht=2,untwisted}\ref{item:V-chains_c=1} & $r=2$, $T=[3]$, $T_1=[2]$, $T_2=[2,2]$ \\
				(8) & \ref{lem:ht=1_types}\ref{item:chains_columnar_T2=0} & $m=3$, $T=[3]$ \\
				(9) & \ref{lem:ht=1_types}\ref{item:chains_columnar} & $m=2$, $T_1,T_2=[3]$ \\
			\end{tabular}
			&
			\begin{tabular}{r|rl}
				\cite{Kojima_index-2} & \multicolumn{2}{c}{this article} \\ \hline\hline 
				(10) & \ref{lem:ht=2,untwisted}\ref{item:tau=id_chains} & $T=T_1=[3]$, $T_2=[2]$ \\
				(11) & \ref{lem:ht=2,untwisted}\ref{item:V-chains_c=1} & $r=2$, $T=[2,2]$, $T_1=[2]$, $T_2=[3]$\\
				(12) & \ref{lem:ht=2,untwisted}\ref{item:tau=id_fork_chain} & $m,n=2$, $T=[2,2]$, $T_1,T_2=[2]$ \\
				(13) & \ref{lem:ht=2,untwisted}\ref{item:tau=id_forks} & $m,n=2$, $T_1,T_2,T_3,T_4=[2]$ \\
				(14) & \ref{lem:ht=1_types}\ref{item:TH_long_columnar_b>2} & $b,m=2$, $T=[2,2,2]$ \\
				(15) & \ref{lem:ht=1_types}\ref{item:chains_columnar_T2=0} & $m=2$, $T=[2,2,2]$ \\
				(16) & \ref{lem:ht=1_types}\ref{item:chains_both} &  $m,r=2$, $T_1=[2,2,2]$, $T_2=[2]$ \\
				(17) & \ref{lem:ht=1_types}\ref{item:chains_columnar} & $m=2$, $T_1=[2,2,2]$, $T_2=[2]$ \\
				(18) & \ref{lem:ht=1_types}\ref{item:chains_columnar} & $m=2$, $T_1,T_2=[2,2,2]$ \\
			\end{tabular}
		\end{tabular}
		
		\caption{\cite{Kojima_index-2}: del Pezzo surfaces of rank one and index two, $\cha\kk=0$.}
		\label{table:index=2}
	\end{table}

	\begin{table}[ht]
		\begin{tabular}{cc}
			\begin{tabular}{r|rl}
				\cite{Kojima_Takahashi} & \multicolumn{2}{c}{this article} \\ \hline\hline 
				(1) & \ref{lem:ht=1_types}\ref{item:beta=3_columnar} & $m=2$, $T_1=[(2)_{m-1}]$, $T_2=[2]$  
				\\
				(2) & \ref{lem:ht=1_types}\ref{item:beta=3_other} & $m=r=2$, $T=[(2)_{m-1}]$ 
				\\
				(3) $r=1$ & \ref{lem:ht=1_types}\ref{item:chains_columnar_T2=0} & $T=[2]$ 
				\\
				$r=2$ & \ref{lem:ht=1_types}\ref{item:chains_columnar} & $T_1=T_2=[2]$ 
				\\
				$r=3$ & \ref{lem:ht=1_types}\ref{item:beta=3_columnar} & $T_1=T_2=[2]$ 
				\\
				$r=4$ &  \ref{lem:ht=1_types}\ref{item:ht=1_bench} &  
				\\
				(4) & \ref{lem:ht=1_types}\ref{item:chains_both} & $r=2$, $T_1=T_2=[2]$ 
				\\
				(5) & \ref{lem:ht=1_types}\ref{item:TH=[2,2]_tip_[3,2]} & $T=[2]$ 
				\\
				(6) & \ref{lem:ht=1_types}\ref{item:beta=3_other} & $r=2$, $T=[2]$ \\
			\end{tabular}
			&
			\begin{tabular}{r|rl}
				\cite{Kojima_Takahashi} & \multicolumn{2}{c}{this article} \\ \hline\hline 
				(7) & \ref{lem:ht=1_types}\ref{item:ht=1_bench-big_b>2}& $b=2$ \\
				(8) & \ref{lem:ht=1_types}\ref{item:beta=3_other_3} &  \\
				(9) & \ref{lem:ht=2,untwisted}\ref{item:c=2} & $m=n=2$, $T_1=T_2=[\tfrac{m+1}{2}]$ \\
				(10) & \ref{lem:ht=2,untwisted}\ref{item:c=3} & $n=m$, $T_1=T_2=[2]$ \\
				(11) & \ref{lem:ht=2,untwisted}\ref{item:c=2} & $n=m$, $T_1=T_2=[2]$ \\
				(12) & \ref{lem:ht=2,untwisted}\ref{item:V-chains_c=2_-2_twig_long}  & $n=m$, $T=[2]$ \\
				(13) & \ref{lem:ht=2,untwisted}\ref{item:tau=id_chains} & $T=[2]$, $T_1=T_2=[m]$ \\
				(14) & \ref{lem:ht=2,untwisted}\ref{item:V-chains_c=1} & $r=2$, $T=[2]$, $T_1=T_2=[m]$ 
			\end{tabular}	
		\end{tabular}
		\caption{\cite[Theorem 4.1, \sec 5.1]{Kojima_Takahashi}: del Pezzo surfaces of rank one and type IIa, $\cha\kk=0$.}
		\label{table:IIa}
	\end{table}
	
	\begin{table}[ht]
		\begin{tabular}{cc}
			\begin{tabular}{r|rl}
				\cite{Zhang_dP3} & \multicolumn{2}{c}{this article} \\ \hline\hline 
				(1) & \ref{lem:ht=1_types}\ref{item:chains_columnar_T1=0,T2=0} & $m=3$  
				\\
				(2) & \ref{lem:ht=2,untwisted}\ref{item:c=1_meeting} & $m,n=2$, $T=[3]$  
				\\
				(3) & \ref{lem:ht=2,untwisted}\ref{item:V-chains_c=1_rivet} & $r_j=2$, $T,T_1=[2]$, $T_2=[3]$ 
				\\
				(4) & \ref{lem:ht=2,untwisted}\ref{item:rivet_nu=3} & $m,n,r=2$, $T_1=[2]$, $T_2=[3]$ 
				\\
				(5) & \ref{lem:ht=2,untwisted}\ref{item:V-chains_c=2_[2]_T1=0} & $n=3$, $m,r=2$, $T_2=\emptyset$
				\\
				(6) & \ref{lem:ht=2_twisted-separable}\ref{item:twisted_off_nu=3} & $k=4$, $l=2$, $T=[2]$ 
				\\
				(7) & \ref{lem:ht=2_twisted-separable}\ref{item:twisted_off_nu=3} & $k=3$, $l=2$, $T=[2]$   
				\\
				(8) & \ref{lem:ht=2,untwisted}\ref{item:R_0-non-branching} & $m,n,r=2$, $T=[3]$  
				\\
				(9) & \ref{lem:ht=2,untwisted}\ref{item:c=1_meeting}  & $m,n=2$, $T=[3]$ 
				\\
				(10) & \ref{lem:ht=2_twisted-separable}\ref{item:twisted_off_nu=3} & $k,l=2$, $T=[2]$   
				\\
				(11) & \ref{lem:ht=2,untwisted}\ref{item:c=2} & $m,n=2$, $T_1=[3]$, $T_2=[2]$ 
				\\
				(12) & \ref{lem:ht=2,untwisted}\ref{item:V-chains_c=2_-2_twig_long} & $m,n=2$, $T=[2,2]$  
				\\
				(13) & $\height=3$ &  $\rA_1+\rD_6+[3,2]$
				\\
				(14) & $\height=3$ &  $\rA_7+[3,2]$    
				\\
				(15) & \ref{lem:ht=2,untwisted}\ref{item:[T_1,n]=[2,2]_r=2} & $m=3$  
				\\
				(16) & $\height=3$ &  $\rD_8+[3]$
				\\
				(17) & $\height=3$ &  $\rA_{1}+\rE_7+[3]$ 
				\\
				(18) & \ref{lem:ht=1_types}\ref{item:T2=[2]_b>2} & $b=2$, $r=3$, $T=[2,2]$   
				\\
				(19) & $\height=4$ &  $\rA_{1}+\rA_7+[3]$
				\\
				(20) & $\height=4$ &  $2\rA_4+[3]$
				\\
				(21) & $\height=4$ &  $\rA_{8}+[3]$
				\\
				(22) & $\height=4$ &  $\rA_{1}+\rA_2+\rA_5+[3]$
				\\
				(23) & $\height=4$ &  $3\rA_2+[3,2,2]$
				\\
				(24) & $\height=4$ &  $\rA_{2}+\rA_5+[3,2]$
				\\
				(25) & $\height=4$ &  $\rA_{2}+\rE_6+[3]$	
				\\
				(26) & $\height=4$ &  $\rA_{3}+\rD_5+[3]$ 
				\\
				(27) & $\height=4$ &  $\rA_{1}+2\rA_3+[3,2]$
				\\
				(28) & \ref{lem:ht=1_types}\ref{item:beta=3_columnar} & $m=2$, $T_1=[2,2]$, $T_2=[2]$
				\\
				(29) & \ref{lem:ht=1_types}\ref{item:beta=3_other} & $m,r=2$, $T=[2,2]$   
				\\
				(30) & \ref{lem:ht=1_types}\ref{item:chains_columnar_T2=0} & $m=3$, $T=[2]$   
				\\
				(31) & \ref{lem:ht=1_types}\ref{item:chains_columnar} & $m=3$, $T_1,T_2=[2]$   
				\\
				(32) & \ref{lem:ht=1_types}\ref{item:chains_both} & $m=3$, $r=2$, $T_1,T_2=[2]$   
				\\
				(33) & \ref{lem:ht=1_types}\ref{item:TH_long_columnar_b>2} & $b=2$, $m=3$, $T=[2]$  
				\\
				(34) & \ref{lem:ht=1_types}\ref{item:beta=3_columnar} & $m=3$, $T_1,T_2=[2]$    
				\\
				(35) & \ref{lem:ht=1_types}\ref{item:beta=3_other} & $m=3$, $T_1,T_2=[2]$ 
				\\
				(36) & \ref{lem:ht=1_types}\ref{item:chains_both_T1=0} & $m=3$, $r=2$, $T=[2]$ 
				\\
				(37) & \ref{lem:ht=1_types}\ref{item:chains_not-columnar} & $m=3$, $r_j=2$, $T,T_j=[2]$ 
				\\
				(38) & \ref{lem:ht=1_types}\ref{item:chains_both_T1=0} & $m,r=2$, $T=[3]$ 
				\\
				(39) & \ref{lem:ht=1_types}\ref{item:chains_both} & $m=3$, $r=2$, $T_1,T_2=[2]$ 
				\\
				(40) & \ref{lem:ht=1_types}\ref{item:chains_both} & $m,r=2$, $T_1=[2]$, $T_2=[3]$ 
				\\
				(41) & \ref{lem:ht=1_types}\ref{item:chains_both} & $m,r=2$, $T_1=[2]$, $T_2=[2,2]$ 
				\\
				(42) & \ref{lem:ht=1_types}\ref{item:chains_not-columnar} & $m,r_j=2$, $T_1=[3]$, $T_2=[2]$ 
				\\
				(43) & \ref{lem:ht=1_types}\ref{item:chains_not-columnar} & $m,r_j=2$, $T_1=[2,2]$, $T_2=[2]$ 
				\\
				(44) & \ref{lem:ht=2,untwisted}\ref{item:V-chains_c=1} & $r=2$, $T_1=[3]$, $T,T_2=[2]$
				\\
				(45) & \ref{lem:ht=2,untwisted}\ref{item:V-chains_c=1_rivet} & $r_j=2$, $T=[2,2]$, $T_j=[2]$  
				\\
				(46) & \ref{lem:ht=2,untwisted}\ref{item:V-chains_c=1_rivet} & $r_j=2$, $T=[3]$, $T_j=[2]$  
				\\
				(47) & \ref{lem:ht=2,untwisted}\ref{item:V-chains_c=1_rivet} & $r_1=3$, $r_2=2$, $T,T_j=[2]$ 
				\\
				(48) & \ref{lem:ht=1_types}\ref{item:TH_long_other_b>2} & $b=2$, $m=3$, $r=2$, $T=[2]$ 
				\\
				(49) & \ref{lem:ht=1_types}\ref{item:TH_long_other_b>2} & $b,m,r=2$, $T=[2,2]$ 
				\\		
			\end{tabular}
			&
			\begin{tabular}{r|rl}
				\cite{Zhang_dP3} & \multicolumn{2}{c}{this article}\\ \hline\hline
				(50) & \ref{lem:ht=1_types}\ref{item:TH_long_other_b>2} & $b,m,r=2$, $T=[3]$ 
				\\
				(51) & \ref{lem:ht=1_types}\ref{item:beta=3_other} & $m,r=2$, $T=[3]$ 
				\\
				(52) & \ref{lem:ht=1_types} \ref{item:T1=0,T=T2=[2]_b>2} & $b=2$, $m=3$  
				\\
				(53) & \ref{lem:ht=1_types}\ref{item:TH=[3,2]_T=[3]_b>2} & $b=2$, $T=[2,2]$
				\\	
				(54) & \ref{lem:ht=1_types}\ref{item:T2*=[2,2]_b>2} & $b=2$, $T=[2,2]$ 
				\\ 
				(55) &\ref{lem:ht=1_types}\ref{item:T2*=[2,2]_b>2} & $b=2$, $T=[3]$\\
				(56) & \ref{lem:ht=1_types}\ref{item:T2=[2,2]_b>2}  & $b=2$  \\
				(57) & \ref{lem:ht=1_types}\ref{item:TH=[2,2]_tip}  & $b=3$ \\
				(58) & \ref{lem:ht=1_types}\ref{item:T2=[2],T=[2,2],tip}  & $r=2$ \\
				(59) & \ref{lem:ht=2,untwisted}\ref{item:V-chains_c=2_[2]_T1=0}  & $m,r=2$, $n=3$, $T_2=\emptyset$ \\
				(60) & \ref{lem:ht=2,untwisted}\ref{item:V-chains_c=2_[2]}  & $m,n,r=2$, $T_1=[2]$, $T_2=\emptyset$ \\
				(61) & \ref{lem:ht=2,untwisted}\ref{item:V-chains_c=2_[2]_T1=0}  & $m,n,r=2$, $T_2=[2]$ \\
				(62) & \ref{lem:ht=2,untwisted}\ref{item:c=3}  &  $m,n=2$, $T_1=[3]$, $T_2=[2]$ \\
				(63) & \ref{lem:ht=1_types}\ref{item:chains_columnar_T2=0}  & $m=2$, $T=[3]$ \\
				(64) & \ref{lem:ht=1_types}\ref{item:chains_columnar}   & $m=2$, $T_1=[3]$, $T_2=[2]$ \\
				(65) & \ref{lem:ht=1_types}\ref{item:chains_both}  & $m,r=2$, $T_1=[3]$,  $T_2=[2]$ \\
				(66) & \ref{lem:ht=1_types}\ref{item:chains_columnar_T2=0}  & $m=2$, $T=[2,2]$ \\
				(67) & \ref{lem:ht=1_types}\ref{item:chains_columnar}  & $m=2$, $T_1=[2,2]$, $T_2=[2]$  \\
				(68) & \ref{lem:ht=1_types}\ref{item:chains_both}  & $m,r=2$, $T_1=[2,2]$, $T_2=[2]$ \\
				(69) & \ref{lem:ht=2,untwisted}\ref{item:c=2}  & $m,n=2$, $T_1=[2,2]$, $T_2=[2]$ \\
				(70) & \ref{lem:ht=2,untwisted}\ref{item:tau=id_chains}  & $T=[2,2]$, $T_j=[2]$  \\
				(71) & \ref{lem:ht=2,untwisted}\ref{item:V-chains_c=1}  & $r=2$, $T=[2,2]$, $T_j=[2]$  \\
				(72) & \ref{lem:ht=2,untwisted}\ref{item:V-chains_c=1}  &  $r=2$, $T=[3]$, $T_j=[2]$ \\
				(73) & \ref{lem:ht=1_types}\ref{item:chains_both_T1=0}   & $m=2$, $r=3$, $T=[2]$ \\
				(74) & \ref{lem:ht=1_types}\ref{item:chains_both}  & $m=2$, $r=3$, $T_j=[2]$ \\
				(75) & \ref{lem:ht=1_types}\ref{item:chains_not-columnar}  & $m,r_2=2$, $r_1=3$, $T_j=[2]$ \\
				(76) & \ref{lem:ht=2,untwisted}\ref{item:V-chains_c=1}  & $r=3$, $T,T_j=[2]$  \\
				(77) & \ref{lem:ht=2,untwisted}\ref{item:V-chains_c=1_rivet}  & $r_1=3$, $r_2=2$, $T=T_j=[2]$ \\
				(78) & \ref{lem:ht=2,untwisted}\ref{item:rivet_nu=3}  & $m,n=2$, $r=3$, $T,T_j=[2]$ \\
				(79) & \ref{lem:ht=2,untwisted}\ref{item:c=2}  & $m=3$, $n=2$, $T_j=[2]$  \\
				(80) & \ref{lem:ht=2,untwisted}\ref{item:rivet_nu=3}  &  \\
				(81) & \ref{lem:ht=2,untwisted}\ref{item:R_0-non-branching}  & $m,n=2$, $r=3$, $T=[2]$ \\
				(82) & \ref{lem:ht=2,untwisted}\ref{item:R_0-non-branching}  & $n=3$, $m,r=2$, $T=[2]$ \\
				(83) & \ref{lem:ht=2,untwisted}\ref{item:c=1_meeting}  & $m=3$, $n=2$, $T=[2]$  \\
				(84) & \ref{lem:ht=2,untwisted}\ref{item:tau=id_fork_chain}  & $m,n=2$, $T,T_j=[2]$  \\
				(85) & \ref{lem:ht=1_types}\ref{item:TH=[2,2]_b>2}  & $b=2$, $T=[3]$ \\
				(86) & \ref{lem:ht=1_types}\ref{item:TH=[2,2]_b>2}  & $b=2$, $T=[3,2]$  \\
				(87) & \ref{lem:ht=1_types}\ref{item:T2=[2]_b>2}  & $b=r=2$, $T=[3]$ \\
				(88) & \ref{lem:ht=1_types}\ref{item:beta=3_columnar}  & $m=2$, $T_1=[3]$, $T_2=[2]$ \\
				(89) & \ref{lem:ht=1_types}\ref{item:TH_long_columnar_b>2}  & $b=m=2$, $T=[3]$ \\
				(90) & \ref{lem:ht=2,untwisted}\ref{item:[T_1,n]=[3,2]_r=2}  &  \\
				(91) & \ref{lem:ht=2,untwisted}\ref{item:V-chains_c=2_-2_twig_long}   & $m,n=2$, $T=[3]$ \\
				(92) & \ref{lem:ht=1_types}\ref{item:beta=3_other_2}  &  $m=2$, $T=[3]$ \\
				(93) & \ref{lem:ht=2,untwisted}\ref{item:c=3}  & $n=3$, $m=2$, $T_j=[2]$ \\
				(94) & \ref{lem:ht=1_types}\ref{item:TH_long_other_b>2}  & $b=m=2$, $r=3$, $T=[2]$ \\
				(95) & \ref{lem:ht=1_types}\ref{item:beta=3_other}  & $m=2$, $r=3$, $T=[2]$ \\
				(96) & \ref{lem:ht=2,untwisted}\ref{item:V-chains_c=2_[2]_T1=0}  & $m=2$, $r=3$, $T_2=\emptyset$ \\
				(97) & \ref{lem:ht=1_types}\ref{item:TH_long_columnar_b>2}  & $b=m=2$, $T=[2,2]$ \\	
				&&\\
			\end{tabular}	
		\end{tabular}
		\smallskip	
		
		\caption{\cite{Zhang_dP3}: dP3-surfaces, $\cha\kk=0$.}
		\label{table:dP3}
	\end{table}
\end{small}

\clearpage
\newgeometry{top=2cm, bottom=0cm, left=2cm, right=2cm, twoside=false}

\section{Tables}

We now summarize all singularity types of log canonical del Pezzo surfaces of rank one and height at most two, see Theorem \ref{thm:ht=1,2}. We use conventions summarized in Section \ref{sec:notation}, which we now briefly recall. A singularity type is written as $T_1+\dots+T_{k}$, where $T_{i}$ is a type of a rational chain fork, or bench, which is a minimal resolution graph of the corresponding singular point. Namely, $[a_1,\dots,a_k]$ is a rational chain whose subsequent components have self-intersection numbers $-a_1,\dots,-a_n$; $\langle b;T_1,T_2,T_3\rangle$ is a rational fork with branching component $[b]$ and twigs $T_1$, $T_2$, $T_3$; and $\lbr C \rbr$ is a bench with central chain $C$, see Section \ref{sec:log_surfaces}. We recall that a fork $F=\langle b;T_1,T_2,T_3\rangle$ is admissible or log canonical if and only if
\begin{equation}\tag{$\dagger$}\label{eq:fork}
	\delta_{F}\de \frac{1}{d(T_1)}+\frac{1}{d(T_2)}+\frac{1}{d(T_3)}\geq 1\quad\mbox{and if } \delta_{F}=1 \mbox{ then $F$ is not a $(-2)$-fork}.
\end{equation}
For an integer $k\geq 0$, we write $(2)_{k}$ for a sequence consisting of the integer $2$ repeated $k$ times, which corresponds to a chain of $k$ $(-2)$-curves. For a definition of \enquote{$*$} see formula \eqref{eq:convention-Tono_star}.

\subsection{Log canonical del Pezzo surfaces of rank one and height 1, see Theorem \ref{thm:ht=1,2}\ref{item:ht=1}}

\begin{small}\vspace{-0.2cm}
	\begin{table}[h!]
		\begin{tabular}{r|llll|l|c}
			\ref{lem:ht=1_types} & \multicolumn{5}{c|}{singularity types} & $\chi$ \\ \hline\hline
			\ref{item:chains_columnar_T1=0,T2=0} & $[{m}]$ &&&& & $m+5$ \\ 
			\ref{item:chains_columnar_T2=0} & $[T,m]$ & $+\ T^{*}$ &&& & $m+3$ \\
			\ref{item:chains_columnar} & $[T_{1},{m},T_{2}]$ & $+\ T_{1}^{*}$ & $+\ T_{2}^{*}$ &&& $m+1$ \\ 
			\ref{item:chains_both_T1=0} & $[{m},T^{*},r,T]$ & $+\ [(2)_{r-2}]$ &&& $r\geq 2$ & $m+2$ \\ 
			\ref{item:chains_both} & $[T_{1},{m},T_{2}^{*},r,T_{2}]$ & $+\ T_{1}^{*}$ & $+\ [(2)_{r-2}]$ && $r\geq 2$ & $m$ \\
			\ref{item:chains_not-columnar} & $[T_{1},r_{1},T_{1}^{*},{m},T_{2}^{*},r_{2},T_{2}]$ & $+\ [(2)_{r_{1}-2}]$ & $+\ [(2)_{r_{2}-2}]$ && $r_1,r_2\geq 2$  & $m-1$\\ 
			\ref{item:beta=3_columnar} & $\langle  {m},T_{1},T_{2},T_{3}\rangle$ & 
			$+\ T_{1}^{*}$ & $+\ T_{2}^{*}$ & $+\ T_{3}^{*}$ & the fork satisfies \eqref{eq:fork} & $m-1$ \\ 
			\ref{item:beta=3_other} & $\langle {m};[T^{*},r,T],[2],[2] \rangle$ & 
			$+\ [(2)_{r-2}]$ & $+\ [2]$ & $+\ [2]$ & $r\geq 2$ & $m-2$ \\
			\ref{item:beta=3_other_2} & $\langle {m}; T,[2,2,2],[2]\rangle$ & $+\ (T^{*})$ & $+\ [2]$ && $d(T)\in \{3,4\}$, & $m-2$ \\
			&&&&& $m\geq 3$ if $d(T)=4$ & \\
			\ref{item:beta=3_other_3} &$\langle m; [2],[2,2,2],[2,2,2]\rangle$ & $+\ [2]$ &&&  $m\geq 3$ & $m-3$ \\
			\ref{item:T2*=[2,2]_b>2} & $\langle b;[2,3],[2,2],[2]\rangle$ & $+\ [3,(2)_{b-3}]$ &&& & $4$ \\
			 & $\langle 2;[2,3],[2,2],[2]\rangle$ &&&& & $6$ \\
			 & $\langle b;[2,2,2],[3],[2]\rangle$& $+\ [3,(2)_{b-3}]$ &&& & $4$ \\
			 & $\langle 2;[2,2,2],[3],[2]\rangle$ &&&& & $5$\\
			\ref{item:T1=0,T=T2=[2]_b>2} & $\langle b; [{m},2],[2],[2]\rangle$ & $+\ [3,(2)_{b-3}]$ &&&& $m+2$ \\
			 & $\langle 2; [m,2],[2],[2]\rangle$ &&&&& $m+1$ \\
			\ref{item:TH=[3,2]_T=[3]_b>2}  & $\langle b;[3,2],T,[2]\rangle$ & $+\ [(2)_{b-2}]*T^{*}$ &&& $d(T)=3$ & $5$ \\
			 & $\langle 2;[3,2],[3],[2]\rangle$ & $+\ [2]$ &&& & $4$ \\
			 & $\langle 2;[3,2],[2,2],[2]\rangle$ &&&& & $3$ \\
			 \ref{item:TH=[2,2]_b>2} & $\langle b;T, [{2},2],[2]\rangle$ & $+\ [(2)_{b-2}]*T^{*}$ &&& $d(T)\in \{3,4,5,6\}$ & $4$ \\
			 & $\langle 2;[2,3], [{2},2],[2]\rangle$ & $+\ [3]$  &&&& $3$  \\
			 & $\langle 2;[3,2], [{2},2],[2]\rangle$ & $+\ [2]$ &&&& $2$ \\
			 & $\langle 2;[(2)_{k}],[2,2],[2]\rangle$ &&&& $k\in \{2,3,4\}$ & $4-k$  \\
			 & $\langle 2;[k],[{2},2],[2]\rangle$ & $+\ [(2)_{k-2}]$ &&& $k\in \{3,4,5,6\}$ & $3$ \\
			\ref{item:TH=[3,2]_tip} & $\langle 3;[3,2],[2,2],[2]\rangle$ &&&& & $4$ \\
			\ref{item:TH=[2,2]_tip} & $\langle b;[(2)_{b-1}],[2,2],[2]\rangle$ &&&& $b\in \{3,4,5,6\}$ & $3$  \\
			\ref{item:TH=[2,2]_tip_[3,2]} & $\langle 3;[2,3],[{2},2],[2]\rangle$ & $+\ [2]$ &&& & $3$ \\
			\ref{item:TH_long_columnar_b>2} & $\langle b;[T,{m},2],[2],[2]\rangle$ & $+\ T^{*}$ & $+\ [3,(2)_{b-3}]$ &&& $m$ \\
			& $\langle 2;[T,{m},2],[2],[2]\rangle$ & $+\ T^{*}$ &&& & $m-1$ \\
			\ref{item:TH_long_other_b>2} & $\langle b; [T,r,T^{*},{m},2],[2],[2]\rangle$ & $+\ [(2)_{r-2}]$ & $+\ [3,(2)_{b-3}]$ && $r\geq 2$ & $m-1$  \\
			& $\langle 2; [T,r,T^{*},{m},2],[2],[2]\rangle$ & $+\ [(2)_{r-2}]$ &&& $r\geq 2$ & $m-2$  \\
			\ref{item:F'_fork_T=[3]_b>2} & $\langle b,[2,2,2,2,2],[3],[2]\rangle$ & $+\ [(2)_{b-3},3,2]$ &&&& $1$ \\
			& $\langle 2,[2,2,2,2,2],[3],[2]\rangle$ & $+\ [2]$ &&& & $0$ \\
			\ref{item:F'_fork_T=[2,2]} & $\langle 3,[2,2,2,2,2],[2,2],[2]\rangle$ &&&&& $0$ \\
			\ref{item:T2=[2,2]_b>2} & $\langle b;[(2)_{r+2}],[3],[2]\rangle$ & $+\ [r]$ & $+\ [3,(2)_{b-3}]$ && $r\in \{2,3\}$ & $2$\\
			 & $\langle 2;[(2)_{r+2}],[3],[2]\rangle$ & $+\ [r]$ &&& $r\in \{2,3\}$ & $1$ \\
			\ref{item:T2=[2]_b>2} & $\langle b;[(2)_{r+1}],T,[2]\rangle$ & $+\ [(2)_{b-2}]*T^{*}$ & $+\ [r]$ && $r\in \{2,3,4\}$, $d(T)=3$  & $2$ \\ 
			& $\langle 2;[(2)_{r+1}],[3],[2]\rangle$ & $+\ [r]$ & $+\ [2]$ && $r\in \{2,3,4\}$ & $1$ \\
			& $\langle 2;[(2)_{r+1}],[2,2],[2]\rangle$ & $+\ [r]$ &&& $r\in \{2,3\}$  & $1$ \\
			\ref{item:T2=[2],T=[2,2],tip} & $\langle 3;[(2)_{r+1}],[2,2],[2]\rangle$ & $+\ [r]$ &&& $r\in \{2,3\}$  & $1$  \\
			\ref{item:ht=1_bench}   & $\lbr m\rbr$ & $+\ 4\cdot [2]$ &&& $m\geq 3$ & $m-3$ \\
			\ref{item:ht=1_bench-big_b>2} & $\lbr b,2, m\rbr$ & $+\ [3,(2)_{b-3}]$ & $+\ [2]$ & $+\ [2]$ & & $m-2$ \\ 			
			& $\lbr 2,2, m\rbr$ & $+\ [2]$ & $+\ [2]$ && $m\geq 3$ & $m-3$\\ 
			\ref{item:ht=1_bench_two_a,b>2} & $\lbr a,2, m,2,b\rbr$ & $+\ [3,(2)_{a-3}]$  & $+\ [3,(2)_{b-3}]$ && & $m-1$ \\
			& $\lbr 2,2, m,2,b\rbr$ & $+\ [3,(2)_{b-3}]$ &&& & $m-2$\\
			& $\lbr 2,2, m,2,2\rbr$ &&&& $m\geq 3$ & $m-3$ \\
		\end{tabular}
		\caption{Rational log canonical del Pezzo surfaces of rank one and height $1$, see Lemma \ref{lem:ht=1_types}. Here $m\geq 2$, $a,b\geq 3$; $T,T_j$ are admissible chains; and $\chi\de \chi(\lts{X}{D})$.}
		\label{table:ht=1}
	\end{table}
\end{small} 

\clearpage
\subsection{Log canonical del Pezzo surfaces of rank one and height 2, see Theorem \ref{thm:ht=1,2}\ref{item:ht=2_width=2},\ref{item:ht=2_width=1}}

\begin{small}	
{\renewcommand{\arraystretch}{1.2}	
	\begin{table}[h!]\vspace{-1em}
		\addtolength{\leftskip} {-2cm} 
		\addtolength{\rightskip}{-2cm}
		\begin{tabular}{r|lll|l|c|c|c}
			  & \multicolumn{4}{c|}{singularity types}  & $\#\Pht$ & swaps to & $\chi$ \\ \hline\hline
			 \ref{lem:ht=2,untwisted}\ref{item:tau=id_chains} & 
			$[{T_2},T^{*}]$ & $+\ [T,{T_1}]$ & $+\ [T_1^{*},T_2^{*}]$ && 1 &  \ref{ex:ht=2}\ref{item:3A2_construction} 
			& 2 \\   \ref{item:V-chains_c=1} & $[T,{T_1},r,T_1^{*},T_{2}^{*}]$ & $+\ [{T_{2}},T^{*}]$ & $+\ [(2)_{r-2}]$ & $r\geq 2$ & 
			& &1\\   \ref{item:V-chains_c=1_rivet} & $[{T_{2}},T^{*},r_2,T,{T_1},r_1,T_1^{*},T_2^{*}]$ & $+\ [(2)_{r_1-2}]$ &  $+\ [(2)_{r_2-2}]$ & $r_1,r_2\geq 2$ &
			& & 0\\  \ref{item:lc_rivet} & $\langle r,[2,2,2,2,2],[2,2],[2]\rangle$ & $+\ [(2)_{r-3},3]$ &&  $r\geq 3$ &  
			& & 0\\   \ref{item:[T_1,n]=[2,2]_r>2} &
			$\langle r; [(2)_{m}],[2,2],[2]\rangle$ & $+\ [m,2]$ & $+\ [(2)_{r-3},3]$ &  $r\geq 3$, $m\leq 5$ &  
				& & 1 \\   \ref{item:[T_1,n]=[3,2]_r>2}  & 
			$\langle r; [3,{2}],[2,2],[2]\rangle$ & $+\ [{2},2,2]$ & $+\ [(2)_{r-3},3]$ &  $r\geq 3$ &  
			& & 1\\ \ref{item:lc_T=[2]} & $\langle r;[2,2,2],[2,2,2],[2]\rangle$  & $+\  [3,3]$ & $+\ [(2)_{r-3},3]$ &  $r\geq 3$ & 
			& & 1\\  \ref{item:lc_T=[2,2]} & $\langle r;[2,2],[2,2],[2,2]\rangle$ & $+\ [2,2]$ &   $+\ [(2)_{r-3},4]$ & $r\geq 3$ & 
			& & 1\\\cline{2-6}   \ref{item:[T_1,n]=[2,2]_r=2} &
			$\langle 2; [2,2],[(2)_{m}],[2]\rangle$ & $+\ [m,2]$&& $m \leq 4$ & 2
			& & 0\\   \ref{item:[T_1,n]=[3,2]_r=2} &
			$\langle 2; [3,{2}],[2,2],[2]\rangle$ & $+\ [{2},2,2]$ &&& 
			& & 0 \\	\cline{2-8} 
			 \ref{item:c=2} &
			$[T_{2}^{*},{m},T_{1}^{*}]$ & $+\ [T_{1},{n},T_{2}]$ & $+\ [(2)_{m+ n-3}]$ & $m,n\geq 2$ & 1 & \ref{ex:ht=2}\ref{item:A1+2A3_construction}
			 & 1\\   \ref{item:rivet_nu=3} &
			$[T_{2},{m},T_{1},r,T_{1}^{*},{n},T_{2}^{*}]$ & $+\ [(2)_{m+ n-3}]$ & $+\ [(2)_{r-2}]$ & &   
			& & 0\\   
			\ref{item:rivet_nu=3-fork-1} & $\langle [2,n,T^{*},r,T],
				[2],[2]\rangle$ & $+\ [3,(2)_{m+n-4}]$ & $+\ [(2)_{r-2}]$ & $r\geq 2$ & & & 0 \\
			\ref{item:rivet_nu=3-fork-2} & 
			 $\langle m,
				[2,2,2,2,2]
				,[3],[2]\rangle$ & $+\ [4,(2)_{m-2}]$ &&&&& 0 \\
			&  $\langle m,
			[2,2,2,2,2]
			,[2,2],[2]\rangle$ & $+\ [2,3,(2)_{m-2}]$ && $m\geq 3$ &&& 0
			\\
			\ref{item:tau=id_fork_chain} &
			$\langle {n},T_{1},T_{2}\trp,T\trp \rangle$ & $+\ 
			[T_{1}^{*},{m},T_{2}^{*}]$ & $+\ T^{*}*[(2)_{n+ m-3}]$ 
			&&  
			& & 1 \\   
			\ref{item:V-chains_c=2_[2]_T1=0}  &
			$\langle {n}; T,[2],[2]\rangle$ & $+\ 
			[2,{m},2]$ & $+\ [(2)_{r-2}]$& $r\geq 2$ & 
			&& 0 \\[-0.3em] 
			& \multicolumn{3}{l|}{\hspace{1cm} where $T=[(2)_{m+ n-3}]*[T^{*}_2,r,T_2]$ or  $[(2)_{m+ n-3},r]$} &&
			&& \\   \ref{item:V-chains_c=2_[2,3]} &
			$\langle {2}; T,[2,3],[2]\rangle$ & $+\ [2,{2},T^{*}]$ & $+\ [2]$ &  $d(T)=3$ &   
			&& 0 \\   \ref{item:V-chains_c=2_-2_twig_long} & 
			$\langle {n}; T,[(2)_{m+ n-2}],[2]\rangle$ & $+\ [2,{m},T^{*}]$ &&&
			&& 0 \\& \multicolumn{3}{l|}{\hspace{1cm} where  $m+ n\leq 7$, $m\leq 4$, and $d(T)=3$} &&
			&& \\   \ref{item:V-chains_c=2_-2_twig}  &
			$\langle {2}; T,[2,2],[2]\rangle$ & $+\ [2,{2},T^{*}]$ && $d(T)\in \{4,5\}$ & 
			&& 0 \\   \ref{item:tau=id_forks} & 
			$\langle {n},T_{1}\trp,T_{2}\trp,T_{3}\trp \rangle$ & $+\ \langle {m};T_{1}^{*},T_{2}^{*},T_{4}^{*}\rangle$ & $+\ (T_{3}^{*}*[(2)_{n+ m-3}]*T_{4})$ & & & & 0 \\[-0.3em]
			& \multicolumn{3}{l|}{\hspace{1cm} and $\sum_{i=1}^{3}\frac{1}{d(T_i)}<1$, i.e.\ the first fork is admissible} && 
			&& \\   \ref{item:V-chains_c=2_[2]}  & 
			$\langle {n}; T,[2],[2]\rangle$ & 
			$+\ \langle {m};T_1\trp,[2],[2]\rangle$ & $+\ [(2)_{r-2}]$ & $r\geq 2$ & 	
			&& 0 \\[-0.3em]
			& \multicolumn{3}{l|}{\hspace{1cm} where $T=T_1^{*}*[(2)_{m+ n-3}]*[T^{*}_2,r,T_2]$ or  $T_{1}^{*}*[(2)_{m+ n-3},r]$} &&	
			&& \\ \cline{2-8}
			 \ref{item:R_0-non-branching} &
			$[T,{n},r,{m},T^{*}]$ & $+\ [(2)_{m+ n-2}]$ & $+\ [(2)_{r-1}]$ & $r\geq 2$ & 1 & 
			\ref{ex:ht=2}\ref{item:A1+A2+A5_construction} 
			& 0 \\  \ref{item:R_0-lc}  & $\langle  r,[2,2],[2,2],[2,2]\rangle$& $+\ [(2)_{r-2},4]$ & $+\ [2,2]$ & $r\geq 3$ &
			&& 0 \\ \ref{item:R_0_branching} &
			$\langle r; [2,{m}],[2,{2}],[2]\rangle$ & $+\ [(2)_{r-2},3]$ & $+\ [(2)_{m}]$ & $r\geq 2$, $m\leq 3$ &  
			%
			& & 0\\	\cline{2-8}
			 \ref{item:c=1_meeting} &
			$[T^{*},{m},{n},T]$ & $+\ [(2)_{n+ m}]$ &&& 1 &  \ref{ex:ht=2_meeting}
			& 0 \\
			\ref{item:2A4_[2]} & $\langle n,[2],[2],[2,m]\rangle$ & $+\ [(2)_{n+m-1},3]$ &&&&& 0\\
			\ref{item:2A4_m=3} & $\langle n,[2],T,[2,3]\rangle$ & $+\ [(2)_{n+3}]*T^{*}$ && $d(T)=3$ &&& 0\\
			\ref{item:2A4_m=2}& $\langle n,[2],T,[2,2]\rangle$ & $+\ [(2)_{n+2}]*T^{*}$ && $3\leq d(T)\leq 6$ & 
			&& 0 \\ \cline{2-7}
			 \ref{item:c=3} &
			$\langle {n},T_{1}\trp,T_{2}\trp,[2]\rangle$ & $+\ \langle {m},T_{1}^{*},T_{2}^{*},[2]\rangle$ & $+\ [(2)_{m+ n-4}]$ &&  $\cM^{1}$ & \ref{ex:ht=2}\ref{item:2D4_construction}
			& 0 \\  \ref{item:ht=2_bench} & $\langle m; [2],[2],[2]\rangle$ & $+\ \lbr n \rbr$ & $+\ [(2)_{m+n-5},3]$ &  $n\geq 3$ &
			&& 0 \\ \cline{1-8}
			 \ref{lem:ht=2_twisted-separable}\ref{item:twisted_off_nu=2} &
			$\langle 2;[(2)_{k}],[2],[2]\rangle$ & $+\ [T,{k},T^{*}]$ && $T\neq [2]$, $k\geq 2$ &  1 & \ref{ex:ht=2_twisted}\ref{item:A3+D5_construction}
			& 0 \\  \ref{item:c=1_T=[2]_l=0} & 
			$\langle {m}; [2,2],[3],[2] \rangle$ & $+\ \langle 2;[3,(2)_{m+ 2}],[2],[2]\rangle$ &&&&& 0
		\end{tabular}
	\caption{Log canonical del Pezzo surfaces of rank 1 and height $2$ occurring for any $\cha\kk$.  
		Here $m,n\geq 2$, $T,T_j$ are admissible chains, and all forks satisfy \eqref{eq:fork}, i.e.\ are admissible or lc. }
	\label{table:ht=2_char-any}
	\end{table}
\vfill
\begin{table}
			\begin{tabular}{r|lll|l}
			\ref{lem:ht=2_twisted-separable} & \multicolumn{4}{c}{singularity types}   \\ \hline\hline
			  \ref{item:twisted_off_nu=3} &
			$[T,{k+\ l-3},T^{*}]$ &$+\ \langle 2;[(2)_{k-3}],[2],[2]\rangle$ & $+\ \langle 2;[(2)_{l-3}],[2],[2]\rangle $ & $k,l\geq 4$ 
			\\ &
			$[T,{k},T^{*}]$ &$+\ \langle 2;[(2)_{k-3}],[2],[2]\rangle$ & $+\ [2,2,2]$ & $k\geq 4$ 
			  \\ &
			$[T,{k-1},T^{*}]$ &$+\ \langle 2;[(2)_{k-3}],[2],[2]\rangle$ & $+\ [2]+\ [2]$ & $k\geq 4$ 
			\\ &
			$[T,3,T^{*}]$ &$+\ [2,2,2]$ & $+\ [2,2,2]$ & 
			\\ &
			$[T,2,T^{*}]$ &$+\ [2,2,2]$ & $+\ [2]+\ [2]$ & 
			\\  \ref{item:k2=2} &
			$\langle {k}; [2,n],[2],[2] \rangle$& $+\ [3,(2)_{n-2},3]$ & $+\ \langle 2;[(2)_{k-3}],[2],[2]\rangle$ & $k\geq 4$ 
			\\ & 
			$\langle 3; [2,n],[2],[2] \rangle$& $+\ [3,(2)_{n-2},3]$ & $+\ [2,2,2]$ &
			\\  \ref{item:c=1_T1=[2]} & 
			$\langle {n+ k-2}; [2],[2],T\trp \rangle$ & 
			$+\ \langle 2;T^{*}*[(2)_{n-2}];[2],[2] \rangle$ & 
			$+\ \langle 2;[(2)_{k-3}],[2],[2]\rangle$ & $k\geq 4$
			\\  & 
			$\langle {n+1}; [2],[2],T\trp \rangle$ & 
			$+\ \langle 2;T^{*}*[(2)_{n-2}];[2],[2] \rangle$ & 
			$+\ [2,2,2]$ &
			\\  & 
			$\langle {n+1}; [2],[2],T\trp \rangle$ & 
			$+\ \langle 2;T^{*}*[(2)_{n-2}];[2],[2] \rangle$ & 
			$+\ [2]+\ [2]$ &
			\\  \ref{item:c=1_T=[2]} &
			$\langle {k+\ l-2}; [2,2],[3],[2] \rangle$ &
			$+\ \langle 2;[3,(2)_{k-3}],[2],[2] \rangle$ &
			$+\ \langle 2;[(2)_{l-3}],[2],[2]\rangle$ & $k\geq 3$, $l\geq 4$ 
			\\ &
			$\langle {l}; [2,2],[3],[2] \rangle$ &
			$+\ [2,3,2]$ &
			$+\ \langle 2;[(2)_{l-3}],[2],[2]\rangle$ & $l\geq 4$ 
			\\ &
			$\langle {k+1}; [2,2],[3],[2] \rangle$ &
			$+\ \langle 2;[3,(2)_{k-3}],[2],[2] \rangle$ &
			$+\ [2,2,2]$ & $k\geq 3$ 
			\\ &
			$\langle {3}; [2,2],[3],[2] \rangle$ &
			$+\ [2,3,2]$ &
			$+\ [2,2,2]$ & 
			\\ &
			$\langle {k}; [2,2],[3],[2] \rangle$ &
			$+\ \langle 2;[3,(2)_{k-3}],[2],[2] \rangle$ &
			$+\ [2]+\ [2]$& $k\geq 3$
			\\ &
			$\langle {2}; [2,2],[3],[2] \rangle$ &
			$+\ [2,3,2]$ &
			$+\ [2]+\ [2]$ & 	
		\end{tabular}
		\caption{Del Pezzo surfaces of rank one and  height $2$, occurring only in case $\cha\kk\neq 2$. They swap vertically to the surface from Example \ref{ex:ht=2_twisted}\ref{item:2A1+2A3_construction}. Each singularity type $\cS$ has $\#\Pht(\cS)=1$. In each case we have $h^i(\lts{X}{D})=0$ for $i=0,1,2$.}
		\label{table:ht=2_char-neq-2}
	\end{table}
\medskip

	\begin{table}
		\addtolength{\leftskip} {-2cm} 
		\addtolength{\rightskip}{-2cm}	
		\begin{tabular}{r|lllll|l|c}
			\ref{lem:ht=2_twisted-inseparable} & \multicolumn{6}{c|}{singularity types} & $\#\Pht$ \\ \hline\hline
			\ref{item:beta=0} & 
			$[{k_{1}^{\geq}-4}]$ & $+\ \sum_{j=1}^{\nu}\rD_{k_{j}} $ &&&& & $\cM^{\nu-3}$
			\\ \ref{item:beta=1_k=2} & 
			$ [ {k_{2}^{\geq}-2},l ,2]$ & $+\ [3,(2)_{l-2}]$ & $+\ \sum_{j=2}^{\nu}\rD_{k_{j}} $ &&&&
			\\ \ref{item:beta=1_k>2} &
			$ [ {k_{1}^{\geq}-3},T]$ & $+\ 
			\gfork{T^{*}*[(2)_{k_{1}-1}]}$ & $+\ \sum_{j=2}^{\nu}\rD_{k_{j}} $ &&&&
			\\ \ref{item:beta=2_k1=k2=2} &
			$[2,l_{1} , {k_{3}^{\geq}},l_2,2]$ & $+\  \sum_{j=1}^{2}[(2)_{l_{j}-2},3]$ & $+\ \sum_{j=3}^{\nu}\rD_{k_{j}} $ &&&&
			\\ \ref{item:beta=2_k1=2,k2>2} &
			$[2,l , {k_{2}^{\geq}-1},T]$ & $+\ 
			\gfork{T^{*}*[(2)_{k_2-1}]}$ & $+\  [(2)_{l-2},3]$ & $+\ \sum_{j=3}^{\nu}\rD_{k_{j}} $ &&&
			\\ \ref{item:beta=2_k1>2,k2>2} &
			$ [T_{1}\trp, {k_{1}^{\geq}-2},T_{2}]$ & $+\ 
			\sum_{j=1}^{2}\gfork{T_{j}^{*}*[(2)_{k_{j}-1}]}$ & $
			+\ \sum_{j=3}^{\nu}\rD_{k_{j}} $ &&&&
			\\ \ref{item:beta=3_notD_k1=k2=2_V1=V2=[2,2]} &
			$\langle  {k_{3}^{\geq}+1};[2,2] ,[2,2],[2]\rangle$ & $+\ 
			\gfork{[3,(2)_{k_3-2}]}$ 
			& $+\ [3]+\ [3]$ & $+\ \sum_{j=4}^{\nu}\rD_{k_{j}} $ &&&
			\\ \ref{item:beta=3_notD_k1=k2=2_V1=[3,2],V2=[2,2]} &
			$\langle  {k_{3}^{\geq}+1};[2,3] ,[2,2],[2]\rangle$ 
			& $+\ \gfork{[3,(2)_{k_3-2}]}$ &
			$+\ [3,2]+\ [3]$& $+\ \sum_{j=4}^{\nu}\rD_{k_{j}} $ &&&
			\\ \ref{item:beta=3_U2=U3=[2]_k1=2} & 
			$\langle  {k_{2}^{\geq}};[2,l] ,[2],[2]\rangle$ & $+\ 
			\sum_{j=2}^{3}\gfork{[3,(2)_{k_j-2}]}$ & $+\  [(2)_{l-2},3]$  & $+\ \sum_{j=4}^{\nu}\rD_{k_{j}} $ &&&
			\\ \ref{item:beta=3_notD_k1=2_V1=[3,2]_k2>2} &
			$\langle  {k_{2}^{\geq}};[2,3] ,T\trp,[2]\rangle$ & $+\  
			\gfork{T^{*}*[(2)_{k_{2}-1}]}$ & $+\ \gfork{[3,(2)_{k_3-2}]}+ 
			[3,2]\!\!\!\!\! $ & $+\ \sum_{j=4}^{\nu}\rD_{k_{j}} $ && $d(T)=3$,& \\
			\ref{item:beta=3_notD_k1=2_V1=[2,2]_k2>2} & 
			$\langle  {k_{2}^{\geq}};[2,2] ,T\trp,[2]\rangle$ & 
			$+\ \gfork{T^{*}*[(2)_{k_{2}-1}]}$ 
			& $+\ \gfork{[3,(2)_{k_3-2}]}+[3]$ & $+\ \sum_{j=4}^{\nu}\rD_{k_{j}} $  && $3\leq d(T)\leq 5$, & 
			\\ \ref{item:beta=3_notD_kj>2} &
			$\langle  {k_{1}^{\geq}-1}; T_{1}\trp,T_2\trp,[2]\rangle$ & $+\ 
			\sum_{j=1}^{2}
			\gfork{T_{j}^{*}*[(2)_{k_{j}-1}]}$ & $+\ \gfork{[3,(2)_{k_3-2}]}$ & $+\ \sum_{j=4}^{\nu}\rD_{k_{j}} $ && $\tfrac{1}{d(T_{1})}+\tfrac{1}{ d(T_{2})}>\tfrac{1}{2}$& 
			\\ \ref{item:C_[2]_L2H=1} & $ \langle k;[2] ,[2],[ {k_{2}^{\geq}-2}] \rangle$ &
			$+\ [3,(2)_{k-2}]*[2]$ &  $+\ \sum_{j=2}^{\nu}\rD_{k_{j}} $ &&&&
			\\ \ref{item:C_[2]_k2=2} & 
			$\langle k;[2] ,[2],[2,l, {k_{3}^{\geq}}]\rangle$ &
			$+\ [3,(2)_{k-2}]*[2]$ & $+\ 
			[(2)_{l-2},3]$ & $
			+\ \sum_{j=3}^{\nu}\rD_{k_{j}} $ &&&
			\\ \ref{item:C_[2]_k2>2} & 
			$\langle k;[2] ,[2],[T\trp, {k_{2}^{\geq}-1}]\rangle$ & $+\ [3,(2)_{k-2}]*[2]$ & $+\ 
			\gfork{T^{*}*[(2)_{k_2-2}]}$ & 
			$+\ \sum_{j=3}^{\nu}\rD_{k_{j}} $ &&&
			\\ 
			\hline
			\ref{item:C_columnar_nu=5} & $\langle k;[2],T\trp,[6]\rangle$ & $+\ [3,(2)_{k-2}]*(T^{*})$ & $+\ 6 \cdot [2]$ &&& $d(T)=3$, $k\geq 3$ & $\cM^{2}$
			\\ \hline
			\ref{item:C_columnar_nu=4_k2=3} &
			$\langle k;[2], T\trp,[ {5}]\rangle$ & $+\ [3,(2)_{k-2}]*(T^{*})$ & $+\ [2,2,2]$ & $+\ 4\cdot [2]$ && $d(T)=3$ & $\cM^{1}$
			\\ \ref{item:C_columnar_nu=4_k2=2} &
			$\langle k;[2], T\trp,[ {4}]\rangle$ & $+\ [3,(2)_{k-2}]*(T^{*})$ & $+\ 6\cdot [2]$ &&& $d(T)\in \{3,4\}$&  
			\\ \hline
			\ref{item:C_columnar_nu=3_k2=3,k3=2} &
			$\langle k;[2], T\trp,[ {3}]\rangle$ & $+\ [3,(2)_{k-2}]*(T^{*})$ & $+\ [2,2,2]$ & $+\ [2]+\ [2]$ && $3\leq d(T)\leq 6$ & $1$
			\\ \ref{item:C_columnar_nu=3_k2=k3=2} &
			$\langle k;[2], T\trp,[ {2}]\rangle$ & $+\ [3,(2)_{k-2}]*(T^{*})$ & $+\ 4\cdot [2]$ &&& $T\neq [2]$ & 
			\\ \ref{item:C_columnar_nu=3_k2+k3>5} &
			$\langle k;[2], T\trp,[ {k_2+\ k_3-2}]\rangle$ & $+\ [3,(2)_{k-2}]*(T^{*})$ & $+\ \sum_{j=2}^{3}\rD_{k_j} $ &&& $d(T)= 3$, & \\
			&&&&&&$k_2+\ k_3\in \{6,7\}$ & 
			\\ \ref{item:C_columnar_L2H=0_k2=2} &
			$\langle k;[2], T\trp,[2,2, {2}]\rangle$ & $+\ [3,(2)_{k-2}]*(T^{*})$ & $+\ [2]$ & $+\ [2]+\ [2]$ && $d(T)=3$ or & \\
			&&&&&& $d(T)=4$, $k\geq 3$  &
			\\ \ref{item:C_columnar_L2H=0_k2=3} & $\langle k;[2], T\trp,[2, {3}]\rangle$ & $+\ [3,(2)_{k-2}]*(T^{*})$ & $+\ [2,3,2]$ & $+\ [2]+\ [2]$ && $d(T)=3$ &
			\\ 
			\hline
			\ref{item:C_fork_L2H=1} &
			$\langle k ;[2],[2],[ {k_{2}^{\geq}-2}] \rangle$&  $+\ [(2)_{k-2}]$ & $+\ \sum_{j=2}^{\nu}\rD_{k_{j}} $ &&&& $\cM^{\nu-2}$
			\\ \ref{item:C_fork_k2=2} &
			$\langle k ;[2],[2],[2,l, {k_{3}^{\geq}}]\rangle$  & $+\ 
			[(2)_{l-2},3]$&  $+\ [(2)_{k-2}]$ & $+\ \sum_{j=3}^{\nu}\rD_{k_{j}} $&&&
			\\ \ref{item:C_fork_k2>2} &
			$\langle k ;[2],[2],[T\trp, {k_{2}^{\geq}-1}]\rangle$ & $+\ 
			\gfork{T^{*}*[(2)_{k_2-2}]}$ &  $+\ [(2)_{k-2}]$ & $
			+\ \sum_{j=3}^{\nu}\rD_{k_{j}} $ &&&
			\\
			\hline 
			\ref{item:twisted-bench} & $\lbr k \rbr$ & $+\ [(2)_{k-3},3]$ & $+\ 4\cdot [2]$ &&&  $k\geq 3$ & $\cM^{1}$
			\\
		\end{tabular}
		\caption{Del Pezzo surfaces of rank one, height $2$, occurring only in case $\cha\kk=2$. They all swap vertically to one of the surfaces from Example \ref{ex:ht=2_twisted_cha=2}. 
			The numbers $h^i\de h^i(\lts{X}{D})$ are as follows: $h^0=0$, $h^1$ is the moduli dimension listed above, and $h^2=\nu-2$. 
			\\  Here $\nu\geq 3$, $k,l,k_j\geq 2$; $k^{\geq}_{i}\de \sum_{j=i}^{\nu}k_j$, and we use symbols $\rD_2\de 2\rA_1$, $\rD_3\de \rA_3$ and $\gfork{\cdot}$ introduced  in Notation \ref{not:Dk,gfork}. We also use the convention $(2)_{-1}$ introduced in formula \eqref{eq:convention_2-1}.}
		\label{table:ht=2_char=2}
	\end{table}
}
\end{small}

\clearpage
\subsection{Non-uniqueness and moduli, see Propositions \ref{prop:moduli}, \ref{prop:moduli-hi}}

\begin{small}
{\renewcommand{\arraystretch}{1.1}		
	\begin{table}[htbp]\vspace{-1.5em}
		\addtolength{\leftskip} {-2cm} 
		\addtolength{\rightskip}{-2cm}
		\begin{tabular}{c|r|llll|cccc}	
			\multirow{2}{*}{$\height$} & \multirow{2}{*}{construction} & \multicolumn{4}{c|}{\multirow{2}{*}{singularity type}} & \multicolumn{4}{c}{$\#$ iso classes if $\cha \kk$} \\ 
			&&&&&& $\neq 2,3,5$ & $=5$ & $=3$ & $=2$ \\ \hline\hline
			$2$ &  \ref{lem:ht=2,untwisted}\ref{item:[T_1,n]=[2,2]_r=2} &
			$\langle 2; [(2)_{r}],[2,2],[2]\rangle$ & $+\ [r,2]$&& $4\geq r\geq 2$ & 2 & 2 & 2 & 2 \\
			& \ref{item:[T_1,n]=[3,2]_r=2} &
			$\langle 2; [3,{2}],[2,2],[2]\rangle$ & $+\ [{2},2,2]$ &&& 2& 2 & 2 & 2 \\
			& \ref{item:c=3} &
			$\langle {n},T_{1}\trp,T_{2}\trp,[2]\rangle$ & $+\ \langle {m},T_{1}^{*},T_{2}^{*},[2]\rangle$ & $+\ [(2)_{m+ n-4}]$ & $m,n\geq 2$ &
			$\cM^{1}$&
			$\cM^{1}$ & $\cM^{1}$ & $\cM^{1}$
			\\
			& \ref{item:ht=2_bench} & $\lbr n \rbr$ &$+\ \langle m,[2],[2],[2]\rangle$ & $+\ [3,(2)_{m+n-5}]$ & $m\geq 2$, $n\geq 3$ &
			$\cM^{1}$&
			$\cM^{1}$ & $\cM^{1}$ & $\cM^{1}$ \\
			\cline{2-6}
			&\ref{lem:ht=2_twisted-inseparable}\ref{item:beta=0}--\ref{item:C_columnar_L2H=0_k2=3}  & \multicolumn{4}{c|}{\multirow{2}{*}{types from  Table \ref{table:ht=2_char=2}; here $\nu\geq 3$ can be arbitrary }} & -& - & - & $\cM^{\nu-3}$ \\
			& \ref{item:C_fork_L2H=1}--\ref{item:twisted-bench} &&&& & -& - & - & $\cM^{\nu-2}$ \\[0.3em]
			\hline
			$1$ 
			& \ref{lem:ht=1_types}\ref{item:beta=3_other_3} & $\langle m; [2],[2,2,2],[2,2,2]\rangle$ & $+\ [2]$ & & $m\geq 3$ & $\cM^{1}$ & $\cM^{1}$ & $\cM^{1}$ & $\cMst^{1}$ \\ 
			& \ref{item:T2*=[2,2]_b>2} & $\langle 2;[2,3],[2,2],[2]\rangle$ &&&& 2& 2 & 2 & 2 \\
			&  & $\langle 2;[2,2,2],[3],[2]\rangle$ &&&& 1& 1 & 2 & 1 \\
			& \ref{item:T1=0,T=T2=[2]_b>2} & $\langle 2; [m,2],[2],[2]\rangle$ &&& $m\geq 2$ & 1& 1 & 1 & 2 \\
			& \ref{item:TH=[3,2]_T=[3]_b>2} & $\langle 2;[3,2],[3],[2]\rangle$ & $+\ [2]$ &&&1& 1 & 1 & 2 \\
			&  & 
			$\langle 2;[3,2],[2,2],[2]\rangle$ &&&& 1 & 2 & 1 & 2 \\
			&  \ref{item:TH=[2,2]_b>2}  & $\langle 2;[2,3],[2,2],[2] \rangle$ & $+\ [3]$ &&& 1 & 1 & 1 & 2 \\
			& &$\langle 2;[3,2], [2,2],[2]\rangle$ & $+\ [2]$ &&&1& 1 & 2 & 2 \\
			& & $\langle 2;[{2},2],[2,2],[2]\rangle$ &&&& 1& 1 & 2 & 2 \\
			& & $\langle 2;[2,2,2],[2,2],[2]\rangle$ &&&& 1 & 1 & 2 & 3 \\
			& & $\langle 2;[2,2,2,2],[2,2],[2]\rangle$ &&&& 2 & 2 & 3 & $3_{\star}$ \\
			& & $\langle 2;[k],[2,2],[2]\rangle$ & $+\ [(2)_{k-2}]$ && $6\geq k\geq 3$ &1& 1 & 1 & 2 \\
			& \ref{item:TH_long_columnar_b>2} & $\langle 2;[T,m,2],[2],[2] \rangle$ & $+\ T^{*}$ && $m\geq 2$& 1 & 1 & 1 & 2 \\
			&\ref{item:TH_long_other_b>2} & $\langle 2; [T,r,T^{*},{m},2],[2],[2]\rangle$ & $+\ [(2)_{r-2}]$ && $r,m\geq 2$ & 1 & 1 & 1 & $\cMst^{1}$ \\
			& \ref{item:F'_fork_T=[3]_b>2} & $\langle 2,[2,2,2,2,2],[3],[2]\rangle$ & $+\ [2]$ &&& 1 & 1 & 1 & $\cMst^{1}$ \\
			& \ref{item:F'_fork_T=[2,2]} & $\langle 3,[2,2,2,2,2],[2,2],[2]\rangle$ &&&& $\cM^{1}$ & $\cM^{1}$ & $\cM^{1}$ & $\cM^{1}$ \\ 
			& \ref{item:T2=[2,2]_b>2} & $\langle 2;[2,{2},2,2],[3],[2]\rangle$ & $+\ [2]$ &&& 1& 1 & 2 & 1 \\
			& \ref{item:T2=[2]_b>2} & $\langle 2;[(2)_{r+1}],[3],[2] \rangle$ & $+\ [r]$ & $+\ [2]$ &  $4\geq r \geq 2$ & 1 & 1 & 1& 2 \\	
			& & $\langle 2;[(2)_{r+1}],[2,2],[2]\rangle$ & $+\ [r]$ && $3\geq r \geq 2$ & 2& 2 & 2 & 2 \\
			& \ref{item:ht=1_bench} & $\lbr m \rbr $ & $+\ [2]+\ [2]$ & $+\ [2]+\ [2]$ &  $m\geq 3$ & $\cM^{1}$ & $\cM^{1}$ & $\cM^{1}$ & $\cM^{1}$ \\
			& \ref{item:ht=1_bench-big_b>2} & $\lbr 2,2, m \rbr $ & $+\ [2]$ & $+\ [2]$ & $m\geq 3$  & 1 & 1 & 1 & $\cMst^1$ \\
			& \ref{item:ht=1_bench_two_a,b>2}  & $\lbr 2,2, m,2,b\rbr$ & $+\ [3,(2)_{b-3}]$ && $b\geq 3$, $m\geq 2$ & 1 & 1 & 1 & $\cMst^1$ \\
			& & $\lbr 2,2, m,2,2\rbr$ &&& $m\geq 3$ & 1 & 1 & 1 & $\cMst^2$ 
		\end{tabular}
		\caption{Proposition \ref{prop:moduli}: rational log canonical del Pezzo surfaces of rank one and height at most two, which, for some $\cha\kk$, are \textbf{not} uniquely determined by their singularity types. Symbol \enquote{$\cM^{d}$} means that  $\Pht(\cS)$ is represented by an almost faithful family of dimension $d$, see Definition  \ref{def:moduli}\ref{item:def-family-faithful}, and subindex \enquote{$\star$} marks exceptions in Proposition \ref{prop:moduli-hi}\ref{item:moduli-hi-infinite},\ref{item:moduli-hi-finite}, see Table \ref{table:exceptions-to-moduli}.}
		\label{table:exceptions}	
	\end{table}
}
\vspace{-2em}
{\renewcommand{\arraystretch}{1.2}	
\begin{table}[htpb]
	\addtolength{\leftskip} {-2cm} 
	\addtolength{\rightskip}{-2cm}
	\begin{tabular}{r|c|c|rlll|c|c}
		\ref{prop:moduli-hi} & $\height$ & $\cha \kk$ divides & \multicolumn{4}{c|}{singularity type} &  $\#\Pht$ & $h^{1}$ \\ \hline \hline 
		\ref{item:moduli-hi-infinite} & 1 & 2 & \ref{lem:ht=1_types}\ref{item:beta=3_other_3} & $\langle m,[2],[2,2,2],[2,2,2]\rangle$ & $+\ [2]$ & & $\cM^{1}$ & $2$ \\ \cline{4-9} 
		&&& \ref{lem:ht=1_types}\ref{item:TH_long_other_b>2} & $\langle 2; [T,r,T^{*},{m},2],[2],[2]\rangle$ & $+\ [(2)_{r-2}]$ & $2|d(T)$ & $\cM^1$ & $1,2$  \\
		&&& \ref{item:F'_fork_T=[3]_b>2} & $\langle 2,[2,2,2,2,2],[3],[2]\rangle$ &$+\ [2]$& & &  \\ 		
		&&& \ref{item:ht=1_bench_two_a,b>2} & $\lbr 2,2, m,2,b\rbr$ & $+\ [3,(2)_{b-3}]$&  & & \\
		&&& \ref{item:ht=1_bench-big_b>2} & $\lbr 2,2, m \rbr$ & $+\ [2]$ & $+\ [2]$ & &    \\
		\cline{4-9} 
		&&& \ref{lem:ht=1_types}\ref{item:ht=1_bench_two_a,b>2} & $\lbr 2,2, m,2,2\rbr$ &&& $\cM^2$ & 2, 3 \\
		\hline 
		\ref{item:moduli-hi-finite} & 1 & 2 & \ref{lem:ht=1_types}\ref{item:TH=[2,2]_b>2} & $\langle 2;[2,2,2,2],[2,2],[2]\rangle$ &&& $3$ & $0,1,3$ \\
		\hline
		\ref{item:moduli-h1} & 1 & $d([T_{1}^{*},m,T_{2}^{*}])$ & \ref{lem:ht=1_types}\ref{item:chains_not-columnar} & $[T_{1},r_{1},T_{1}^{*},{m},T_{2}^{*},r_{2},T_{2}]$ & $+\ [(2)_{r_{1}-2}]$ & $+\ [(2)_{r_{2}-2}]$ & 1 & 1 \\		 
		& & $d(T)$ & 
		\ref{item:beta=3_other} & $\langle {m};[T^{*},r,T],[2],[2] \rangle$ & $+\ [(2)_{r-2}]$ & $+\ [2]+\ [2]$ & & \\ \cline{3-7}
		&& $d([2,m,T^{*}])$ & \ref{lem:ht=1_types}\ref{item:TH_long_other_b>2} & $\langle b; [T,r,T^{*},{m},2],[2],[2]\rangle$ & $ +\ [(2)_{r-2}]$ & $+\ [3,(2)_{b-3}]$ && \\ 
		&& && $\langle 2; [T,r,T^{*},{m},2],[2],[2]\rangle$ & $+\ [(2)_{r-2}]$ & $2\nmid d(T)$ & &  \\  \cline{3-7}
		&& $2(m-1)$ & \ref{lem:ht=1_types}\ref{item:ht=1_bench_two_a,b>2} & $\lbr a,2, m,2,b\rbr$ & $+\ [(2)_{a-3},3]$ & $+\ [3,(2)_{b-3}]$ && \\
		\cline{3-7}
		&& $m-1$ & \ref{lem:ht=1_types}\ref{item:ht=1_bench_two_a,b>2} & $\lbr 2,2, m,2,b\rbr$ & $+\ [3,(2)_{b-3}]$ &&&  \\
		&& but not $2$ &  & $\lbr 2,2, m,2,2\rbr$ &&&&  \\
		\cline{3-7}		 
		&& 2 & \ref{lem:ht=1_types}\ref{item:beta=3_other_2} & $\langle {m}; T,[2,2,2],[2]\rangle$ & $+\ (T^{*})$ & $+\ [2]$ && \\
		&&  & \ref{item:ht=1_bench-big_b>2} & $\lbr b,2, m\rbr$ & $+\ [(2)_{b-3},3]$ & $+\ [2]+\ [2]$ && \\   
		&&  & \ref{item:F'_fork_T=[3]_b>2} & $\langle b,[2,2,2,2,2],[3],[2]\rangle$ & $+\ [(2)_{b-3},3,2]$ &&& \\ 
		\cline{2-7}
		& 2 & $d([T,T_1])$ & \ref{lem:ht=2,untwisted}\ref{item:V-chains_c=1_rivet} & $[{T_{2}},T^{*},r_2,T,{T_1},r_1,T_1^{*},T_2^{*}]$ & $+\ [(2)_{r_1-2}]$ & $+\ [(2)_{r_2-2}]$ && \\
		\cline{3-7}
		&& $3$ & \ref{lem:ht=2,untwisted}\ref{item:lc_rivet} & $\langle r,[2,2,2,2,2],[2,2],[2]\rangle$ & $+\ [(2)_{r-3},3]$ &&& \\ \cline{3-7}
		&& $d(T_1)$ & \ref{lem:ht=2,untwisted}\ref{item:rivet_nu=3} &
		$[T_{2},{m},T_{1},r,T_{1}^{*},{n},T_{2}^{*}]$ & $+\ [(2)_{m+ n-3}]$ & $+\ [(2)_{r-2}]$ &&  \\  
		&&&
		\ref{item:rivet_nu=3-fork-1} & $\langle [2,n,T^{*},r,T],
		[2],[2]\rangle$ & $+\ [3,(2)_{m+n-4}]$ & $+\ [(2)_{r-2}]$ &&  \\
		&&&
		\ref{item:rivet_nu=3-fork-2} & 
		$\langle m,
		[2,2,2,2,2]
		,[3],[2]\rangle$ & $+\ [4,(2)_{m-2}]$ &&& \\
		&&&&  $\langle m,
		[2,2,2,2,2]
		,[2,2],[2]\rangle$ & $+\ [2,3,(2)_{m-2}]$ &&&
		\\
		\cline{3-7}
		&& $d(T_2)$ &  \ref{lem:ht=2,untwisted}\ref{item:V-chains_c=2_[2]_T1=0} &
		$\langle {n}; [(2)_{m+ n-3}]*[T^{*}_2,r,T_2],[2],[2]\rangle$ & $+\ 
		[2,{m},2]$ & $+\ [(2)_{r-2}]$& & \\ 
		&&& \ref{item:V-chains_c=2_[2]}  & 
		$\langle {n}; T_1^{*}*[(2)_{m+ n-3}]*[T^{*}_2,r,T_2],[2],[2]\rangle$ & $+\ \langle {m};T_1\trp,[2],[2]\rangle$ & $+\ [(2)_{r-2}]$ & &
	\end{tabular}
	\caption{Exceptions in Proposition \ref{prop:moduli-hi}, cf.\ Table \ref{table:exceptions-to-exceptions}.}
	\label{table:exceptions-to-moduli}
\end{table}
}
\end{small}

\clearpage
\restoregeometry

\subsection{Canonical del Pezzo surfaces of rank one}

\begin{small}
\begin{table}[htbp]
	\begin{tabular}{c|c|c|ccc| cccc}	
		\multirow{2}{*}{$\height$} &\multirow{2}{*}{construction}& {\multirow{2}{*}{singularity types}} & \multicolumn{3}{c|}{$\#$ iso.\ classes if $\cha \kk$} & \multicolumn{4}{c}{rational elliptic boundaries if  $\cha \kk$} \\ 
		&&& $\neq 2,3$ & $=3$ & $=2$ & $\neq 2,3,5$ & $\ \ =5\ \ $ & $\ \ =3\ \ $ & $\ \ =2\ \ $  \\ \hline\hline
		1 & \ref{rem:canonical_ht=1} & $\uline{\rA_{1}}$, $\uline{\rA_{1}+\rA_{2}}$, $\rA_{4}$ & 1 & 1 & 1 & C, N & C, N & C, N & C, N  \\
		&& $\uline{2\rA_{1}+\rA_{3}}$, $\rA_{1}+\rA_{5}$, $\rA_{7}$ & 1 & 1 & 1 & N & N & N & C\\
		&& $\uline{3\rA_{1}+\rD_{4}}$, $2\rA_{1}+\rD_{6}$ & 1 & 1 & 1 & - & - & - & $\rC_1$ \\
		&& $\rD_{5}$ & 1 & 1 & 2 & C, N & C, N & C, N & C, N \\
		&& $\rA_1+\rD_6$ & 1 & 1 & 2 & N & N & N & C\\
		&& $\rE_6$ & 1 & 2 & 2 & C, N & C, N & C, N & C, N \\
		&& $\rE_7$ & 1 & 2 & 3 & C, N & C, N & C, N & C, N \\
		&& $\rE_8$ & 2 & 3 & 3 & C, N & C, N & C, N & C, N \\
		&& $\rD_8$ & 1 & 1 & $\cM^{1}$ & N & N & N & C\\
		&& $\rA_1+\rE_7$ & 2 & 2 & 2 & N & N & N & C\\
		\hline
		2 & \ref{ex:ht=2}\ref{item:3A2_construction} & $\uline{3\rA_2}$ & 1 & 1 & 1 & N & N & C & N\\
		& \ref{ex:ht=2}\ref{item:A1+2A3_construction} & $\uline{\rA_1+2\rA_3}$ & 1 & 1 & 1 & N & N & N & - \\
		& \ref{ex:ht=2}\ref{item:A1+A2+A5_construction} & $\uline{\rA_1+\rA_2+\rA_5}$ & 1 & 1 & 1 & N & N & - & - \\
		& \ref{ex:ht=2}\ref{item:2D4_construction} & $2\rD_4$ & $\cM^1$ & $\cM^1$ & $\cM^1$ & - & - & - & - \\
		& \ref{ex:ht=2_meeting} & $\uline{2\rA_4}$ & 1 & 1 & 1 & N & C & N & N \\
		& \ref{ex:ht=2_twisted}\ref{item:2A1+2A3_construction} & $\uline{2\rA_1+2\rA_3}$ & 1 & 1 & - & - & - & - & - \\
		& \ref{ex:ht=2_twisted_cha=2} & $\uline{7\rA_1}$ & - & - & 1 & - & - & - & $\rC_{2}$\\ 
		& \ref{ex:remaining_canonical}\ref{item:3A2_tower}  & $\rA_2+\rA_5$, $\rA_8$ & 1 & 1 & 1 & N & N & C & N \\
		& & $\rA_2+\rE_6$ & 2 & 2 & 2 & N & N & C & N \\
		& \ref{ex:remaining_canonical}\ref{item:A1+2A3_tower} & $\rA_3+\rD_5$, $\rA_1+\rA_7$ & 1 & 1 & 1 & N & N & N & - \\
		& \ref{ex:remaining_canonical}\ref{item:7A1_tower} & $4\rA_1+\rD_4$ & - & - & $\cM^1$ & - & - & - & $\rC_{2}$\\
		\hline 
		4 & \cite[Example 7.1]{PaPe_MT} & $\uline{4\rA_2}$ & 1 & 1 & 1  & - & - & $\rC_{1}$ & - \\
		& \cite[Example 8.1(c)]{PaPe_MT} & $\uline{8\rA_1}$ & - & - & $\cM^{2}$ & - & - & - & $\rC_{3}$  
	\end{tabular}
	\caption{Canonical del Pezzo surfaces of rank one, other than $\P^2$. Underlined types refer to primitive surfaces, see Definition \ref{def:vertical_swap}\ref{item:def-swap-basic} and Remark \ref{rem:primitive_ht=2}.}
	\label{table:canonical}
\end{table}
\end{small}

\bibliographystyle{amsalpha}
\bibliography{bibl}

\end{document}